\documentclass[12pt]{amsart}

\calclayout
\usepackage{amssymb}

\usepackage{graphicx}
\usepackage{multirow, multicol, makecell, booktabs}
\usepackage[font=scriptsize]{caption}
\usepackage[font=scriptsize, margin=-.3cm]{subcaption}
\usepackage{amsthm}
\usepackage[T1]{fontenc}

\newcommand{\etalchar}[1]{$^{#1}$}

\usepackage{hyperref}

\IfFileExists{mathrsfs.sty}{\usepackage{mathrsfs}}
{\IfFileExists{euscript.sty}{\usepackage[mathscr]{euscript}}
{\let\mathscr\mathcal}}

\DeclareFontEncoding{OT2}{}{} 
 \newcommand{\textcyr}[1]{%
   {\fontencoding{OT2}\fontfamily{wncyr}\fontseries{m}\fontshape{n}%
    \selectfont #1}}
\newcommand{\Sha}{{\mbox{\textcyr{Sh}}}}

\newcommand{\Z}{\mathbb{Z}}

\newcommand{\Q}{\mathbb{Q}}

\newcommand{\C}{\mathbb{C}}
\newcommand{\R}{\mathbb{R}}
\newcommand{\PP}{\operatorname{\mathbf{Prob}}}
\newcommand{\X}{\mathscr X}
\newcommand{\Gal}{\operatorname{Gal}}
\newcommand{\rk}{\operatorname{rank\,}_{\Z}}
\newcommand{\ord}{{\operatorname{ord}}}
\newcommand{\f}{\mathfrak f_{\chi}} 
\newcommand{\p}{\mathfrak p} 
\newcommand{\ff}{\mathfrak f} 

\newcommand{\Qbar}{\overline{\Q}}

\newcommand{\Nm}{\operatorname{\bf{Nm}}}

\newtheorem{thm}{Theorem}[section] 
\newtheorem{prop}[thm]{Proposition}
\newtheorem{lem}[thm]{Lemma}

\newtheorem{pred}[thm]{Prediction}
\theoremstyle{definition}

\newtheorem{remark}{Remark}
\numberwithin{equation}{section}

\title[Small Values ]%
{Small Algebraic Central Values of Twists of Elliptic $L$-Functions }

\date{\today}

\author[H.~Kisilevsky and JB Nam]{Hershy Kisilevsky and Jungbae Nam}

\address[H.~Kisilevsky]%
{Department of Mathematics and Statistics and CICMA\\
     Concordia University \\
     1455 de Maisonneuve  Blvd. West\\
     Montr\'eal, Qu\'ebec, H3G 1M8, CANADA}
\email{hershy.kisilevsky@concordia.ca}

\address[Jungbae ~Nam]%
{Department of Mathematics and Statistics and CICMA\\
     Concordia University \\
     1455 de Maisonneuve  Blvd. West\\
     Montr\'eal, Qu\'ebec, H3G 1M8, CANADA}
\email{jungbae.nam@concordia.ca}

\begin{document}

\begin{abstract}
We consider heuristic predictions for small non-zero algebraic central values of twists of the $L$-function of an elliptic 
curve $E/\Q$ by Dirichlet characters.
We provide computational evidence for these predictions and consequences of them for instances 
of an analogue of the Brauer-Siegel theorem associated to $E/\Q$ extended to chosen families of cyclic extensions of fixed degree.
\end{abstract}

\maketitle

\section{Introduction}\label{Introduction}

Let $E$ be an elliptic curve defined over the rational field $\Q$ with $L$-function $L(E/\Q,s)$, conductor $N_E$, and 
denote by $w_E$ the sign of its functional equation. 
Then the Birch \& Swinnerton-Dyer Conjecture
relates the leading term of the Taylor expansion of $L(E/\Q,s)$ at $s=1$ to the arithmetic invariants of $E/\Q.$ In particular, 
it predicts that the order of vanishing of  $L(E/\Q,s)$ at $s=1$ is
equal to $r$, the $\Z$-rank of the Mordell-Weil group $E(\Q),$ and that 
\begin{equation}\label{BSD}
\frac{L^{(r)}(E/\Q,1)}{r!}=\frac{\vert\Sha(E/\Q)\vert\Omega_E R(E(\Q))\prod_{p <\infty }c_p}{\vert E(\Q)_{\rm{tors}}\vert^2},
\end{equation}
where the right hand side involves the usual invariants of $E/\Q.$ The number $\vert\Sha(E/\Q)\vert$ defined by Equation~\eqref{BSD} is sometimes called the {\it{analytic}} order of $\Sha(E/\Q)$ and denoted $\Sha_{\rm{an}}(E/\Q)$ (see {\it{e.g.}}~\cite{LMFDB}).
    
There are corresponding (generalized) conjectures for $E$ over finite extensions  $K/\Q.$  We focus on the Birch \& Swinnerton-Dyer conjecture which asserts that the $\Z$-rank of $E(K)$ (=$r_K$ say)   is equal to the order of vanishing of the $L$-function
$L(E/K,s)$ at $s=1$  and that
\begin{equation}\label{genBSD}
\frac{L^{(r_K)}(E/K,1)}{r_K!}=\frac{\vert\Sha(E/K)\vert\Omega_{E/K}R(E(K))\prod_{\p <\infty}c_{\p}}{\sqrt{\vert \Delta(K)\vert}(\vert E(K)_{\rm{tors}}\vert)^2},
\end{equation}
where the Shafarevich-Tate group $\Sha(E/K)$ is conjecturally finite, $\Omega_{E/K}$ is a product of the periods of $E$, the $c_{\p}$'s are the Tamagawa numbers at $\p$ adjusted as in 
~\cite{D-E-W} (see also~\cite{An}), and $\Delta(K)$ is the field discriminant of $K.$

 In the case of {\it abelian} extensions $K/\Q,$  we have the factorisation
 \begin{equation}\label{Lfact}
 L(E/K,s) = \prod_{\chi} L(E,s,\chi),
 \end{equation}
where the product is taken over all primitive Dirichlet characters $\chi$ attached to the field $K$.
 It follows that the behaviour of $L(E/K,s)$ at $s=1$ is 
determined by the values $L(E,1,\chi)$. The case of
quadratic extensions has been extensively studied (see {\it{e.g.\/}}~\cite{CKRS} ,~\cite{G} or~\cite{Go-M}).

For complex characters (characters of order $\geq 3$), the
{\it {vanishing}} of $L(E,1,\chi)$ is considered in~\cite {M-R},~\cite {DFK1},~\cite {DFK2},~\cite{FKK}, and~\cite{BRY} and some predictions are 
made regarding the frequency of such vanishings. Specifically, Conjecture 10.1 of~\cite{M-R} predicts that 
there are only finitely many {\it{even}} primitive Dirichlet characters $\chi$ of order $k$ for which both Euler's $\phi$ function $\phi(k) \geq 6$ and  
$L(E,1,\chi)=0.$ Rephrasing this in terms of the Birch \& Swinnerton-Dyer conjecture, the authors predict (Conjecture 10.2~\cite{M-R})
that for any (infinite) real abelian extension 
$F/\Q$ with only a finite number of subfields of degree $2,3$ or $5$, the Mordell-Weil group $E(F)$ is finitely generated.
 
In this article, we consider  the distribution of {\it{non-vanishing}} central values $L(E,1,\chi)$ as $\chi$ varies 
over certain families of primitive Dirichlet characters.  
The corresponding 
fields $K_{\chi}$ are cyclic $\Z/k_{\chi}\Z$-extensions of the rational field $\Q$, and Equation~\eqref{Lfact} holds.
The heuristics of~\cite {M-R},~\cite {DFK1},~\cite {DFK2}
would indicate that for almost all such characters ($100\%$), we would have  $L(E,1,\chi)\neq 0.$
In $\S~4$, we propose a probabilistic model from which we predict the distributions of these non-zero central values. We construct the series of probabilities (see Equation~\eqref{growth}) whose convergence and divergence informs the rates of occurrence of these values.  

In order to develop a heuristic from which we derive our speculations, we must make several assumptions. Note that each hypothesis  is
  arithmetic  and statistical in nature.


\subsection{\bf{Arithmetic Hypotheses}}\label{Arithmetic Hyp}

 For an elliptic curve $E$ defined over the rational field $\Q,$ we  assume  the generalized Birch \& Swinnerton-Dyer conjecture (as in Equation~\eqref{genBSD}), for $E$ viewed as a curve over  finite extensions $K/\Q$ .  We also assume
 the  generalized Lindel\"of hypothesis (see~\cite{G-H-P}) which states that for any $\epsilon >0$ we have $L(E,1,\chi) = O(\f^{\epsilon})$, where $\f$ is the conductor of $\chi$, 
 and the implied constant depends only on $\epsilon$ and $E$.


\subsection{\bf{Statistical Hypotheses}}\label{Statistical Hyp}

In \S~\ref{Probabilities} we consider a totally real field $F/\Q,$ with $[F:\Q]=n$, and the usual embedding $\psi:F \rightarrow \R^n$  
 sends $\alpha  \mapsto (\gamma_1(\alpha), \gamma_2(\alpha), \ldots,\gamma_n(\alpha))\in \R^n$ where
 $\gamma_1, \gamma_2, \ldots,\gamma_n$ are the distinct embeddings of $F$ into $\R$. Then the image of the ring of
 integers $\psi(\mathcal O_F)\subset \R^n$ is a sublattice of $\R^n$. We assume that the  probability that the image of 
 an integer  lies in a region  $\mathcal T \subset \R^n$ is given by the relative volume of $\mathcal T$, and we further assume that 
 the coordinates $\gamma(\alpha)$ are (generically) independent identically distributed random variables. (We say 
 ``generically'' since this is certainly not the case for $\alpha$ lying in a proper subfield of $F$).
\begin{remark}  In the following we shall always use the word ``twist''  to refer to the $L$-function twisted by a character, and not the geometric twist of an
algebraic variety. We also restrict ourselves to twists of the $L$-function by primitive Dirichlet characters $\chi$ with $\f$ coprime to its analytic conductor.  This allows us to control the conductor and the order 
of the resulting character of the twisted  $L$-function. Otherwise,  had we started, for example, with an $L$-function $L(E,s,\psi)$ where $\psi$ is a primitive Dirichlet character
of some even order $2k$, and twisted by characters $\chi$ of the same order $2k$ with $\chi = \psi^{-1}\chi_D$ where $\chi_D$ is the quadratic character of 
conductor $D$ (with $\gcd(D,\ff_{\psi}N_E)=1$), then the resulting twisted $L$-functions $L(E,s,\psi\chi)$ corresponding to the primitive character associated to $\psi\chi$
would be the quadratic twists $L(E,s,\chi_D).$ The central values of such quadratic twists would be substantially different from those of twists by characters $\chi$ of order $2k$ and with coprime conductors. 
\end{remark}

In \S~\ref{Algebraic central values}, we follow~\cite{M-T-T} to write
 \begin{equation}\label{algpart}
L(E,1,\chi)= \frac{\Omega_{\chi}}{2\tau(\overline\chi)}\times L_E^{\rm{alg}}(\chi).
\end{equation}
Here $\overline\chi=\chi^{-1}$ is the complex conjugate character, $\Omega_{\chi}$ and $\tau(\overline\chi)$ are as in \S~\ref{Algebraic central values}, and
$L_E^{\rm{alg}}(\chi) $ is an algebraic integer in $ \Q(\chi)$, the cyclotomic field generated over $\Q$ by the values of $\chi$. 
We will see below that
\begin{equation*}
L(E,1,\chi) = 0 \Longleftrightarrow  L_E^{\rm{alg}}(\chi)= 0
\end{equation*}
and that $\sigma( L_E^{\rm{alg}}(\chi)) =  L_E^{\rm{alg}}(\chi^{\sigma})$ for all $\sigma \in \Gal(\Q(\chi)/\Q).$ 
Hence $L(E,1,\chi) = 0 \Longleftrightarrow L(E,1,\chi^{\sigma}) = 0$ for all $\sigma \in \Gal(\Q(\chi)/\Q).$

More generally, if $\chi$ is a primitive Dirichlet character of order $k$,  and $K= K_{\chi}$ is the cyclic extension of $\Q$ of degree $k$ corresponding to $\chi$, and $L(E/\Q,s,\chi)$
is the twist of $L(E,s)$ by $\chi$, then (see \cite{Ro}) the order of vanishing of $L(E/\Q,s,\chi)$ at $s=1$ is conjectured (generalized BSD) to be the rank of the
``$\chi$-component'' $E(K)^{\chi}$ of $E(K)$.  Here $\rk E(K)^{\chi}= \dim_{\C} \bigl(\C\otimes E(K)\bigr)^{\chi}$ is the dimension of the $\chi$ eigenspace of $\C\otimes_{\Z} E(K)$ as a $\Gal (\Qbar/\Q)$-space. 
Kato's  result \cite{Kato} generalizing Kolyavagin's theorem \cite{Ko} (see  \cite{Sch}) shows
that if $L(E/\Q,1,\chi)\neq 0 $, then both $E(K_{\chi})$ is finite  and that the ``$\chi$-part'' of the Shafarevich-Tate group $\Sha(E/K)(\chi)$ is finite
 where $\Sha(E/K)(\chi)$ is
the subgroup of $\Sha(E/K)$ on which $\Gal(K/\Q)$ acts via the  (rational) $\Q$-representation associated to the character $\chi$ (and its Galois conjugates).

If $\ord_{s=1}L(E/K,s)=\ord_{s=1}L(E/\Q,s)$, then we may
substitute $L(E,1,\chi)$ in Equation~\eqref{algpart} into the factored expression for $L(E/K,1)$ of Equation~\eqref{Lfact}. Then in view of Equation~\eqref{genBSD}, the Galois equivariance of $L_E^{\rm{alg}}(\chi)$ and the conductor-discriminant formula, we think of the integer 
$\vert \Nm_{\Q(\chi)/\Q}(L_E^{\rm{alg}}(\chi))\vert$ as  an  ``analytic order''  
$\vert\Sha_{\text{an}}(E/K)(\chi)\vert $
of the ``$\chi$-part'' of the Shafarevich-Tate group $\Sha(E/K)$. Under the general BSD conjecture we expect this to hold
 up to finitely many factors depending only on $E$ and the order of $\chi$ (see~\cite{D-E-W}). (For example, since the 
group order  $\vert \Sha(E/K)\vert$ is not an isogeny invariant, it might be slightly different from  $\vert\Sha_{\text{an}}(E/K)(\chi)\vert $).

The authors would like to thank the anonymous referees for their valuable insights and helpful suggestions in the preparation of this paper.

\section{Predictions}\label{Predictions}

\subsection{\bf{Main Prediction}}\label{Main Prediction}

 As mentioned in the Introduction (\S~\ref{Introduction}) the non-vanishing of $L(E,1,\chi)$ is expected to occur for $100\%$ of characters of fixed order $k\geq3.$
Kato's result then translates these non-zero values to inform the orders of the corresponding components of $\Sha(E/K)$ which play a role in the arithmetic of elliptic curves analagous to that of ideal class groups of number fields. Studying the distributions of these groups, their growth rates in
twist  families, is a motivating interest of this article.

Based on the Arithmetic and Statistical hypotheses of \S~\ref{Introduction}, 
we propose a model in \S~\ref{Probabilities}  which allows us to make a number of predictions on the behaviour of  
$\vert\Sha_{\text{an}}(E/K)(\chi)\vert =\vert \Nm_{\Q(\chi)/\Q}(L_E^{\rm{alg}}(\chi))\vert$
as $\chi$ runs over a chosen  family of primitive Dirichlet characters, and $K=K_{\chi}$.
Note that the generalized Lindel\"of hypothesis implies that 
\[
\vert\Sha_{\text{an}}(E/K)(\chi)\vert =\vert \Nm_{\Q(\chi)/\Q}(L_E^{\rm{alg}}(\chi))\vert \leq \f^{\frac{\phi(k)}{2}+\epsilon}
\]
where $\chi$ is a Dirichlet character of order $k$, and $\f$ is the conductor of $\chi.$
For $L_{\chi}>0,$ let $n_{k,E} (X;L_{\chi})$ denote the number of  primitive 
Dirichlet characters $\chi$ of order $k $ and with conductor $\f$ relatively prime to $N_E$ and $\f \leq X$
such that $0 \neq \vert \Nm_{\Q(\chi)/\Q}(L_E^{\rm{alg}}(\chi))\vert  \leq L_{\chi}^2$ (so we may take $L_{\chi} \leq \f^{\phi(k)/4}$).

\begin{pred}\label{A}
Fix an integer $k\geq 3$ and a real number $c$, with $ 0\leq c < \phi(k)/4.$ Let $L=L_{\chi} =O(\f^{c})$ and let $X>0.$
Then under the hypotheses of \S~\ref{Introduction}, as $X\to\infty,$ the number $n_{k,E} (X;L)$ grows as:
\begin{enumerate} 
\item
\text{Case:\ \ } $c=0$
\[
\begin{aligned}
n_{k,E}(X;L)   &\asymp X^{1/2}\log^B(X) & {\text{ if \ \   }} \phi(k)=2\\
         &\asymp  \log^{B+1}(X) & {\text{if \ \ }} \phi(k)=4\\
         & {\phantom{\asymp}}  {\text{\ \   is bounded\ \ }} & {\text{\  if \ \  }} \phi(k)\geq 6         
\end{aligned}
\]

\item
\text{Case: \ \ }$0<c<\phi(k)/4$
\[
\begin{aligned}
          n_{k,E}(X:L) &\asymp X^{c-(\phi(k)/4 -1)}\log^B(X) &{\text{ if \ \ }} &\max{\{0, \frac{\phi(k)}{4} -1\}}< c<\frac{\phi(k)}{4}  \\
         &\asymp \log^{B+1}(X) & {\text{ if \ \ }} &c =\frac{\phi(k)}{4}-1 > 0 \\
         & {\phantom{\asymp}}  {\text{is bounded\ \ }} &{\text{ if \ \ }}  &0 < c<\frac{\phi(k)}{4}-1.\\
\end{aligned}
\]

\end{enumerate}
  \end{pred}
  

Here, as usual, $\phi(k)$ is the Euler $\phi$-function (totient) evaluated at $k,$ and $B=\sigma_0(k)+\phi(k)/2-3$ where $\sigma_0(k)$ 
is the number of positive divisors of $k$ (sometimes denoted $d(k)$ or $\tau(k)).$ The value of $B$ arises in Prediction~\ref{A} in the  power of the logarithm 
due to the count of 
the number of Dirichlet characters $\chi$ of order fixed order $k$ with conductors $\f$ coprime to a fixed integer $N$ and for which $\f \le X.$ 
This occurs from the application of a Delange-Ikehara-Tauberian theorem to the generating function counting such characters, which has a pole at $s=1$ 
of order $\sigma_0(k) -1.$

For example, letting  $k=3,$ we have $\phi(k)=2$ and $B=\sigma_0(3)+\phi(3)/2 -3= 0.$ Taking  $c=0$ and $L>0$ a fixed integer,
  the predicted growth rate for ``small'' non-zero algebraic values,  $0 \neq \vert \Nm_{\Q(\chi)/\Q}(L_E^{\rm{alg}}(\chi))\vert \leq L,$ as $\chi$ ranges over cubic characters is predicted to be
  order $\asymp X^{1/2}.$ 
This contrasts with the situation for the distribution of class numbers in quadratic fields. For imaginary quadratic fields, 
the class number is at most $L$ only finitely often for any fixed $L>0$, and the unit group is finite. Moreover, for real quadratic fields, we expect infinitely many 
to have class number $1$ (or at most $L$) and large fundamental units.
For the elliptic curve $E$ and fixed $L > 0,$ we expect  that for infinitely many characters $\chi$ of order $3$ (and cyclic cubic fields $K=K_{\chi}$) we have
    $0 \neq \vert \Nm_{\Q(\chi)/\Q}(L_E^{\rm{alg}}(\chi))\vert \leq L$ and also that  $E(K)/E(\Q)$ is finite.\\

For $k=5,$ we have $\phi(k)=4$ and $B=\sigma_0(5)+\phi(5)/2 -3= 1,$ so the predicted growth rate for ``small'' non-zero algebraic values of 
quintic twists  $0 \neq \vert \Nm_{\Q(\chi)/\Q}(L_E^{\rm{alg}}(\chi))\vert \leq L$
is of the order $\asymp \log^2(X).$\\

For $k=6,$ we have $\phi(k)=2$ and $B=\sigma_0(6)+\phi(6)/2 -3= 2,$ so the predicted growth rate for ``small'' non-zero algebraic values of 
sextic twists  $0 \neq \vert \Nm_{\Q(\chi)/\Q}(L_E^{\rm{alg}}(\chi))\vert \leq L$
is of the order $\asymp X^{1/2}\log^2(X).$

The experimental computations (see \S~\ref{computations}) seem to support Prediction~\ref{Main Prediction}.
 
 \subsection{\bf{Consequences of the Main Prediction}}\label{More Predictions}
 
 \bigskip

 \begin{pred}\label{B}
 Fix an integer $L > 0.$ Let $\mathcal C(N_E)$ be any set of  primitive Dirichlet characters $\chi$ with conductors $\f$ coprime to $N_E,$ 
 whose orders $k_{\chi}$ satisfy 
 $\phi(k_{\chi}) \geq 6.$ Then there are only finitely many $\chi \in \mathcal C(N_E)$ such that 
 $\vert \Nm_{\Q(\chi)/\Q}(L_E^{\rm{alg}}(\chi))\vert \leq L.$
\end{pred}
 
 \begin{remark}\label{RemarkForB}
 The $c=0$ case of Prediction~\ref{A} asserts that for each fixed number $L$ and order $k$ with $\phi(k)\geq 6$ we expect only a finite number of characters of order $k$ for which 
  $\vert \Nm_{\Q(\chi)/\Q}(L_E^{\rm{alg}}(\chi))\vert \leq L.$ Because of the ambiguity of the implied constants involved, this does not  directly imply  Prediction~\ref{B}. However using the method in~\cite{M-R}, we can establish 
 Prediction~\ref{B} for all characters $\chi$ with $\gcd(\f, N_E)=1$ whose orders $k_{\chi}\geq 10.$ Then applying Prediction~\ref{A}  finitely many times yields Prediction~\ref{B}.    
 \end{remark} 
  
  \begin{pred}\label{C} Under the hypotheses above, for almost all ($100\%$) Dirichlet characters of order $k\geq 3$ and any $\epsilon >0,$
  we have 
  \begin{equation}\label{upper_lower_bounds}
  \f^{\frac{1}{2}-\epsilon}\ll \vert L_E^{\rm{alg}}(\chi)  \vert \ \ll \f^{\frac{1}{2}+\epsilon}.
  \end{equation}
  Consequently for almost all  $\chi$ of order $k$, for which $L_E^{\rm{alg}}(\chi)\neq 0$ we have 
  \[
  \f^{\phi(k)/2-\epsilon}\ll \vert\Sha_{\rm{an}}(E/K_{\chi})(\chi)\vert \ll \f^{\phi(k)/2+\epsilon}.
\] 
  \end{pred}

  \begin{remark} The upper bound is the  generalized Lindel\"of hypothesis and is conjectured to hold for all primitive Dirichlet characters. The novelty 
  of Prediction~\ref{C} is that the lower bound holds for {\it almost all} primitive Dirichlet characters of given order $k \ge 3$. In \S~\ref{Probabilities} we 
  take variable $L=\f^c$ for  $ 0<c<\phi(k)/4$ using Equation~\eqref{conj_pos_c} to predict the lower bound for almost all characters $\chi$ of order $k$.
  Similar questions are discussed in~\cite{D-W}.
  \end{remark}
  
  \subsection{\bf{Brauer-Siegel Quotients}}\label{Brauer-Siegel}
  
  The Brauer-Siegel theorem (see~\cite{Lang})  states that if $\{K_j\}_{j \ge 1}$ is an infinite sequence of finite Galois extensions of $\Q$ such that as $j \rightarrow \infty$
\[
  \frac{[K_j:\Q]}{\log \vert \Delta(K_j)\vert} \to 0
  \]
  then as $j \rightarrow \infty$
 \[
  \frac{\log (h_{K_j}R_{K_j})}{\log \sqrt{\vert \Delta(K_j) \vert}} \to 1,
  \]
  where $h_{K_j},R_{K_j},$ and $\Delta(K_j)$ are the class number, the regulator, and the discriminant of $K_j$ respectively.

  There are a number of extensions and conjectural analogues of the Brauer-Siegel theorem that can be found in the literature. The work of Tsfasman and 
  Vl\u{a}dut(\cite{Ts-Vl}) and  Zykin(\cite{Zy}) relax the hypotheses of the theorem and that of Hindry and Pacheco (\cite{Hi-Pa},\cite{Hi}) formulate 
  analogous statements in wider geometric settings. Also Ulmer (\cite{Ul}) examines the conjectures algebraically in the 
  context of certain varieties over function fields of finite characteristic.
  
 Below we consider analogous predictions for the Brauer-Siegel quotients associated
 to an elliptic curve $E/\Q$ extended to chosen families of cyclic extensions $K/\Q$ of fixed degree.
   
Let $L > 0$ be a fixed integer and let $\mathcal F_{k,N}(L)$ be the family of all primitive Dirichlet characters of order $k$ with $\gcd(\f, N)=1$ such that 
 $0 < \vert \Nm_{\Q(\chi)/\Q}(L_E^{\rm{alg}}(\chi))\vert \leq L.$ Then Prediction~\ref{A} implies that $\mathcal F_{k,N}(L)$ is an infinite set 
 if $\phi(k)\leq 4.$

For example for $k = 3,$ let $K_{\chi}/\Q$ be the cyclic cubic extension corresponding to the character $\chi \in \mathcal F_{3,N}(L)$. 
 Then since  $L(E,1,\chi) \neq 0$ 
 the Birch \& Swinnerton-Dyer conjecture predicts that $\vert \Sha(E/K_{\chi})\vert$ is finite and differs from $\vert \Sha(E/\Q)\vert\times  \vert \Nm_{\Q(\chi)/\Q}(L_E^{\rm{alg}}(\chi))\vert$ by a factor bounded independently of $\chi,$
 and that the elliptic regulator $R(E(K_{\chi})) =R(E(\Q))$ up to a factor also bounded independently of $\chi$. Since  families of cyclic cubic extensions $K_{\chi}/\Q$ with   $\chi \in \mathcal F_{3,N}(L)$  are predicted (by Prediction~\ref{A}) to be 
  infinite, we can take the limit (for $K_{\chi}$ in such a family and $\f \le X$ ) as $X\to\infty$ to get an analogue of the  Brauer-Siegel limits.
  
  Similarly if $\phi(k)\leq 4,$ then the Prediction~\ref{A} asserts that the family  of cyclic $\Z/k\Z$-extensions $K_{\chi}/\Q$ with $\chi \in \mathcal F_{k,N}(L)$ is an {\it{infinite}} set and so we have:
  
  \begin{pred}\label{Da} Let $\phi(k)\leq 4.$ For an infinite family of cyclic $\Z/k\Z$-extensions $K_{\chi}/\Q$ with $\chi \in \mathcal F_{k,N}(L)$ we have 
\[
   \lim_{\substack{\f \le X\\X\to\infty}} \frac{\log(\vert\Sha_{\rm{an}}(E/K_{\chi})\vert)R(E(K_{\chi})))}{\log(\sqrt{\vert\Delta(K_{\chi})\vert})}=0.
 \]
\end{pred}

More generally, take $L =\f^c.$  Then for  $\phi(k)/4 - 1<c'< c < \phi(k)/4,$  the set difference $\mathcal F_{k,N}(\f^c)\setminus \mathcal F_{k,N}(\f^{c'})$ is the family of all primitive Dirichlet characters of order $k$ with $\gcd(\f, N)=1$ such that 
 $\f^{c'} < \vert \Nm_{\Q(\chi)/\Q}(L_E^{\rm{alg}}(\chi))\vert \leq \f^c.$ Then Prediction~\ref{A} implies that  $\mathcal F_{k,N}(\f^c)\setminus \mathcal F_{k,N}(\f^{c'})$ is an infinite set (see \S~\ref{c>0}). This allows us to make the following prediction as $X\to\infty$:
   
  \begin{pred}\label{Db} Let  $\phi(k)/4-1 \le c' < c <\phi(k)/4.$ Under the hypotheses of \S~\ref{Introduction},
   the  family of extensions $K_{\chi}/\Q$ with $\chi \in \mathcal F_{k,N}(\f^c) \setminus \mathcal F_{k,N}(\f^{c'})$,   satisfies
 \begin{align*}
  \frac{4c'}{\phi(k)} &\leq
 \liminf_{\substack{\chi \in \mathcal F_{k,N}(\f^c)\\ \f\leq X}} \frac{\log(\vert\Sha_{\rm{an}}( E/K_{\chi})(\chi)\vert) R(E(K_{\chi})(\chi))}{\log(\vert\sqrt{\Delta(\chi)}\vert)} \\
 &\leq
 \limsup_{\substack{\chi \in \mathcal F_{k,N}(\f^c)\\ \f \leq X}} \frac{\log(\vert\Sha_{\rm{an}}(E/K_{\chi})(\chi)\vert) R(E(K_{\chi})(\chi))}{\log(\vert\sqrt{\Delta(\chi)}\vert)}\leq\frac{4c}{\phi(k)}
 \end{align*}
 where $\sqrt{\vert \Delta(\chi)\vert} : = \vert\prod_{\substack{1 \leq i <k \\(i,k)=1}}\tau(\chi^i)\vert$.
\end{pred}
 Here (under generalized BSD) the ``$\chi$-part'' of the regulators $R(E(K_{\chi})(\chi)) =1$ as $L_E^{\rm{alg}}(\chi^i) \neq 0 $ for $\gcd(i,k)=1.$
 
 In the last section (\S~\ref{computations}) we provide
the results of  numerical  computations which  seem to support these predictions.   We would like to thank
Evan Dummit for his computations and Andrew Granville  for  the proof of Lemma~\ref{geombounds}.

\section{ Preliminaries and Notation}\label{Preliminaries}

\subsection{\bf{Number of primitive characters}}

\smallskip

Let $\ff$ be a positive integer, either odd, or divisible by $4.$ Then the multiplicative group $(\Z/\ff \Z)^\ast$
is naturally isomorphic with the Galois group $G=\Gal(\Q(\zeta_{\ff})/\Q)$ of the cyclotomic field of $\ff^{\rm{th}}$ roots of unity.  The group of characters $\widehat G$ of $G$ can then be identified with  the  group of Dirichlet characters 
$\widehat{(\Z/\ff \Z)^\ast}.$ Since each Dirichlet character modulo $\ff$ is induced from a {\it{unique}} primitive character of conductor $\f$ dividing $\ff$, we  identify $\widehat G=\widehat{(\Z/\ff \Z)^\ast}$ with the set of primitive 
characters $\chi$ of conductor $\f $ dividing $\ff.$ The trivial character with $\f =1$ will be denoted $1,$ and for $\chi \in \widehat{(\Z/\ff \Z)^\ast}$ let $\ord(\chi)$  denote the order of $\chi.$ 

In the following we restrict to using only primitive characters as this allows for simpler functional equations and factorizations of $L$-functions  (without extra Euler factors).


For $k \geq 2$ define the following sets of primitive Dirichlet characters:
\[
\mathcal A_k : =\{\chi \mid \chi^k=1\}  {\rm{\ \ and\ \ }}
\mathcal B_k :  =\{\chi \in \mathcal A_k \mid \ord(\chi)=k\}
\]
Fix $N\geq 1$  a positive integer and let  $X>0.$ Define
\begin{align*}
\mathcal A_{k,N}(X) &:=\{\chi \in \mathcal A_k\mid  \gcd(\f,N)=1,\f \leq X\}
 {\rm{\ \ and\ \ }}\\
 \mathcal B_{k,N}(X) &:=\{\chi \in \mathcal B_k\mid \gcd(\f,N)=1,\f \leq X\}.
 \end{align*}
For positive integers $\ff$, let $a_k(\ff)=\#( \{\chi \in \mathcal A_k\mid \f = \ff\} )$ and 
$b_k(\ff)=\#(\mathcal B_k(\ff)):=\#( \{\chi \in \mathcal B_k\mid \f =\ff\} ).$
Then by considering the series
\[
F_k(s,\mathcal A_k) =\sum_{\substack{\ff \leq X\\ \gcd(\ff,N)=1}}  \frac{a_k(\ff)}{\ff^s} 
 {\rm{\ \ and\ \ }}
F_{k}(s,\mathcal B_k) =\sum_{\substack{\ff \leq X\\ \gcd(\ff,N)=1}}\frac{b_k(\ff)}{\ff^s} , 
\]
we obtain for some positive constant $c_k =c_k(N)> 0$ the asymptotic 
estimate (see Corollory 1 of Theorem 8.8~\cite{Nark}), 
\begin{equation}\label{Narksize}
\#(\mathcal B_{k,N}(X))= \sum_{\substack {\ff \leq X\\ \gcd(\ff,N)=1}} b_k(\ff) \sim c_k X \log^{\sigma_0(k)-2}(X) .
 \end{equation}
 Hence for characters of prime order $p$ we have
 \[
\#(\mathcal B_{p,N}(X))= \sum_{\substack {\ff \leq X\\ \gcd(\ff,N)=1}} b_p(\ff) \sim c_p X 
 \]
 where $c_p =c_p(N)> 0.$   
 For characters of order $6$ we have
  \[
\#(\mathcal B_{6,N}(X))= \sum_{\substack {\ff \leq X\\ \gcd(\ff,N)=1}} b_6(\ff) \sim c_6 X \log^2(X).
 \]

\begin{remark}  
 This calculation has been done by many authors numerous times in differing cases, and an 
 account can be found in  Narkiewicz~\cite{Nark}. 
 A more general version of this result appears in~\cite{M-R} and the authors essentially attribute this result to Kubota.  
\end{remark}

\subsection{\bf{Algebraic central values}}\label{Algebraic central values}
\smallskip

Note that for $z\in \C,$ $\overline z$ denotes the complex conjugate of $z.$ For a Dirichlet character $\chi$ and $a\in\Z,$, $\overline\chi(a)=\overline{\chi(a)}=\chi^{-1}(a).$
For this section, we follow~\cite{M-T-T} and~\cite{M-R}. Let $E/\Q$ be an elliptic curve  of conductor $N_E,$ and for a primitive Dirichlet 
character $\chi,$   let
$L(E,s,\chi)$ denote the $L$-function $L(E/\Q,s)$ twisted by $\chi.$  
 Then there are real numbers $\Omega^{\pm}$ such that
 \[
 L(E,1,\chi) =\frac{ \Omega_{\chi}}{2\tau(\overline\chi)}\sum_{a\mod \f}\overline\chi(a)c(a,\f;E)
 \]
 where $\Omega_{\chi}$ equals $\Omega^{\pm}$ according as $\chi(-1)=\pm 1$ and the $c(a,\f;E)$ are 
 integers that do 
 not depend on $\chi$ (but only on $a,\f,$ and $E$) and $\tau(\chi)$ is the Gauss sum 
 corresponding to the character $\chi.$ (One can  ensure that the integers $c(a,\f;E)$ 
have no common divisor by choosing the numbers $\Omega^{\pm}$ appropriately). See~\cite{Wuthrich} and~\cite{An} for more precise statements.

 From~\cite{M-T-T}  the algebraic part of $L(E,1,\chi)$ is defined by
\begin{equation}\label{MTT}
L_E^{\rm{alg}}(\chi) := \frac{2\tau(\overline\chi)L(E,1,\chi)}{ \Omega_{\chi}} =\sum_{a\mod \f}\overline\chi(a)c(a,\f;E).
\end{equation}
Then $L_E^{\rm{alg}}(\chi)$ is an algebraic integer in the cyclotomic field $\Q(\chi)$ 
generated over $\Q$ by the values of $\chi$ and satisfies
 $\sigma( L_E^{\rm{alg}}(\chi))=L_E^{\rm{alg}}(\chi^{\sigma})$ for all $\sigma \in \Gal(\Q(\chi)/\Q).$

Noting that $\chi(-1)=\overline\chi(-1)$ we see that $\Omega_{\chi}=\Omega_{\overline{\chi}}$,  
so from the functional equation, we have
\begin{align*}
L_E^{\rm{alg}}(\chi) &=  \frac{2\tau(\overline\chi)L(E,1,\chi)}{ \Omega_{\chi}}\\
                        &=  \frac{2\tau(\overline\chi)w_E\chi(N_E)\tau(\chi)^2}{\f \Omega_{\chi}}L(E,1,\overline\chi)\\
                        &= \frac{w_E\chi(N_E)\tau(\chi)\tau(\overline\chi)}{\f}\times\frac{2\tau(\chi)L(E,1,\overline\chi)}{\Omega_{\overline\chi}}\\
                        &= \frac{w_E\chi(N_E)\chi(-1)\f}{\f}\times L_E^{\rm{alg}}(\overline\chi)\\
                        &= w_E\chi(-N_E)\overline{L_E^{\rm{alg}}(\chi)}
\end{align*}

If $z \in \C^*$ is any non-zero complex number satisfying $z=w_E\chi(-N_E)\overline{z},$
 then it follows that $L_E^{\rm{alg}}(\chi) /z=x \in \R$ is real.
Let $\zeta_{\chi}=w_E\chi(-N_E).$ Then $\zeta_{\chi}$ is a primitive $n^{\rm{th}}$ root of unity for 
some $n\geq 1$  dividing $2k,$ where $k$ is the order of $\chi.$

 Suppose now that $\chi$ is a complex Dirichlet character of order $k\geq3.$

If $\zeta_{\chi}\neq \pm 1,$ choose
\[ 
\lambda_{\chi}=\frac{1}{1+\overline\zeta_{\chi}}    {\rm \ \ so\  that\ \ } \lambda_{\chi}=\zeta_{\chi}\frac{1}{1+\zeta_{\chi}}=\zeta_{\chi} \overline \lambda_{\chi},
\]
and
 \[
 L_E^{\rm{alg}}(\chi) =\lambda_{\chi}\alpha_{\chi}
 \]
with $\alpha_{\chi} \in \mathcal O_{\chi}^+$ where $\mathcal O_{\chi}^+ $ is the ring of integers in $\Q(\chi)^+ ,$
 the maximal real subfield of $\Q(\chi).$

If $\zeta_{\chi} = -1,$ let $c$ be the least positive integer such that the order of $\chi(c)$ is equal to $k$, the order of $\chi.$
 Choose
 \[
\lambda_{\chi} =\frac{1}{\chi(c)-\overline{\chi(c)}}  {\rm \ \  so\  that\ \ }  \lambda_{\chi}=-\overline \lambda_{\chi} =\zeta_{\chi}\overline \lambda_{\chi}
 \]
  and
 \[
 L_E^{\rm{alg}}(\chi) = \lambda_{\chi}\alpha_{\chi}
 \]
with $\alpha_{\chi} \in \mathcal O_{\chi}^+.$\\ 

If $\zeta_{\chi} = 1,$ then choose $\lambda_{\chi}=1$ and so
\[
L_E^{\rm{alg}}(\chi) =\lambda_{\chi} \alpha_{\chi}
\]
with $\alpha_{\chi} \in \mathcal O_{\chi}^+.$  We have proved the following:

\begin{prop}\label{algval}
Let $E/\Q$ be an elliptic curve defined over $\Q$, and let $\chi$ be a primitive Dirichlet character of order $k \geq 3$ and 
conductor $\f.$   Let $\zeta_{\chi}=w_E\chi(-N_E).$  Then
\[
 L_E^{\rm{alg}}(\chi) =\lambda_{\chi}\alpha_{\chi}
\]
where
\[
\lambda_{\chi} = 
 \left\{
\begin{aligned}\null
& \frac{1}{1+\overline\zeta_{\chi}}  & {\rm\ \ \ \   if \ \ } \zeta_{\chi} \neq \pm 1 \\
 &\frac{1}{\chi(c)-\overline{\chi(c)}}  &{\rm \ \  if\ \  } \zeta_{\chi} =  -1,\\
&1  &{\rm \ \ if  \ \ \ \,} \zeta_{\chi} =  1
  \end{aligned}
\right.
 \]
and $\alpha_{\chi}  \in  \mathcal O_{\chi}^+$ are real cyclotomic integers.  Also we have  
\[
\sigma(\alpha_{\chi}) =\alpha_{\chi^{\sigma}}   {\rm \ \  and\ \  } \sigma(\lambda_{\chi}) =\lambda_{\chi^{\sigma}}
\]
for all $\sigma\in\Gal(\Q(\chi)/\Q).$
\end{prop}

\begin{remark}
The choice of $\lambda_{\chi}$ is {\it{not}} unique. For example, in the case that $E$ is an elliptic curve with $w_E = 1$ and 
$\chi$ is a character of {\it {odd}} order $k,$ then we could take $\lambda'_{\chi} = \zeta^{\frac{k+1}{2}}_{\chi} =\chi(N_E)^{\frac{k+1}{2}}.$
Then $\lambda'_{\chi} / \overline{\lambda'_{\chi}}=(\lambda'_{\chi})^2=\zeta_{\chi}^{k+1}=\zeta_{\chi}$ and 
\[
L_E^{\rm{alg}}(\chi) =\lambda'_{\chi} \beta_{\chi}
\]
where $\beta_{\chi}  \in  \mathcal O_{\chi}^+$ is a real cyclotomic integer and we again have  
\[
\sigma(\beta_{\chi}) =\beta_{\chi^{\sigma}}   {\rm \ \  and\ \  } \sigma(\lambda'_{\chi}) =\lambda'_{\chi^{\sigma}}
\]
for all $\sigma\in\Gal(\Q(\chi)/\Q).$ Note that in this case that
\[
\beta_{\chi}=\frac{\lambda_{\chi}}{\lambda'_{\chi}}\alpha_{\chi}
\] 
and that $\lambda_{\chi}/ \lambda'_{\chi}$ is a circular unit.
\end{remark}

In the Introduction (\S~\ref{Introduction}), we noted that  
 the integer $\vert \Nm_{\Q(\chi)/\Q}(L_E^{\rm{alg}}(\chi))\vert,$ if non-zero, is essentially the 
order of the ``$\chi$-part'' of the  Shafarevich-Tate group $\Sha(E/K_{\chi}).$ Since $\alpha_{\chi} \in \mathcal O_{\chi}^+$ is {\it{real}}, we see that
$A_{\chi}= \Nm_{\Q(\chi)^+/\Q}(\alpha_{\chi})\in \Z$
and we have 
\begin{align}
\Nm_{\Q(\chi)/\Q}(L_E^{\rm{alg}}(\chi)) &=  \Nm_{\Q(\chi)/\Q}(\lambda_{\chi})\Nm_{\Q(\chi)/\Q}(\alpha_{\chi}) \nonumber\\
&= \Nm_{\Q(\chi)/\Q}(\lambda_{\chi})\Nm_{\Q(\chi)^+/\Q}(L_E^{\rm{alg}}(\chi))^2 \label{squareness1}\\
&=\Nm_{\Q(\chi)/\Q}(\lambda_{\chi}) A_{\chi}^2. \nonumber
\end{align}
Since, for characters $\chi$ of fixed order $k$, the possible values of $ \Nm_{\Q(\chi)/\Q}(\lambda_{\chi})$ are finite (either equal to $\pm1$ or a divisor of $k=\ord(\chi)$ when $k$ is a prime power)
we see that this is consistent with the fact that, if finite,
 then 
 \begin{equation}\label{squareness2}
 |\Nm_{\Q(\chi)/\Q}(L_E^{\rm{alg}}(\chi))| = |\Sha_{\rm{an}}(E/K_{\chi})(\chi)|
 \end{equation}
  is essentially a square.

\section{Probabilities for non-zero values }\label{Probabilities}

In this section we shall  consider a na\"ive probabilistic model for  {\it{non-zero}} values of $L^{\rm{alg}}_E(\chi)$ as $\chi$ ranges over primitive Dirichlet characters of fixed order $k\geq 3$ and with $\f$ coprime to $N_E.$ From 
Proposition~\ref{algval} we see that
\[
 L_E^{\rm{alg}}(\chi) = \lambda_{\chi}\alpha_{\chi}
 \]
where the $\lambda_{\chi}$ are taken from a finite set and the $\alpha_{\chi}  \in  \mathcal O_{\chi}^+$ are real cyclotomic integers satisfying 
\[
\sigma(\alpha_{\chi}) =\alpha_{\chi^{\sigma}}   {\rm \ \  and\ \  } \sigma(\lambda_{\chi}) =\lambda_{\chi^{\sigma}}
\]
for all $\sigma\in\Gal(\Q(\chi)/\Q).$

We are interested in the distribution of the norms $A_{\chi}= \Nm_{\Q(\chi)^+/\Q}(\alpha_{\chi})\in \Z$   as $\chi$ 
varies over  families of primitive Dirichlet characters. Recall that the generalized Lindel\"of hypothesis implies that 
 $\vert \gamma(\alpha_{\chi})\vert= O(\f^{\frac{1}{2}+\epsilon})$ for all $\gamma \in \Gal(\Q(\chi)^+/\Q)$ and therefore that
\[ 
0 \leq \vert A_{\chi}\vert= \vert\Nm_{\Q(\chi)^+/\Q}(\alpha_{\chi})\vert = O\bigg(\f^{\frac{\phi(k)}{4}+\epsilon}\bigg). 
\]

 For a totally real field $F/\Q,$ of degree $n$, recall the usual map $\psi:F \rightarrow \R^n$ sending
  $\alpha \mapsto \psi(\alpha)=(\gamma_1(\alpha), \gamma_2(\alpha), \ldots,\gamma_n(\alpha))\in \R^n.$ Here
 $\gamma_1, \gamma_2, \ldots,\gamma_n$ are the $n$ distinct embeddings of $F$ into $\R$. Then the image of the ring of
 integers $\psi(\mathcal O_F)\subset \R^n$
 is a sublattice of $\R^n$. We will be interested in the case when $F=\Q(\chi)^+, n=[F:\Q]=\phi(k)/2$, and $\mathcal O_F=\mathcal O_{\chi}^+.$
 
\begin{lem}\label{geombounds} Let  $n\geq 1$ be a positive integer and let $L$ and $M$  be real numbers
with $0 < L \leq M^n$.  Define subsets $\mathcal T \subseteq \mathcal R \subset \R^n$  by
\[
\mathcal R=\mathcal R(M)= \{(x_1,x_2,\ldots,x_n)\in \R^n \ \vert \  0\leq x_i \leq M, 1\leq i \leq n\}
\]
and 
\[
\mathcal T=\mathcal T(L,M)=\{(x_1,x_2,\ldots,x_n) \in \mathcal R\  \vert \    x_1x_2\cdots x_n\leq L \}.
\]
Then
\[
\frac{\mu(\mathcal T)}{\mu(\mathcal R)} =\frac{L}{M^n}P_{n-1}\bigg(-\log(\frac{L}{M^n})\bigg)
\]
 where $\mu$ is Lebesgue measure and $P_m(x)=\sum_{j=0}^mx^j/j!$ is the $m^{\rm{th}}$ Taylor polynomial of $e^{x}$ at $x=0.$
Then the order of growth of the ratio of their measures as $M\to\infty$ is
\[
\frac{\mu(\mathcal T)}{\mu(\mathcal R)}\sim \frac{Ln^{n-1}\log^{n-1}(M)}{(n-1)!M^n}.
\]
\end{lem}
\begin{proof} (This proof is due to A. Granville)

Clearly $\mu(\mathcal R) = M^n$, so we compute $\mu(\mathcal T)$. Re-scaling and letting $x_i'=x_i/M$ we see that
\[
\mu(\mathcal T) = M^n\int_{\substack{0\leq x_1',x_2',\ldots,x_n' \leq 1 \\ x_1'x_2'\cdots x_n' \leq \frac{L}{M^n}}}\mathrm{d}x_1'\mathrm{d}x_2'\cdots \mathrm{d}x_n' = M^nI(n)
\]
and so $\mu(\mathcal T)/\mu(\mathcal R)=I(n).$\\

Set $C=L/M^n$ and $x_j'=e^{-y_j},$ so that $\mathrm{d}x_j' =-e^{-y_j}\mathrm{d}y_j,$  and then $I(n)$ becomes
\begin{align*}
I(n)&=\int_{\substack {y_1,y_2,\ldots, y_n \geq 0  \\ y_1 + y_2+\cdots+y_n \geq \log (1/C)}} e^{-(y_1 + y_2+\cdots+y_n)}\mathrm{d}y_1\mathrm{d}y_2\cdots \mathrm{d}y_n\\
     &= \int_{x\geq \log(1/C)} e^{-x} \cdot\int_{0\leq y_1 + y_2+\cdots+y_{n-1}\leq x} \mathrm{d}y_1\mathrm{d}y_2\mathrm{d}y_{n-1}\mathrm{d}x.
\end{align*}

Let
\[ 
J_n(x):= \int_{0\leq y_1 + y_2+\cdots+y_n \leq x} \mathrm{d}y_1\mathrm{d}y_2\cdots \mathrm{d}y_n.
\]
Then $J_1(x)=x$ and
\begin{align*}
J_n(x) &=\int_{y=0}^x J_{n-1}(x-y) \mathrm{d}y \\
           &= \int_{t=0}^x J_{n-1}(t)\mathrm{d}t\\
           &= x^n/n!  {\rm{\ \ by\ induction.}} 
\end{align*}
Hence we have 
\[
I(n)=\int_{x\geq \log(1/C)}\frac{e^{-x}x^{n-1}}{(n-1)!}\mathrm{d}x.
\]
Integrating by parts we see
\[
\int_{x\geq A}\frac{e^{-x}x^n}{n!}\mathrm{d}x = \bigg [\frac{-e^{-x}x^n}{n!}\bigg]_A^\infty + \int_{x\geq A}\frac{e^{-x}x^{n-1}}{(n-1)!}\mathrm{d}x,
\]
so by induction we have 
\[
I(n):= \int_{\substack{0\leq x_1,x_2.\ldots,x_n \leq 1 \\ x_1x_2\cdots x_n \leq C}}\mathrm{d}x_1\mathrm{d}x_2\cdots \mathrm{d}x_n = C\sum_{m=0}^{n-1}\frac{(-\log(C))^m}{m!}.
\]
Recalling that $C=L/M^n$ yields the statement.
\end{proof}

  We will estimate in the number of characters $\chi \in \mathcal B_{k,N}(X)$ for which $A_{\chi} $   assumes  ``small'' non-zero values, {\it{i.e.\ }}
  $\#\{\chi \in \mathcal B_{k,N}(X)) : 0< \vert A_{\chi}\vert \leq L_{\chi}\}$ for various choices of $L_{\chi}$ depending only on $\f$ 
  (typically we will take $L_{\chi}= \f^c$ for $0\leq c <\phi(k)/4).$

  We use our Statistical Hypothesis and Lemma~\ref{geombounds} with $L_{\chi}=\f^c, \ 0\leq c <\phi(k)/4$  and $M= \f^{\frac{1}{2}}$ to obtain  ``heuristic probabilities'' for the ``event'' $0<\vert A_{\chi}\vert \leq \f^c$ and denote them by
 $ \PP(0<\vert A_{\chi}\vert \leq \f^c).$

 We note that for all $\gamma\in\Gal(\Q(\chi)^+/\Q)$,
\[
A_{\chi} = 0 \iff \gamma(\alpha_{\chi}) = 0 \iff  \alpha_{\chi^{\gamma}} = 0
 \]
 and that such characters $\chi$ are treated (via different probability models) in~\cite{M-R},~\cite{DFK1}, and~\cite{DFK2},  and that they contribute $0$ 
 to the heuristic probability calculations below.

  For fixed $k$ and fixed $N,$ we are interested in the convergence or divergence of the sum
\begin{equation}\label{finitesum}
\sum_{\chi \in \mathcal B_{k,N}(X)} \PP(0<\vert A_{\chi} \vert \leq \f^c)
\end{equation}
 as $X \to\infty.$
 We interpret the convergence of the sum~\eqref{finitesum} as suggesting that the ``events'' $0<\vert A_{\chi}\vert \leq \f^c$ occur for only a finite number of characters $\chi$ under consideration. On the other hand the divergence of the sum~\eqref{finitesum} would indicate that the events 
 $0<\vert A_{\chi}\vert \leq \f^c$ occur infinitely often and the rate of divergence would inform the frequency of occurence of these events.
 
 For characters $\chi$ of order $k$ and conductor $\f$,  
  the  generalized Lindel\"of hypothesis implies that the image $\psi(\alpha_{\chi})$  lies in  $ \mathcal R'(M)$
  where 
 \[ 
  \mathcal R'(M)= \{ (x_1,x_2,\ldots,x_n)\in \R^n \ :\  \vert x_i\vert  \leq M, 1\leq i \leq n\}
  \]
 with $M\asymp  \f^{\frac{1}{2} + \epsilon}$ and $n=\phi(k)/2 = [\Q(\chi)^+:\Q]$. Taking into account the possible signs of the $x_i$ we have
 $\mu(\mathcal R'(M))=2^n\mu(\mathcal R(M)).$
Similarly, $\mu(\mathcal T'(L,M))=2^n\mu(\mathcal T'(L,M))$ where 
 \[
\mathcal T'(L, M)= \{ (x_1,x_2,\ldots,x_n)\in \mathcal R'(M) \  :\  \vert  \prod_i x_i\vert  \leq L  \}.
\]

Assume that the number of lattice points in a region is proportional to the relative volume of the region, and that the 
coordinates are  independent identically distributed random variables. Then
for fixed $k$, and with $ n=\phi(k)/2,$ we have, by Lemma~\ref{geombounds}
\begin{equation}\label{probability}
 \PP(0<\vert A_{\chi} \vert \leq \f^c)=\frac{\mu(\mathcal T'(\f^c, M))}{\mu( \mathcal R'(M))} \asymp 
\frac{\f^c\log^{\frac{\phi(k)}{2}-1}(\f^{1/2})}{\f^{\frac{\phi(k)}{4}+\epsilon}}.
 \end{equation}
  as $M \sim \f^{\frac{1}{2} + \epsilon}\to \infty.$

  Then for the sum (\ref{finitesum}) we have 
   \begin{equation}
   \begin{split}
\sum_{\chi \in \mathcal B_{k,N}(X)} \PP( 0<\vert A_{\chi}\vert \leq \f^c) &= \sum_{\substack{\ff \le X\\ \gcd(\ff,N)=1}}\sum_{\substack{\chi \in \mathcal B_k \\ \f=\ff}} \PP(0<\vert A_{\chi}\vert \leq \f^c)\\
&=\sum_{\substack{\ff \le X\\ \gcd(\ff,N)=1}}b_k(\ff) \PP(0<\vert A_{\chi}\vert \leq \f^c)\\
&\asymp \sum_{\substack{\ff \le X\\ \gcd(\ff,N)=1}} \frac{\f^cb_k(\ff)\log^{\frac{\phi(k)}{2}-1}(\ff^{1/2})}{\ff^{\frac{\phi(k)}{4}+\epsilon}}.
\end{split}
\end{equation}

In the following, since the values of $\phi(k)/4$ are discrete (half integers for $k \geq 3$) and $\epsilon >0$ can be taken arbitrarily small, the convergence of Equation~\eqref{prob} below is determined by the value of $\phi(k)/4 - c.$ Then,
 \begin{equation}\label{prob}
   \begin{split}
\sum_{\chi \in \mathcal B_{k,N}(X)} \PP( 0<\vert A_{\chi}\vert \leq  \f^c) 
 &= \sum_{\substack{\ff \le X\\ \gcd(\ff,N)=1}}\sum_{\substack{\chi \in \mathcal B_k \\ \f=\ff}} \PP(0<\vert A_{\chi}\vert \leq \f^c)\\
&=\sum_{\substack{\ff \le X\\ \gcd(\ff,N)=1}}b_k(\ff) \PP(0<\vert A_{\chi}\vert \leq \f^c)\\
&\asymp \sum_{\substack{\ff \le X\\ \gcd(\ff,N)=1}} \frac{b_k(\ff)\log^{\frac{\phi(k)}{2}-1}(\ff^{1/2})}{\ff^{\frac{\phi(k)}{4} -c}}.
\end{split}
\end{equation}

 Then by partial summation, and using Equation~\eqref{Narksize}  we have,  
\begin{equation}\label{growth}
\sum_{\chi \in \mathcal B_{k,N}(X)} \PP(0< \vert A_{\chi}\vert \leq \f^c) \asymp \frac{X\log^B(X)}{X^{(\phi(k)/4)-c}} + \int_1^X  \frac{u\log^Bu}{u^{1+(\phi(k)/4)-c}}\mathrm{d}u
 \end{equation}
  as $X\to\infty,$ where   $B=\sigma_0(k)+\phi(k)/2-3.$  
 
\section{Some Consequences}\label{Consequences}

 \vspace*{-0.15cm}
 
 Fix an elliptic curve $E/\Q$ defined over the rational field $\Q$ with conductor $N_E.$ In the computations of 
\S~7 and in our predictions
 we consider only those characters $\chi$ of order $k$ with conductors $\gcd(\f, N_E)=1.$

 \vspace*{-0.25cm}

 \subsection{\bf{For the case that $c = 0$}} \label{c=0}
 
 \smallskip

 If we take $c = 0$, then $L=L_{\chi}=\f^c$ is assumed to be a fixed bounded constant (independent of $\chi$).
 As $X \to \infty,$ the sum~\eqref{growth} converges for $\phi(k)\geq 6$ and diverges for $\phi(k)=2$ or $4$.


This suggests that 
\[
 \#\{\chi \in \mathcal B_k \mid  \gcd(\f,N)=1, \vert A_{\chi}\vert \leq L\} 
\]
 is  infinite for $\phi(k)=2$ or $4,$ and
 finite for
$\phi(k)\geq 6$.

From Equation~\eqref{growth} and the discussion in \S~3, letting  $X\to\infty,$ we predict,
\begin{equation} \label{conj_zero_c}
\begin{aligned}
           \#\{\chi \in \mathcal B _{k,N}(X) \mid 0< \vert A_{\chi} \vert \leq L\}    &\asymp X^{1/2}\log^B(X) &{\text{if \ \   }} \phi(k)=2\\
         & \asymp  \log^{B+1}(X) &{\text{if \ \   }} \phi(k)=4.\\
         & {\phantom{\asymp}}  {\text{\ \   is bounded\ \ }} &{\text{if \ \  }} \phi(k)\geq 6
\end{aligned}
\end{equation}
Noting that  $\Nm_{\Q(\chi)/\Q}(L_E^{\rm{alg}}(\chi)) =\Nm_{\Q(\chi)/\Q}(\lambda_{\chi}) A_{\chi}^2,$
this gives the case that $c=0$ in Prediction~\ref{A}.

For $k=3,$ we have $B=\sigma_0(3)+\phi(3)/2 -3= 0,$ so the predicted growth rate for ``small'' non-zero algebraic central 
values of cubic twists $A_{\chi}$
is of the order $\asymp X^{1/2}.$

For $k=5,$ we have $B=\sigma_0(5)+\phi(5)/2 -3= 1,$ so the predicted growth rate for ``small'' non-zero algebraic central
values of quintic twists $A_{\chi}$ is of the order $\asymp \log^2(X).$

These predicted growth rates seem to be supported by the numerical computations of \S~7 below.

  Let $L > 0$ be fixed. To explain Prediction~\ref{B}, we note that  
$$ \#\{\chi \in \mathcal B _{k,N} \mid 0< \vert A_{\chi} \vert \leq L\}  $$ 
is predicted to be  finite for those $k$ such that $\phi(k) \geq 6.$
Note that since $\mathcal B_k(\ff)$ consists of all   characters of conductor $\ff$ with order $k$, then $\bigcup_k \mathcal B_k(\ff)$ is the
set of all characters of conductor $\ff$ and therefore is a subset of 
the set of all characters mod $\ff$. Hence 
\[
\sum_k b_k(\ff)= \vert\bigcup_k \mathcal B_k(\ff)\vert \leq \vert\widehat{(\Z/\ff\Z)^{\times}} \vert=\phi(\ff)< \vert\ff\vert.
 \]

Since for a character $\chi$ of order $k$
\[
 \PP(0<\vert A_{\chi} \vert \leq L) \asymp 
\frac{L\log^{\frac{\phi(k)}{2}-1}(\f^{1/2})}{\f^{\frac{\phi(k)}{4}+\epsilon}},
 \]
the series 
\[
\sum_{\substack {\ff \leq X\\ \gcd(\ff,N)=1}}
\sum_{\f=\ff} \sum_{k\geq k_0}\PP(0< \vert A_{\chi}\vert \leq L) \ll \sum_{\ff<X}\f\frac{L\log^{\frac{\phi(k_0)}{2}-1}(\f^{1/2})}{\f^{\frac{\phi(k_0)}{4}+\epsilon}}
\]
converges absolutely for some $ k_0\gg0$, so its value equals the value of the re-arranged series 
\[
\sum_{k\geq k_0}\sum_{\chi \in \mathcal B_{k,N}}\PP(0< \vert A_{\chi}\vert \leq L).
\]
Hence the series
\[
\sum_{\ff \leq X}\sum_{k\geq k_0} b_k(\ff)\PP(0< \vert A_{\chi}\vert \leq L)  \leq
\sum_{\ff\leq X}\frac{\phi(\ff)}{\ff^{\phi(k_0)/4}}\leq\sum_{\ff\leq X}\frac{1}{\ff^{\phi(k_0)/4 -1}}
\] 
converges absolutely as $X\to \infty$ for $ \phi(k_0)/4 >2,$ {\it{\i.e.}\/} for $\phi(k_0)\geq 10.$
Since we predict that there are only finitely many characters $\chi$ of order $k$ with $\phi(k)=6$ or $8$ for which 
$  0< \vert A_{\chi} \vert \leq L,$  we would then have that
\[
  \#\{\chi \in \mathcal \bigcup_{\phi(k)>4} \mathcal B_k(X)   \mid 0< \vert A_{\chi} \vert \leq L\} 
\]  
is finite. 
Finally, a similar argument (see~\cite{M-R} and~\cite{DFK2}) would imply that  
\[
  \#\{ \chi \in \mathcal \bigcup_{\phi(k)>4} \mathcal B_k(X)  \mid     A_{\chi} =0\} 
\]  
is (conjecturally) finite if $\phi(k) \geq 6$ so Prediction~\ref{B} follows.

 \subsection{\bf{For the case that $0 < c < \phi(k)/4$}}\label{c>0}
 
 \smallskip

  As $X \to \infty,$ the sum~\eqref{growth} converges for $\phi(k)>4(1+ c)$ and diverges for $\phi(k)\leq 4(1+c ).$ Our discussion in \S~3 would then suggest that
\[
 \#\{\chi \in \mathcal B_k \mid  \gcd(\f,N)=1, \vert A_{\chi}\vert \leq \f^c\} 
\]
 is finite for
$\phi(k)> 4(1+ c) ,$ and 
asymptotically as $X\to\infty,$ we have,
\begin{equation} \label{conj_pos_c}
\begin{aligned}
           \#\{\chi &\in \mathcal B _{k.N}(X) \mid 0< \vert  A_{\chi} \vert \leq  \f^c \} &\phantom{6} &\phantom{6}\\  
             &\asymp X^{c-(\phi(k)/4 -1)}\log^B(X) &{\text{ if \ \ }} & \max{\{0, \frac{\phi(k)}{4} -1\}}< c<\frac{\phi(k)}{4}  \\
         &\asymp \log^{B+1}(X)  &{\text{ if \ \ }} & c =\frac{\phi(k)}{4}-1 > 0 \\
         & {\phantom{\asymp}}  {\text{is bounded\ \ }} &{\text{ if \ \ }} & 0 < c<\frac{\phi(k)}{4}-1.
\end{aligned}
\end{equation}
Note that if for some character $\chi$ we have $\vert\gamma(\alpha_{\chi})\vert \le \f^{\frac{1}{2}-\delta}$ for some
 $\gamma \in \Gal( \Q(\chi)^+/\Q),$ then 
$\vert A_{\chi} \vert \leq  \f^{\frac{\phi(k)}{4}-\delta}$ which, by Equation~\eqref{conj_pos_c}, can only happen for at most 
$O(X^{1-\delta}\log^B(X))$ characters of order $k$, {\it{i.e.\/}} $0\%$ of all characters of order $k$. This is the content of Prediction~\ref{C}.

 \section{Brauer-Siegel Limits}\label{Brauer-Siegel Limits}

Taking $L\asymp\f^c$,  and considering $\chi  \in \mathcal B _{k,N}(X),$     then since $L \leq M^n = X^{\phi(k)/4}$, we have $0 \leq c \leq \phi(k)/4.$
Then from Equation~\eqref{prob} above we find that the model predicts that
 \[
 \#\{\chi \in \mathcal B_k \mid  \gcd(\f,N)=1, 0< \vert A_{\chi}\vert \leq \ff^c\} 
\]
 is finite for $0 <c < \phi(k)/4 - 1$ (so $\phi(k)>4$) and is infinite for $ \phi(k)/4 - 1 \leq c  \leq \phi(k)/4.$
The model then implies that
 \[ 
  \#\{\chi \in \mathcal B _{k,N}(X) :
 0< \vert A_{\chi} \vert \leq \ff^c\} 
  \]
 grows as $\log^{B+1} X$ if  $ c = \phi(k)/4 - 1 $ (with $\phi(k)\geq 4$) and as   $X^{c-(\phi(k)/4 -1)}\log^B X$ if  $ \phi(k)/4>c  > \phi(k)/4 - 1  $ as $X \to\infty.$
  
  For any character $\chi \in  \mathcal B _{k,N}(X), $ let  $K=K_{\chi}/\Q$ be the associated cyclic extension of degree $k.$
  Then $E/K$ is  an elliptic curve over $K$ whose $L$-function satisfies
  \begin{equation}\label{Lfact2}
  L(E/K,s)=\prod_{i=0}^{k-1}L(E,s,\chi^i).
  \end{equation}
  
  Recall, we are always taking the primitive character giving $\chi^i$ so for example $L(E,s,\chi^0)=L(E/\Q,s).$
   This allows us to express the leading term in the Taylor expansion at $s=1$ of $L(E/K,s)$ in terms of the corresponding
    leading terms of  the twists $L(E,s,\chi^i).$
   
   We consider the ``$\chi$-component'' of Equation~\eqref{Lfact2} 
   \begin{equation}\label{Lfact chi}
 L(E/K,s)(\chi)=  \prod_{\substack{1 \leq i <k \\(i,k)=1}}L(E,s,\chi^i).
 \end{equation}
 Then in view of the discussion in \S~\ref{Introduction}, taking the algebraic parts of Equation~\eqref{Lfact chi} we have
 \begin{equation}\label{Lfact chi alg}
 \vert\Sha_{\text{an}}(E/K)(\chi)\vert =\vert \Nm_{\Q(\chi)/\Q}(L_E^{\rm{alg}}(\chi))\vert = \prod_{\substack{1 \leq i <k \\(i,k)=1}}L_E^{\rm{alg}}(\chi^i).
 \end{equation}
 
 The model suggests that for  $ \phi(k)/4 - 1 \leq c  \leq \phi(k)/4,$ there is an infinite family 
 \[
 \mathcal F_{k,N}(\f^c)  = \{\chi \in \mathcal B _{k,N}(X) : 0< \vert A_{\chi} \vert \leq \f^c\}
 \]
 of characters 
 $\chi$ (and therefore fields $K_{\chi}$) such that $0< \vert A_{\chi} \vert \leq \f^c$ and hence that $ L(E/K,1)(\chi)\neq 0.$

  %

Then the
 Birch \& Swinnerton-Dyer Conjecture implies that the $\chi$-component $E(K_{\chi})^{\chi}$
 of the Mordell-Weil group  $E(K_{\chi})$ is finite, and 
 so the corresponding factor of the regulator  $R(E(K_{\chi})(\chi))=1.$ In this case we have
 \[
   \prod_{\substack{1 \leq i <k \\(i,k)=1}}L^{\rm{alg}}_{E}(\chi^i) =d  A_{\chi}^2
 \]
  for some $d$ bounded only in terms of $k$ and that $A_{\chi}^2  $ 
  is essentially 
 (up to constants) the order of  $\vert\Sha_{\text{an}}(E/K)(\chi)\vert$.
  Hence for $\chi \in \mathcal F_{k,N}(\f^c)$ we have
 \begin{equation}\label{Numerator estimate}
 \begin{split}
 \log(  \vert\Sha_{\text{an}}(E/K)(\chi)\vert\cdot R(E_{\chi})(\chi))) & \asymp \log(\prod_{\substack{1 \leq i <k \\(i,k)=1}}L^{\rm{alg}}_{E}(\chi^i))\\
                                                 &\asymp \log( A_{\chi}^2) \\
                                                 &\leq \log(\f^{2c} )\\
                                                 &=2c\log(\f).
\end{split}
 \end{equation}

 Fix a character $\chi \in  \mathcal B _{k,N}(X)$ and let $K=K_{\chi}/\Q$ be the corresponding cyclic $\Z/k\Z$-extension 
of $\Q.$ Viewing $E$ as an elliptic curve over $K$ with $L$-function $L(E/K,s)$ we have
\[
L(E/K,s)= \prod _{0\leq i <k} L(E,s,\chi^i).
\]
Suppose for simplicity that $k$ is an{\it{ odd prime}}, and that $L(E, 1, \chi) \neq 0$. Then $\chi(-1)=1$ and
comparing leading terms at $s=1$ the Birch \& Swinnerton-Dyer conjecture predicts that
\[
\frac{\Omega_{E_K }R(E(K)) \vert \Sha(E(K)\vert \prod_{\p} c_{\p}}{\vert E(K)_{\rm{tors}}\vert^2 \sqrt{\vert \Delta(K/\Q)}\vert}
=\frac{\Omega_E^+ R(E(\Q)) \vert \Sha(E/\Q)\vert \prod_p c_p}{\vert E(\Q)_{\rm{tors}}\vert^2}     \prod_{1 \leq i<k} \frac{\Omega^+ L_E^{\rm{alg}}(\chi^i)}{2\tau(\overline{\chi^i})}
\]
where $\Omega_{E_K}=(\Omega^+)^k$ and the $c_{\p},c_p$ are the Tamagawa numbers for $E/K$ and $E/\Q$ respectively.
Since for $\chi(-1)=1$ we have 
\begin{equation}\label{Delta growth}
\sqrt{\vert \Delta(K/\Q) \vert}  =\vert\prod_{0\leq i <k}\tau(\chi^i)\vert =\f^{(k-1)/2}.
\end{equation}

For general $k$, for the character $\chi$, we define   
 \[
 \sqrt{\vert \Delta(\chi)\vert} : = \vert\prod_{\substack{1 \leq i <k \\(i,k)=1}}\tau(\chi^i)\vert =(\f)^{\phi(k)/2}
 \]
 so that $\log(\sqrt{\vert \Delta(\chi)\vert}) = \frac{\phi(k)}{2}\log(\f).$ 
 
 For $\phi(k)/4 - 1\le c < \phi(k)/4$,  the families $\mathcal F_{k,N}(\f^c)$ are predicted to be infinite so we can take the  upper limit (as $X\to\infty$) 
 of the Brauer-Siegel quotients, using Equation~\eqref{Numerator estimate}:
 \[
 \limsup_{\substack{\chi \in \mathcal F_{k,N}(\f^c)\\ \f \leq X}} \frac{\log(\Sha_{\rm{an}}(E/K_{\chi}(\chi))\cdot R(E(K_{\chi})(\chi))}{\log(\sqrt{\Delta(\chi)})}\leq\frac{4c}{\phi(k)}.
 \]
 
 If we choose $ \phi(k)/4 - 1 \leq c' < c  \leq \phi(k)/4,$ then Equation~\eqref{conj_pos_c} would predict that the set difference  $\mathcal F_{k,N}(\f^c) \setminus \mathcal F_{k,n}(\f^{c'})$
 is an infinite set. So we may consider the upper and lower limits as $X\to \infty$ for $\chi \in \mathcal F_{k,N}(\f^c) \setminus \mathcal F_{k,n}(\f^{c'})$ and use Equations~\eqref{Numerator estimate} and ~\eqref{Delta growth} to obtain:
 \begin{align*}
  \frac{4c'}{\phi(k)} &\leq
 \liminf_{\substack{\chi \in \mathcal F_{k,N}(\f^{c'})\\ \f \leq X}} \frac{\log( \Sha_{\rm{an}}(E/K_{\chi}(\chi))) R(E(K_{\chi})(\chi))}{\log(\sqrt{\Delta(\chi)})}\\
& \leq
 \limsup_{\substack{\chi \in \mathcal F_{k,N}(\f^c)\\ \f \leq X}} \frac{\log( \Sha_{\rm{an}}(E/K_{\chi}(\chi))) R(E(K_{\chi})(\chi))}{\log(\sqrt{\Delta(\chi)})}\leq\frac{4c}{\phi(k)}.
 \end{align*}
 
 This is the statement in Prediction~\ref{Db}. When $\phi(k)=2$ or $4$, the families of \S~\ref{Consequences} are predicted to be infinite
by Prediction~\ref{A} and then
 Prediction~\ref{Da} follows by taking $c=0.$

 Note that for characters $\chi$ of prime order $k$, the fields $K_{\chi}/\Q$ are cyclic extensions of prime degree $k$, and
 $ \Sha_{\rm{an}}(E/K_{\chi}(\chi))$ is essentially (under generalized BSD) just the order of relative Shafarevich-Tate group $\Sha(E/K_{\chi}))/\Sha(E/\Q).$

\vspace{-5pt}

\section{Computational Results}\label{computations}

For computing those central $L$-values, we use the following well-known formula:
\begin{equation}\label{l-formula}
L(E, 1, \chi) = \sum_{n \ge 1}(\chi(n) + w_E c_{\chi}\overline{\chi}(n))\frac{a_n}{n}\exp\big(-\frac{2\pi n}{\f \sqrt{N}}\big),
\end{equation}
where $a_n$ are the coefficients of $L(E, s)$, $w_E$ is the sign of the functional equation of $L(E, s)$, $c_{\chi} = \chi(N) \tau^2(\chi)/\f$, and $N := N_E$. Using those values we compute $L_E^{\rm{alg}}(\chi)$ and then $A_{\chi}$ using Proposition~\ref{algval}. For computing $L_E^{\rm{alg}}(\chi)$, we compute the period of $E$, $\Omega_{\chi}$, and the integer values of $A_{\chi}$ by computing the integer coefficients of $\alpha_\chi \in \mathcal O_{\chi}^+$ in Proposition~\ref{algval} by SageMath~\cite{sage}. Then, we divide the values of $A_{\chi}$ for $\chi \in \mathcal B _{k,N}(X)$ by their greatest common divisor $\text{g}_{k, E}$ which theoretically depends only on $k$ and $E$.

For our numerical computations, $a_n$ for each elliptic curve are computed by using PARI/GP~\cite{PARI2} up to $n = 15\times 10^{7}$ to maintain at least 4 decimal place accuracy in computing the values of $\alpha_\chi$, the values of $L(E, 1, \chi)$ are computed by using the second author's codes written with FLINT~\cite{FLINT} and CUDA~\cite{NVF}, and $L_E^{\rm{alg}}(\chi)$, $\alpha_\chi$ and $A_\chi$ are computed by using SageMath~\cite{sage}. 

The computations are conducted on the second author's personal Linux system with an NVIDIA GTX 1080 Ti GPU. Computing $L(E, 1, \chi)$ given by Equation~\eqref{l-formula} is the primary task, significantly accelerated by the GPU, achieving a throughput approximately about two thousand times faster than a general CPU. More precisely, leveraging the GPU's three thousand cores, we assign a task of computing the exponential sum in Equation~\eqref{l-formula} for each $\chi \in \mathcal B _{k,N}(X)$ into each core in GPU and execute several thousands of tasks in parallel.

For the software codes, refer to the second author's website~\cite{Nam}. The dataset of the values for elliptic curves of conductor up to 100 and $k = $3, 5, 6, 7, 13 is publicly accessible also in~\cite{Nam}.  Additionally, further numerical support for these predictions for several more elliptic curves is available in~\cite{Ki-Na}.

In this section, we present the computational results supporting Predictions~\ref{Main Prediction} and~\ref{C}. Recall that Prediction~\ref{B} is yielded by a finite number of applications of Prediction~\ref{Main Prediction} (see Remark~\ref{RemarkForB}) and Predictions~\ref{Da}, and~\ref{Db} are other consequences of Prediction~\ref{Main Prediction} (see $\S~\ref{Brauer-Siegel}$). We also present computational results for some other statistics of some small norms $A_{\chi}= \Nm_{\Q(\chi)^+/\Q}(\alpha_{\chi})\in \Z$ associated with the central values for $L(E, s, \chi)$ for 
$$E: \text{11a1}, \text{14a1}, \text{15a1},\text{17a1}, \text{19a1}, \text{37b1}$$ 
in the Cremona's labels and $\mathcal B_{k,N}(X)$ for $X = 3\times 10^6$ for $k = 3, 5, 7, 13$ and $X = 10^6$ for $k = 6$ where $N = N_E$ is the conductor of $E$. 

\begin{remark}
In~\cite{DFK1}, the authors successfully computed the implied constants of their conjectural asymptotics concerning the frequency of vanishings for the cubic twists using the moment conjecture of Keating and Snaith. However, we have no method to determine the implied constants for the lower and upper bounds in Prediction~\ref{Main Prediction} with our model at present.
\end{remark}

Abusing notation, we denote those values by $A_{\chi}$. Moreover, for the number of vanishings, define
$$
\mathcal V_{k, N}(X) := \{\chi \in \mathcal B_{k, N}(X) \mid A_\chi = 0\}.
$$
Table~\ref{table:combined} presents the values of $\#(\mathcal B_{k, N}(X)), \#(\mathcal V_{k, N}(X)), \text{ and } \text{g}_{k, E}$ for the choices of $E$ (with Cremona labels), $k$, $X$ used in our numerical experiment.

\begin{table}[h!]
\centering
\scalebox{0.9}{
\begin{tabular}{ccrrcr}
\hline
\makecell{$E$} & \makecell{$k$} & \makecell{$X$} & \makecell{$\#(\mathcal B_{k, N}(X))$} & \makecell{$\text{g}_{k, E}$} & \makecell{$\#(\mathcal V_{k, N}(X))$} \\
\hline
\hline
\multirow{5}*{$\text{11a1}$} & $3$ & $3 \times 10^6$ & $951116$ & $10$ & $2842$ \\
 & $5$ & $3 \times 10^6$ & $577692$ & $10^2$ & $68$ \\
 & $6$ & $10^6$ & $7103800$ & $2$ & $65846$ \\
 & $7$ & $3 \times 10^6$ & $592938$ & $10^3$ & $12$ \\
 & $13$ & $3 \times 10^6$ & $514620$ & $10^6$ & $0$ \\
\hline
\multirow{5}*{$\text{14a1}$} & $3$ & $3 \times 10^6$ & $739810$ & $6$ & $10946$ \\
 & $5$ & $3 \times 10^6$ & $787584$ & $6^2$ & $44$ \\
 & $6$ & $10^6$ & $3207738$ & $2$ & $52492$ \\
 & $7$ & $3 \times 10^6$ & $528852$ & $6^3$ & $0$ \\
 & $13$ & $3 \times 10^6$ & $514620$ & $6^6$ & $0$ \\
\hline
\multirow{5}*{$\text{15a1}$} & $3$ & $3 \times 10^6$ & $778150$ & $10$ & $5134$ \\
 & $5$ & $3 \times 10^6$ & $678796$ & $10^2$ & $48$ \\
 & $6$ & $10^6$ & $3791698$ & $2$ & $57214$ \\
 & $7$ & $3 \times 10^6$ & $592938$ & $10^3$ & $0$ \\
 & $13$ & $3 \times 10^6$ & $514620$ & $10^6$ & $0$ \\
\hline
\multirow{5}*{$\text{17a1}$} & $3$ & $3 \times 10^6$ & $951116$ & $6$ & $4240$ \\
 & $5$ & $3 \times 10^6$ & $787584$ & $6^2$ & $20$ \\
 & $6$ & $10^6$ & $7293076$ & $2$ & $110044$ \\
 & $7$ & $3 \times 10^6$ & $592938$ & $6^3$ & $0$ \\
 & $13$ & $3 \times 10^6$ & $514620$ & $6^6$ & $0$ \\
\hline
\multirow{5}*{$\text{19a1}$} & $3$ & $3 \times 10^6$ & $860578$ & $6$ & $8098$ \\
 & $5$ & $3 \times 10^6$ & $787584$ & $6^2$ & $4$ \\
 & $6$ & $10^6$ & $6308258$ & $2$ & $79062$ \\
 & $7$ & $3 \times 10^6$ & $592938$ & $6^3$ & $0$ \\
 & $13$ & $3 \times 10^6$ & $514620$ & $6^6$ & $0$ \\
\hline
\multirow{5}*{$\text{37b1}$} & $3$ & $3 \times 10^6$ & $902370$ & $6$ & $15062$ \\
 & $5$ & $3 \times 10^6$ & $787584$ & $6^2$ & $8$ \\
 & $6$ & $10^6$ & $6941862$ & $2$ & $97290$ \\
 & $7$ & $3 \times 10^6$ & $592938$ & $6^3$ & $6$ \\
 & $13$ & $3 \times 10^6$ & $514620$ & $6^6$ & $0$ \\
\hline
\end{tabular}
}
\caption{$E, k, X, \#(\mathcal B_{k, N}(X)), \text{g}_{k, E}, \text{ and } \#(\mathcal V_{k, N}(X))$.}
\label{table:combined}
\end{table}


\subsection{\bf{Numerical Support for Prediction~\ref{Main Prediction}}}

\medskip


In this section, we present numerical evidence that supports Equations~\eqref{conj_zero_c} and~\eqref{conj_pos_c}, which correspond to the validation of Prediction~\ref{Main Prediction} for the cases $c = 0$ and $0 < c < \phi(k)/4$, respectively, in accordance with Equations~\eqref{squareness1} and~\eqref{squareness2}.

\subsubsection{\bf{Numerical Support for Equation~\eqref{conj_zero_c}}}

\medskip

Abusing notation, we use the same notation $n_{k,E}(X;L)$ as one defined with $\Nm_{\Q(\chi)/\Q}(L_E^{\rm{alg}}(\chi))$ in Prediction~\ref{Main Prediction}: i.e. $$n_{k,E}(X;L) := \#(\{\chi \in \mathcal B _{k,N}(X) \mid 0 <  |A_{\chi}| \le L\}).$$
Notice that using this notation does not affect the order of growth in Prediction~\ref{Main Prediction}. For $k = 3, 5, 6$, we compute the following ratios
\begin{equation}\label{ratio_A}
\frac{n_{k,E}(X;L)}{X^{1/2}\log^B(X)} \hspace{0.25em} \text{ if } k = 3 \text{ or } 6 \hspace{0.5em} \text{ and } \hspace{0.5em} \frac{n_{k,E}(X;L)}{\log^{B+1}(X)} \hspace{0.25em} \text{ if } k = 5,
\end{equation}
where $B = \sigma_0(k)+\phi(k)/2-3$.

Notice that for $k = 7, 13$  with $\phi(k) \ge 6$, $m_{k,E}(X;c)$ seems to be very small compared with the other choices of $k$ when $\phi(k) \ge 6$ ($L = 5$ for $k = 7, 13$ in Table~\ref{table:finite_c_zero} as examples) and to be of the order $\asymp \log^{B+1}(X)$ when $\phi(k) = 4$ ($L = 5$ for $k = 5$ in Table~\ref{table:finite_c_zero} as an example).

\begin{table}[h!]
\centering
\scalebox{0.9}{
\begin{tabular}{ccrc|ccrc}
\hline
\makecell{$E$} & \makecell{$k$} & \makecell{$n_{k,E}(X;L)$} & \makecell{$\text{ratio}$} & \makecell{$E$} & \makecell{$k$} & \makecell{$n_{k,E}(X;L)$} & \makecell{$\text{ratio}$}\\
\hline
\hline
\multirow{5}*{$\text{11a1}$} & $3$ & $19356$ & $0.02035$ & \multirow{5}*{$\text{17a1}$} & $3$ & $16296$ & $0.01713$ \\
 & $5$ & $1440$ & $0.00249$ &  & $5$ & $1620$ & $0.00206$ \\
 & $6$ & $32392$ & $0.00456$ & & $6$ & $60282$ & $0.00827$  \\
 & $7$ & $120$ & $0.00020$ & & $7$ & $108$ & $0.00018$ \\
 & $13$ & $0$ & $0.00000$ & & $13$ & $12$ & $0.00002$ \\
\hline
\multirow{5}*{$\text{14a1}$} & $3$ & $6524$ & $0.00882$ & \multirow{5}*{$\text{19a1}$} & $3$ & $6512$ & $0.00757$ \\
 & $5$ & $1360$ & $0.00173$ &  & $5$ & $760$ & $0.00097$ \\
 & $6$ & $8370$ & $0.00261$ & & $6$ & $30124$ & $0.00478$ \\
 & $7$ & $48$ & $0.00009$ &  & $7$ & $66$ & $0.00011$ \\
 & $13$ & $36$ & $0.00007$ &  & $13$ & $24$ & $0.00005$ \\
\hline
\multirow{5}*{$\text{15a1}$} & $3$ & $10560$ & $0.01357$ & \multirow{5}*{$\text{37b1}$} & $3$ & $10838$ & $0.01201$ \\
 & $5$ & $1068$ & $0.00157$ &  & $5$ & $1204$ & $0.00153$ \\
 & $6$ & $19070$ & $0.00503$ & & $6$ & $26864$ & $0.00387$ \\
 & $7$ & $42$ & $0.00007$  & & $7$ & $138$ & $0.00023$ \\
 & $13$ & $12$ & $0.00002$  & & $13$ & $48$ & $0.00009$ \\ 
\hline
\end{tabular}
}
\caption{$n_{k,E}(X;L)$ and $\text{ratio} := n_{k,E}(X;L)/\#(\mathcal B_{k, N}(X))$ for $L = 5$ and $k =$ 3, 5, 7, 13 with $X = 3 \times 10^6$ and $k = $ 6 with $X = 10^6$.}
\label{table:finite_c_zero}
\end{table}

For the other choices of $k$ (i.e. $\phi(k) = $ 2, 4) and, we choose 
$$L = 1, 2, 3 \text{ for } k = 3, 6 \quad \text{and} \quad L = 1, 4, 5 \text{ for } k = 5,$$ 
and compute their ratios with $X$ in Table~\ref{table:combined} and depict them in Figures~\ref{fig:A_acc_11_14},~\ref{fig:A_acc_15_17},~\ref{fig:A_acc_19_37}.

In these figures, each subfigure corresponding to $E$ and $k$ contains three graphs associated with different choices of $L$. Each graph illustrates the convergence of the ratios within a moderate range. Additionally, the ratio values exhibit a regularity and stability as $X$ increases. For each $E, k$ and a fixed $X$, if $L \le L'$, $0 <  |A_{\chi}| \le L \le L'$, hence the ratio of $L$ does not exceed that of $L'$. Consequently, the graph representing the ratios of $L$ is positioned lower than that of $L'$ in each subfigure. Note that the possible values of $\vert A_{\chi} \vert$ with $0 < \vert A_{\chi} \vert \le 5$ are $1, 2, 3, 4$, and $5$ for $k = 3$ and $6$, $1, 4$, and $5$ for $k = 5$, and $1$ for $k = 7$ and $13$.

\subsubsection{\bf{Numerical Support for Equation~\eqref{conj_pos_c}}}

\medskip

We define
\[
m_{k,E}(X;c) := \#(\{\chi \in \mathcal B _{k,N}(X) \mid 0 < |A_{\chi}| \le \f^c\})
\] 
and consider the following ratios to support the predictions in Equation~\eqref{conj_pos_c}
\begin{equation}\label{ratio_c}
\begin{aligned}
\frac{m_{k,E}(X;c)}{X^{c-(\phi(k)/4-1)}\log^B(X)} & \hspace{0.25em} \text{ if } \hspace{0.25em} \max\{0, \frac{\phi(k)}{4} -1 \} < c \le \frac{\phi(k)}{4},\\
\frac{m_{k,E}(X;c)}{\log^{B+1}(X)} & \hspace{0.25em} \text{ if } \hspace{0.25em} c = \frac{\phi(k)}{4} -1.
\end{aligned}
\end{equation}

Notice that for $k = 7, 13$  with $\phi(k) \ge 6$, $m_{k,E}(X;c)$ seems to be very small compared with the other choices of $k$ when $0 < c < \phi(k)/4 -1$ ($c = 1/4$ for $k = 7$ and $c = 1/4, 1/2, 5/4$ for $k = 13$ in Table~\ref{table:finite_c_pos} for $c = 1/4, 1/2, 5/4$ as examples) and to be of the order $\asymp \log^{B+1}(X)$ when $c = \phi(k)/4 -1$ ($c = 1/2$ for $k = 7$ in Table~\ref{table:finite_c_pos} as examples).


\begin{table}[h!]
\centering
\scalebox{0.9}{
\begin{tabular}{ccrrrrrr}
\hline
\multirow{2}{*}{\makecell{$E$}} & \multirow{2}{*}{\makecell{$k$}} & \multicolumn{3}{r}{\makecell{$m_{k,E}(X;c)$}} & \multicolumn{3}{r}{\makecell{$\text{ratio}$}} \\
 & & \makecell{$c = 1/4$} & \makecell{$c = 1/2$} & \makecell{$c = 5/4$}  & \makecell{$c = 1/4$} & \makecell{$c = 1/2$} & \makecell{$c = 5/4$} \\
\hline
\hline
\multirow{5}*{$\text{11a1}$} & $3$ & $107672$ & $822140$ & $948274$ & $0.11321$ & $0.86440$ & $0.99701$ \\
& $5$ & $3448$ & $282056$ & $577624$ & $0.00597$ & $0.07340$ & $0.99988$  \\
& $6$ & $233834$ & $3965168$ & $7037954$ & $0.03292$ & $0.55818$ & $0.99073$ \\
& $7$ & $336$ & $3108$ & $469266$ & $0.00057$ & $0.00524$ & $0.79143$  \\
& $13$ & $0$ & $36$ & $1656$ & $0.00000$ & $0.00007$ & $0.00322$  \\
\hline
\multirow{5}*{$\text{14a1}$} & $3$ & $69440$ & $666356$ & $728864$ & $0.09386$ & $0.90071$ & $0.98520$  \\
& $5$ & $3364$ & $47576$ & $787540$ & $0.00427$ & $0.06041$ & $0.99994$ \\
& $6$ & $98102$ & $1957574$ & $3155246$ & $0.03058$ & $0.61027$ & $0.98364$ \\
& $7$ & $246$ & $2088$ & $409428$ & $0.00047$ & $0.00395$ & $0.77418$  \\
& $13$ & $36$ & $36$ & $1044$ & $0.00007$ & $0.00007$ & $0.00203$ \\
\hline
\multirow{5}*{$\text{15a1}$} & $3$ & $67410$ & $664334$ & $773016$ & $0.08663$ & $0.85374$ & $0.99340$ \\
& $5$ & $2660$ & $39028$ & $678748$ & $0.00392$ & $0.05750$ & $0.99993$  \\
& $6$ & $156876$ & $2606364$ & $3734484$ & $0.04137$ & $0.68739$ & $0.98491$ \\
& $7$ & $150$ & $1770$ & $436668$ & $0.00025$ & $0.00299$ & $0.73645$  \\
& $13$ & $12$ & $12$ & $864$ & $0.00002$ & $0.00002$ & $0.00168$  \\
\hline
\multirow{5}*{$\text{17a1}$} & $3$ & $94728$ & $815242$ & $946876$ & $0.09960$ & $0.85714$ & $0.99554$ \\
& $5$ & $4348$ & $57408$ & $787564$ & $0.00552$ & $0.07289$ & $0.99998$  \\
& $6$ & $465904$ & $5642244$ & $7183032$ & $0.06388$ & $0.77364$ & $0.98491$ \\
& $7$ & $264$ & $2304$ & $461328$ & $0.00045$ & $0.00389$ & $0.77804$  \\
& $13$ & $12$ & $24$ & $1248$ & $0.00002$ & $0.00005$ & $0.00243$  \\
\hline
\multirow{5}*{$\text{19a1}$} & $3$ & $63702$ & $709306$ & $852480$ & $0.07402$ & $0.82422$ & $0.99059$ \\
& $5$ & $1992$ & $28828$ & $787576$ & $0.00253$ & $0.03660$ & $0.99999$  \\
& $6$ & $263396$ & $3892956$ & $6229196$ & $0.04175$ & $0.61712$ & $0.98747$ \\
& $7$ & $126$ & $1038$ & $350796$ & $0.00021$ & $0.00175$ & $0.59162$  \\
& $13$ & $24$ & $36$ & $432$ & $0.00005$ & $0.00007$ & $0.00084$  \\
\hline
\multirow{5}*{$\text{37b1}$} & $3$ & $105334$ & $857782$ & $887308$ & $0.11673$ & $0.95059$ & $0.98331$ \\
& $5$ & $3464$ & $47708$ & $787576$ & $0.00440$ & $0.06058$ & $0.99999$  \\
& $6$ & $271734$ & $4643896$ & $6844572$ & $0.03914$ & $0.66897$ & $0.98599$ \\
& $7$ & $252$ & $2310$ & $492528$ & $0.00043$ & $0.00390$ & $0.83066$  \\
& $13$ & $48$ & $60$ & $1056$ & $0.00009$ & $0.00012$ & $0.00205$  \\
\hline
\end{tabular}
}
\caption{$m_{k,E}(X;c)$ and $\text{ratio} := m_{k,E}(X;c)/\#(\mathcal B_{k, N}(X))$ for $c = 1/4, 1/2, 5/4$ and $k =$ 3, 5, 7, 13 with $X = 3 \times 10^6$ and $k = $ 6 with $X = 10^6$.}
\label{table:finite_c_pos}
\end{table} 

For $k$ and $c$ such that $0 < c < \phi(k)/4$ when $\phi(k)/4 < 1$ and $\phi(k)/4 -1 \le c < \phi(k)/4$ when $\phi(k)/4 \ge 1$, we choose 
$$c = 
\begin{cases}
 0.3, 0.4 &\text{ for } k = 3, 6,\\
 0.3, 0.5 & \text{ for } k = 5,\\
 1.2, 1.3 & \text{ for } k = 7,\\
 2.4, 2.5 & \text{ for } k = 13.
\end{cases}
$$
Then, we compute those ratios for $E$, $k$ and $X$ given in Table~\ref{table:combined} and depict them in Figures ~\ref{fig:c_11_14_acc_3_5_6},~\ref{fig:c_11_14_acc_7_13},~\ref{fig:c_15_17_acc_3_5_6},~\ref{fig:c_15_17_acc_7_13},~\ref{fig:c_19_37_acc_3_5_6},~\ref{fig:c_19_37_acc_7_13}. 

In these figures, each subfigure corresponding to $E$ and $k$ contains two graphs associated with different choices of $c$. Each graph illustrates the convergence of the ratios within a moderate range. Additionally, the ratio values exhibit a regularity and stability as $X$ increases. Moreover, observe the first estimate in Equation~\eqref{ratio_c} includes $c$ in its denominator. Consequently, the ratios in Equation~\eqref{ratio_c} do not consistently increase as $c$ increases (cf. Figures~\ref{fig:A_acc_11_14},~\ref{fig:A_acc_15_17},~\ref{fig:A_acc_19_37}), and it seems rather the opposite as illustrated in these figures.

\subsection{\bf{Numerical Support for Prediction~\ref{C}}}


In this section, we present numerical evidence supporting Equation~\eqref{upper_lower_bounds} within Prediction~\ref{C}. Subsequently, the second assertion in Prediction~\ref{C} follows from Prediction~\ref{Main Prediction}, in accordance with Equations~\eqref{squareness2} (see \S~\ref{c>0}). 

For the implied constants of the lower and upper bounds depending on each $E$ and $k$, we define
 $M_{k, E}(X) := \max_{\chi \in \mathcal B_{k, N}(X)}\{\lvert L_E^{\text{alg}}(\chi) \rvert / \sqrt{\f} \}$.
Abusing notation, we let $M := M_{k, E}(X)$. Then, we compute 
\begin{equation} \label{gen_lindeloff}
l_{k, E}(X; \epsilon) := \#\{\chi \in \mathcal B_{k, N}(X) \mid (M \hspace{0.2em} \f^{\epsilon})^{-1} \le \lvert L_E^{\text{alg}}(\chi) \rvert \le M \hspace{0.2em} \f^{\epsilon}\}
\end{equation}
for $\epsilon = 10^{-1}$ and $10^{-8}$. In Table~\ref{table:gen_lindeloff}, the values of $l_{k, E}(X; \epsilon)$ at $\epsilon = 10^{-1}$ and $10^{-8}$ are present for $E$, $k$, and $X$. At least $97.9 \%$ of $\chi \in \mathcal B_{k, N}(X)$ satisfy the inequality in Equation~\eqref{gen_lindeloff}, which seems to support Equation~\eqref{upper_lower_bounds}.

\begin{table}[h!]
\centering
\scalebox{0.9}{
\begin{tabular}{ccrrrrcr}
\hline
\multirow{2}{*}{\makecell{$E$}} & \multirow{2}{*}{\makecell{$k$}} & \multicolumn{2}{r}{\makecell{$l_{k, E}(X; \epsilon)$}} & \multicolumn{2}{r}{\makecell{$\text{ratio}$}} & \multirow{2}{*}{\makecell{$M_{k, E}^{-1}$}} & \multirow{2}{*}{\makecell{$M_{k, E}$}} \\
 & & \makecell{$\epsilon = 10^{-1}$} & \makecell{$\epsilon = 10^{-8}$} & \makecell{$\epsilon = 10^{-1}$} & \makecell{$\epsilon = 10^{-8}$} & & \\
\hline
\hline
\multirow{5}*{$\text{11a1}$} & $3$ & $948274$ & $945738$ & $0.99701$ & $0.99435$ & $0.01289$ & $77.59927$ \\
 & $5$ & $577084$ & $575118$ & $0.99895$ & $0.99554$ & $0.00977$ & $102.32258$ \\
 & $6$ & $7036444$ & $7009180$ & $0.99052$ & $0.98668$ & $0.01316$ & $75.99970$ \\
 & $7$ & $592208$ & $590236$ & $0.99877$ & $0.99544$ & $0.00941$ & $106.31960$ \\
 & $13$ & $513978$ & $512102$ & $0.99875$ & $0.99511$ & $0.01052$ & $95.01895$ \\
\hline
\multirow{5}*{$\text{14a1}$} & $3$ & $728602$ & $724758$ & $0.98485$ & $0.97965$ & $0.02263$ & $44.19867$ \\
 & $5$ & $785266$ & $778102$ & $0.99706$ & $0.98796$ & $0.01932$ & $51.74859$ \\
 & $6$ & $3154516$ & $3141180$ & $0.98341$ & $0.97925$ & $0.01765$ & $56.63447$ \\
 & $7$ & $527106$ & $521846$ & $0.99670$ & $0.98675$ & $0.02013$ & $49.67612$ \\
 & $13$ & $513070$ & $508718$ & $0.99699$ & $0.98853$ & $0.01752$ & $57.08883$ \\
\hline
\multirow{5}*{$\text{15a1}$} & $3$ & $773016$ & $771428$ & $0.99340$ & $0.99136$ & $0.01392$ & $71.85634$ \\
 & $5$ & $677206$ & $672616$ & $0.99766$ & $0.99090$ & $0.02078$ & $48.12782$ \\
 & $6$ & $3734468$ & $3726672$ & $0.98491$ & $0.98285$ & $0.01602$ & $62.44145$ \\
 & $7$ & $591906$ & $588864$ & $0.99826$ & $0.99313$ & $0.01583$ & $63.16885$ \\
 & $13$ & $514044$ & $512352$ & $0.99888$ & $0.99559$ & $0.01001$ & $99.90364$ \\
\hline
\multirow{5}*{$\text{17a1}$} & $3$ & $946876$ & $944638$ & $0.99554$ & $0.99319$ & $0.01302$ & $76.77703$ \\
 & $5$ & $785704$ & $780138$ & $0.99761$ & $0.99054$ & $0.01774$ & $56.37136$ \\
 & $6$ & $7183032$ & $7166310$ & $0.98491$ & $0.98262$ & $0.01351$ & $74.01081$ \\
 & $7$ & $591936$ & $588866$ & $ 0.99831$ & $0.99313$ & $0.01386$ & $72.14272$ \\
 & $13$ & $513646$ & $510592$ & $0.99811$ & $0.99217$ & $0.01466$ & $68.20034$ \\
\hline
\multirow{5}*{$\text{19a1}$} & $3$ & $852480$ & $851034$ & $0.99059$ & $0.98891$ & $0.01297$ & $77.10545$ \\
 & $5$ & $786682$ & $783836$ & $0.99885$ & $0.99524$ & $0.01003$ & $99.68147$ \\
 & $6$ & $6226182$ & $6186182$ & $0.98699$ & $0.98065$ & $0.01663$ & $60.14557$ \\
 & $7$ & $592192$ & $589918$ & $0.99874$ & $0.99491$ & $0.01064$ & $93.94585$ \\
 & $13$ & $513956$ & $512022$ & $0.99871$ & $0.99495$ & $0.01081$ & $92.50724$ \\
\hline
\multirow{5}*{$\text{37b1}$} & $3$ & $887308$ & $886838$ & $0.98331$ & $0.98279$ & $0.01506$ & $66.40528$ \\
 & $5$ & $786912$ & $784710$ & $0.99915$ & $0.99635$ & $0.01059$ & $94.41026$ \\
 & $6$ & $6844572$ & $6835160$ & $0.98599$ & $0.98463$ & $0.01567$ & $63.81421$ \\
 & $7$ & $592460$ & $590810$ & $0.99919$ & $0.99641$ & $0.01024$ & $97.61092$ \\
 & $13$ & $514166$ & $512682$ & $0.99912$ & $0.99623$ & $0.01087$ & $91.96074$ \\
\hline
\end{tabular}
}
\caption{$l_{k, E}(X; \epsilon)$, $\text{ratio} := l_{k, E}(X; \epsilon)/\#(\mathcal B_{k, N}(X))$, and $M = M_{k, E}(X)$ for $\epsilon = 10^{-1}, 10^{-8}$ and $k =$ 3, 5, 7, 13 with $X = 3 \times 10^6$ and $k = $ 6 with $X = 10^6$.}
\label{table:gen_lindeloff}
\end{table}

\subsection{Distributions for fixed integer values of $A_\chi$}\label{dist_exact_values}
\hfill

For a non-zero integer $l$, define 
\[
x_{k,E}(X;l) := \#\{\chi \in \mathcal B _{k,N}(X) \mid A_{\chi} = l\}
\]
and we consider the ratios
\begin{equation}\label{ratio_A_exact}
\frac{x_{k,E}(X;l)}{X^{1/2}\log^B(X)} \hspace{0.25em} \text{ for } k = 3 \text{ or } 6 \hspace{0.5em} \text{ and } \hspace{0.5em} \frac{x_{k,E}(X;l)}{\log^{B+1}(X)} \hspace{0.25em} \text{ for } k = 5.\\
\end{equation}
 
Note that for $k$ such that $\phi(k) \ge 6$, $x_{k,E}(X;l)$ is predicted finite. Therefore, we only take $k = 3, 5, 6$. It would be interesting to see the dependencies of frequencies between $x_{k,E}(X;l)$ and $x_{k,E}(X;-l)$. As before, we take $l = \pm1, \pm2, \ldots, \pm9$ for $k = 3 , 6$ and $l = \pm1, \pm4, \pm5, \pm9, \pm11, \pm16, \pm19, \pm20, \pm25$ for $k = 5$.

Those ratios for $k = 3, 5, 6$ are depicted on Figures~\ref{fig:11a1_3_A_exact},~\ref{fig:11a1_5_A_exact},~\ref{fig:11a1_6_A_exact} for 11a1, Figures~\ref{fig:14a1_3_A_exact},~\ref{fig:14a1_5_A_exact},~\ref{fig:14a1_6_A_exact} for 14a1, Figures~\ref{fig:15a1_3_A_exact},~\ref{fig:15a1_5_A_exact},~\ref{fig:15a1_6_A_exact} for 15a1, Figures~\ref{fig:17a1_3_A_exact},~\ref{fig:17a1_5_A_exact},~\ref{fig:17a1_6_A_exact} for 17a1, Figures~\ref{fig:19a1_3_A_exact},~\ref{fig:19a1_5_A_exact},~\ref{fig:19a1_6_A_exact} for 19a1 and Figures~\ref{fig:37b1_3_A_exact},~\ref{fig:37b1_5_A_exact},~\ref{fig:37b1_6_A_exact} for 37b1, respectively. 

The numerical data seem to suggest that the quantities $x_{k,E}(X;l)$ increase with similar rates of growth, as in~\eqref{ratio_A_exact}, for each small admissable value of $\vert l \vert$.

Moreover, we have found no $\chi \in \mathcal B _{k,N}(X)$ such that $A_\chi \equiv 2 \bmod 3$, as shown in Figures~\ref{fig:19a1_3_A_exact} and~\ref{fig:37b1_3_A_exact} for 19a1 and 37b1 and $k = 3$, respectively. At present we have no explanation for this phenomenon.

\subsection{Distributions of $A_\chi$ depending on $\chi(-1)$ for $k = 6$}
\hfill

 In this section, we fix $k = 6$ and split $\mathcal B _{6,N}(X)$ inito two subfamilies by the signs of $\chi$'s for $E$ and $X$ given in Table~\ref{table:combined}, and present their distributions as the previous sections. Table~\ref{table:gcd_for_k_6} presents the greatest common divisors (gcd) of $A_\chi$'s of those subfamilies. Notice that the gcd's for $\chi(-1) = 1$ are some multiples of those for $\chi(-1) = -1$;  multiples of some divisors of the orders of torsions of $E$. 

\begin{table}[h!]
\centering
\scalebox{0.9}{
\begin{tabular}{rrrrrrr}
\hline
\makecell{$\chi(-1)$} & \makecell{$\text{11a1}$} & \makecell{$\text{14a1}$}  & \makecell{$\text{15a1}$} & \makecell{$\text{17a1}$} & \makecell{$\text{19a1}$} & \makecell{$\text{37b1}$}\\
\hline
$1$ & $10$ & $6$ & $8$ & $8$ & $6$ & $12$ \\
$-1$ & $2$ & $2$ & $4$ & $4$ & $2$ & $4$ \\
\hline
\end{tabular}
}
\caption{The greatest common divisors for $k = 6$ with $X = 10^6$.}
\label{table:gcd_for_k_6}
\end{table}

Now, similarly as before, define for a positive integer $L$, a real number $0 < c < \phi(k)/4$ and a non-zero integer $l$,
\[
\begin{split}
n_{6,E}^\pm(X;L) &:= \#\{\chi \in \mathcal B _{6,N}(X) \mid 0 <  |A_{\chi}| \le L \text{ and } \chi(-1) = \pm 1\},\\
m_{6,E}^\pm(X;c) &:= \#\{\chi \in \mathcal B _{6,N}(X) \mid 0 < |A_{\chi}| \le \f^c \text{ and } \chi(-1) = \pm 1\},\\
x_{6,E}^\pm(X;l) &:= \#\{\chi \in \mathcal B _{6,N}(X) \mid A_{\chi} = l \text{ and } \chi(-1) = \pm 1\}.
\end{split}
\]

Then, applying the ratio predictions~\eqref{ratio_A},~\eqref{ratio_c} and~\eqref{ratio_A_exact} for $k = 6$, consider the following ratios for those three families above:
\begin{equation}
{n_{6,E}^{\pm}(X;L)}/{X^{1/2}\log^2(X)},\label{ratio_N_pm}
\end{equation}
\vspace{-30pt}
\begin{equation}
{m_{6,E}^{\pm}(X;c)}/{X^{c +1/2}\log^2(X)},\label{ratio_M_pm}
\end{equation}
\vspace{-30pt}
\begin{equation}
{x_{6,E}^{\pm}(X;l)}/{X^{1/2}\log^2(X)}\label{ratio_n_pm}
\end{equation}

The graphs for~\eqref{ratio_N_pm} are presented in Figures~\ref{fig:6_even_odd_A_acc_11_14_15} and~\ref{fig:6_even_odd_A_acc_17_19_37} for $L = $ 1, 2, 3. Moreover, the graphs for~\eqref{ratio_M_pm} are presented in Figures~\ref{fig:c_11_14_15_pm_acc_6} and~\ref{fig:c_17_19_37_pm_acc_6} for $c = $ 0.3, 0.4. Lastly, the graphs for~\eqref{ratio_n_pm} are presented for $l = \pm1, \pm2, \ldots, \pm9$ in Figures~\ref{fig:11a1_6_even_A_exact},~\ref{fig:14a1_6_even_A_exact},~\ref{fig:15a1_6_even_A_exact},~\ref{fig:17a1_6_even_A_exact},~\ref{fig:19a1_6_even_A_exact} and~\ref{fig:37b1_6_even_A_exact} for $\chi(-1) = 1$ and Figures~\ref{fig:11a1_6_odd_A_exact},~\ref{fig:14a1_6_odd_A_exact},~\ref{fig:15a1_6_odd_A_exact},~\ref{fig:17a1_6_odd_A_exact},~\ref{fig:19a1_6_odd_A_exact} and~\ref{fig:37b1_6_odd_A_exact} for $\chi(-1) = -1$.

Notice that as shown in Figure~\ref{fig:11a1_6_odd_A_exact}, for 11a1, we have found no $\chi \in \mathcal B_{6,11}(X)$ such that $\chi(-1) = -1$ and $A_\chi = \pm 5$ for any $X$.

\subsection{Distributions of $A_\chi$ depending on $\chi(-1)$ and $\chi(N)$ for $k = 6$}
\hfill

In this section, as in the previous section, we we fix $k = 6$ and split $\mathcal B _{6,N}(X)$ inito four subfamilies by the parities of $\ord(\chi(-1))$ and $\ord(\chi(N))$ for $E$ and $X$ given in Table~\ref{table:combined}, and present their distributions as the previous sections. 

More precisely, we partition $\mathcal B _{k,N}(X)$ by the following four subfamilies:
\[
\begin{split}
\mathcal B_{6,N}^{(1,3)}(X) &:= \{\chi \in \mathcal B _{6,N}(X) \mid \ord(\chi(-1)) = 1 \text{ and } \ord(\chi(N)) = 1 \text{ or } 3\},\\
\mathcal B_{6,N}^{(2,3)}(X) &:= \{\chi \in \mathcal B _{6,N}(X) \mid \ord(\chi(-1)) = 2 \text{ and } \ord(\chi(N)) = 1 \text{ or } 3\},\\
\mathcal B_{6,N}^{(1,6)}(X) &:= \{\chi \in \mathcal B _{6,N}(X) \mid \ord(\chi(-1)) = 1 \text{ and } \ord(\chi(N)) = 2 \text{ or } 6\},\\
\mathcal B_{6,N}^{(2,6)}(X) &:= \{\chi \in \mathcal B _{6,N}(X) \mid \ord(\chi(-1)) = 2 \text{ and } \ord(\chi(N)) = 2 \text{ or } 6\},\\
\end{split}
\]
Denote $\alpha = 1, 2$ and $\beta = 3, 6$. Then, define for a positive integer $L$, a real number $0 < c < \phi(k)/4$ and a non-zero integer $l$
\[
\begin{split}
n_{6,E}^{(\alpha, \beta)}(X;L) &:= \#\{\chi \in \mathcal B _{6,N}^{(\alpha, \beta)}(X) \mid 0 <  |A_{\chi}| \le L \},\\
m_{6,E}^{(\alpha, \beta)}(X;c) &:= \#\{\chi \in \mathcal B _{6,N}^{(\alpha, \beta)}(X) \mid 0 < |A_{\chi}| \le \f^c \},\\
x_{6,E}^{(\alpha, \beta)}(X;l) &:= \#\{\chi \in \mathcal B _{6,N}^{(\alpha, \beta)}(X) \mid A_{\chi} = l \}.\\
\end{split}
\]
Notice that the image of $\chi(-N)$ for $\chi \in \mathcal B_{6,N}^{(\alpha, \beta)}(X)$ is a cubic root of unity (or sixth root of unity) for $(\alpha, \beta) = (1,3) \text{ or } (2,6)$ (or $(2,3) \text{ or } (1,6)$, respectively).  As in the previous section, applying the ratio predictions~\eqref{ratio_A},~\eqref{ratio_c} and~\eqref{ratio_A_exact} for $k = 6$, consider the following ratios for those three families above:

\vspace{-20pt}
\begin{equation}
{n_{6,E}^{(\alpha, \beta)}(X;L)}/{X^{1/2}\log^2(X)},\label{ratio_N_orders}
\end{equation}
\vspace{-20pt}
\begin{equation}
{m_{6,E}^{(\alpha, \beta)}(X;c)}/{X^{c +1/2}\log^2(X)},\label{ratio_M_orders}
\end{equation}
\vspace{-20pt}
\begin{equation}
{x_{6,E}^{(\alpha, \beta)}(X;l)}/{X^{1/2}\log^2(X)}.\label{ratio_n_orders}
\end{equation}

Table~\ref{table:gcd_for_k_6_orders} presents the greatest common divisors of $A_\chi$'s of the subfamilies $\mathcal B _{6,N}^{(\alpha, \beta)}(X)$. 

\begin{table}[h!]
\centering
\scalebox{0.9}{
\begin{tabular}{crrrrrr}
\hline
\makecell{$(\alpha, \beta)$} & \makecell{$\text{11a1}$} & \makecell{$\text{14a1}$}  & \makecell{$\text{15a1}$} & \makecell{$\text{17a1}$} & \makecell{$\text{19a1}$} & \makecell{$\text{37b1}$}\\
\hline
$(1,3)$ & $10$ & $6$ & $8$ & $8$ & $6$ & $12$\\
$(2,3)$ & $6$ & $6$ & $12$ & $12$ & $6$ & $12$\\
$(1,6)$ & $30$ & $18$ & $24$ & $24$ & $18$ & $36$\\
$(2,6)$ & $2$ & $2$ & $4$ & $4$ & $2$ & $4$\\
\hline
\end{tabular}
}
\caption{The greatest common divisors for $k = 6$ with $X = 10^6$.}
\label{table:gcd_for_k_6_orders}
\end{table}

The graphs for~\eqref{ratio_N_orders},~\eqref{ratio_M_orders} and~\eqref{ratio_n_orders} are presented in Figures~\ref{fig:6_alpha_beta_A_acc_11_14},~\ref{fig:6_alpha_beta_A_acc_15_17},~\ref{fig:6_alpha_beta_A_acc_19_37} for $L =$ 1, 2, 3, Figures~\ref{fig:c_11_14_acc_6_orders},~\ref{fig:c_15_17_acc_6_orders},~\ref{fig:c_19_37_acc_6_orders} for $c = $ 0.3, 0.4 and Figures~\ref{fig:11a1_6_1_3_A_exact} to~\ref{fig:37b1_6_2_6_A_exact} for $l = \pm 1, \pm 2, \ldots, \pm 9$, respectively.

 Similarly to the previous section, for 11a1, $l = \pm 5$ and $(\alpha, \beta) = (2, 6)$, we have found no $\chi \in \mathcal B_{6,11}^{(2,6)}(X)$ such that $A_\chi = \pm 5$ for any $X$ as shown in Figure~\ref{fig:11a1_6_2_6_A_exact}. Moreover, surprisingly again, we find similar phenomenons for $(\alpha, \beta) = (2, 3)$ as mentioned in the last paragraph of Section~\ref{dist_exact_values}.  More precisely, for 19a1, we have found no $\chi \in \mathcal B_{k,N}^{(2,3)}(X)$ such that $A_\chi \equiv 1 \bmod 3$, except for $\f = 9$, as shown in Figure~\ref{fig:19a1_6_2_3_A_exact} and for 38b1, we have found no $\chi \in \mathcal B_{k,N}^{(2,3)}(X)$ such that $A_\chi \equiv 2 \bmod 3$ as shown in Figure~\ref{fig:37b1_6_2_3_A_exact}. 


\begin{figure}[h!] 
\hspace*{-.7cm}
\begin{subfigure}[h]{0.4\linewidth}
\includegraphics[width=\linewidth]{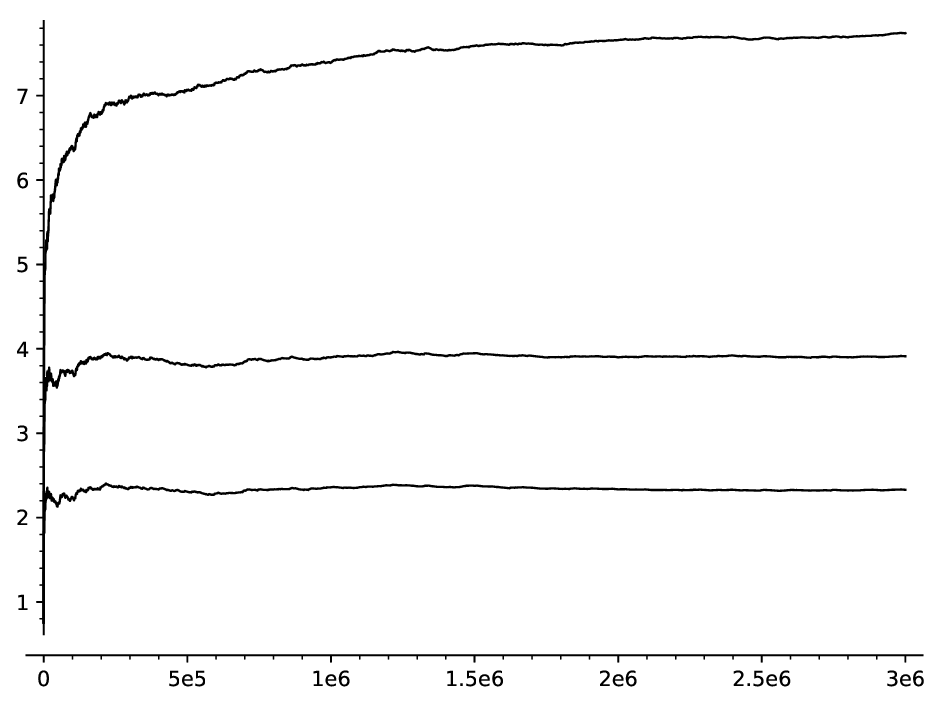}
\caption{11a1 $k = 3$: $n_{k,E}(X;L)/X^{1/2}$} \label{fig:11_3_acc_A}
\end{subfigure}\hspace*{\fill}
\begin{subfigure}[h]{0.4\linewidth}
\includegraphics[width=\linewidth]{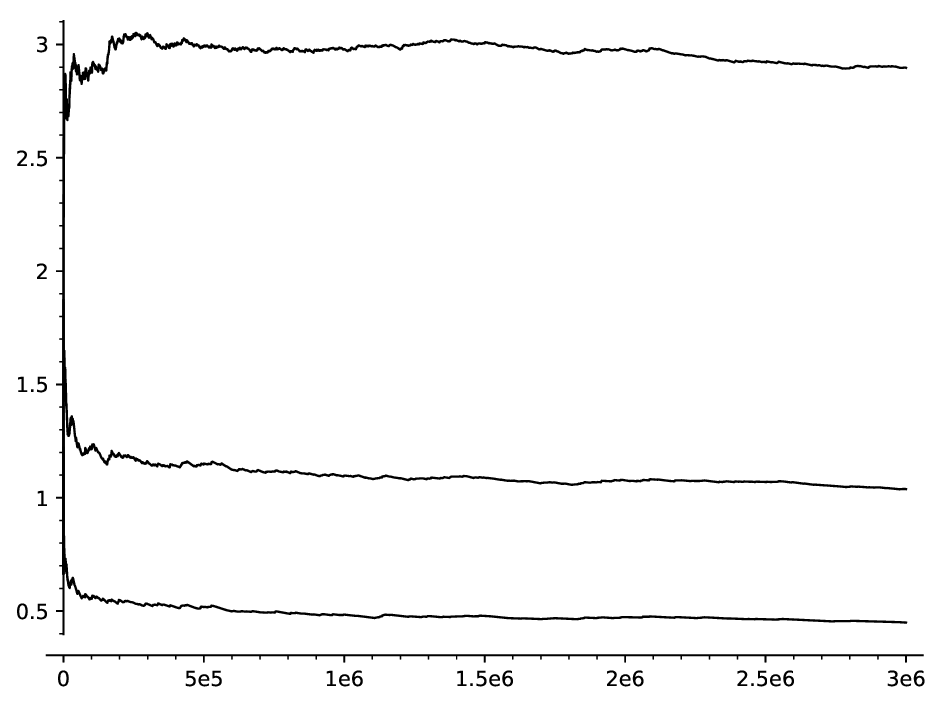}
\caption{14a1 $k = 3$: $n_{k,E}(X;L)/X^{1/2}$} \label{fig:14_3_acc_A}
\end{subfigure}
\hspace*{-.7cm}
\begin{subfigure}[h]{0.4\linewidth}
\includegraphics[width=\linewidth]{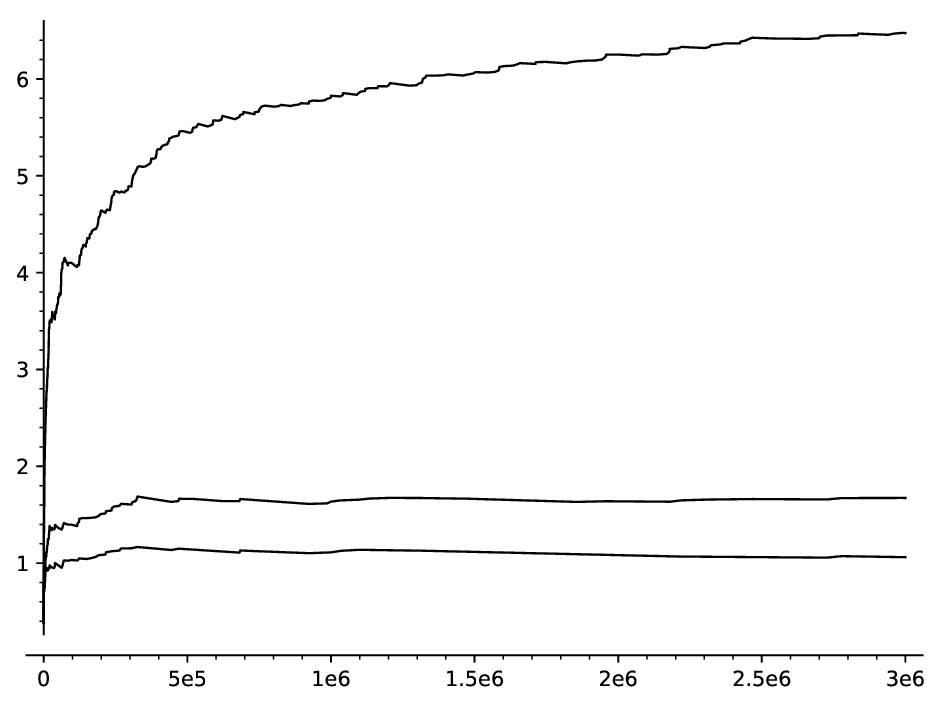}
\caption{11a1 $k = 5$: $n_{k,E}(X;L)/\log^{2}(X)$} \label{fig:11_5_acc_A}
\end{subfigure}\hspace*{\fill}
\begin{subfigure}[h]{0.4\linewidth}
\includegraphics[width=\linewidth]{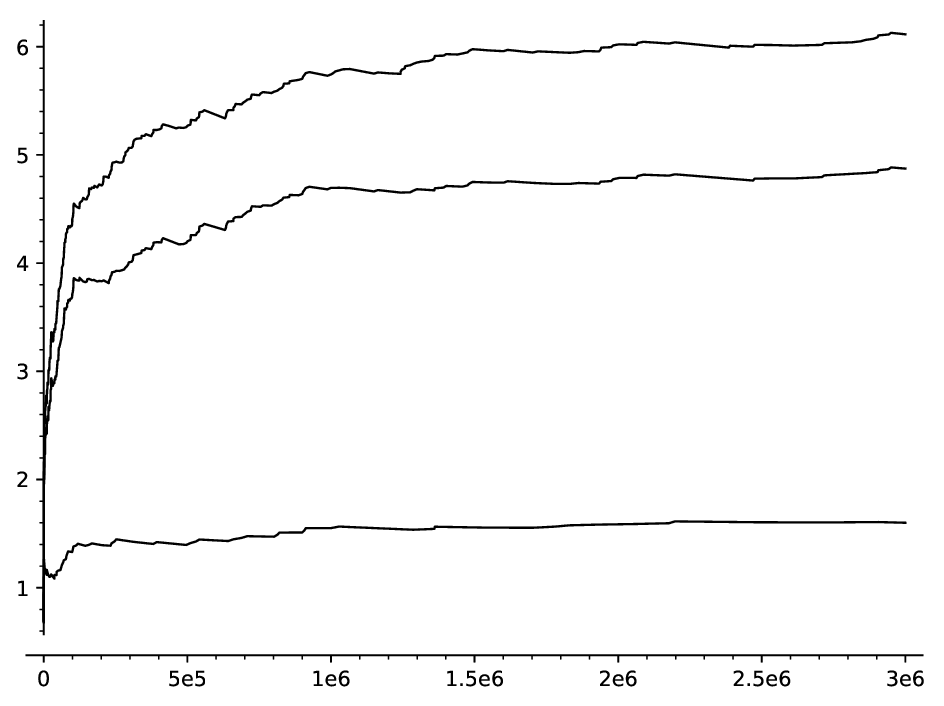}
\caption{14a1 $k = 5$: $n_{k,E}(X;L)/\log^{2}(X)$} \label{fig:14_5_acc_A}
\end{subfigure}
\hspace*{-.7cm}
\begin{subfigure}[h]{0.4\linewidth}
\includegraphics[width=\linewidth]{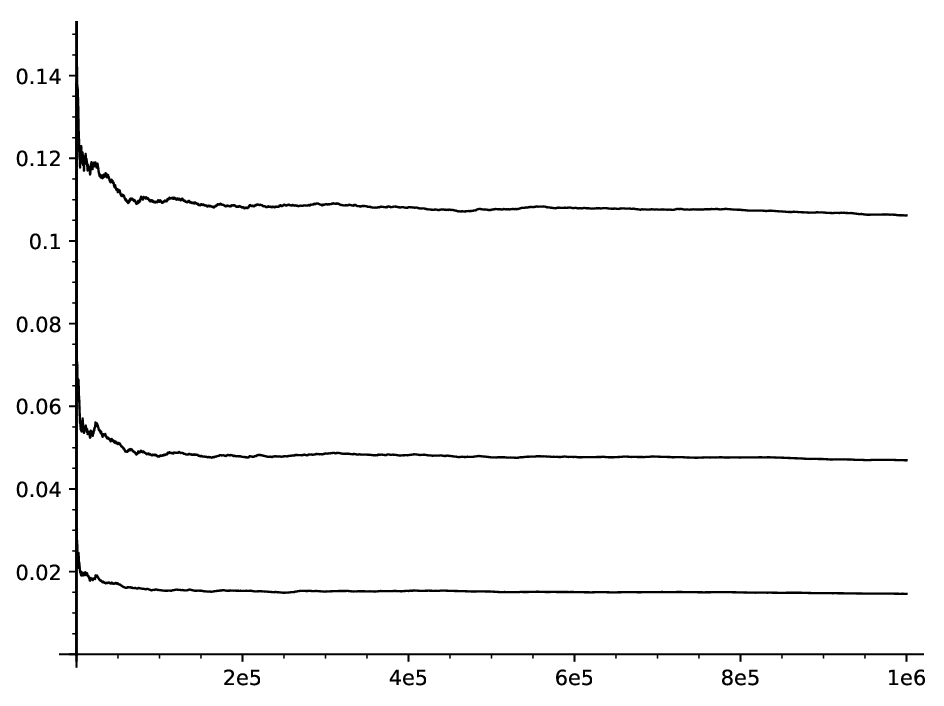}
\caption{11a1 $k = 6$: $n_{k,E}(X;L)/X^{1/2}\log^2(X)$} \label{fig:11_6_acc_A}
\end{subfigure}\hspace*{\fill}
\begin{subfigure}[h]{0.4\linewidth}
\includegraphics[width=\linewidth]{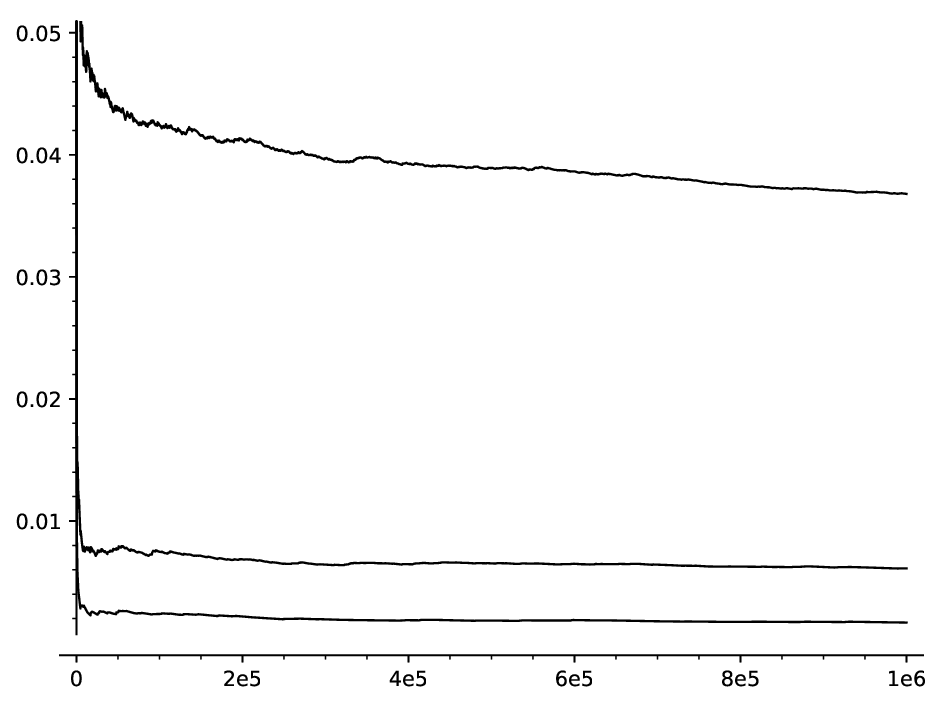}
\caption{14a1 $k = 6$: $n_{k,E}(X;L)/X^{1/2}\log^2(X)$} \label{fig:14_6_acc_A}
\end{subfigure}
\caption{Ratio~\eqref{ratio_A}: $L =$ 1, 2, 3 for $k =$ 3, 6 and $L =$ 1, 4, 5 for $k =$ 5 from the bottom to the top. Note that for each $E, k$ and a fixed $X$, the ratio in Equation~\eqref{ratio_A} for $L$ is less than or equal to that for $L'$ if $L \le L$ since $0 <  |A_{\chi}| \le L \le L'$.} \label{fig:A_acc_11_14}
\end{figure}

\clearpage

\begin{figure}[t!] 
\hspace*{-.7cm}
\begin{subfigure}[h]{0.4\linewidth}
\includegraphics[width=\linewidth]{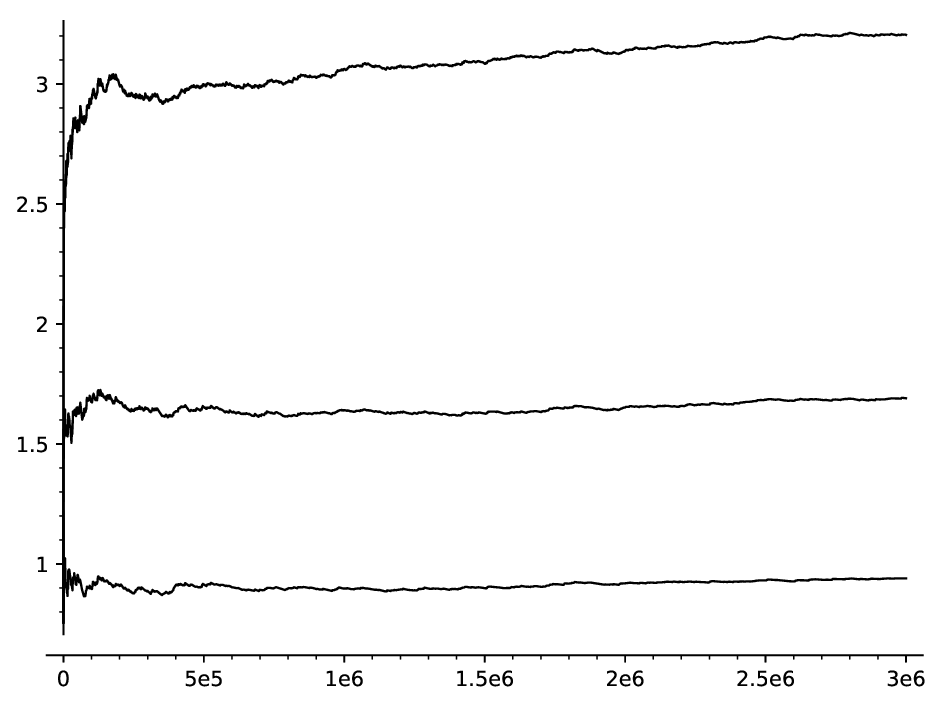}
\caption{15a1 $k = 3$: $n_{k,E}(X;L)/X^{1/2}$} \label{fig:15_3_acc_A}
\end{subfigure}\hspace*{\fill}
\begin{subfigure}[h]{0.4\linewidth}
\includegraphics[width=\linewidth]{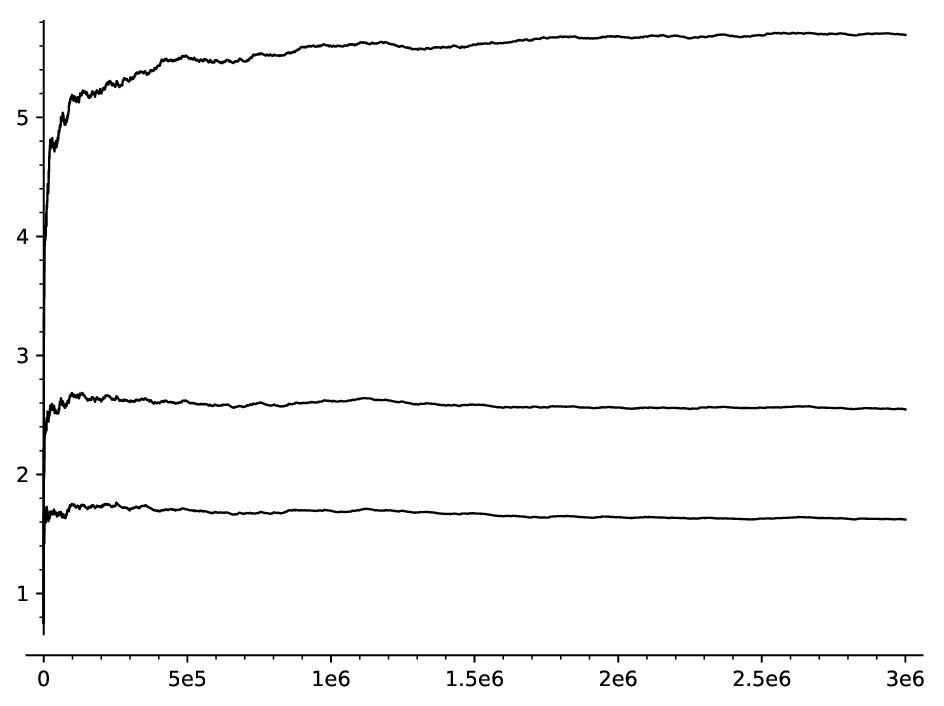}
\caption{17a1 $k = 3$: $n_{k,E}(X;L)/X^{1/2}$} \label{fig:17_3_acc_A}
\end{subfigure}
\hspace*{-.7cm}
\begin{subfigure}[h]{0.4\linewidth}
\includegraphics[width=\linewidth]{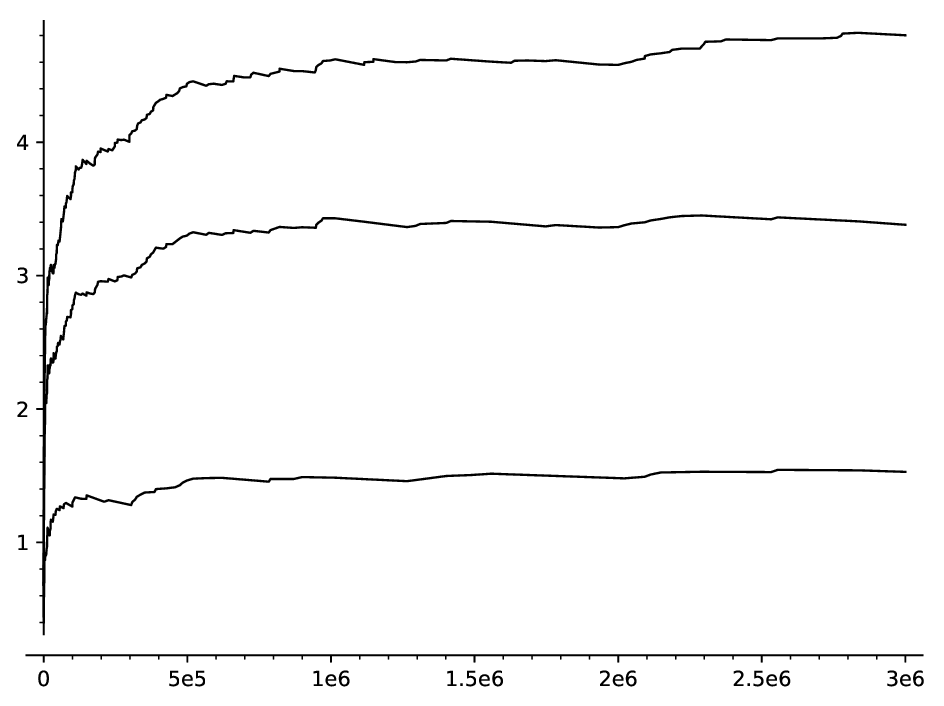}
\caption{15a1 $k = 5$: $n_{k,E}(X;L)/\log^{2}(X)$} \label{fig:15_5_acc_A}
\end{subfigure}\hspace*{\fill}
\begin{subfigure}[h]{0.4\linewidth}
\includegraphics[width=\linewidth]{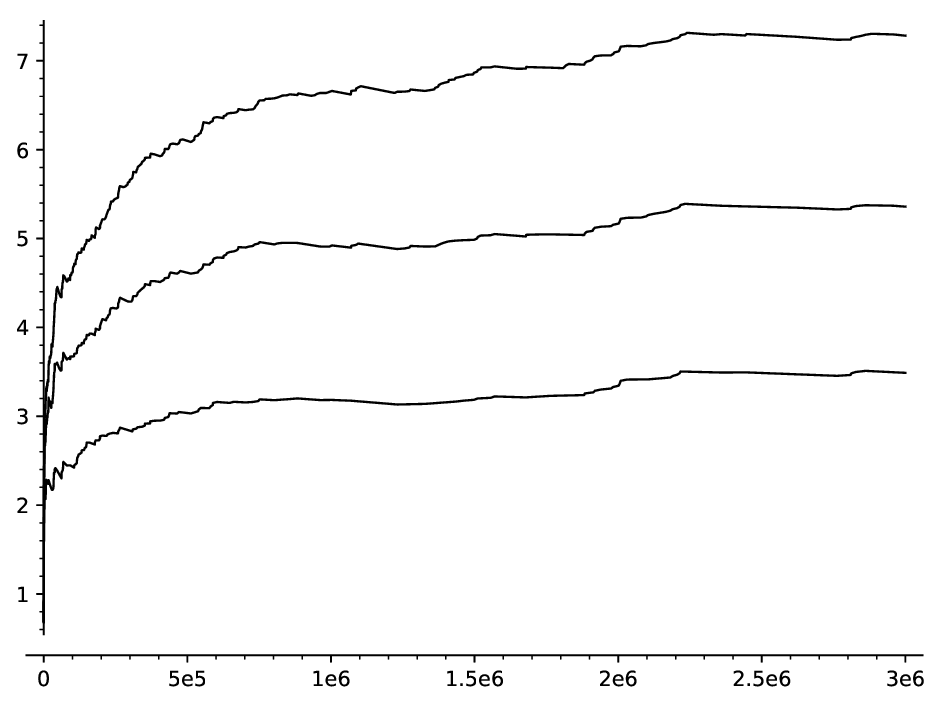}
\caption{17a1 $k = 5$: $n_{k,E}(X;L)/\log^{2}(X)$} \label{fig:17_5_acc_A}
\end{subfigure}
\hspace*{-.7cm}
\begin{subfigure}[h]{0.4\linewidth}
\includegraphics[width=\linewidth]{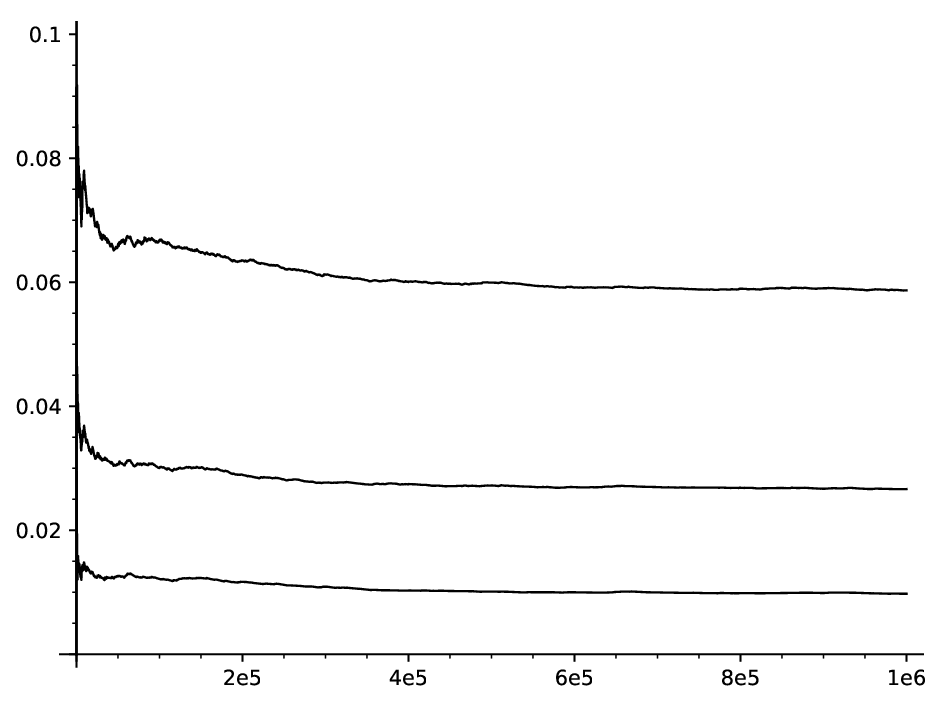}
\caption{15a1 $k = 6$: $n_{k,E}(X;L)/X^{1/2}\log^2(X)$} \label{fig:15_6_acc_A}
\end{subfigure}\hspace*{\fill}
\begin{subfigure}[h]{0.4\linewidth}
\includegraphics[width=\linewidth]{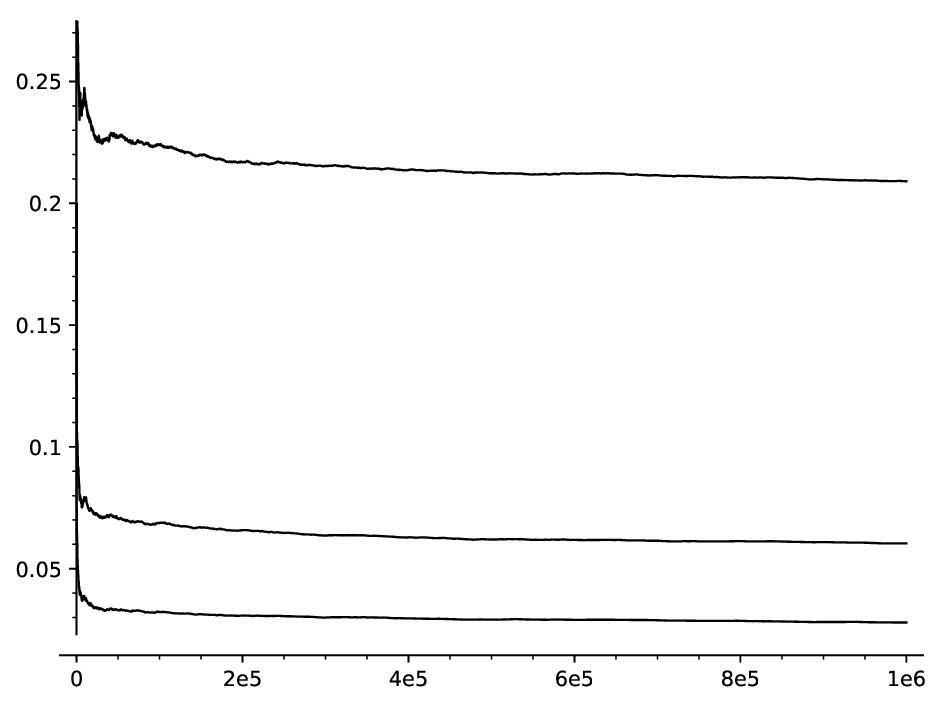}
\caption{17a1 $k = 6$: $n_{k,E}(X;L)/X^{1/2}\log^2(X)$} \label{fig:17_6_acc_A}
\end{subfigure}
\caption{Ratio~\eqref{ratio_A}: $L =$ 1, 2, 3 for $k =$ 3, 6 and $L =$ 1, 4, 5 for $k =$ 5 from the bottom to the top. Note that for each $E, k$ and a fixed $X$, the ratio in Equation~\eqref{ratio_A} for $L$ is less than or equal to that for $L'$ if $L \le L$ since $0 <  |A_{\chi}| \le L \le L'$.} \label{fig:A_acc_15_17}
\end{figure}

\clearpage

\begin{figure}[t!] 
\hspace*{-.7cm}
\begin{subfigure}[h]{0.4\linewidth}
\includegraphics[width=\linewidth]{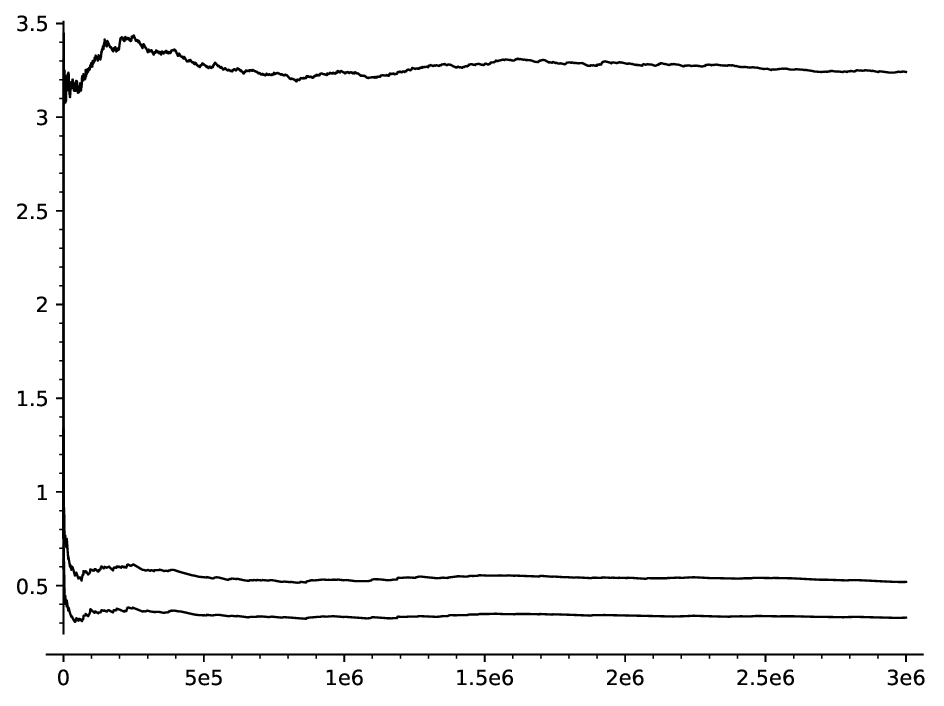}
\caption{19a1 $k = 3$: $n_{k,E}(X;L)/X^{1/2}$} \label{fig:19_3_acc_A}
\end{subfigure}\hspace*{\fill}
\begin{subfigure}[h]{0.4\linewidth}
\includegraphics[width=\linewidth]{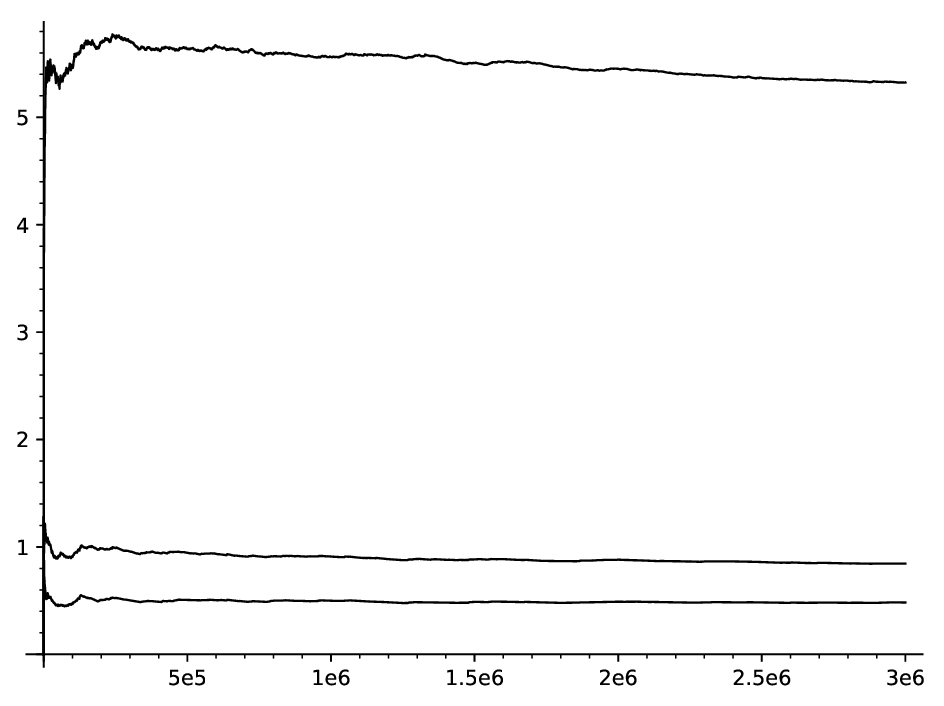}
\caption{37b1 $k = 3$: $n_{k,E}(X;L)/X^{1/2}$} \label{fig:37_3_acc_A}
\end{subfigure}
\hspace*{-.7cm}
\begin{subfigure}[h]{0.4\linewidth}
\includegraphics[width=\linewidth]{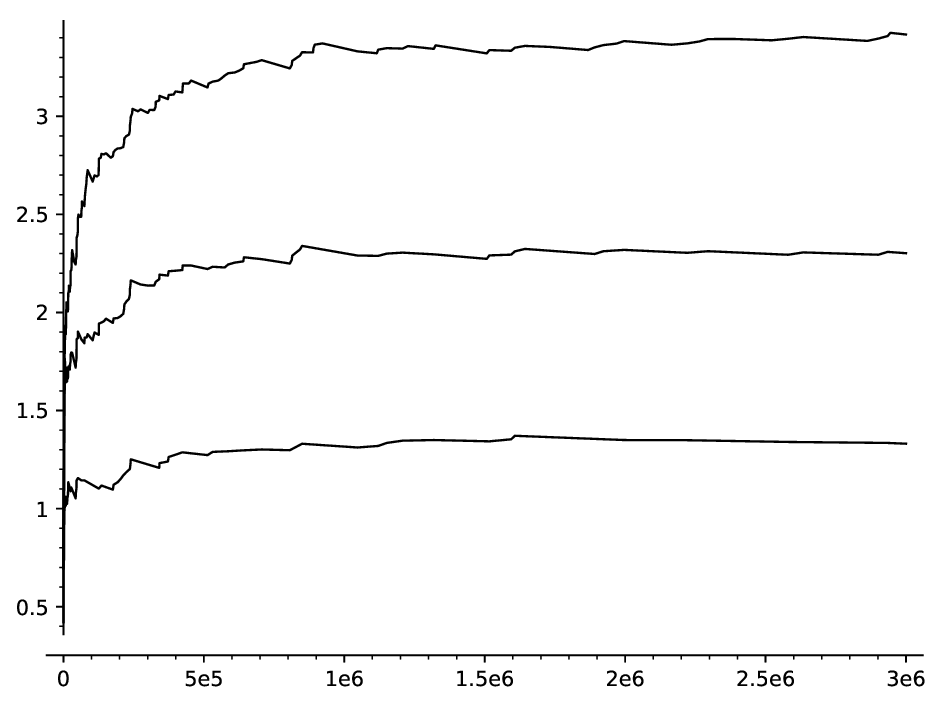}
\caption{19a1 $k = 5$: $n_{k,E}(X;L)/\log^{2}(X)$} \label{fig:19_5_acc_A}
\end{subfigure}\hspace*{\fill}
\begin{subfigure}[h]{0.4\linewidth}
\includegraphics[width=\linewidth]{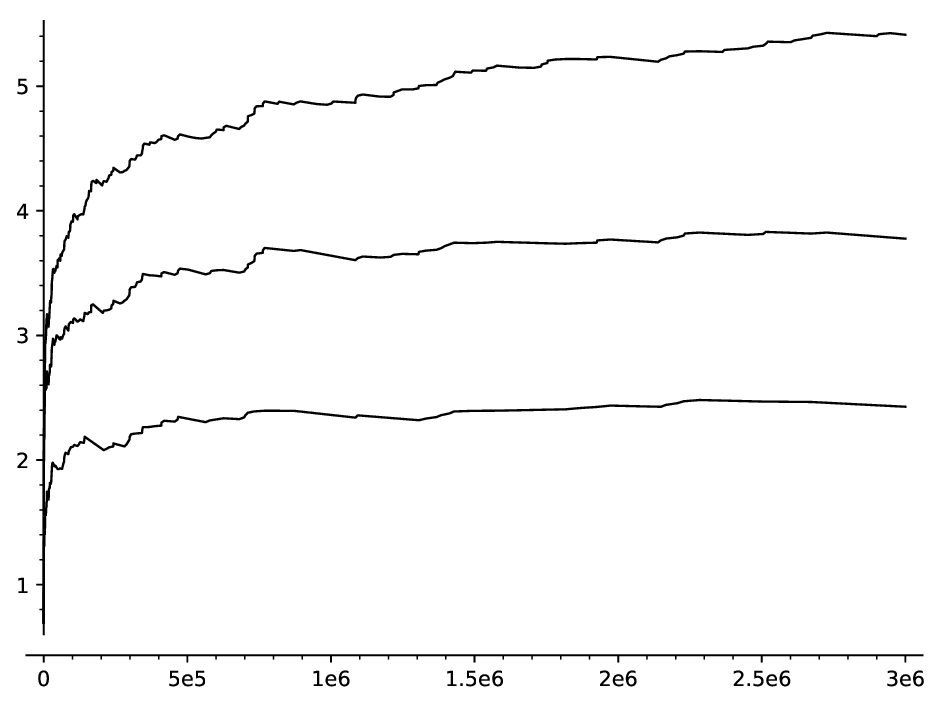}
\caption{37b1 $k = 5$: $n_{k,E}(X;L)/\log^{2}(X)$} \label{fig:37_5_acc_A}
\end{subfigure}
\hspace*{-.7cm}
\begin{subfigure}[h]{0.4\linewidth}
\includegraphics[width=\linewidth]{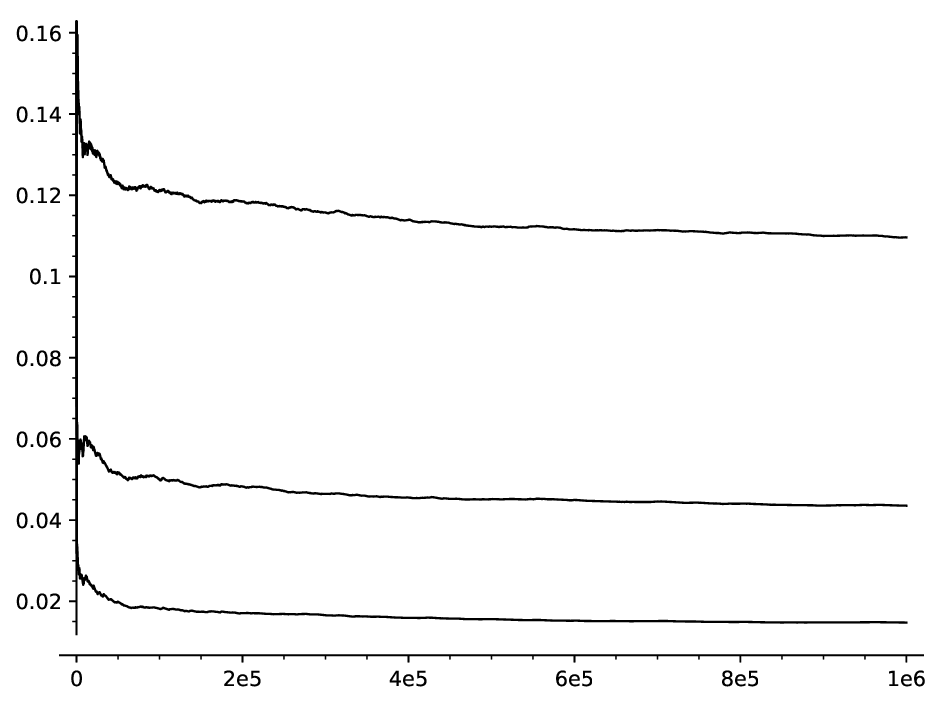}
\caption{19a1 $k = 6$: $n_{k,E}(X;L)/X^{1/2}\log^2(X)$} \label{fig:19_6_acc_A}
\end{subfigure}\hspace*{\fill}
\begin{subfigure}[h]{0.4\linewidth}
\includegraphics[width=\linewidth]{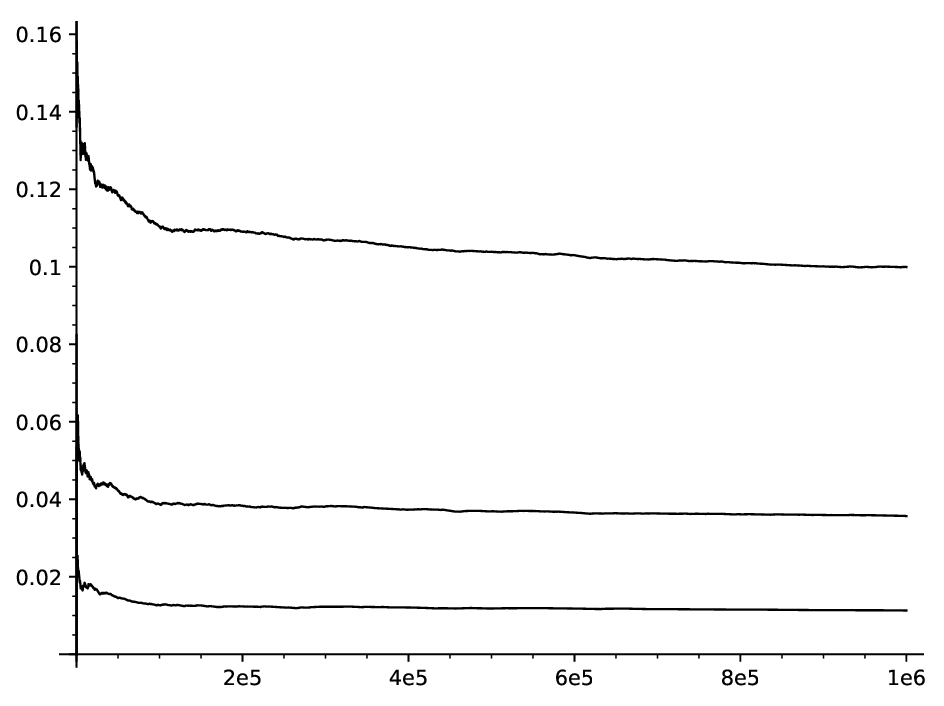}
\caption{37b1 $k = 6$: $n_{k,E}(X;L)/X^{1/2}\log^2(X)$} \label{fig:37_6_acc_A}
\end{subfigure}
\caption{Ratio~\eqref{ratio_A}: $L =$ 1, 2, 3 for $k =$ 3, 6 and $L =$ 1, 4, 5 for $k =$ 5 from the bottom to the top. Note that for each $E, k$ and a fixed $X$, the ratio in Equation~\eqref{ratio_A} for $L$ is less than or equal to that for $L'$ if $L \le L$ since $0 <  |A_{\chi}| \le L \le L'$.} \label{fig:A_acc_19_37}
\end{figure}

\clearpage

\begin{figure}[t!] 
\hspace*{-.7cm}
\begin{subfigure}[b]{0.4\linewidth}
\includegraphics[width=\linewidth]{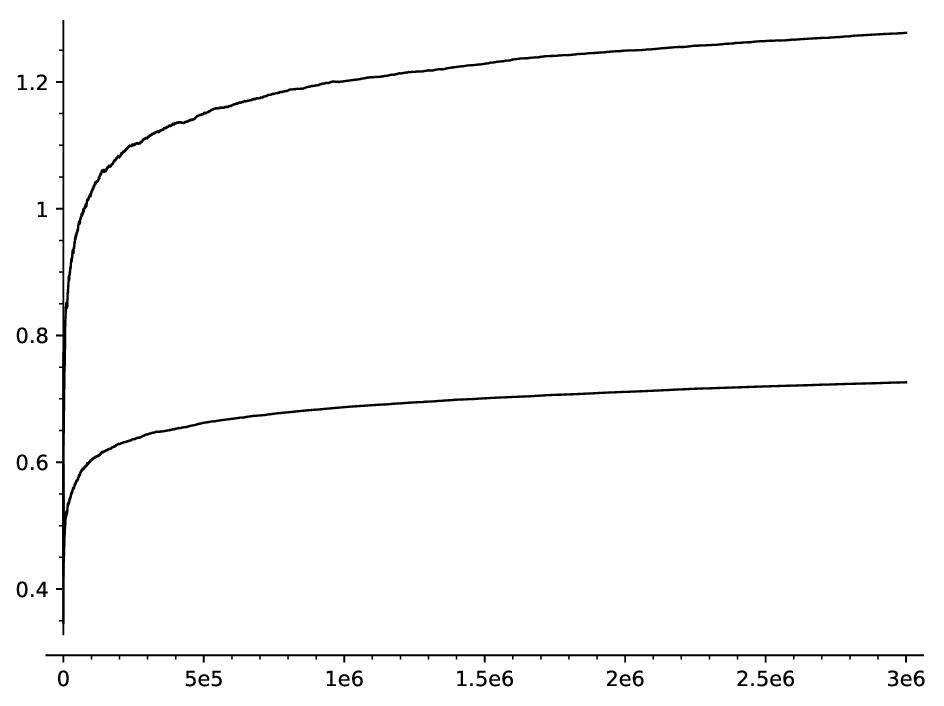}
\caption{11a1 $k = 3$: $m_{k,E}(X;c)/X^{c +1/2}$\\ Top to bottom $c =$ 0.3, 0.4} \label{fig:11_3_acc_c}
\end{subfigure}\hspace*{\fill}
\begin{subfigure}[b]{0.4\linewidth}
\includegraphics[width=\linewidth]{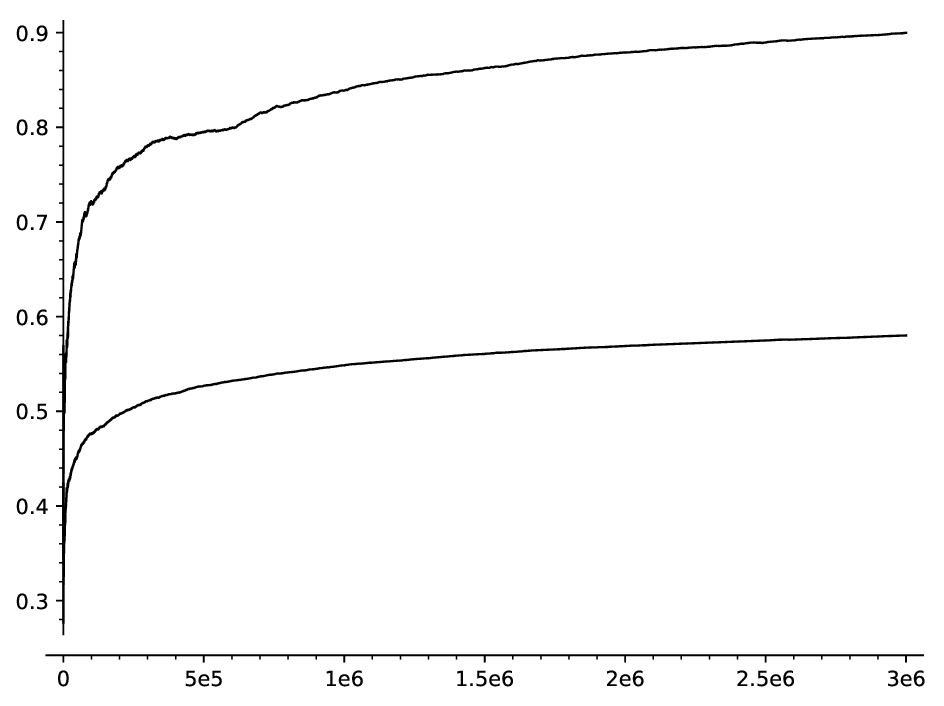}
\caption{14a1 $k = 3$: $m_{k,E}(X;c)/X^{c +1/2}$\\ Top to bottom $c =$ 0.3, 0.4} \label{fig:14_3_acc_c}
\end{subfigure}
\hspace*{-.7cm}
\begin{subfigure}[b]{0.4\linewidth}
\includegraphics[width=\linewidth]{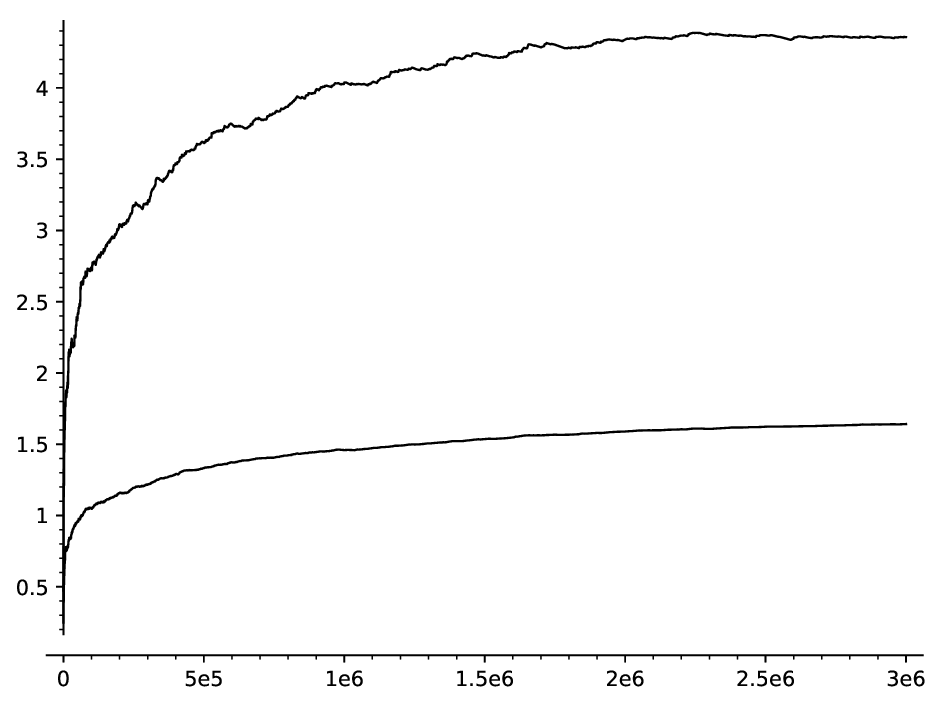}
\caption{11a1 $k = 5$: $m_{k,E}(X;c)/X^{c}\log(X)$\\ Top to bottom $c =$ 0.3, 0.5} \label{fig:11_5_acc_c}
\end{subfigure}\hspace*{\fill}
\begin{subfigure}[b]{0.4\linewidth}
\includegraphics[width=\linewidth]{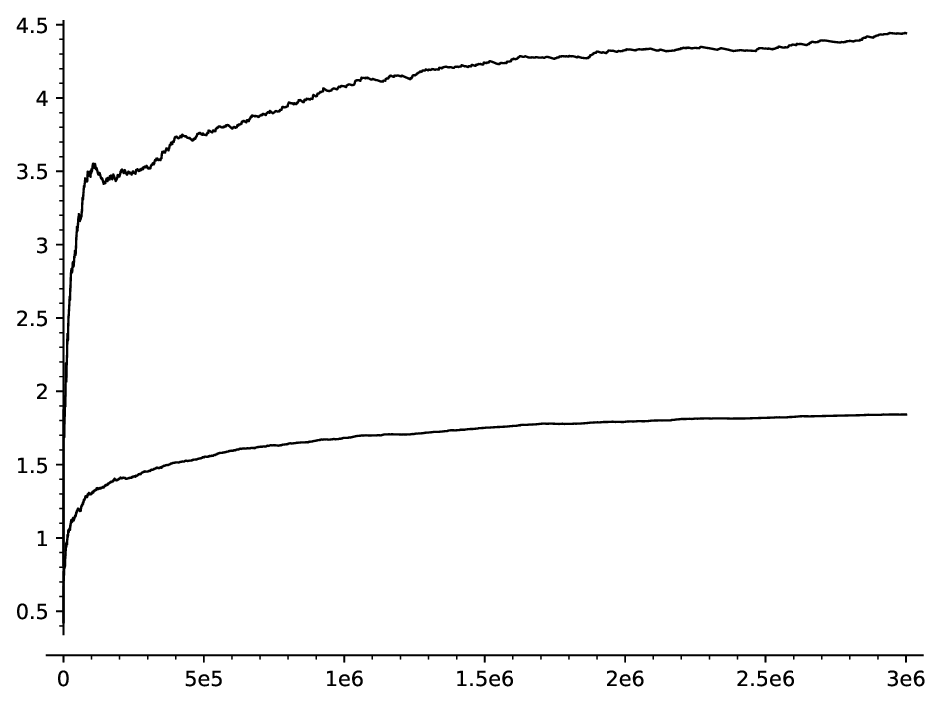}
\caption{14a1 $k = 5$: $m_{k,E}(X;c)/X^{c}\log(X)$\\ Top to bottom $c =$ 0.3, 0.5} \label{fig:14_5_acc_c}
\end{subfigure}
\hspace*{-.7cm}
\begin{subfigure}[b]{0.4\linewidth}
\includegraphics[width=\linewidth]{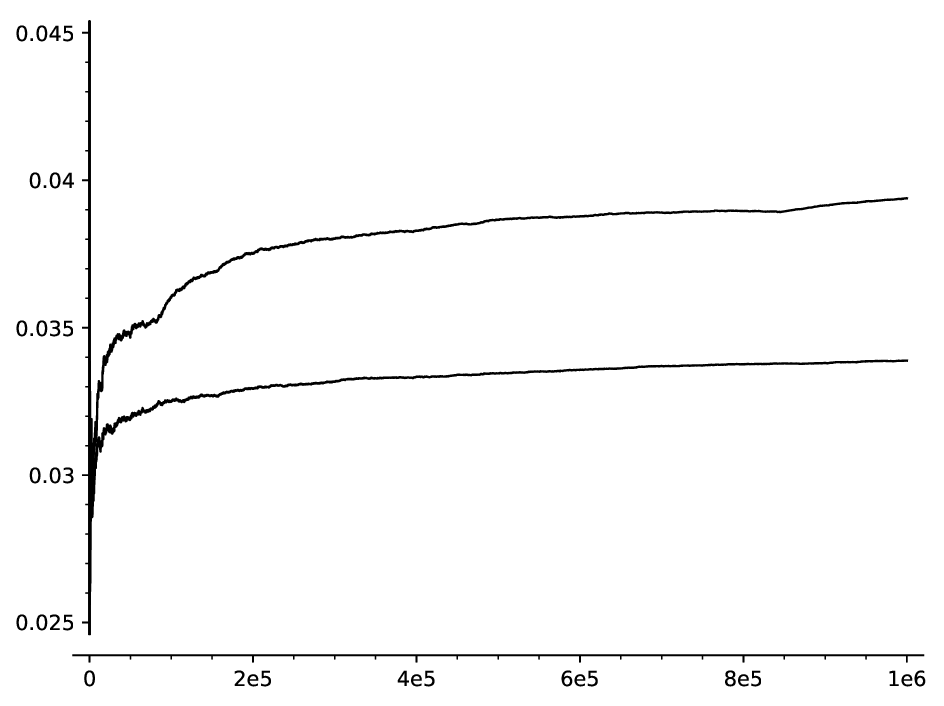}
\caption{11a1 $k = 6$: $m_{k,E}(X;c)/X^{c+1/2}\log^2(X)$\\ Top to bottom $c =$ 0.3, 0.4} \label{fig:11_6_acc_c}
\end{subfigure}\hspace*{\fill}
\begin{subfigure}[b]{0.4\linewidth}
\includegraphics[width=\linewidth]{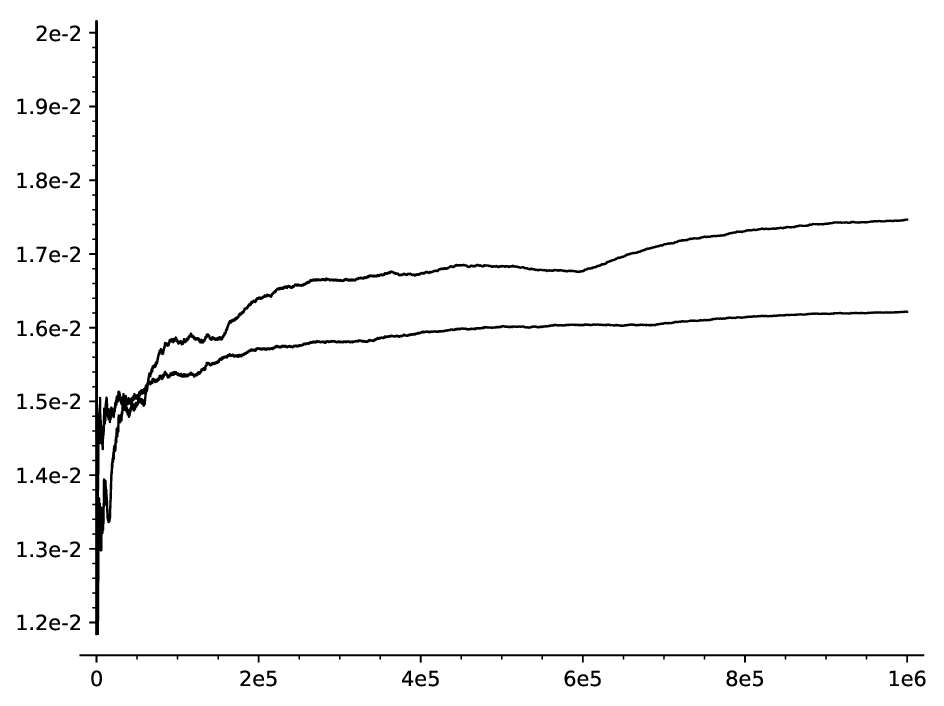}
\caption{14a1 $k = 6$: $m_{k,E}(X;c)/X^{c +1/2}\log^2(X)$\\ Top to bottom $c =$ 0.3, 0.4} \label{fig:14_6_acc_c}
\end{subfigure}
\caption{Ratio~\eqref{ratio_c} for $k = 3, 5, 6$ and $\phi(k)/4 -1 < c \le \phi(k)/4$} \label{fig:c_11_14_acc_3_5_6}
\end{figure}

\clearpage

\begin{figure}[t!] 
\hspace*{-.7cm}
\begin{subfigure}[b]{0.4\linewidth}
\includegraphics[width=\linewidth]{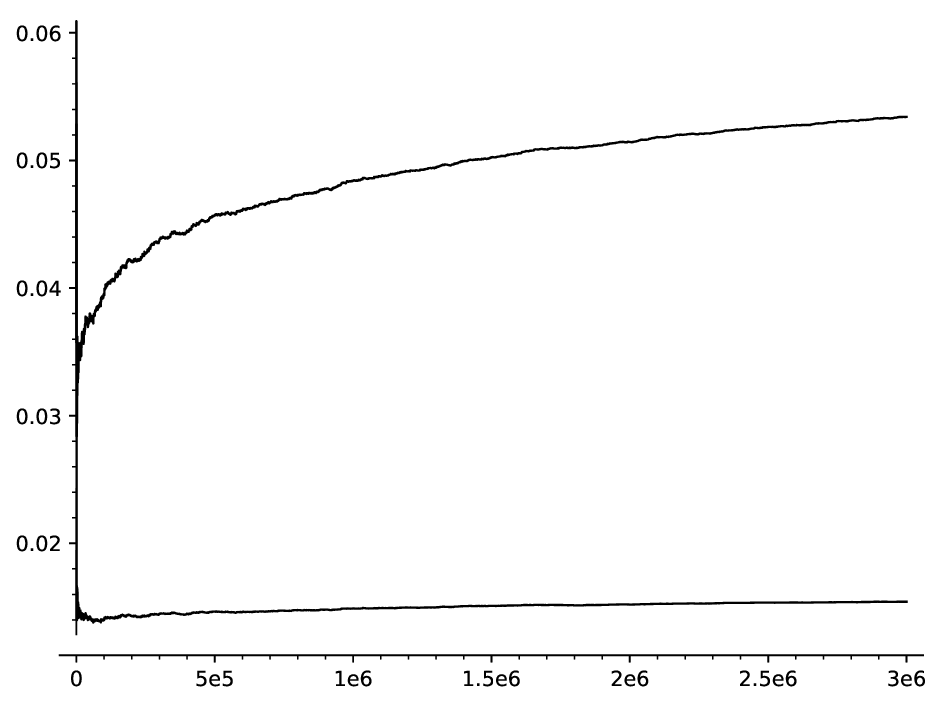}
\caption{11a1 $k = 7$: $m_{k,E}(X;c)/X^{c-1/2}\log^2(X)$\\ Top to bottom $c =$ 1.2, 1.3} \label{fig:11_7_acc_c}
\end{subfigure}\hspace*{\fill}
\begin{subfigure}[b]{0.4\linewidth}
\includegraphics[width=\linewidth]{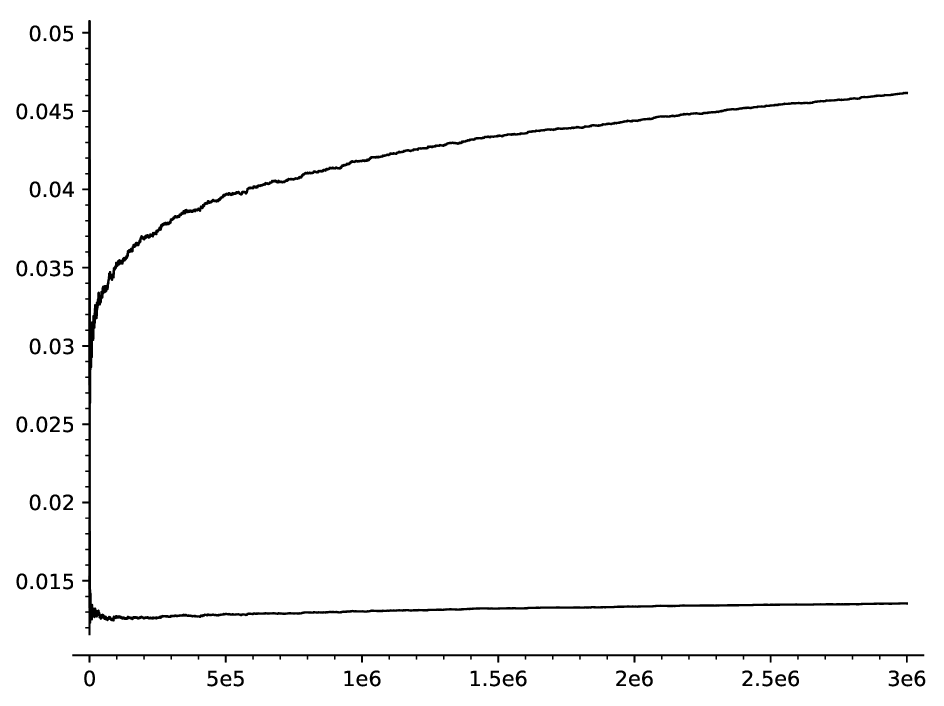}
\caption{14a1 $k = 7$: $m_{k,E}(X;c)/X^{c-1/2}\log^2(X)$\\ Top to bottom $c =$ 1.2, 1.3} \label{fig:14_7_acc_c}
\end{subfigure}
\hspace*{-.7cm}
\begin{subfigure}[b]{0.4\linewidth}
\includegraphics[width=\linewidth]{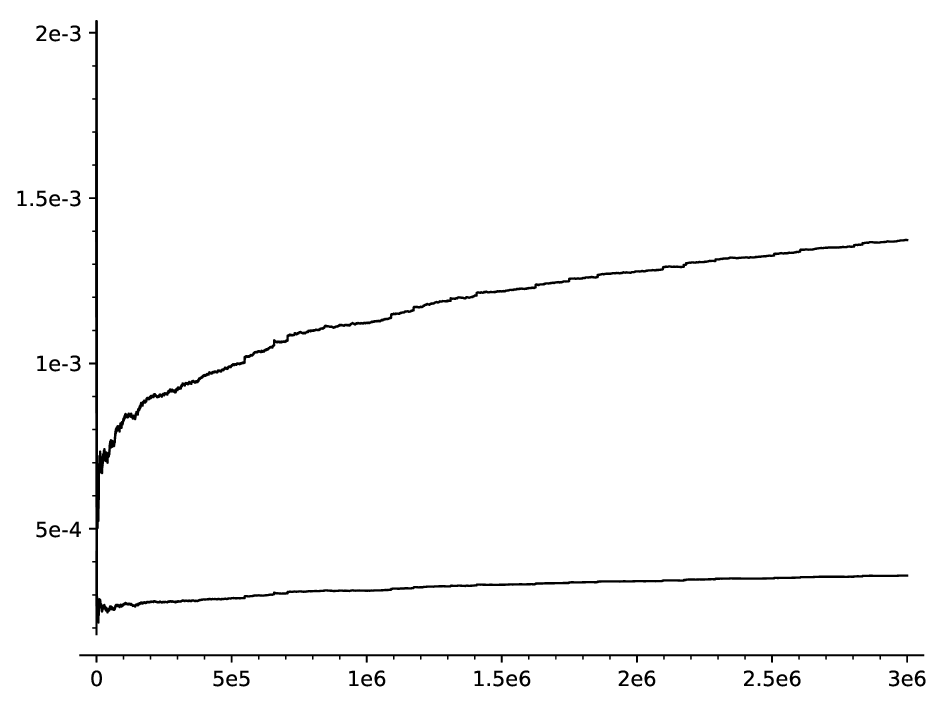}
\caption{11a1 $k = 13$: $m_{k,E}(X;c)/X^{c-2}\log^5(X)$\\ Top to bottom $c =$ 2.4, 2.5} \label{fig:11_13_acc_c}
\end{subfigure}\hspace*{\fill}
\begin{subfigure}[b]{0.4\linewidth}
\includegraphics[width=\linewidth]{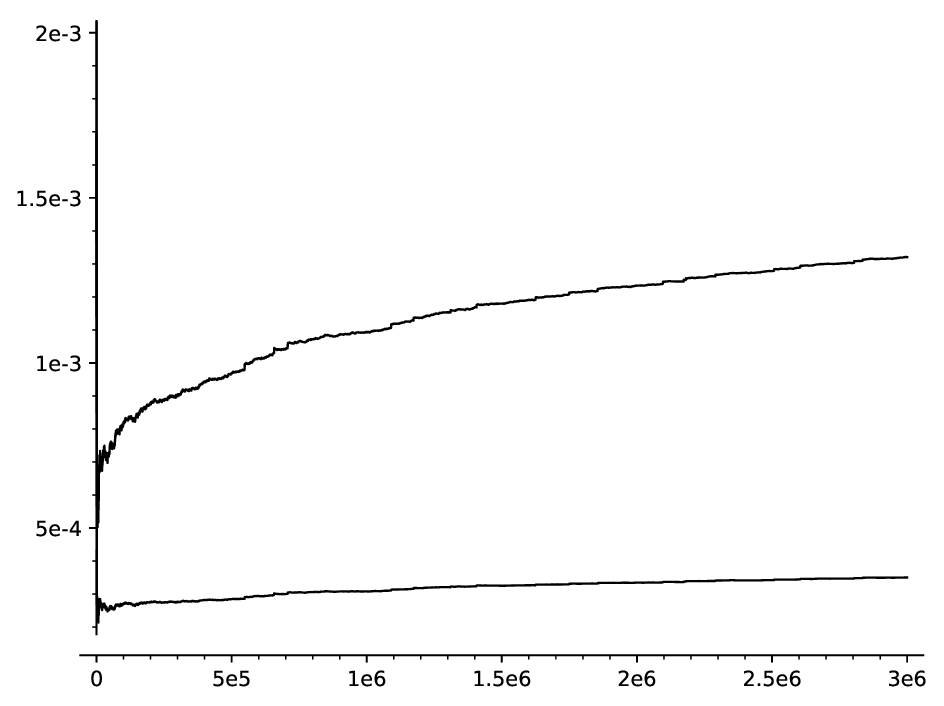}
\caption{14a1 $k = 13$: $m_{k,E}(X;c)/X^{c-2}\log^5(X)$\\ Top to bottom $c =$ 2.4, 2.5} \label{fig:14_13_acc_c}
\end{subfigure}
\caption{Ratio~\eqref{ratio_c} for $k = 7, 13$ and $\phi(k)/4 -1 \le c \le \phi(k)/4$} \label{fig:c_11_14_acc_7_13}
\end{figure}

\clearpage

\begin{figure}[t!] 
\hspace*{-.7cm}
\begin{subfigure}[b]{0.4\linewidth}
\includegraphics[width=\linewidth]{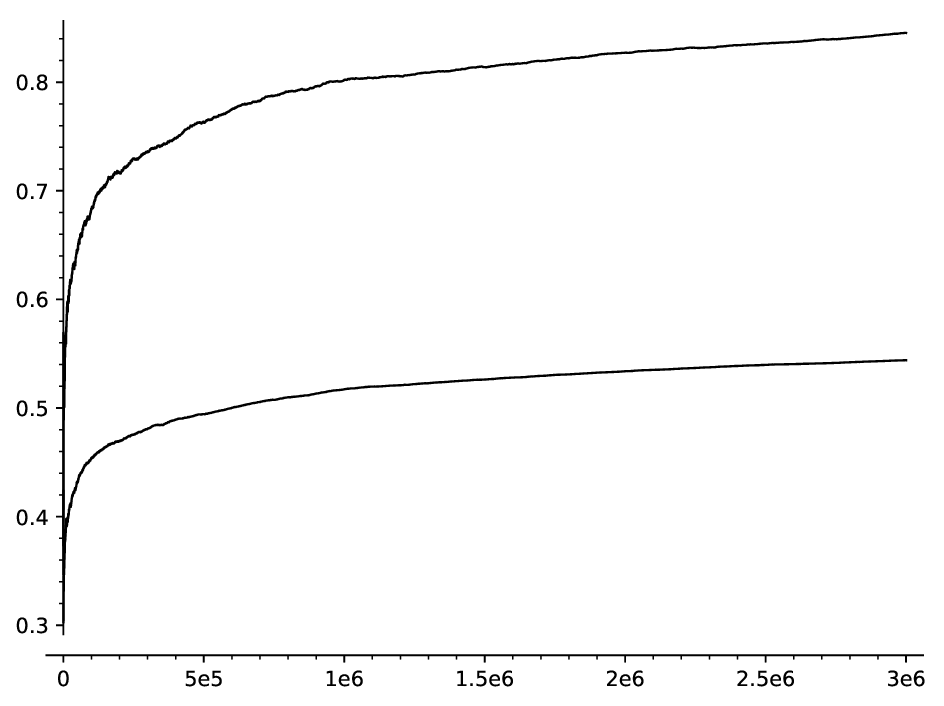}
\caption{15a1 $k = 3$: $m_{k,E}(X;c)/X^{c +1/2}$\\ Top to bottom $c =$ 0.3, 0.4} \label{fig:15_3_acc_c}
\end{subfigure}\hspace*{\fill}
\begin{subfigure}[b]{0.4\linewidth}
\includegraphics[width=\linewidth]{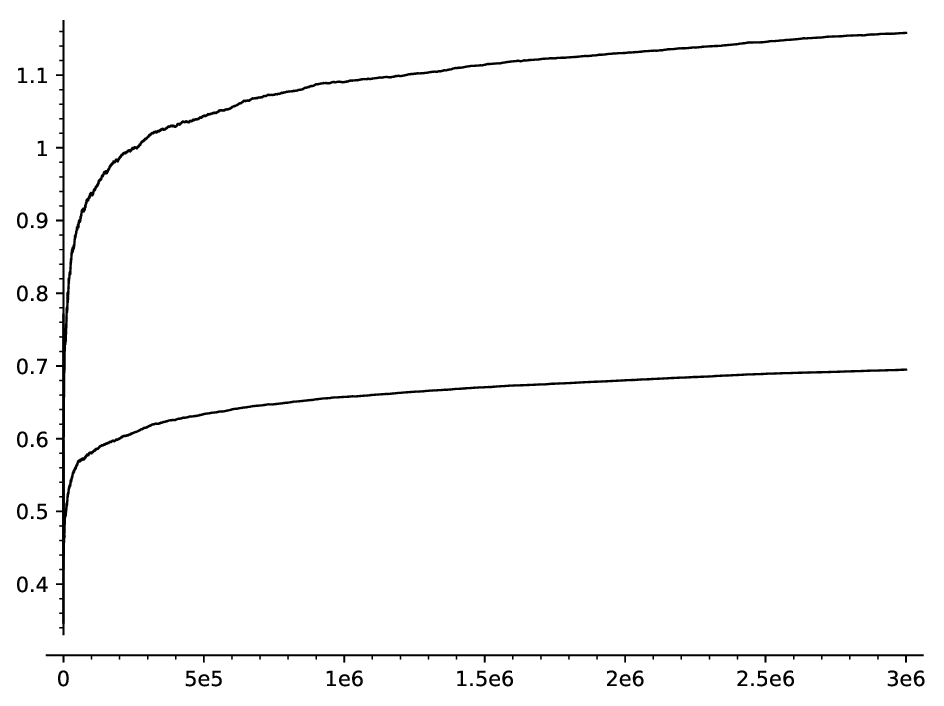}
\caption{17a1 $k = 3$: $m_{k,E}(X;c)/X^{c +1/2}$\\ Top to bottom $c =$ 0.3, 0.4} \label{fig:17_3_acc_c}
\end{subfigure}
\hspace*{-.7cm}
\begin{subfigure}[b]{0.4\linewidth}
\includegraphics[width=\linewidth]{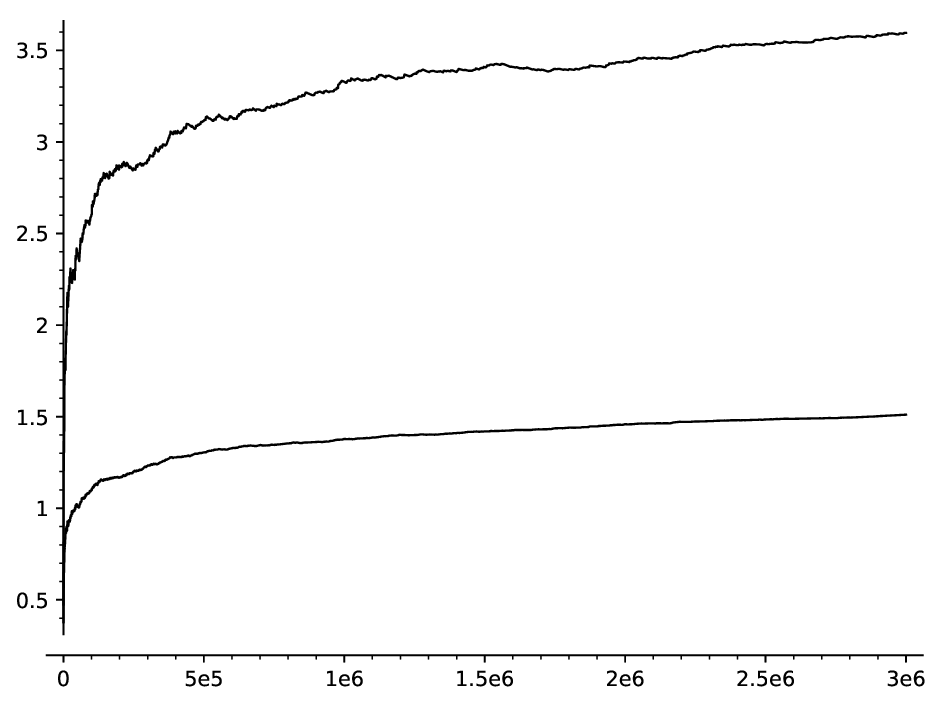}
\caption{15a1 $k = 5$: $m_{k,E}(X;c)/X^{c}\log(X)$\\ Top to bottom $c =$ 0.3, 0.5} \label{fig:15_5_acc_c}
\end{subfigure}\hspace*{\fill}
\begin{subfigure}[b]{0.4\linewidth}
\includegraphics[width=\linewidth]{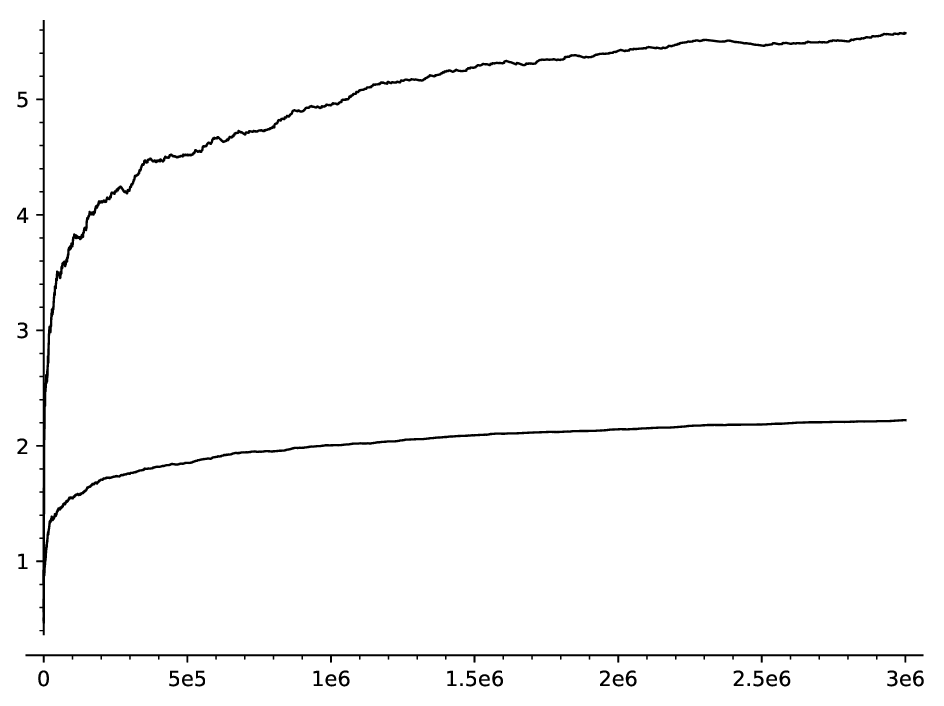}
\caption{17a1 $k = 5$: $m_{k,E}(X;c)/X^{c}\log(X)$\\ Top to bottom $c =$ 0.3, 0.5} \label{fig:17_5_acc_c}
\end{subfigure}
\hspace*{-.7cm}
\begin{subfigure}[b]{0.4\linewidth}
\includegraphics[width=\linewidth]{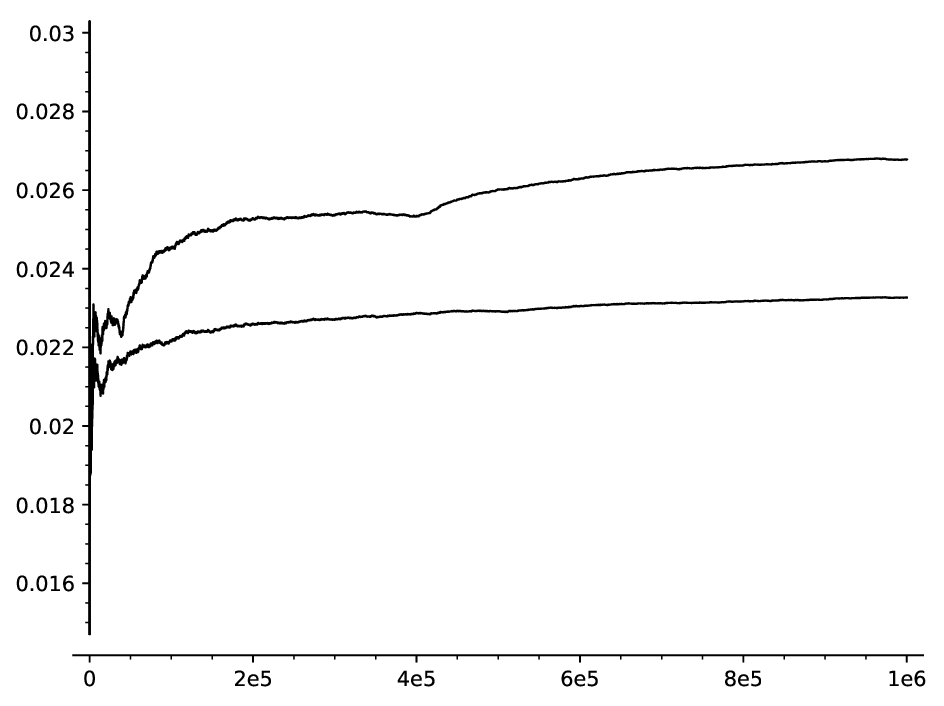}
\caption{15a1 $k = 6$: $m_{k,E}(X;c)/X^{c+1/2}\log^2(X)$\\ Top to bottom $c =$ 0.3, 0.4} \label{fig:15_6_acc_c}
\end{subfigure}\hspace*{\fill}
\begin{subfigure}[b]{0.4\linewidth}
\includegraphics[width=\linewidth]{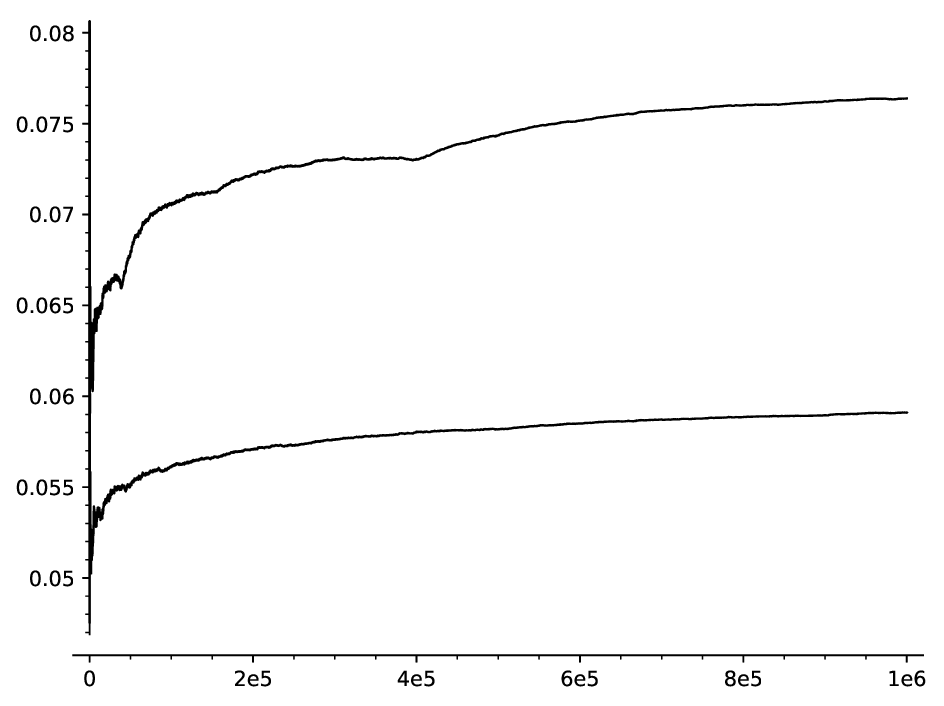}
\caption{17a1 $k = 6$: $m_{k,E}(X;c)/X^{c +1/2}\log^2(X)$\\ Top to bottom $c =$ 0.3, 0.4} \label{fig:17_6_acc_c}
\end{subfigure}
\caption{Ratio~\eqref{ratio_c} for $k = 3, 5, 6$ and $\phi(k)/4 -1 < c \le \phi(k)/4$} \label{fig:c_15_17_acc_3_5_6}
\end{figure}

\clearpage

\begin{figure}[t!] 
\hspace*{-.7cm}
\begin{subfigure}[b]{0.4\linewidth}
\includegraphics[width=\linewidth]{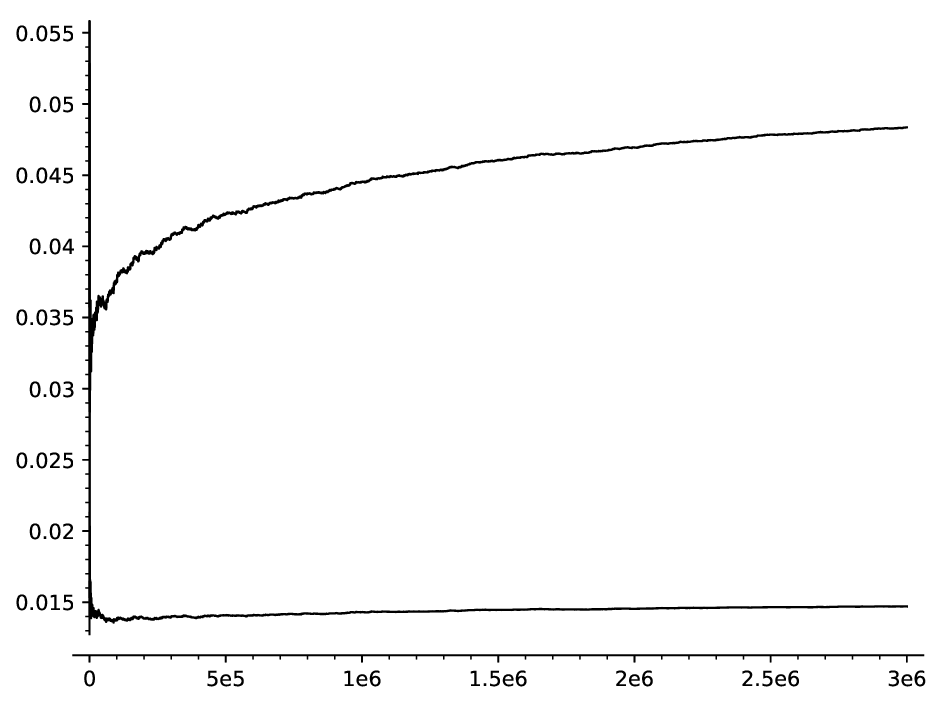}
\caption{15a1 $k = 7$: $m_{k,E}(X;c)/X^{c-1/2}\log^2(X)$\\ Top to bottom $c =$ 1.2, 1.3} \label{fig:15_7_acc_c}
\end{subfigure}\hspace*{\fill}
\begin{subfigure}[b]{0.4\linewidth}
\includegraphics[width=\linewidth]{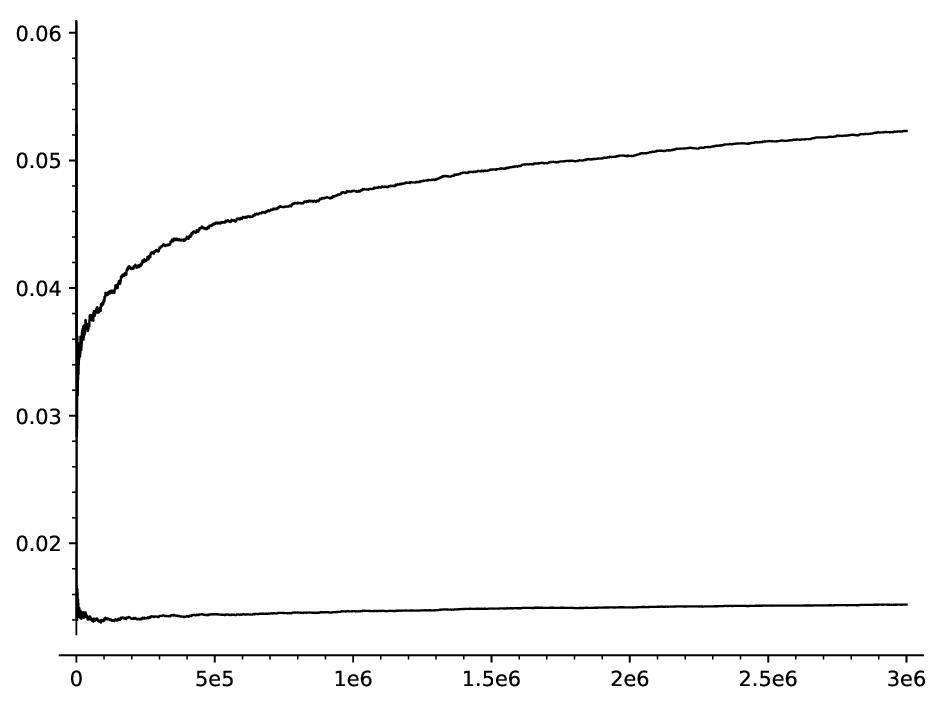}
\caption{17a1 $k = 7$: $m_{k,E}(X;c)/X^{c-1/2}\log^2(X)$\\ Top to bottom $c =$ 1.2, 1.3} \label{fig:17_7_acc_c}
\end{subfigure}
\hspace*{-.7cm}
\begin{subfigure}[b]{0.4\linewidth}
\includegraphics[width=\linewidth]{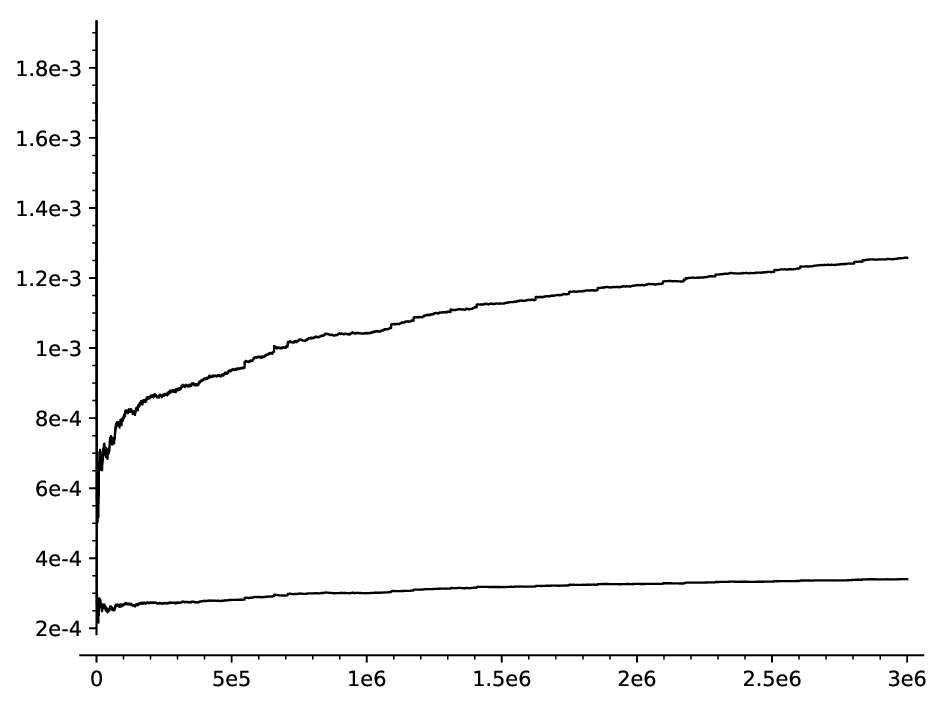}
\caption{15a1 $k = 13$: $m_{k,E}(X;c)/X^{c-2}\log^5(X)$\\ Top to bottom $c =$ 2.4, 2.5} \label{fig:15_13_acc_c}
\end{subfigure}\hspace*{\fill}
\begin{subfigure}[b]{0.4\linewidth}
\includegraphics[width=\linewidth]{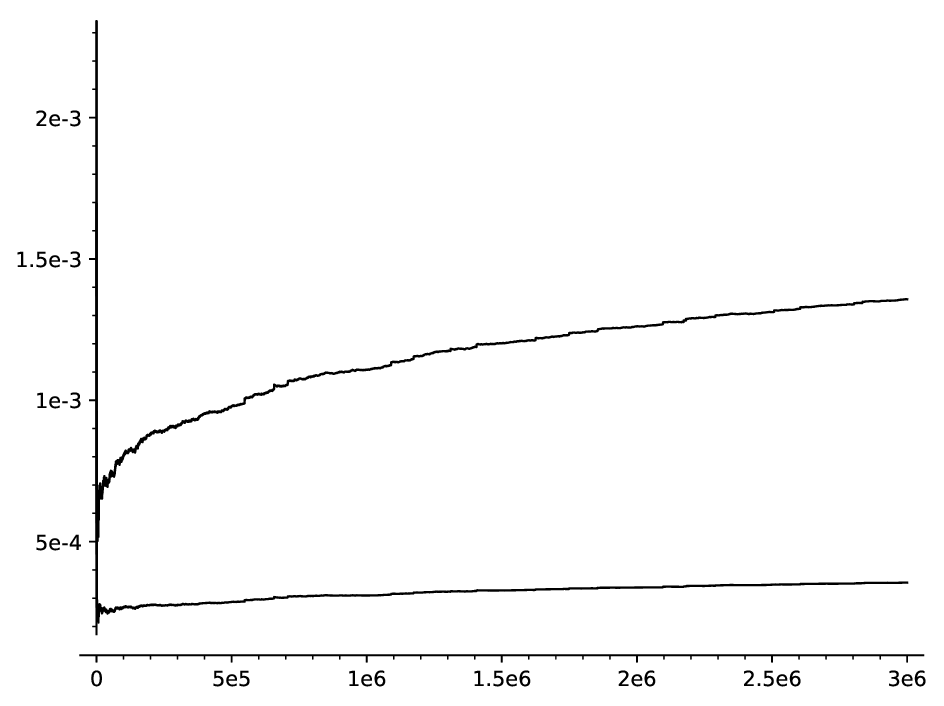}
\caption{17a1 $k = 13$: $m_{k,E}(X;c)/X^{c-2}\log^5(X)$\\ Top to bottom $c =$ 2.4, 2.5} \label{fig:17_13_acc_c}
\end{subfigure}
\caption{Ratio~\eqref{ratio_c} for $k = 7, 13$ and $\phi(k)/4 -1 \le c \le \phi(k)/4$} \label{fig:c_15_17_acc_7_13}
\end{figure}

\clearpage

\begin{figure}[t!] 
\hspace*{-.7cm}
\begin{subfigure}[b]{0.4\linewidth}
\includegraphics[width=\linewidth]{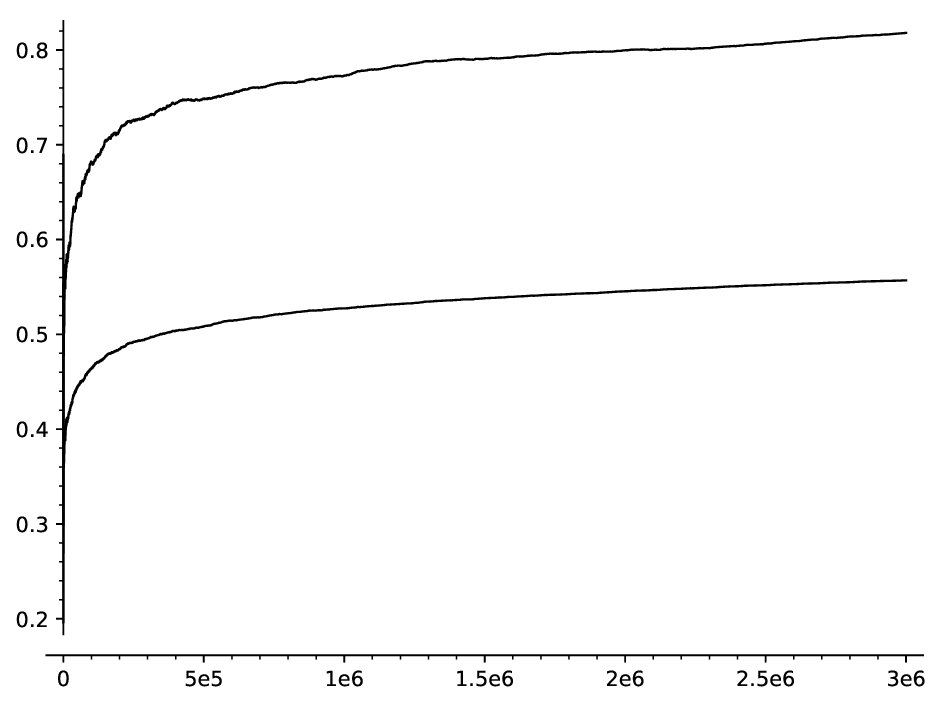}
\caption{19a1 $k = 3$: $m_{k,E}(X;c)/X^{c +1/2}$\\ Top to bottom $c =$ 0.3, 0.4} \label{fig:19_3_acc_c}
\end{subfigure}\hspace*{\fill}
\begin{subfigure}[b]{0.4\linewidth}
\includegraphics[width=\linewidth]{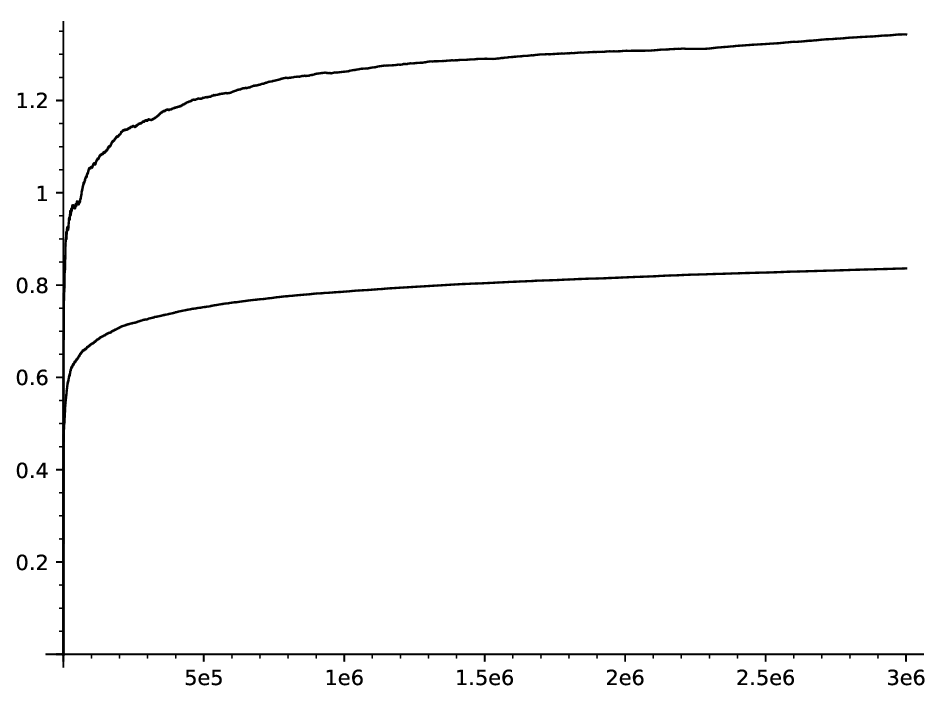}
\caption{37b1 $k = 3$: $m_{k,E}(X;c)/X^{c +1/2}$\\ Top to bottom $c =$ 0.3, 0.4} \label{fig:37_3_acc_c}
\end{subfigure}
\hspace*{-.7cm}
\begin{subfigure}[b]{0.4\linewidth}
\includegraphics[width=\linewidth]{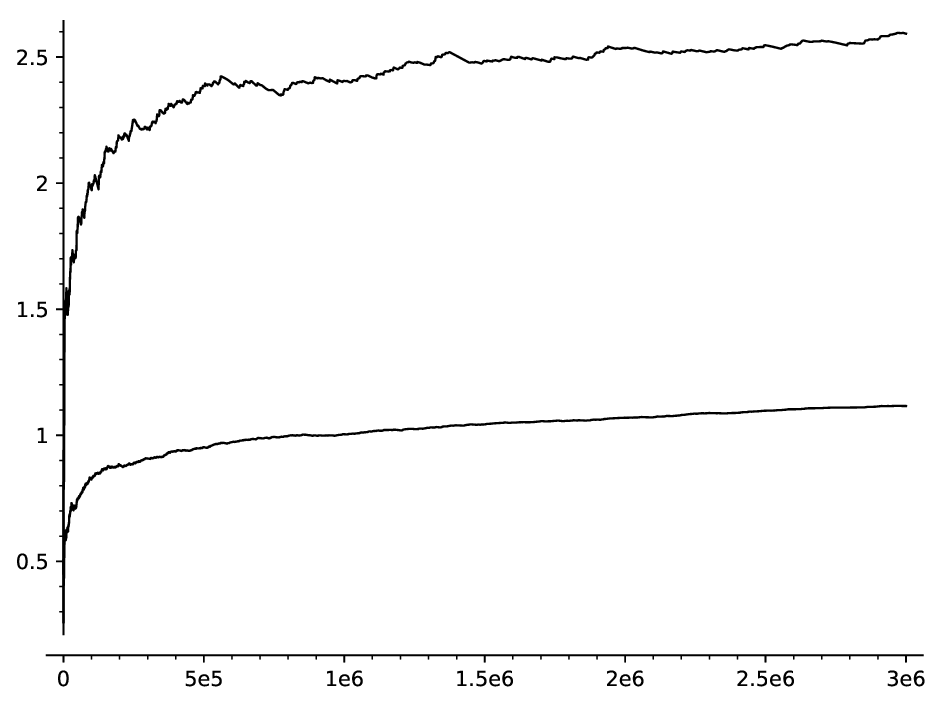}
\caption{19a1 $k = 5$: $m_{k,E}(X;c)/X^{c}\log(X)$\\ Top to bottom $c =$ 0.3, 0.5} \label{fig:19_5_acc_c}
\end{subfigure}\hspace*{\fill}
\begin{subfigure}[b]{0.4\linewidth}
\includegraphics[width=\linewidth]{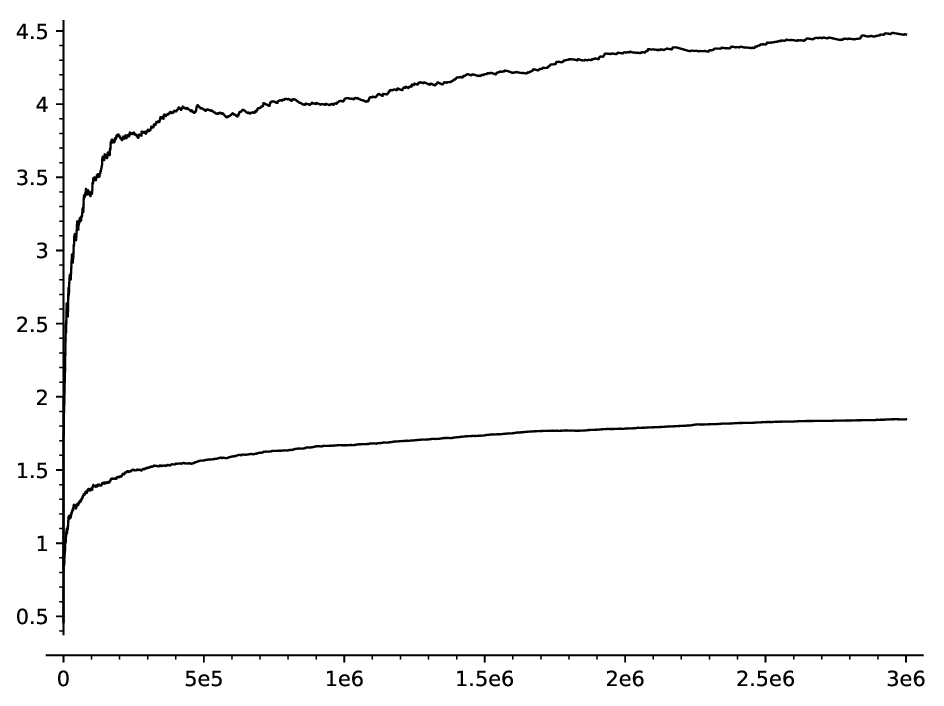}
\caption{37b1 $k = 5$: $m_{k,E}(X;c)/X^{c}\log(X)$\\ Top to bottom $c =$ 0.3, 0.5} \label{fig:37_5_acc_c}
\end{subfigure}
\hspace*{-.7cm}
\begin{subfigure}[b]{0.4\linewidth}
\includegraphics[width=\linewidth]{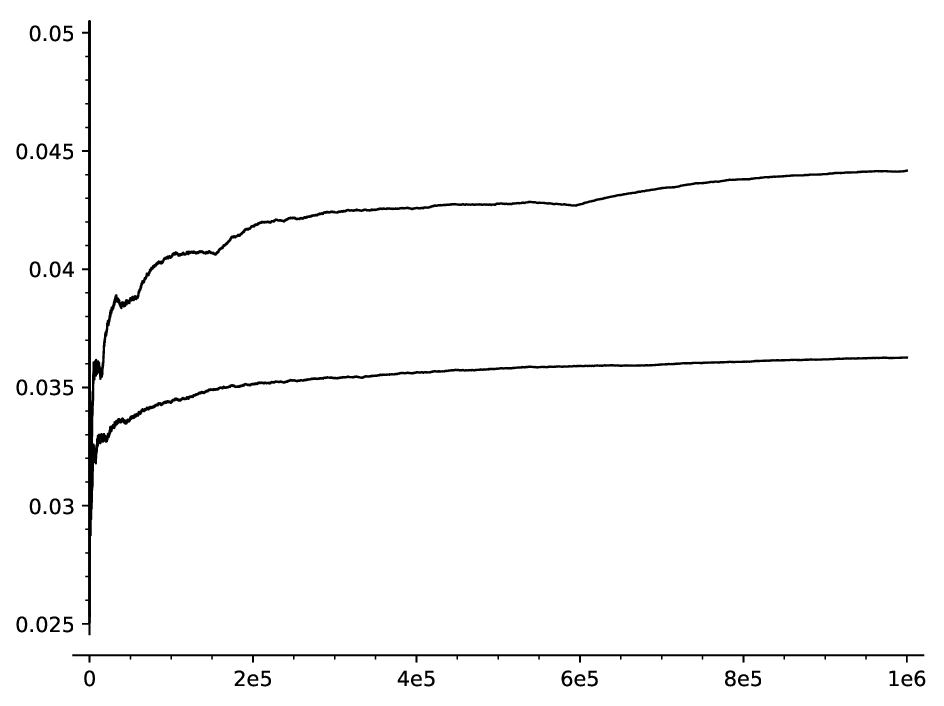}
\caption{19a1 $k = 6$: $m_{k,E}(X;c)/X^{c+1/2}\log^2(X)$\\ Top to bottom $c =$ 0.3, 0.4} \label{fig:19_6_acc_c}
\end{subfigure}\hspace*{\fill}
\begin{subfigure}[b]{0.4\linewidth}
\includegraphics[width=\linewidth]{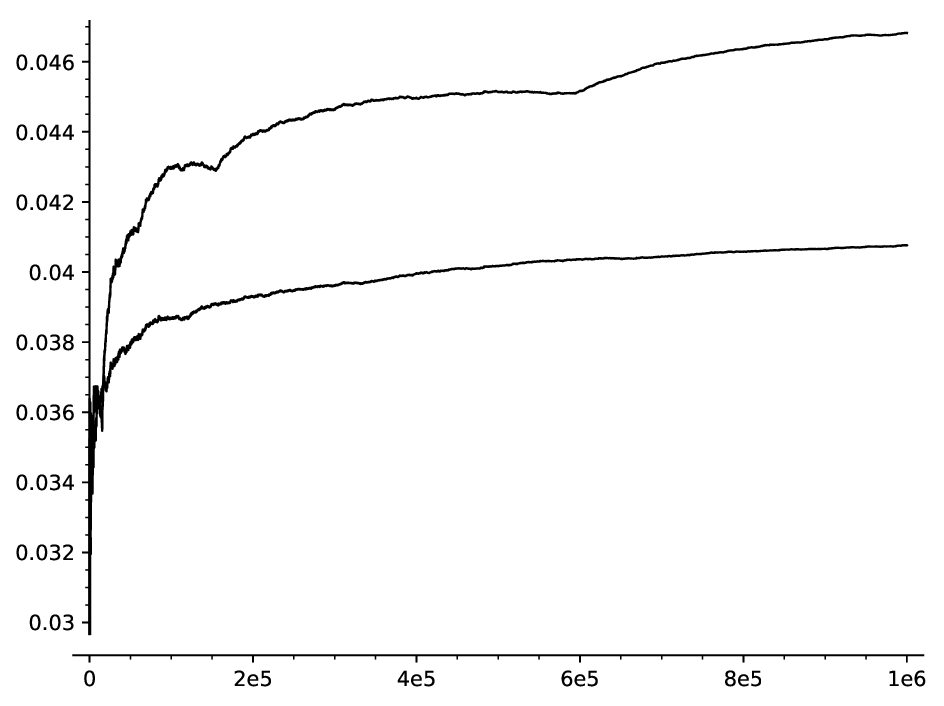}
\caption{37b1 $k = 6$: $m_{k,E}(X;c)/X^{c +1/2}\log^2(X)$\\ Top to bottom $c =$ 0.3, 0.4} \label{fig:37_6_acc_c}
\end{subfigure}
\caption{Ratio~\eqref{ratio_c} for $k = 3, 5, 6$ and $\phi(k)/4 -1 < c \le \phi(k)/4$} \label{fig:c_19_37_acc_3_5_6}
\end{figure}

\clearpage

\begin{figure}[t!] 
\hspace*{-.7cm}
\begin{subfigure}[b]{0.4\linewidth}
\includegraphics[width=\linewidth]{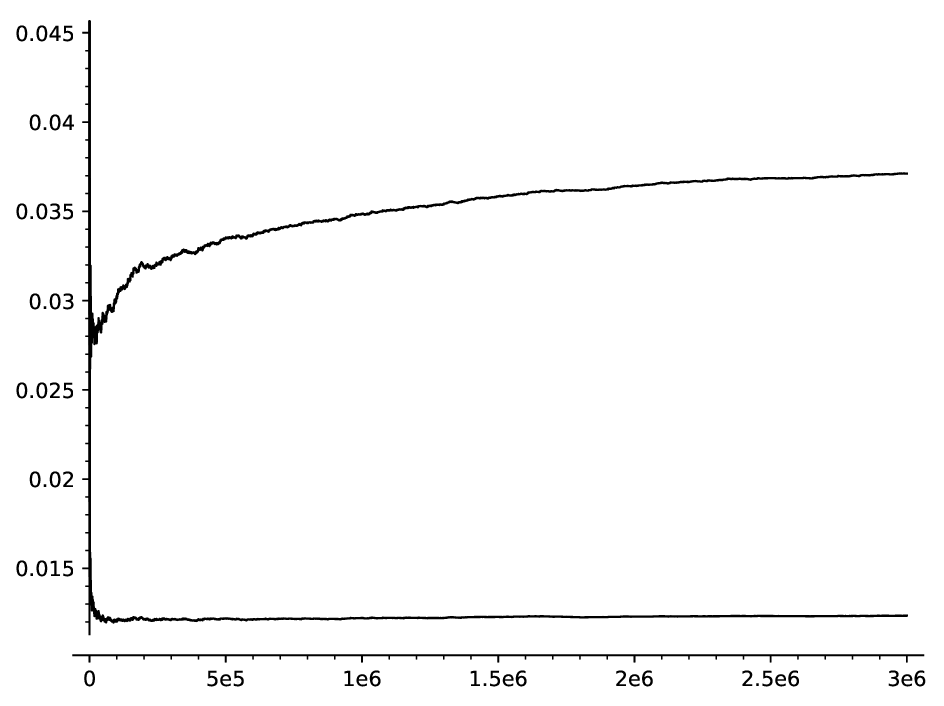}
\caption{19a1 $k = 7$: $m_{k,E}(X;c)/X^{c-1/2}\log^2(X)$\\ Top to bottom $c =$ 1.2, 1.3} \label{fig:19_7_acc_c}
\end{subfigure}\hspace*{\fill}
\begin{subfigure}[b]{0.4\linewidth}
\includegraphics[width=\linewidth]{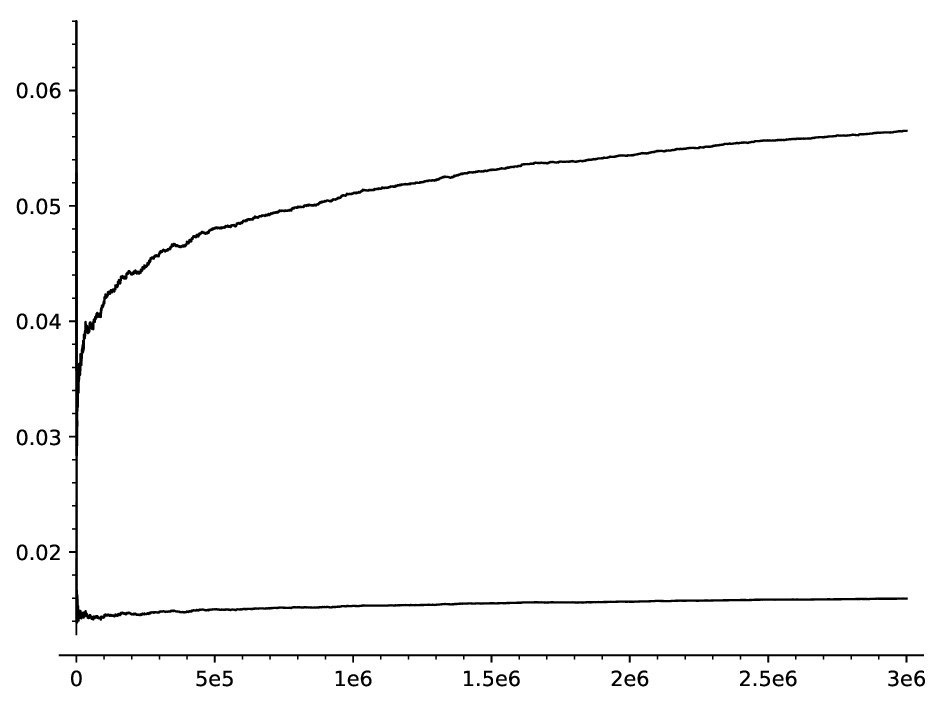}
\caption{37b1 $k = 7$: $m_{k,E}(X;c)/X^{c-1/2}\log^2(X)$\\ Top to bottom $c =$ 1.2, 1.3} \label{fig:37_7_acc_c}
\end{subfigure}
\hspace*{-.7cm}
\begin{subfigure}[b]{0.4\linewidth}
\includegraphics[width=\linewidth]{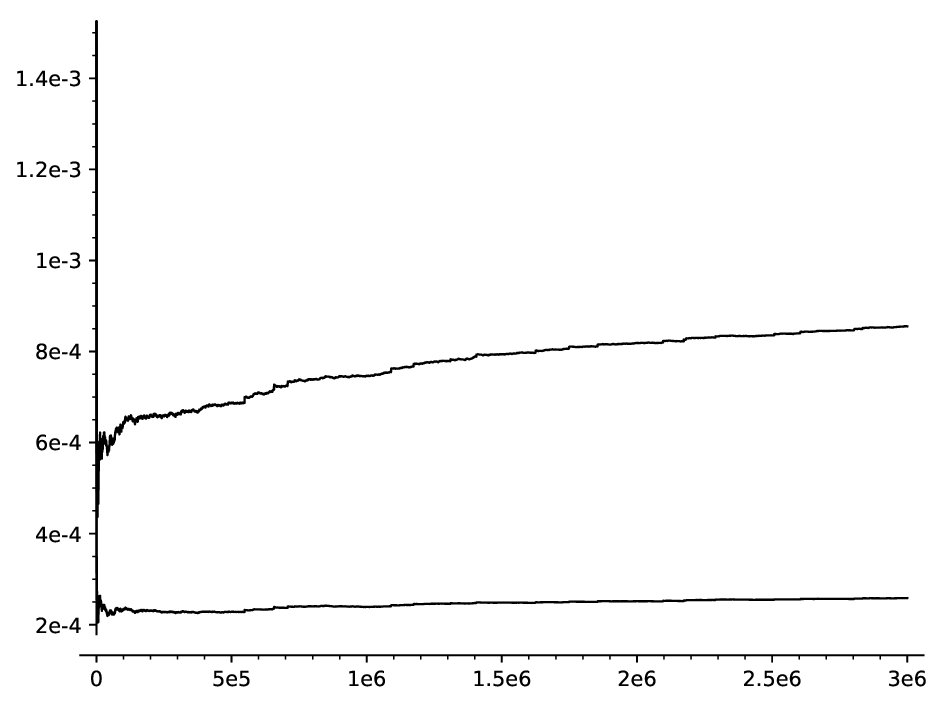}
\caption{19a1 $k = 13$: $m_{k,E}(X;c)/X^{c-2}\log^5(X)$\\ Top to bottom $c =$ 2.4, 2.5} \label{fig:19_13_acc_c}
\end{subfigure}\hspace*{\fill}
\begin{subfigure}[b]{0.4\linewidth}
\includegraphics[width=\linewidth]{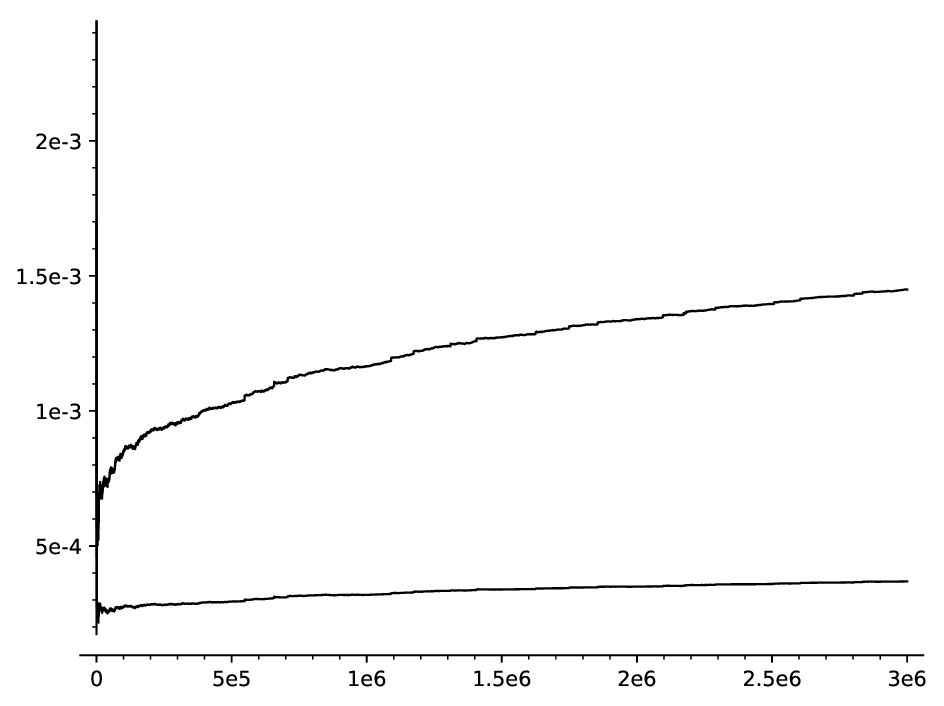}
\caption{37b1 $k = 13$: $m_{k,E}(X;c)/X^{c-2}\log^5(X)$\\ Top to bottom $c =$ 2.4, 2.5} \label{fig:37_13_acc_c}
\end{subfigure}
\caption{Ratio~\eqref{ratio_c} for $k = 7, 13$ and $\phi(k)/4 -1 \le c \le \phi(k)/4$} \label{fig:c_19_37_acc_7_13}
\end{figure}

\clearpage

\begin{figure}[t!] 
\hspace*{-2.3cm}
\begin{subfigure}[b]{0.43\linewidth}
\includegraphics[width=\linewidth]{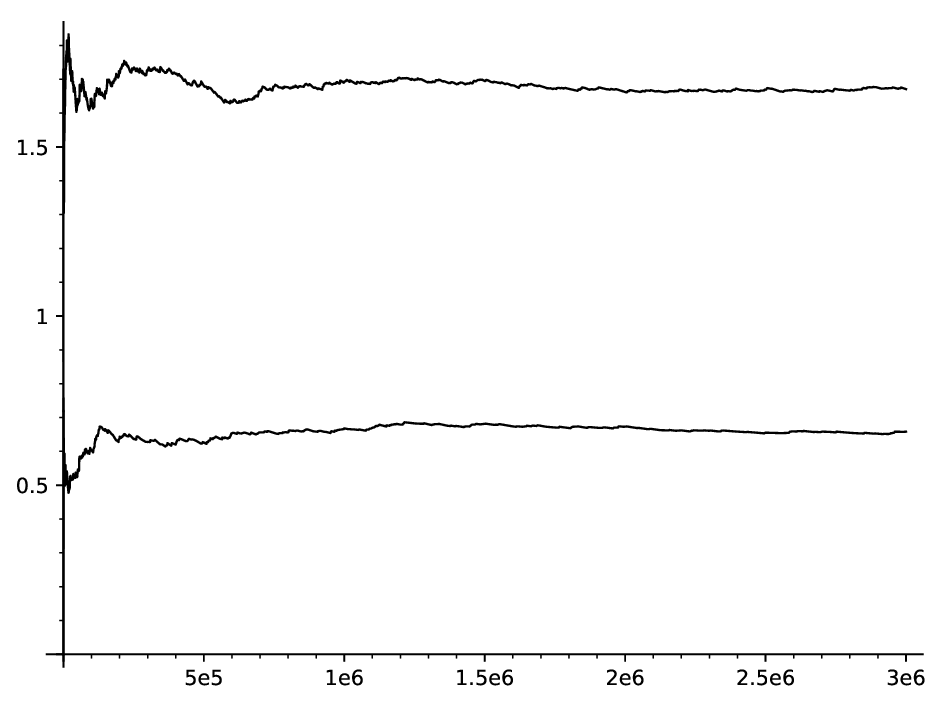}
\caption{$|l| = 1$: Top 1 bottom -1} \label{fig:11_3_A_1}
\end{subfigure}\hspace*{\fill}
\begin{subfigure}[b]{0.43\linewidth}
\includegraphics[width=\linewidth]{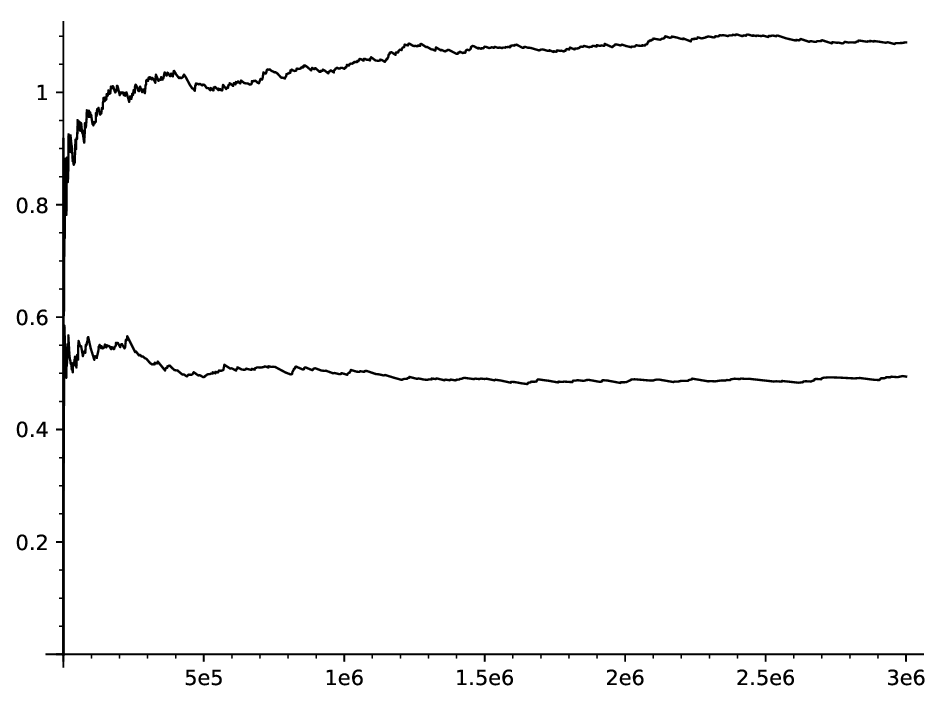}
\caption{$|l| = 2$: Top -2 bottom 2} \label{fig:11_3_A_2}
\end{subfigure}\hspace*{\fill}
\begin{subfigure}[b]{0.43\linewidth}
\includegraphics[width=\linewidth]{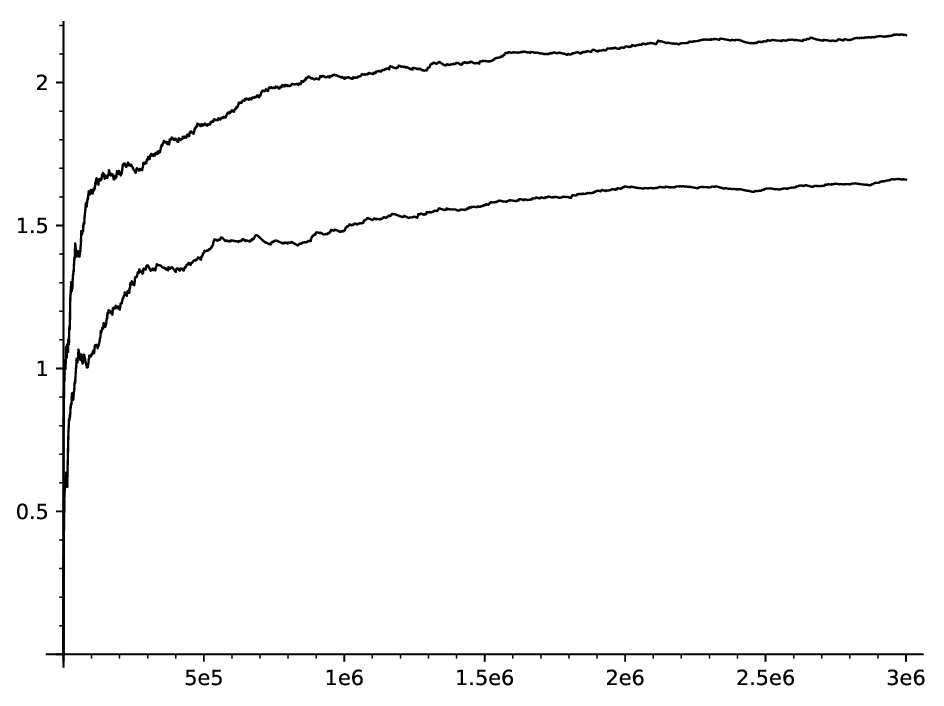}
\caption{$|l| = 3$: Top 3 bottom -3} \label{fig:11_3_A_3}
\end{subfigure}
\hspace*{-2.3cm}
\begin{subfigure}[b]{0.43\linewidth}
\includegraphics[width=\linewidth]{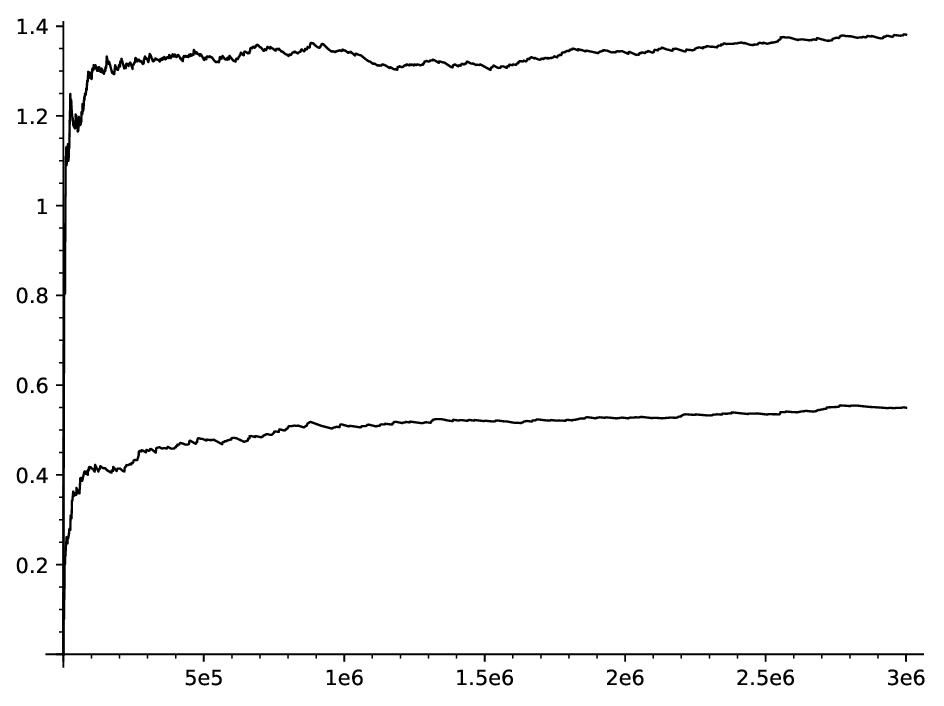}
\caption{$|l| = 4$: Top 4 bottom -4} \label{fig:11_3_A_4}
\end{subfigure}\hspace*{\fill}
\begin{subfigure}[b]{0.43\linewidth}
\includegraphics[width=\linewidth]{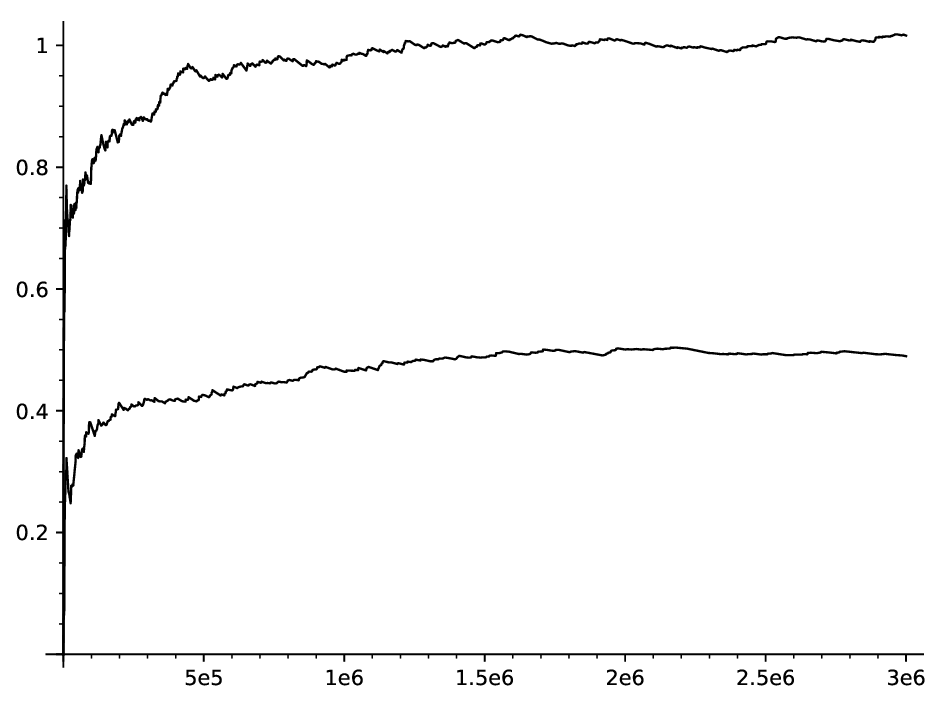}
\caption{$|l| = 5$: Top -5 bottom 5} \label{fig:11_3_A_5}
\end{subfigure}\hspace*{\fill}
\begin{subfigure}[b]{0.43\linewidth}
\includegraphics[width=\linewidth]{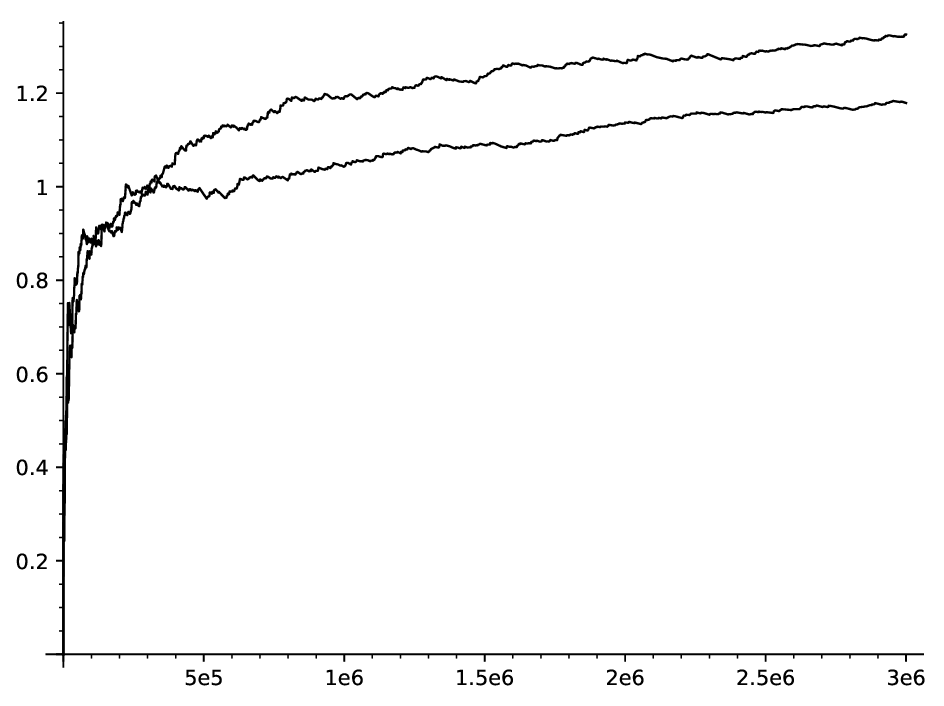}
\caption{$|l| = 6$: Top -6 bottom 6} \label{fig:11_3_A_6}
\end{subfigure}
\hspace*{-2.3cm}
\begin{subfigure}[b]{0.43\linewidth}
\includegraphics[width=\linewidth]{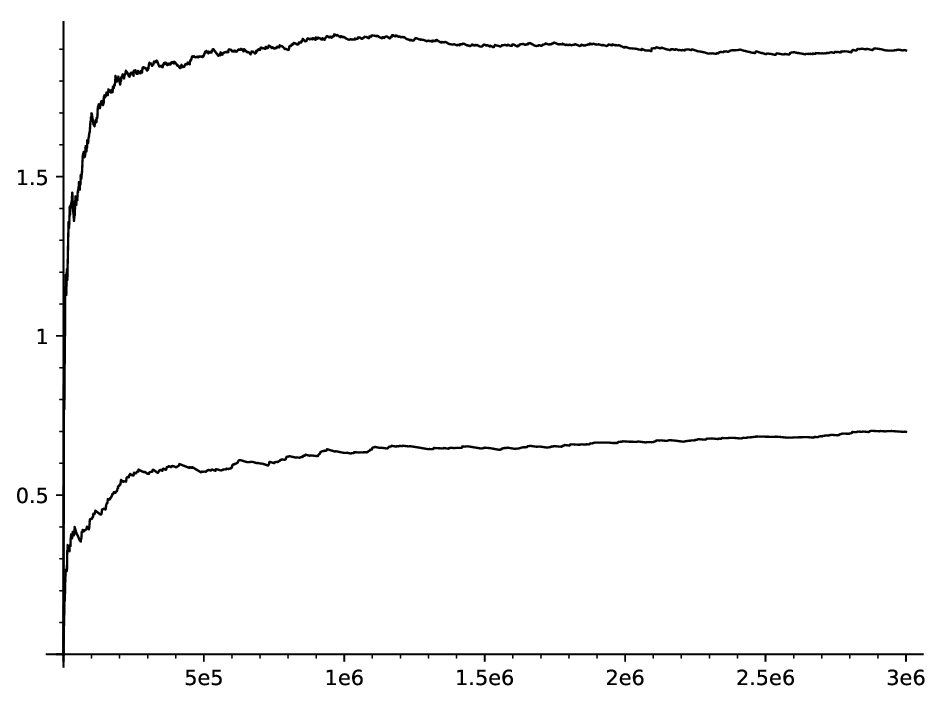}
\caption{$|l| = 7$: Top 7 bottom -7} \label{fig:11_3_A_7}
\end{subfigure}\hspace*{\fill}
\begin{subfigure}[b]{0.43\linewidth}
\includegraphics[width=\linewidth]{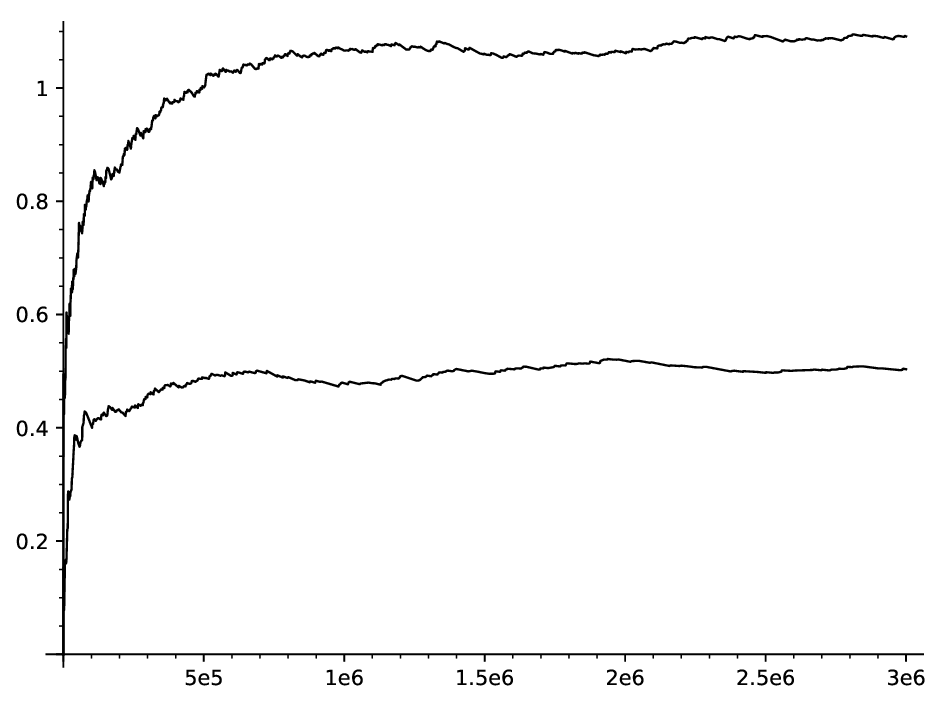}
\caption{$|l| = 8$: Top -8 bottom 8} \label{fig:11_3_A_8}
\end{subfigure}\hspace*{\fill}
\begin{subfigure}[b]{0.43\linewidth}
\includegraphics[width=\linewidth]{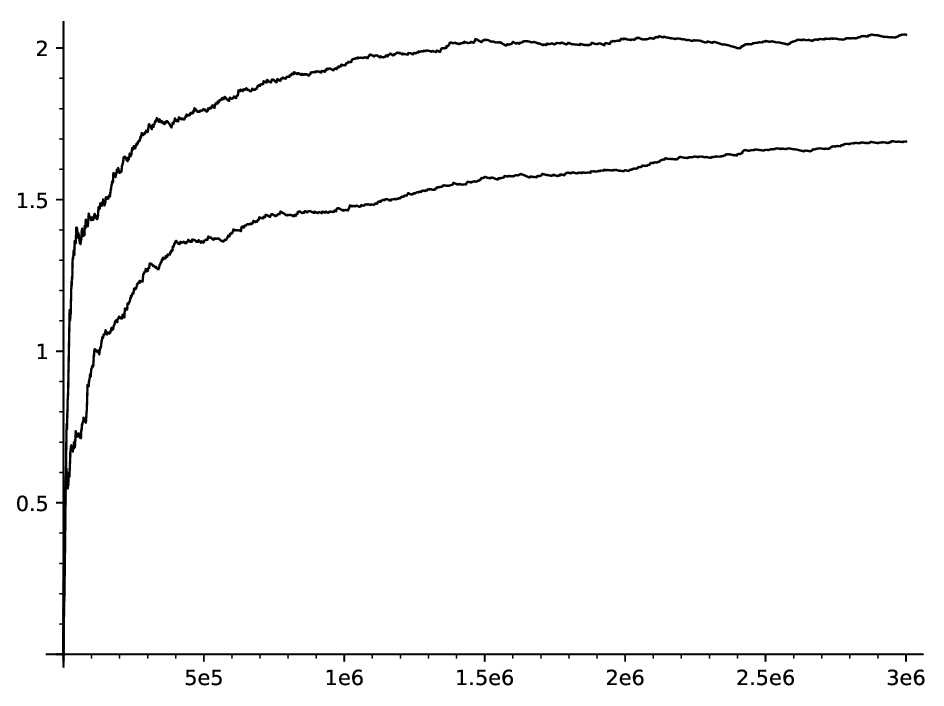}
\caption{$|l| = 9$: Top 9 bottom -9} \label{fig:11_3_A_9}
\end{subfigure}
\caption{Ratio~\eqref{ratio_A_exact} 11a1: $x(X;l)/X^{1/2}$ for $k = 3$} \label{fig:11a1_3_A_exact}
\end{figure}

\clearpage

\begin{figure}[t!] 
\hspace*{-2.3cm}
\begin{subfigure}[b]{0.43\linewidth}
\includegraphics[width=\linewidth]{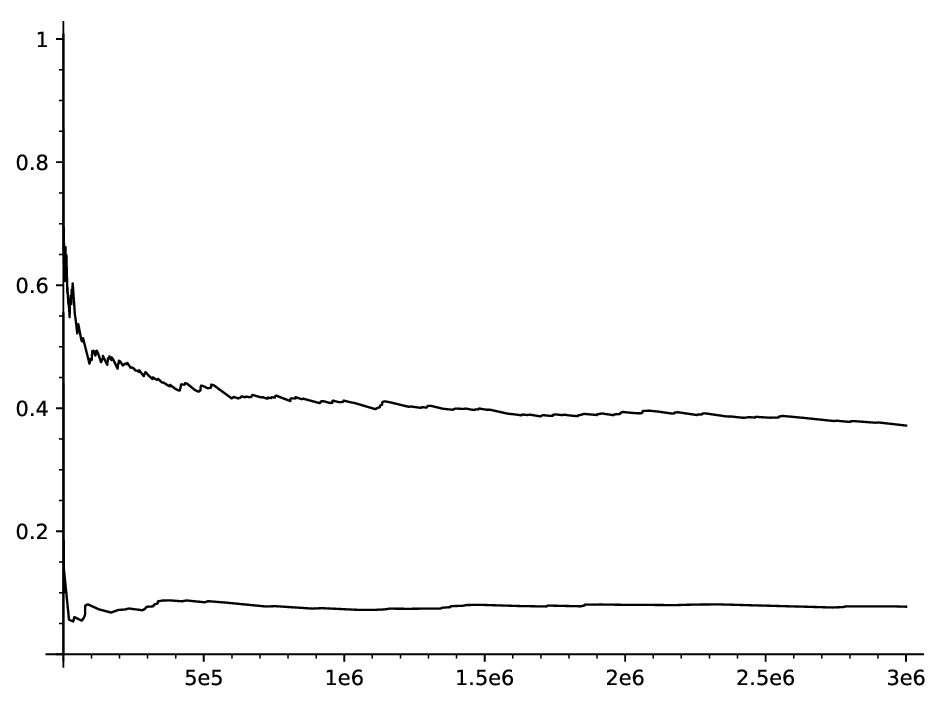}
\caption{$|l| = 1$: Top 1 bottom -1} \label{fig:14_3_A_1}
\end{subfigure}\hspace*{\fill}
\begin{subfigure}[b]{0.43\linewidth}
\includegraphics[width=\linewidth]{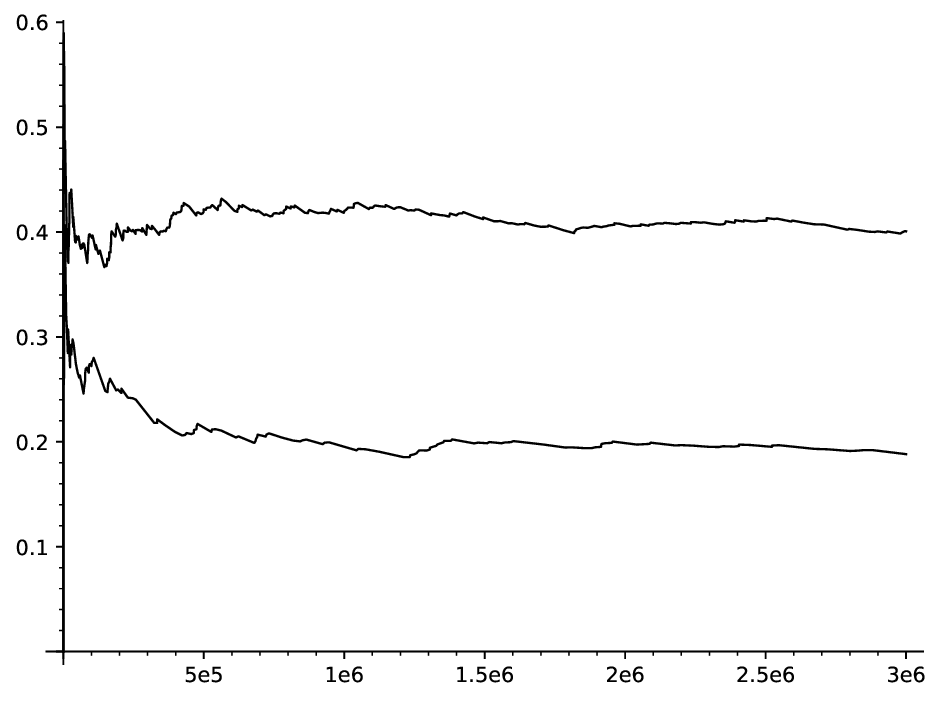}
\caption{$|l| = 2$: Top 2 bottom -2} \label{fig:14_3_A_2}
\end{subfigure}\hspace*{\fill}
\begin{subfigure}[b]{0.43\linewidth}
\includegraphics[width=\linewidth]{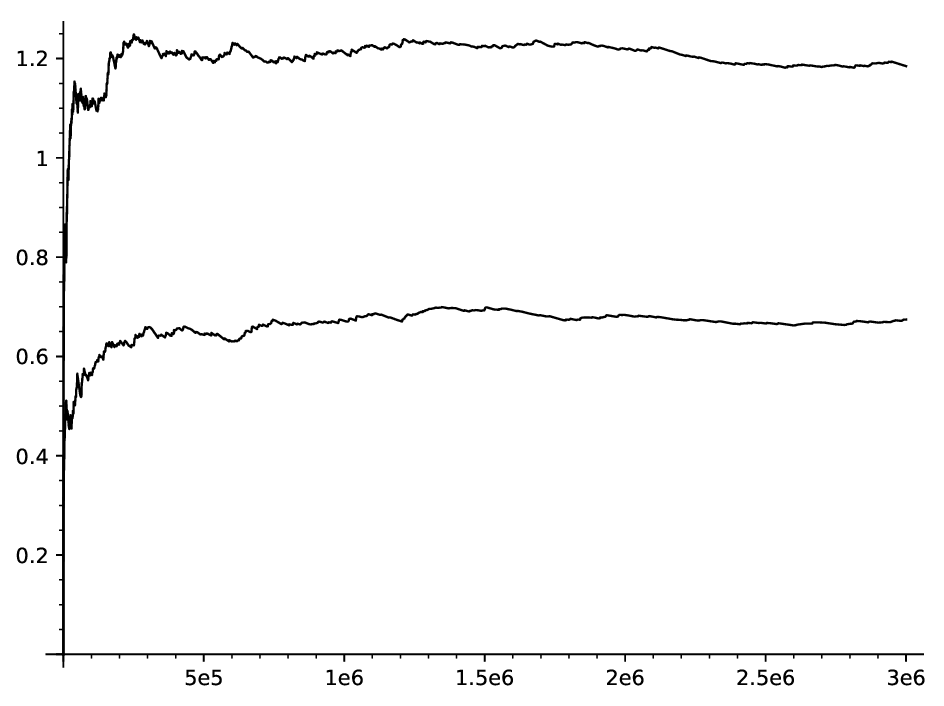}
\caption{$|l| = 3$: Top 3 bottom -3} \label{fig:14_3_A_3}
\end{subfigure}
\hspace*{-2.3cm}
\begin{subfigure}[b]{0.43\linewidth}
\includegraphics[width=\linewidth]{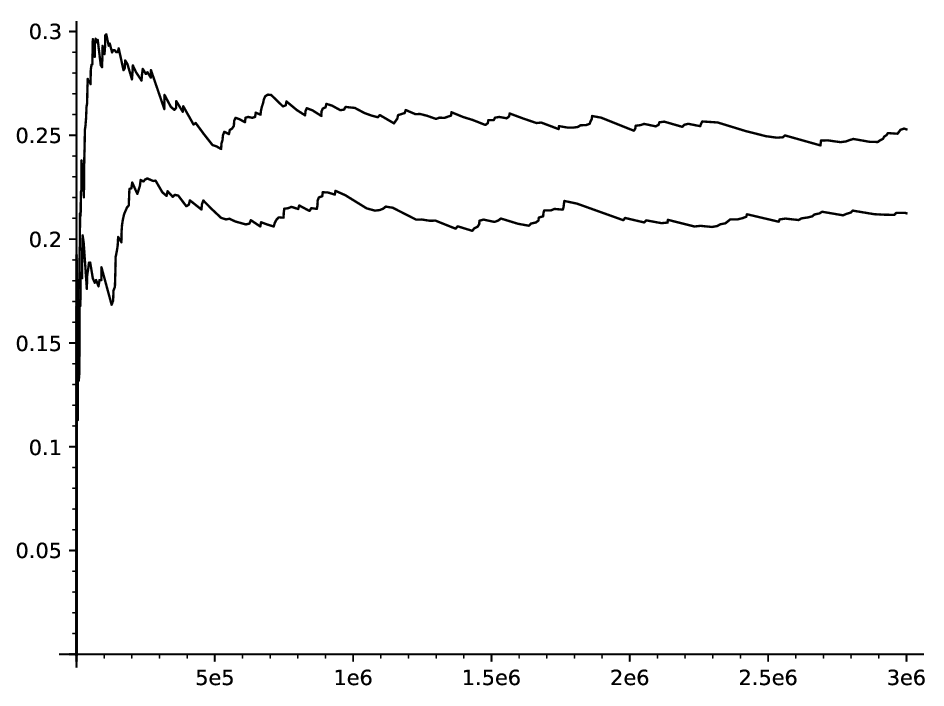}
\caption{$|l| = 4$: Top 4 bottom -4} \label{fig:14_3_A_4}
\end{subfigure}\hspace*{\fill}
\begin{subfigure}[b]{0.43\linewidth}
\includegraphics[width=\linewidth]{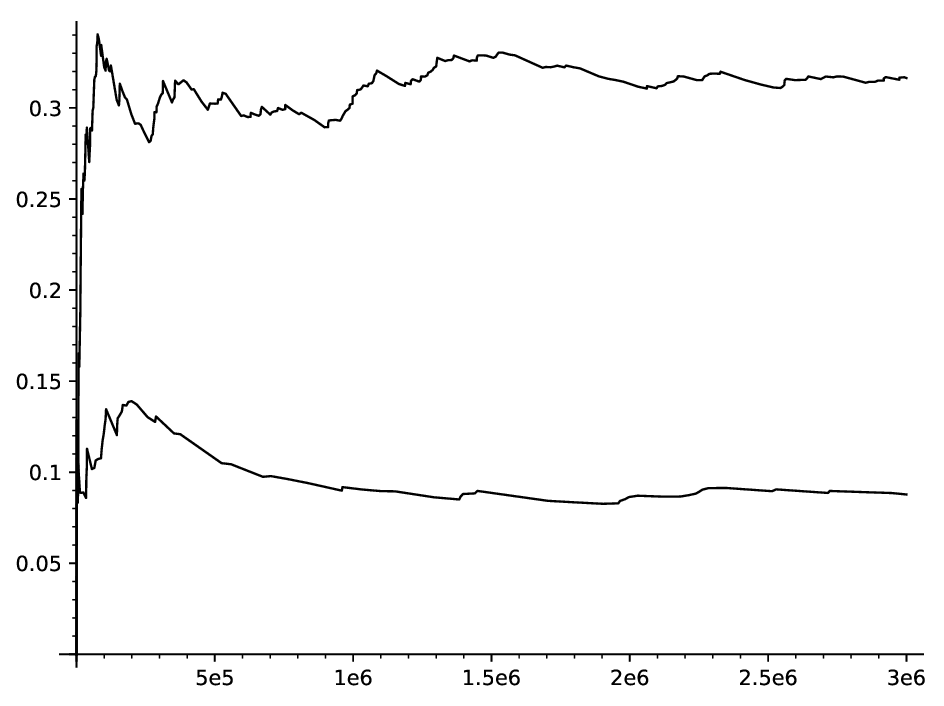}
\caption{$|l| = 5$: Top -5 bottom 5} \label{fig:14_3_A_5}
\end{subfigure}\hspace*{\fill}
\begin{subfigure}[b]{0.43\linewidth}
\includegraphics[width=\linewidth]{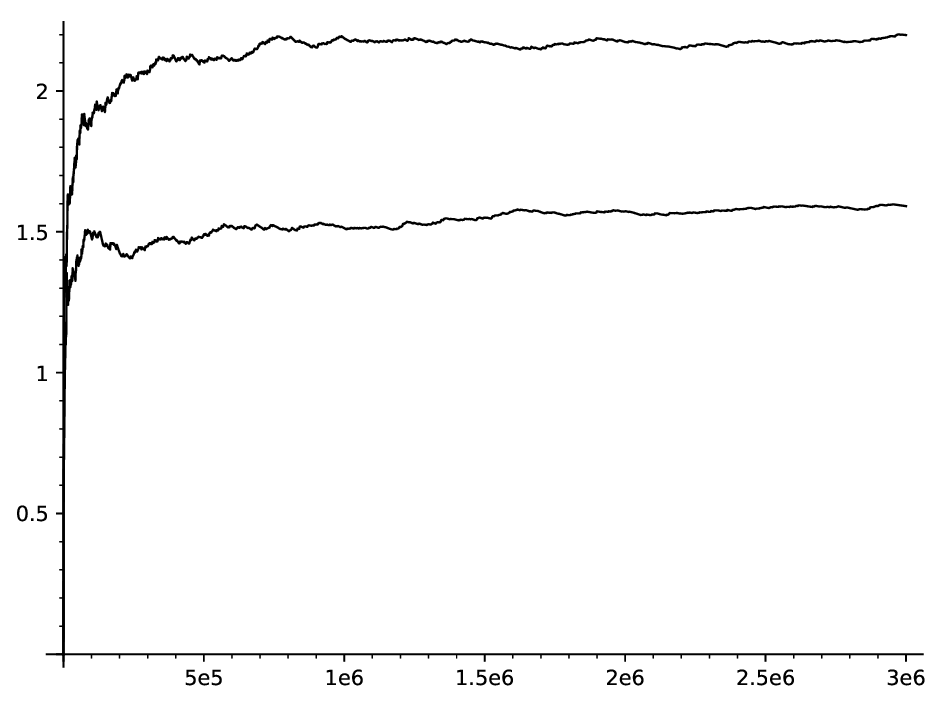}
\caption{$|l| = 6$: Top 6 bottom -6} \label{fig:14_3_A_6}
\end{subfigure}
\hspace*{-2.3cm}
\begin{subfigure}[b]{0.43\linewidth}
\includegraphics[width=\linewidth]{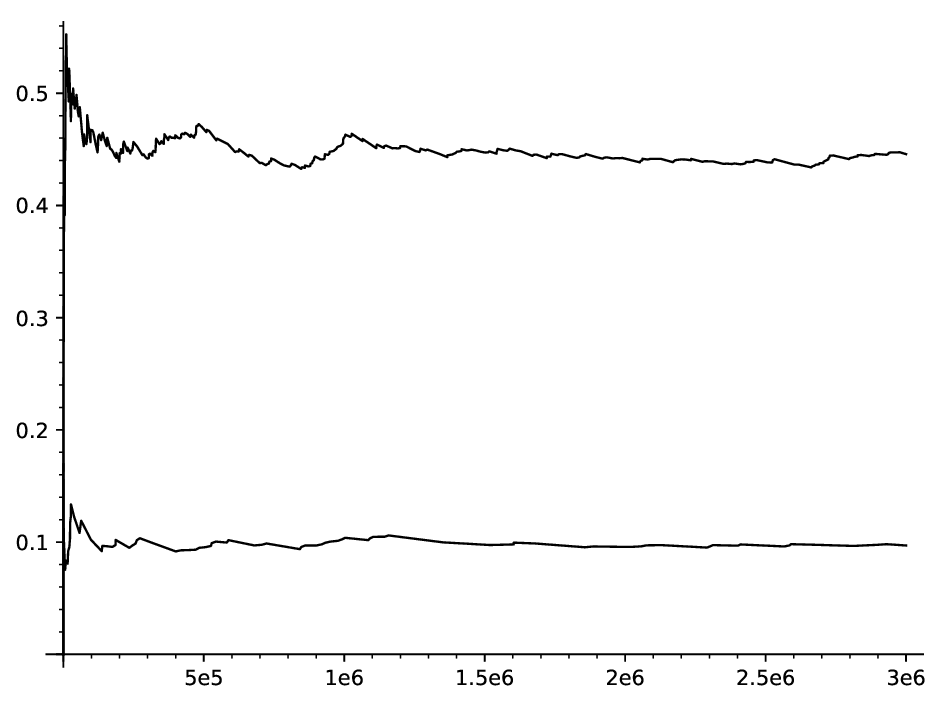}
\caption{$|l| = 7$: Top 7 bottom -7} \label{fig:14_3_A_7}
\end{subfigure}\hspace*{\fill}
\begin{subfigure}[b]{0.43\linewidth}
\includegraphics[width=\linewidth]{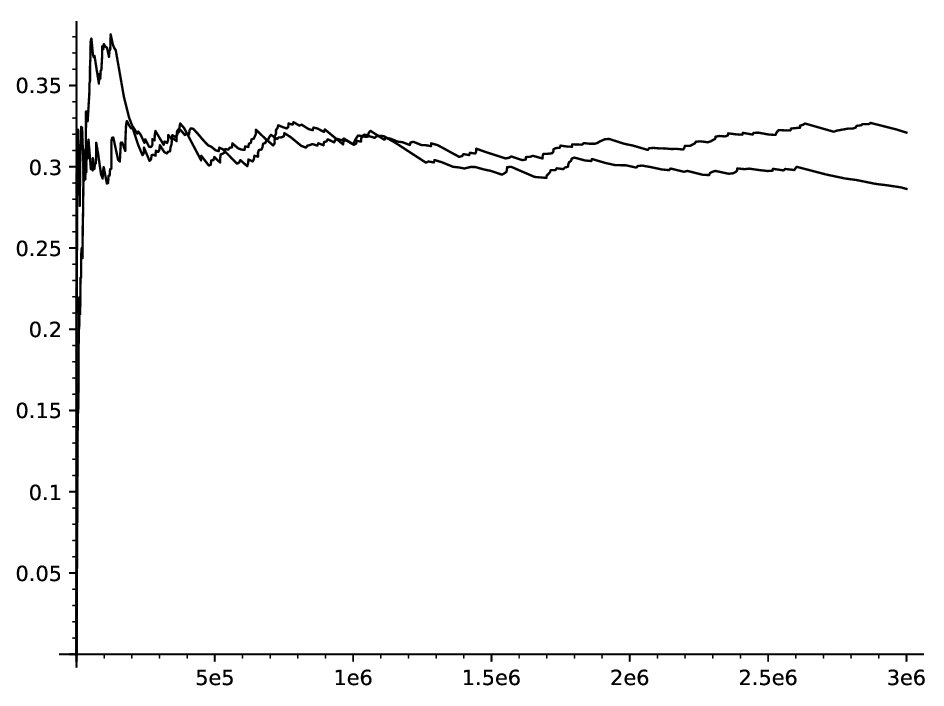}
\caption{$|l| = 8$: Top 8 bottom -8} \label{fig:14_3_A_8}
\end{subfigure}\hspace*{\fill}
\begin{subfigure}[b]{0.43\linewidth}
\includegraphics[width=\linewidth]{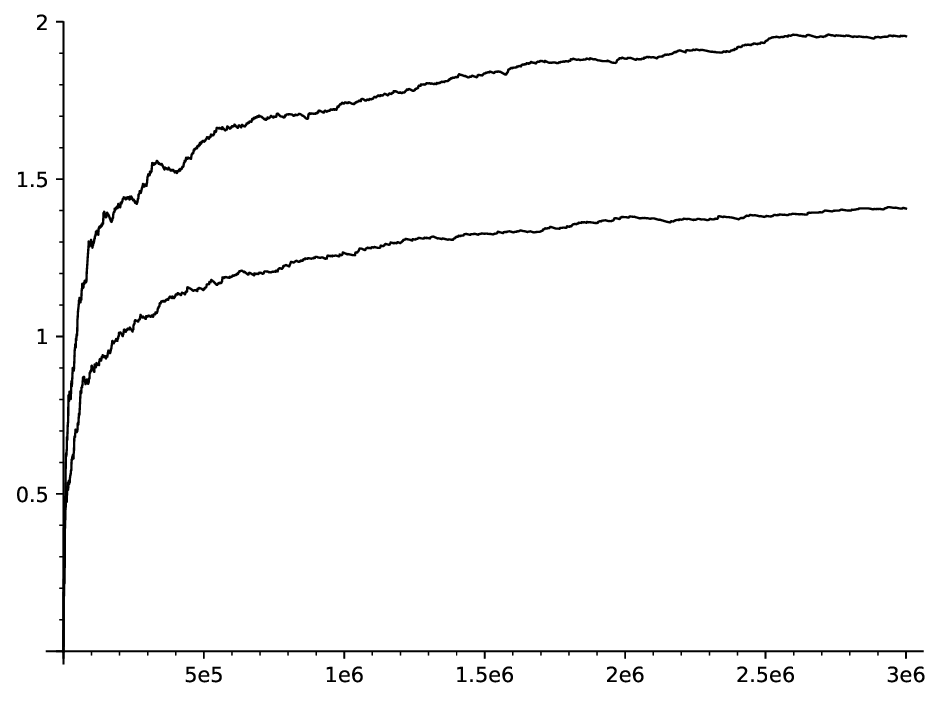}
\caption{$|l| = 9$: Top 9 bottom -9} \label{fig:14_3_A_9}
\end{subfigure}
\caption{Ratio~\eqref{ratio_A_exact} 14a1: $x(X;l)/X^{1/2}$ for $k = 3$} \label{fig:14a1_3_A_exact}
\end{figure}

\clearpage

\begin{figure}[t!] 
\hspace*{-2.3cm}
\begin{subfigure}[b]{0.43\linewidth}
\includegraphics[width=\linewidth]{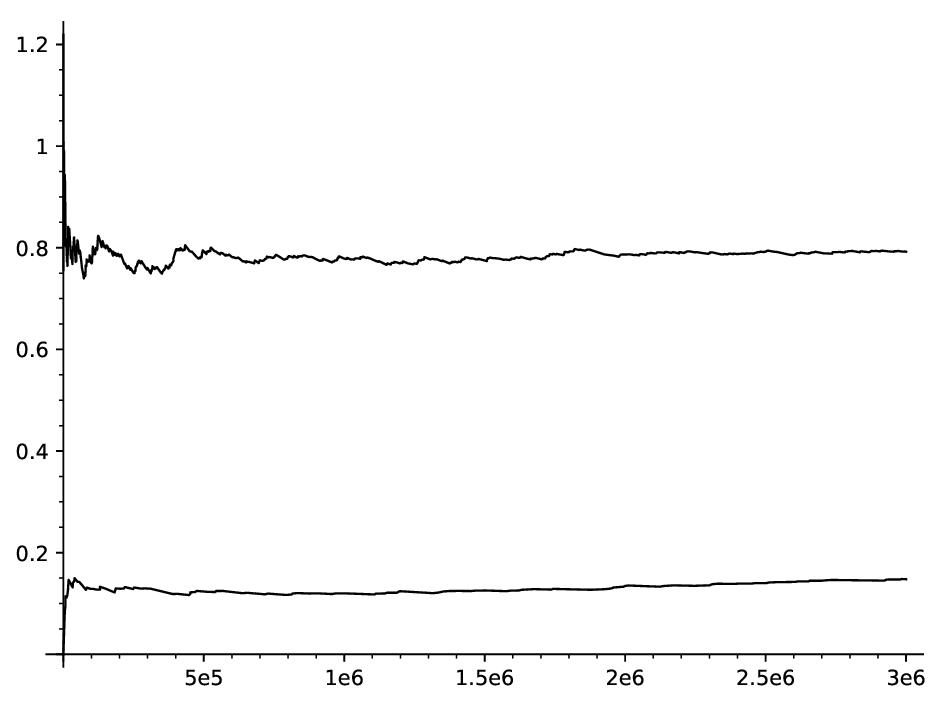}
\caption{$|l| = 1$: Top 1 bottom -1} \label{fig:15_3_A_1}
\end{subfigure}\hspace*{\fill}
\begin{subfigure}[b]{0.43\linewidth}
\includegraphics[width=\linewidth]{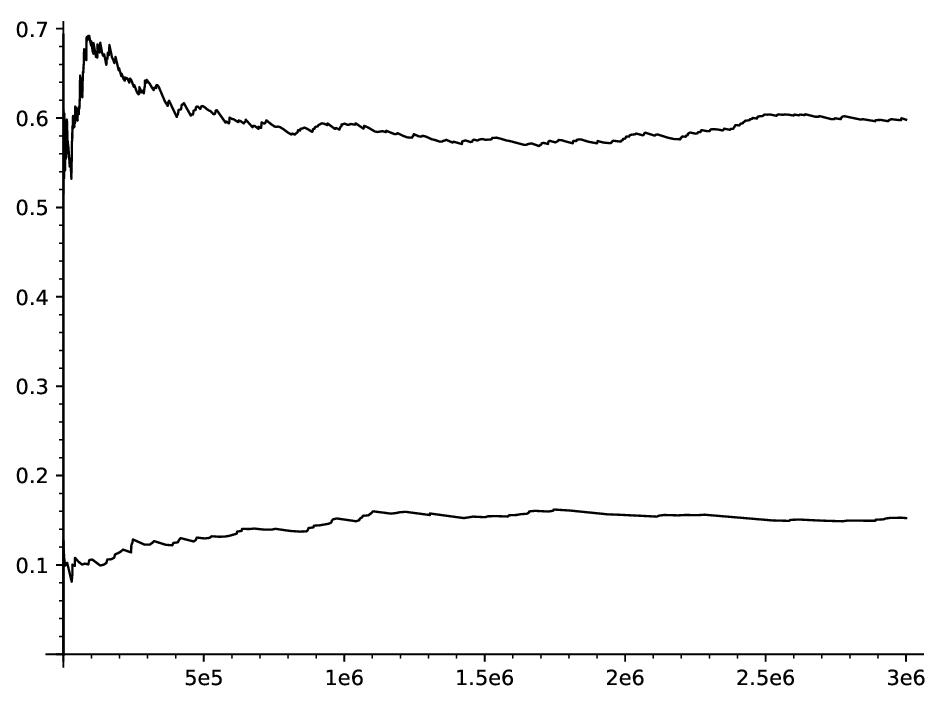}
\caption{$|l| = 2$: Top -2 bottom 2} \label{fig:15_3_A_2}
\end{subfigure}\hspace*{\fill}
\begin{subfigure}[b]{0.43\linewidth}
\includegraphics[width=\linewidth]{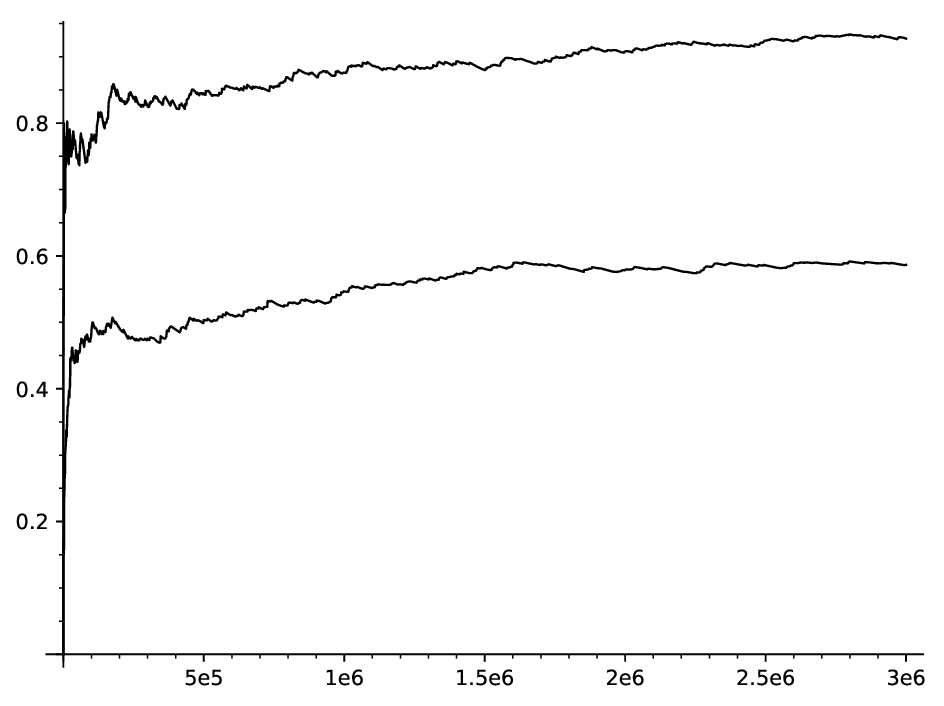}
\caption{$|l| = 3$: Top 3 bottom -3} \label{fig:15_3_A_3}
\end{subfigure}
\hspace*{-2.3cm}
\begin{subfigure}[b]{0.43\linewidth}
\includegraphics[width=\linewidth]{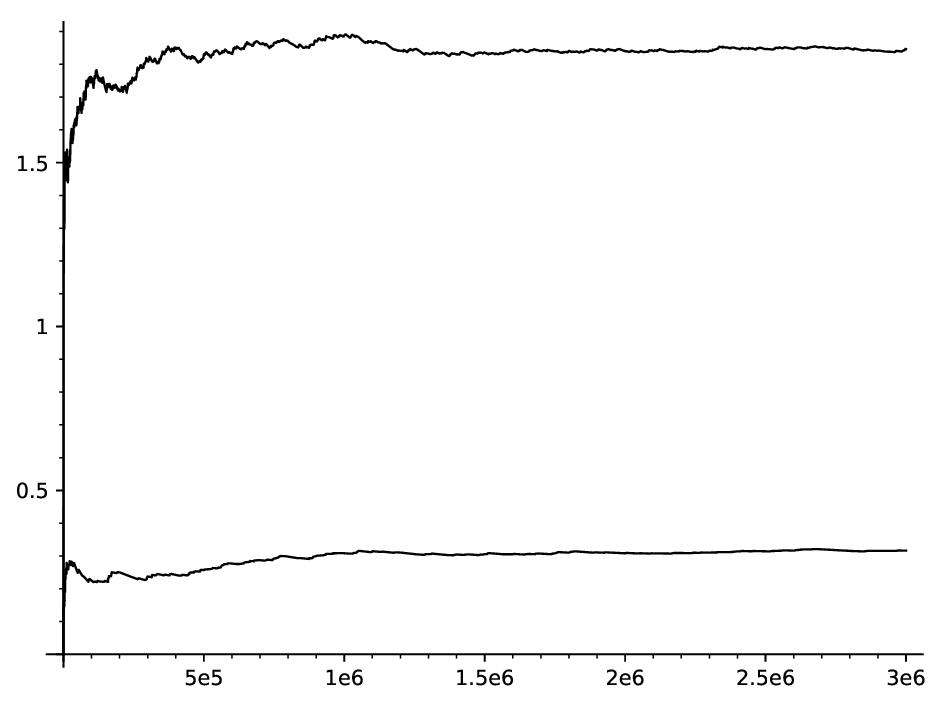}
\caption{$|l| = 4$: Top 4 bottom -4} \label{fig:15_3_A_4}
\end{subfigure}\hspace*{\fill}
\begin{subfigure}[b]{0.43\linewidth}
\includegraphics[width=\linewidth]{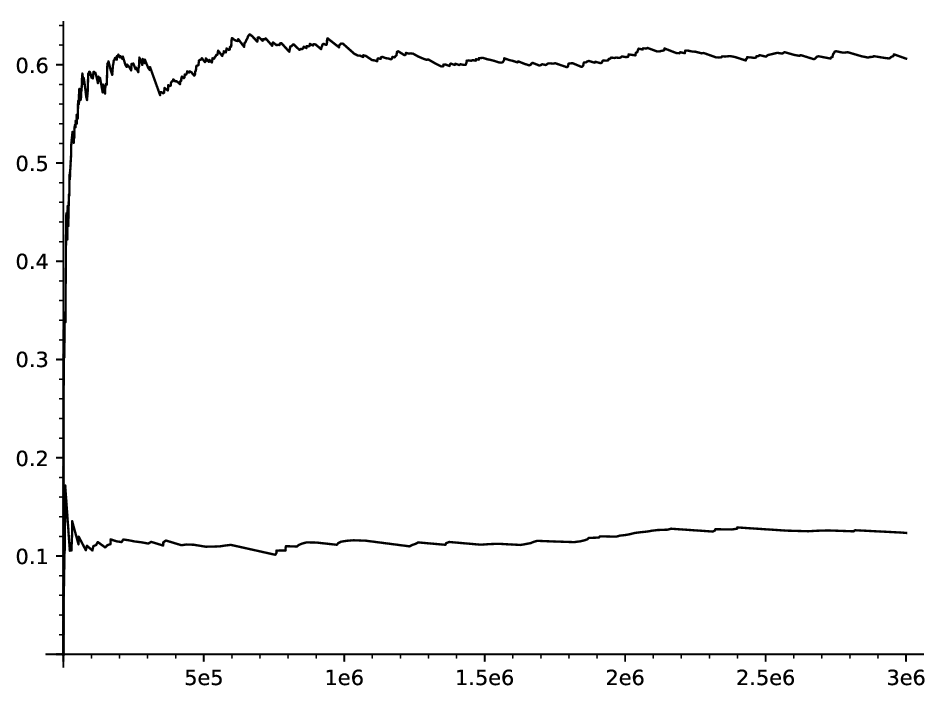}
\caption{$|l| = 5$: Top -5 bottom 5} \label{fig:15_3_A_5}
\end{subfigure}\hspace*{\fill}
\begin{subfigure}[b]{0.43\linewidth}
\includegraphics[width=\linewidth]{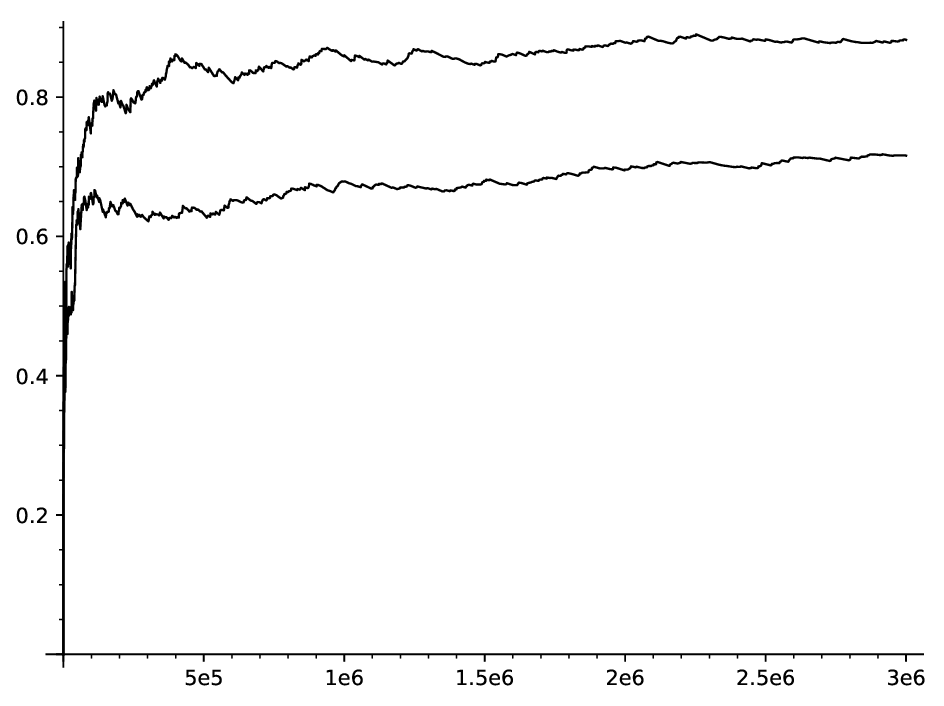}
\caption{$|l| = 6$: Top -6 bottom 6} \label{fig:15_3_A_6}
\end{subfigure}
\hspace*{-2.3cm}
\begin{subfigure}[b]{0.43\linewidth}
\includegraphics[width=\linewidth]{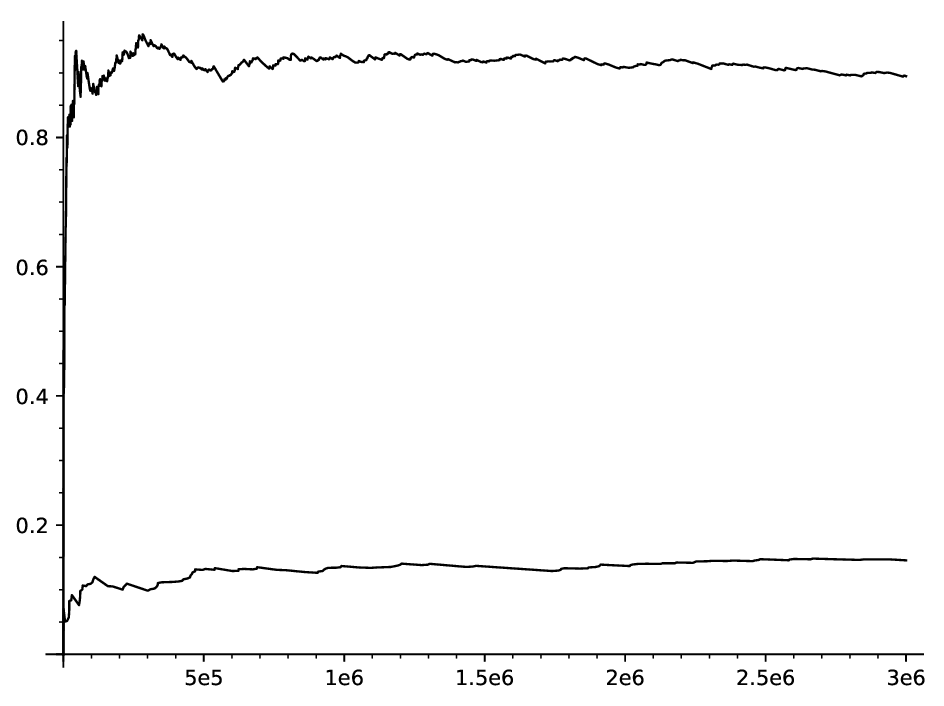}
\caption{$|l| = 7$: Top 7 bottom -7} \label{fig:15_3_A_7}
\end{subfigure}\hspace*{\fill}
\begin{subfigure}[b]{0.43\linewidth}
\includegraphics[width=\linewidth]{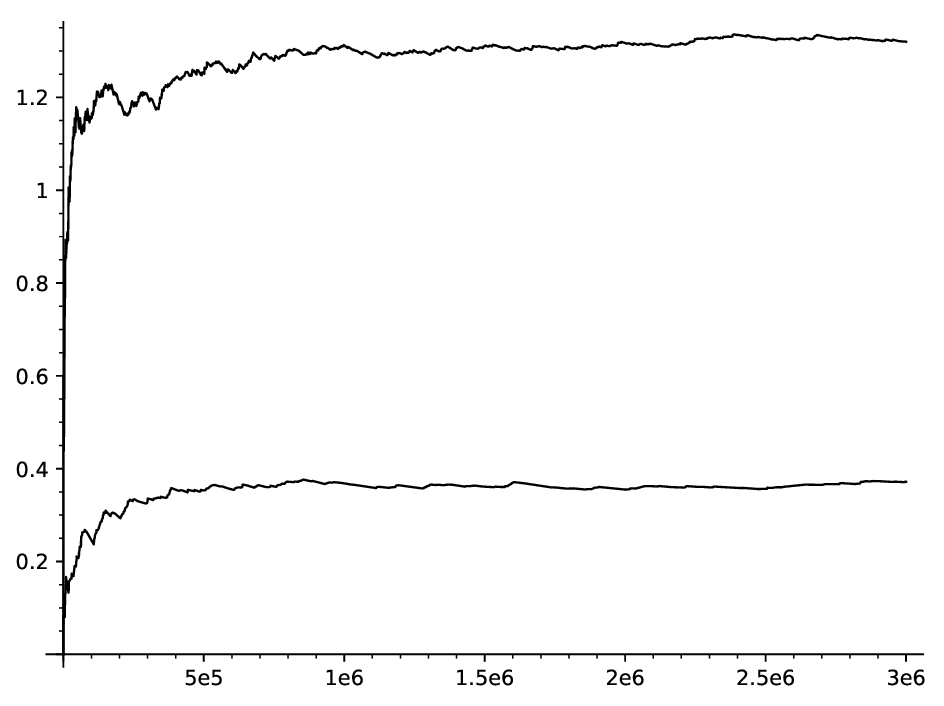}
\caption{$|l| = 8$: Top -8 bottom 8} \label{fig:15_3_A_8}
\end{subfigure}\hspace*{\fill}
\begin{subfigure}[b]{0.43\linewidth}
\includegraphics[width=\linewidth]{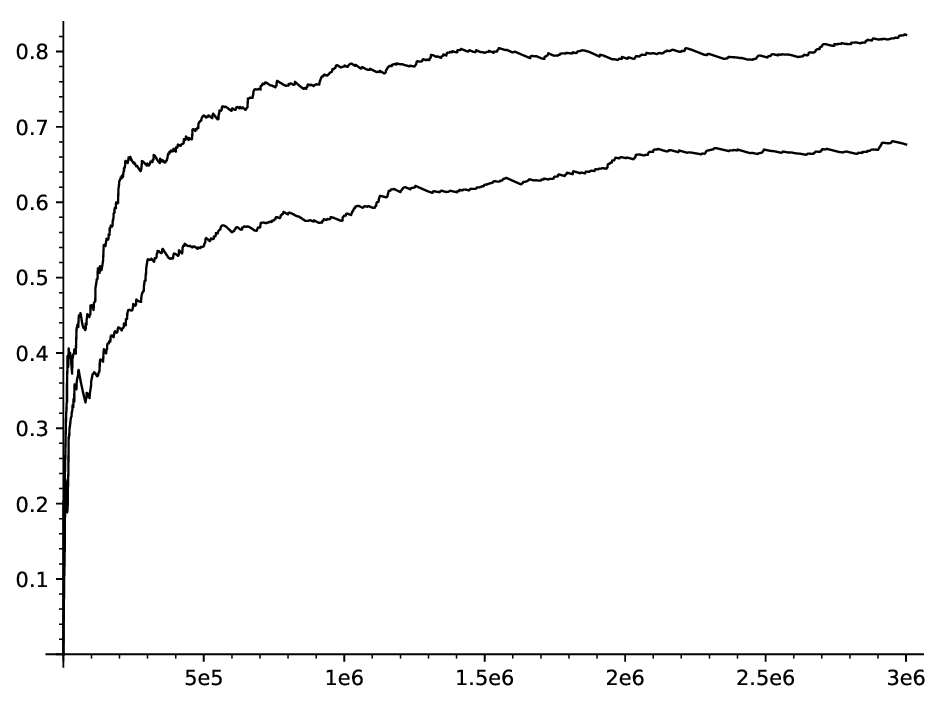}
\caption{$|l| = 9$: Top 9 bottom -9} \label{fig:15_3_A_9}
\end{subfigure}
\caption{Ratio~\eqref{ratio_A_exact} 15a1: $x(X;l)/X^{1/2}$ for $k = 3$} \label{fig:15a1_3_A_exact}
\end{figure}

\clearpage

\begin{figure}[t!] 
\hspace*{-2.3cm}
\begin{subfigure}[b]{0.43\linewidth}
\includegraphics[width=\linewidth]{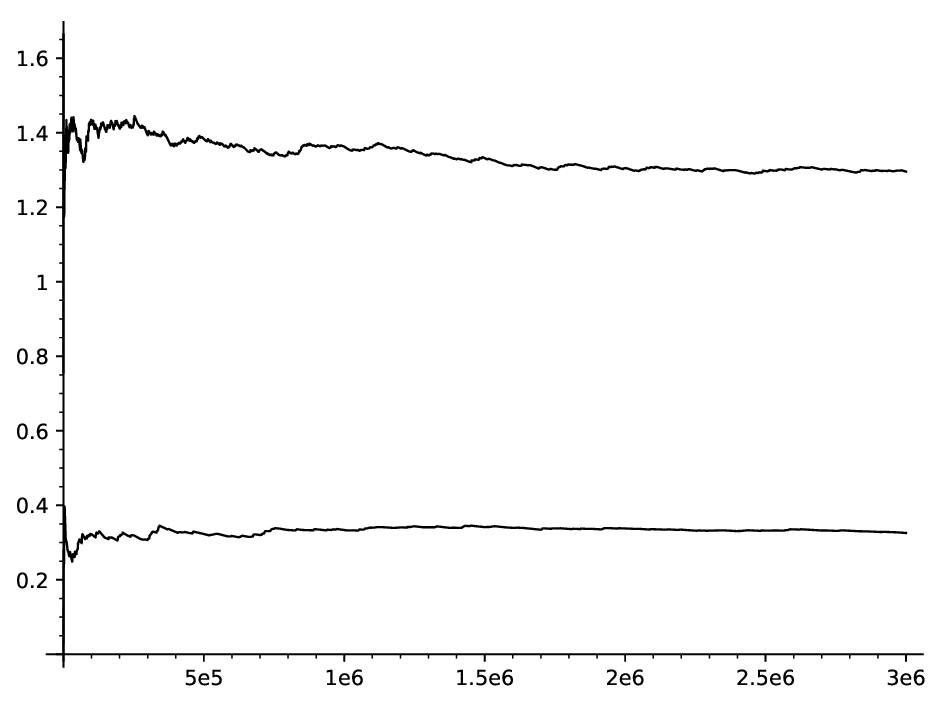}
\caption{$|l| = 1$: Top 1 bottom -1} \label{fig:17_3_A_1}
\end{subfigure}\hspace*{\fill}
\begin{subfigure}[b]{0.43\linewidth}
\includegraphics[width=\linewidth]{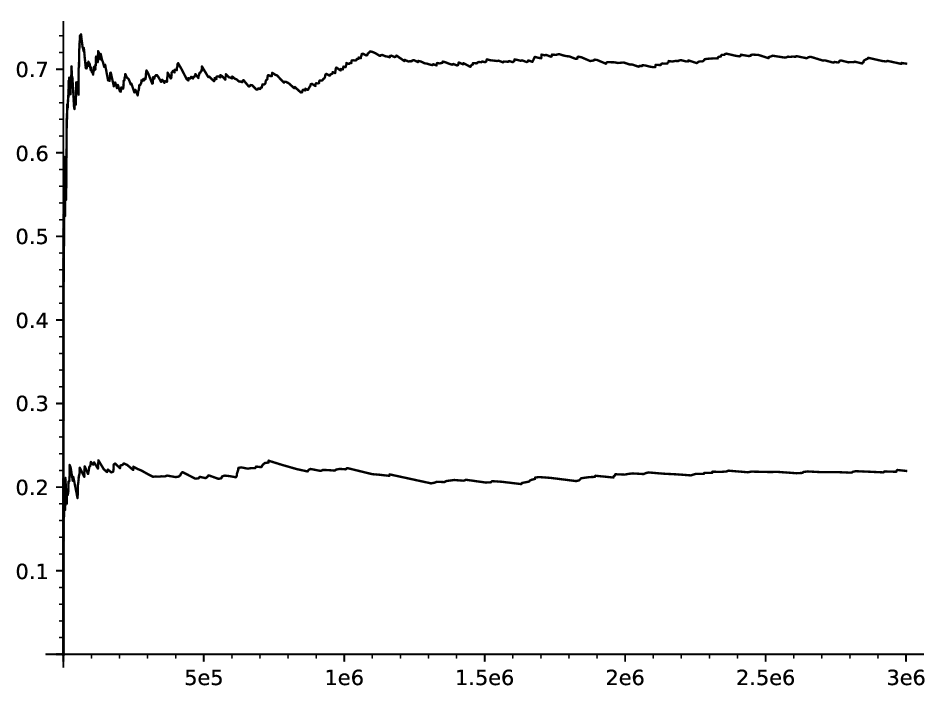}
\caption{$|l| = 2$: Top -2 bottom 2} \label{fig:17_3_A_2}
\end{subfigure}\hspace*{\fill}
\begin{subfigure}[b]{0.43\linewidth}
\includegraphics[width=\linewidth]{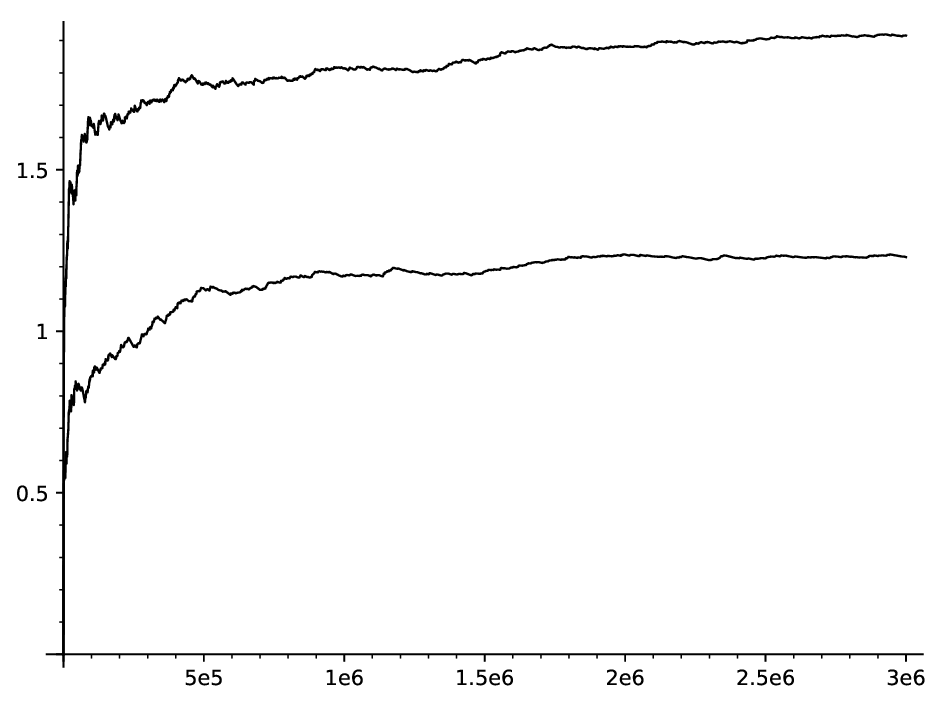}
\caption{$|l| = 3$: Top 3 bottom -3} \label{fig:17_3_A_3}
\end{subfigure}
\hspace*{-2.3cm}
\begin{subfigure}[b]{0.43\linewidth}
\includegraphics[width=\linewidth]{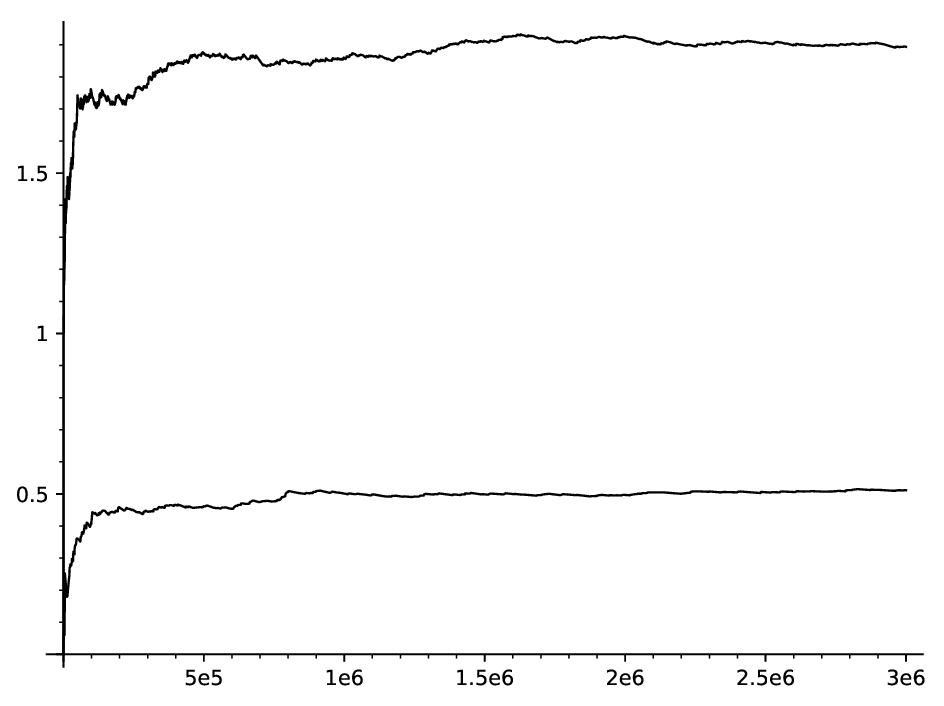}
\caption{$|l| = 4$: Top 4 bottom -4} \label{fig:17_3_A_4}
\end{subfigure}\hspace*{\fill}
\begin{subfigure}[b]{0.43\linewidth}
\includegraphics[width=\linewidth]{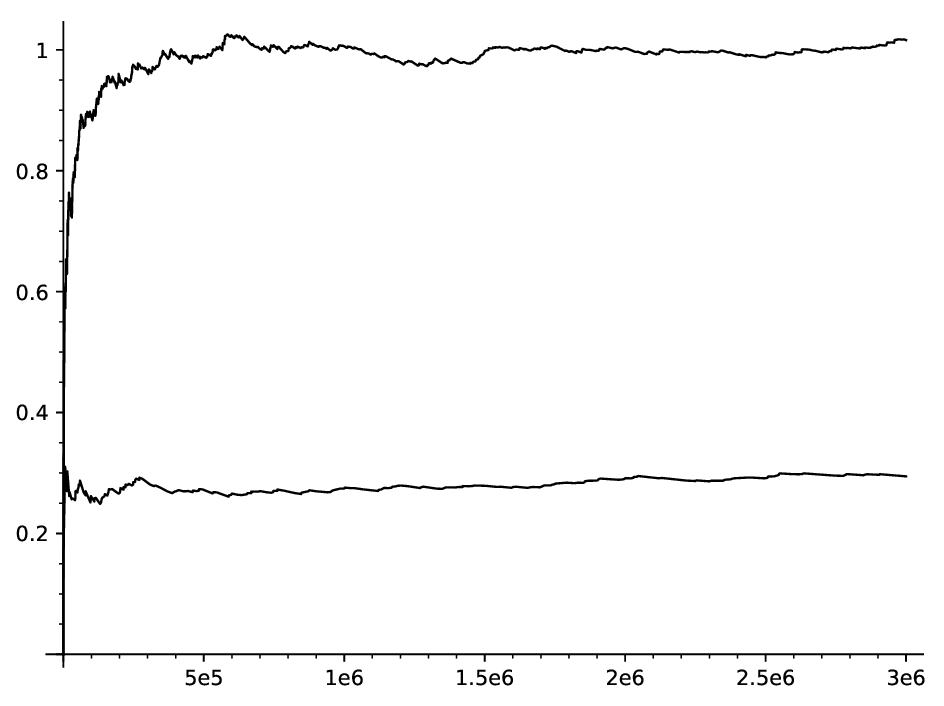}
\caption{$|l| = 5$: Top -5 bottom 5} \label{fig:17_3_A_5}
\end{subfigure}\hspace*{\fill}
\begin{subfigure}[b]{0.43\linewidth}
\includegraphics[width=\linewidth]{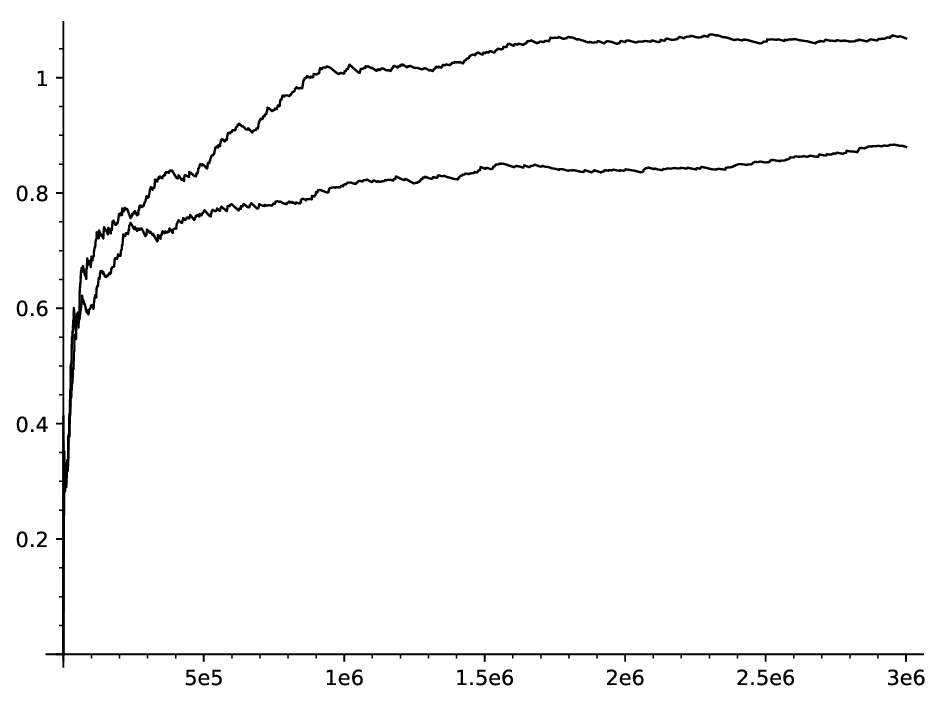}
\caption{$|l| = 6$: Top -6 bottom 6} \label{fig:17_3_A_6}
\end{subfigure}
\hspace*{-2.3cm}
\begin{subfigure}[b]{0.43\linewidth}
\includegraphics[width=\linewidth]{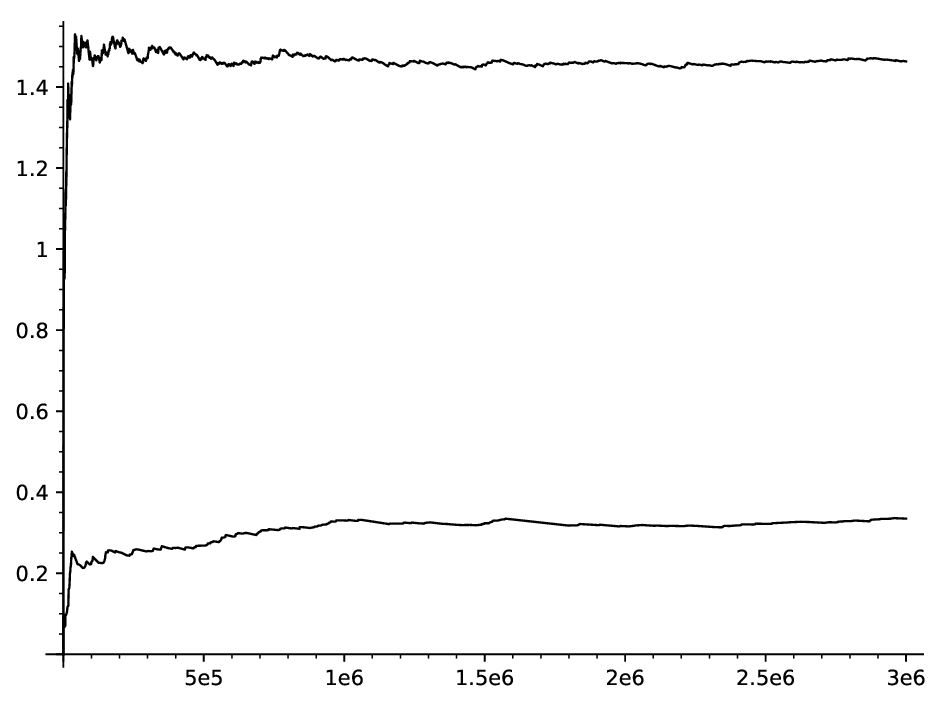}
\caption{$|l| = 7$: Top 7 bottom -7} \label{fig:17_3_A_7}
\end{subfigure}\hspace*{\fill}
\begin{subfigure}[b]{0.43\linewidth}
\includegraphics[width=\linewidth]{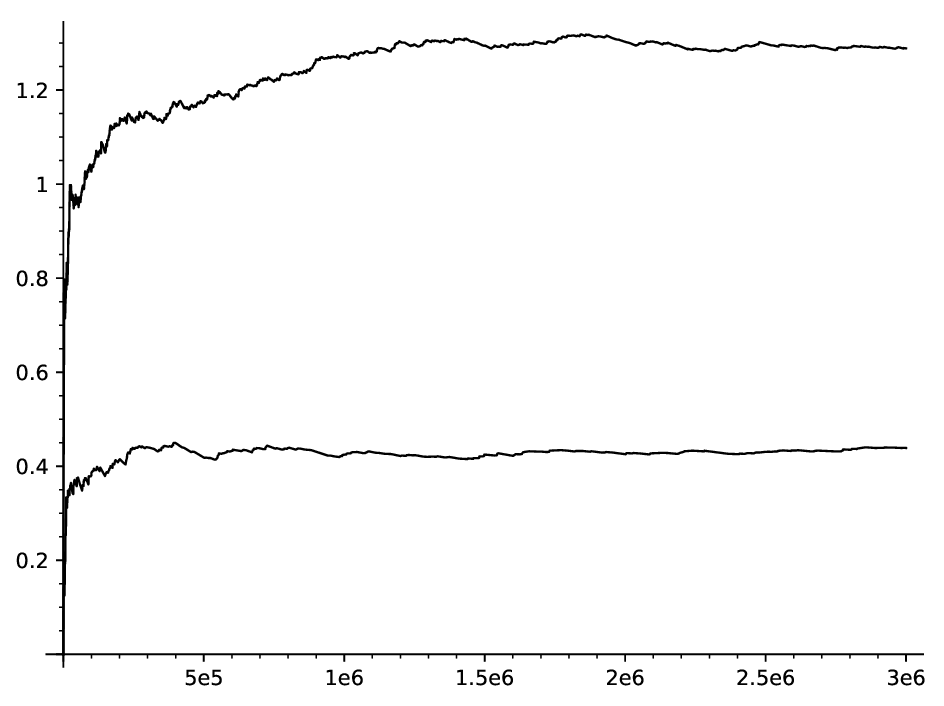}
\caption{$|l| = 8$: Top -8 bottom 8} \label{fig:17_3_A_8}
\end{subfigure}\hspace*{\fill}
\begin{subfigure}[b]{0.43\linewidth}
\includegraphics[width=\linewidth]{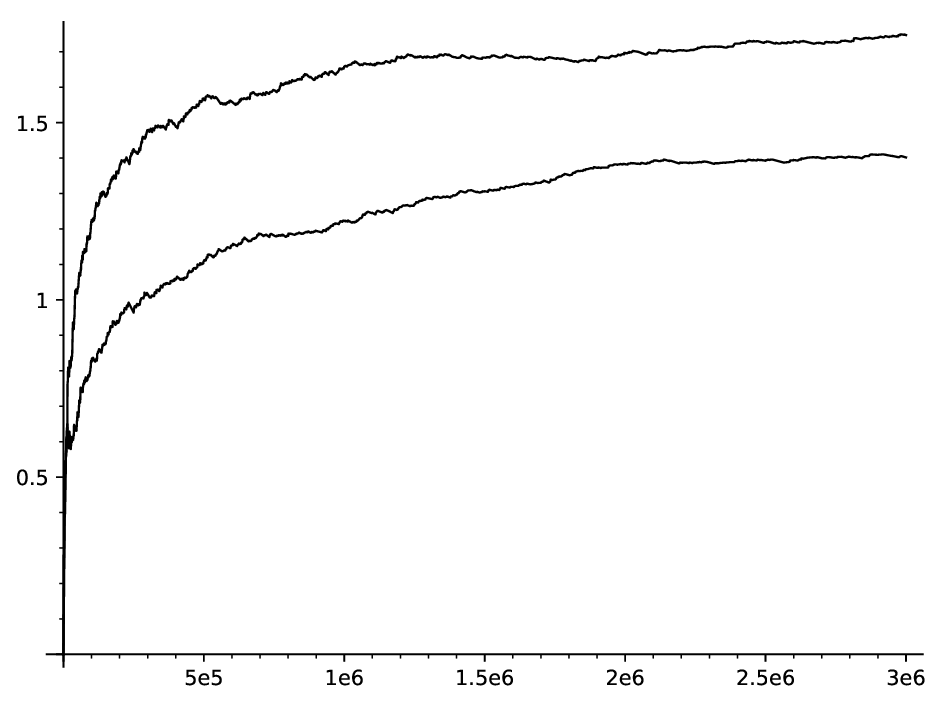}
\caption{$|l| = 9$: Top 9 bottom -9} \label{fig:17_3_A_9}
\end{subfigure}
\caption{Ratio~\eqref{ratio_A_exact} 17a1: $x(X;l)/X^{1/2}$ for $k = 3$} \label{fig:17a1_3_A_exact}
\end{figure}

\clearpage

\begin{figure}[t!] 
\hspace*{-2.3cm}
\begin{subfigure}[b]{0.43\linewidth}
\includegraphics[width=\linewidth]{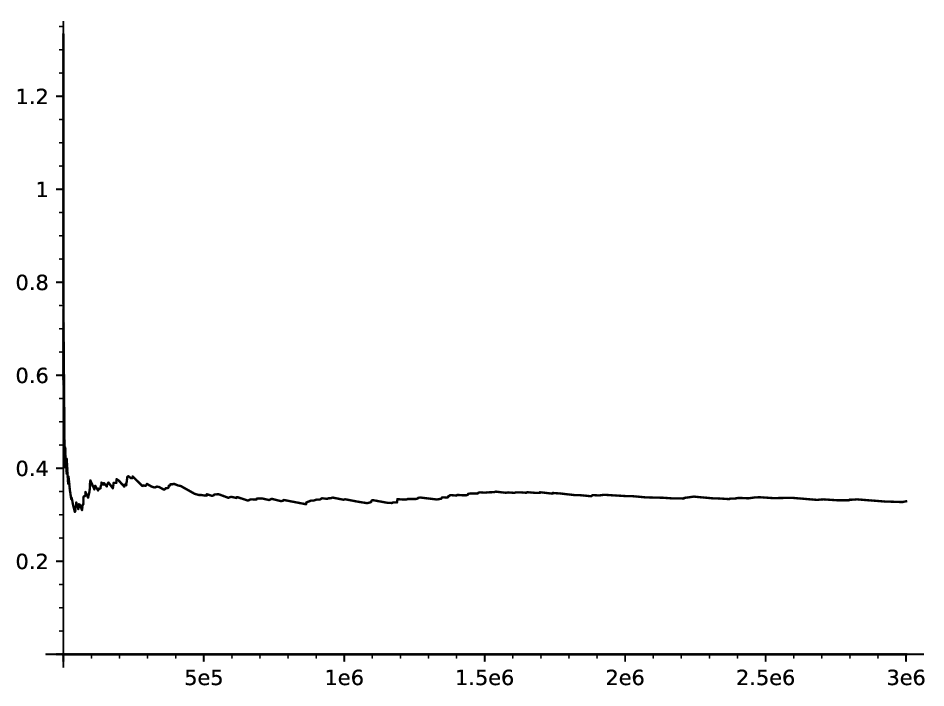}
\caption{$|l| = 1$: No -1 exists} \label{fig:19_3_A_1}
\end{subfigure}\hspace*{\fill}
\begin{subfigure}[b]{0.43\linewidth}
\includegraphics[width=\linewidth]{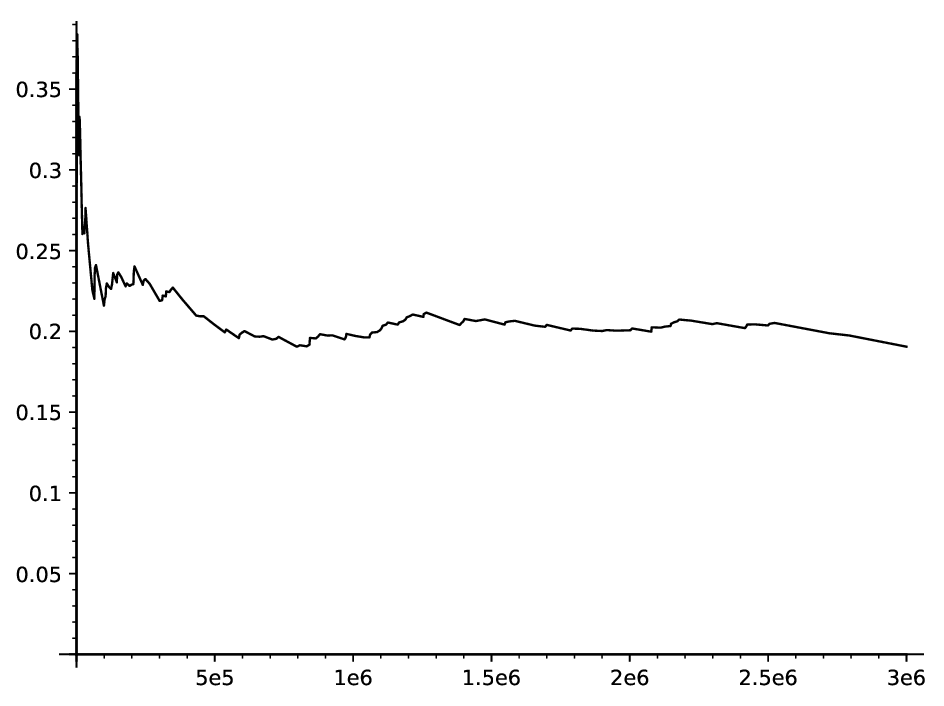}
\caption{$|l| = 2$: No 2 exists} \label{fig:19_3_A_2}
\end{subfigure}\hspace*{\fill}
\begin{subfigure}[b]{0.43\linewidth}
\includegraphics[width=\linewidth]{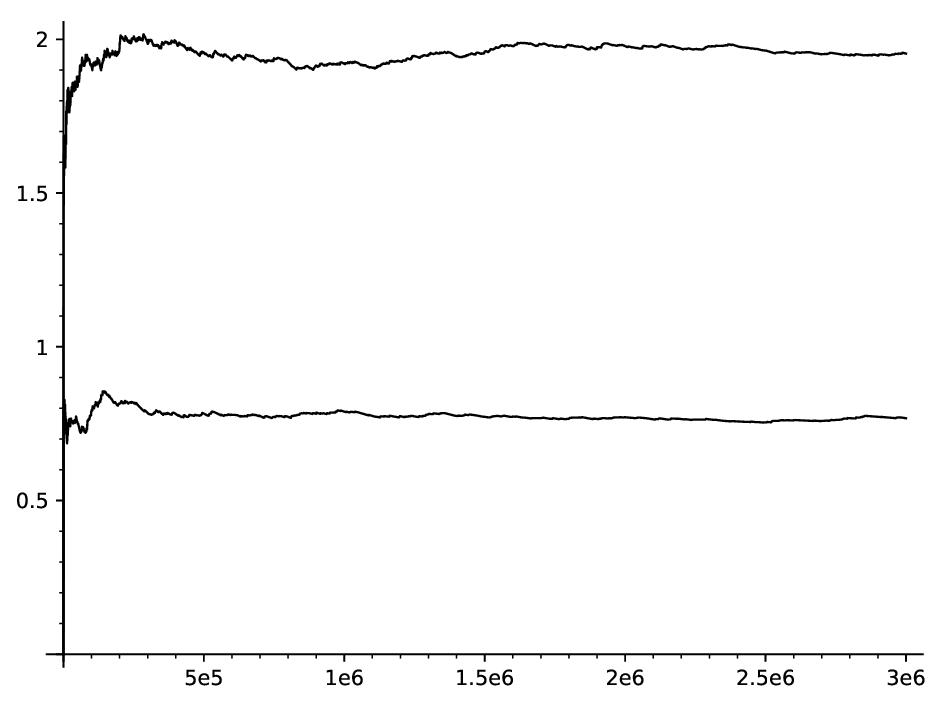}
\caption{$|l| = 3$: Top 3 bottom -3} \label{fig:19_3_A_3}
\end{subfigure}
\hspace*{-2.3cm}
\begin{subfigure}[b]{0.43\linewidth}
\includegraphics[width=\linewidth]{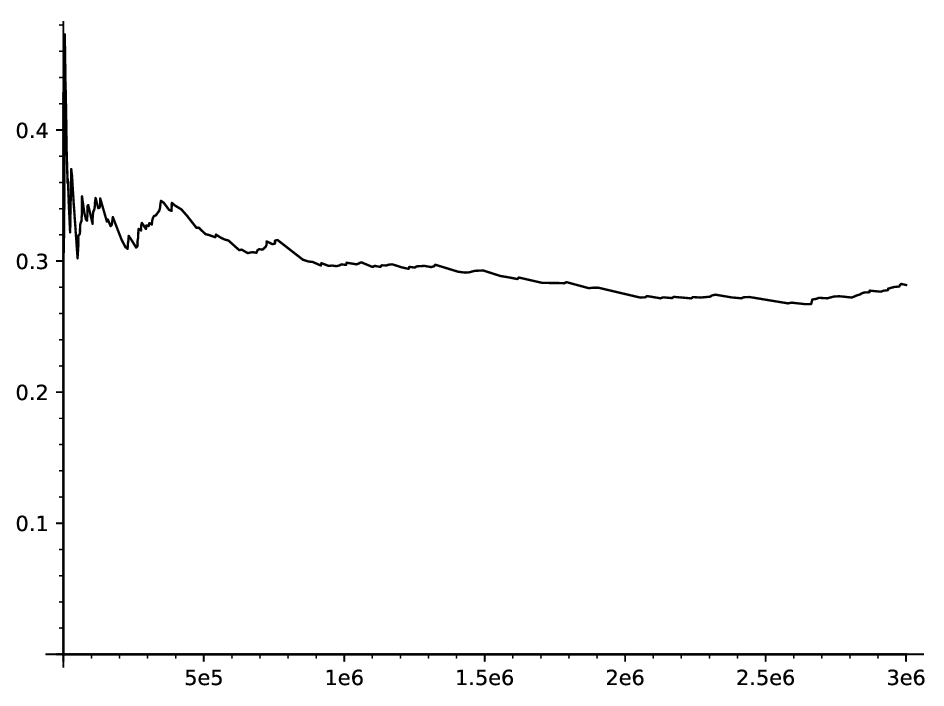}
\caption{$|l| = 4$: No -4 exists} \label{fig:19_3_A_4}
\end{subfigure}\hspace*{\fill}
\begin{subfigure}[b]{0.43\linewidth}
\includegraphics[width=\linewidth]{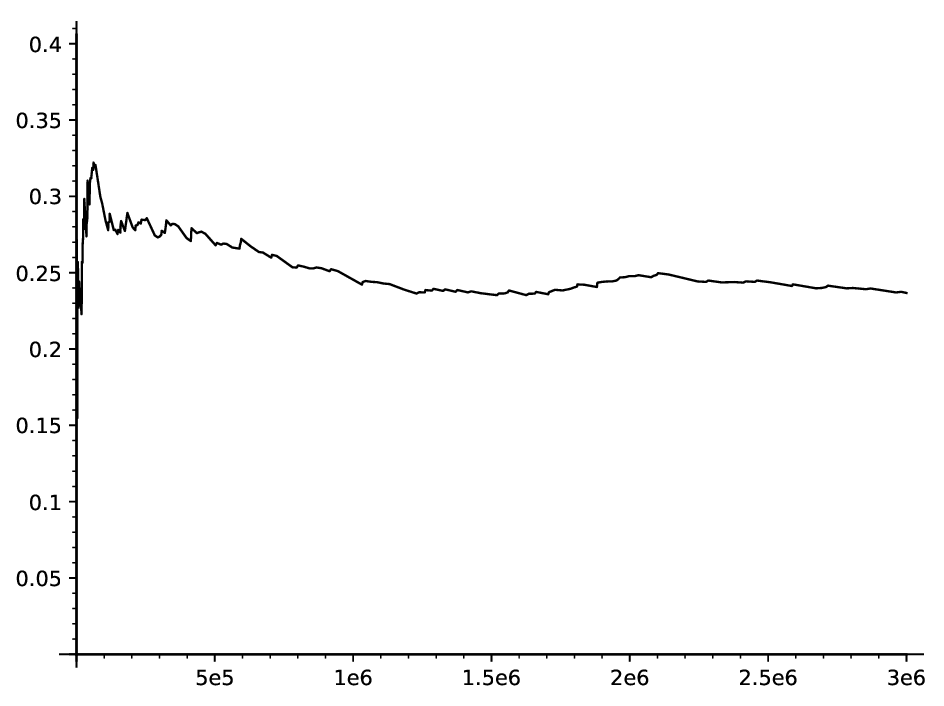}
\caption{$|l| = 5$: No 5 exists} \label{fig:19_3_A_5}
\end{subfigure}\hspace*{\fill}
\begin{subfigure}[b]{0.43\linewidth}
\includegraphics[width=\linewidth]{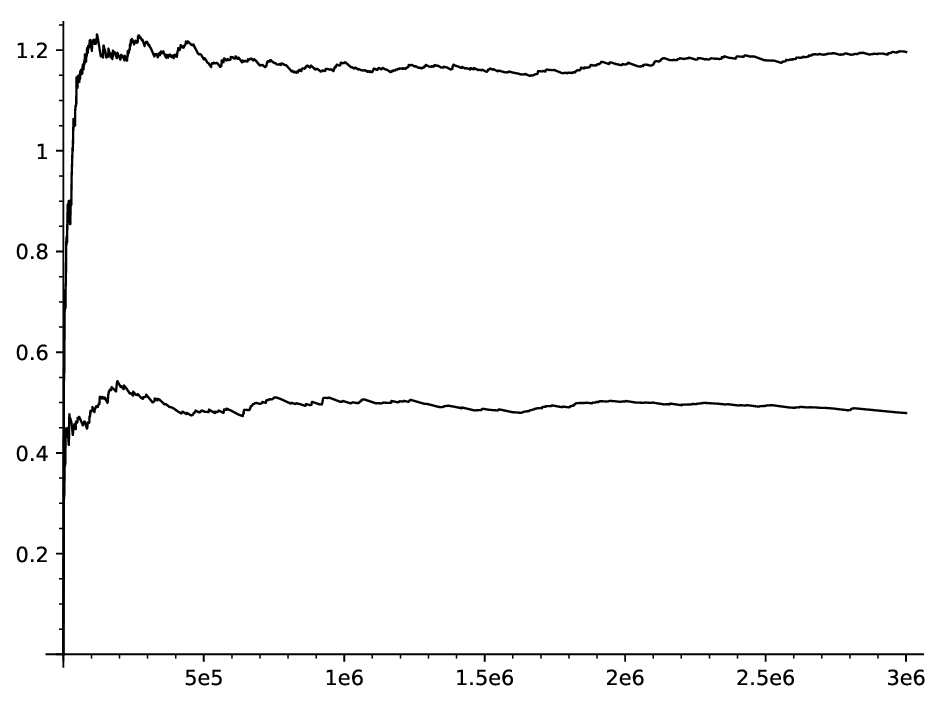}
\caption{$|l| = 6$: Top -6 bottom 6} \label{fig:19_3_A_6}
\end{subfigure}
\hspace*{-2.3cm}
\begin{subfigure}[b]{0.43\linewidth}
\includegraphics[width=\linewidth]{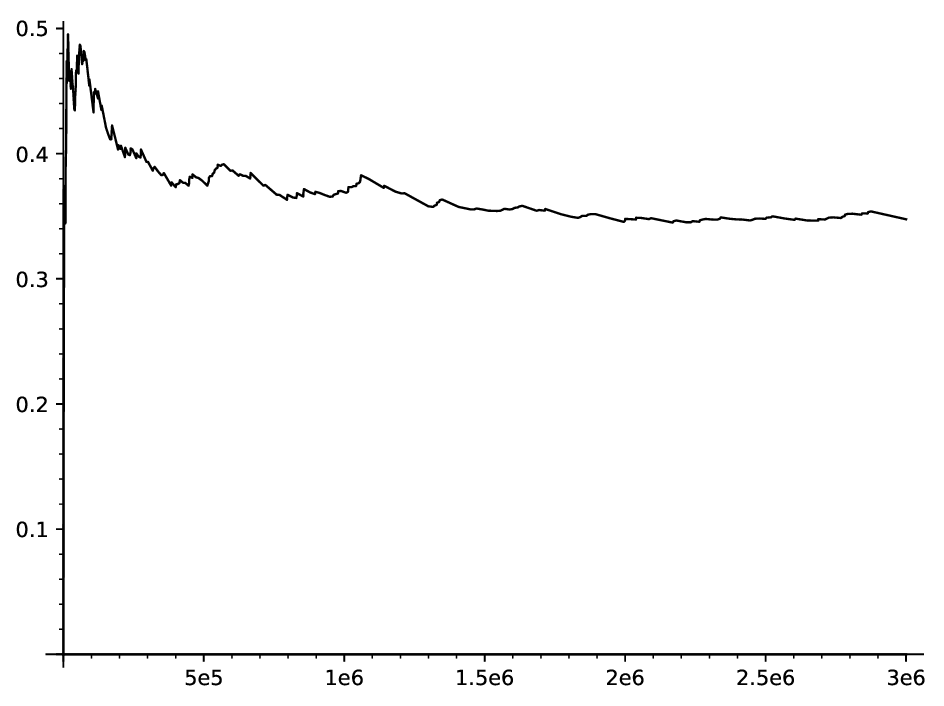}
\caption{$|l| = 7$: No -7 exists} \label{fig:19_3_A_7}
\end{subfigure}\hspace*{\fill}
\begin{subfigure}[b]{0.43\linewidth}
\includegraphics[width=\linewidth]{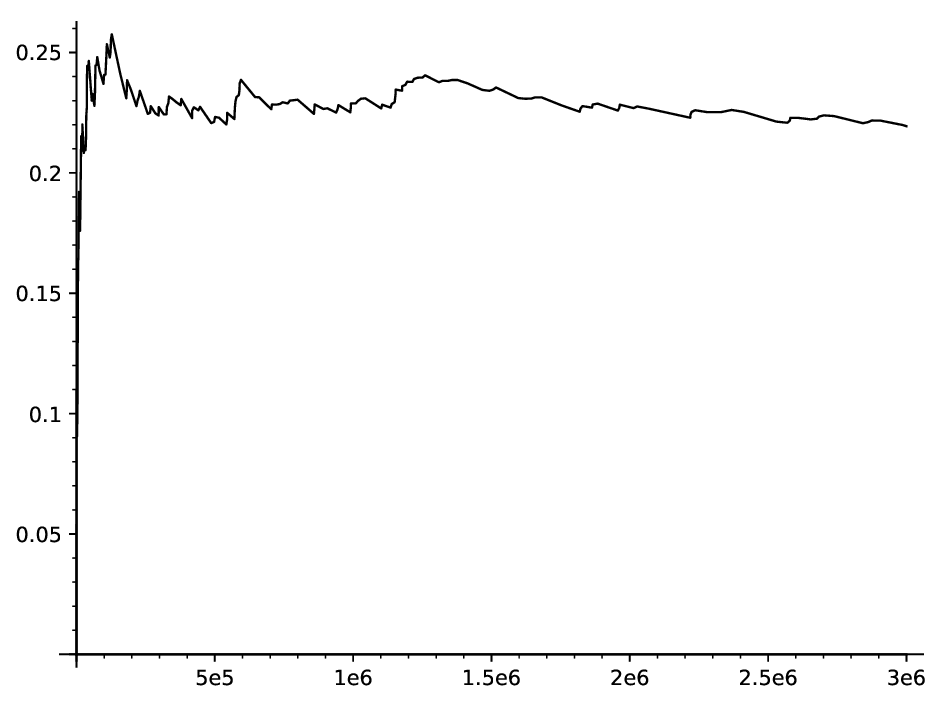}
\caption{$|l| = 8$: No 8 exists} \label{fig:19_3_A_8}
\end{subfigure}\hspace*{\fill}
\begin{subfigure}[b]{0.43\linewidth}
\includegraphics[width=\linewidth]{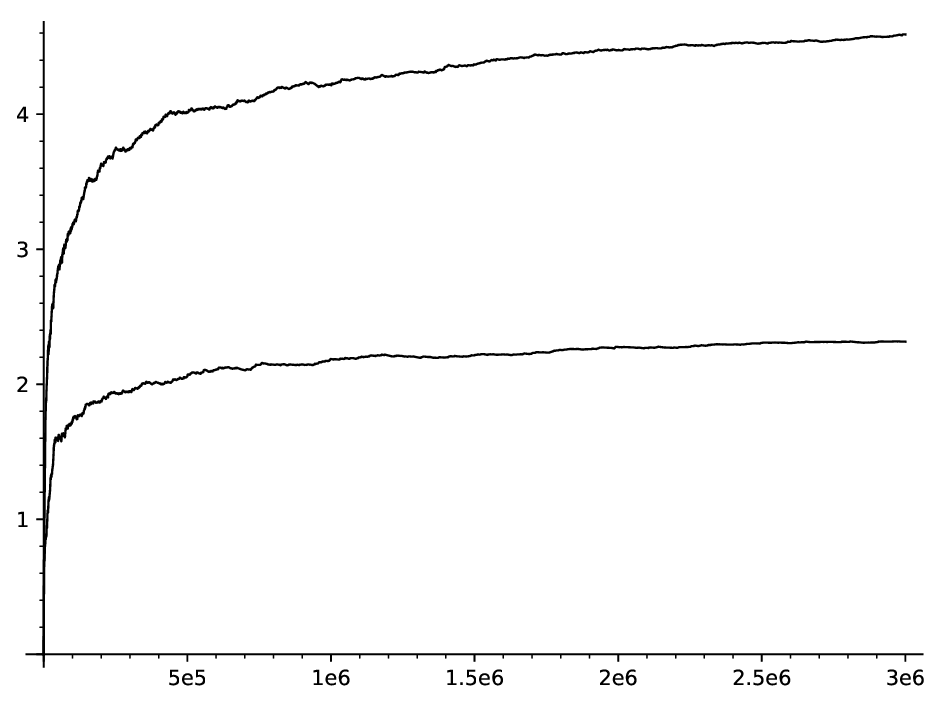}
\caption{$|l| = 9$: Top 9 bottom -9} \label{fig:19_3_A_9}
\end{subfigure}
\caption{Ratio~\eqref{ratio_A_exact} 19a1: $x(X;l)/X^{1/2}$ for $k = 3$} \label{fig:19a1_3_A_exact}
\end{figure}

\clearpage

\begin{figure}[t!] 
\hspace*{-2.3cm}
\begin{subfigure}[b]{0.43\linewidth}
\includegraphics[width=\linewidth]{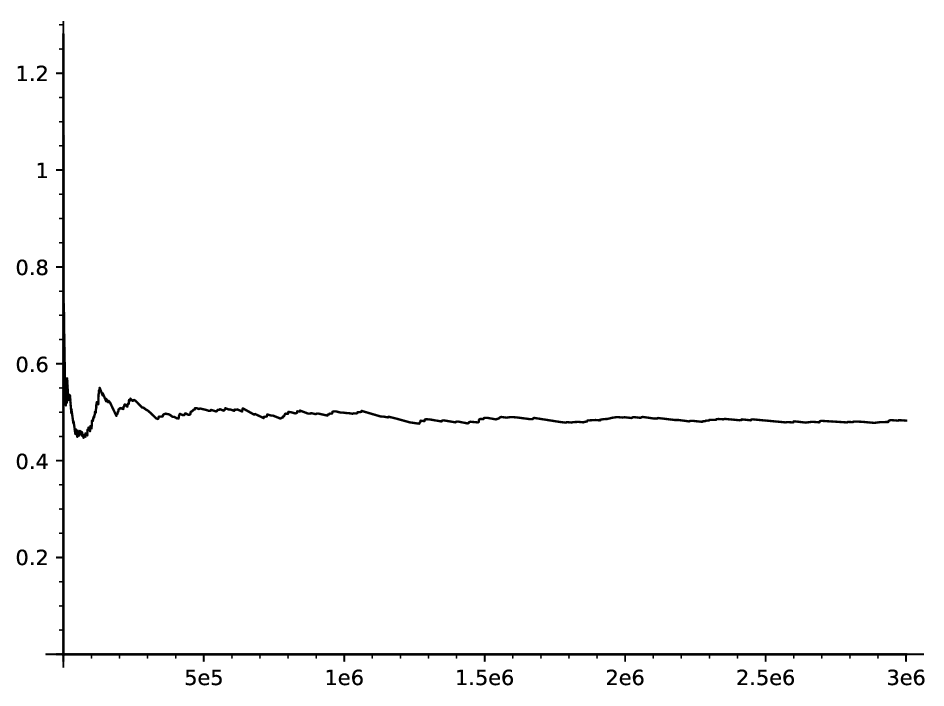}
\caption{$|l| = 1$: No -1 exists} \label{fig:37_3_A_1}
\end{subfigure}\hspace*{\fill}
\begin{subfigure}[b]{0.43\linewidth}
\includegraphics[width=\linewidth]{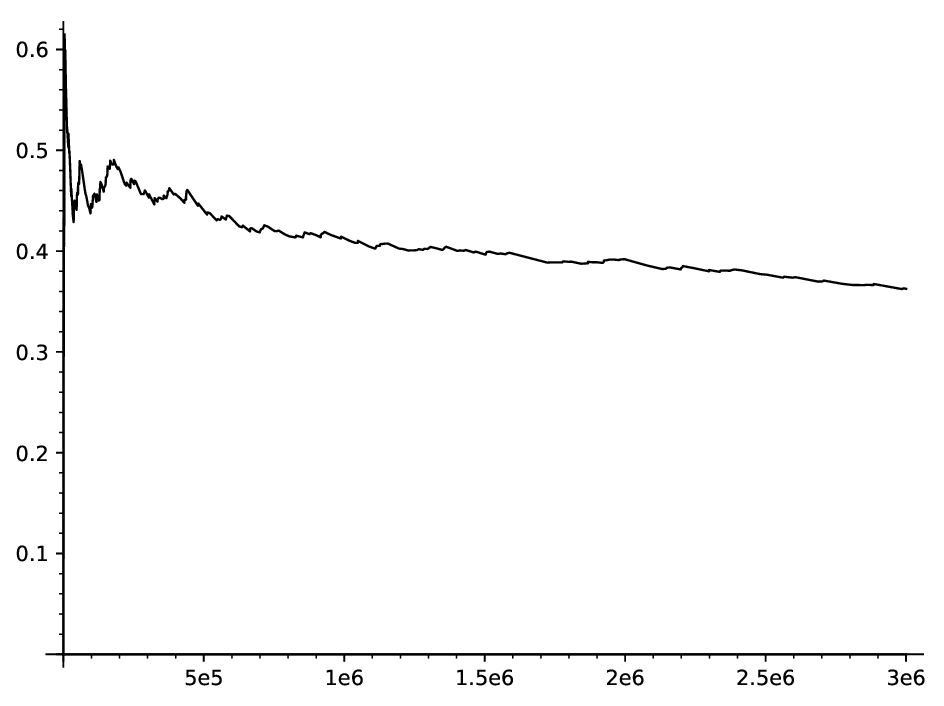}
\caption{$|l| = 2$: No 2 exists} \label{fig:37_3_A_2}
\end{subfigure}\hspace*{\fill}
\begin{subfigure}[b]{0.43\linewidth}
\includegraphics[width=\linewidth]{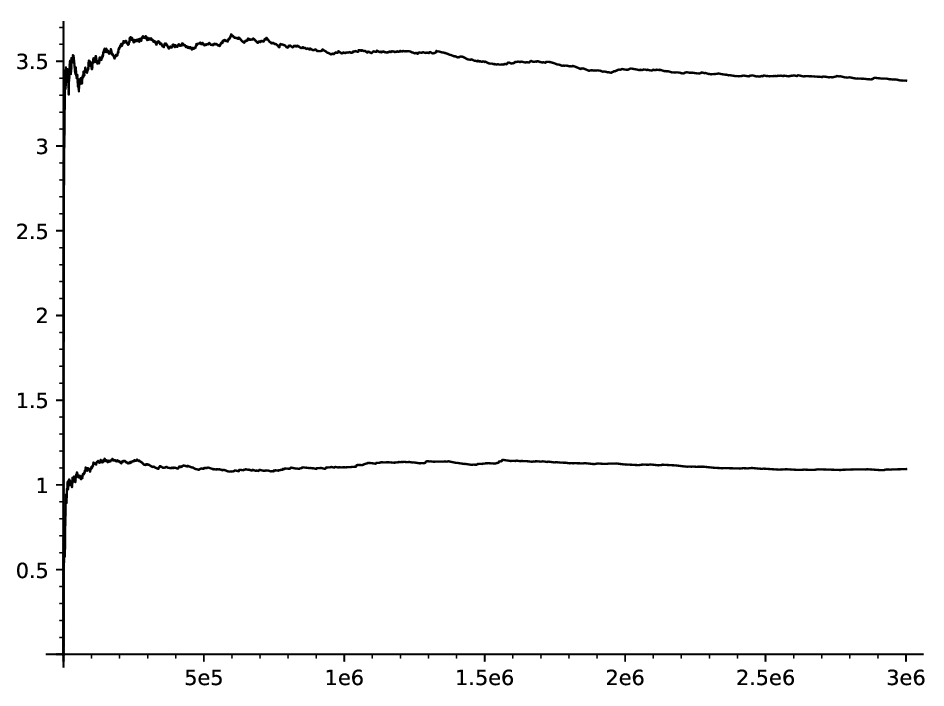}
\caption{$|l| = 3$: Top 3 bottom -3} \label{fig:37_3_A_3}
\end{subfigure}
\hspace*{-2.3cm}
\begin{subfigure}[b]{0.43\linewidth}
\includegraphics[width=\linewidth]{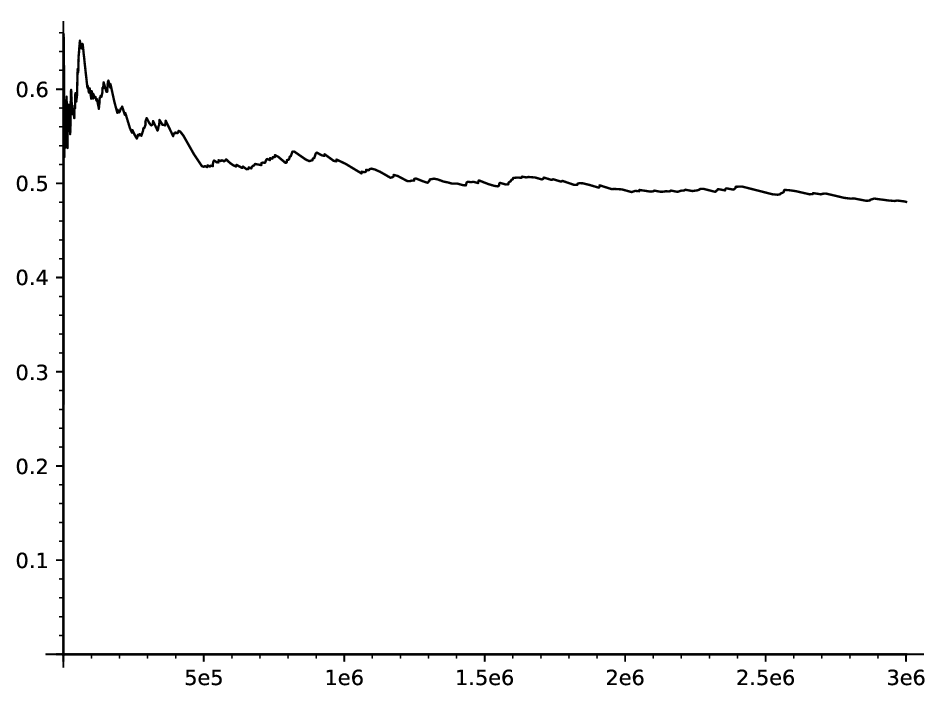}
\caption{$|l| = 4$: No -4 exists} \label{fig:37_3_A_4}
\end{subfigure}\hspace*{\fill}
\begin{subfigure}[b]{0.43\linewidth}
\includegraphics[width=\linewidth]{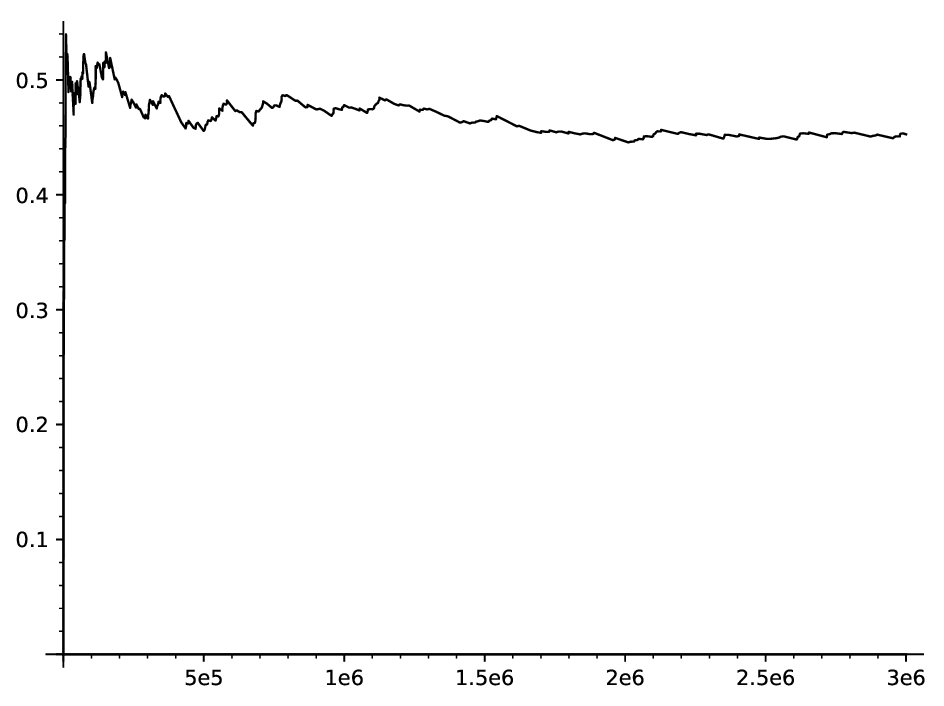}
\caption{$|l| = 5$: No 5 exists} \label{fig:37_3_A_5}
\end{subfigure}\hspace*{\fill}
\begin{subfigure}[b]{0.43\linewidth}
\includegraphics[width=\linewidth]{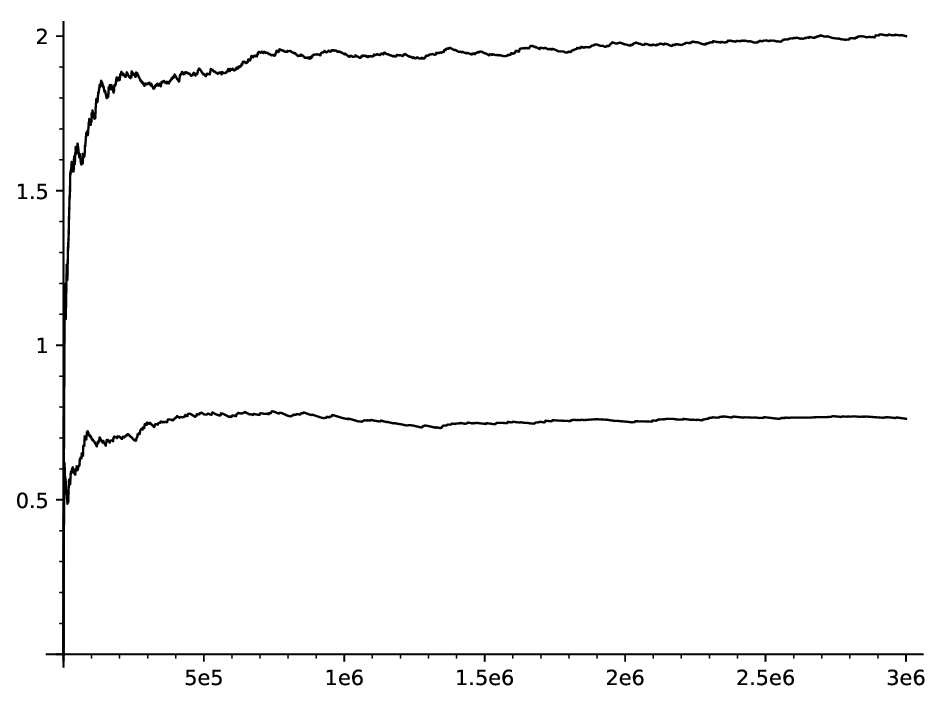}
\caption{$|l| = 6$: Top -6 bottom 6} \label{fig:37_3_A_6}
\end{subfigure}
\hspace*{-2.3cm}
\begin{subfigure}[b]{0.43\linewidth}
\includegraphics[width=\linewidth]{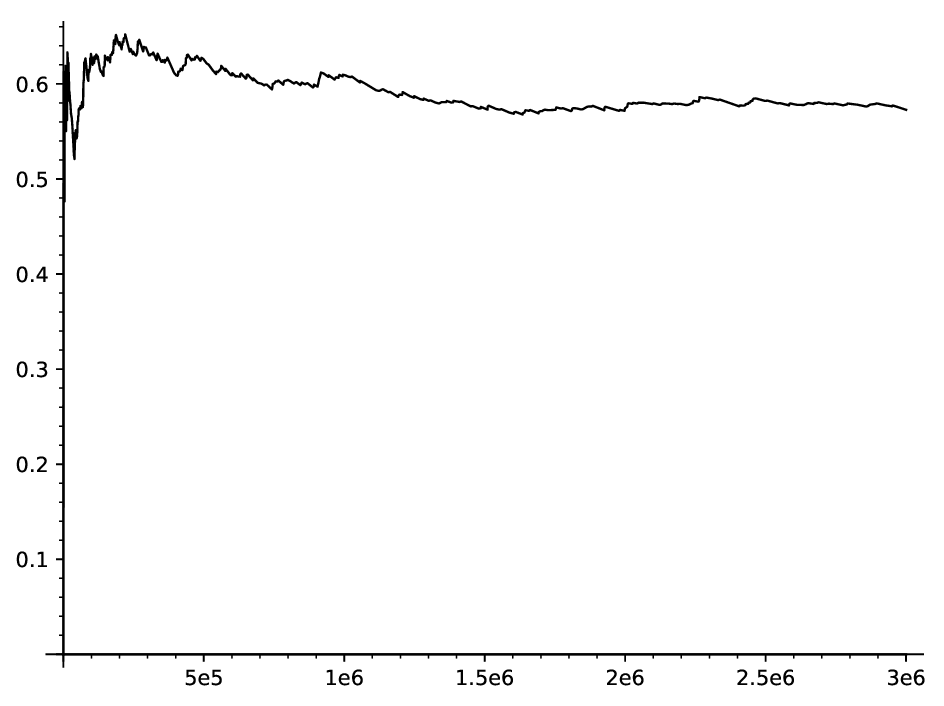}
\caption{$|l| = 7$: No -7 exists} \label{fig:37_3_A_7}
\end{subfigure}\hspace*{\fill}
\begin{subfigure}[b]{0.43\linewidth}
\includegraphics[width=\linewidth]{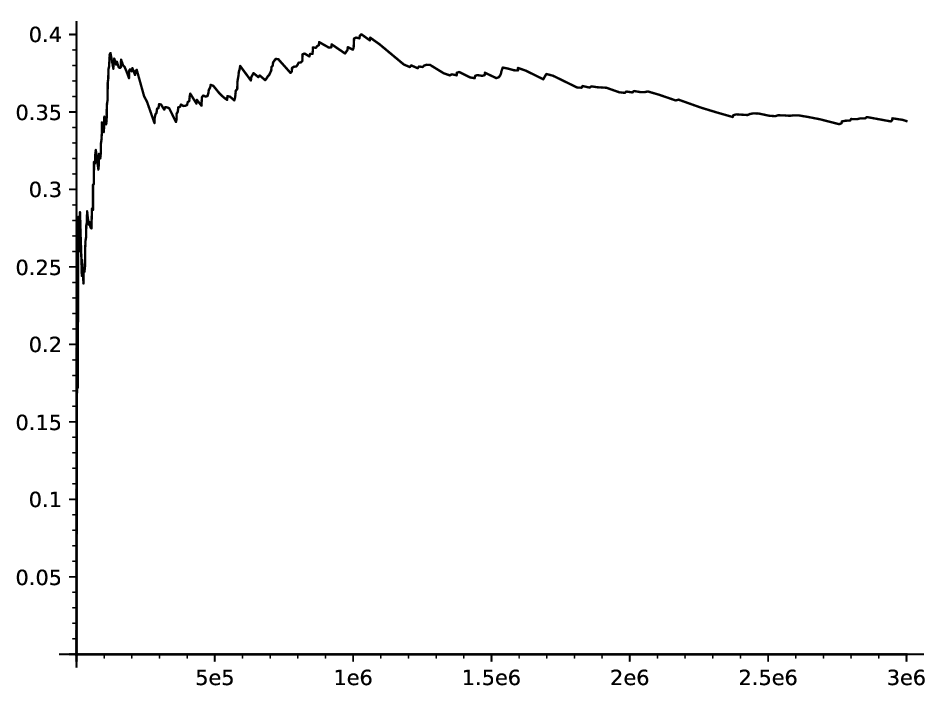}
\caption{$|l| = 8$: No 8 exists} \label{fig:37_3_A_8}
\end{subfigure}\hspace*{\fill}
\begin{subfigure}[b]{0.43\linewidth}
\includegraphics[width=\linewidth]{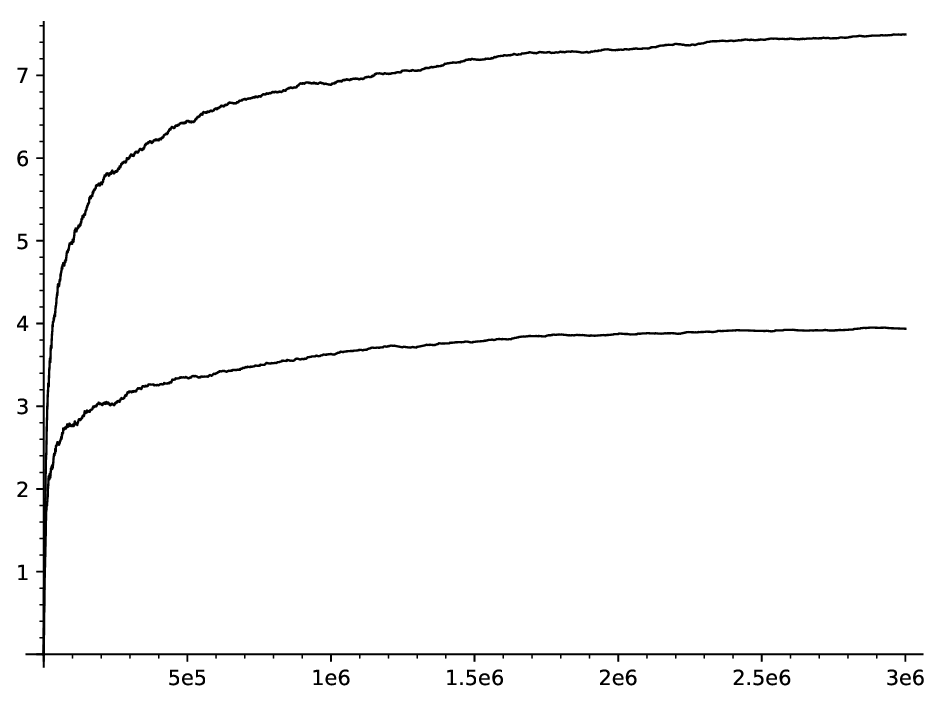}
\caption{$|l| = 9$: Top 9 bottom -9} \label{fig:37_3_A_9}
\end{subfigure}
\caption{Ratio~\eqref{ratio_A_exact} 37b1: $x(X;l)/X^{1/2}$ for $k = 3$} \label{fig:37b1_3_A_exact}
\end{figure}

\clearpage

\begin{figure}[t!] 
\hspace*{-2.3cm}
\begin{subfigure}[b]{0.43\linewidth}
\includegraphics[width=\linewidth]{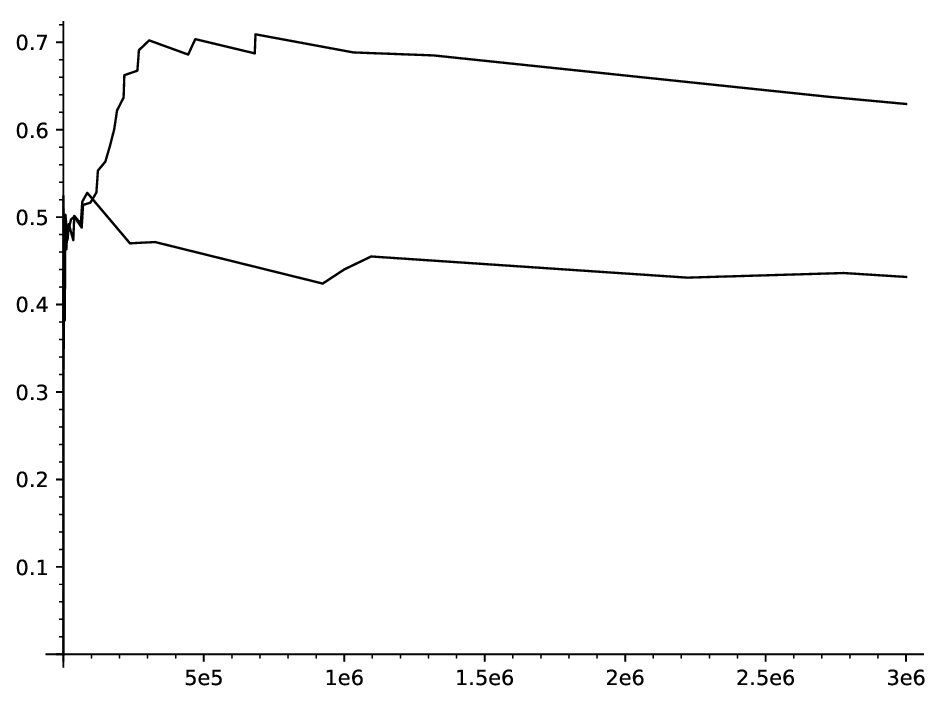}
\caption{$|l| = 1$: Top 1 bottom -1} \label{fig:11_5_A_1}
\end{subfigure}\hspace*{\fill}
\begin{subfigure}[b]{0.43\linewidth}
\includegraphics[width=\linewidth]{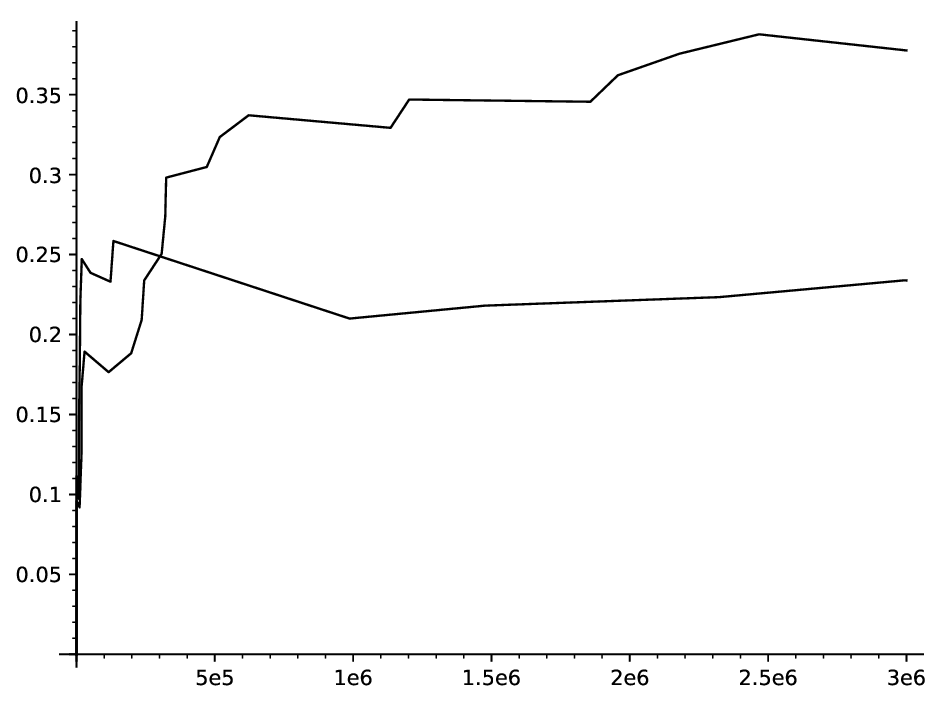}
\caption{$|l| = 4$: Top -4 bottom 4} \label{fig:11_5_A_4}
\end{subfigure}\hspace*{\fill}
\begin{subfigure}[b]{0.43\linewidth}
\includegraphics[width=\linewidth]{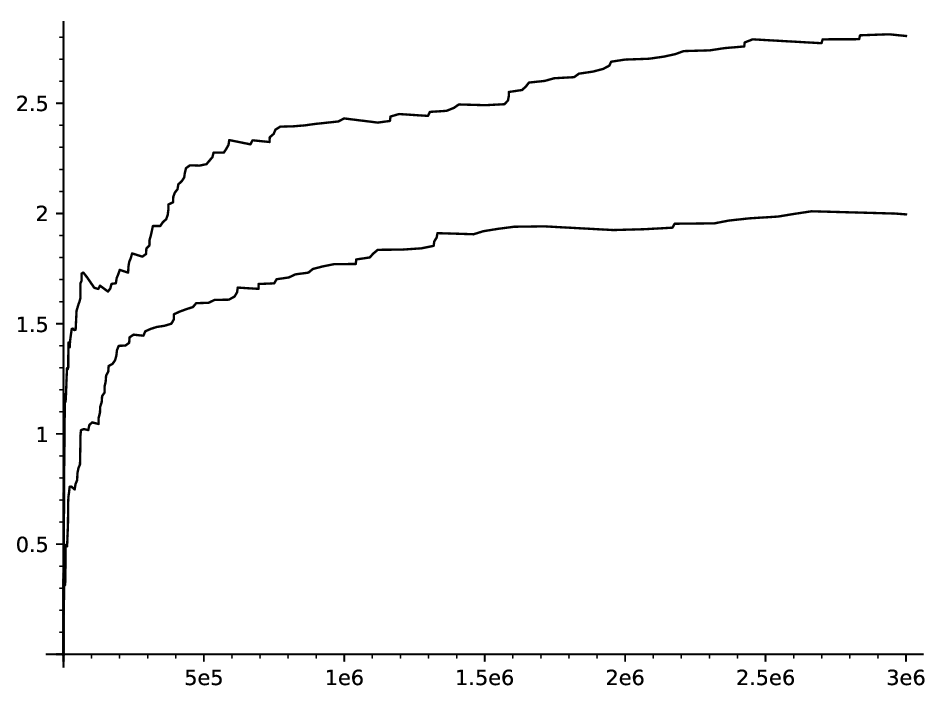}
\caption{$|l| = 5$: Top 5 bottom -5} \label{fig:11_5_A_5}
\end{subfigure}
\hspace*{-2.3cm}
\begin{subfigure}[b]{0.43\linewidth}
\includegraphics[width=\linewidth]{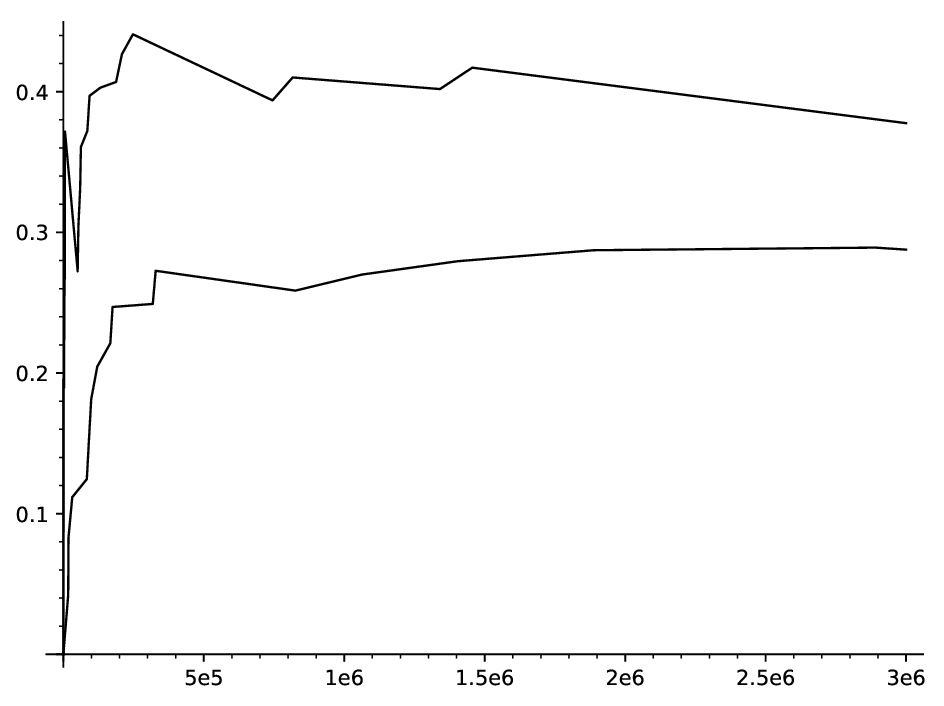}
\caption{$|l| = 9$: Top 9 bottom -9} \label{fig:11_5_A_9}
\end{subfigure}\hspace*{\fill}
\begin{subfigure}[b]{0.43\linewidth}
\includegraphics[width=\linewidth]{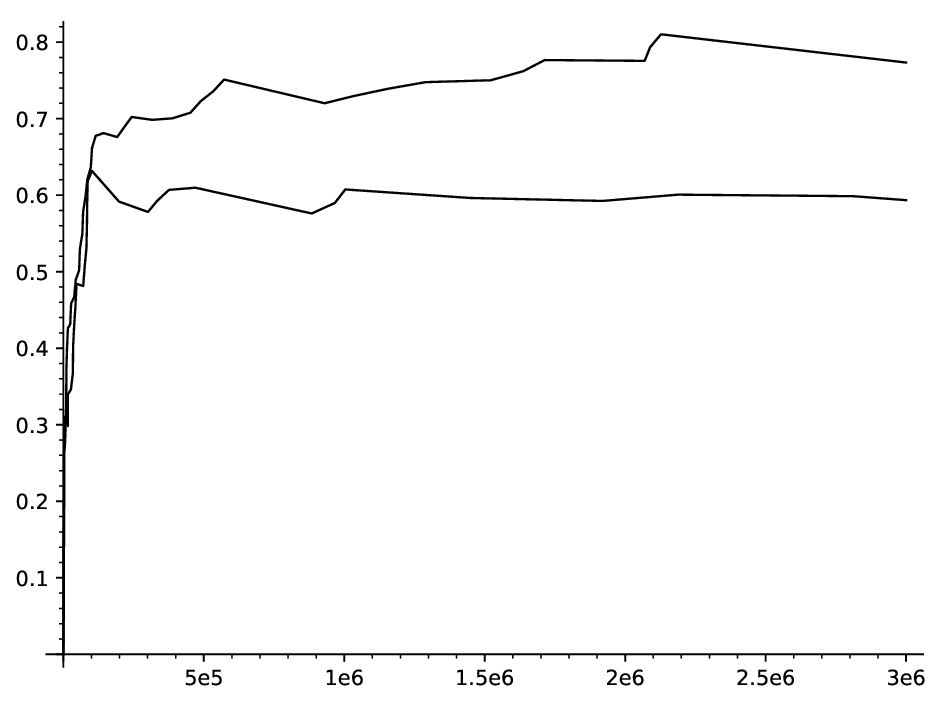}
\caption{$|l| = 11$: Top 11 bottom -11} \label{fig:11_5_A_11}
\end{subfigure}\hspace*{\fill}
\begin{subfigure}[b]{0.43\linewidth}
\includegraphics[width=\linewidth]{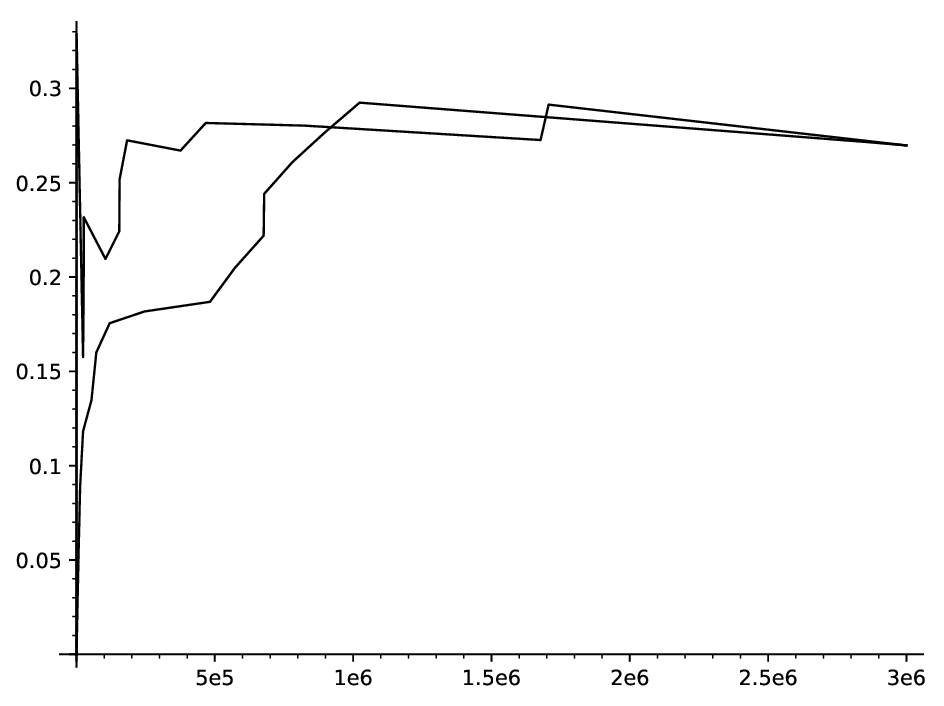}
\caption{$|l| = 16$: Top -16 bottom 16} \label{fig:11_5_A_16}
\end{subfigure}
\hspace*{-2.3cm}
\begin{subfigure}[b]{0.43\linewidth}
\includegraphics[width=\linewidth]{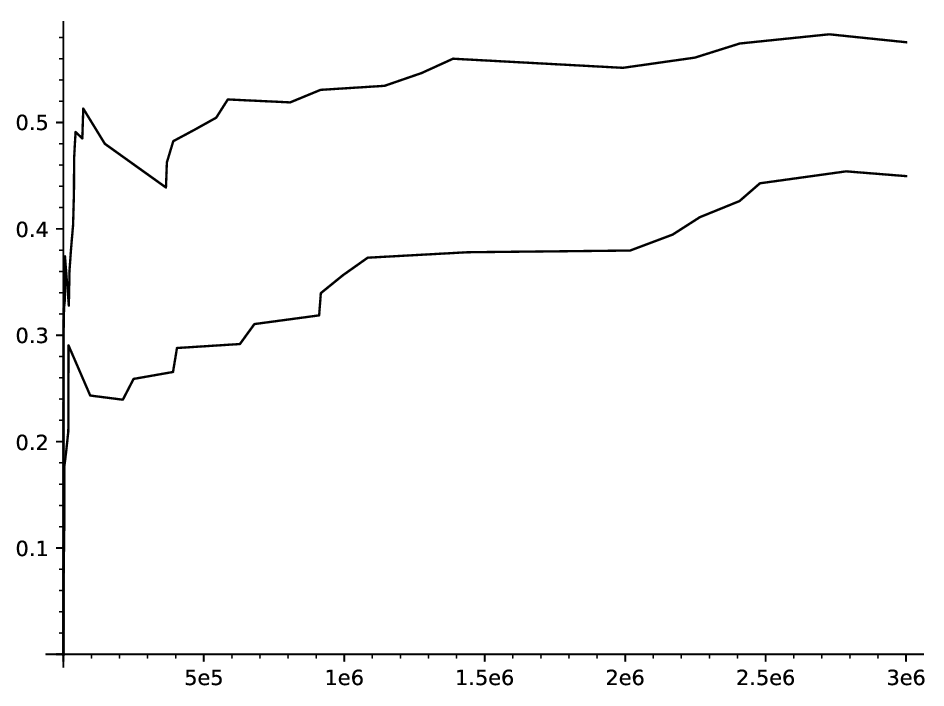}
\caption{$|l| = 19$: Top -19 bottom 19} \label{fig:11_5_A_19}
\end{subfigure}\hspace*{\fill}
\begin{subfigure}[b]{0.43\linewidth}
\includegraphics[width=\linewidth]{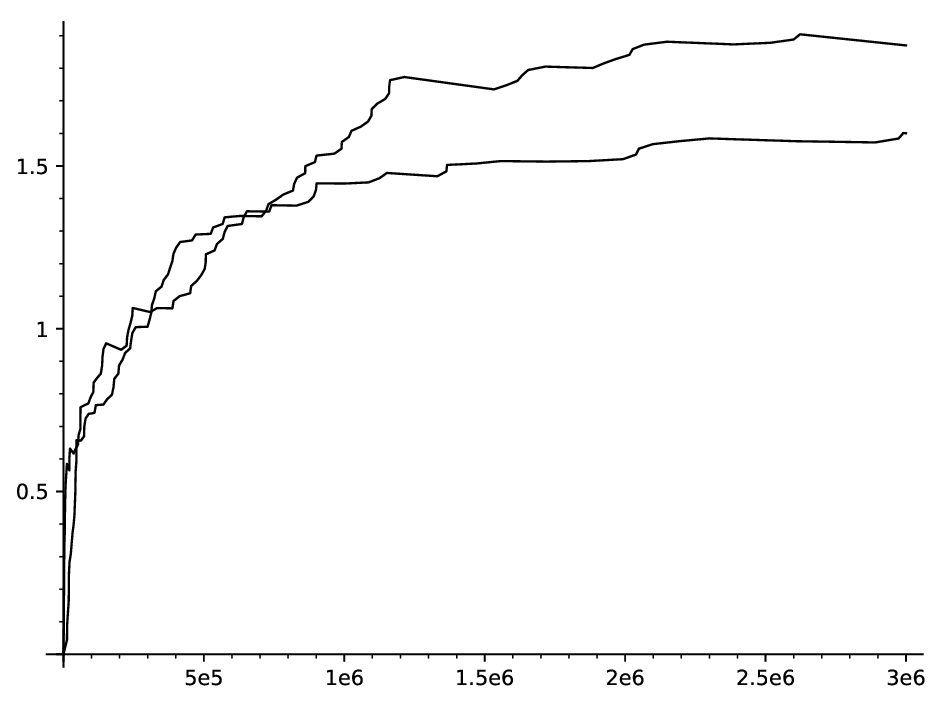}
\caption{$|l| = 20$: Top -20 bottom 20} \label{fig:11_5_A_20}
\end{subfigure}\hspace*{\fill}
\begin{subfigure}[b]{0.43\linewidth}
\includegraphics[width=\linewidth]{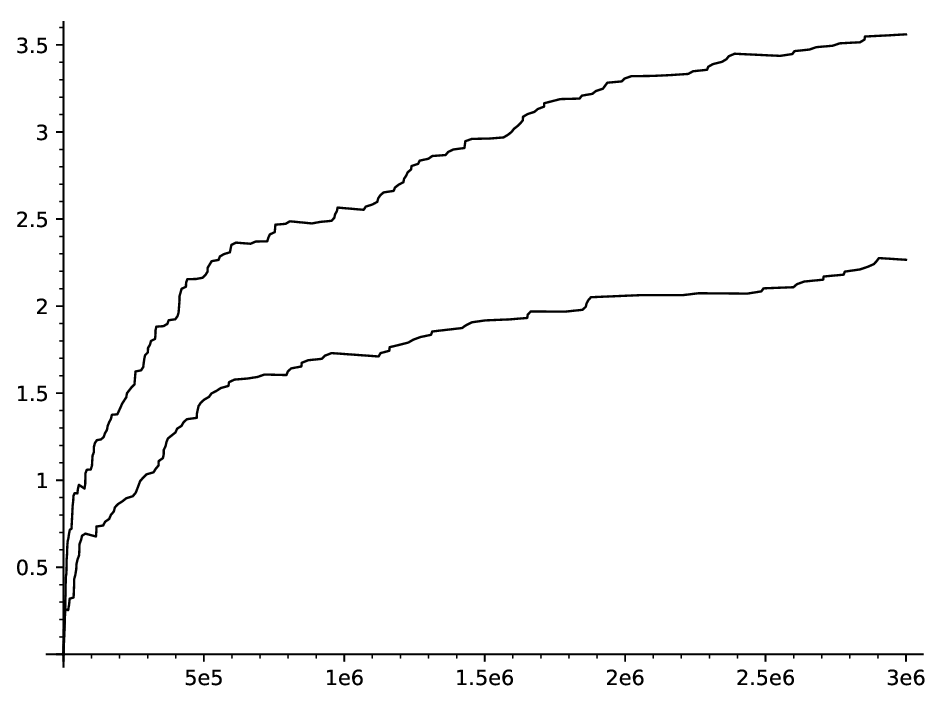}
\caption{$|l| = 25$: Top 25 bottom -25} \label{fig:11_5_A_25}
\end{subfigure}
\caption{Ratio~\eqref{ratio_A_exact} 11a1: $x(X;l)/\log^2(X)$ for $k = 5$} \label{fig:11a1_5_A_exact}
\end{figure}

\clearpage

\begin{figure}[t!] 
\hspace*{-2.3cm}
\begin{subfigure}[b]{0.43\linewidth}
\includegraphics[width=\linewidth]{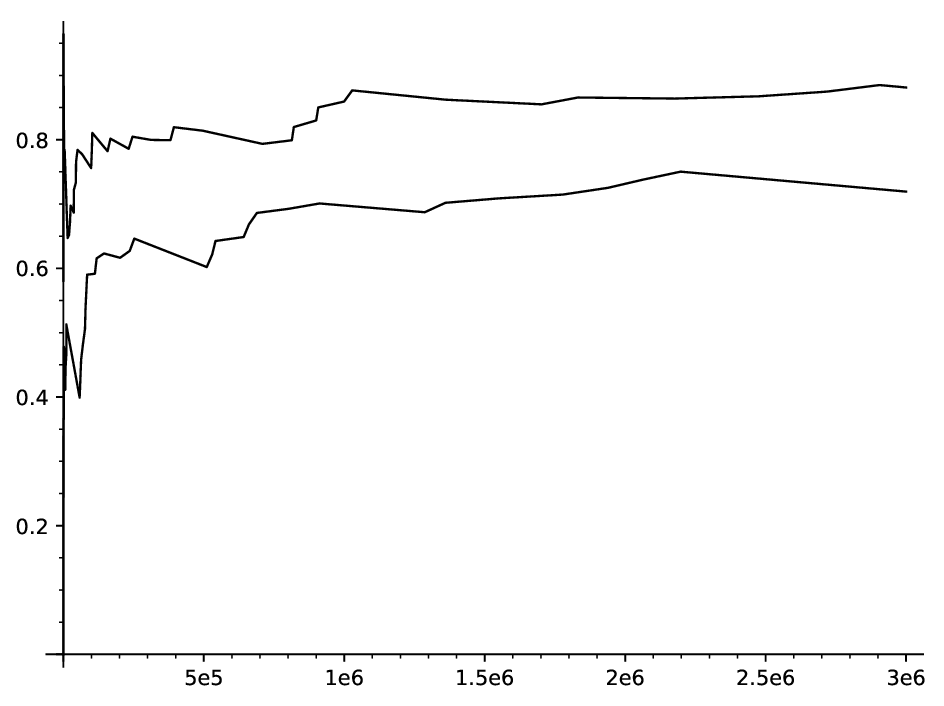}
\caption{$|l| = 1$: Top -1 bottom 1} \label{fig:14_5_A_1}
\end{subfigure}\hspace*{\fill}
\begin{subfigure}[b]{0.43\linewidth}
\includegraphics[width=\linewidth]{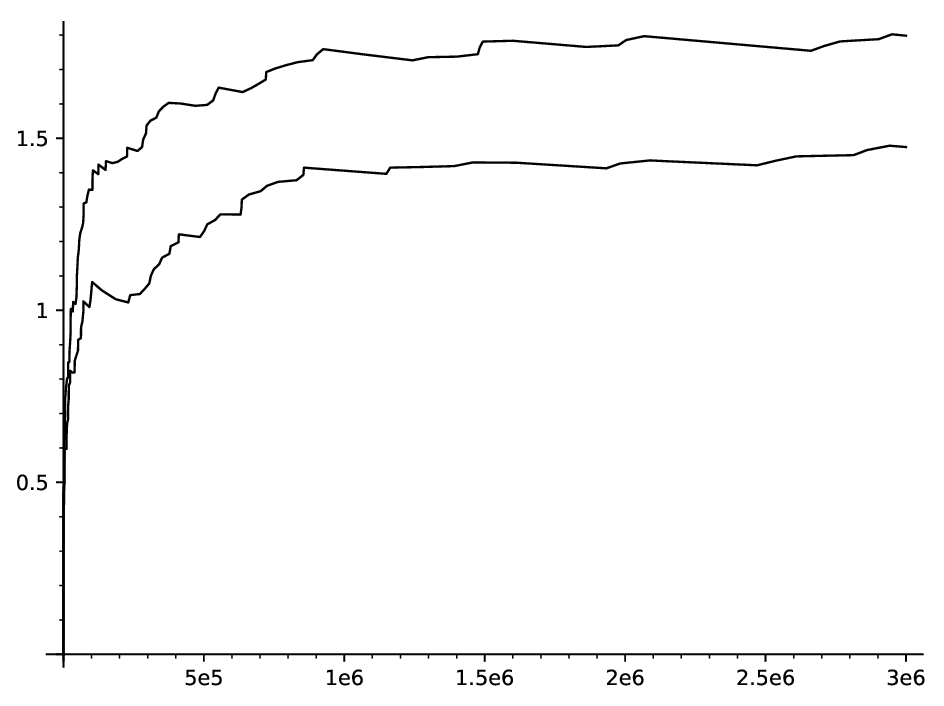}
\caption{$|l| = 4$: Top 4 bottom -4} \label{fig:14_5_A_4}
\end{subfigure}\hspace*{\fill}
\begin{subfigure}[b]{0.43\linewidth}
\includegraphics[width=\linewidth]{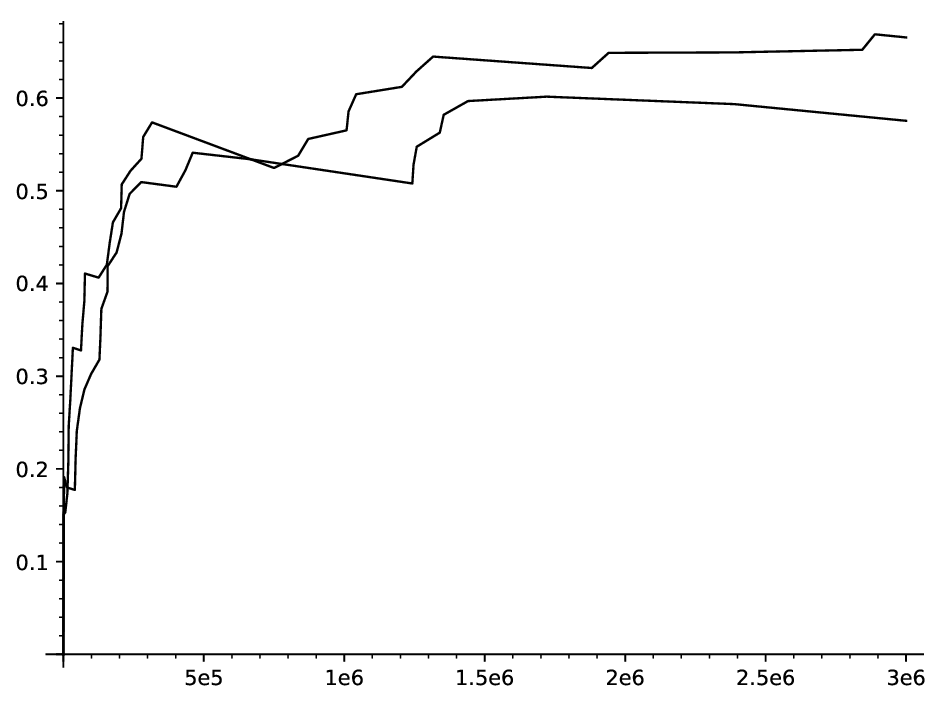}
\caption{$|l| = 5$: Top -5 bottom 5} \label{fig:14_5_A_5}
\end{subfigure}
\hspace*{-2.3cm}
\begin{subfigure}[b]{0.43\linewidth}
\includegraphics[width=\linewidth]{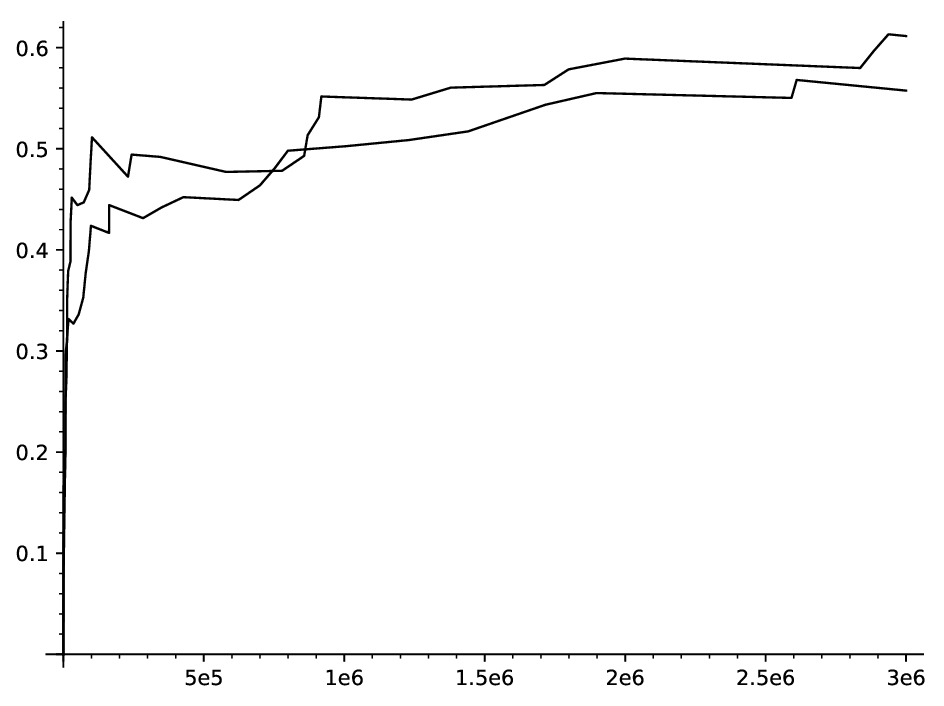}
\caption{$|l| = 9$: Top -9 bottom 9} \label{fig:14_5_A_9}
\end{subfigure}\hspace*{\fill}
\begin{subfigure}[b]{0.43\linewidth}
\includegraphics[width=\linewidth]{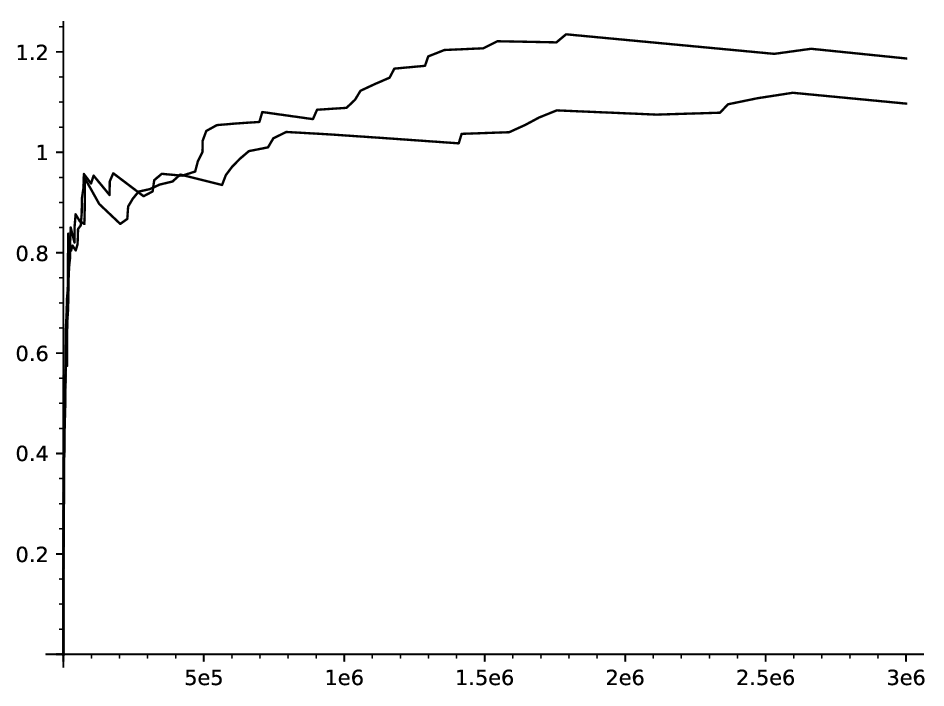}
\caption{$|l| = 11$: Top -11 bottom 11} \label{fig:14_5_A_11}
\end{subfigure}\hspace*{\fill}
\begin{subfigure}[b]{0.43\linewidth}
\includegraphics[width=\linewidth]{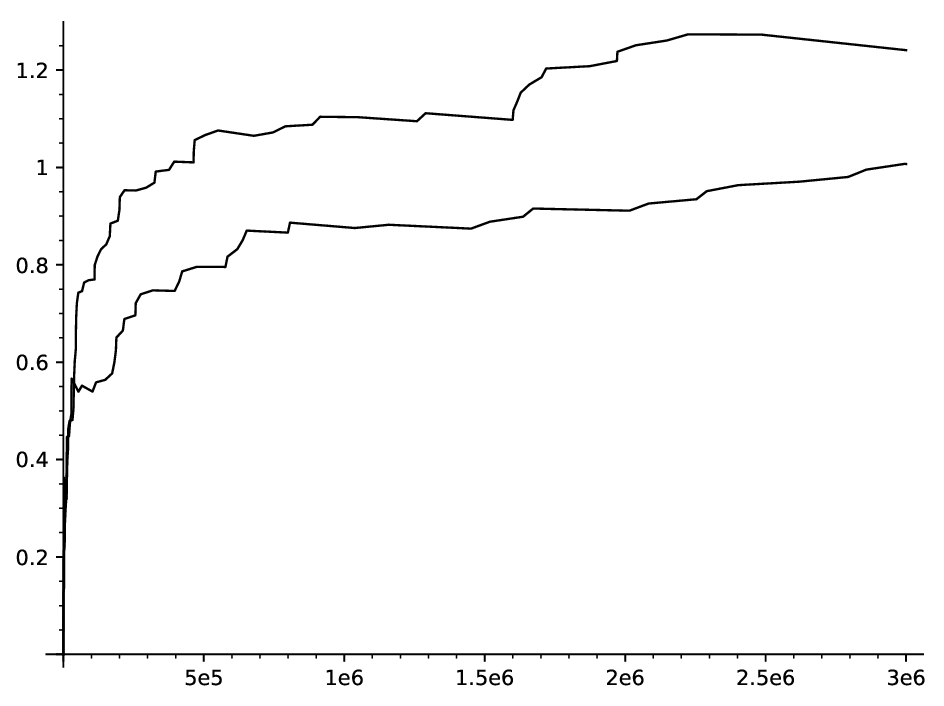}
\caption{$|l| = 16$: Top -16 bottom 16} \label{fig:14_5_A_16}
\end{subfigure}
\hspace*{-2.3cm}
\begin{subfigure}[b]{0.43\linewidth}
\includegraphics[width=\linewidth]{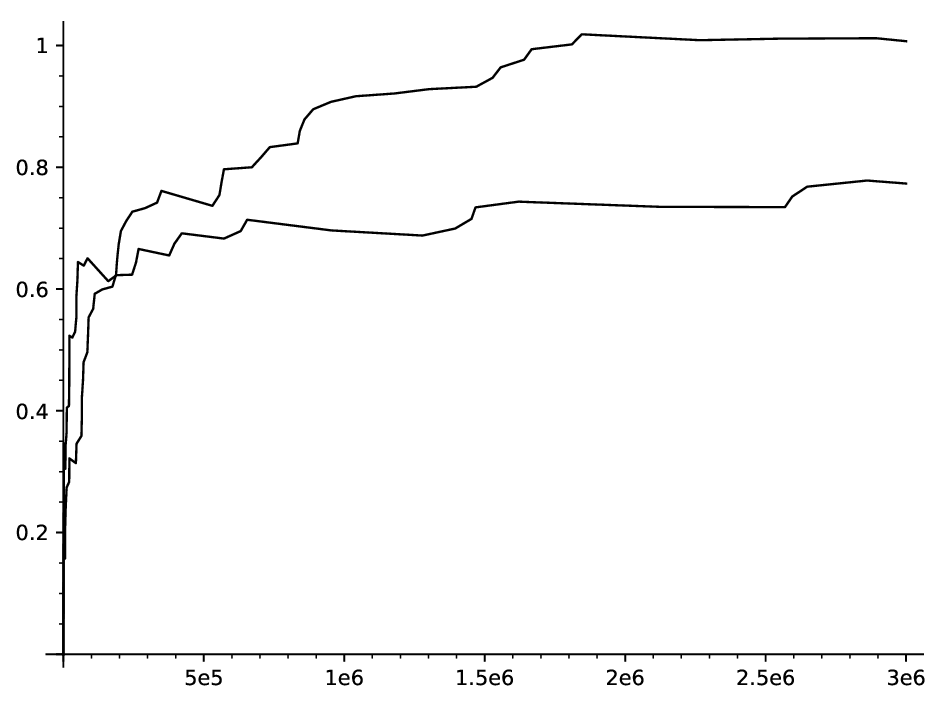}
\caption{$|l| = 19$: Top 19 bottom -19} \label{fig:14_5_A_19}
\end{subfigure}\hspace*{\fill}
\begin{subfigure}[b]{0.43\linewidth}
\includegraphics[width=\linewidth]{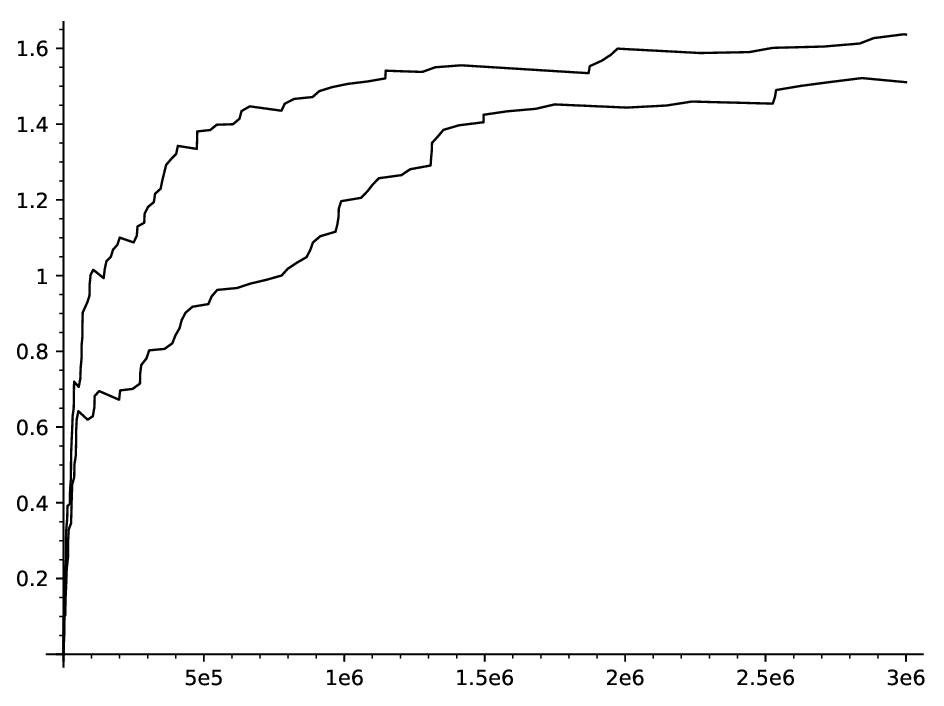}
\caption{$|l| = 20$: Top 20 bottom -20} \label{fig:14_5_A_20}
\end{subfigure}\hspace*{\fill}
\begin{subfigure}[b]{0.43\linewidth}
\includegraphics[width=\linewidth]{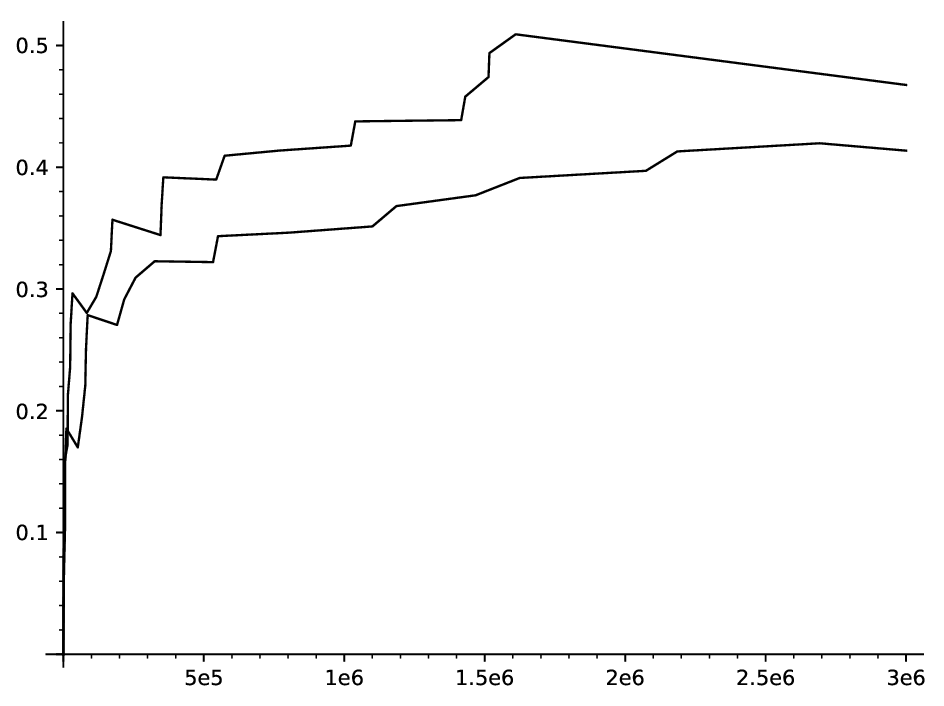}
\caption{$|l| = 25$: Top 25 bottom -25} \label{fig:14_5_A_25}
\end{subfigure}
\caption{Ratio~\eqref{ratio_A_exact} 14a1: $x(X;l)/\log^2(X)$ for $k = 5$} \label{fig:14a1_5_A_exact}
\end{figure}

\clearpage

\begin{figure}[t!] 
\hspace*{-2.3cm}
\begin{subfigure}[b]{0.43\linewidth}
\includegraphics[width=\linewidth]{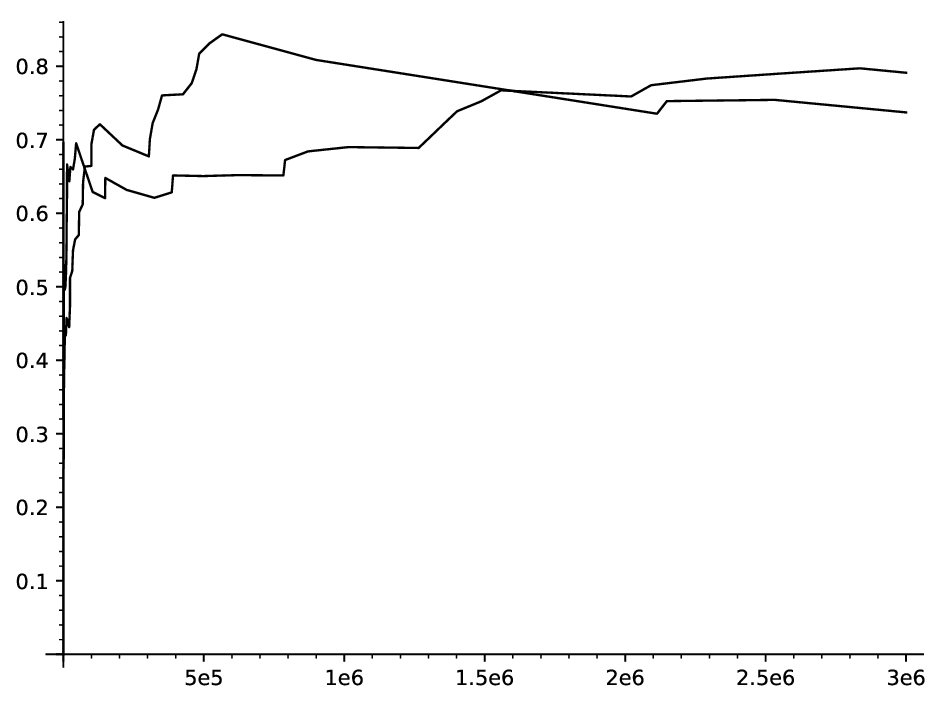}
\caption{$|l| = 1$: Top -1 bottom 1} \label{fig:15_5_A_1}
\end{subfigure}\hspace*{\fill}
\begin{subfigure}[b]{0.43\linewidth}
\includegraphics[width=\linewidth]{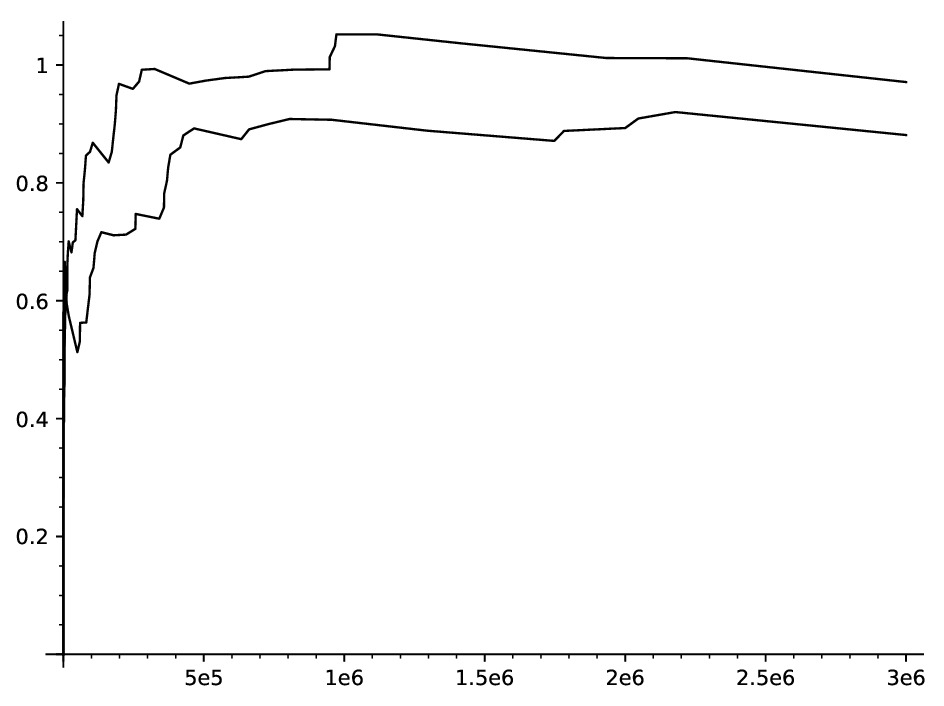}
\caption{$|l| = 4$: Top 4 bottom -4} \label{fig:15_5_A_4}
\end{subfigure}\hspace*{\fill}
\begin{subfigure}[b]{0.43\linewidth}
\includegraphics[width=\linewidth]{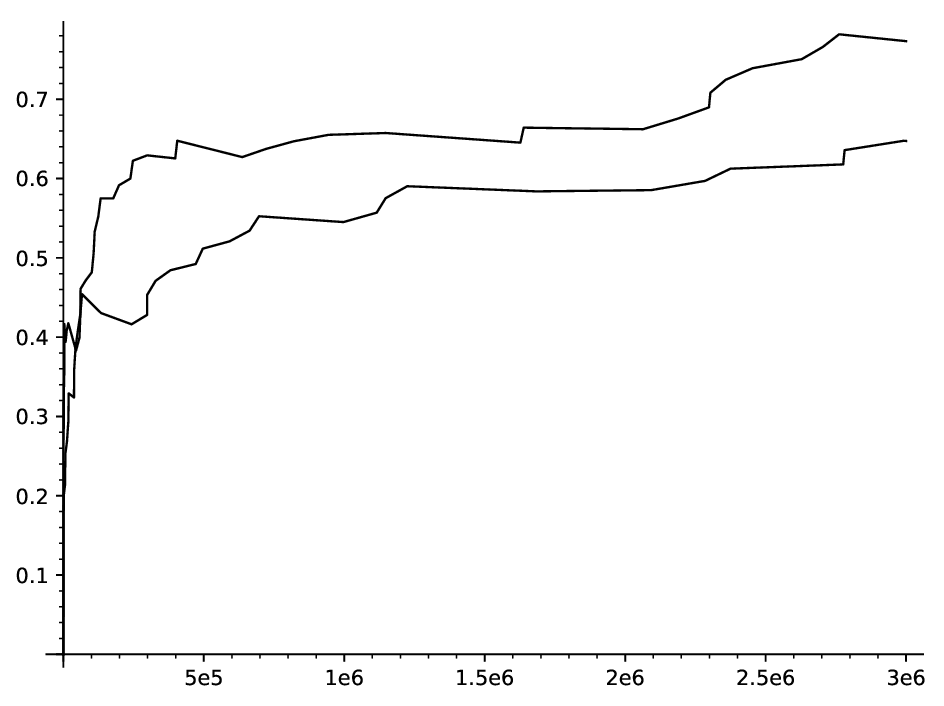}
\caption{$|l| = 5$: Top -5 bottom 5} \label{fig:15_5_A_5}
\end{subfigure}
\hspace*{-2.3cm}
\begin{subfigure}[b]{0.43\linewidth}
\includegraphics[width=\linewidth]{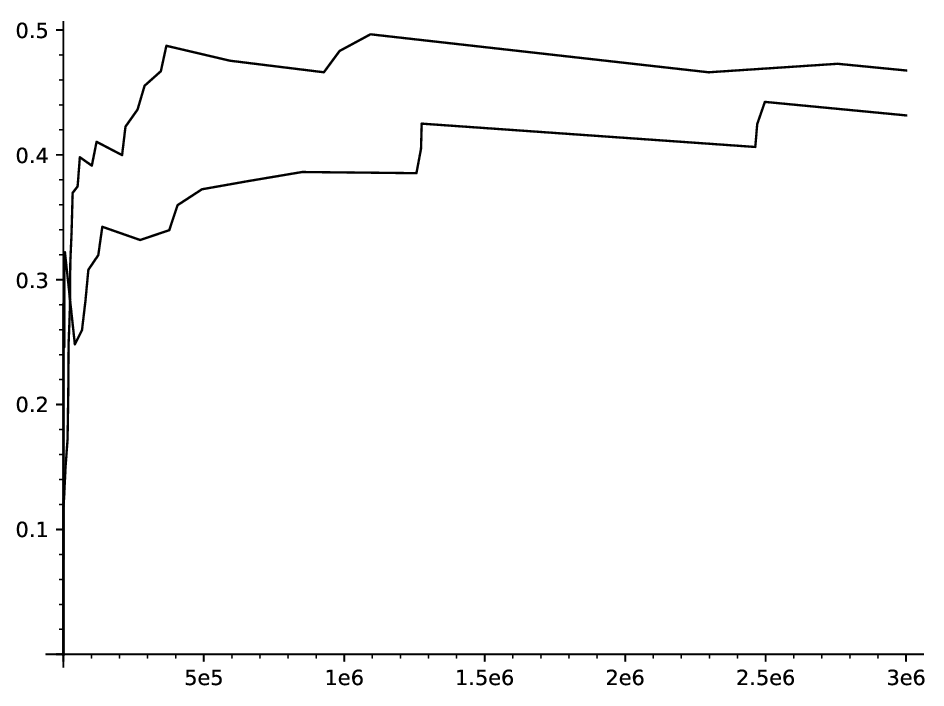}
\caption{$|l| = 9$: Top -9 bottom 9} \label{fig:15_5_A_9}
\end{subfigure}\hspace*{\fill}
\begin{subfigure}[b]{0.43\linewidth}
\includegraphics[width=\linewidth]{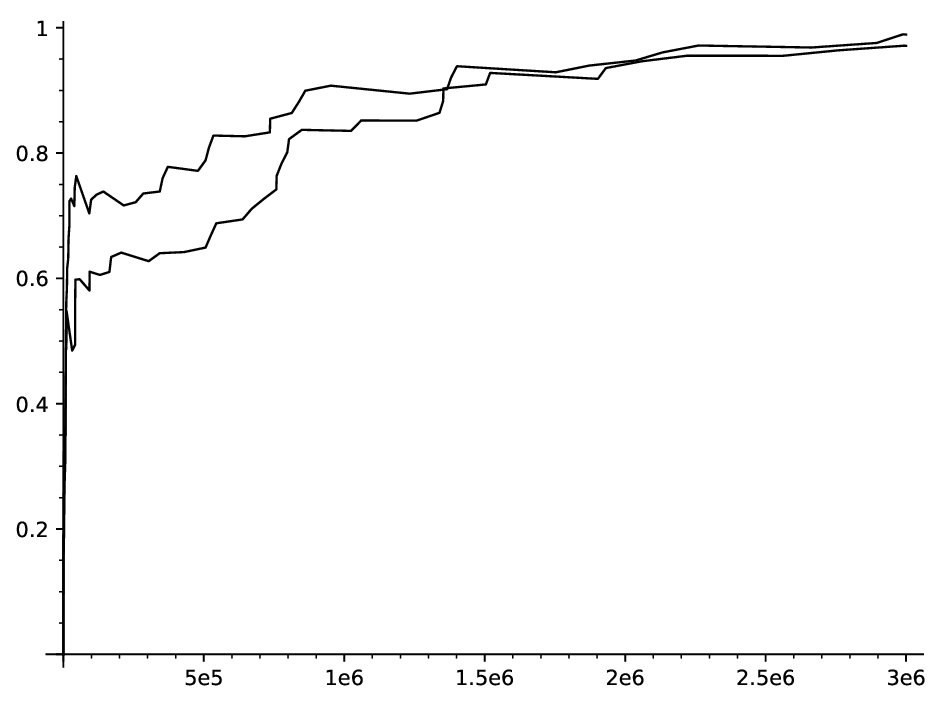}
\caption{$|l| = 11$: Top 11 bottom -11} \label{fig:15_5_A_11}
\end{subfigure}\hspace*{\fill}
\begin{subfigure}[b]{0.43\linewidth}
\includegraphics[width=\linewidth]{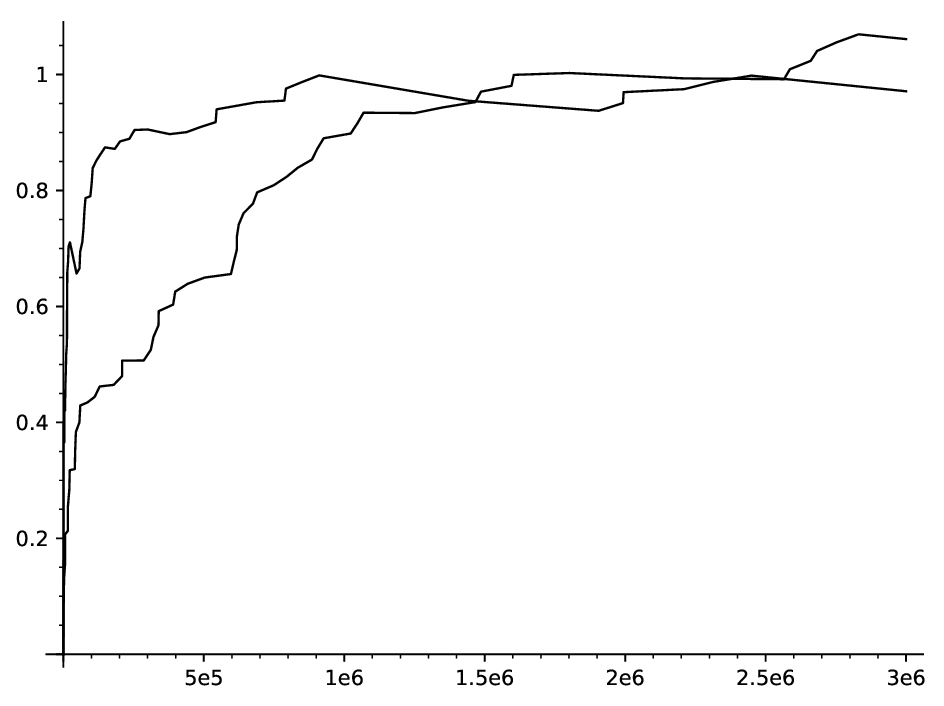}
\caption{$|l| = 16$: Top 16 bottom -16} \label{fig:15_5_A_16}
\end{subfigure}
\hspace*{-2.3cm}
\begin{subfigure}[b]{0.43\linewidth}
\includegraphics[width=\linewidth]{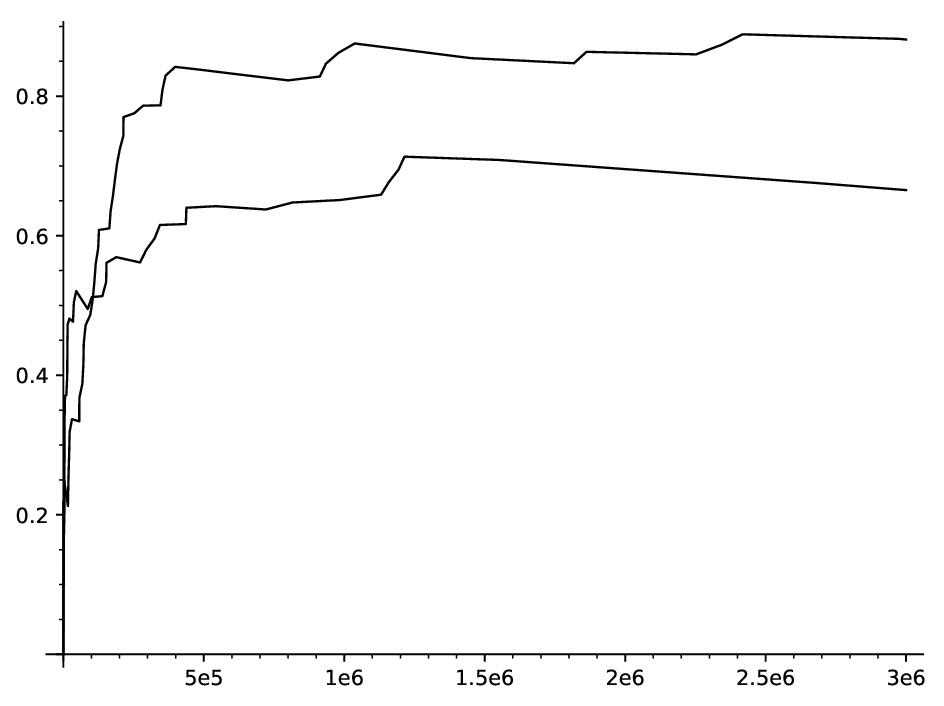}
\caption{$|l| = 19$: Top 19 bottom -19} \label{fig:15_5_A_19}
\end{subfigure}\hspace*{\fill}
\begin{subfigure}[b]{0.43\linewidth}
\includegraphics[width=\linewidth]{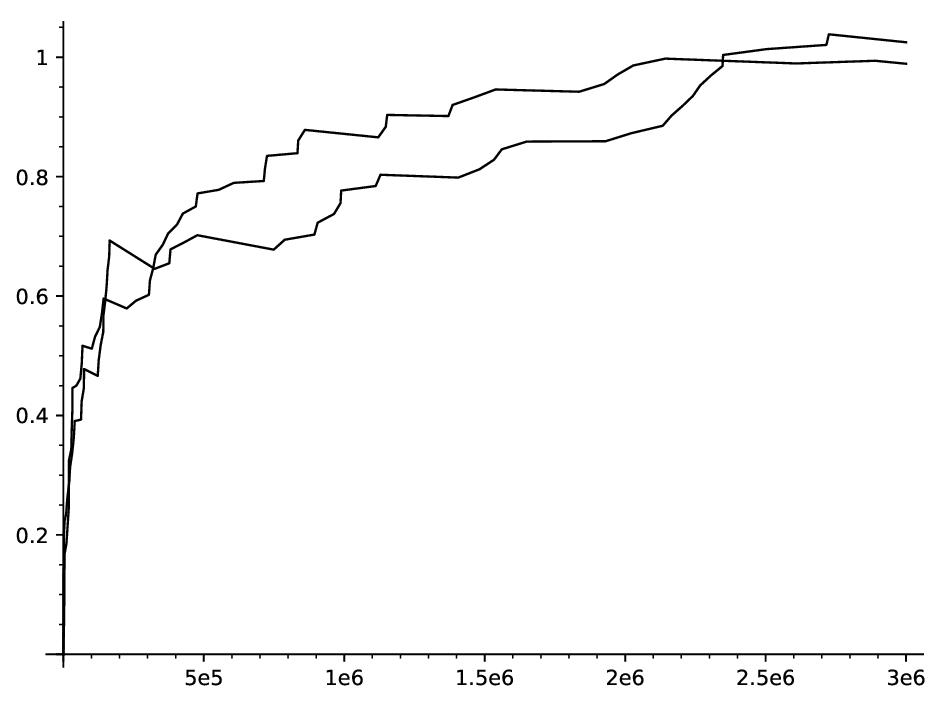}
\caption{$|l| = 20$: Top 20 bottom -20} \label{fig:15_5_A_20}
\end{subfigure}\hspace*{\fill}
\begin{subfigure}[b]{0.43\linewidth}
\includegraphics[width=\linewidth]{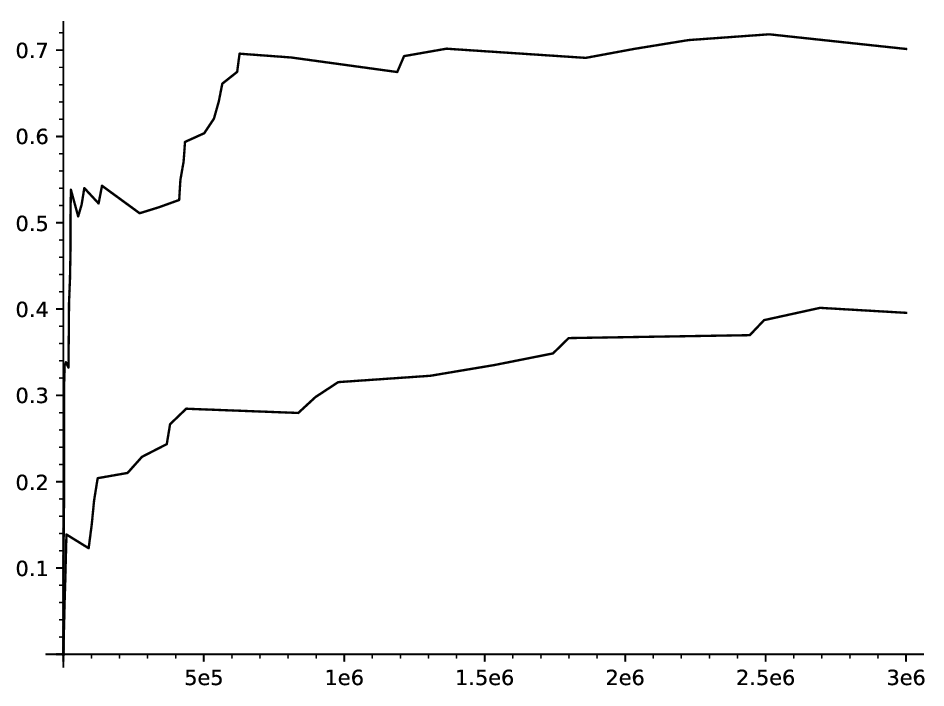}
\caption{$|l| = 25$: Top -25 bottom 25} \label{fig:15_5_A_25}
\end{subfigure}
\caption{Ratio~\eqref{ratio_A_exact} 15a1: $x(X;l)/\log^2(X)$ for $k = 5$} \label{fig:15a1_5_A_exact}
\end{figure}

\clearpage

\begin{figure}[t!] 
\hspace*{-2.3cm}
\begin{subfigure}[b]{0.43\linewidth}
\includegraphics[width=\linewidth]{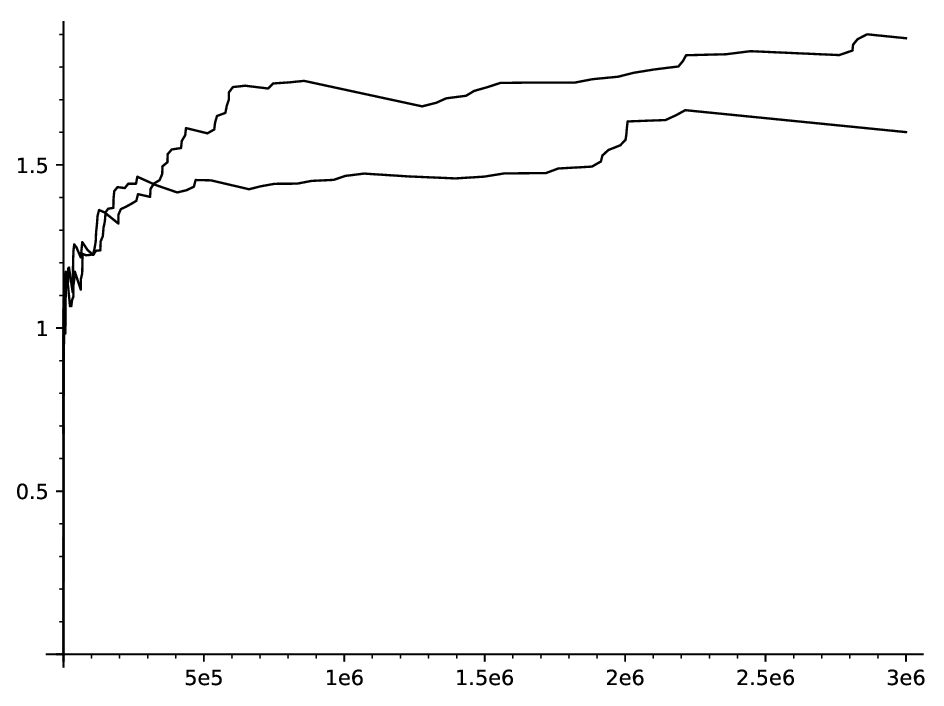}
\caption{$|l| = 1$: Top -1 bottom 1} \label{fig:17_5_A_1}
\end{subfigure}\hspace*{\fill}
\begin{subfigure}[b]{0.43\linewidth}
\includegraphics[width=\linewidth]{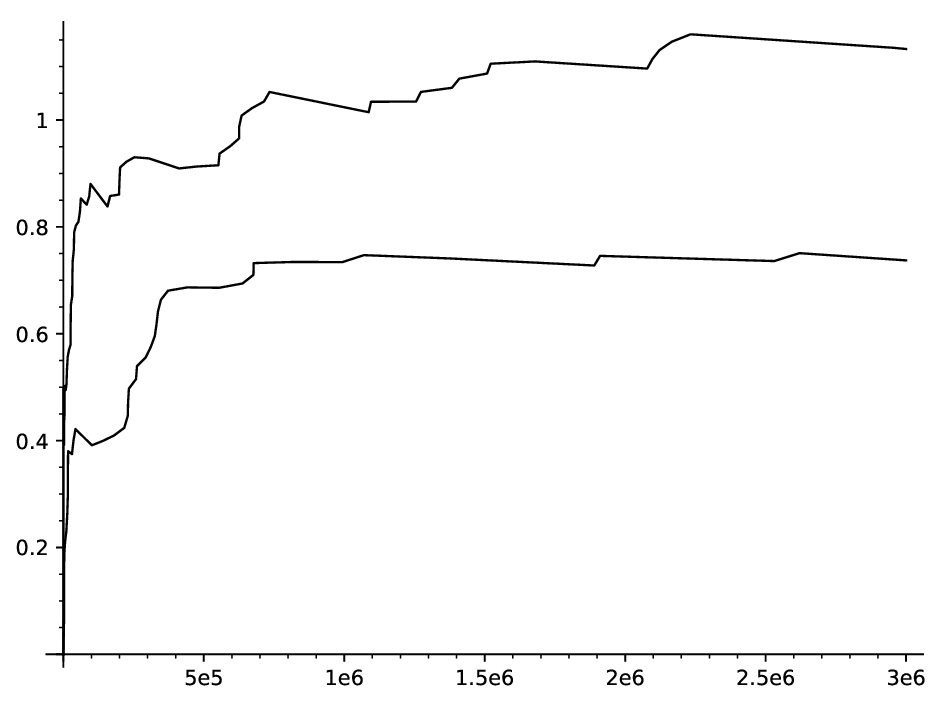}
\caption{$|l| = 4$: Top 4 bottom -4} \label{fig:17_5_A_4}
\end{subfigure}\hspace*{\fill}
\begin{subfigure}[b]{0.43\linewidth}
\includegraphics[width=\linewidth]{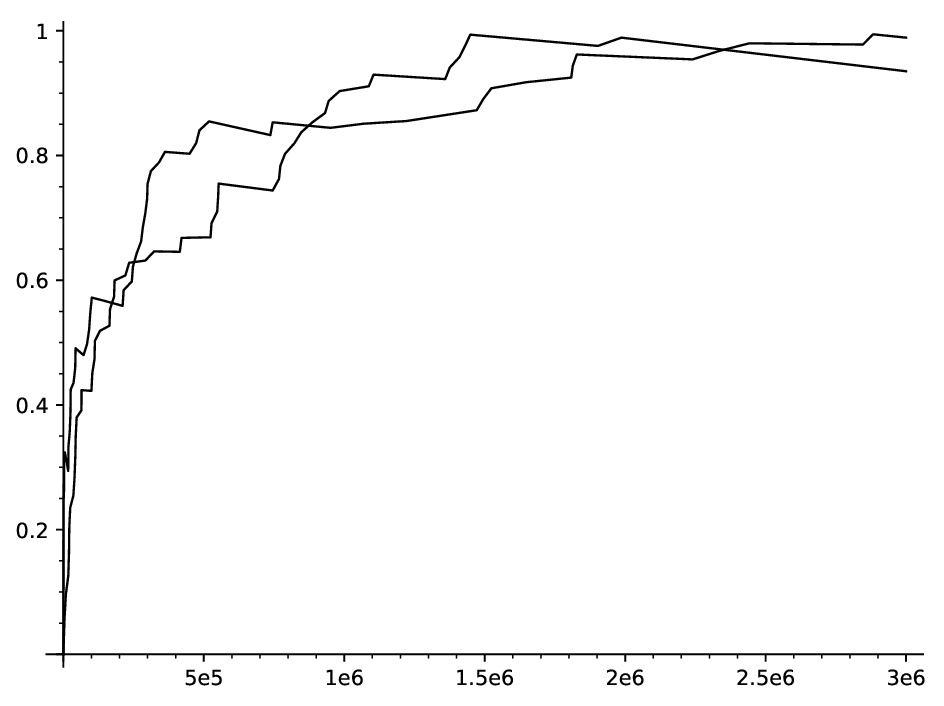}
\caption{$|l| = 5$: Top 5 bottom -5} \label{fig:17_5_A_5}
\end{subfigure}
\hspace*{-2.3cm}
\begin{subfigure}[b]{0.43\linewidth}
\includegraphics[width=\linewidth]{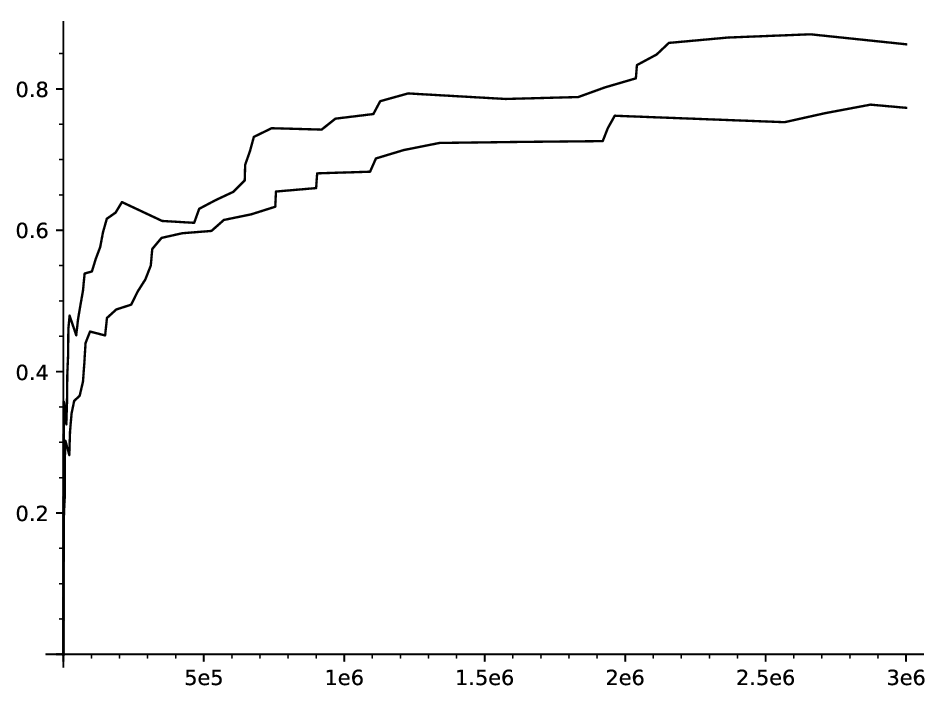}
\caption{$|l| = 9$: Top 9 bottom -9} \label{fig:17_5_A_9}
\end{subfigure}\hspace*{\fill}
\begin{subfigure}[b]{0.43\linewidth}
\includegraphics[width=\linewidth]{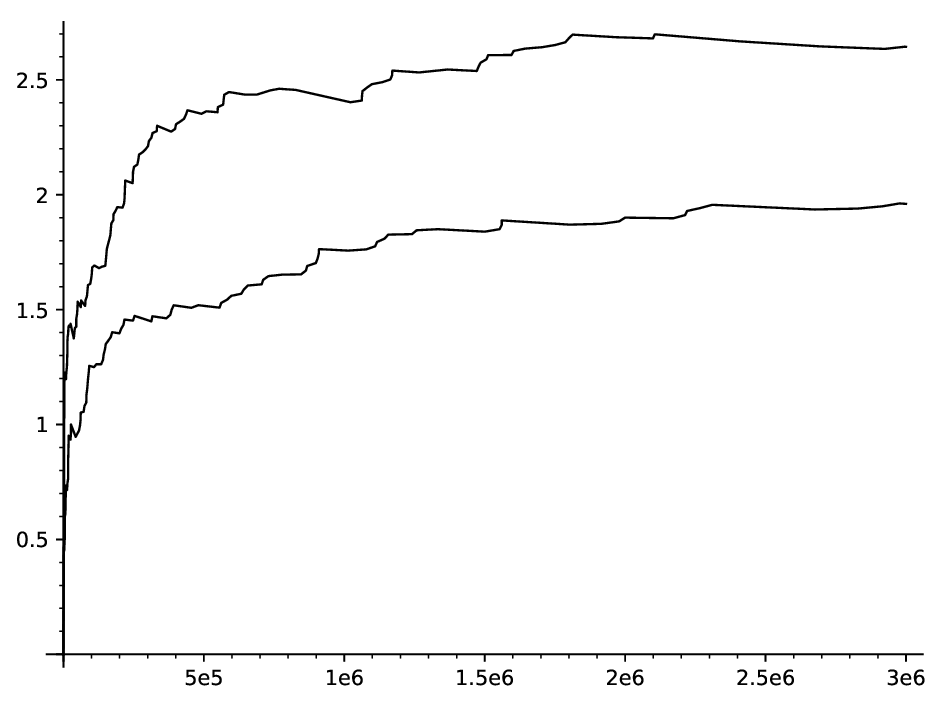}
\caption{$|l| = 11$: Top 11 bottom -11} \label{fig:17_5_A_11}
\end{subfigure}\hspace*{\fill}
\begin{subfigure}[b]{0.43\linewidth}
\includegraphics[width=\linewidth]{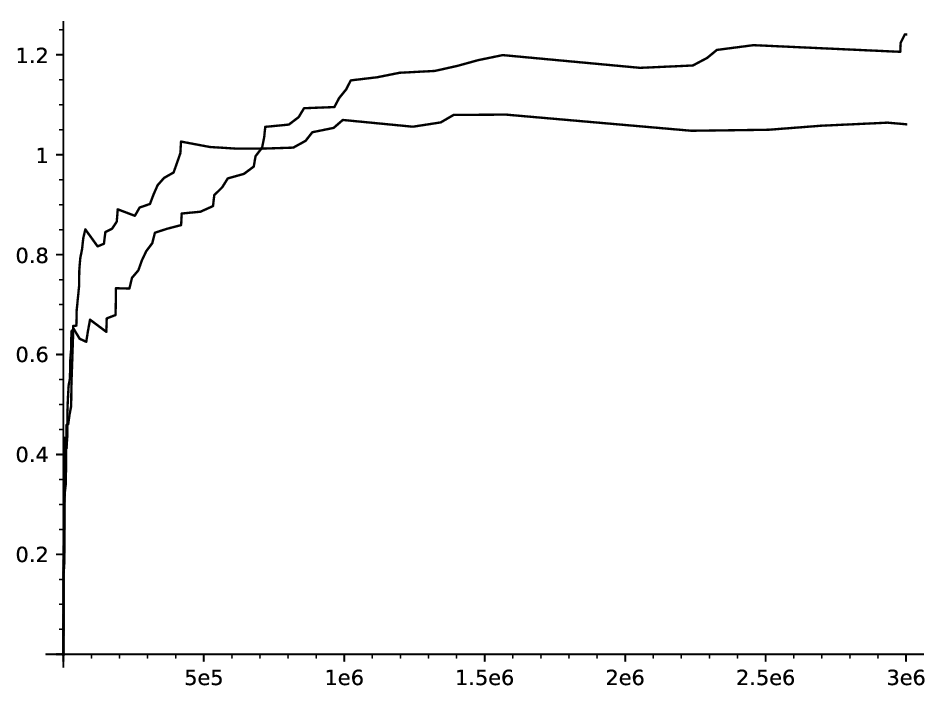}
\caption{$|l| = 16$: Top 16 bottom -16} \label{fig:17_5_A_16}
\end{subfigure}
\hspace*{-2.3cm}
\begin{subfigure}[b]{0.43\linewidth}
\includegraphics[width=\linewidth]{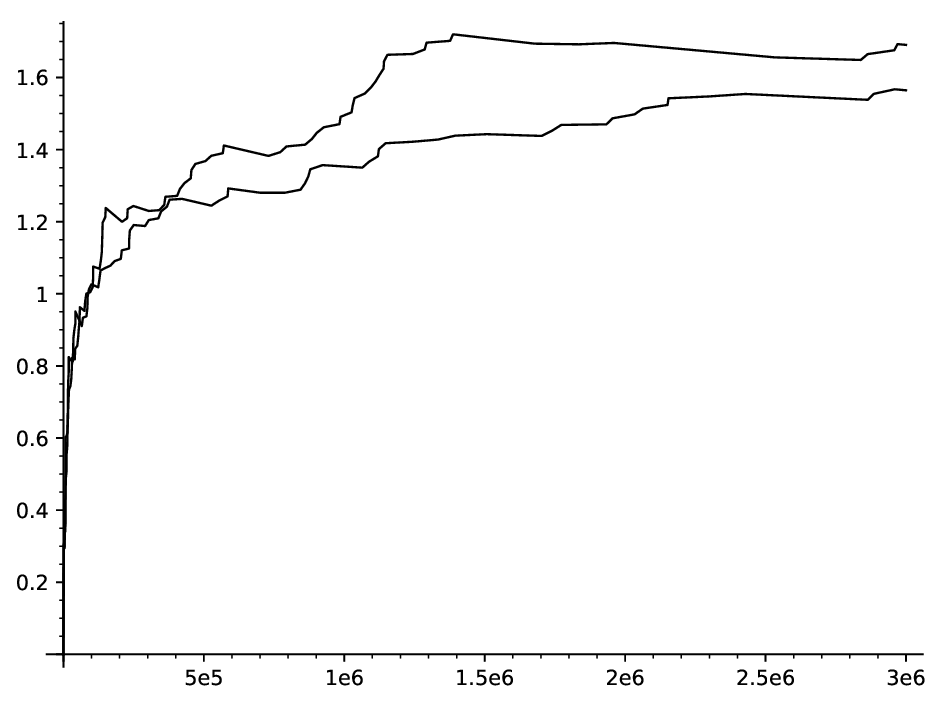}
\caption{$|l| = 19$: Top 19 bottom -19} \label{fig:17_5_A_19}
\end{subfigure}\hspace*{\fill}
\begin{subfigure}[b]{0.43\linewidth}
\includegraphics[width=\linewidth]{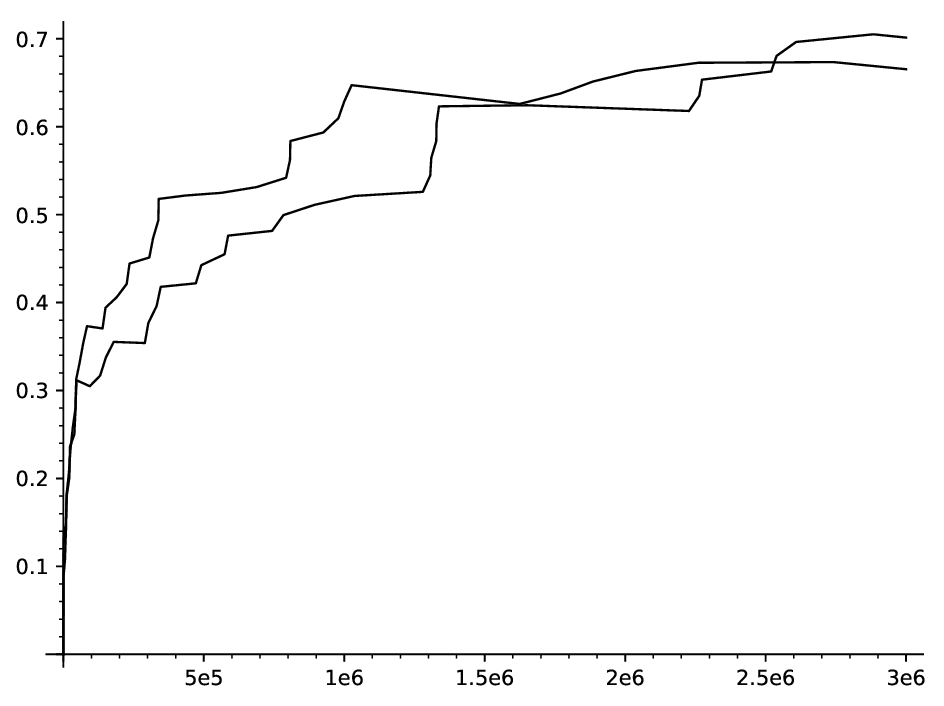}
\caption{$|l| = 20$: Top -20 bottom 20} \label{fig:17_5_A_20}
\end{subfigure}\hspace*{\fill}
\begin{subfigure}[b]{0.43\linewidth}
\includegraphics[width=\linewidth]{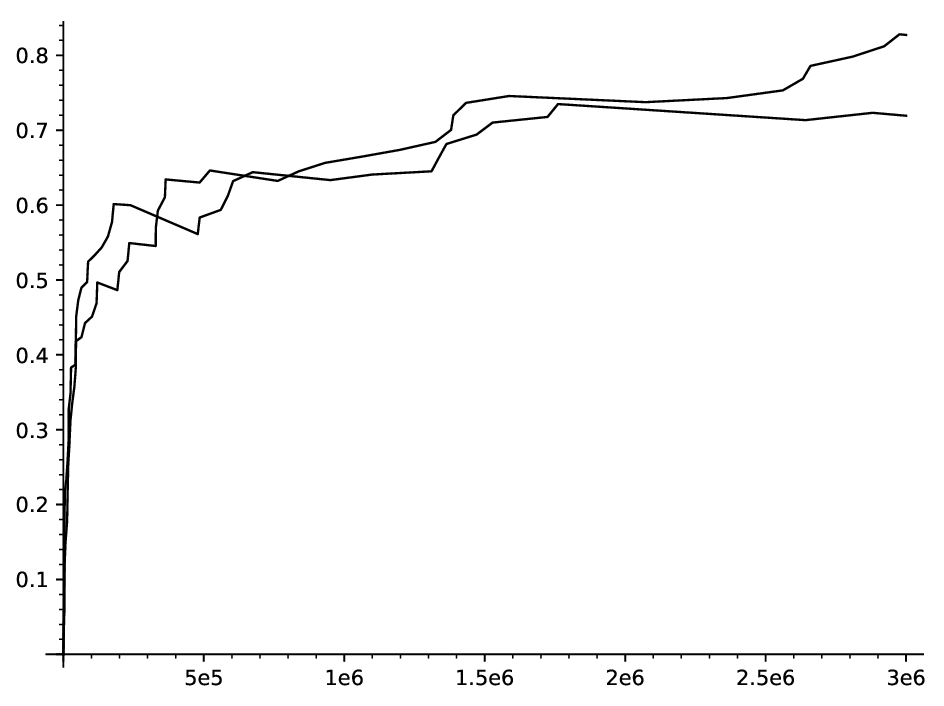}
\caption{$|l| = 25$: Top 25 bottom -25} \label{fig:17_5_A_25}
\end{subfigure}
\caption{Ratio~\eqref{ratio_A_exact} 17a1: $x(X;l)/\log^2(X)$ for $k = 5$} \label{fig:17a1_5_A_exact}
\end{figure}

\clearpage

\begin{figure}[t!] 
\hspace*{-2.3cm}
\begin{subfigure}[b]{0.43\linewidth}
\includegraphics[width=\linewidth]{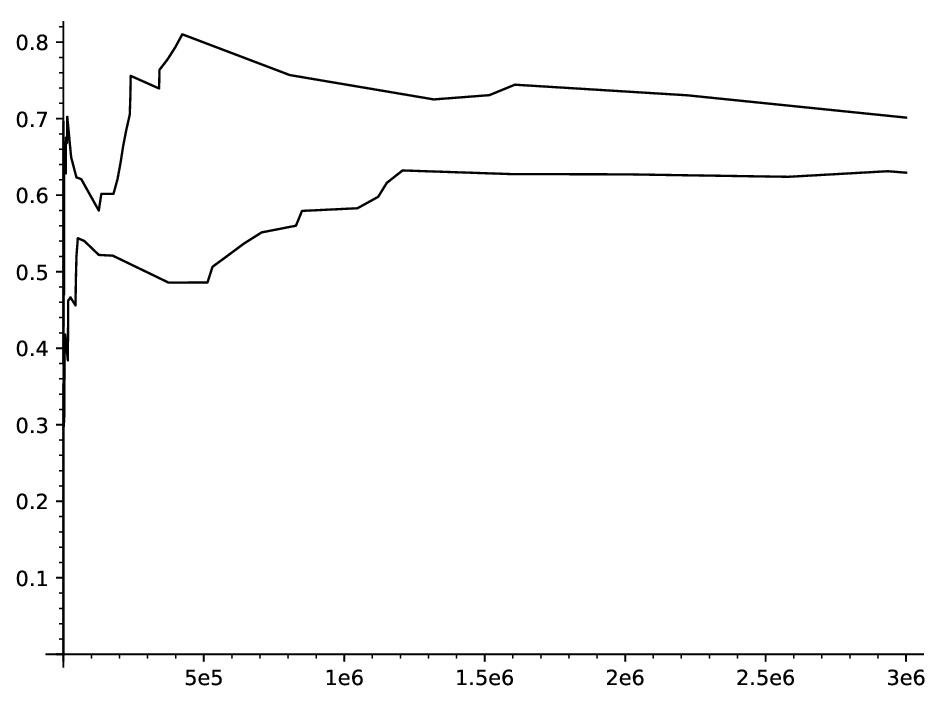}
\caption{$|l| = 1$: Top -1 bottom 1} \label{fig:19_5_A_1}
\end{subfigure}\hspace*{\fill}
\begin{subfigure}[b]{0.43\linewidth}
\includegraphics[width=\linewidth]{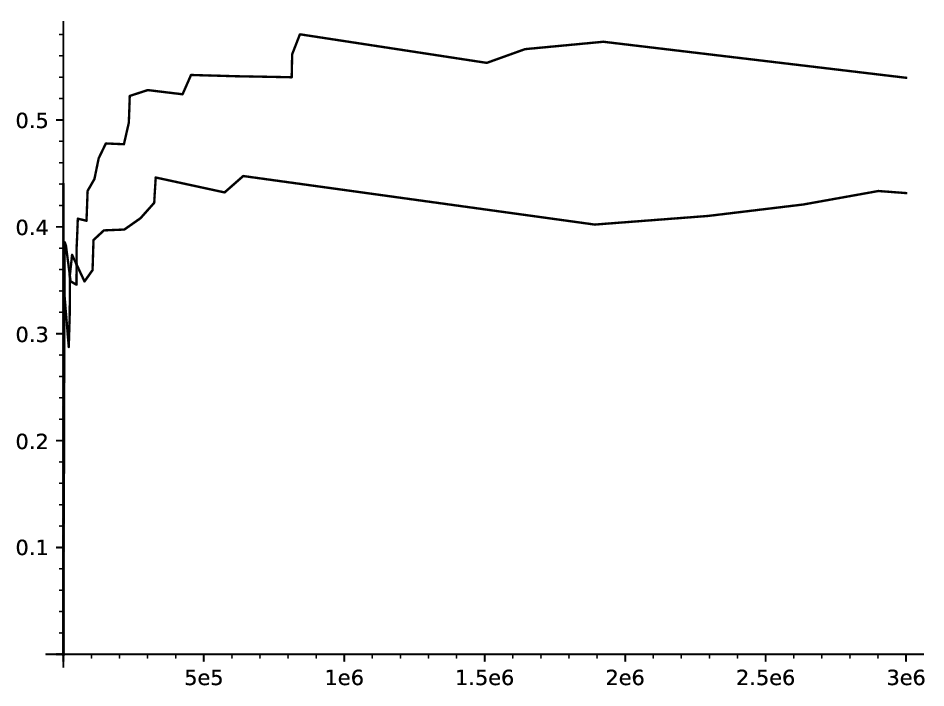}
\caption{$|l| = 4$: Top -4 bottom 4} \label{fig:19_5_A_4}
\end{subfigure}\hspace*{\fill}
\begin{subfigure}[b]{0.43\linewidth}
\includegraphics[width=\linewidth]{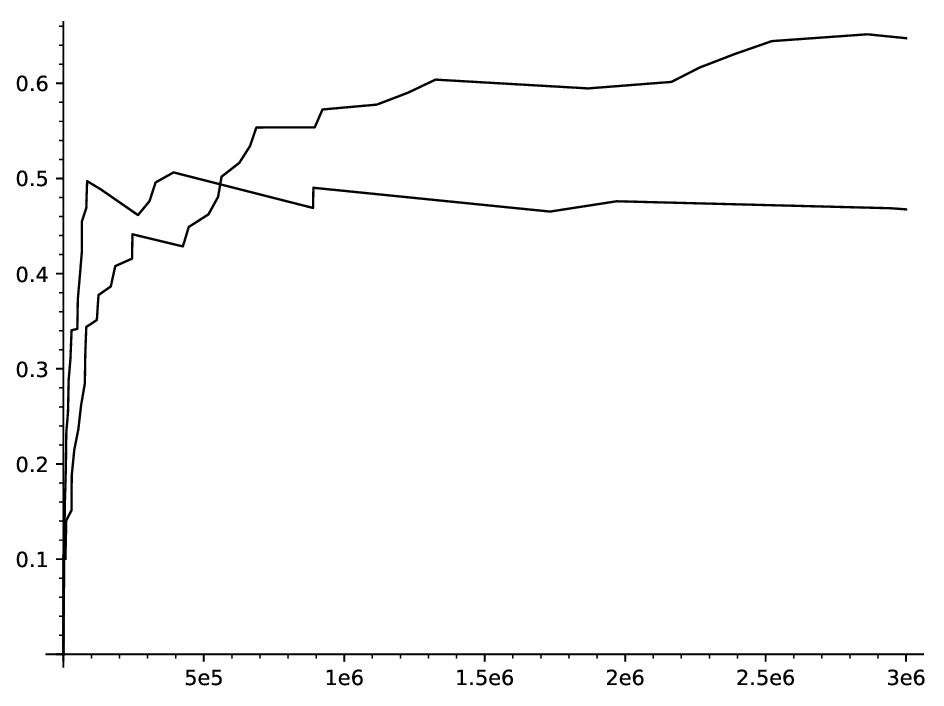}
\caption{$|l| = 5$: Top 5 bottom -5} \label{fig:19_5_A_5}
\end{subfigure}
\hspace*{-2.3cm}
\begin{subfigure}[b]{0.43\linewidth}
\includegraphics[width=\linewidth]{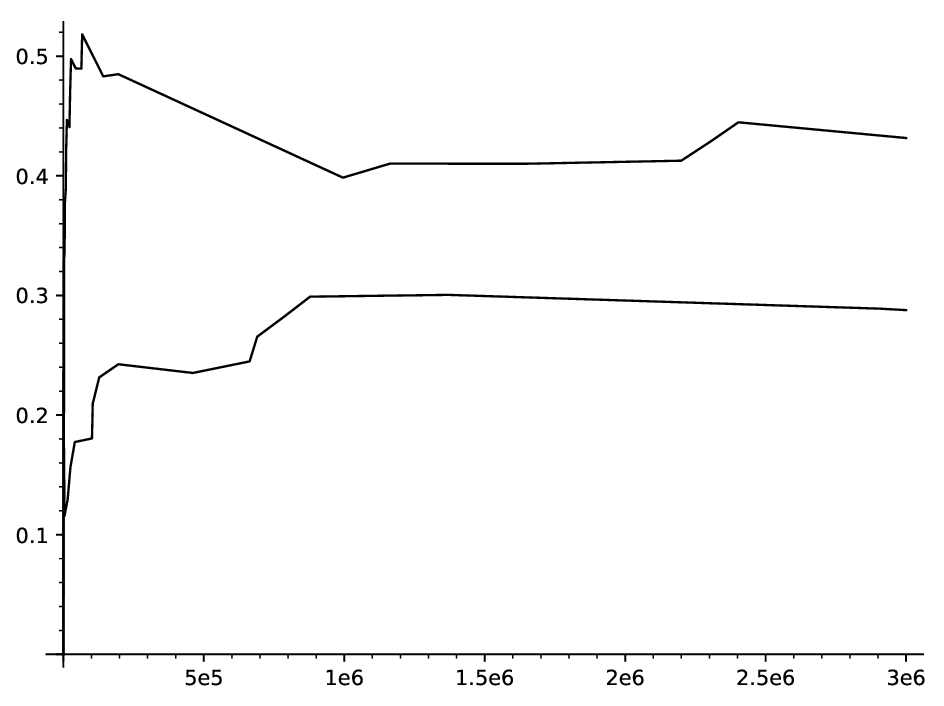}
\caption{$|l| = 9$: Top 9 bottom -9} \label{fig:19_5_A_9}
\end{subfigure}\hspace*{\fill}
\begin{subfigure}[b]{0.43\linewidth}
\includegraphics[width=\linewidth]{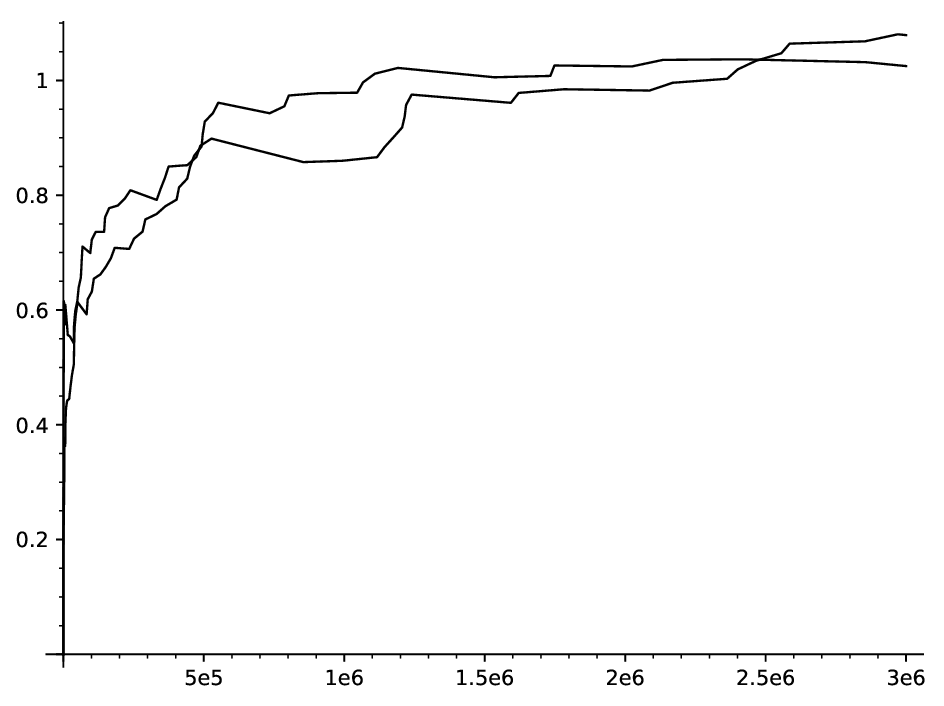}
\caption{$|l| = 11$: Top -11 bottom 11} \label{fig:19_5_A_11}
\end{subfigure}\hspace*{\fill}
\begin{subfigure}[b]{0.43\linewidth}
\includegraphics[width=\linewidth]{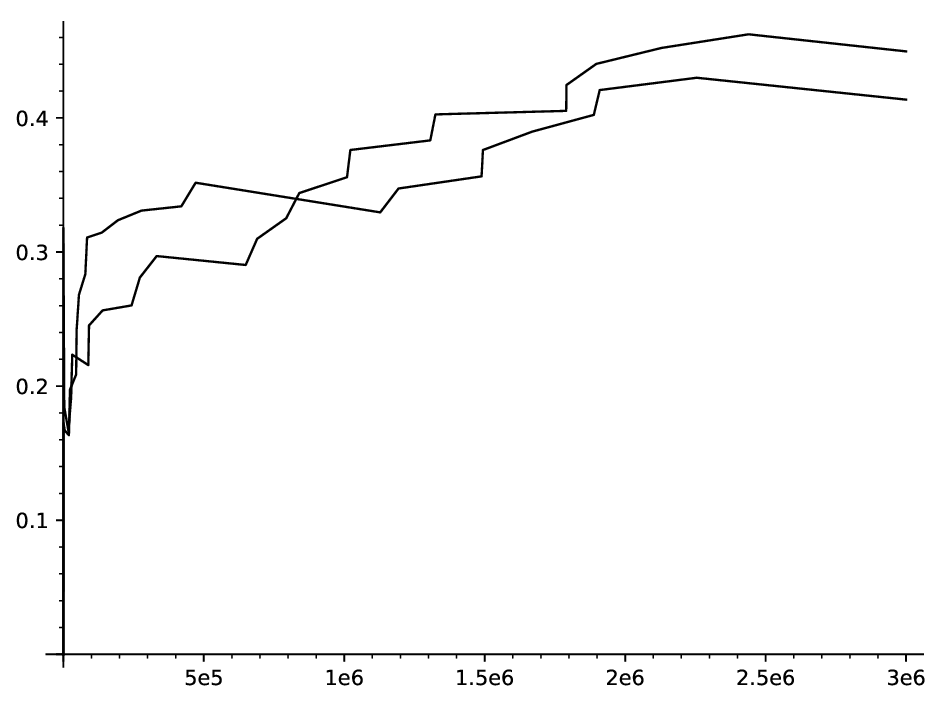}
\caption{$|l| = 16$: Top -16 bottom 16} \label{fig:19_5_A_16}
\end{subfigure}
\hspace*{-2.3cm}
\begin{subfigure}[b]{0.43\linewidth}
\includegraphics[width=\linewidth]{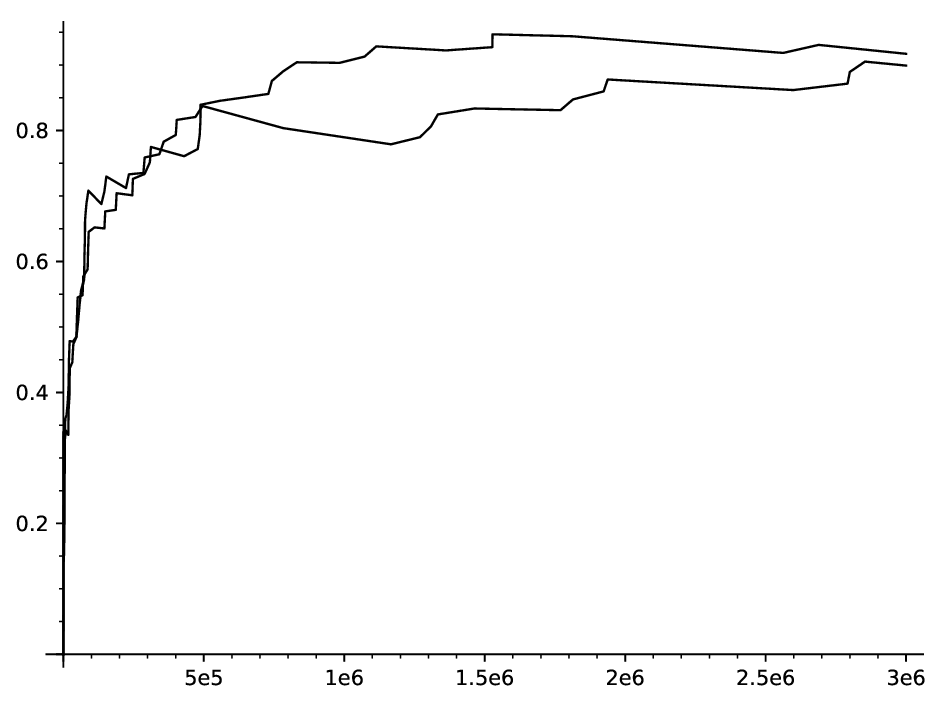}
\caption{$|l| = 19$: Top -19 bottom 19} \label{fig:19_5_A_19}
\end{subfigure}\hspace*{\fill}
\begin{subfigure}[b]{0.43\linewidth}
\includegraphics[width=\linewidth]{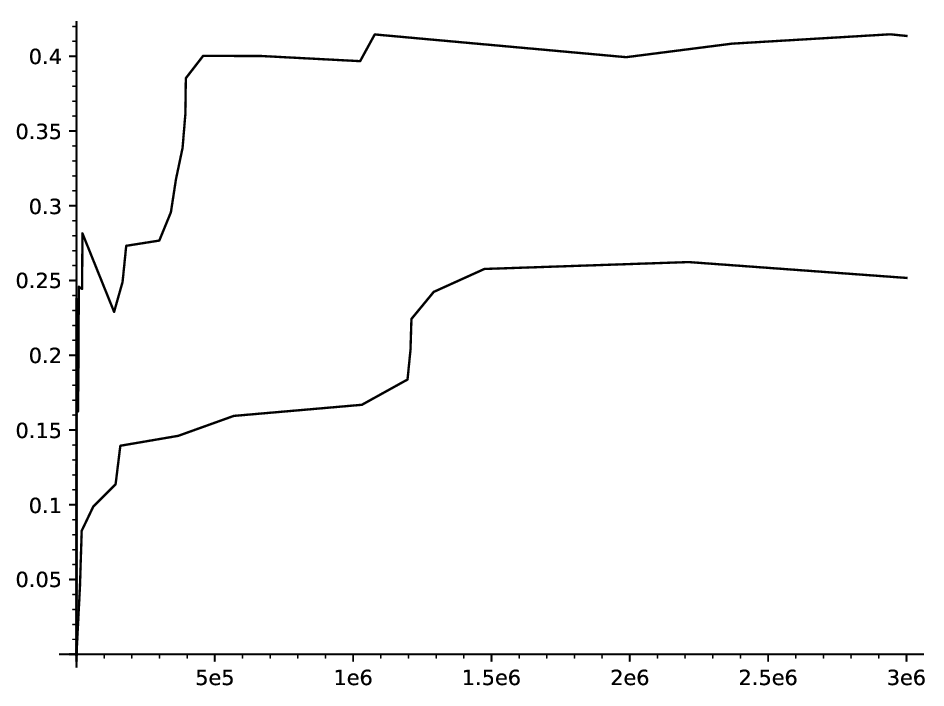}
\caption{$|l| = 20$: Top -20 bottom 20} \label{fig:19_5_A_20}
\end{subfigure}\hspace*{\fill}
\begin{subfigure}[b]{0.43\linewidth}
\includegraphics[width=\linewidth]{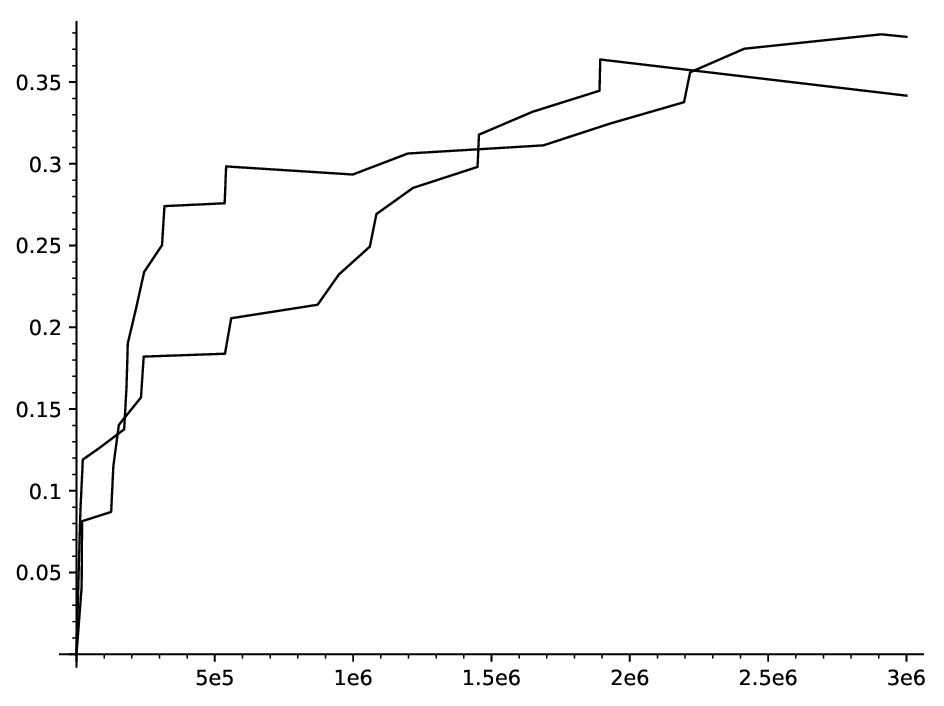}
\caption{$|l| = 25$: Top -25 bottom 25} \label{fig:19_5_A_25}
\end{subfigure}
\caption{Ratio~\eqref{ratio_A_exact} 19a1: $x(X;l)/\log^2(X)$ for $k = 5$} \label{fig:19a1_5_A_exact}
\end{figure}

\clearpage

\begin{figure}[t!] 
\hspace*{-2.3cm}
\begin{subfigure}[b]{0.43\linewidth}
\includegraphics[width=\linewidth]{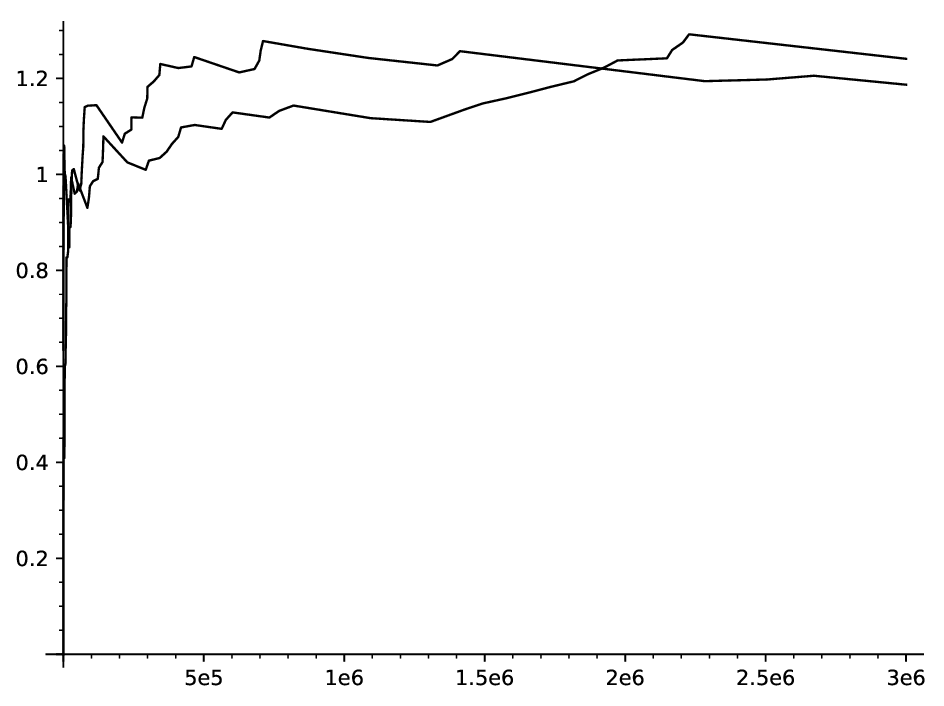}
\caption{$|l| = 1$: Top 1 bottom -1} \label{fig:37_5_A_1}
\end{subfigure}\hspace*{\fill}
\begin{subfigure}[b]{0.43\linewidth}
\includegraphics[width=\linewidth]{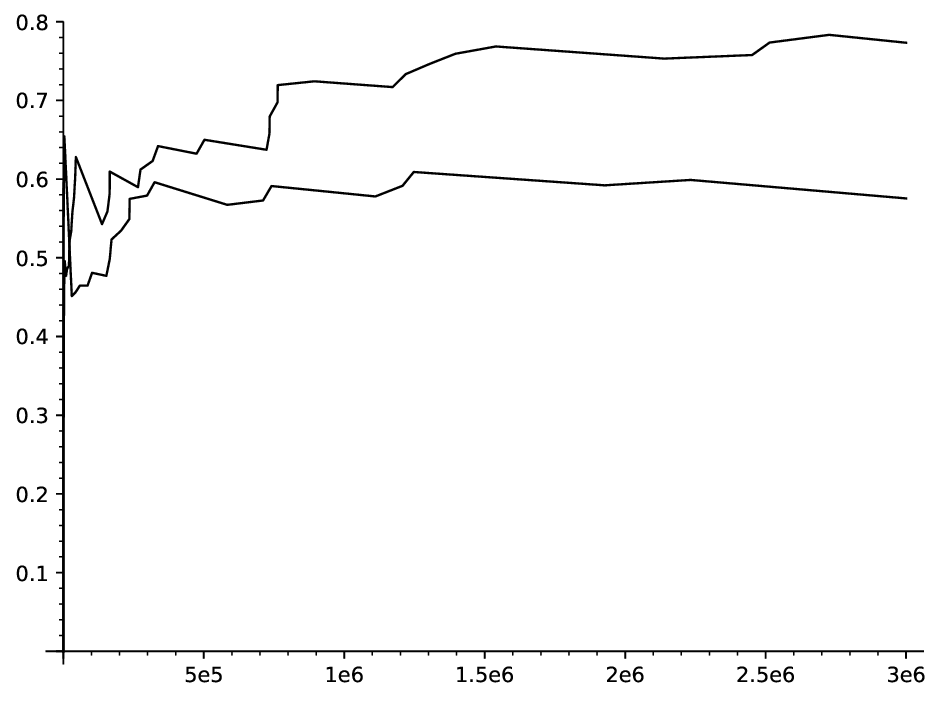}
\caption{$|l| = 4$: Top -4 bottom 4} \label{fig:37_5_A_4}
\end{subfigure}\hspace*{\fill}
\begin{subfigure}[b]{0.43\linewidth}
\includegraphics[width=\linewidth]{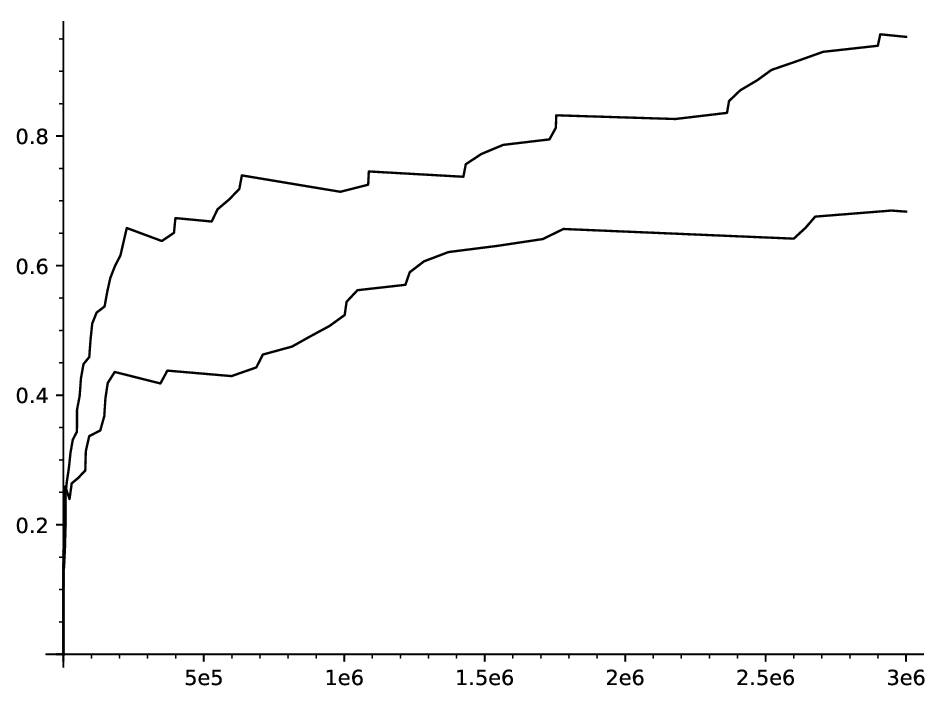}
\caption{$|l| = 5$: Top -5 bottom 5} \label{fig:37_5_A_5}
\end{subfigure}
\hspace*{-2.3cm}
\begin{subfigure}[b]{0.43\linewidth}
\includegraphics[width=\linewidth]{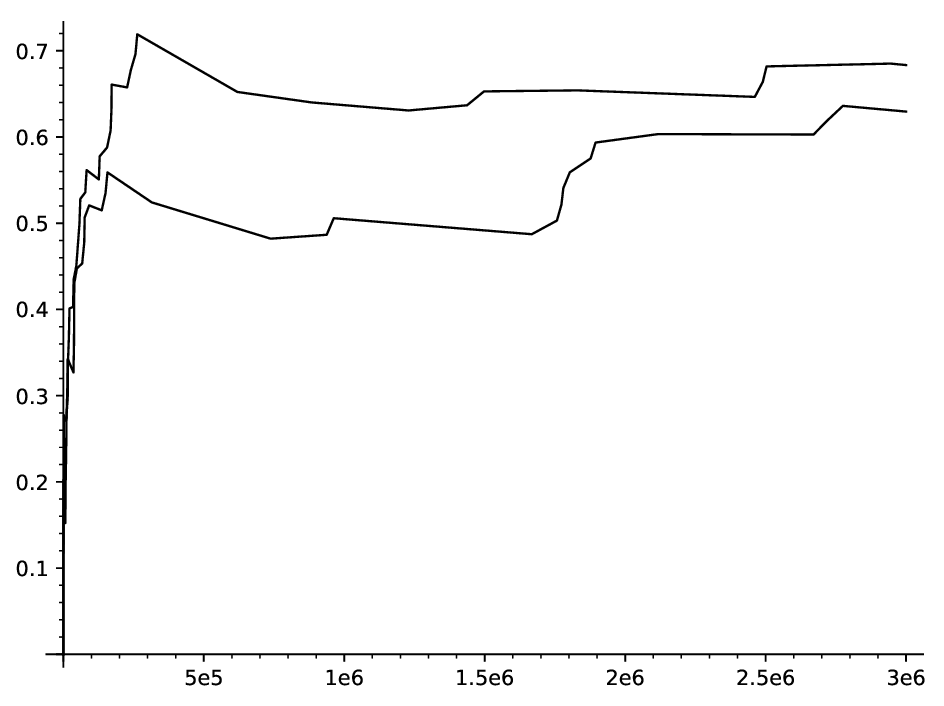}
\caption{$|l| = 9$: Top -9 bottom 9} \label{fig:37_5_A_9}
\end{subfigure}\hspace*{\fill}
\begin{subfigure}[b]{0.43\linewidth}
\includegraphics[width=\linewidth]{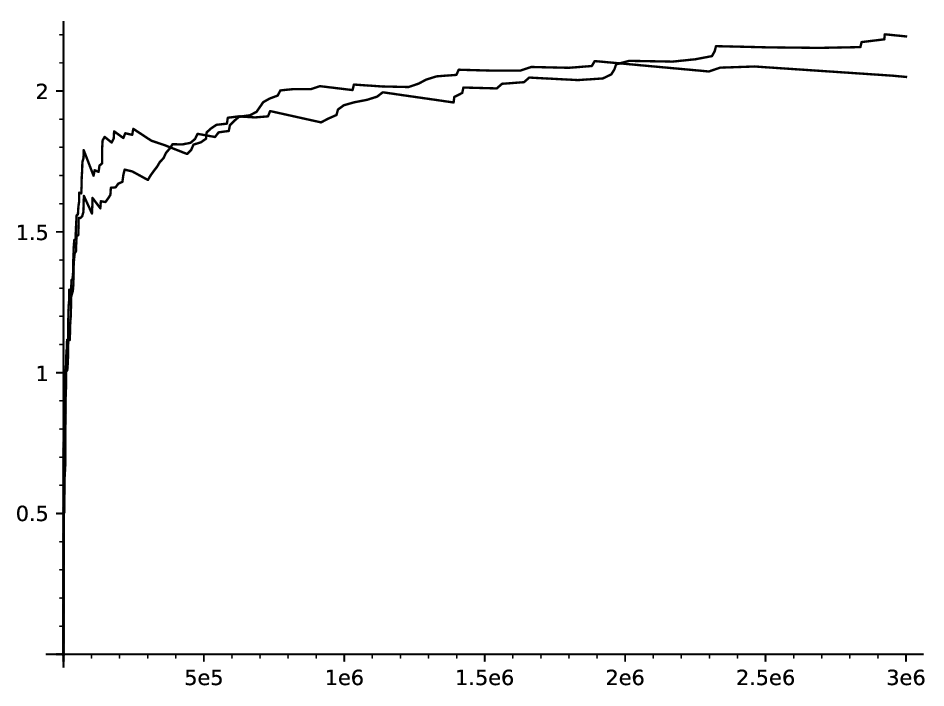}
\caption{$|l| = 11$: Top -11 bottom 11} \label{fig:37_5_A_11}
\end{subfigure}\hspace*{\fill}
\begin{subfigure}[b]{0.43\linewidth}
\includegraphics[width=\linewidth]{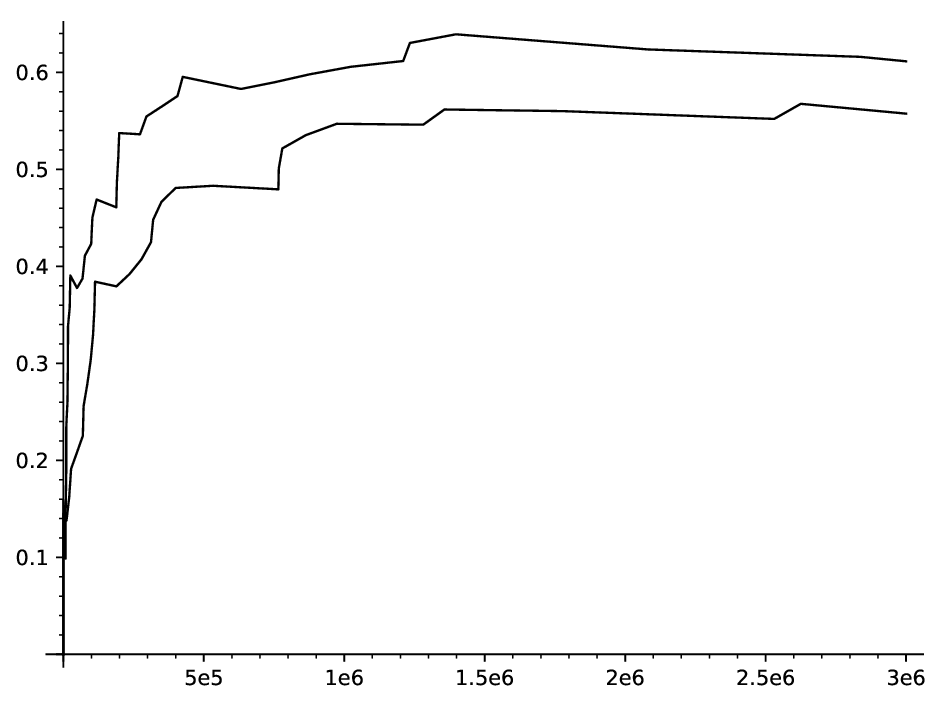}
\caption{$|l| = 16$: Top 16 bottom -16} \label{fig:37_5_A_16}
\end{subfigure}
\hspace*{-2.3cm}
\begin{subfigure}[b]{0.43\linewidth}
\includegraphics[width=\linewidth]{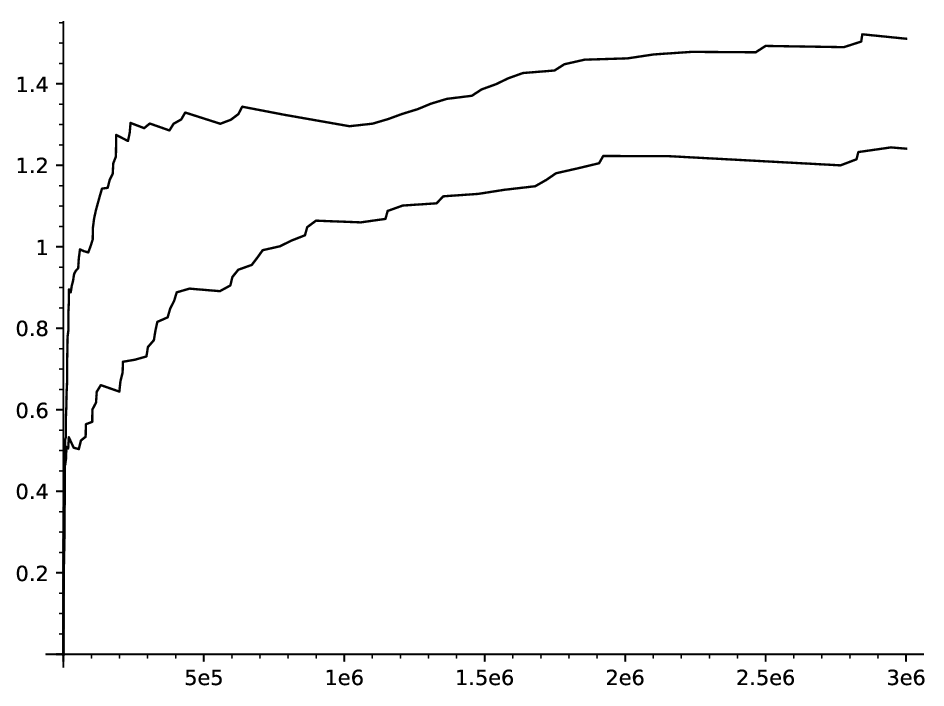}
\caption{$|l| = 19$: Top 19 bottom -19} \label{fig:37_5_A_19}
\end{subfigure}\hspace*{\fill}
\begin{subfigure}[b]{0.43\linewidth}
\includegraphics[width=\linewidth]{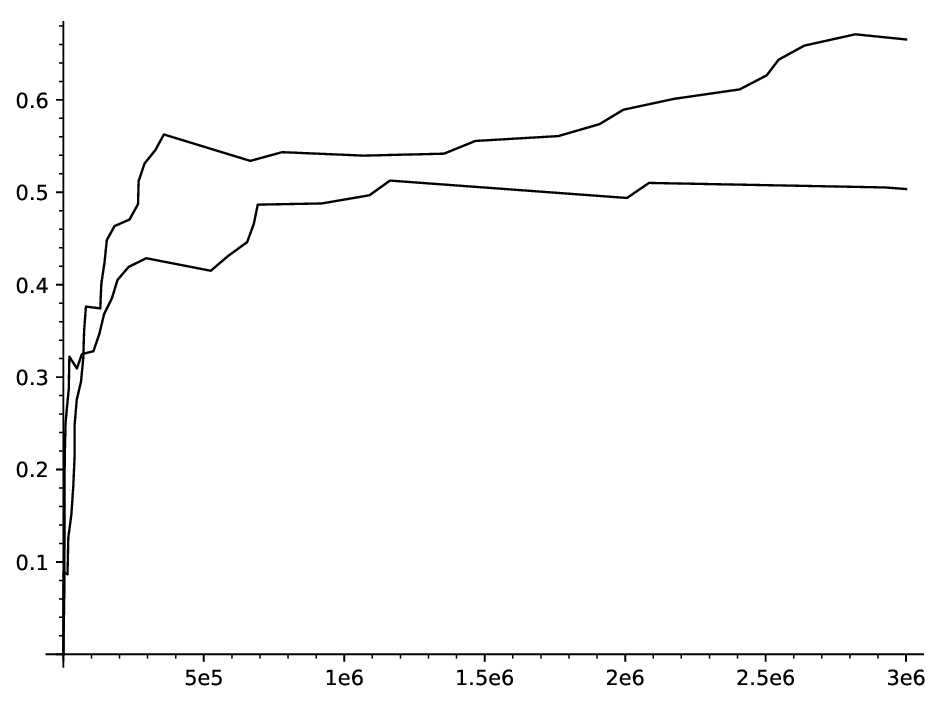}
\caption{$|l| = 20$: Top -20 bottom 20} \label{fig:37_5_A_20}
\end{subfigure}\hspace*{\fill}
\begin{subfigure}[b]{0.43\linewidth}
\includegraphics[width=\linewidth]{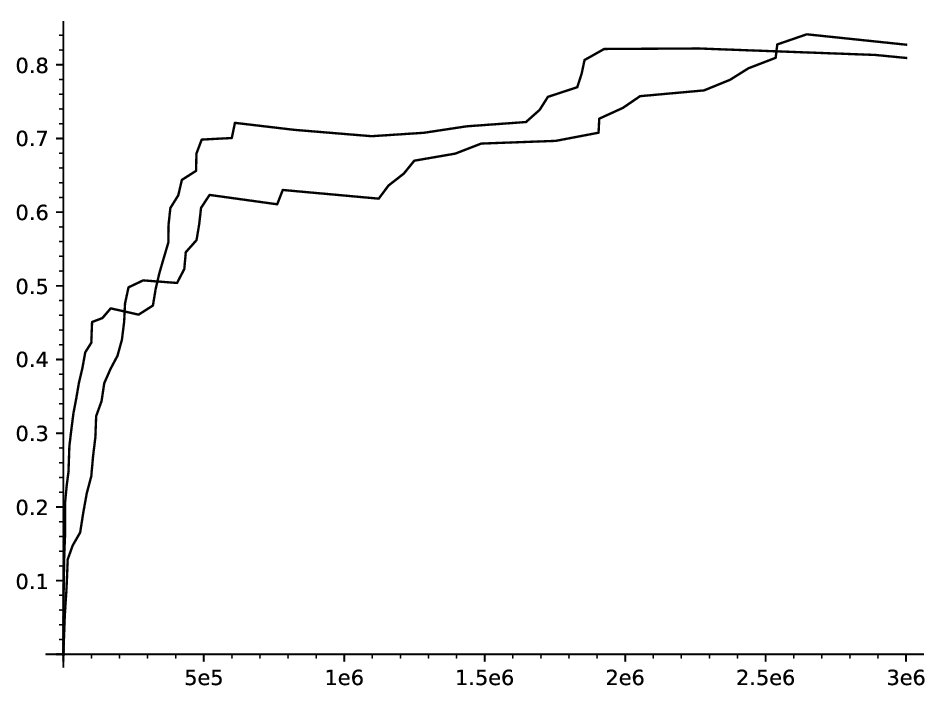}
\caption{$|l| = 25$: Top 25 bottom -25} \label{fig:37_5_A_25}
\end{subfigure}
\caption{Ratio~\eqref{ratio_A_exact} 37b1: $x(X;l)/\log^2(X)$ for $k = 5$} \label{fig:37b1_5_A_exact}
\end{figure}

\clearpage

\begin{figure}[t!] 
\hspace*{-2.3cm}
\begin{subfigure}[b]{0.43\linewidth}
\includegraphics[width=\linewidth]{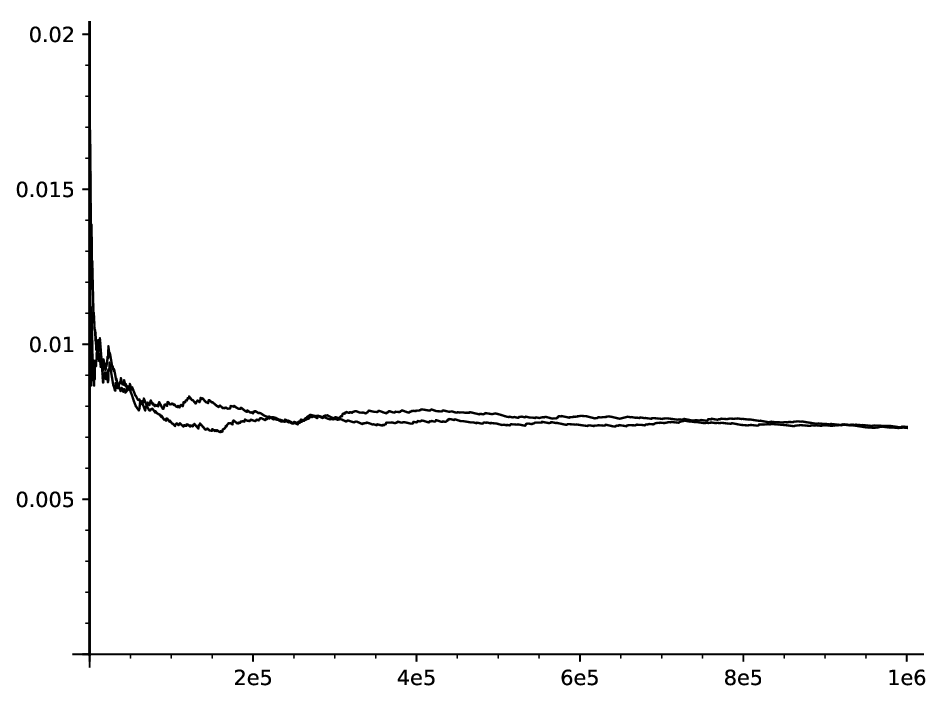}
\caption{$|l| = 1$: Top -1 bottom 1} \label{fig:11_6_A_1}
\end{subfigure}\hspace*{\fill}
\begin{subfigure}[b]{0.43\linewidth}
\includegraphics[width=\linewidth]{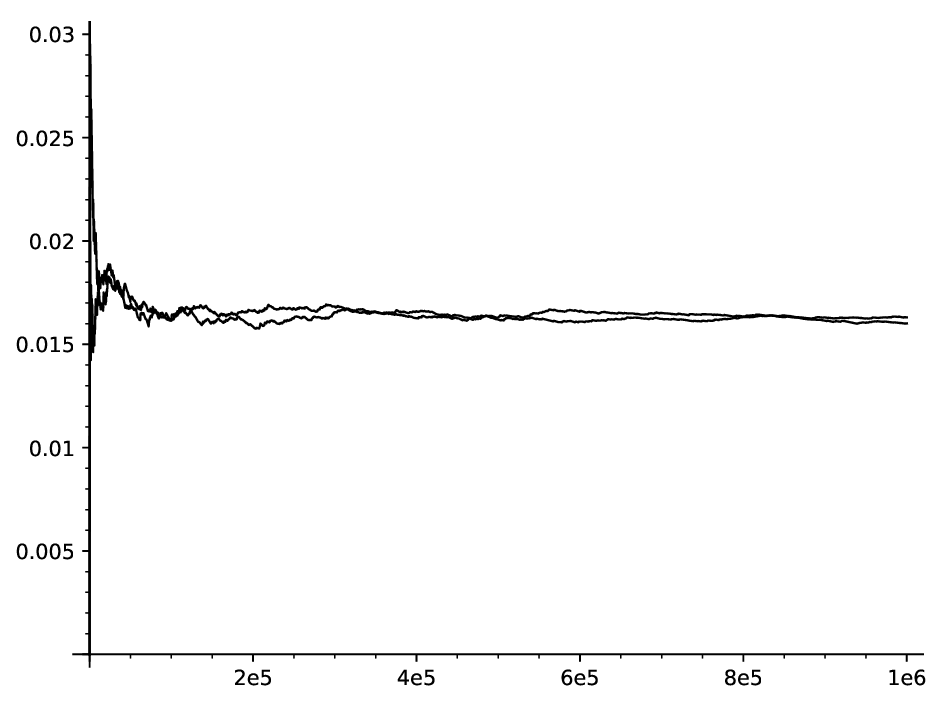}
\caption{$|l| = 2$: Top -2 bottom 2} \label{fig:11_6_A_2}
\end{subfigure}\hspace*{\fill}
\begin{subfigure}[b]{0.43\linewidth}
\includegraphics[width=\linewidth]{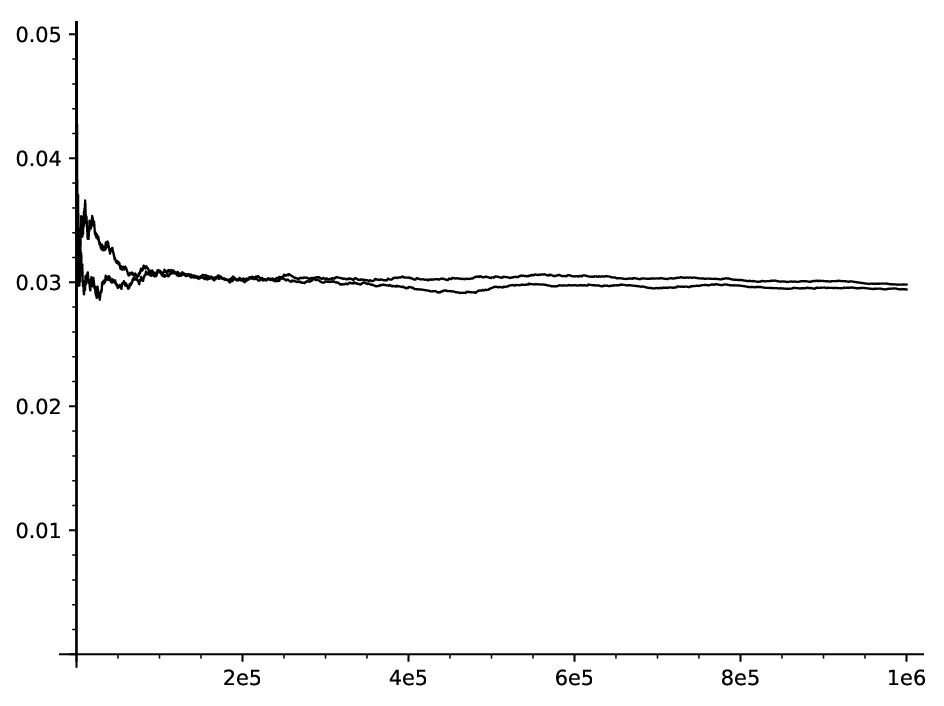}
\caption{$|l| = 3$: Top 3 bottom -3} \label{fig:11_6_A_3}
\end{subfigure}
\hspace*{-2.3cm}
\begin{subfigure}[b]{0.43\linewidth}
\includegraphics[width=\linewidth]{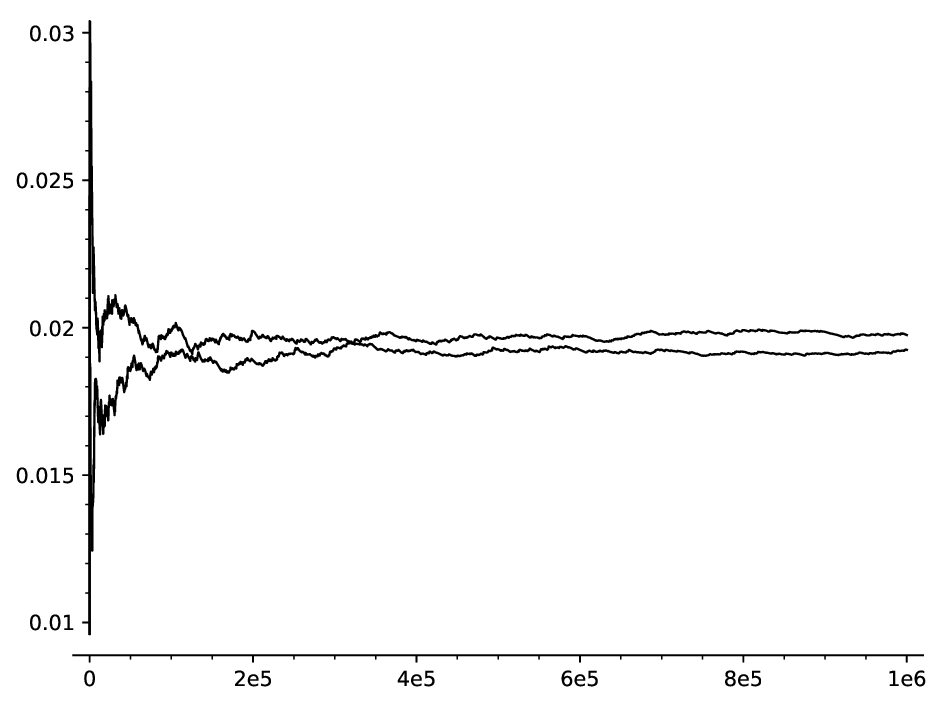}
\caption{$|l| = 4$: Top -4 bottom 4} \label{fig:11_6_A_4}
\end{subfigure}\hspace*{\fill}
\begin{subfigure}[b]{0.43\linewidth}
\includegraphics[width=\linewidth]{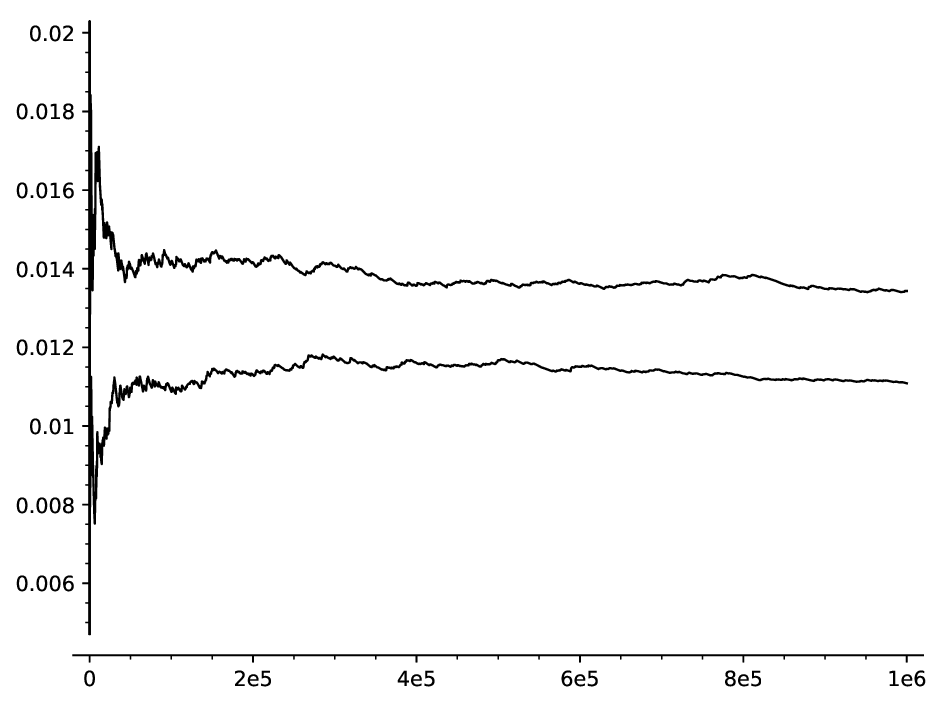}
\caption{$|l| = 5$: Top -5 bottom 5} \label{fig:11_6_A_5}
\end{subfigure}\hspace*{\fill}
\begin{subfigure}[b]{0.43\linewidth}
\includegraphics[width=\linewidth]{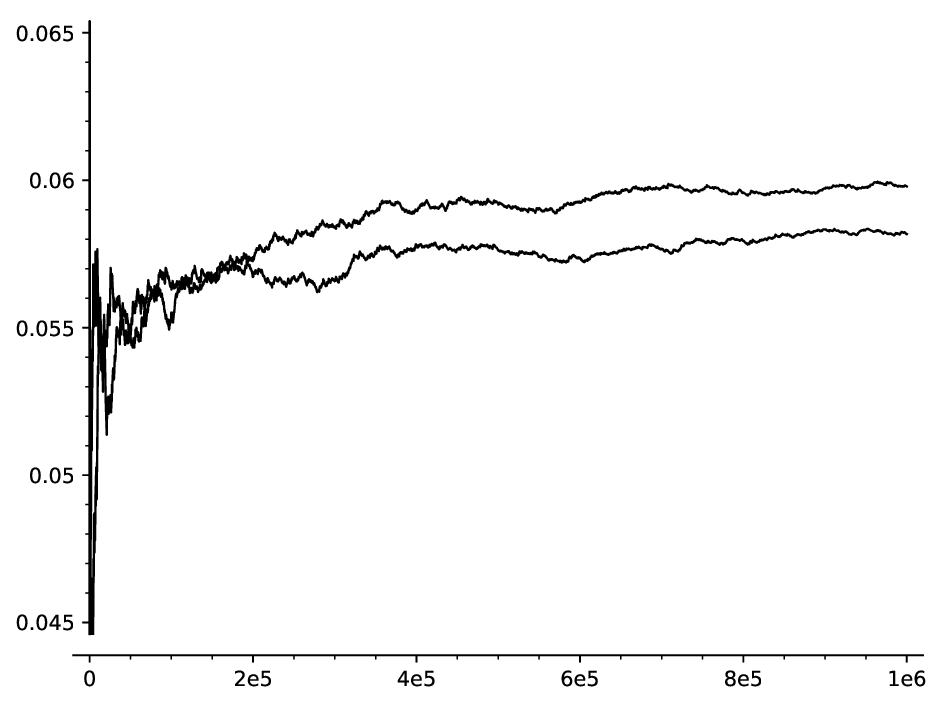}
\caption{$|l| = 6$: Top -6 bottom 6} \label{fig:11_6_A_6}
\end{subfigure}
\hspace*{-2.3cm}
\begin{subfigure}[b]{0.43\linewidth}
\includegraphics[width=\linewidth]{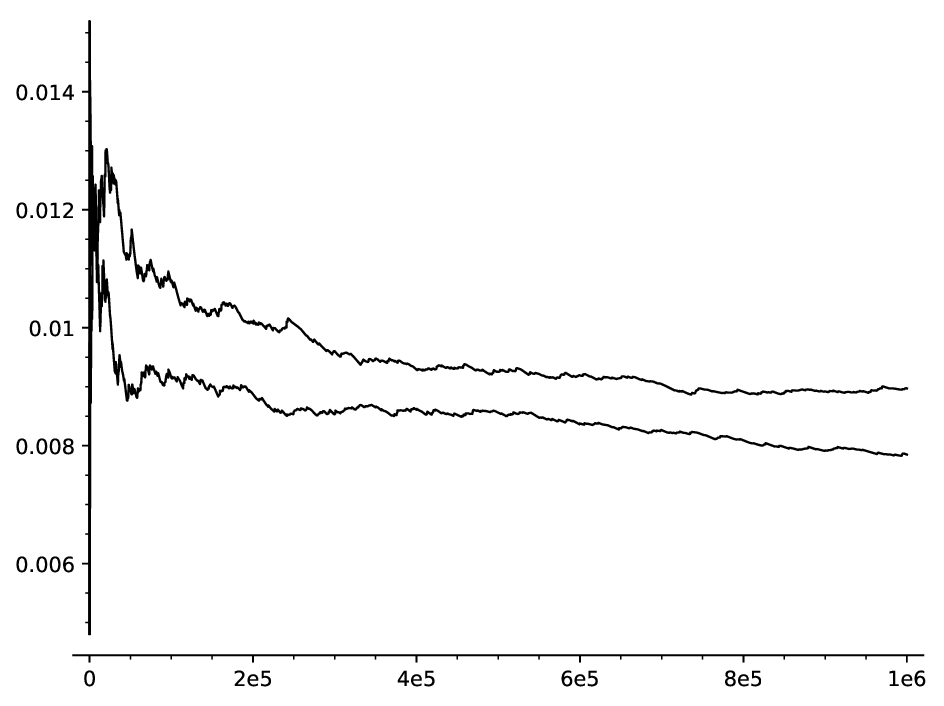}
\caption{$|l| = 7$: Top -7 bottom 7} \label{fig:11_6_A_7}
\end{subfigure}\hspace*{\fill}
\begin{subfigure}[b]{0.43\linewidth}
\includegraphics[width=\linewidth]{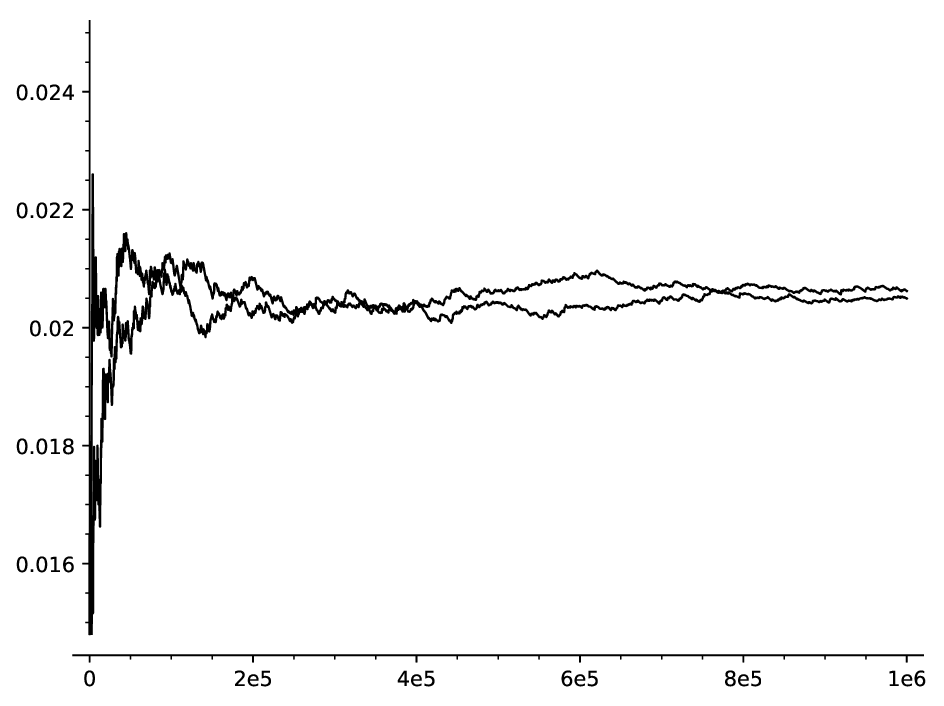}
\caption{$|l| = 8$: Top 8 bottom -8} \label{fig:11_6_A_8}
\end{subfigure}\hspace*{\fill}
\begin{subfigure}[b]{0.43\linewidth}
\includegraphics[width=\linewidth]{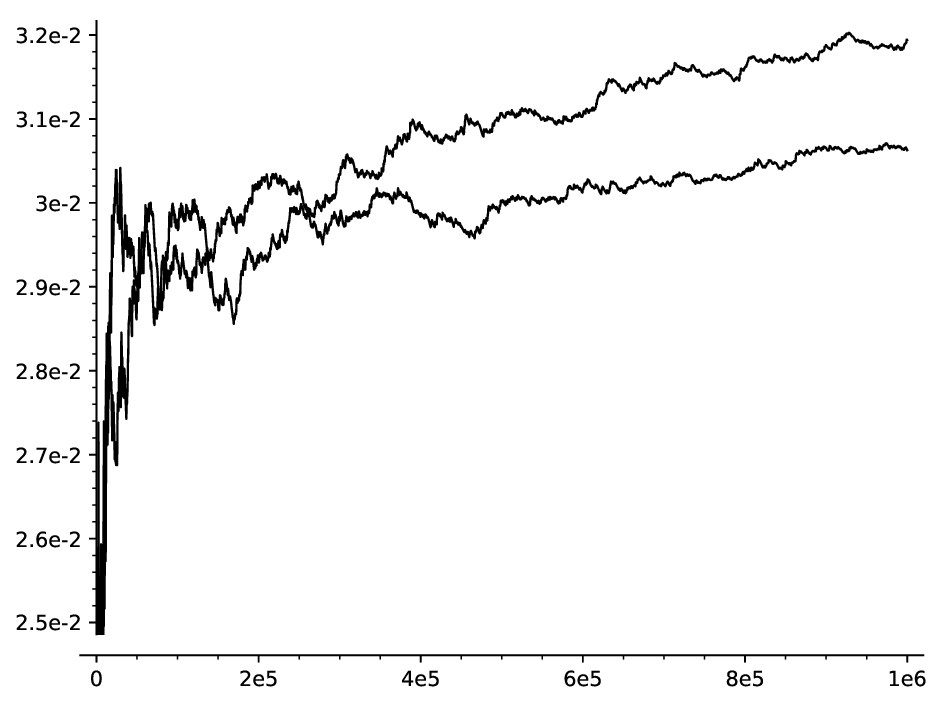}
\caption{$|l| = 9$: Top -9 bottom 9} \label{fig:11_6_A_9}
\end{subfigure}
\caption{Ratio~\eqref{ratio_A_exact} 11a1: $x(X;l)/X^{1/2}\log^2(X)$ for $k = 6$} \label{fig:11a1_6_A_exact}
\end{figure}

\clearpage

\begin{figure}[t!] 
\hspace*{-2.3cm}
\begin{subfigure}[b]{0.43\linewidth}
\includegraphics[width=\linewidth]{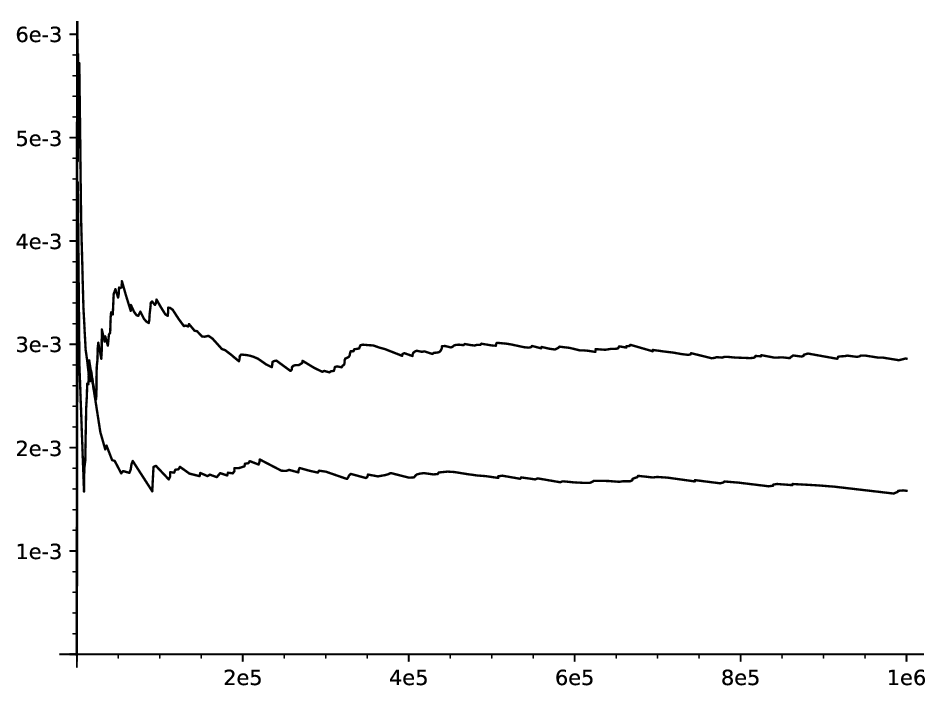}
\caption{$|l| = 1$: Top -1 bottom 1} \label{fig:14_6_A_1}
\end{subfigure}\hspace*{\fill}
\begin{subfigure}[b]{0.43\linewidth}
\includegraphics[width=\linewidth]{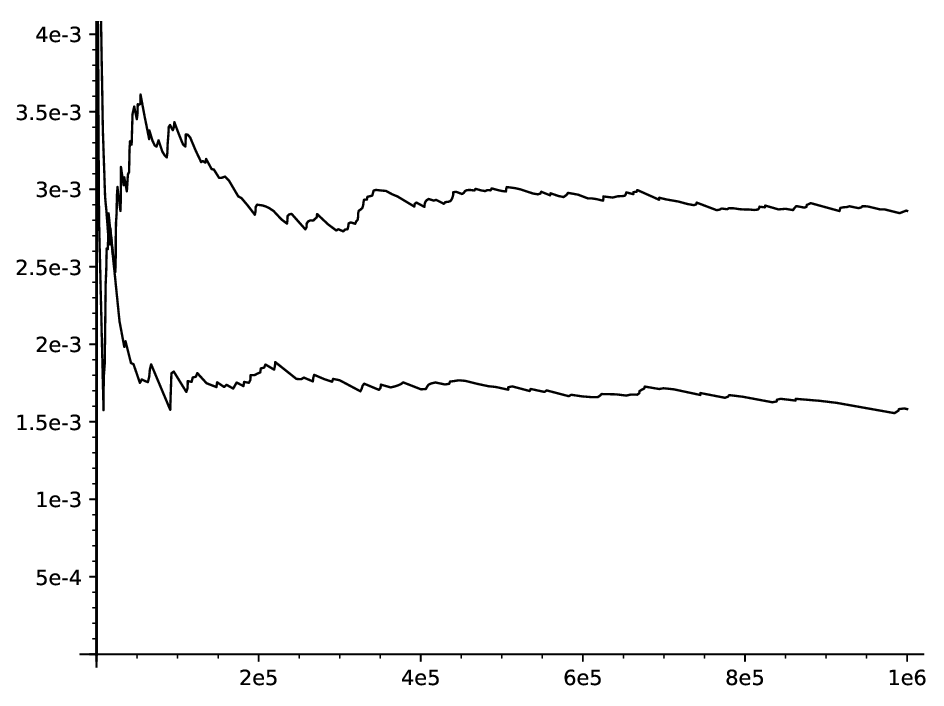}
\caption{$|l| = 2$: Top -2 bottom 2} \label{fig:14_6_A_2}
\end{subfigure}\hspace*{\fill}
\begin{subfigure}[b]{0.43\linewidth}
\includegraphics[width=\linewidth]{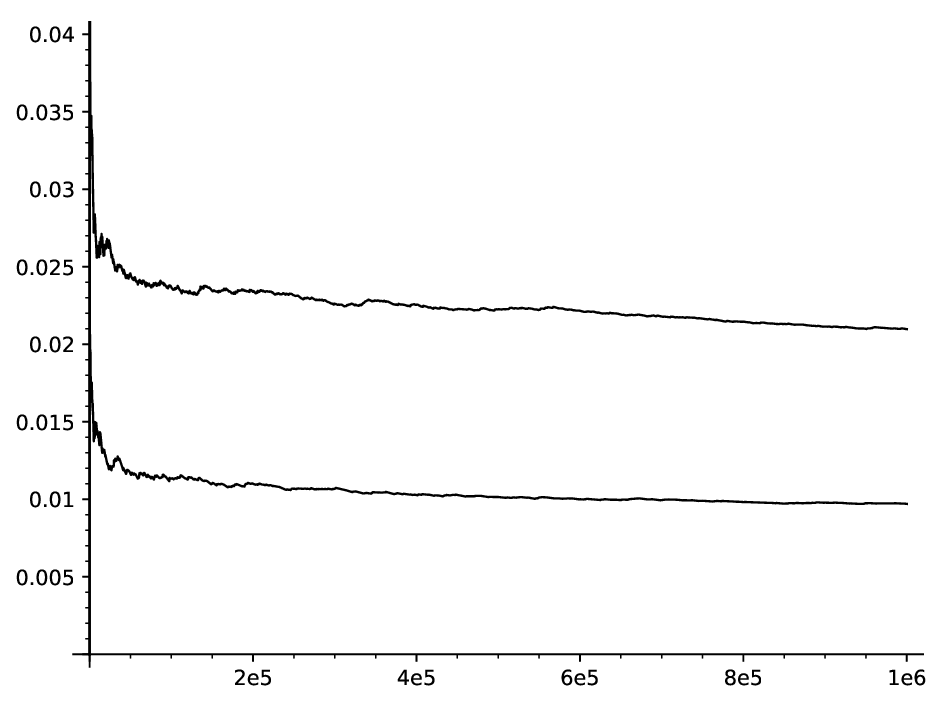}
\caption{$|l| = 3$: Top -3 bottom 3} \label{fig:14_6_A_3}
\end{subfigure}
\hspace*{-2.3cm}
\begin{subfigure}[b]{0.43\linewidth}
\includegraphics[width=\linewidth]{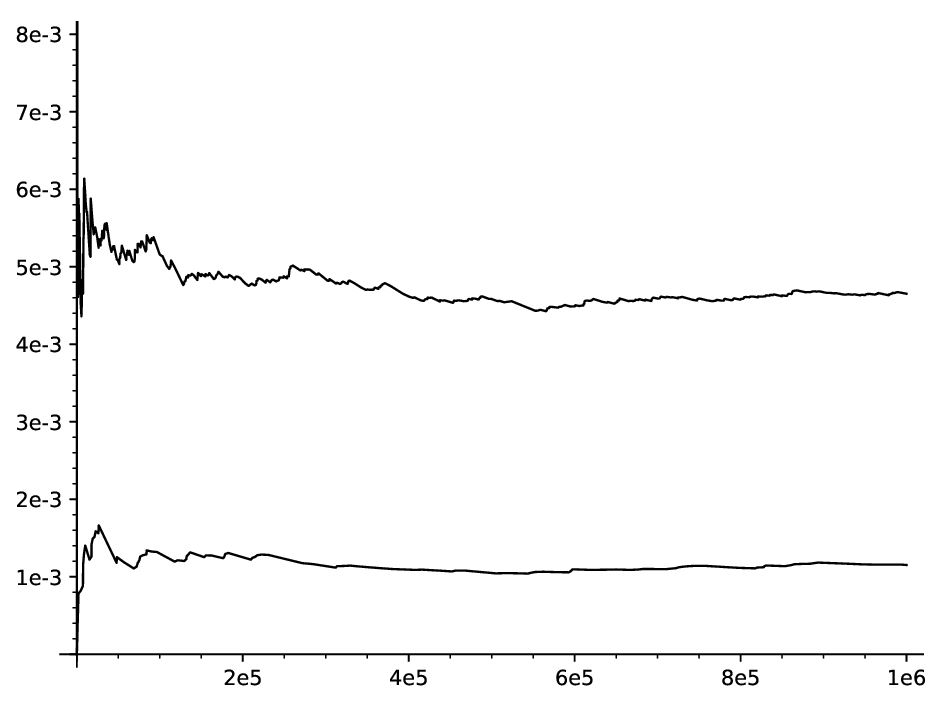}
\caption{$|l| = 4$: Top 4 bottom -4} \label{fig:14_6_A_4}
\end{subfigure}\hspace*{\fill}
\begin{subfigure}[b]{0.43\linewidth}
\includegraphics[width=\linewidth]{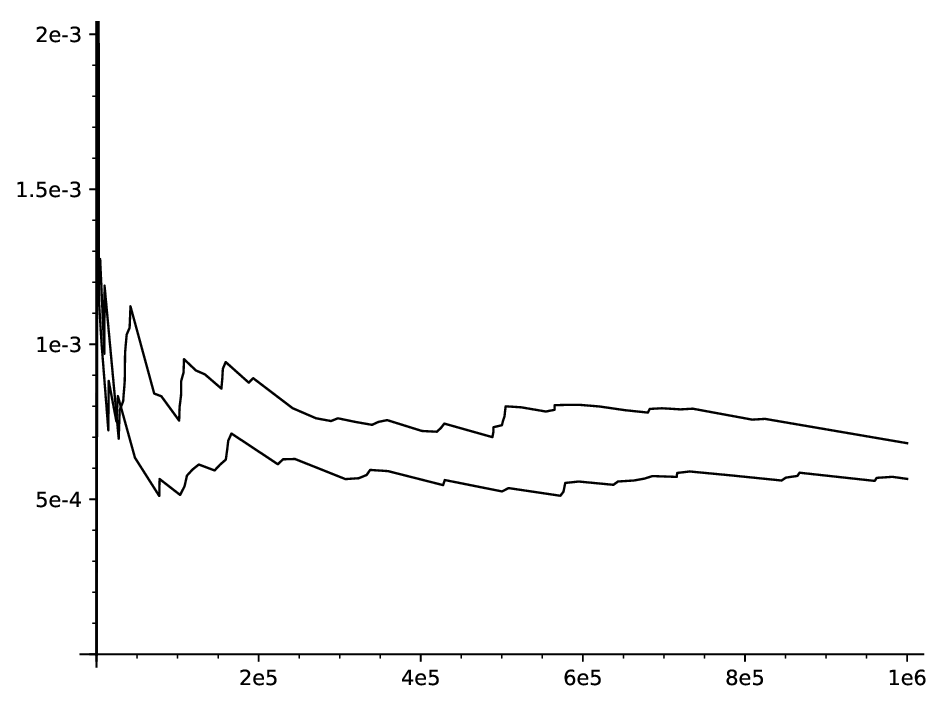}
\caption{$|l| = 5$: Top 5 bottom -5} \label{fig:14_6_A_5}
\end{subfigure}\hspace*{\fill}
\begin{subfigure}[b]{0.43\linewidth}
\includegraphics[width=\linewidth]{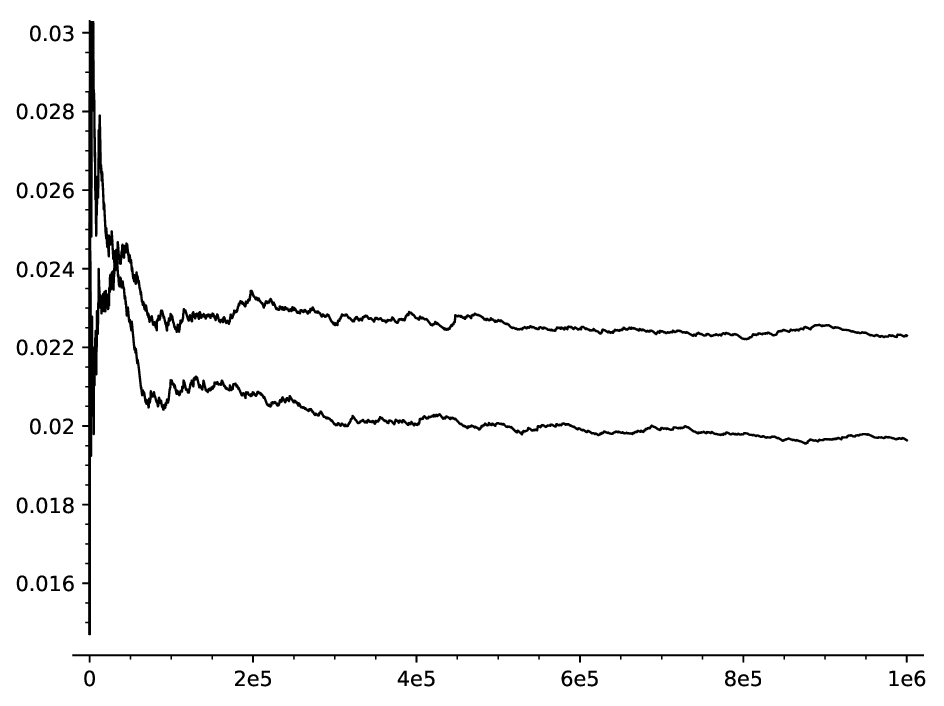}
\caption{$|l| = 6$: Top -6 bottom 6} \label{fig:14_6_A_6}
\end{subfigure}
\hspace*{-2.3cm}
\begin{subfigure}[b]{0.43\linewidth}
\includegraphics[width=\linewidth]{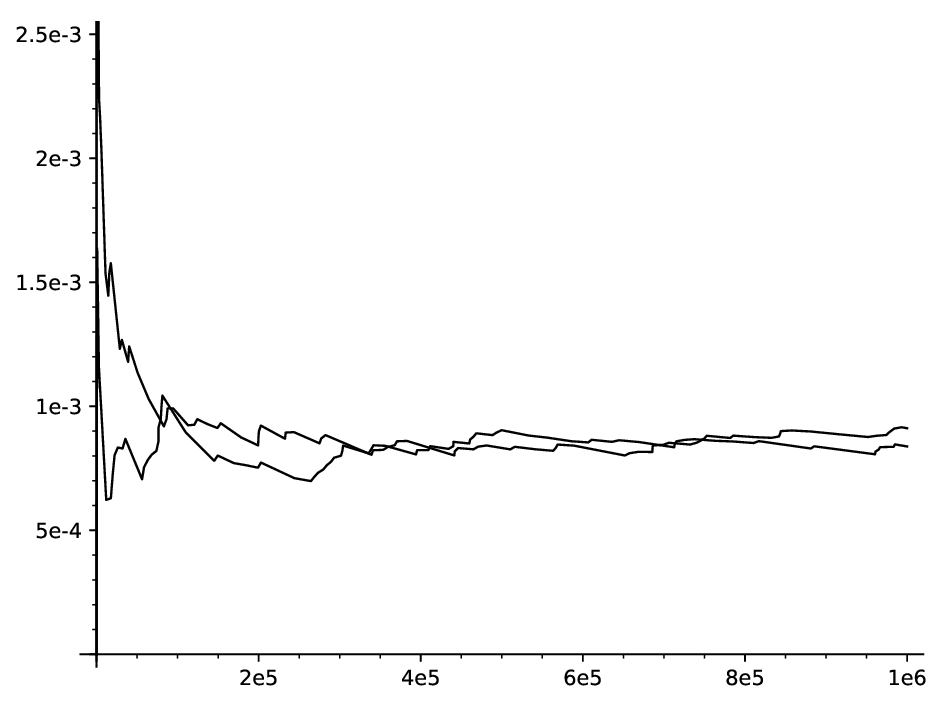}
\caption{$|l| = 7$: Top 7 bottom -7} \label{fig:14_6_A_7}
\end{subfigure}\hspace*{\fill}
\begin{subfigure}[b]{0.43\linewidth}
\includegraphics[width=\linewidth]{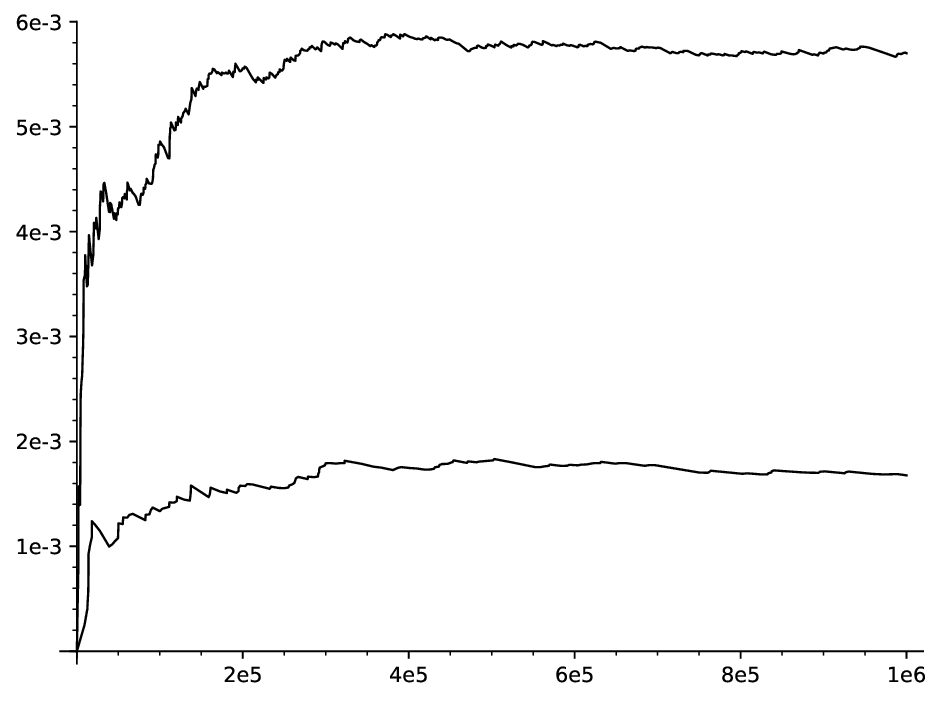}
\caption{$|l| = 8$: Top -8 bottom 8} \label{fig:14_6_A_8}
\end{subfigure}\hspace*{\fill}
\begin{subfigure}[b]{0.43\linewidth}
\includegraphics[width=\linewidth]{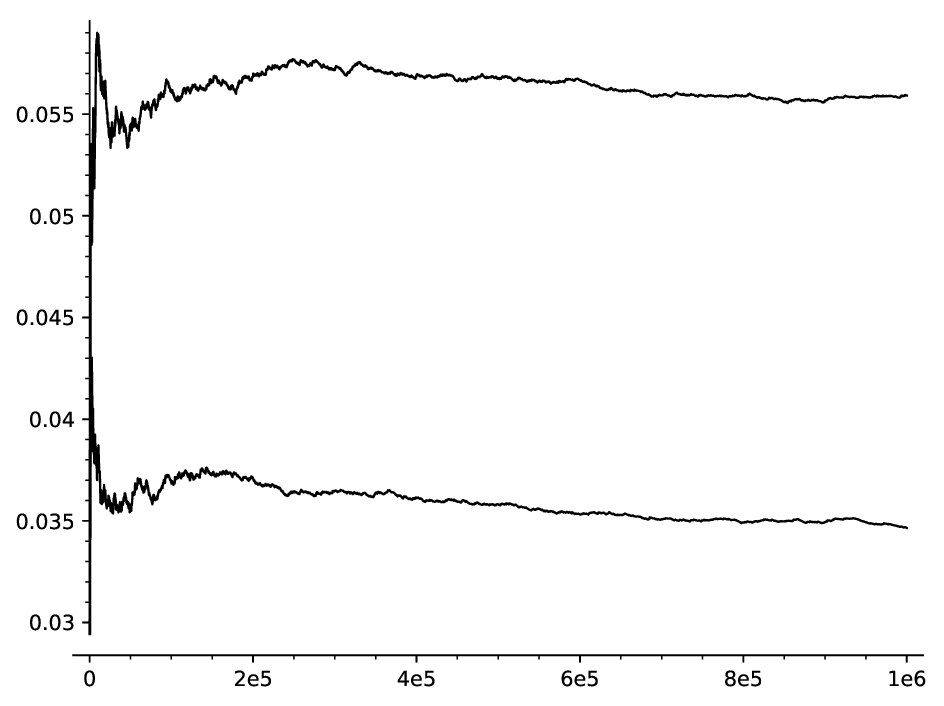}
\caption{$|l| = 9$: Top -9 bottom 9} \label{fig:14_6_A_9}
\end{subfigure}
\caption{Ratio~\eqref{ratio_A_exact} 14a1: $x(X;l)/X^{1/2}\log^2(X)$ for $k = 6$} \label{fig:14a1_6_A_exact}
\end{figure}

\clearpage

\begin{figure}[t!] 
\hspace*{-2.3cm}
\begin{subfigure}[b]{0.43\linewidth}
\includegraphics[width=\linewidth]{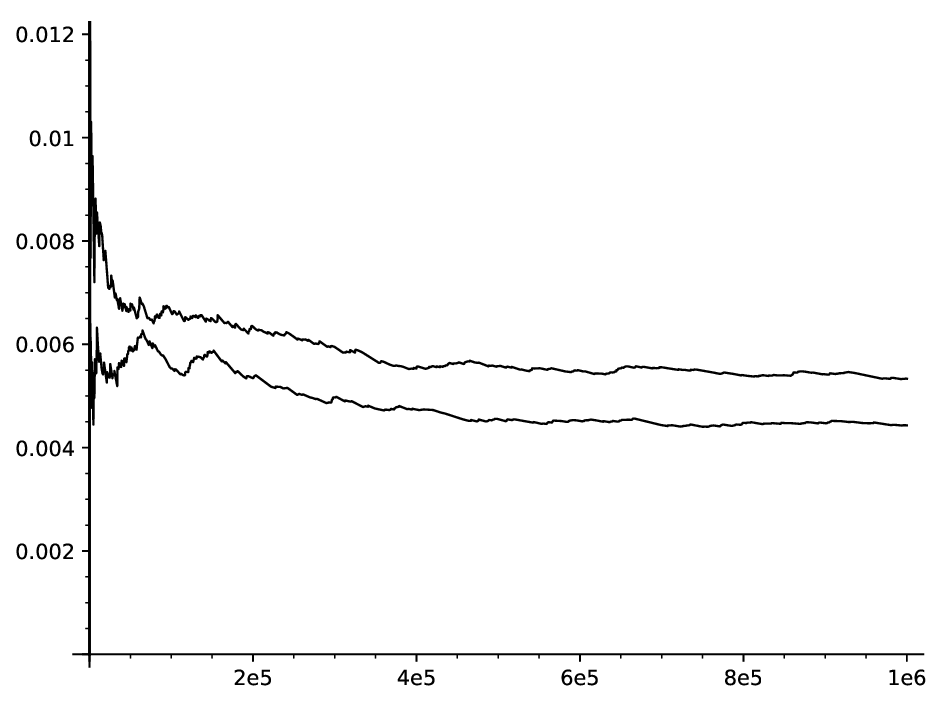}
\caption{$|l| = 1$: Top 1 bottom -1} \label{fig:15_6_A_1}
\end{subfigure}\hspace*{\fill}
\begin{subfigure}[b]{0.43\linewidth}
\includegraphics[width=\linewidth]{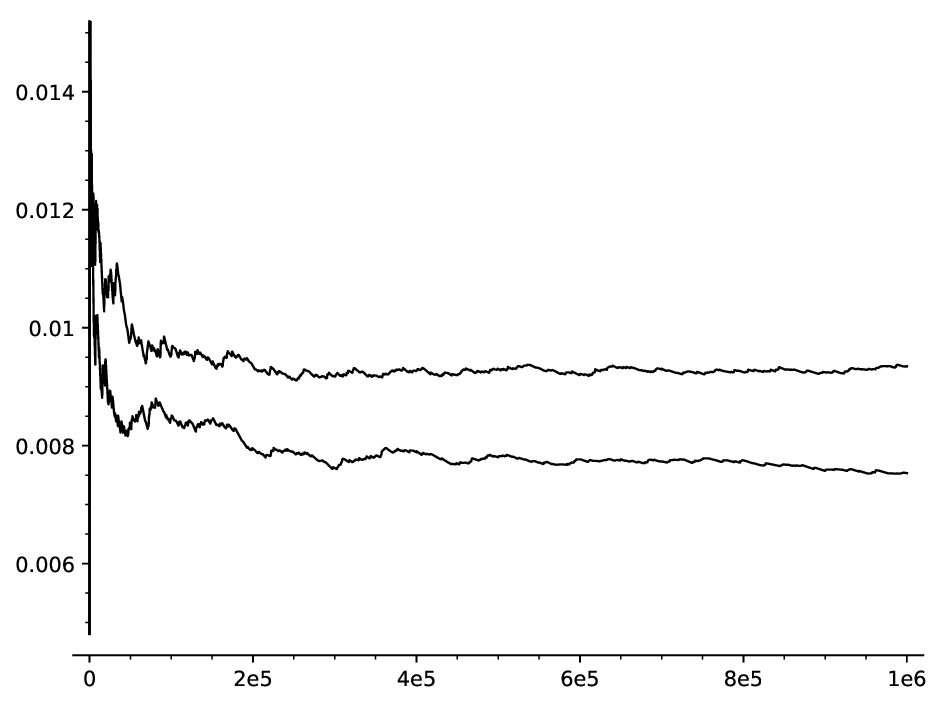}
\caption{$|l| = 2$: Top 2 bottom -2} \label{fig:15_6_A_2}
\end{subfigure}\hspace*{\fill}
\begin{subfigure}[b]{0.43\linewidth}
\includegraphics[width=\linewidth]{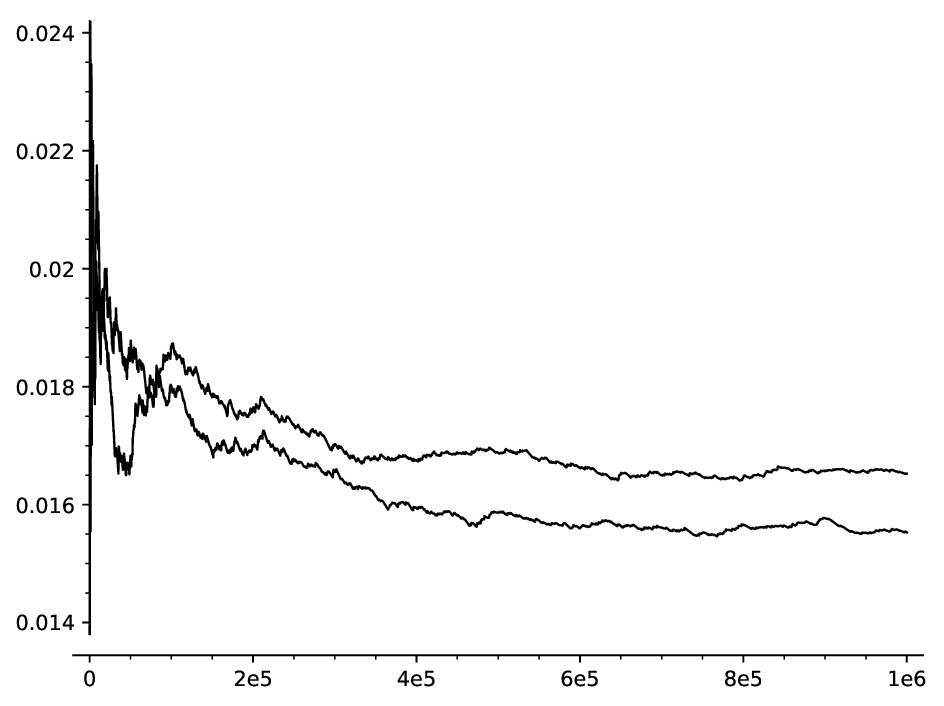}
\caption{$|l| = 3$: Top -3 bottom 3} \label{fig:15_6_A_3}
\end{subfigure}
\hspace*{-2.3cm}
\begin{subfigure}[b]{0.43\linewidth}
\includegraphics[width=\linewidth]{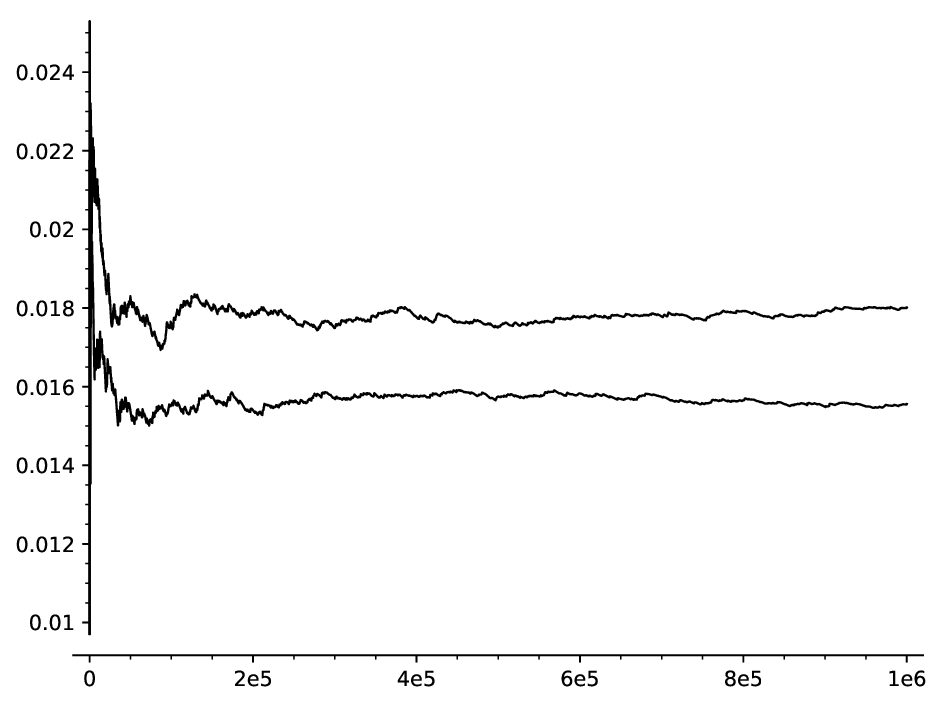}
\caption{$|l| = 4$: Top 4 bottom -4} \label{fig:15_6_A_4}
\end{subfigure}\hspace*{\fill}
\begin{subfigure}[b]{0.43\linewidth}
\includegraphics[width=\linewidth]{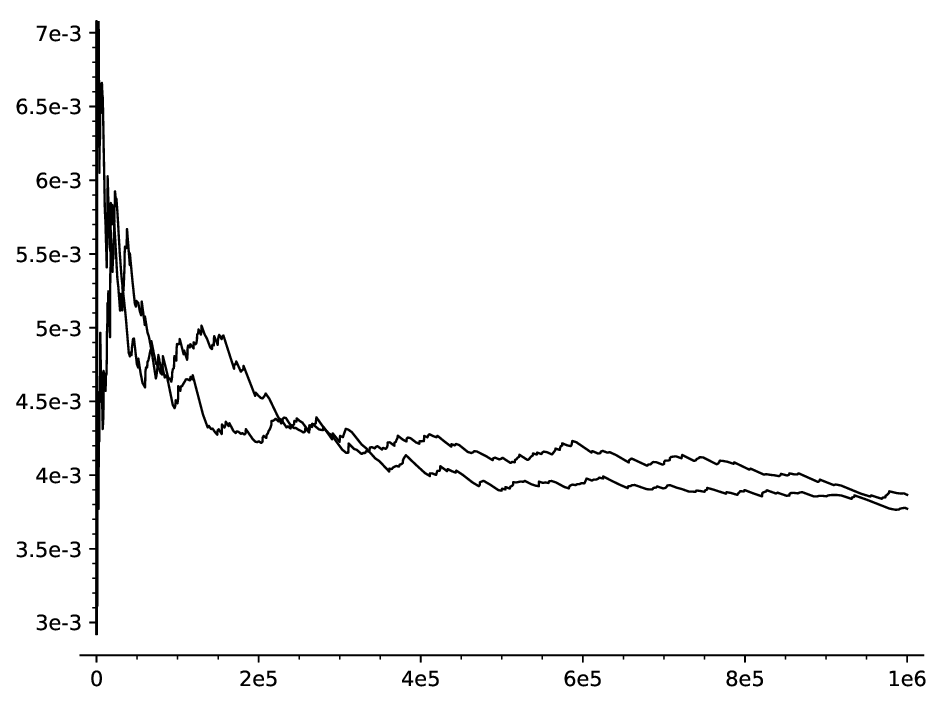}
\caption{$|l| = 5$: Top 5 bottom -5} \label{fig:15_6_A_5}
\end{subfigure}\hspace*{\fill}
\begin{subfigure}[b]{0.43\linewidth}
\includegraphics[width=\linewidth]{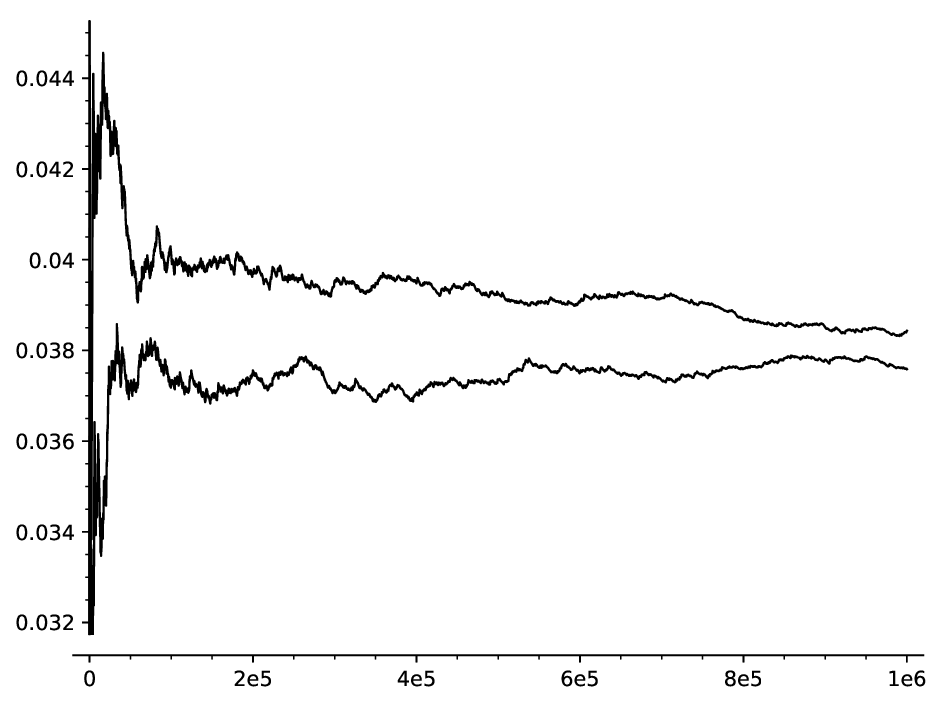}
\caption{$|l| = 6$: Top 6 bottom -6} \label{fig:15_6_A_6}
\end{subfigure}
\hspace*{-2.3cm}
\begin{subfigure}[b]{0.43\linewidth}
\includegraphics[width=\linewidth]{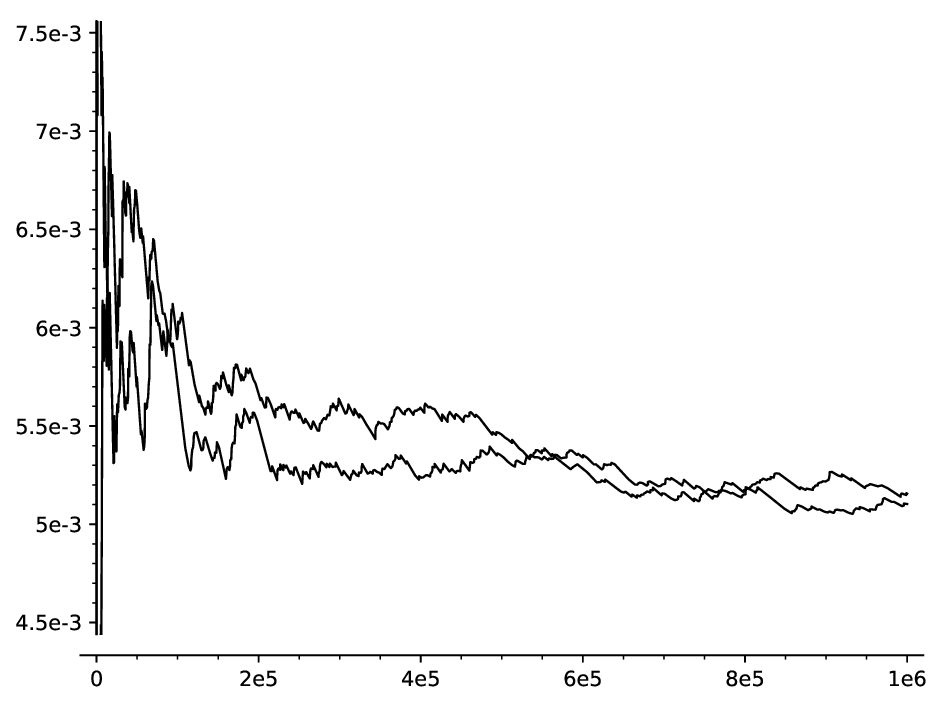}
\caption{$|l| = 7$: Top -7 bottom 7} \label{fig:15_6_A_7}
\end{subfigure}\hspace*{\fill}
\begin{subfigure}[b]{0.43\linewidth}
\includegraphics[width=\linewidth]{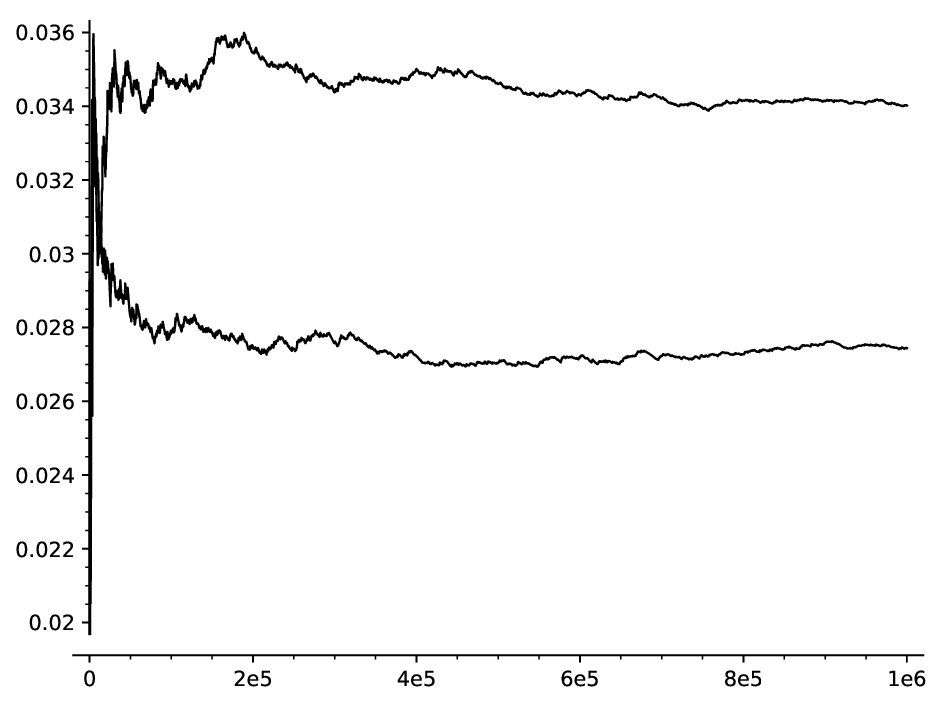}
\caption{$|l| = 8$: Top 8 bottom -8} \label{fig:15_6_A_8}
\end{subfigure}\hspace*{\fill}
\begin{subfigure}[b]{0.43\linewidth}
\includegraphics[width=\linewidth]{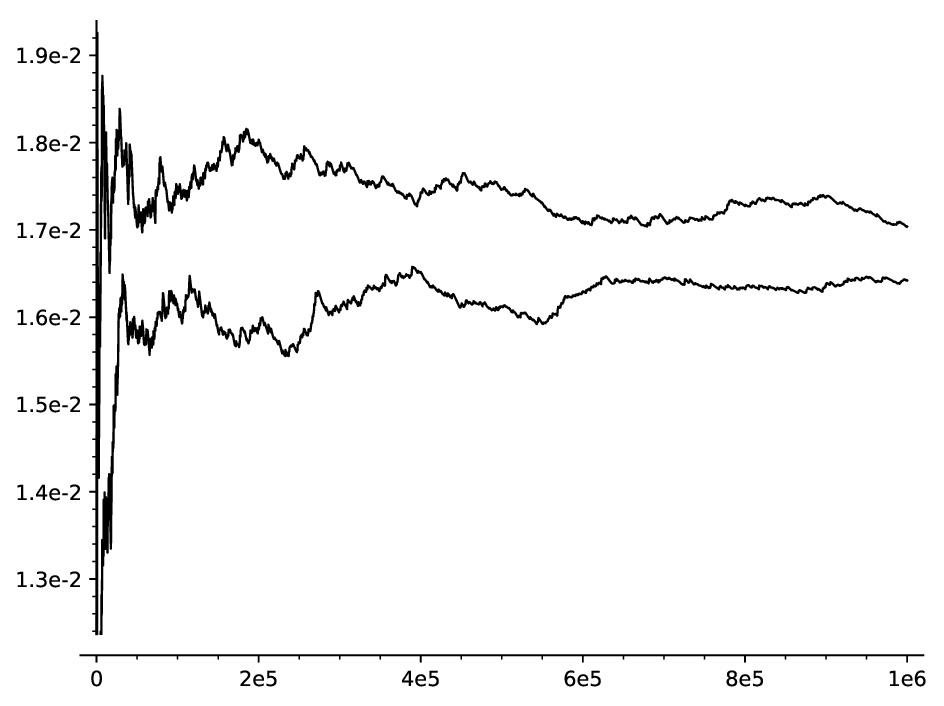}
\caption{$|l| = 9$: Top 9 bottom -9} \label{fig:15_6_A_9}
\end{subfigure}
\caption{Ratio~\eqref{ratio_A_exact} 15a1: $x(X;l)/X^{1/2}\log^2(X)$ for $k = 6$} \label{fig:15a1_6_A_exact}
\end{figure}

\clearpage

\begin{figure}[t!] 
\hspace*{-2.3cm}
\begin{subfigure}[b]{0.43\linewidth}
\includegraphics[width=\linewidth]{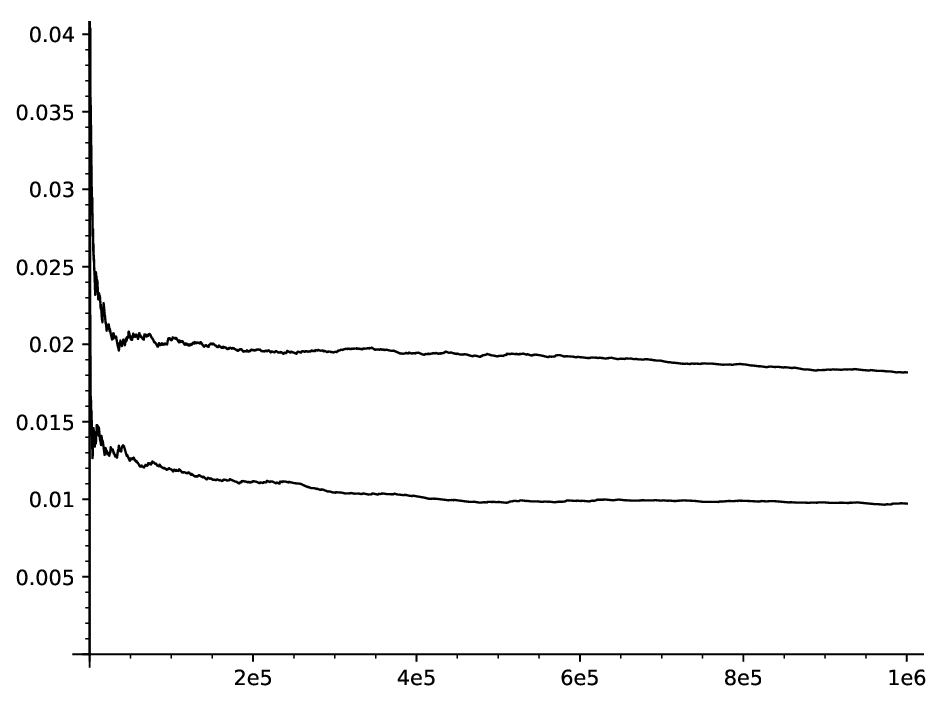}
\caption{$|l| = 1$: Top 1 bottom -1} \label{fig:17_6_A_1}
\end{subfigure}\hspace*{\fill}
\begin{subfigure}[b]{0.43\linewidth}
\includegraphics[width=\linewidth]{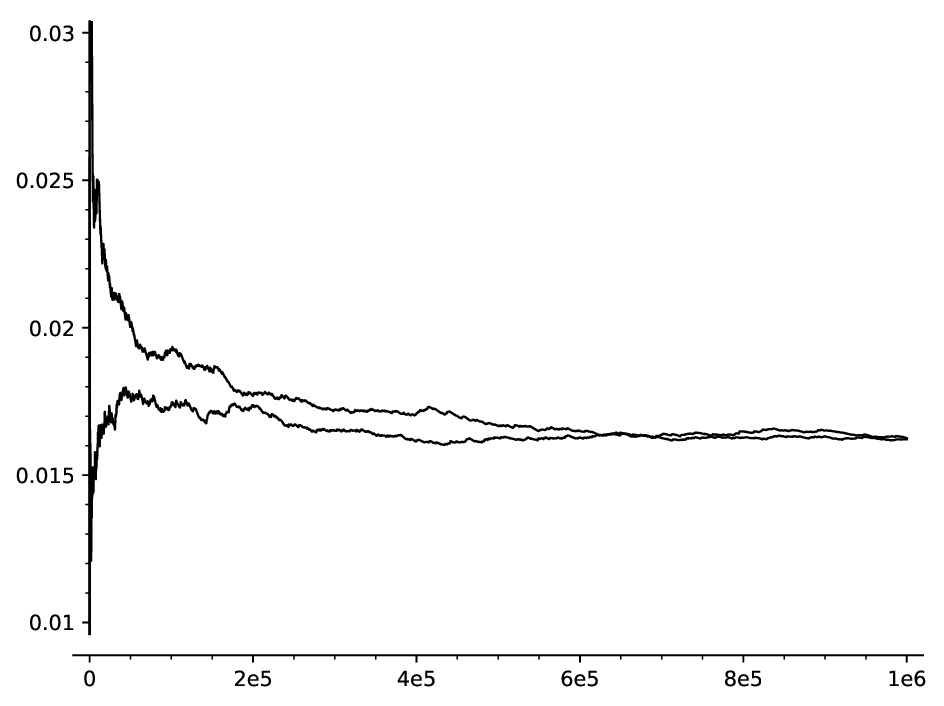}
\caption{$|l| = 2$: Top -2 bottom 2} \label{fig:17_6_A_2}
\end{subfigure}\hspace*{\fill}
\begin{subfigure}[b]{0.43\linewidth}
\includegraphics[width=\linewidth]{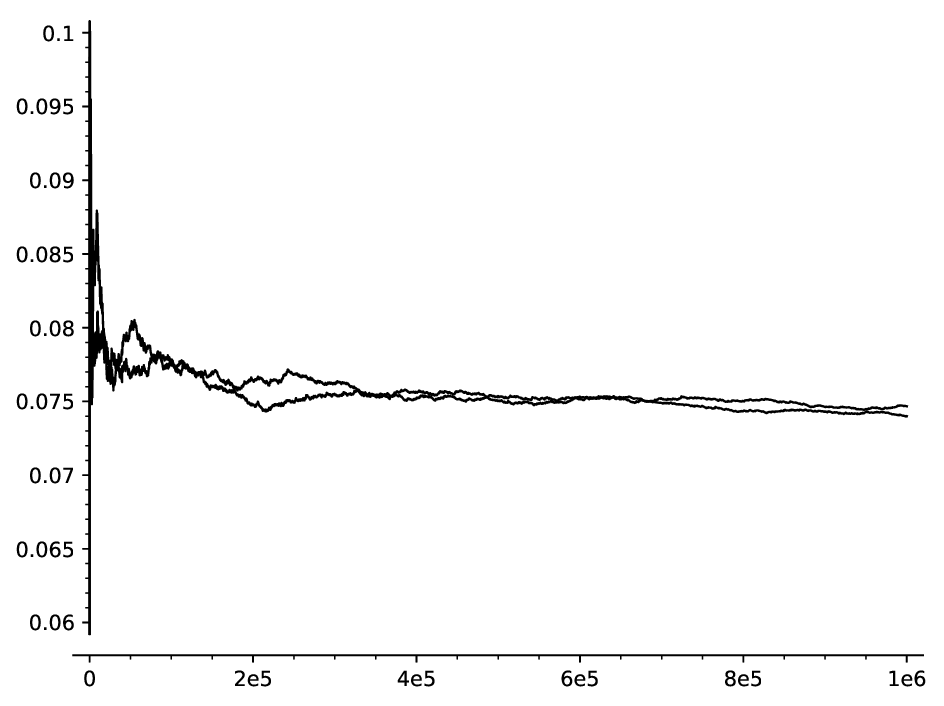}
\caption{$|l| = 3$: Top 3 bottom -3} \label{fig:17_6_A_3}
\end{subfigure}
\hspace*{-2.3cm}
\begin{subfigure}[b]{0.43\linewidth}
\includegraphics[width=\linewidth]{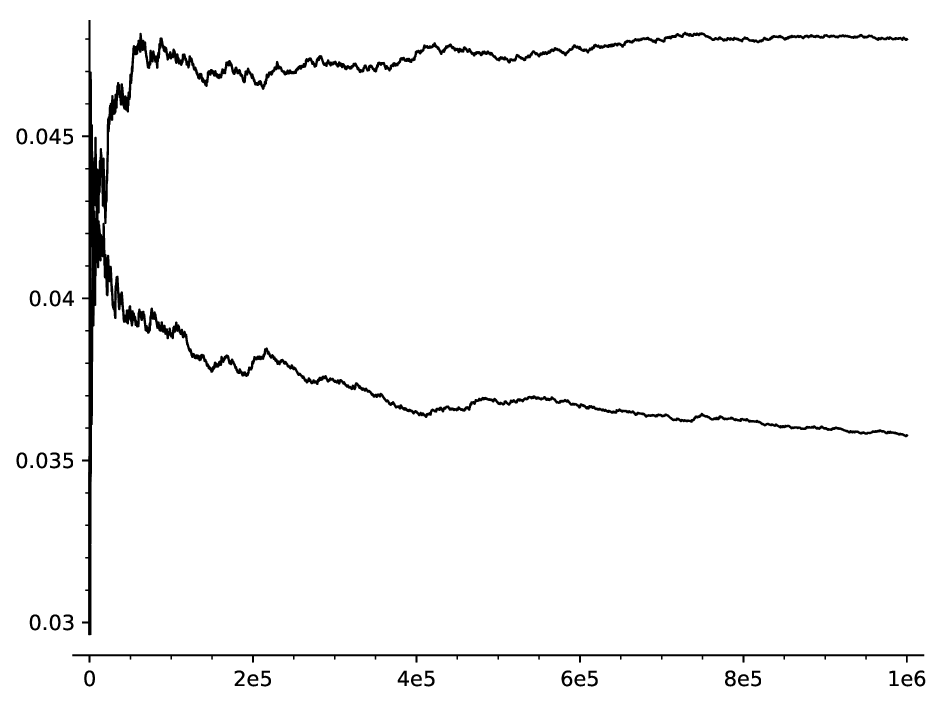}
\caption{$|l| = 4$: Top 4 bottom -4} \label{fig:17_6_A_4}
\end{subfigure}\hspace*{\fill}
\begin{subfigure}[b]{0.43\linewidth}
\includegraphics[width=\linewidth]{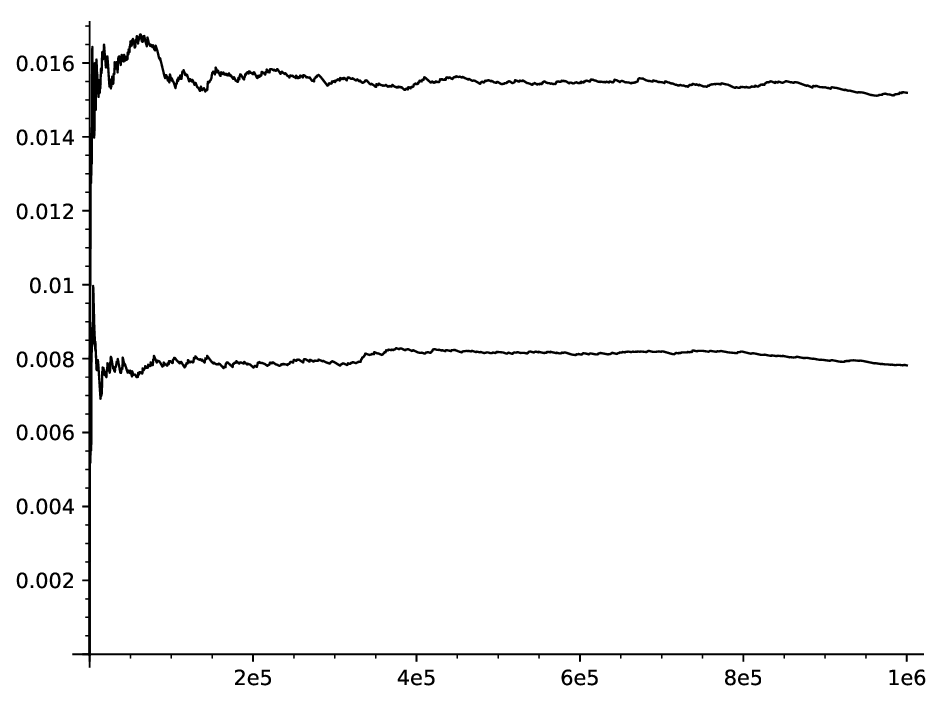}
\caption{$|l| = 5$: Top -5 bottom 5} \label{fig:17_6_A_5}
\end{subfigure}\hspace*{\fill}
\begin{subfigure}[b]{0.43\linewidth}
\includegraphics[width=\linewidth]{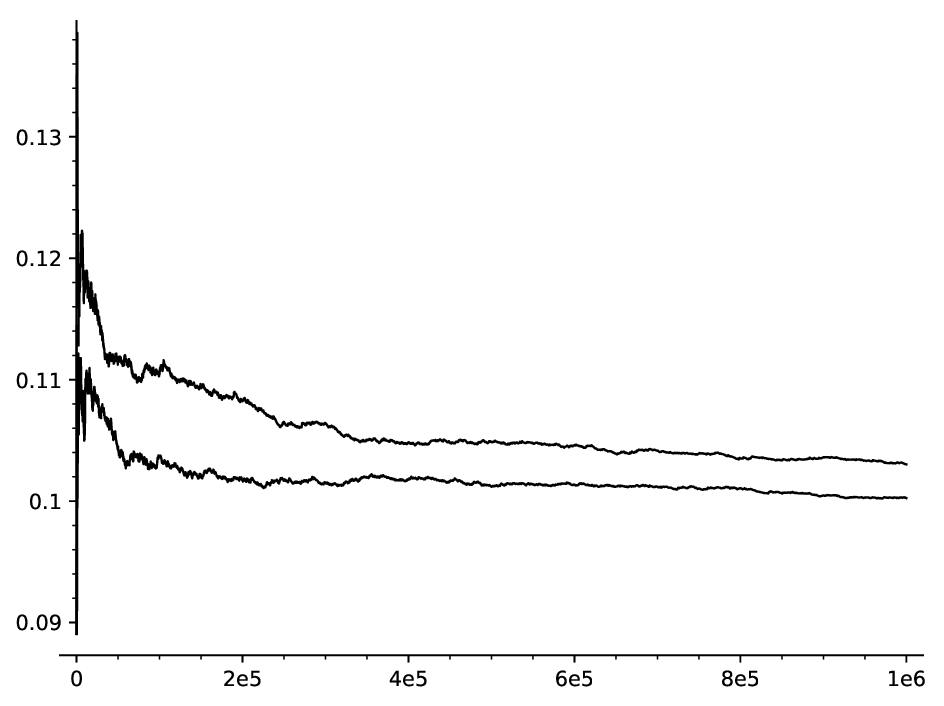}
\caption{$|l| = 6$: Top 6 bottom -6} \label{fig:17_6_A_6}
\end{subfigure}
\hspace*{-2.3cm}
\begin{subfigure}[b]{0.43\linewidth}
\includegraphics[width=\linewidth]{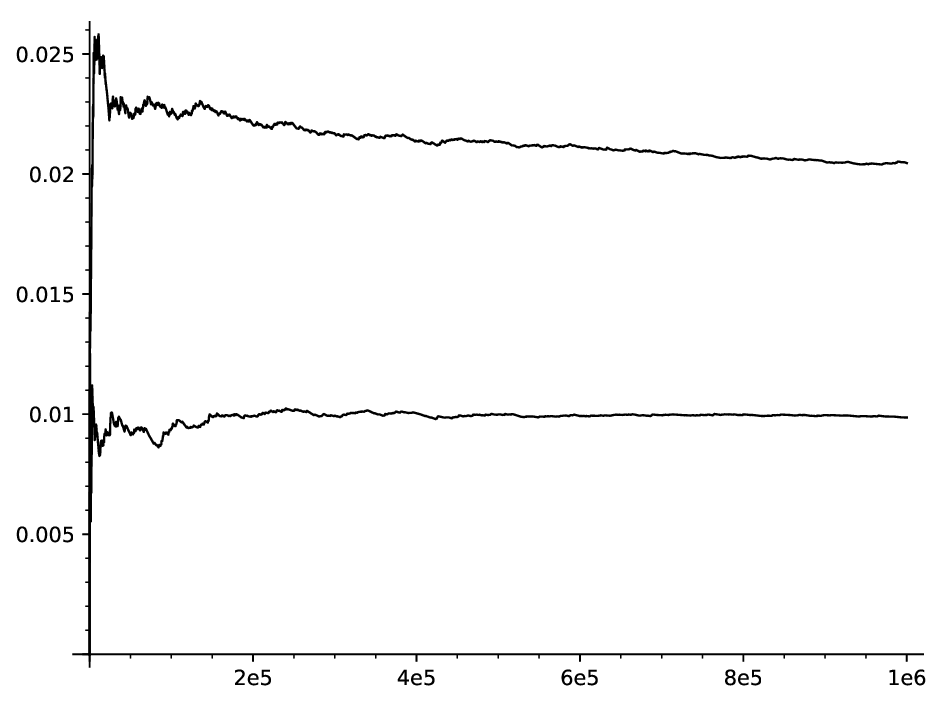}
\caption{$|l| = 7$: Top 7 bottom -7} \label{fig:17_6_A_7}
\end{subfigure}\hspace*{\fill}
\begin{subfigure}[b]{0.43\linewidth}
\includegraphics[width=\linewidth]{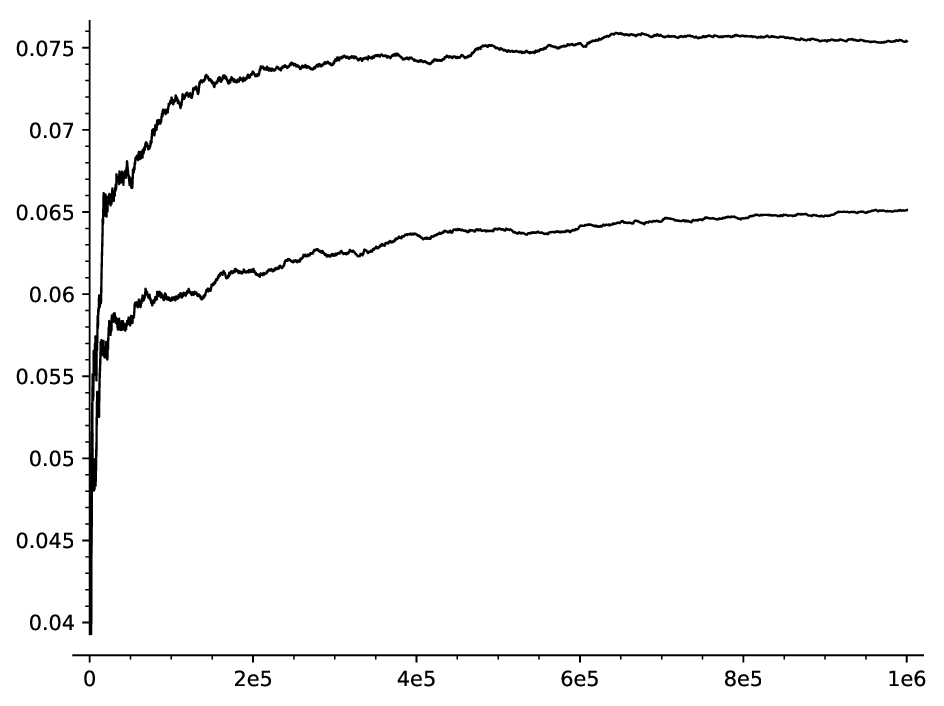}
\caption{$|l| = 8$: Top 8 bottom -8} \label{fig:17_6_A_8}
\end{subfigure}\hspace*{\fill}
\begin{subfigure}[b]{0.43\linewidth}
\includegraphics[width=\linewidth]{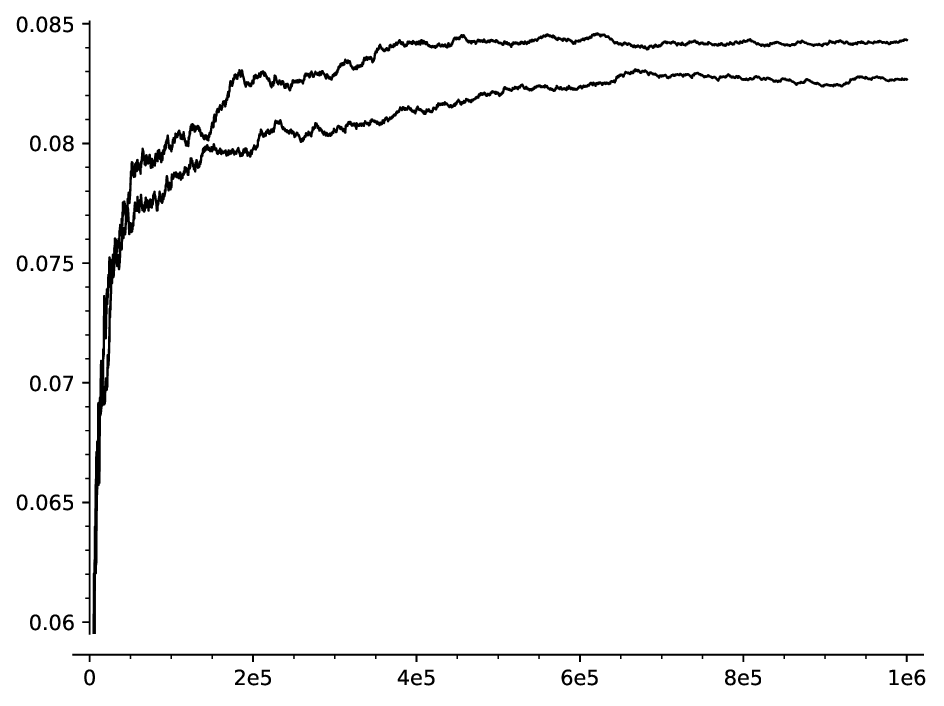}
\caption{$|l| = 9$: Top 9 bottom -9} \label{fig:17_6_A_9}
\end{subfigure}
\caption{Ratio~\eqref{ratio_A_exact} 17a1: $x(X;l)/X^{1/2}\log^2(X)$ for $k = 6$} \label{fig:17a1_6_A_exact}
\end{figure}

\clearpage

\begin{figure}[t!] 
\hspace*{-2.3cm}
\begin{subfigure}[b]{0.43\linewidth}
\includegraphics[width=\linewidth]{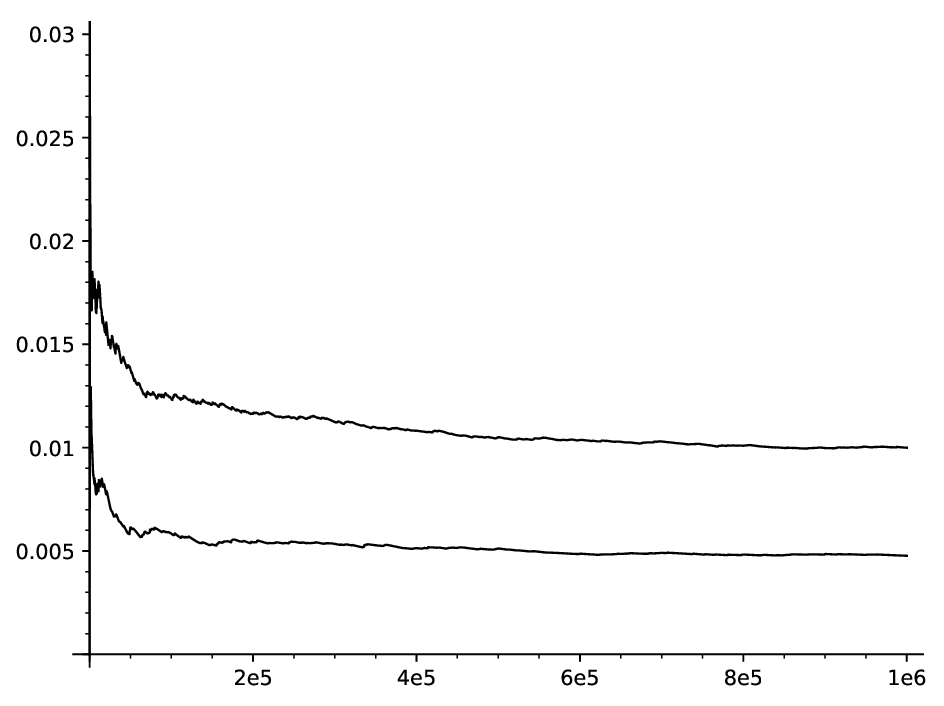}
\caption{$|l| = 1$: Top -1 bottom 1} \label{fig:19_6_A_1}
\end{subfigure}\hspace*{\fill}
\begin{subfigure}[b]{0.43\linewidth}
\includegraphics[width=\linewidth]{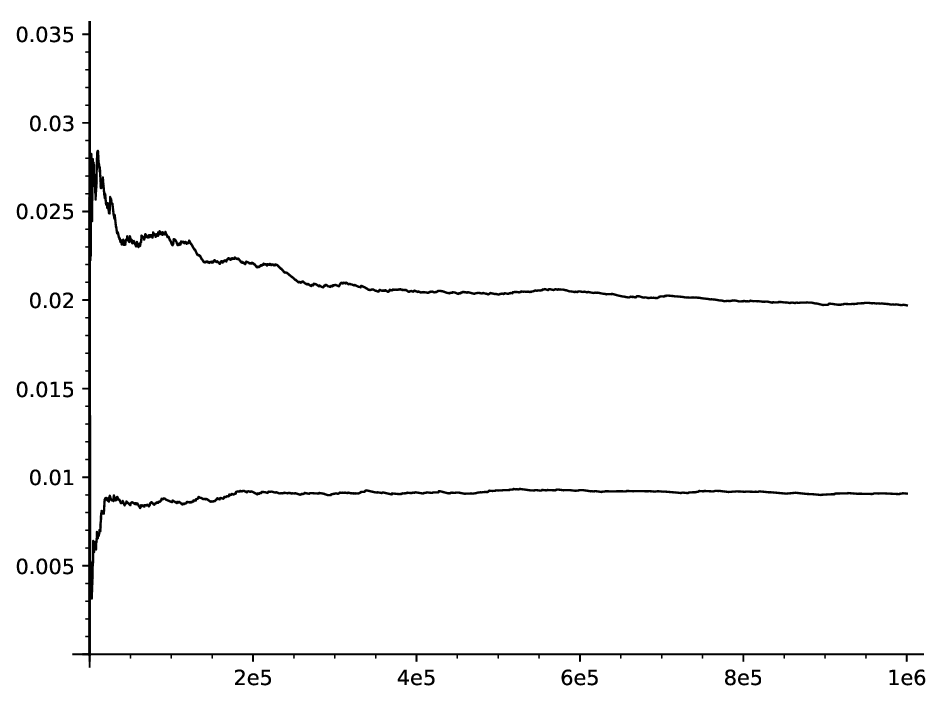}
\caption{$|l| = 2$: Top 2 bottom -2} \label{fig:19_6_A_2}
\end{subfigure}\hspace*{\fill}
\begin{subfigure}[b]{0.43\linewidth}
\includegraphics[width=\linewidth]{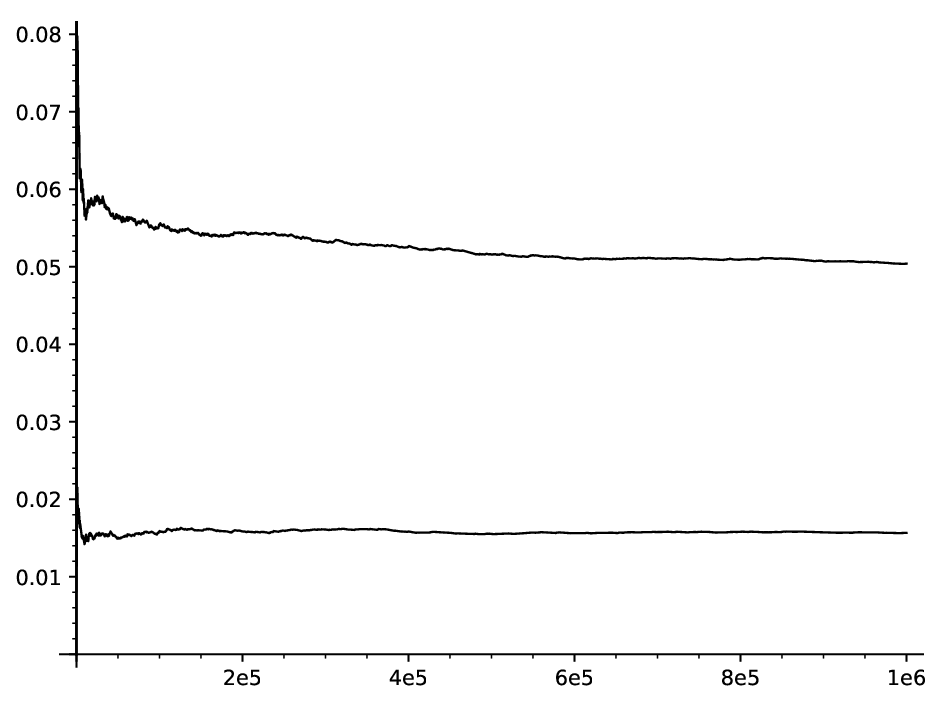}
\caption{$|l| = 3$: Top -3 bottom 3} \label{fig:19_6_A_3}
\end{subfigure}
\hspace*{-2.3cm}
\begin{subfigure}[b]{0.43\linewidth}
\includegraphics[width=\linewidth]{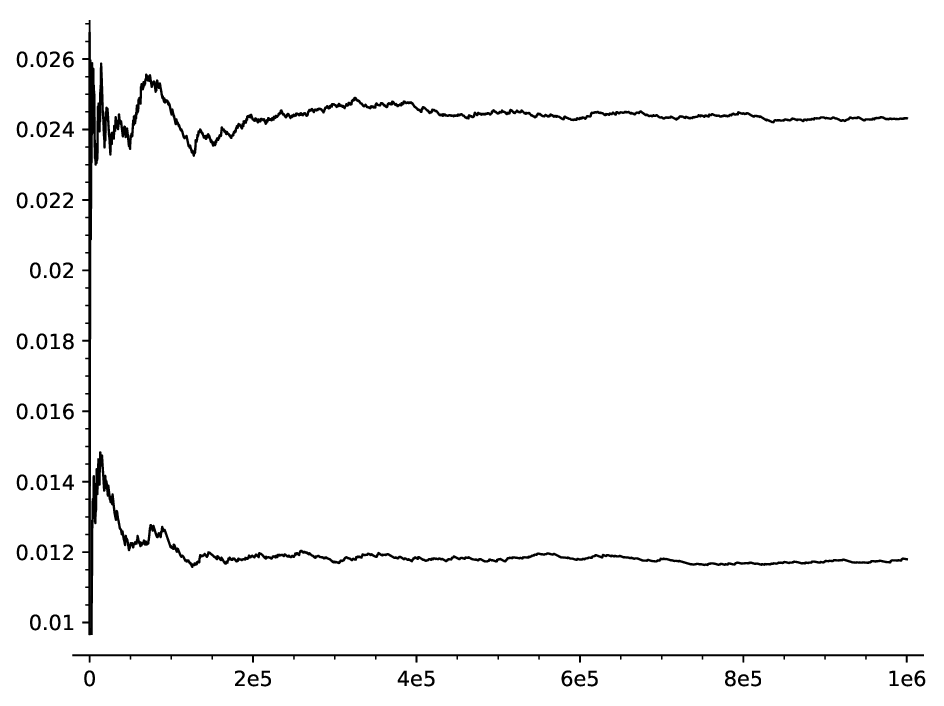}
\caption{$|l| = 4$: Top -4 bottom 4} \label{fig:19_6_A_4}
\end{subfigure}\hspace*{\fill}
\begin{subfigure}[b]{0.43\linewidth}
\includegraphics[width=\linewidth]{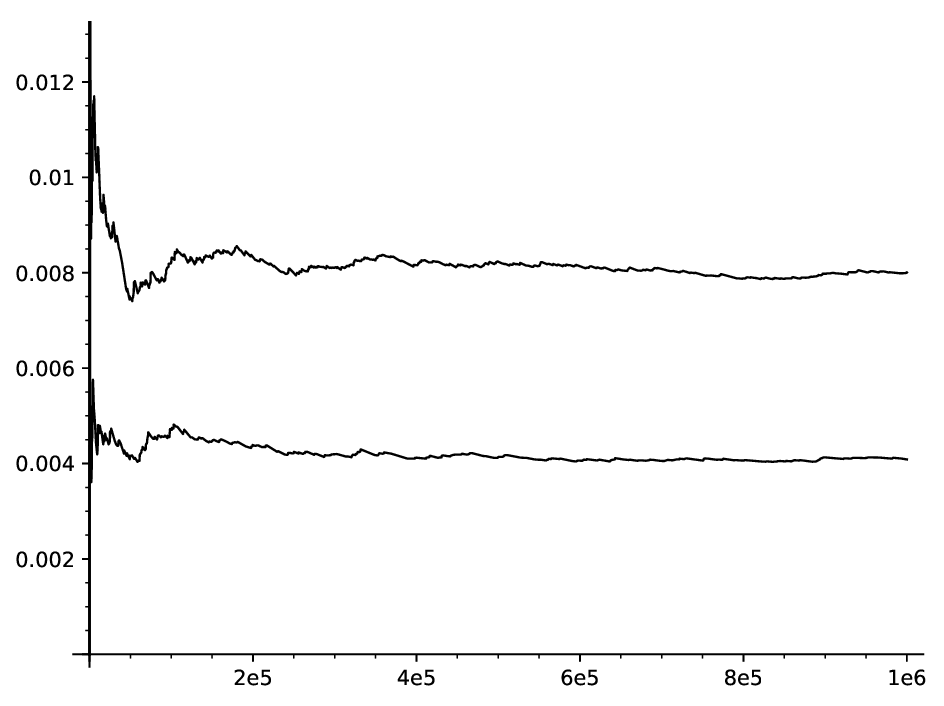}
\caption{$|l| = 5$: Top 5 bottom -5} \label{fig:19_6_A_5}
\end{subfigure}\hspace*{\fill}
\begin{subfigure}[b]{0.43\linewidth}
\includegraphics[width=\linewidth]{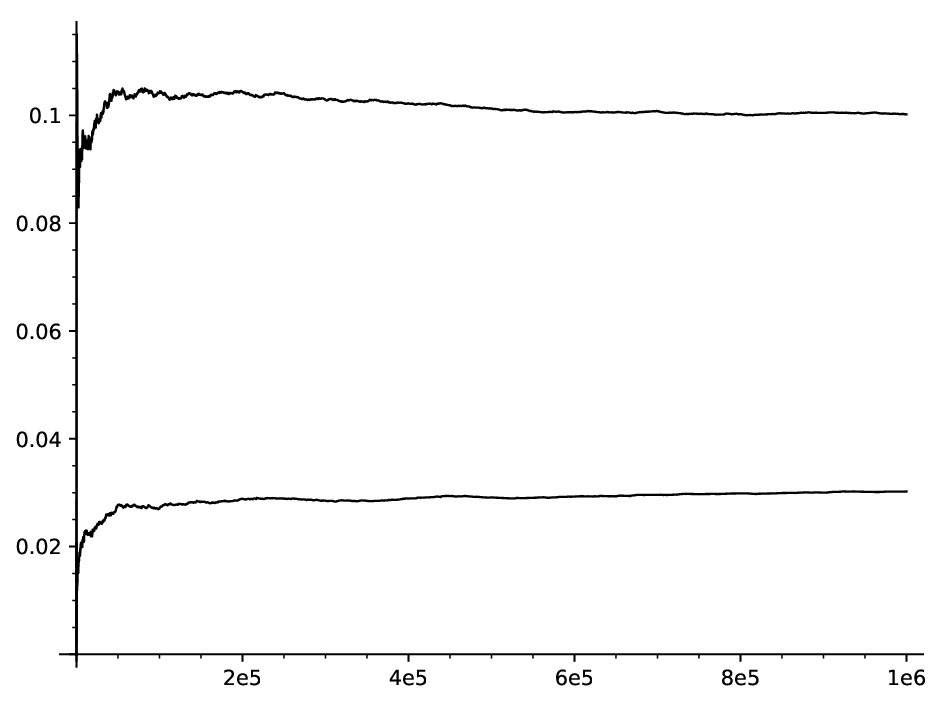}
\caption{$|l| = 6$: Top 6 bottom -6} \label{fig:19_6_A_6}
\end{subfigure}
\hspace*{-2.3cm}
\begin{subfigure}[b]{0.43\linewidth}
\includegraphics[width=\linewidth]{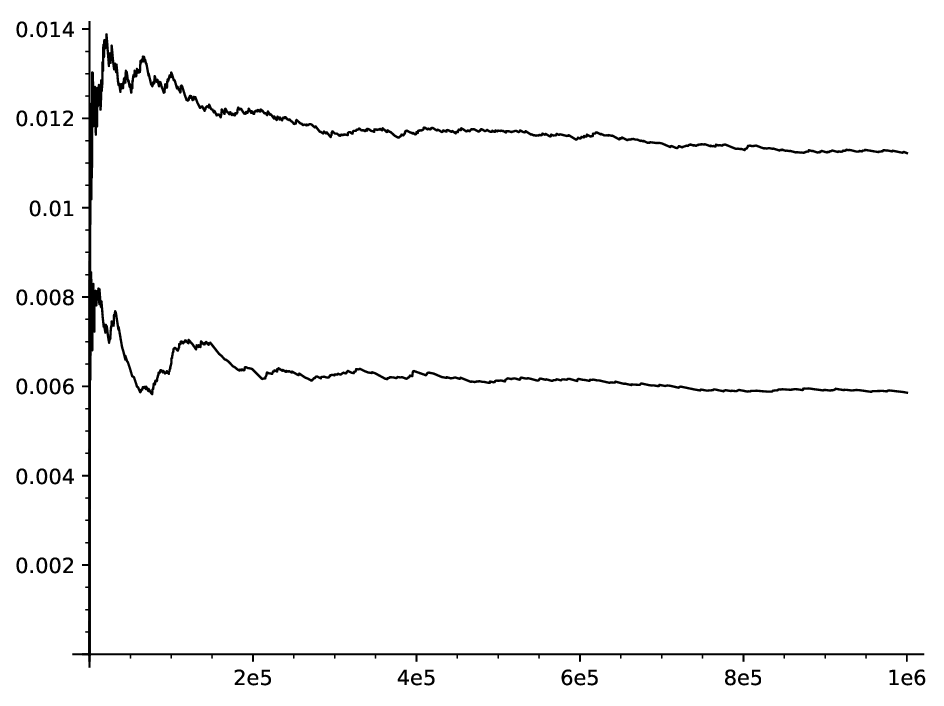}
\caption{$|l| = 7$: Top -7 bottom 7} \label{fig:19_6_A_7}
\end{subfigure}\hspace*{\fill}
\begin{subfigure}[b]{0.43\linewidth}
\includegraphics[width=\linewidth]{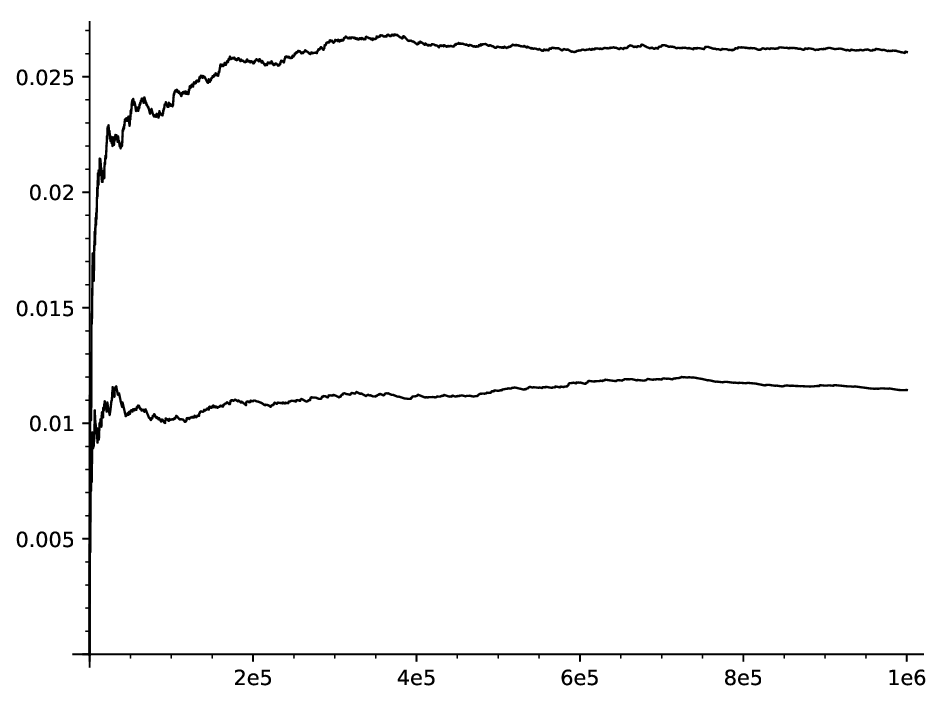}
\caption{$|l| = 8$: Top 8 bottom -8} \label{fig:19_6_A_8}
\end{subfigure}\hspace*{\fill}
\begin{subfigure}[b]{0.43\linewidth}
\includegraphics[width=\linewidth]{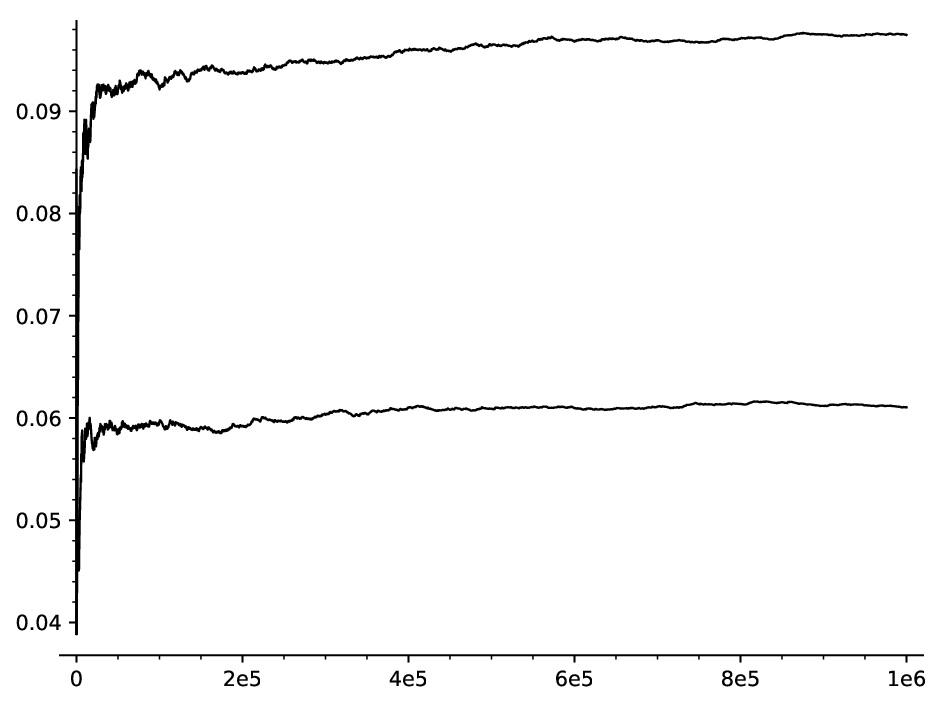}
\caption{$|l| = 9$: Top -9 bottom 9} \label{fig:19_6_A_9}
\end{subfigure}
\caption{Ratio~\eqref{ratio_A_exact} 19a1: $x(X;l)/X^{1/2}\log^2(X)$ for $k = 6$} \label{fig:19a1_6_A_exact}
\end{figure}

\clearpage

\begin{figure}[t!] 
\hspace*{-2.3cm}
\begin{subfigure}[b]{0.43\linewidth}
\includegraphics[width=\linewidth]{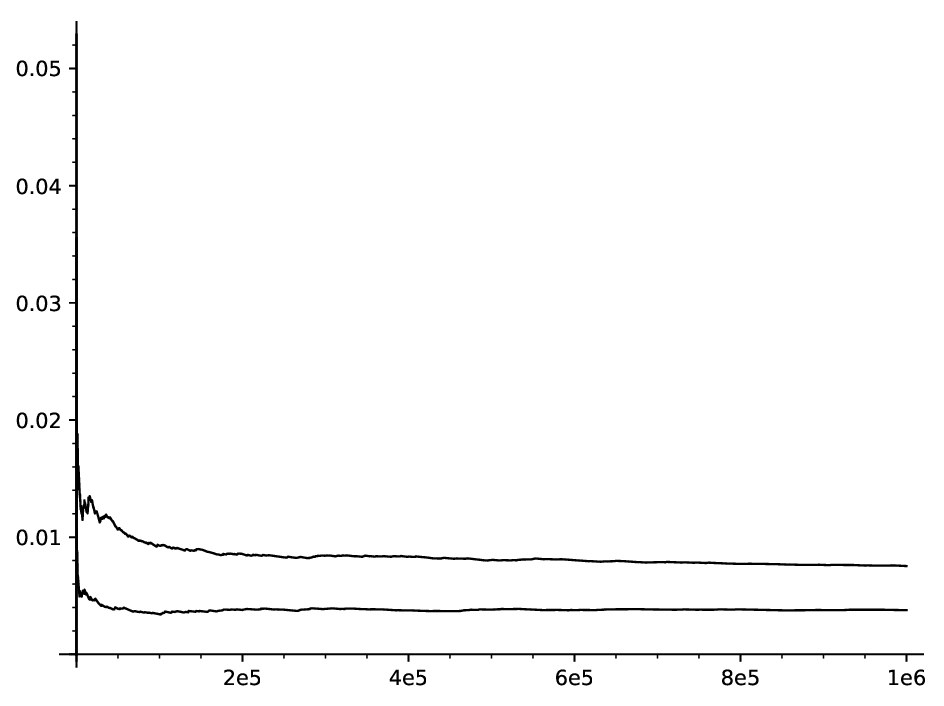}
\caption{$|l| = 1$: Top 1 bottom -1} \label{fig:37_6_A_1}
\end{subfigure}\hspace*{\fill}
\begin{subfigure}[b]{0.43\linewidth}
\includegraphics[width=\linewidth]{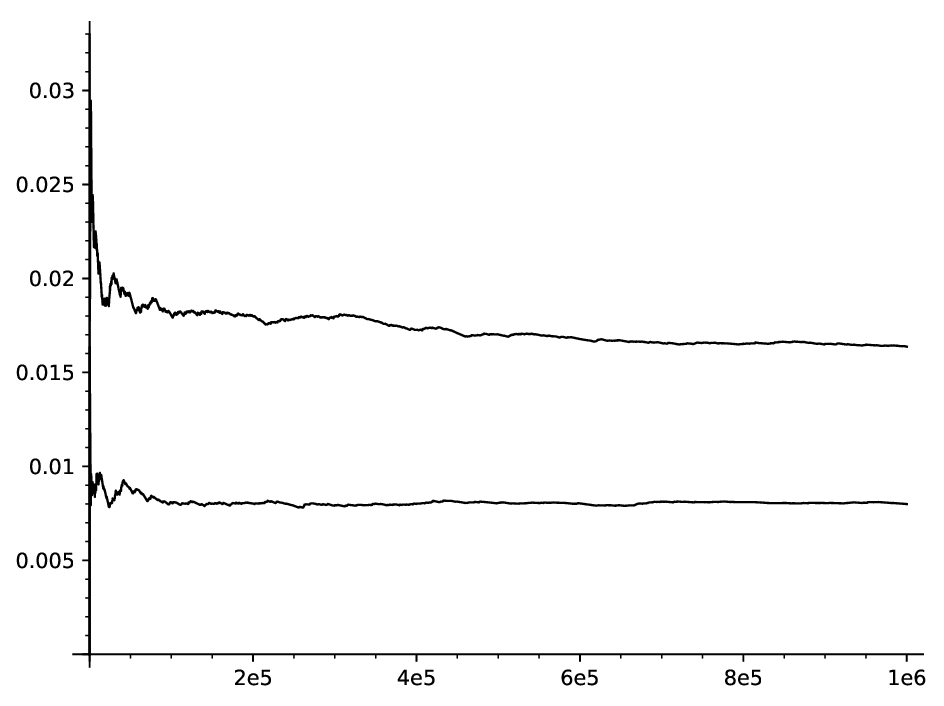}
\caption{$|l| = 2$: Top -2 bottom 2} \label{fig:37_6_A_2}
\end{subfigure}\hspace*{\fill}
\begin{subfigure}[b]{0.43\linewidth}
\includegraphics[width=\linewidth]{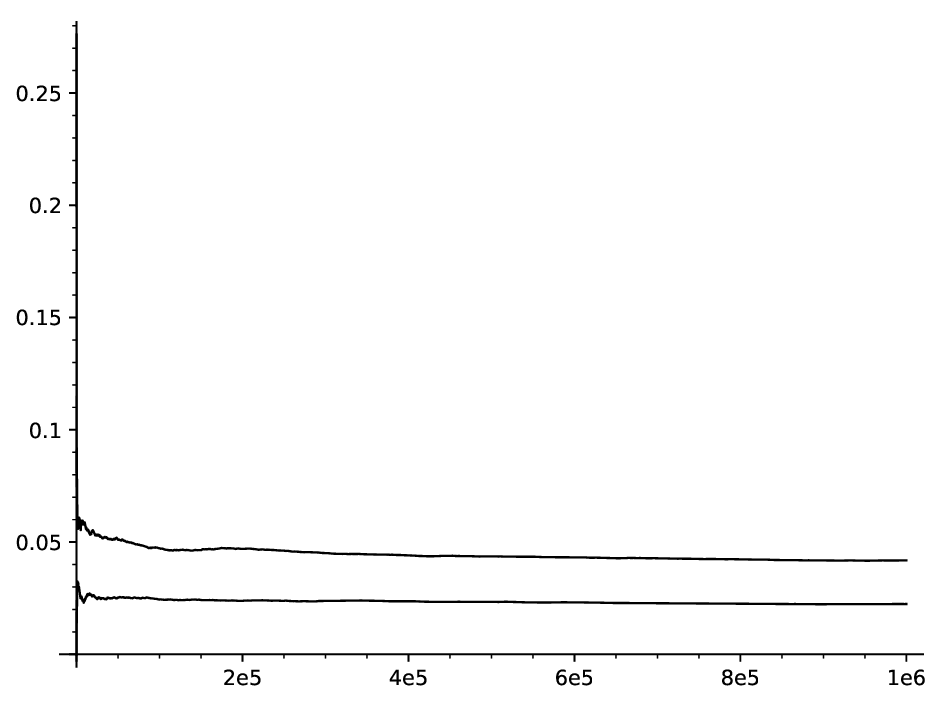}
\caption{$|l| = 3$: Top 3 bottom -3} \label{fig:37_6_A_3}
\end{subfigure}
\hspace*{-2.3cm}
\begin{subfigure}[b]{0.43\linewidth}
\includegraphics[width=\linewidth]{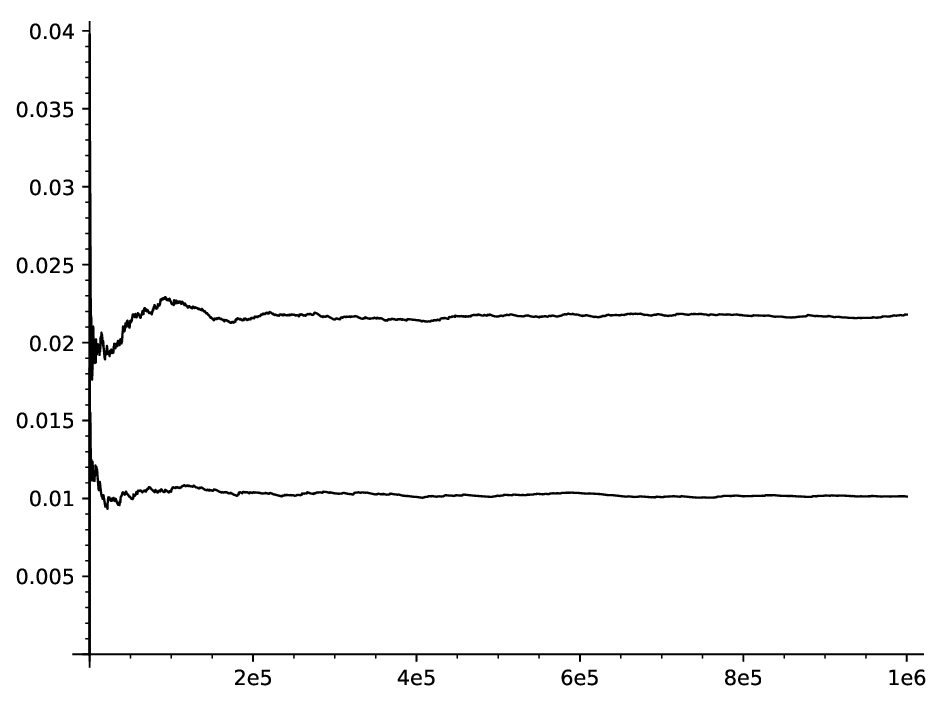}
\caption{$|l| = 4$: Top 4 bottom -4} \label{fig:37_6_A_4}
\end{subfigure}\hspace*{\fill}
\begin{subfigure}[b]{0.43\linewidth}
\includegraphics[width=\linewidth]{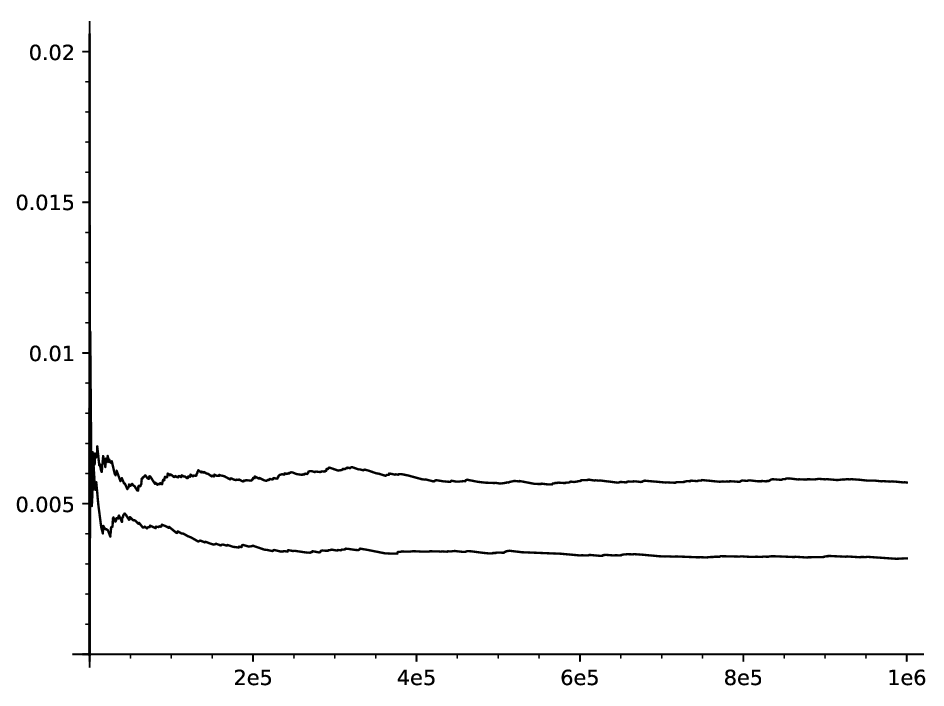}
\caption{$|l| = 5$: Top -5 bottom 5} \label{fig:37_6_A_5}
\end{subfigure}\hspace*{\fill}
\begin{subfigure}[b]{0.43\linewidth}
\includegraphics[width=\linewidth]{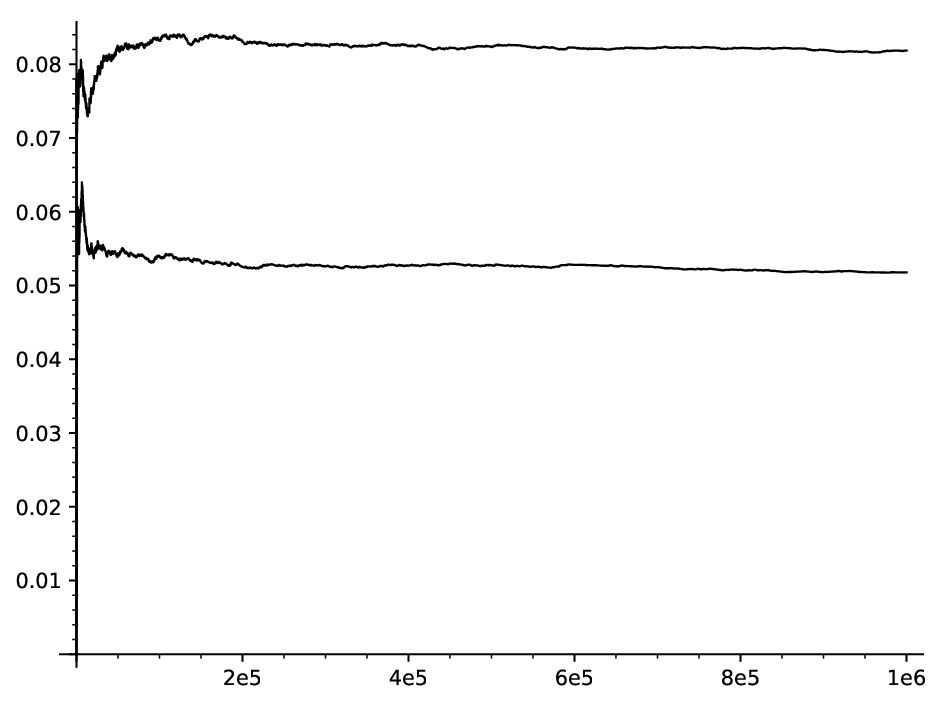}
\caption{$|l| = 6$: Top -6 bottom 6} \label{fig:37_6_A_6}
\end{subfigure}
\hspace*{-2.3cm}
\begin{subfigure}[b]{0.43\linewidth}
\includegraphics[width=\linewidth]{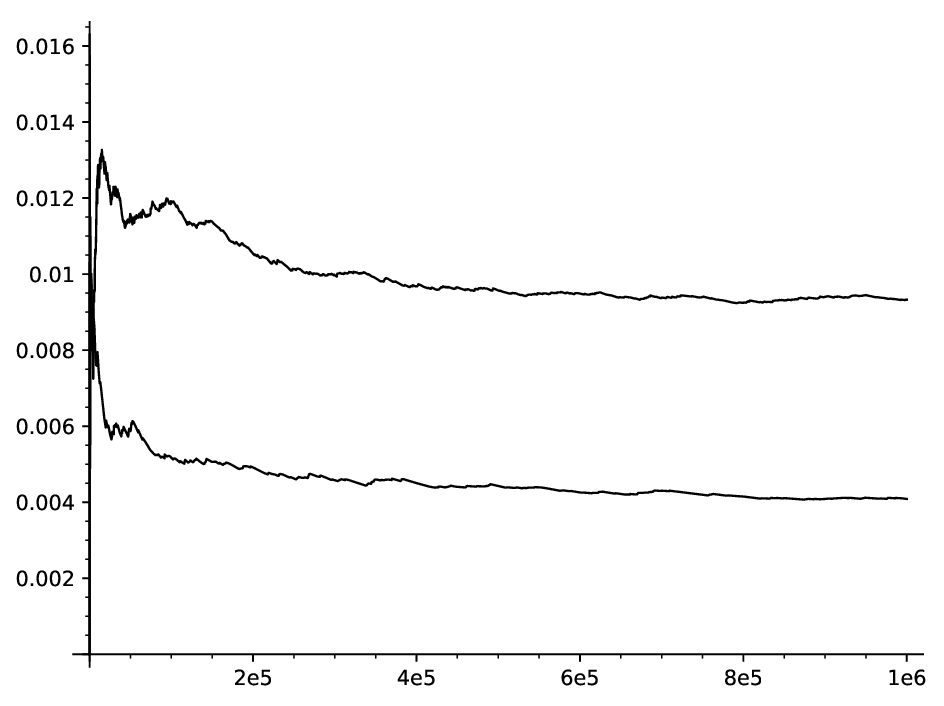}
\caption{$|l| = 7$: Top 7 bottom -7} \label{fig:37_6_A_7}
\end{subfigure}\hspace*{\fill}
\begin{subfigure}[b]{0.43\linewidth}
\includegraphics[width=\linewidth]{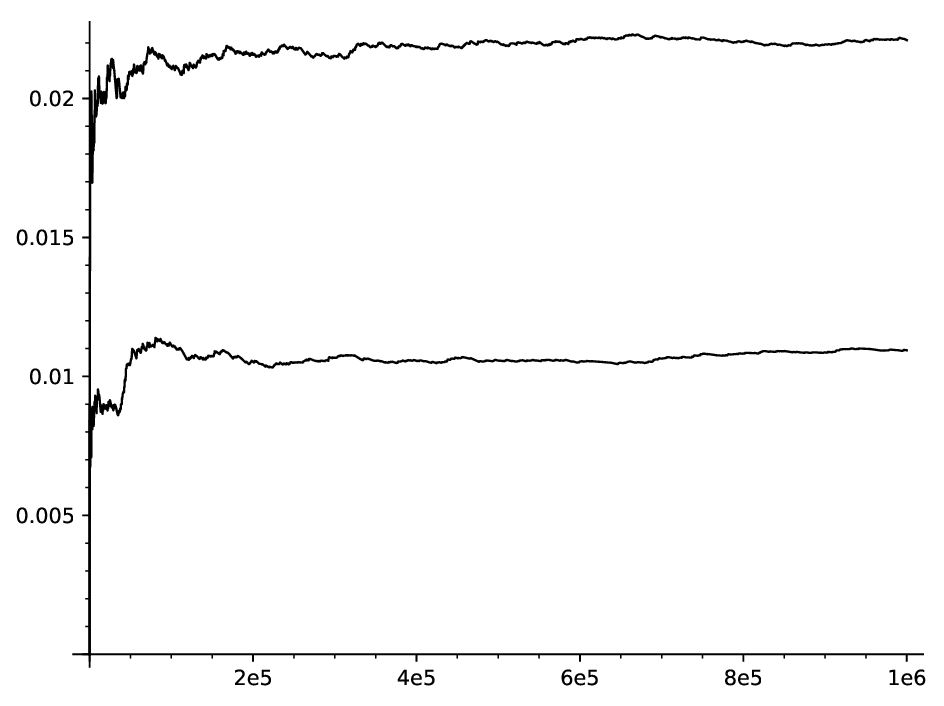}
\caption{$|l| = 8$: Top -8 bottom 8} \label{fig:37_6_A_8}
\end{subfigure}\hspace*{\fill}
\begin{subfigure}[b]{0.43\linewidth}
\includegraphics[width=\linewidth]{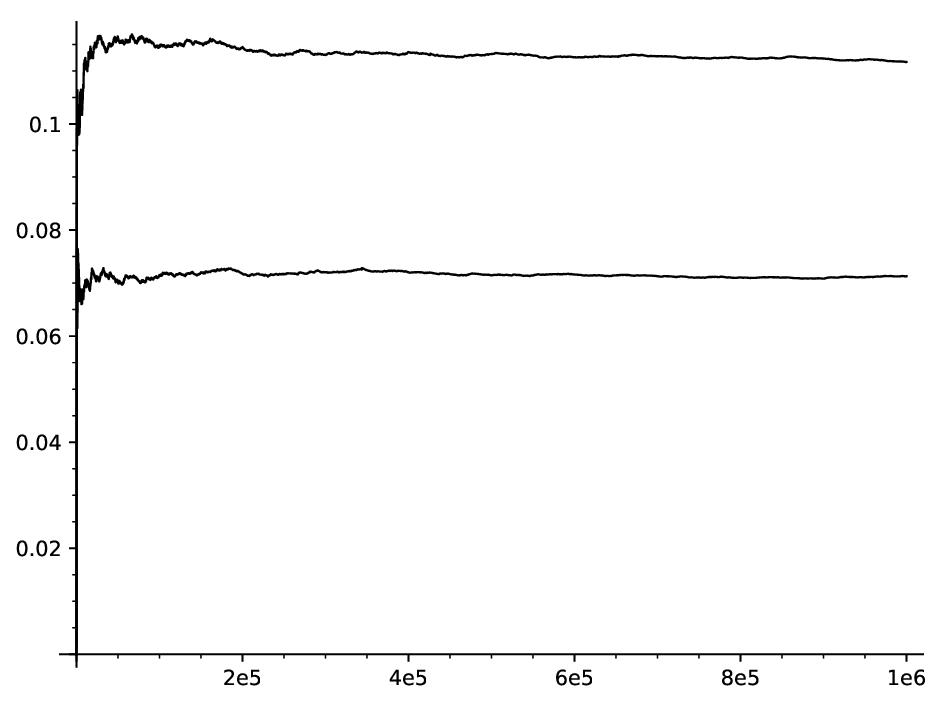}
\caption{$|l| = 9$: Top 9 bottom -9} \label{fig:37_6_A_9}
\end{subfigure}
\caption{Ratio~\eqref{ratio_A_exact} 37b1: $x(X;l)/X^{1/2}\log^2(X)$ for $k = 6$} \label{fig:37b1_6_A_exact}
\end{figure}

\clearpage

\begin{figure}[b!] 
\hspace*{-.7cm}
\begin{subfigure}[b]{0.4\linewidth}
\includegraphics[width=\linewidth]{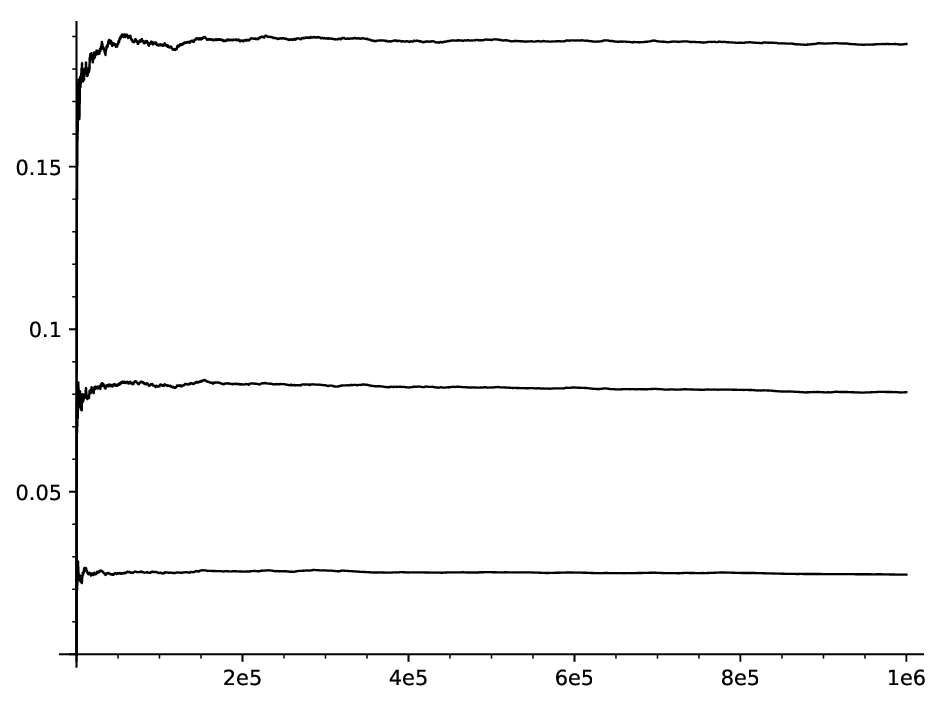}
\caption{11a1: $n_{6,E}^+(X;L)/X^{1/2}\log^2(X)$} \label{fig:11_6_even_acc_A}
\end{subfigure}\hspace*{\fill}
\begin{subfigure}[b]{0.4\linewidth}
\includegraphics[width=\linewidth]{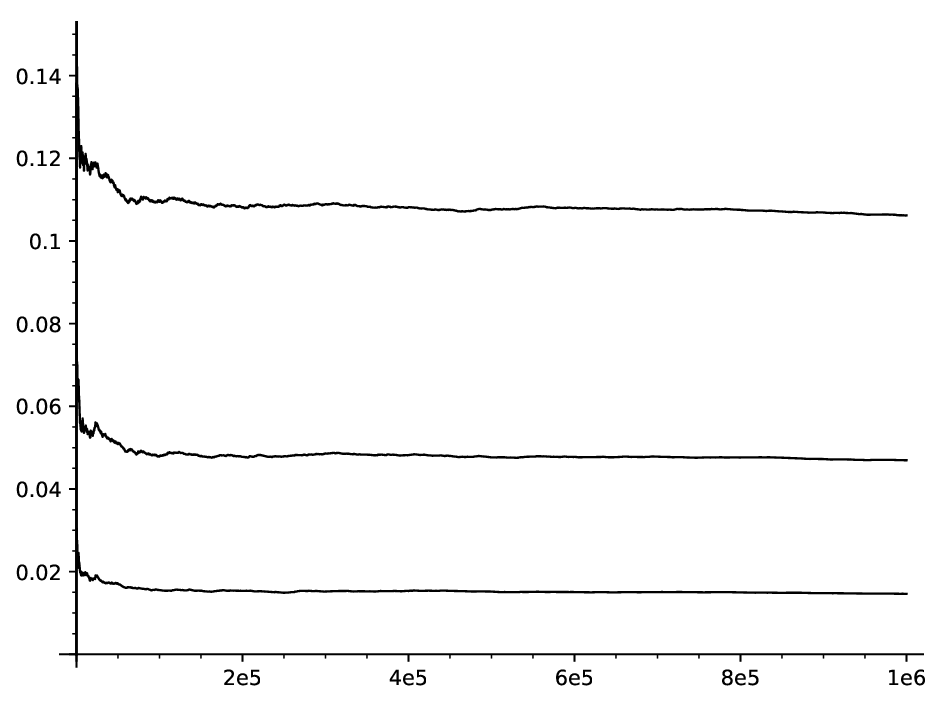}
\caption{11a1: $n_{6,E}^-(X;L)/X^{1/2}\log^2(X)$} \label{fig:11_6_odd_acc_A}
\end{subfigure}
\hspace*{-.7cm}
\begin{subfigure}[b]{0.4\linewidth}
\includegraphics[width=\linewidth]{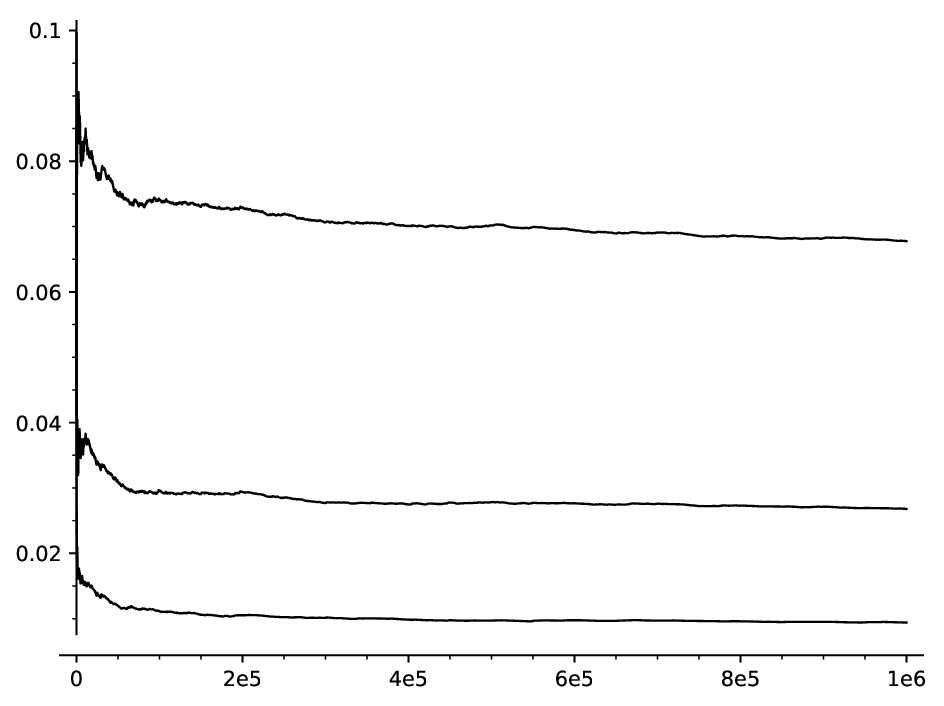}
\caption{14a1: $n_{6,E}^+(X;L)/X^{1/2}\log^2(X)$} \label{fig:14_6_even_acc_A}
\end{subfigure}\hspace*{\fill}
\begin{subfigure}[b]{0.4\linewidth}
\includegraphics[width=\linewidth]{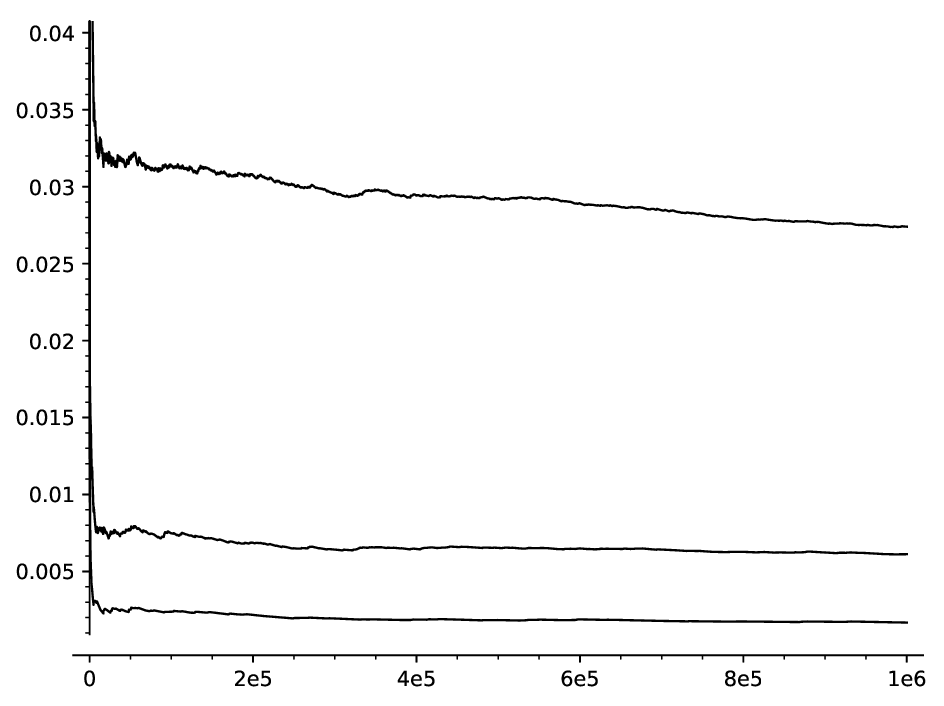}
\caption{14a1: $n_{6,E}^-(X;L)/X^{1/2}\log^2(X)$} \label{fig:14_6_odd_acc_A}
\end{subfigure}
\hspace*{-.7cm}
\begin{subfigure}[b]{0.4\linewidth}
\includegraphics[width=\linewidth]{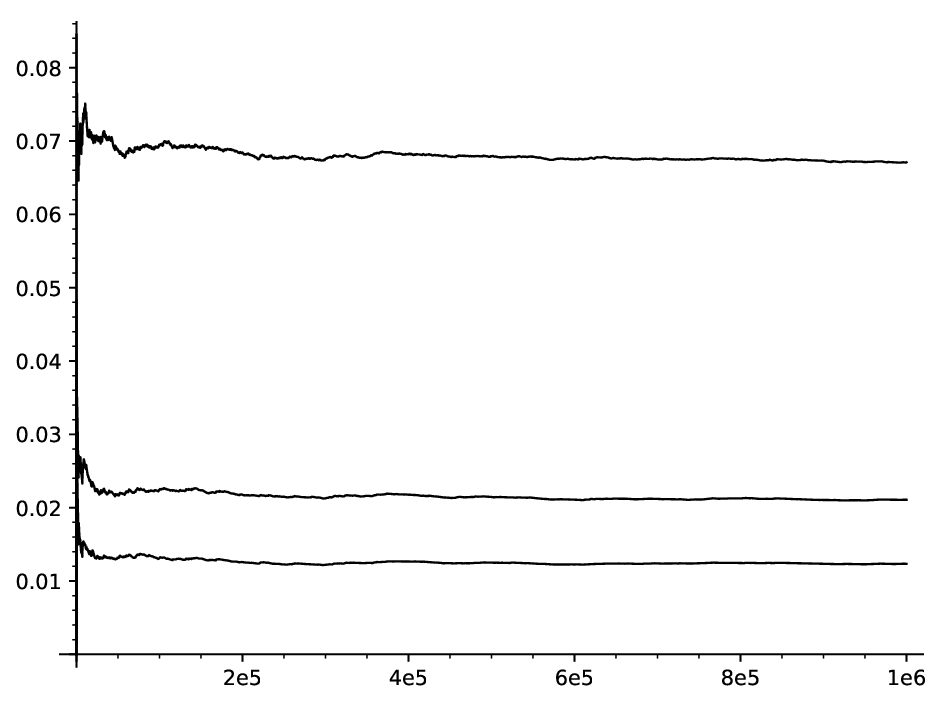}
\caption{15a1: $n_{6,E}^+(X;L)/X^{1/2}\log^2(X)$} \label{fig:15_6_even_acc_A}
\end{subfigure}\hspace*{\fill}
\begin{subfigure}[b]{0.4\linewidth}
\includegraphics[width=\linewidth]{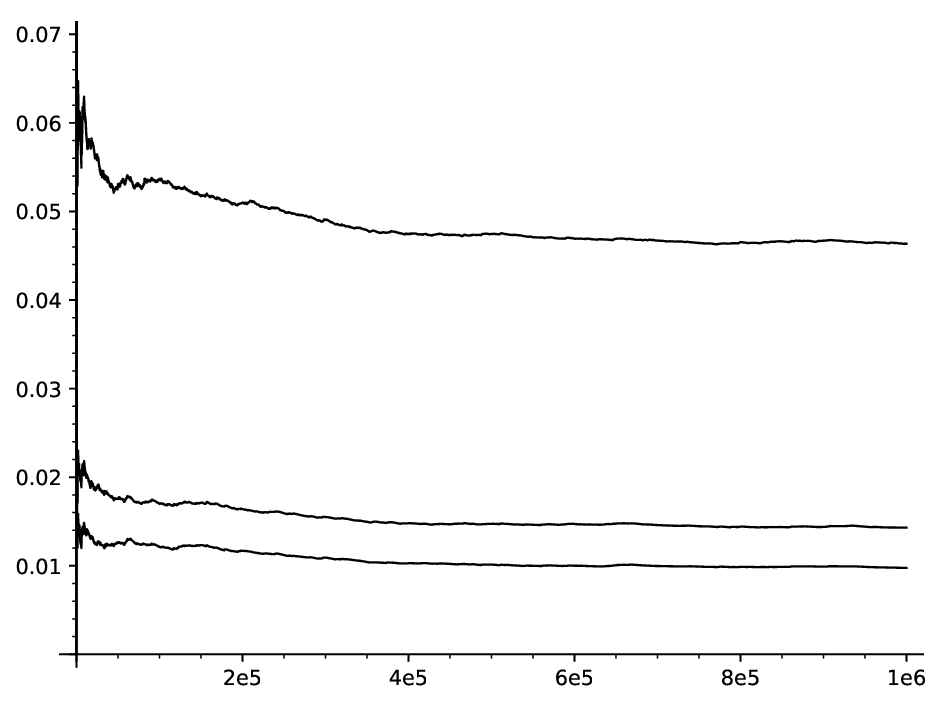}
\caption{15a1: $n_{6,E}^-(X;L)/X^{1/2}\log^2(X)$} \label{fig:15_6_odd_acc_A}
\end{subfigure}
\caption{11a1, 14a1, 15a1: Ratio~\eqref{ratio_N_pm} $n_{6,E}^\pm(X;L)/X^{1/2}\log^2(X)$ depending on $\chi(-1) = \pm 1$ for $k = 6$ and $L = $ 1, 2, 3. Note that the larger $L$ the higher its ratio graph is depicted.} \label{fig:6_even_odd_A_acc_11_14_15}
\end{figure}

\clearpage

\begin{figure}[b!] 
\hspace*{-.7cm}
\begin{subfigure}[b]{0.4\linewidth}
\includegraphics[width=\linewidth]{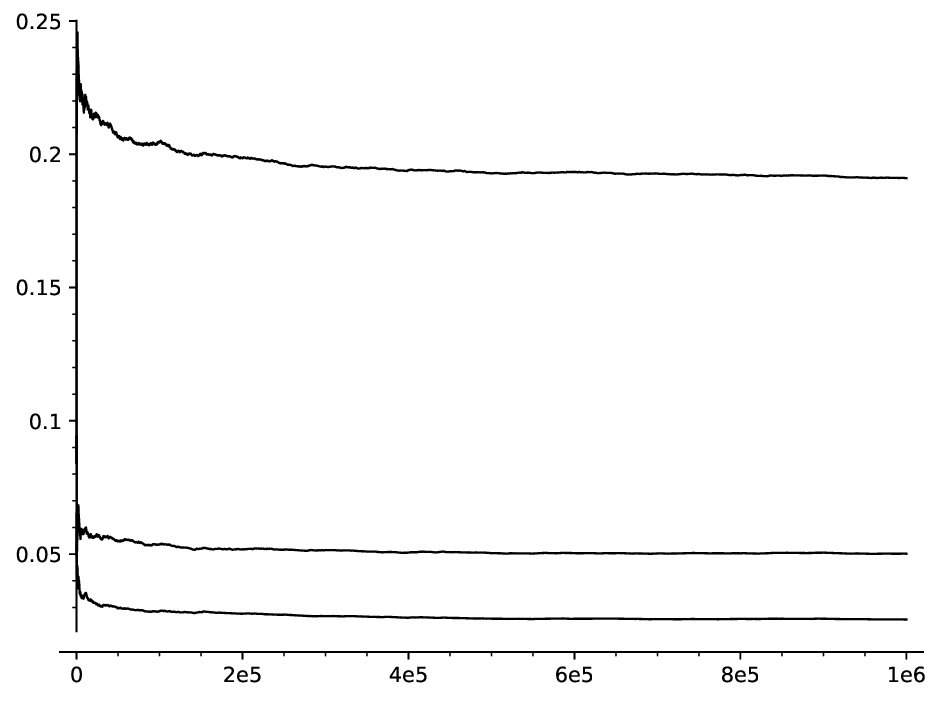}
\caption{17a1: $n_{6,E}^+(X;L)/X^{1/2}\log^2(X)$} \label{fig:17_6_even_acc_A}
\end{subfigure}\hspace*{\fill}
\begin{subfigure}[b]{0.4\linewidth}
\includegraphics[width=\linewidth]{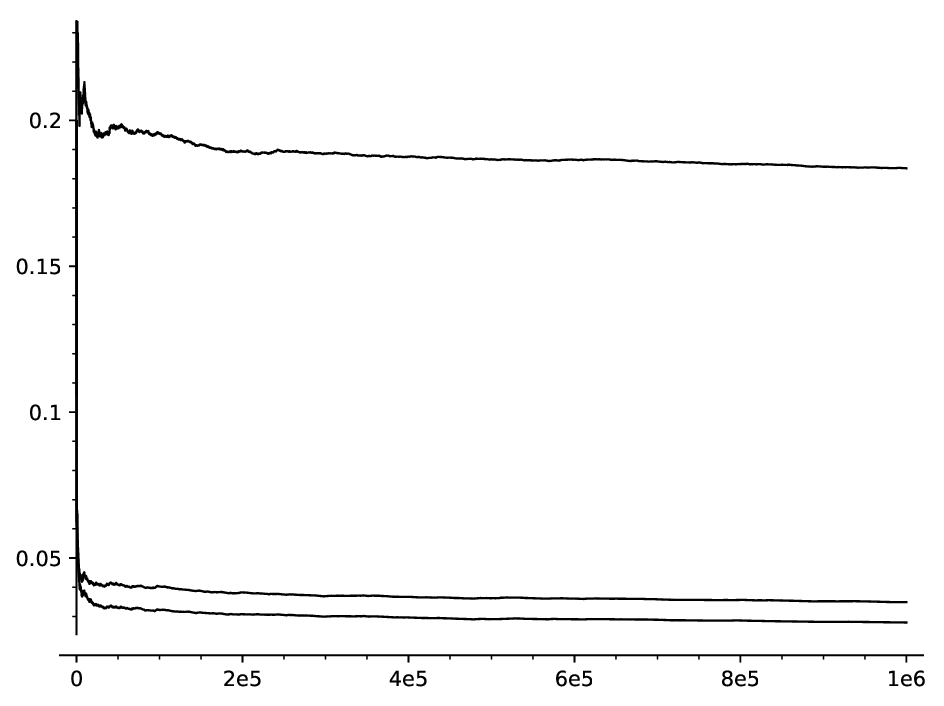}
\caption{17a1: $n_{6,E}^-(X;L)/X^{1/2}\log^2(X)$} \label{fig:17_6_odd_acc_A}
\end{subfigure}
\hspace*{-.7cm}
\begin{subfigure}[b]{0.4\linewidth}
\includegraphics[width=\linewidth]{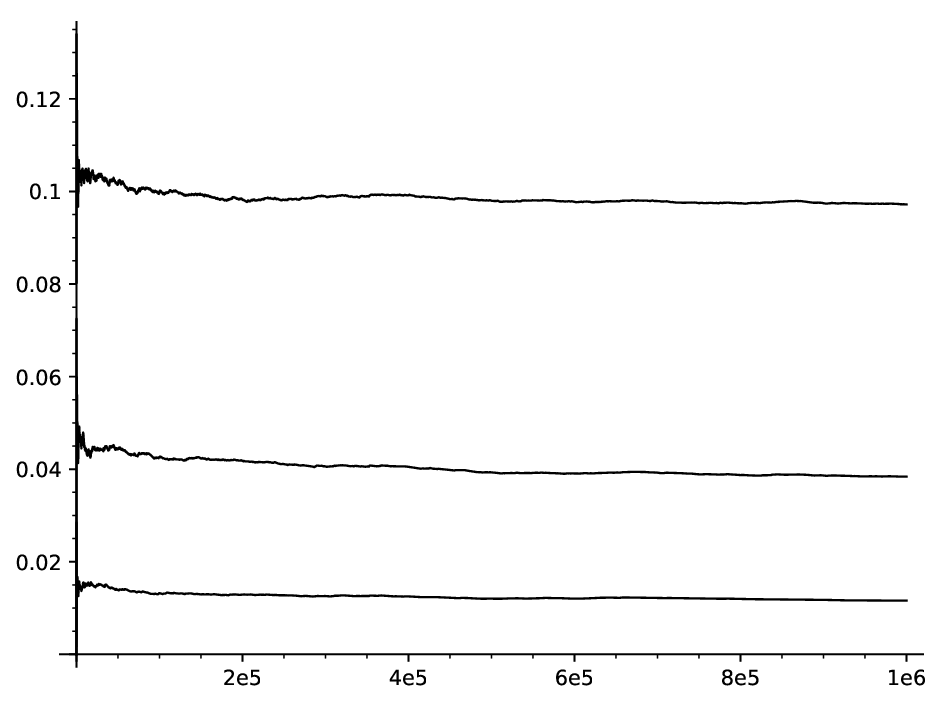}
\caption{19a1: $n_{6,E}^+(X;L)/X^{1/2}\log^2(X)$} \label{fig:19_6_even_acc_A}
\end{subfigure}\hspace*{\fill}
\begin{subfigure}[b]{0.4\linewidth}
\includegraphics[width=\linewidth]{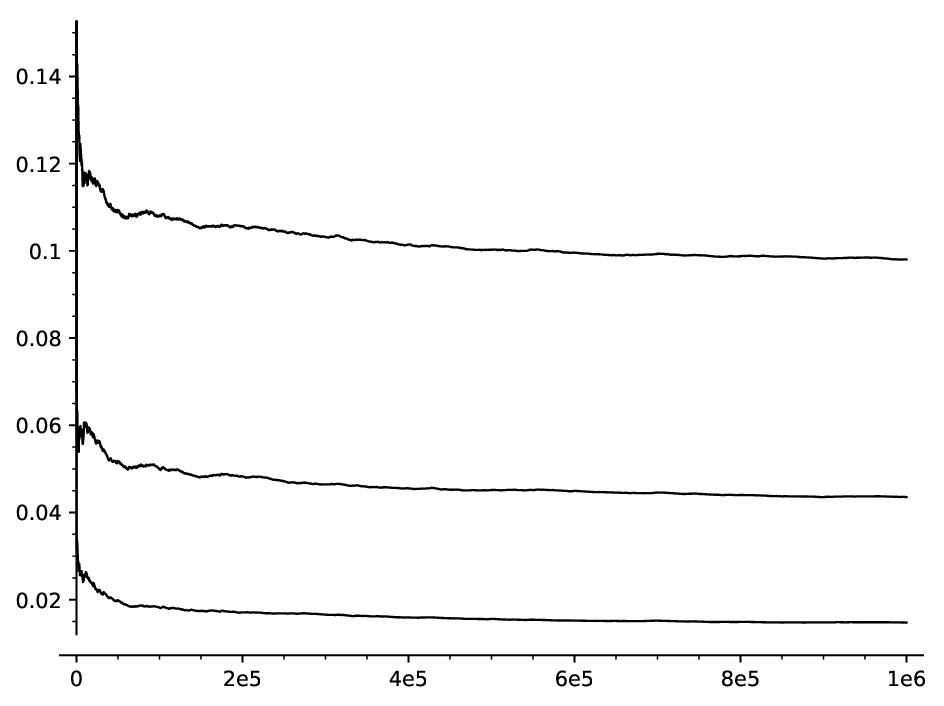}
\caption{19a1: $n_{6,E}^-(X;L)/X^{1/2}\log^2(X)$} \label{fig:19_6_odd_acc_A}
\end{subfigure}
\hspace*{-.7cm}
\begin{subfigure}[b]{0.4\linewidth}
\includegraphics[width=\linewidth]{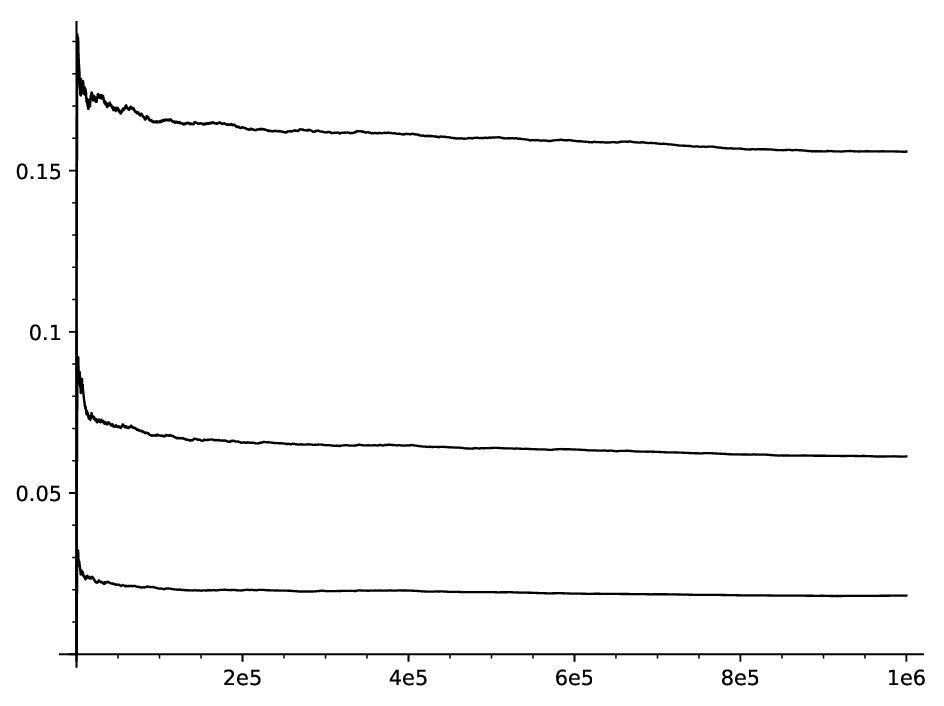}
\caption{37b1: $n_{6,E}^+(X;L)/X^{1/2}\log^2(X)$} \label{fig:37_6_even_acc_A}
\end{subfigure}\hspace*{\fill}
\begin{subfigure}[b]{0.4\linewidth}
\includegraphics[width=\linewidth]{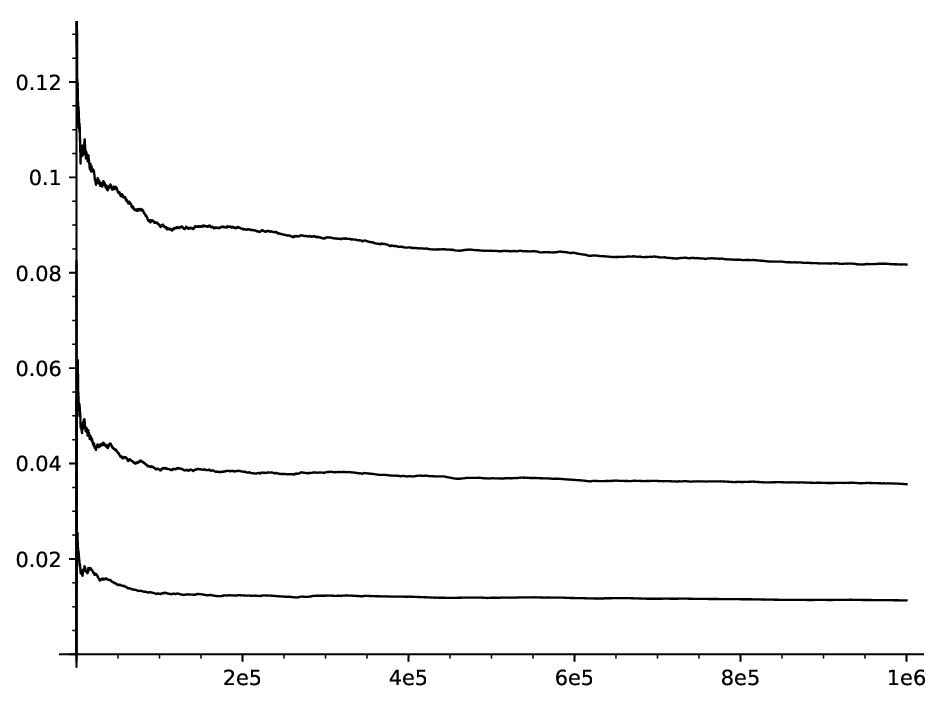}
\caption{37b1: $n_{6,E}^-(X;L)/X^{1/2}\log^2(X)$} \label{fig:37_6_odd_acc_A}
\end{subfigure}
\caption{17a1, 19a1, 37b1: Ratio~\eqref{ratio_N_pm} $n_{6,E}^\pm(X;L)/X^{1/2}\log^2(X)$ depending on $\chi(-1) = \pm 1$ for $k = 6$ and $L = $ 1, 2, 3. Note that the larger $L$ the higher its ratio graph is depicted.} \label{fig:6_even_odd_A_acc_17_19_37}
\end{figure}

\clearpage

\begin{figure}[t!] 
\hspace*{-.7cm}
\begin{subfigure}[b]{0.4\linewidth}
\includegraphics[width=\linewidth]{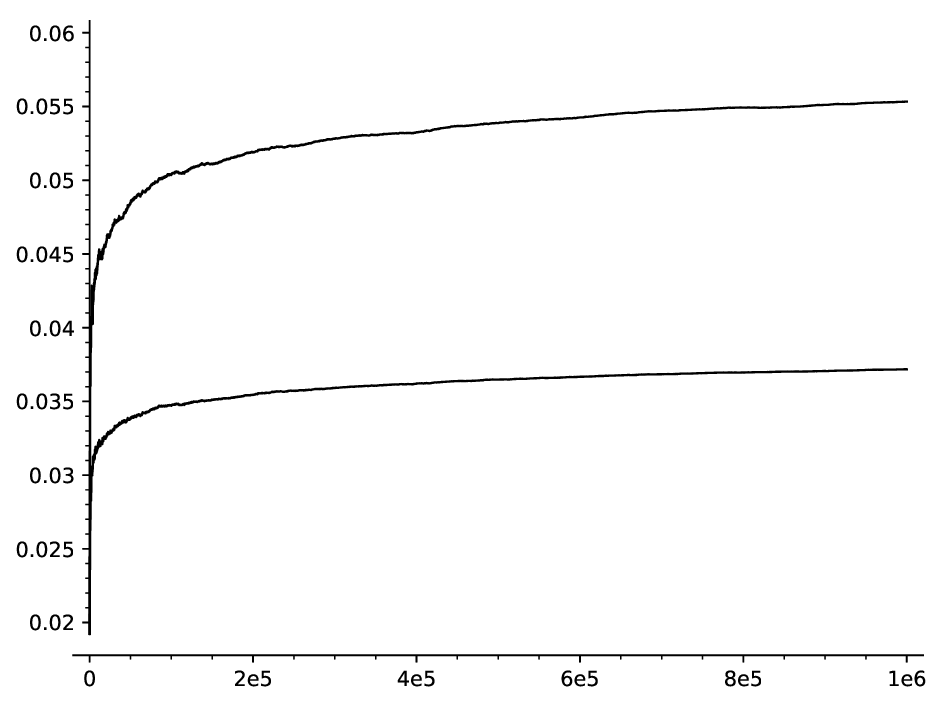}
\caption{11a1:  $m_{6,E}^+(X;c)/X^{c +1/2}\log^2(X)$ Top to bottom $c =$ 0.3, 0.4} \label{fig:11_6_even_acc_c}
\end{subfigure}\hspace*{\fill}
\begin{subfigure}[b]{0.4\linewidth}
\includegraphics[width=\linewidth]{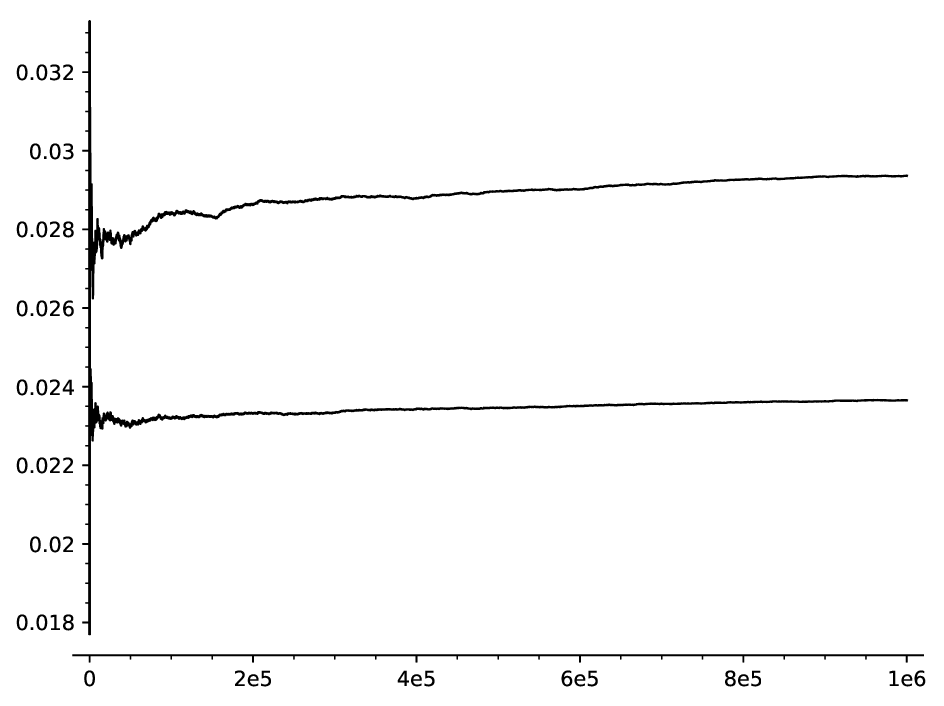}
\caption{11a1: $m_{6,E}^-(X;c)/X^{c +1/2}\log^2(X)$ Top to bottom $c =$ 0.3, 0.4} \label{fig:11_6_odd_acc_c}
\end{subfigure}
\hspace*{-.7cm}
\begin{subfigure}[b]{0.4\linewidth}
\includegraphics[width=\linewidth]{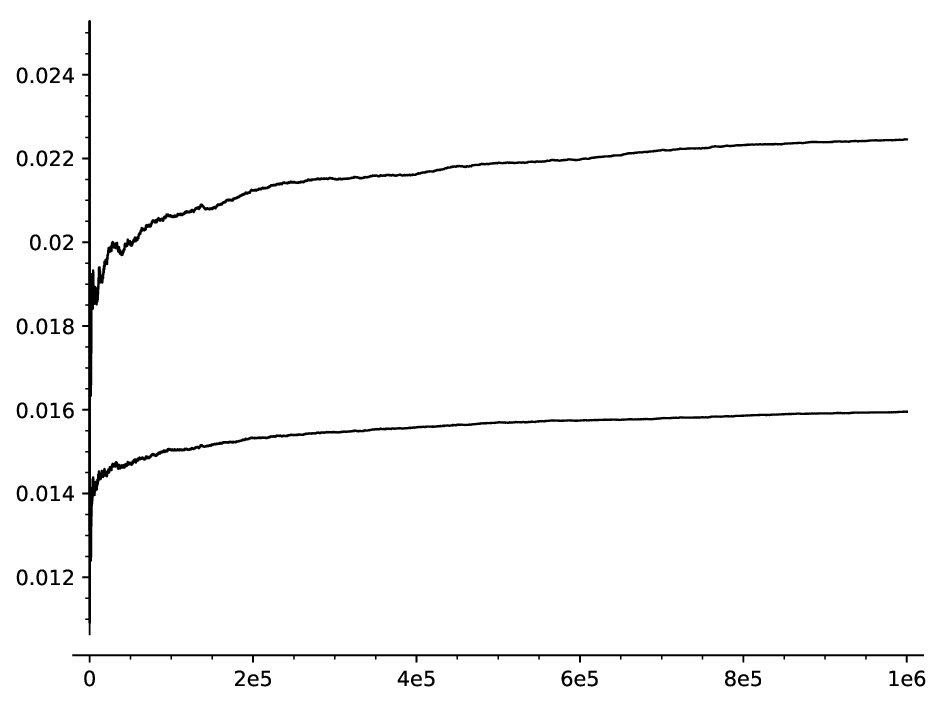}
\caption{14a1: $m_{6,E}^+(X;c)/X^{c +1/2}\log^2(X)$ Top to bottom $c =$ 0.3, 0.4} \label{fig:14_6_even_acc_c}
\end{subfigure}\hspace*{\fill}
\begin{subfigure}[b]{0.4\linewidth}
\includegraphics[width=\linewidth]{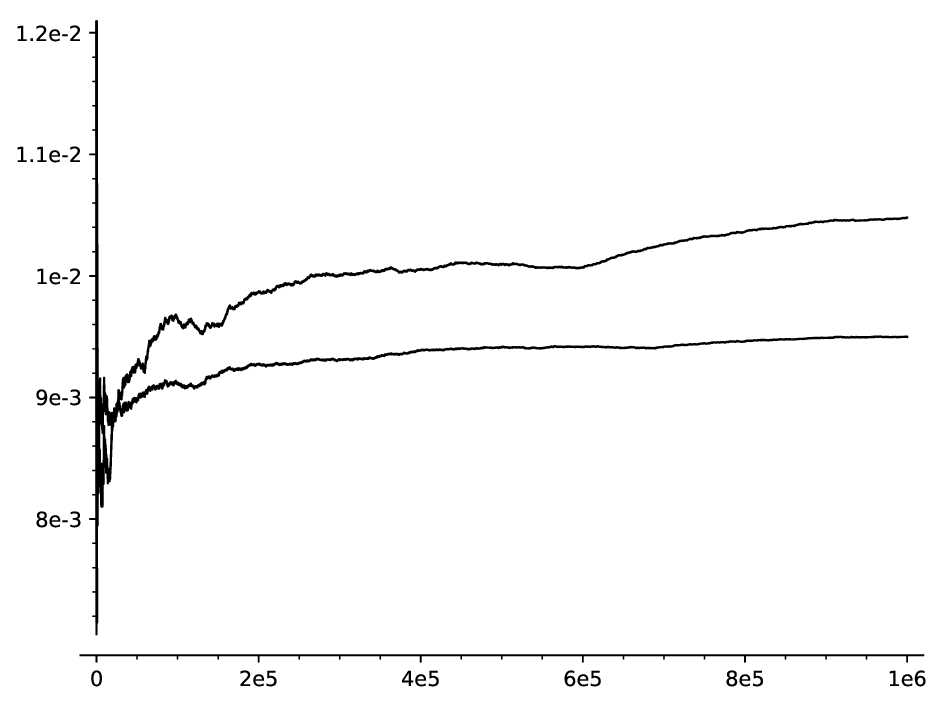}
\caption{14a1: $m_{6,E}^-(X;c)/X^{c +1/2}\log^2(X)$ Top to bottom $c =$ 0.3, 0.4} \label{fig:14_6_odd_acc_c}
\end{subfigure}
\hspace*{-.7cm}
\begin{subfigure}[b]{0.4\linewidth}
\includegraphics[width=\linewidth]{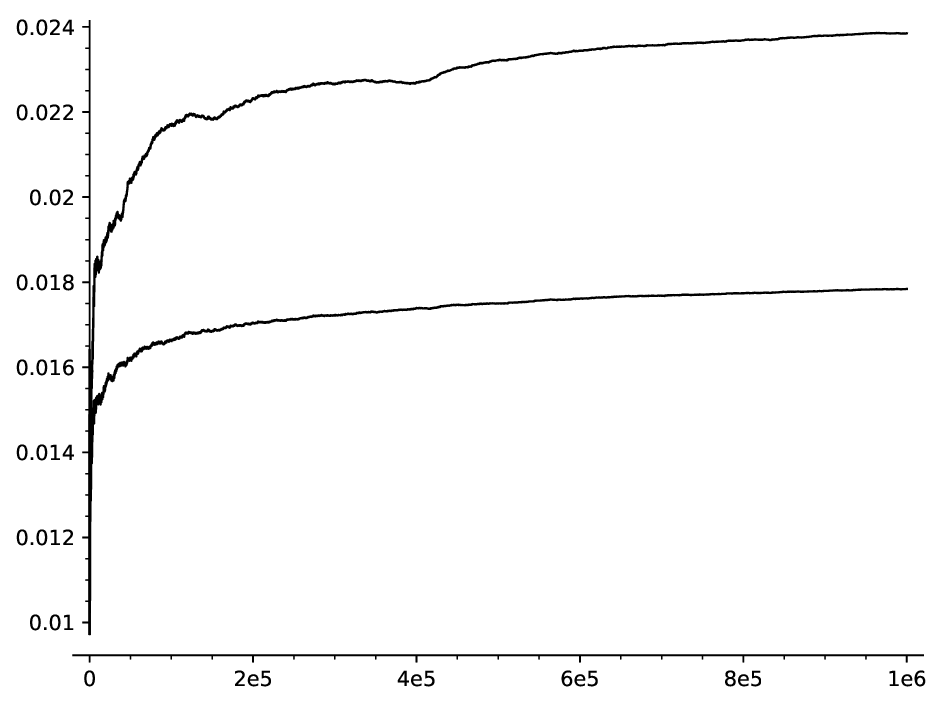}
\caption{15a1: $m_{6,E}^+(X;c)/X^{c +1/2}\log^2(X)$ Top to bottom $c =$ 0.3, 0.4} \label{fig:15_6_even_acc_c}
\end{subfigure}\hspace*{\fill}
\begin{subfigure}[b]{0.4\linewidth}
\includegraphics[width=\linewidth]{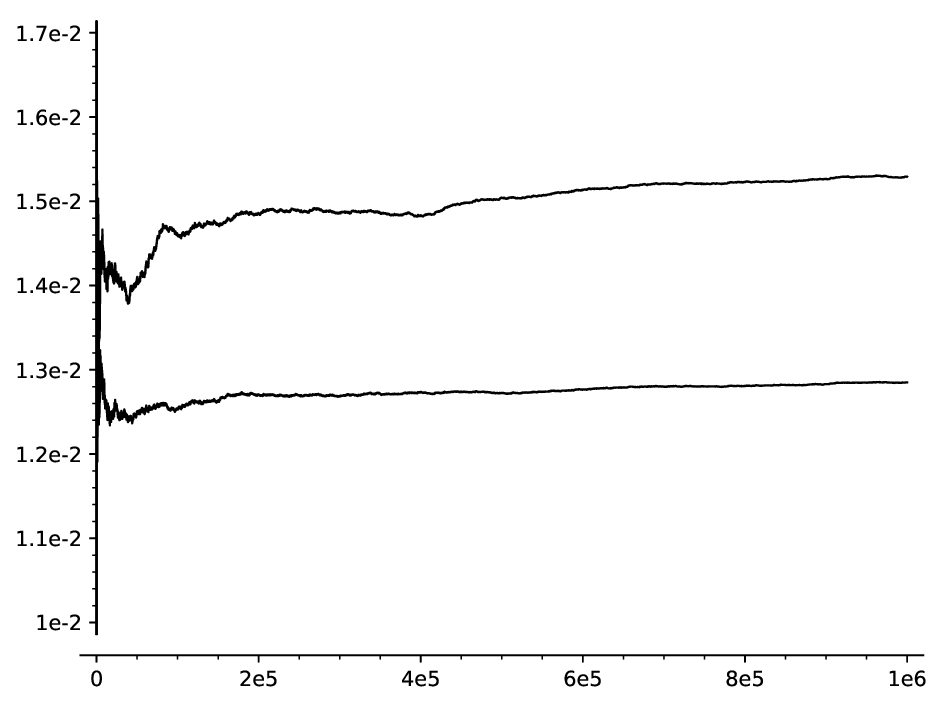}
\caption{15a1: $m_{6,E}^-(X;c)/X^{c +1/2}\log^2(X)$ Top to bottom $c =$ 0.3, 0.4} \label{fig:15_6_odd_acc_c}
\end{subfigure}
\caption{11a1, 14a1, 15a1: Ratio~\eqref{ratio_M_pm} $m_{6,E}^\pm(X;c)/X^{c +1/2}\log^2(X)$ depending on $\chi(-1) = \pm 1$ for $k = 6$ and $c = $ 0.3, 0.4 } \label{fig:c_11_14_15_pm_acc_6}
\end{figure}
\clearpage

\begin{figure}[t!] 
\hspace*{-.7cm}
\begin{subfigure}[b]{0.4\linewidth}
\includegraphics[width=\linewidth]{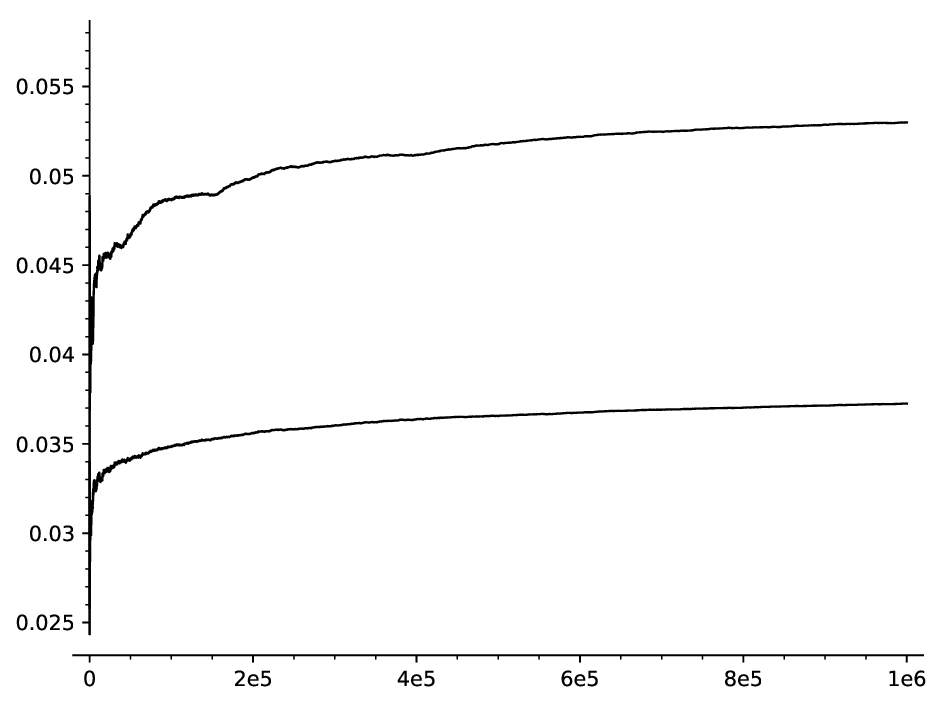}
\caption{17a1:  $m_{6,E}^+(X;c)/X^{c +1/2}\log^2(X)$ Top to bottom $c =$ 0.3, 0.4} \label{fig:17_6_even_acc_c}
\end{subfigure}\hspace*{\fill}
\begin{subfigure}[b]{0.4\linewidth}
\includegraphics[width=\linewidth]{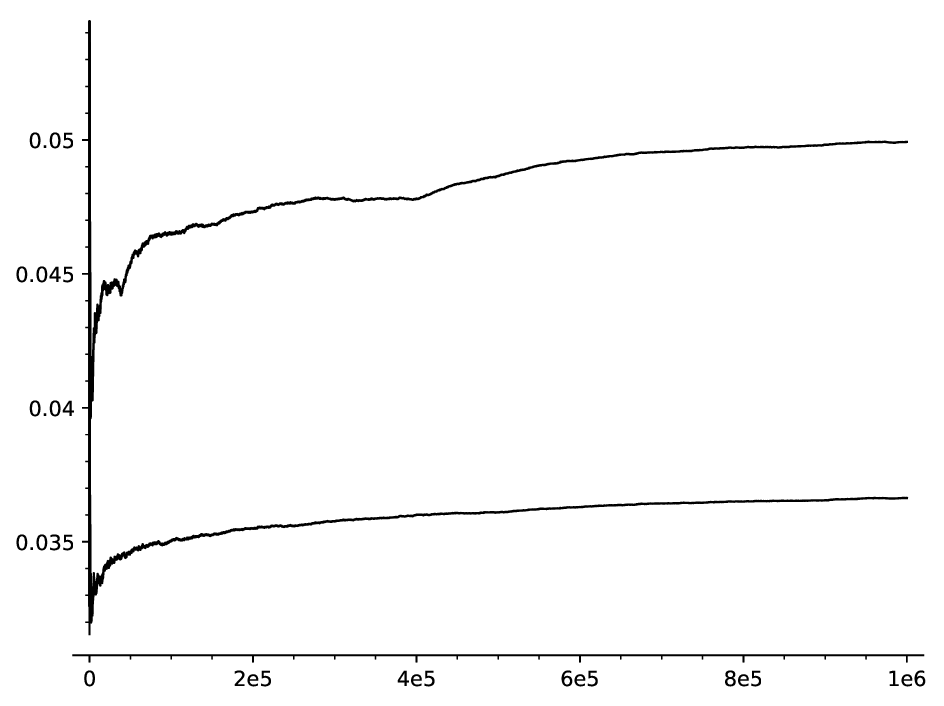}
\caption{17a1: $m_{6,E}^-(X;c)/X^{c +1/2}\log^2(X)$ Top to bottom $c =$ 0.3, 0.4} \label{fig:17_6_odd_acc_c}
\end{subfigure}
\hspace*{-.7cm}
\begin{subfigure}[b]{0.4\linewidth}
\includegraphics[width=\linewidth]{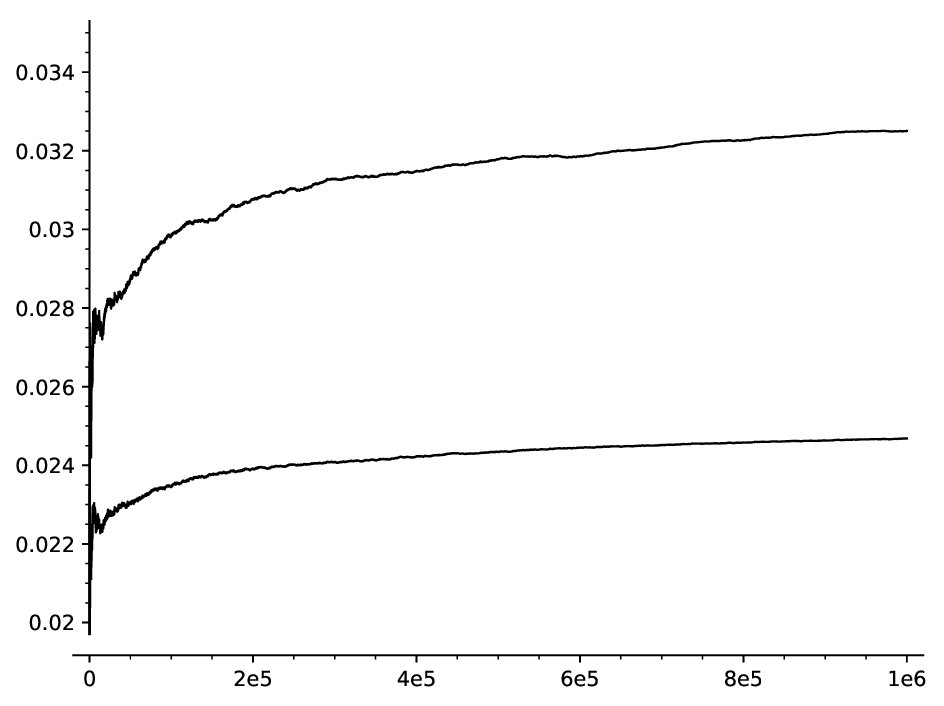}
\caption{19a1: $m_{6,E}^+(X;c)/X^{c +1/2}\log^2(X)$ Top to bottom $c =$ 0.3, 0.4} \label{fig:19_6_even_acc_c}
\end{subfigure}\hspace*{\fill}
\begin{subfigure}[b]{0.4\linewidth}
\includegraphics[width=\linewidth]{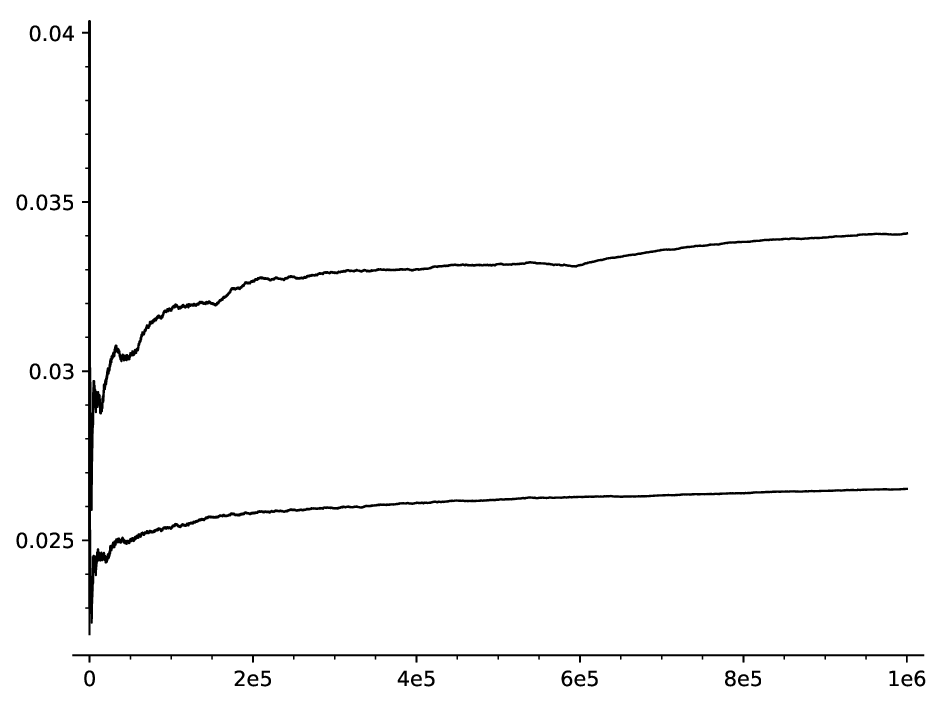}
\caption{19a1: $m_{6,E}^-(X;c)/X^{c +1/2}\log^2(X)$ Top to bottom $c =$ 0.3, 0.4} \label{fig:19_6_odd_acc_c}
\end{subfigure}
\hspace*{-.7cm}
\begin{subfigure}[b]{0.4\linewidth}
\includegraphics[width=\linewidth]{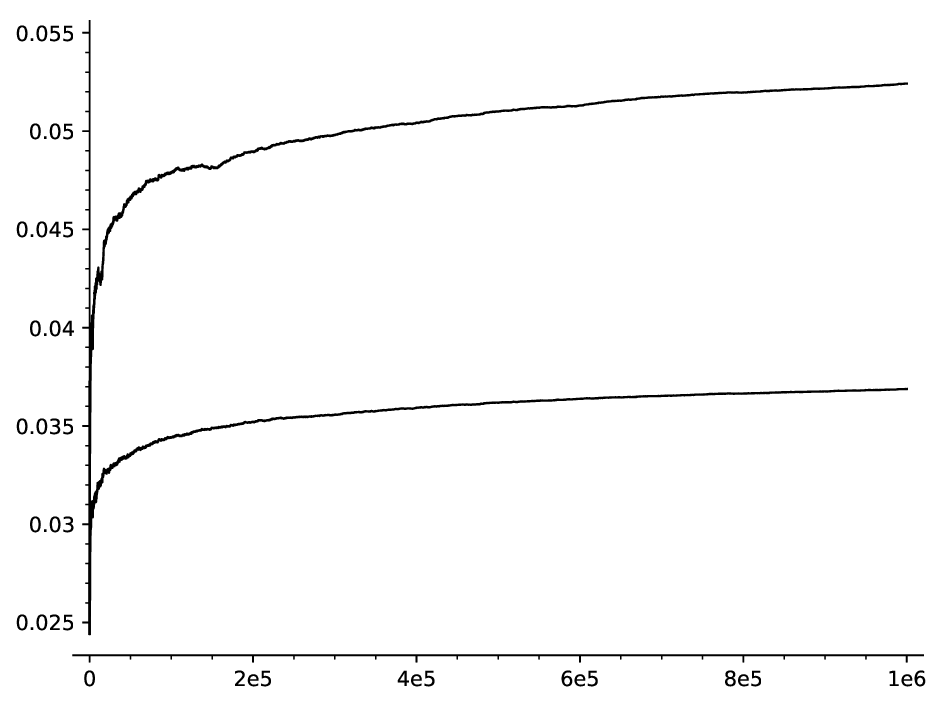}
\caption{37b1: $m_{6,E}^+(X;c)/X^{c +1/2}\log^2(X)$ Top to bottom $c =$ 0.3, 0.4} \label{fig:37_6_even_acc_c}
\end{subfigure}\hspace*{\fill}
\begin{subfigure}[b]{0.4\linewidth}
\includegraphics[width=\linewidth]{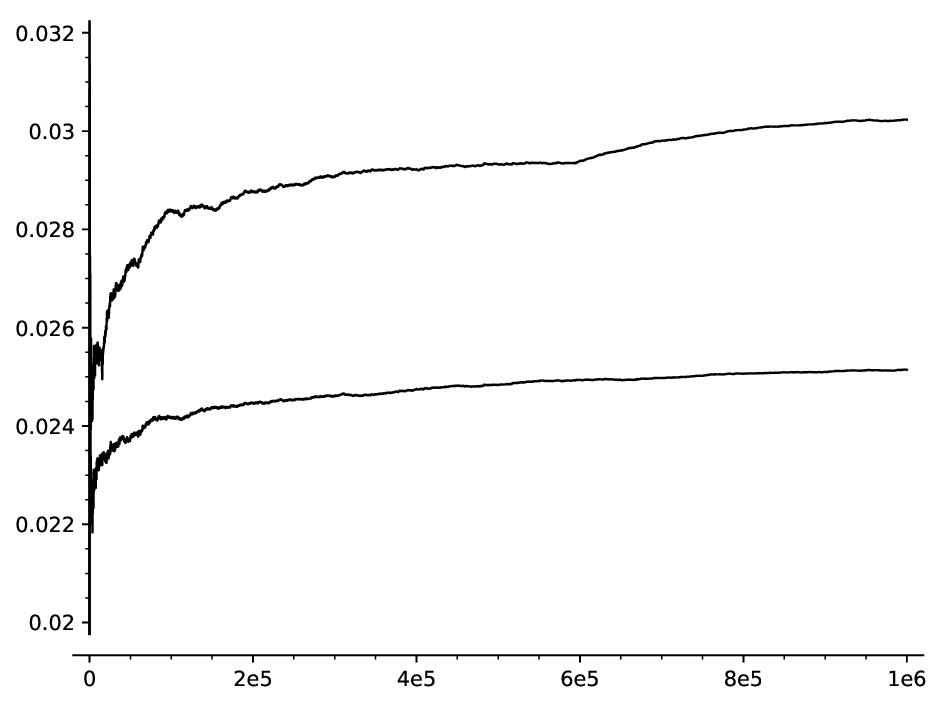}
\caption{37b1: $m_{6,E}^-(X;c)/X^{c +1/2}\log^2(X)$ Top to bottom $c =$ 0.3, 0.4} \label{fig:37_6_odd_acc_c}
\end{subfigure}
\caption{17a1, 19a1, 37b1: Ratio~\eqref{ratio_M_pm} $m_{6,E}^\pm(X;c)/X^{c +1/2}\log^2(X)$ depending on $\chi(-1) = \pm 1$ for $k = 6$ and $c = $ 0.3, 0.4} \label{fig:c_17_19_37_pm_acc_6}
\end{figure}

\clearpage

\begin{figure}[t] 
\hspace*{-2.3cm}
\begin{subfigure}[b]{0.4\linewidth}
\includegraphics[width=\linewidth]{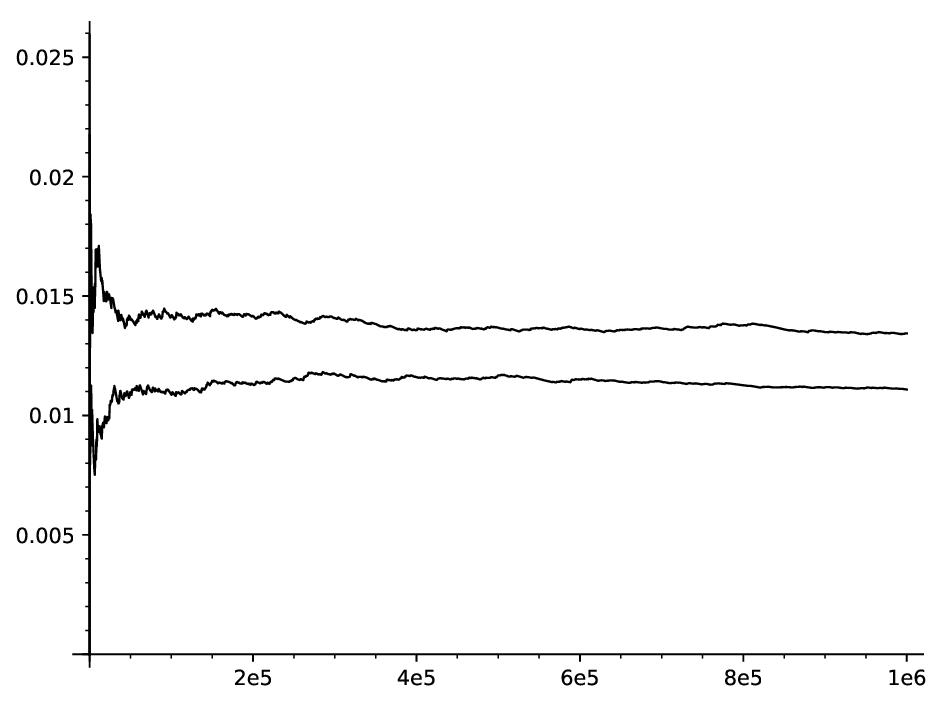}
\caption{$|l| = 1$: Top -1 bottom 1} \label{fig:11_6_even_A_1}
\end{subfigure}\hspace*{\fill}
\begin{subfigure}[b]{0.4\linewidth}
\includegraphics[width=\linewidth]{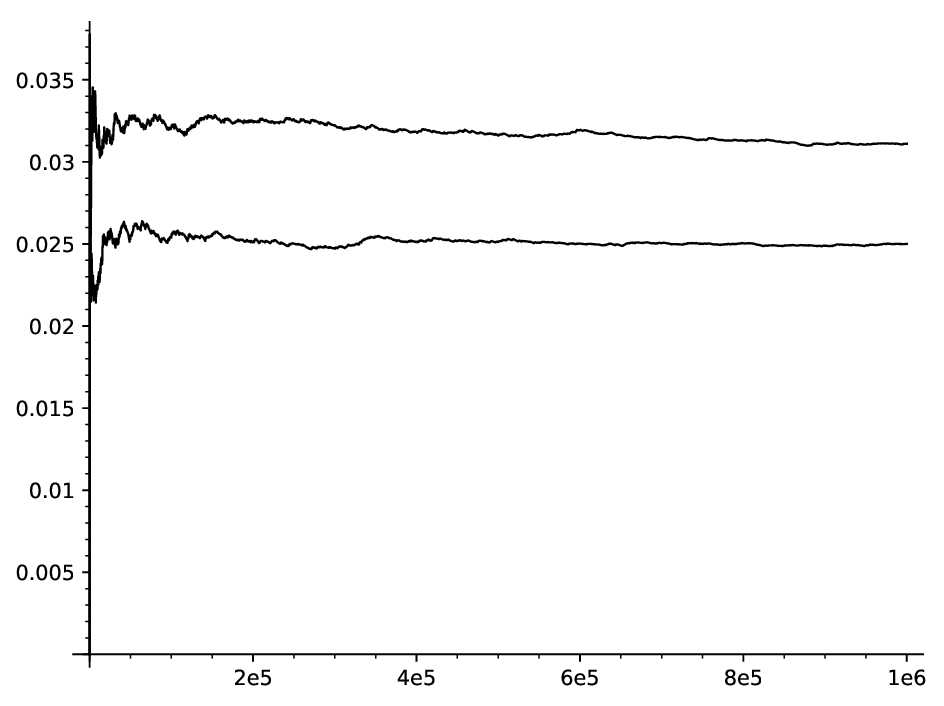}
\caption{$|l| = 2$: Top 2 bottom -2} \label{fig:11_6_even_A_2}
\end{subfigure}\hspace*{\fill}
\begin{subfigure}[b]{0.4\linewidth}
\includegraphics[width=\linewidth]{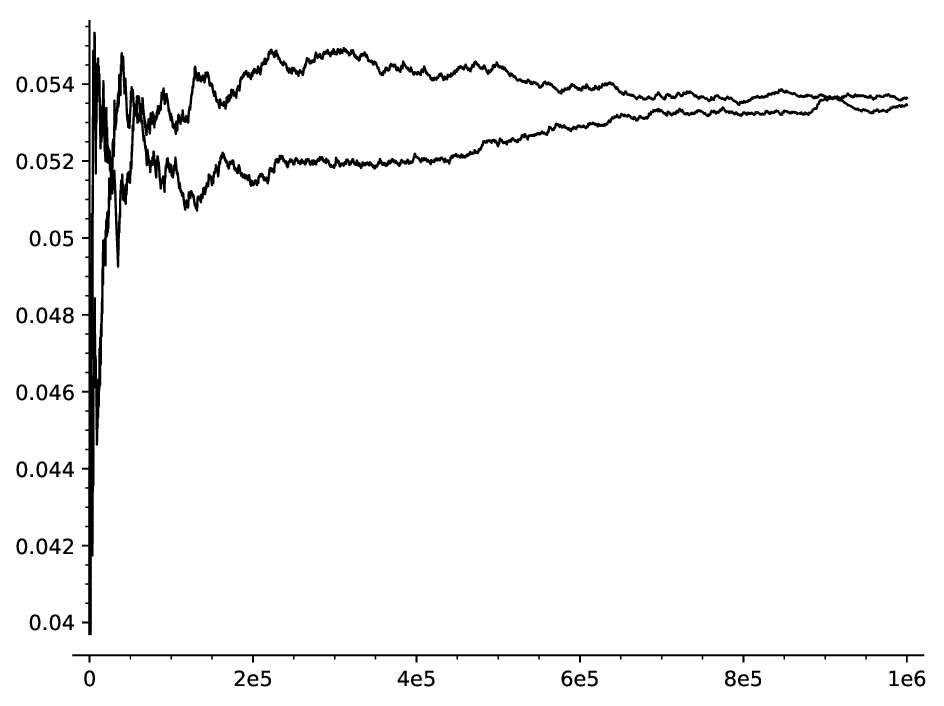}
\caption{$|l| = 3$: Top -3 bottom 3} \label{fig:11_6_even_A_3}
\end{subfigure}
\hspace*{-2.3cm}
\begin{subfigure}[b]{0.4\linewidth}
\includegraphics[width=\linewidth]{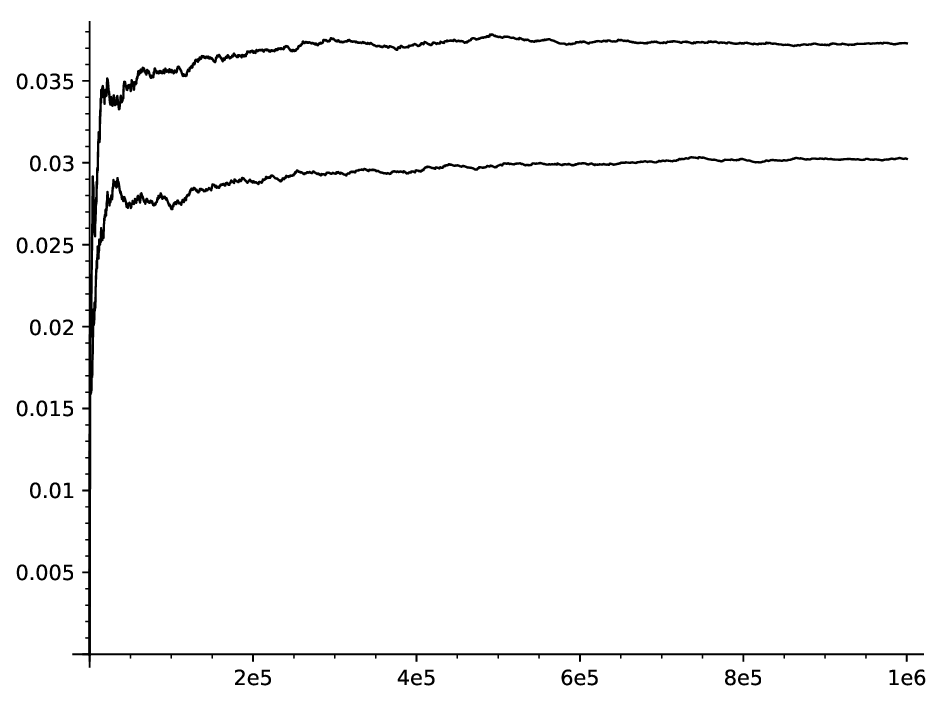}
\caption{$|l| = 4$: Top -4 bottom 4} \label{fig:11_6_even_A_4}
\end{subfigure}\hspace*{\fill}
\begin{subfigure}[b]{0.4\linewidth}
\includegraphics[width=\linewidth]{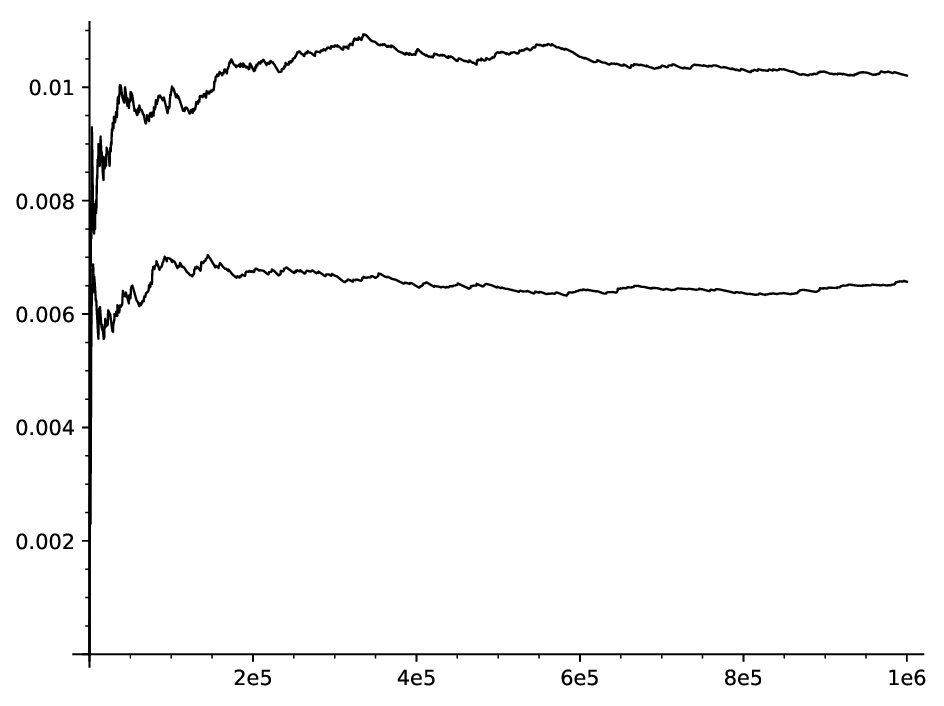}
\caption{$|l| = 5$: Top 5 bottom -5} \label{fig:11_6_even_A_5}
\end{subfigure}\hspace*{\fill}
\begin{subfigure}[b]{0.4\linewidth}
\includegraphics[width=\linewidth]{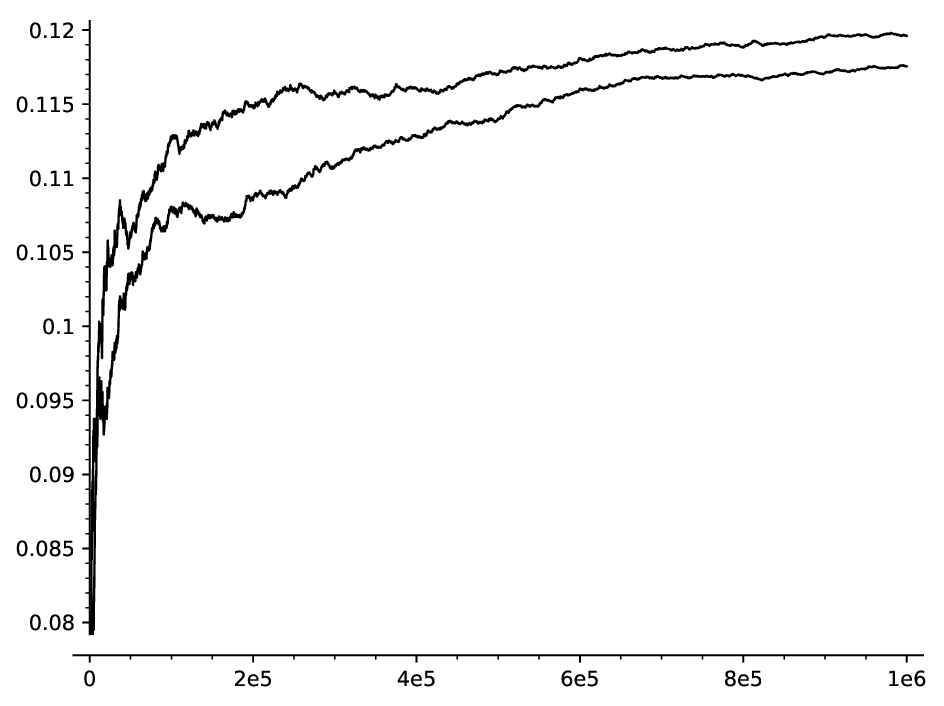}
\caption{$|l| = 6$: Top 6 bottom -6} \label{fig:11_6_even_A_6}
\end{subfigure}
\hspace*{-2.3cm}
\begin{subfigure}[b]{0.4\linewidth}
\includegraphics[width=\linewidth]{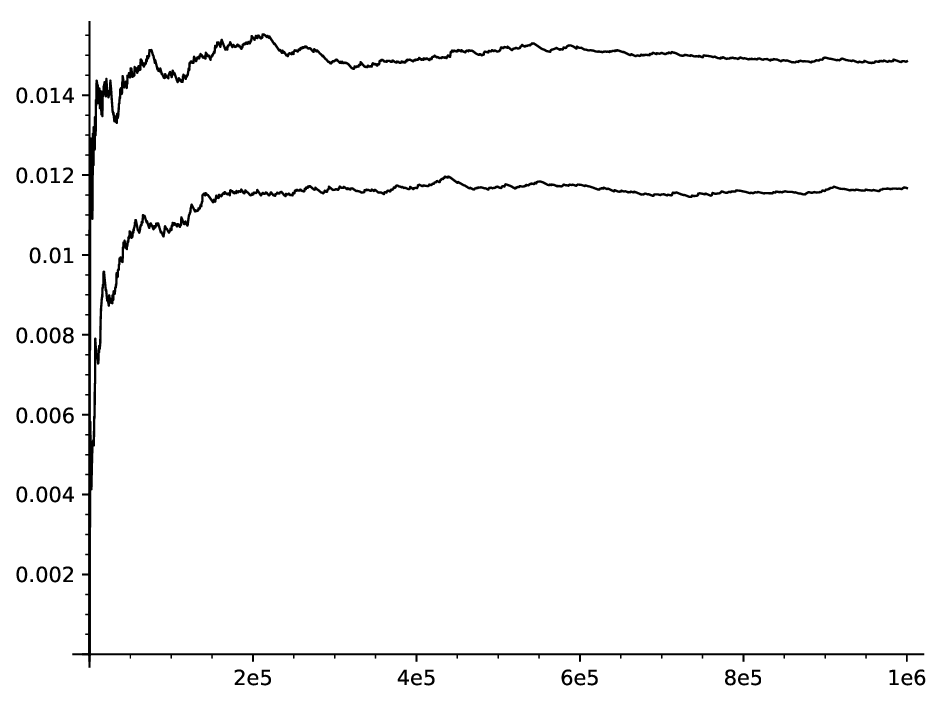}
\caption{$|l| = 7$: Top -7 bottom 7} \label{fig:11_6_even_A_7}
\end{subfigure}\hspace*{\fill}
\begin{subfigure}[b]{0.4\linewidth}
\includegraphics[width=\linewidth]{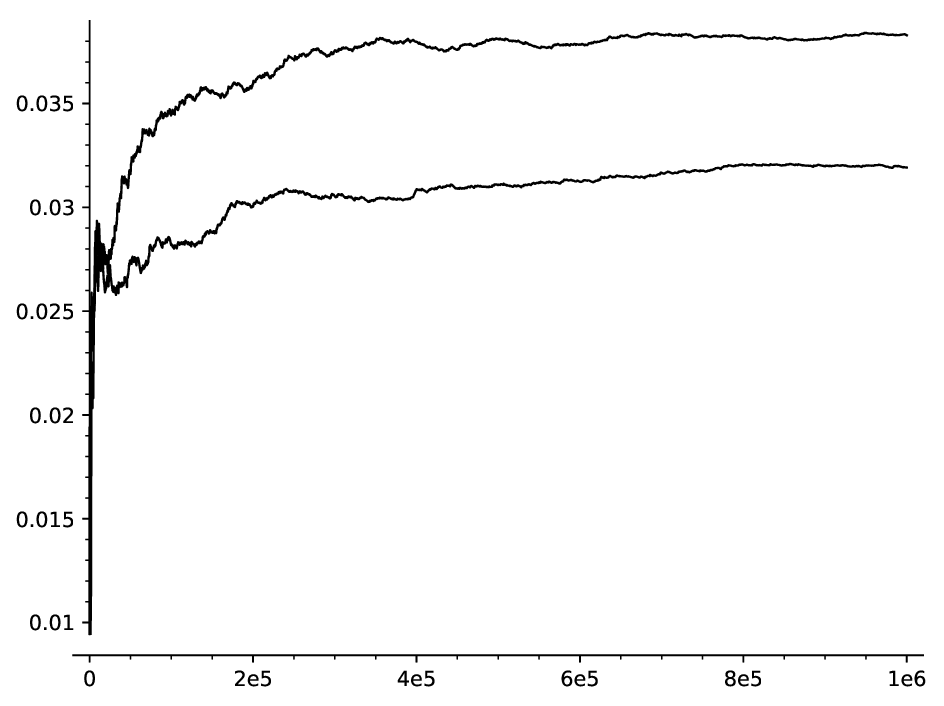}
\caption{$|l| = 8$: Top 8 bottom -8} \label{fig:11_6_even_A_8}
\end{subfigure}\hspace*{\fill}
\begin{subfigure}[b]{0.4\linewidth}
\includegraphics[width=\linewidth]{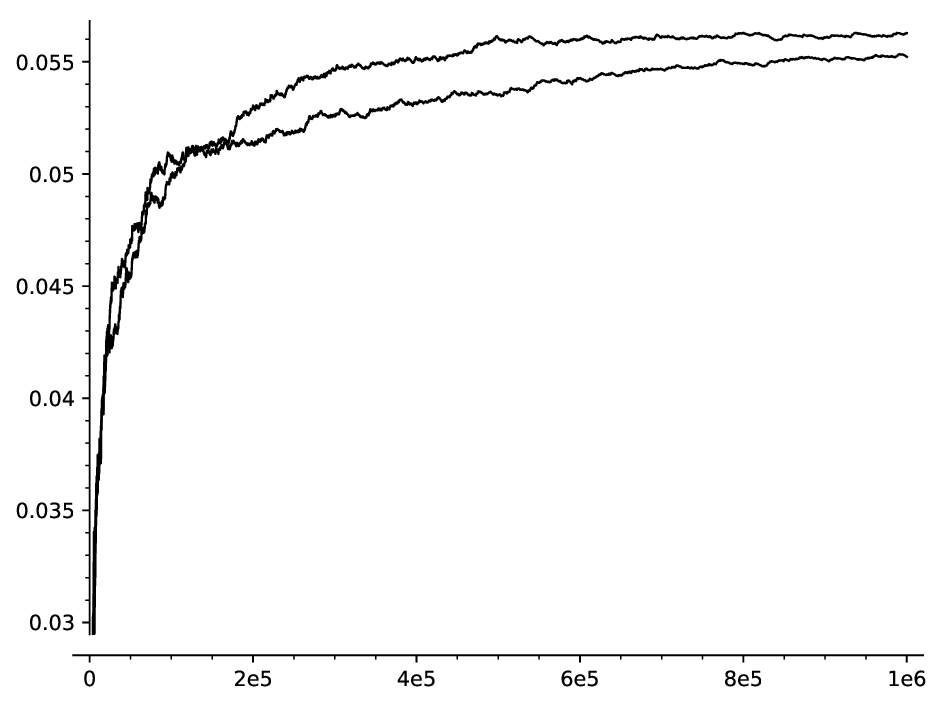}
\caption{$|l| = 9$: Top 9 bottom -9} \label{fig:11_6_even_A_9}
\end{subfigure}
\caption{11a1: Ratio~\eqref{ratio_n_pm} $x_{6,E}^+(X;l)/X^{1/2}\log^2(X)$} \label{fig:11a1_6_even_A_exact}
\end{figure}

\begin{figure}[b] 
\hspace*{-2.3cm}
\begin{subfigure}[b]{0.4\linewidth}
\includegraphics[width=\linewidth]{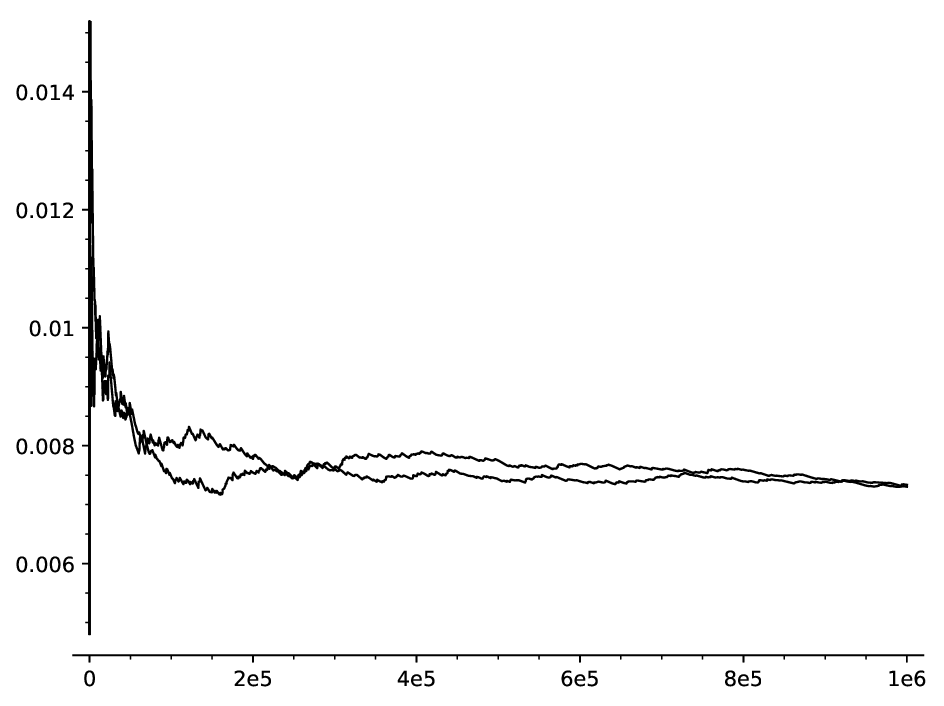}
\caption{$|l| = 1$: Top -1 bottom 1} \label{fig:11_6_odd_A_1}
\end{subfigure}\hspace*{\fill}
\begin{subfigure}[b]{0.4\linewidth}
\includegraphics[width=\linewidth]{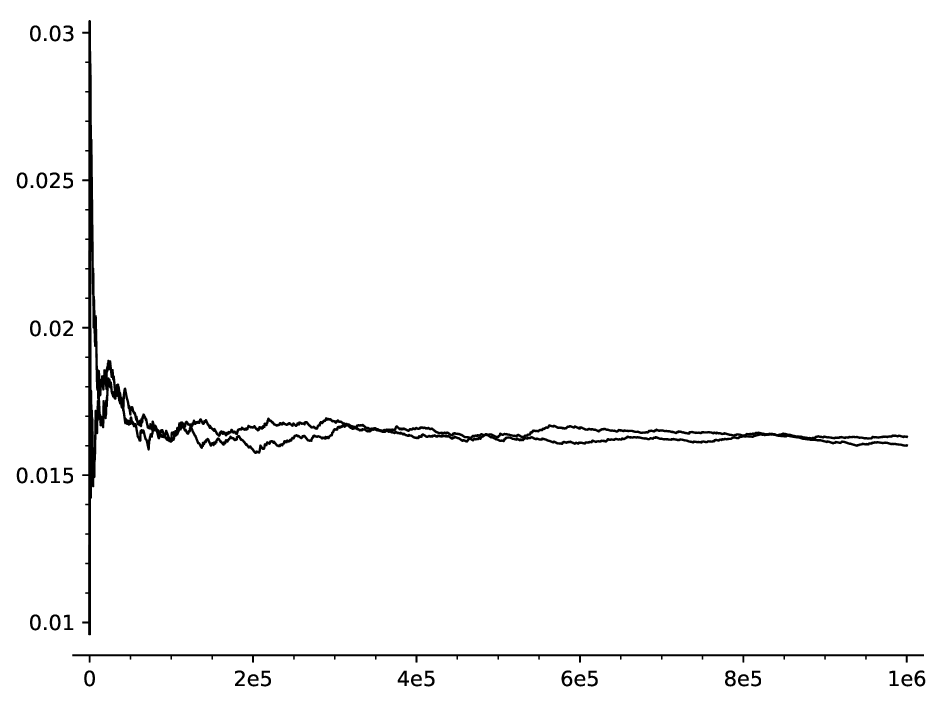}
\caption{$|l| = 2$: Top 2 bottom -2} \label{fig:11_6_odd_A_2}
\end{subfigure}\hspace*{\fill}
\begin{subfigure}[b]{0.4\linewidth}
\includegraphics[width=\linewidth]{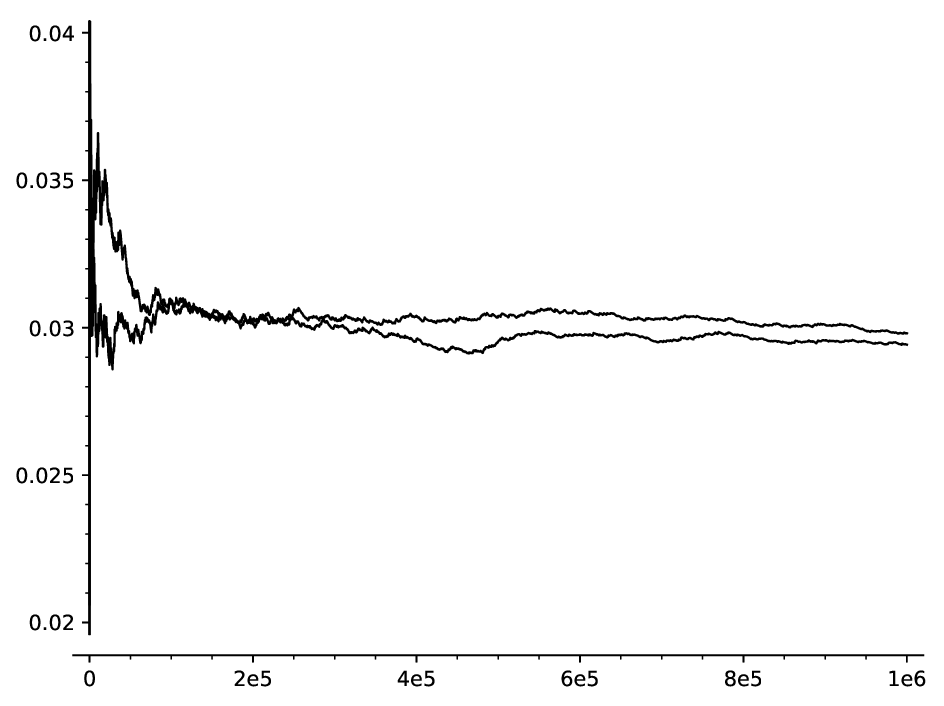}
\caption{$|l| = 3$: Top -3 bottom 3} \label{fig:11_6_odd_A_3}
\end{subfigure}
\hspace*{-2.3cm}
\begin{subfigure}[b]{0.4\linewidth}
\includegraphics[width=\linewidth]{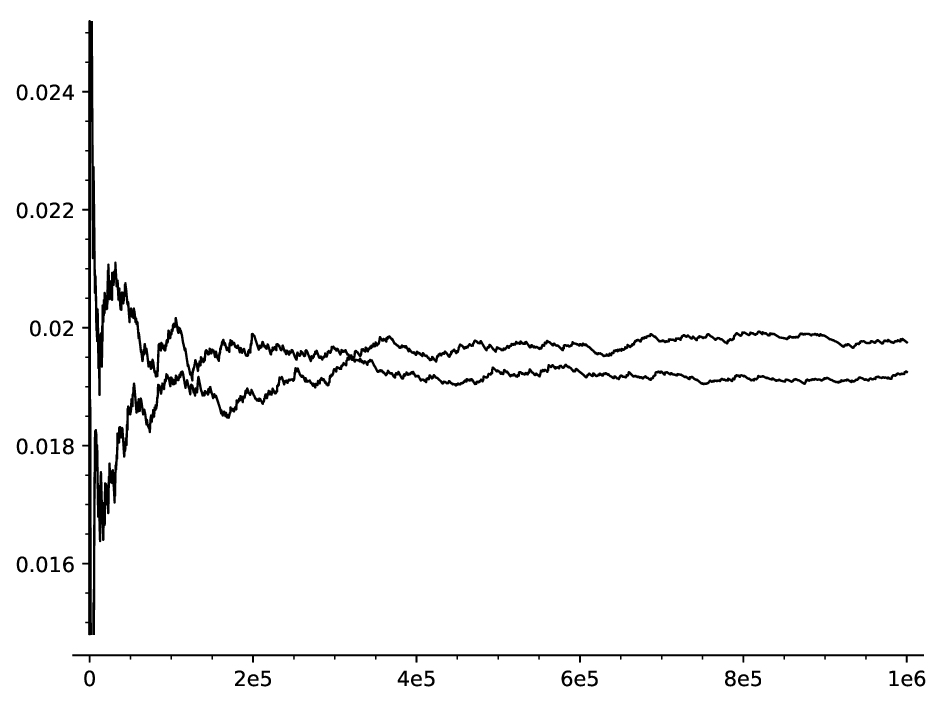}
\caption{$|l| = 4$: Top -4 bottom 4} \label{fig:11_6_odd_A_4}
\end{subfigure}\hspace*{\fill}
\begin{subfigure}[b]{0.4\linewidth}
\caption{$|l| = 5$: Not occur} \label{fig:11_6_odd_A_5}
\end{subfigure}\hspace*{\fill}
\begin{subfigure}[b]{0.4\linewidth}
\includegraphics[width=\linewidth]{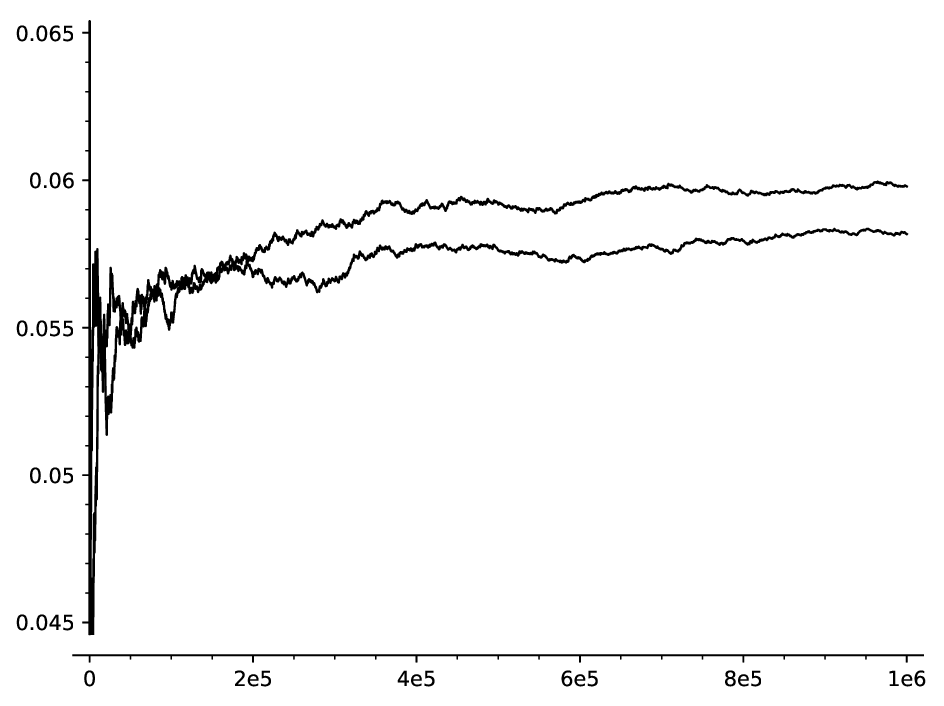}
\caption{$|l| = 6$: Top 6 bottom -6} \label{fig:11_6_odd_A_6}
\end{subfigure}
\hspace*{-2.3cm}
\begin{subfigure}[b]{0.4\linewidth}
\includegraphics[width=\linewidth]{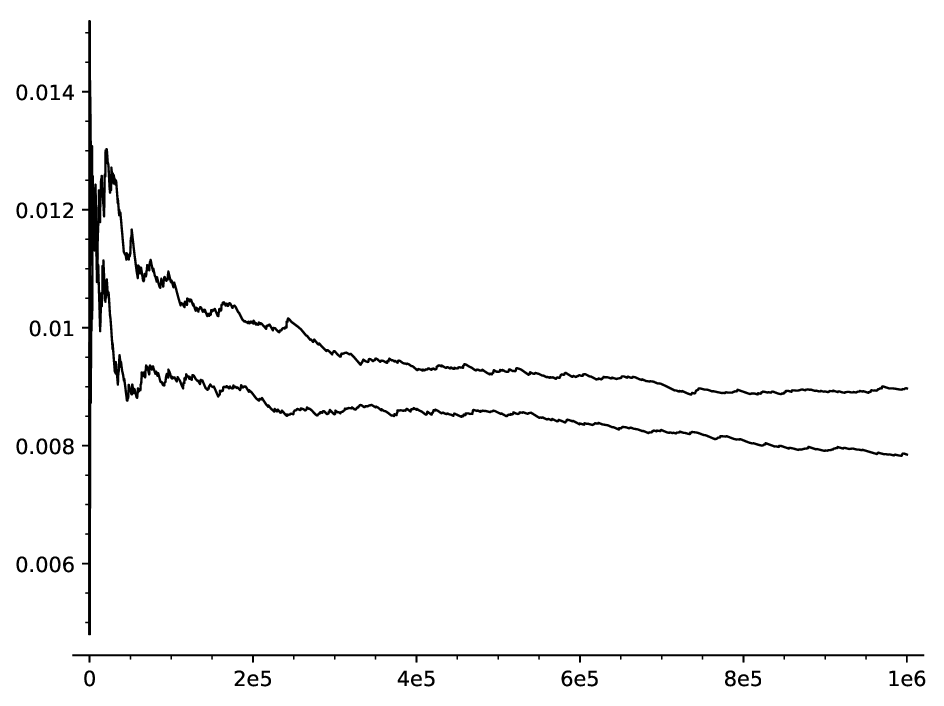}
\caption{$|l| = 7$: Top -7 bottom 7} \label{fig:11_6_odd_A_7}
\end{subfigure}\hspace*{\fill}
\begin{subfigure}[b]{0.4\linewidth}
\includegraphics[width=\linewidth]{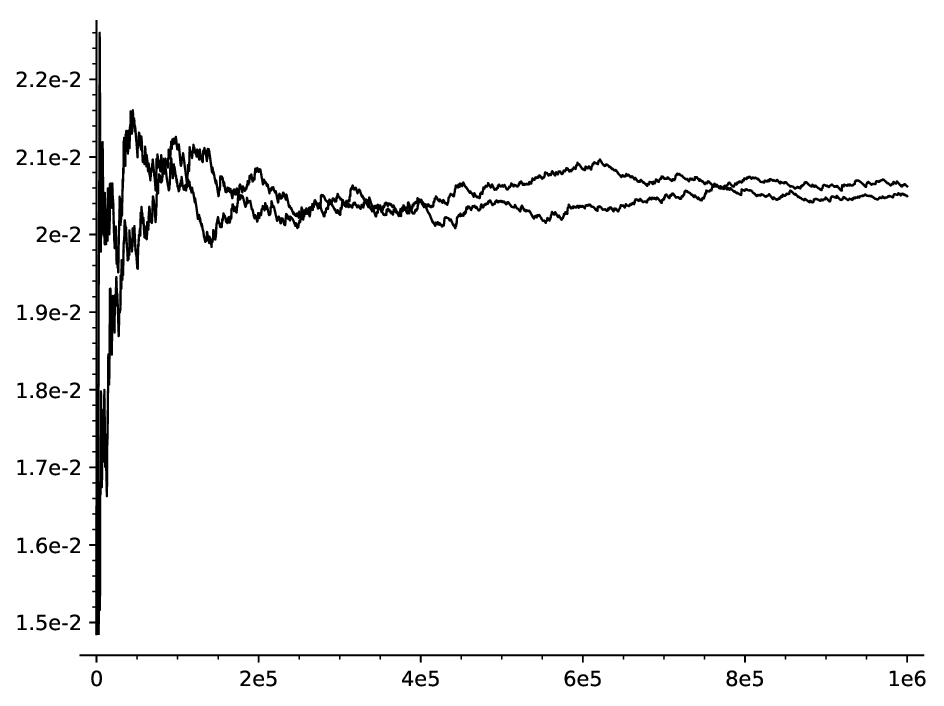}
\caption{$|l| = 8$: Top 8 bottom -8} \label{fig:11_6_od_A_8}
\end{subfigure}\hspace*{\fill}
\begin{subfigure}[b]{0.4\linewidth}
\includegraphics[width=\linewidth]{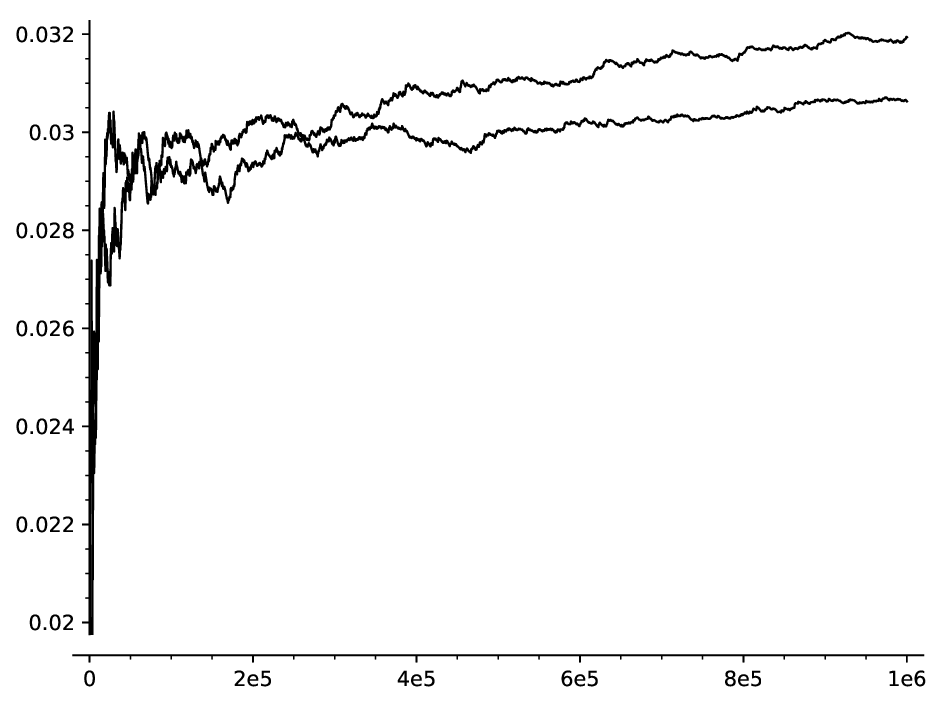}
\caption{$|l| = 9$: Top 9 bottom -9} \label{fig:11_6_odd_A_9}
\end{subfigure}
\caption{11a1: Ratio~\eqref{ratio_n_pm} $x_{6,E}^-(X;l)/X^{1/2}\log^2(X)$} \label{fig:11a1_6_odd_A_exact}
\end{figure}

\clearpage

\begin{figure}[t] 
\hspace*{-2.3cm}
\begin{subfigure}[b]{0.4\linewidth}
\includegraphics[width=\linewidth]{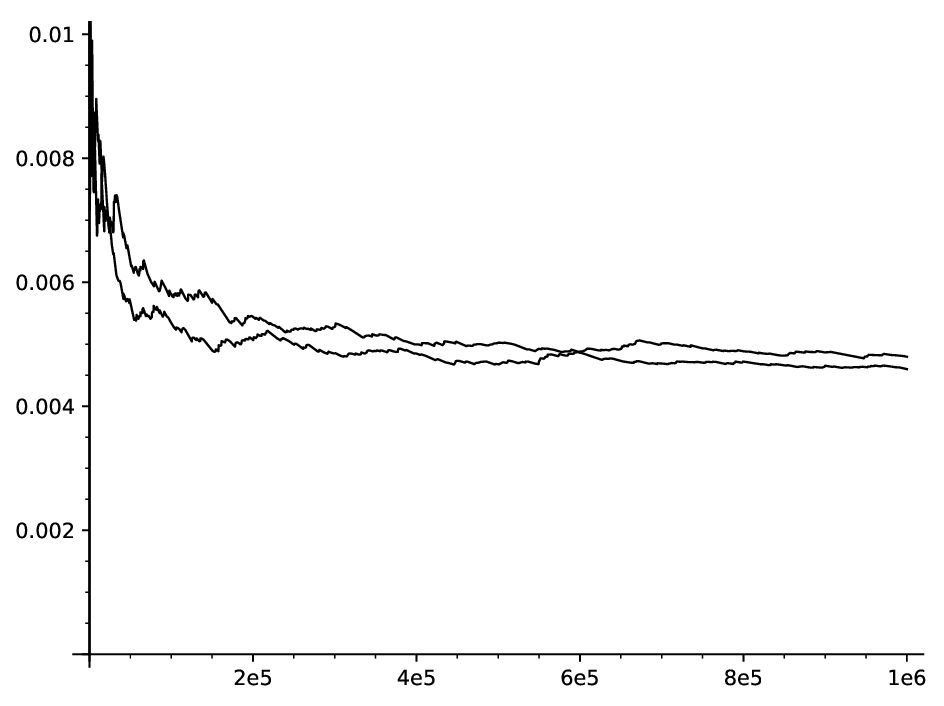}
\caption{$|l| = 1$: Top 1 bottom -1} \label{fig:14_6_even_A_1}
\end{subfigure}\hspace*{\fill}
\begin{subfigure}[b]{0.4\linewidth}
\includegraphics[width=\linewidth]{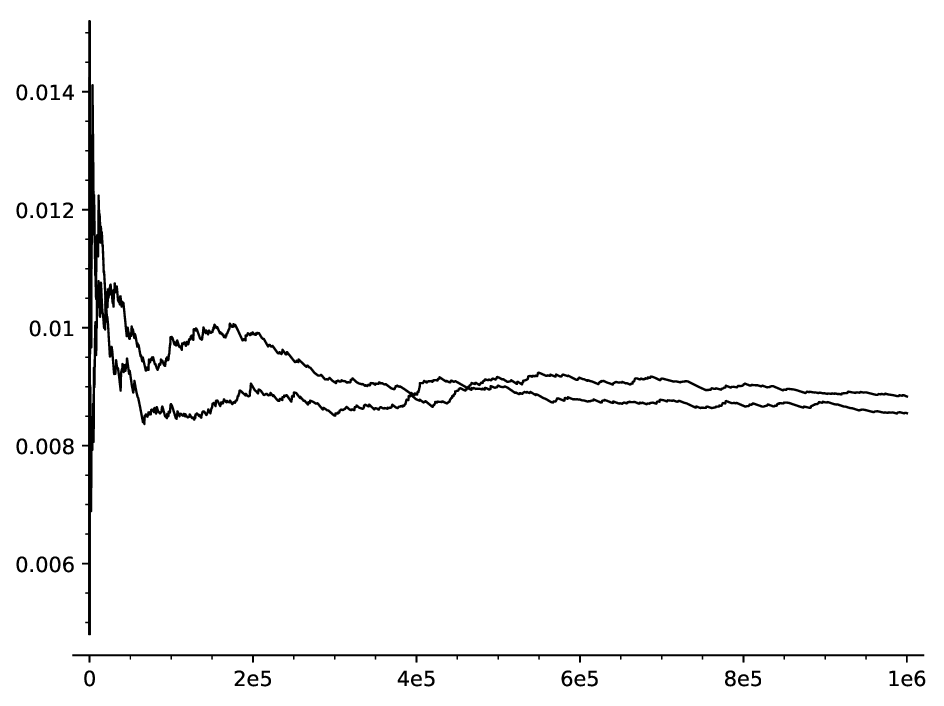}
\caption{$|l| = 2$: Top 2 bottom -2} \label{fig:14_6_even_A_2}
\end{subfigure}\hspace*{\fill}
\begin{subfigure}[b]{0.4\linewidth}
\includegraphics[width=\linewidth]{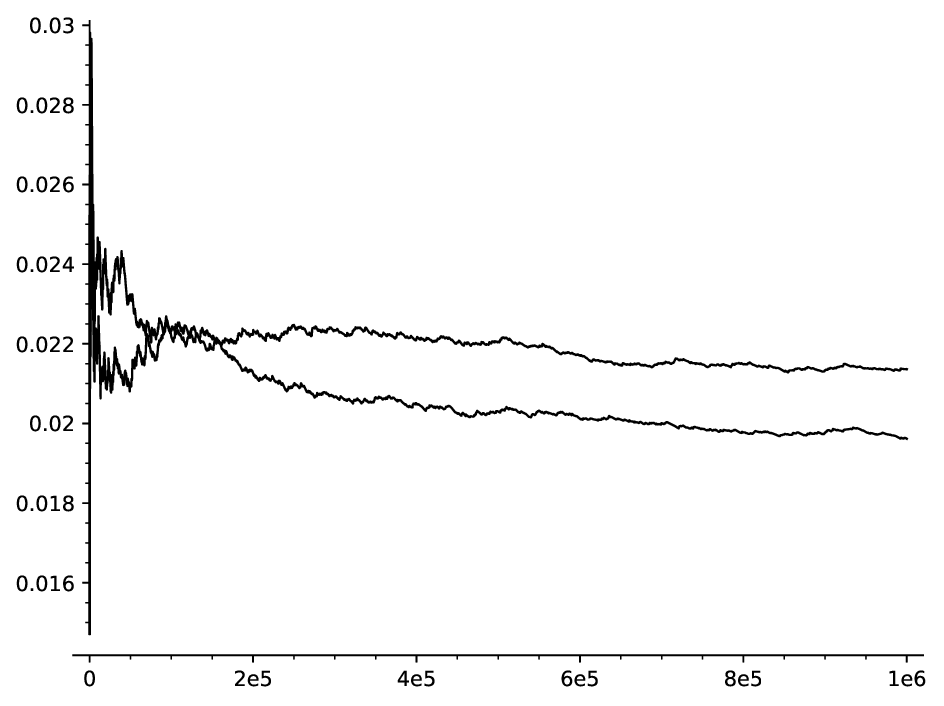}
\caption{$|l| = 3$: Top -3 bottom 3} \label{fig:14_6_even_A_3}
\end{subfigure}
\hspace*{-2.3cm}
\begin{subfigure}[b]{0.4\linewidth}
\includegraphics[width=\linewidth]{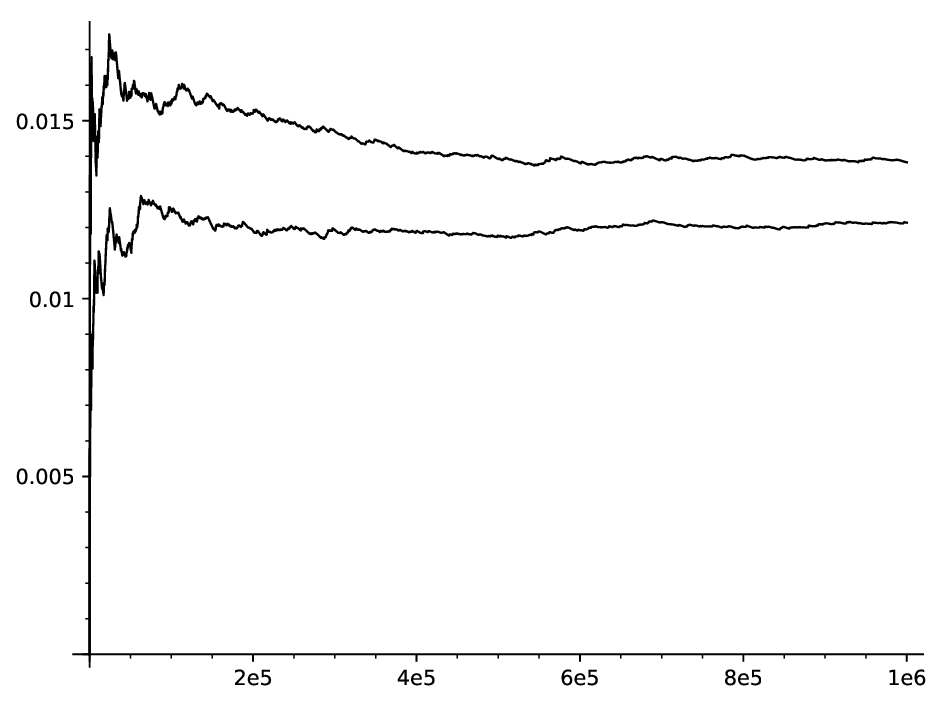}
\caption{$|l| = 4$: Top -4 bottom 4} \label{fig:14_6_even_A_4}
\end{subfigure}\hspace*{\fill}
\begin{subfigure}[b]{0.4\linewidth}
\includegraphics[width=\linewidth]{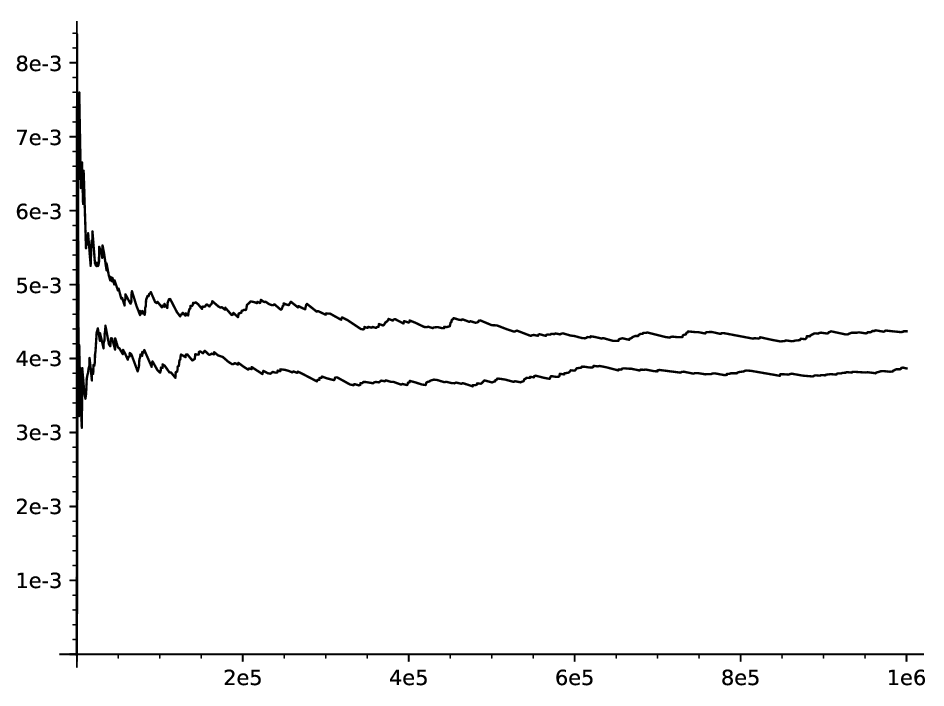}
\caption{$|l| = 5$: Top -5 bottom 5} \label{fig:14_6_even_A_5}
\end{subfigure}\hspace*{\fill}
\begin{subfigure}[b]{0.4\linewidth}
\includegraphics[width=\linewidth]{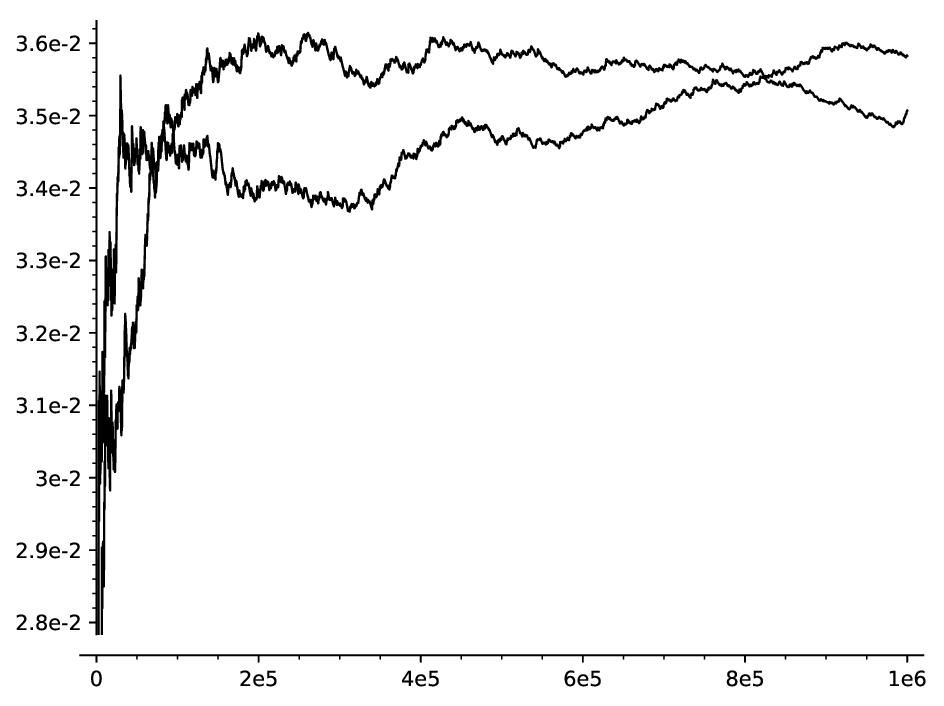}
\caption{$|l| = 6$: Top -6 bottom 6} \label{fig:14_6_even_A_6}
\end{subfigure}
\hspace*{-2.3cm}
\begin{subfigure}[b]{0.4\linewidth}
\includegraphics[width=\linewidth]{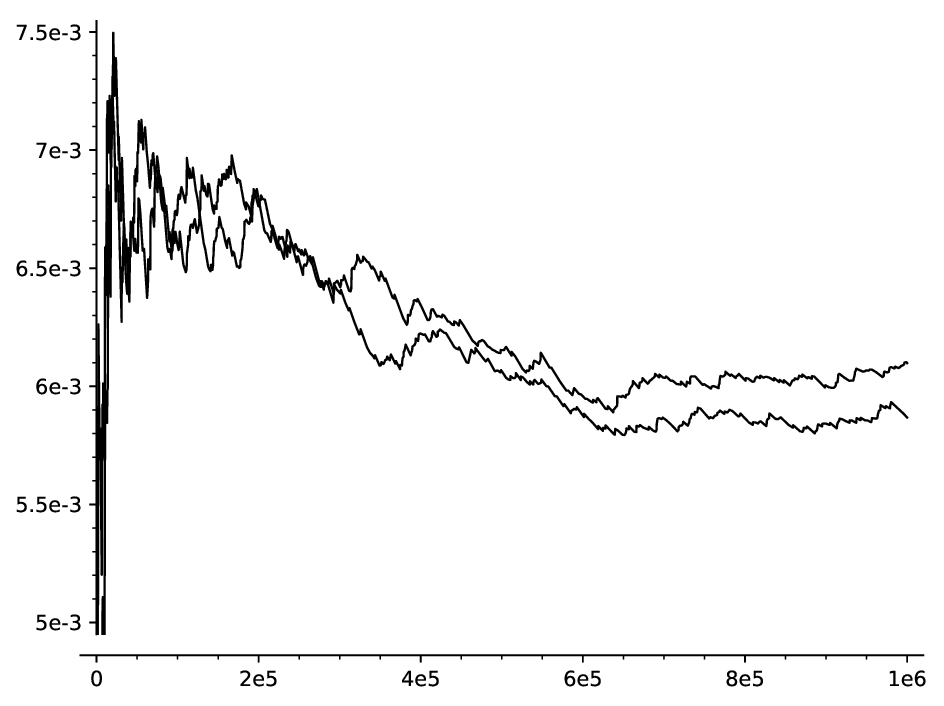}
\caption{$|l| = 7$: Top -7 bottom 7} \label{fig:14_6_even_A_7}
\end{subfigure}\hspace*{\fill}
\begin{subfigure}[b]{0.4\linewidth}
\includegraphics[width=\linewidth]{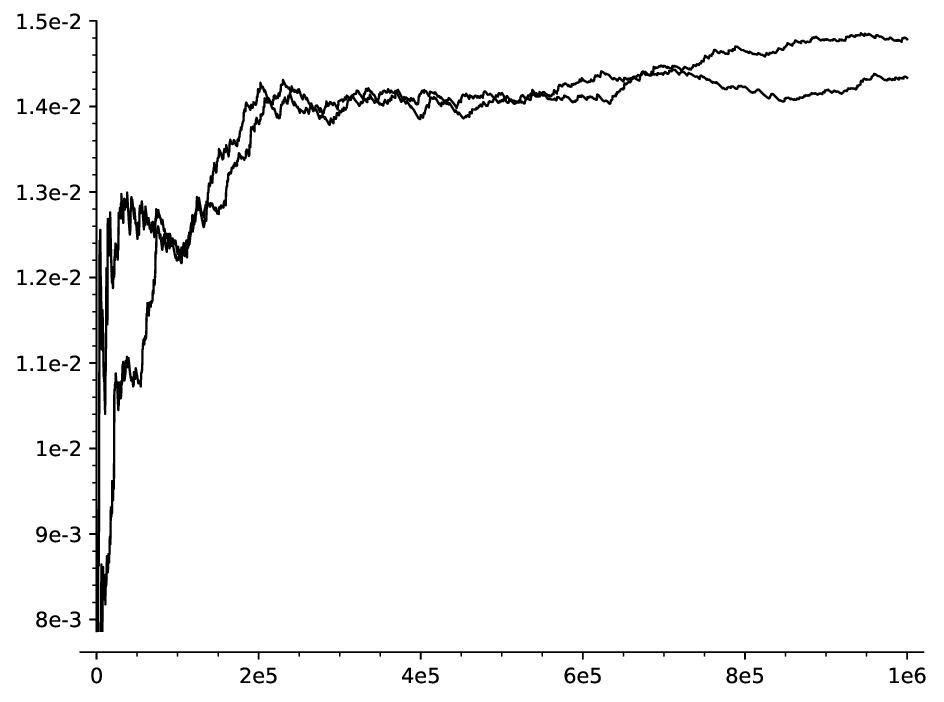}
\caption{$|l| = 8$: Top -8 bottom 8} \label{fig:14_6_even_A_8}
\end{subfigure}\hspace*{\fill}
\begin{subfigure}[b]{0.4\linewidth}
\includegraphics[width=\linewidth]{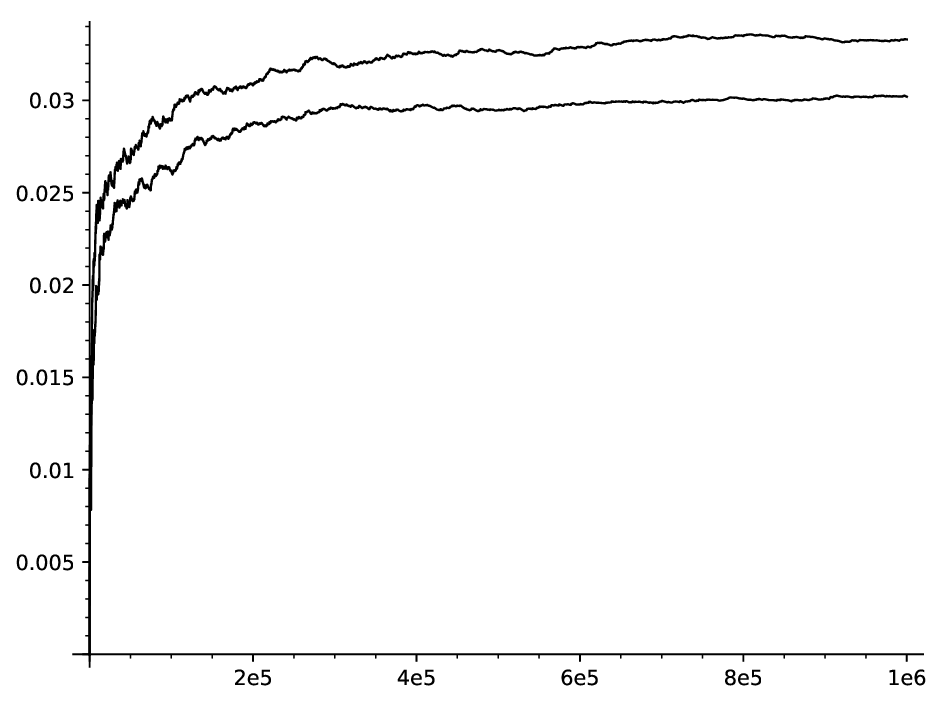}
\caption{$|l| = 9$: Top -9 bottom 9} \label{fig:14_6_even_A_9}
\end{subfigure}
\caption{14a1: Ratio~\eqref{ratio_n_pm} $x_{6,E}^+(X;l)/X^{1/2}\log^2(X)$} \label{fig:14a1_6_even_A_exact}
\end{figure}

\begin{figure}[b] 
\hspace*{-2.3cm}
\begin{subfigure}[b]{0.4\linewidth}
\includegraphics[width=\linewidth]{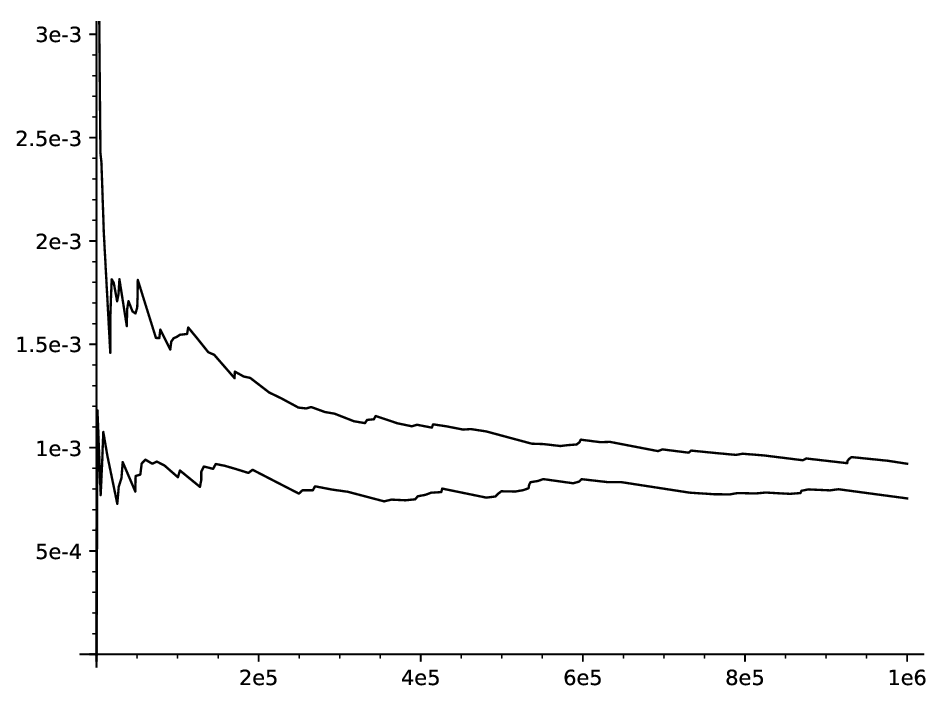}
\caption{$|l| = 1$: Top -1 bottom 1} \label{fig:14_6_odd_A_1}
\end{subfigure}\hspace*{\fill}
\begin{subfigure}[b]{0.4\linewidth}
\includegraphics[width=\linewidth]{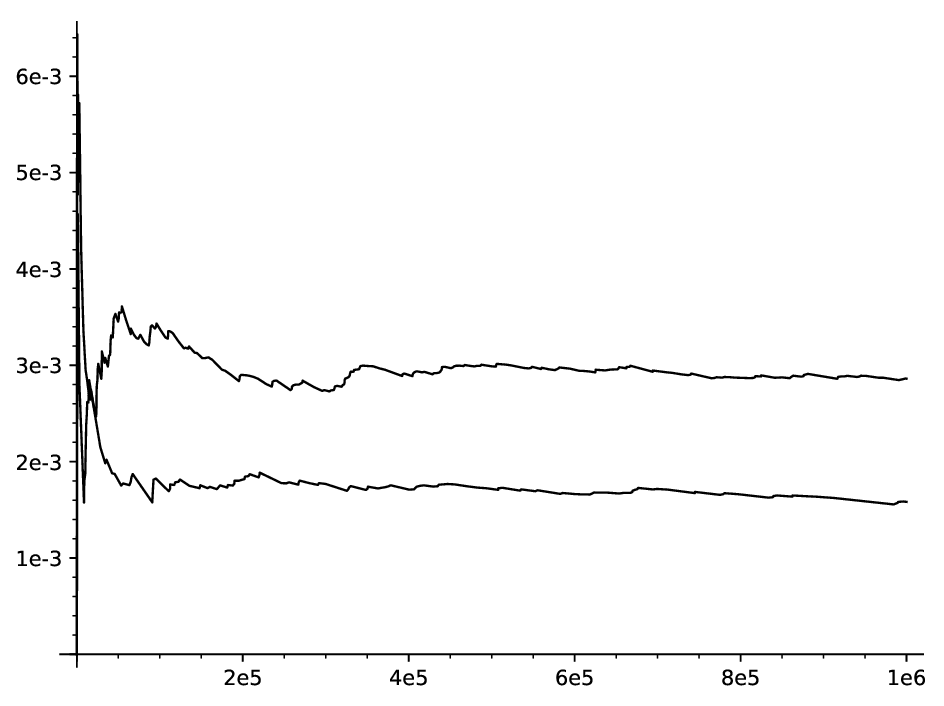}
\caption{$|l| = 2$: Top -2 bottom 2} \label{fig:14_6_odd_A_2}
\end{subfigure}\hspace*{\fill}
\begin{subfigure}[b]{0.4\linewidth}
\includegraphics[width=\linewidth]{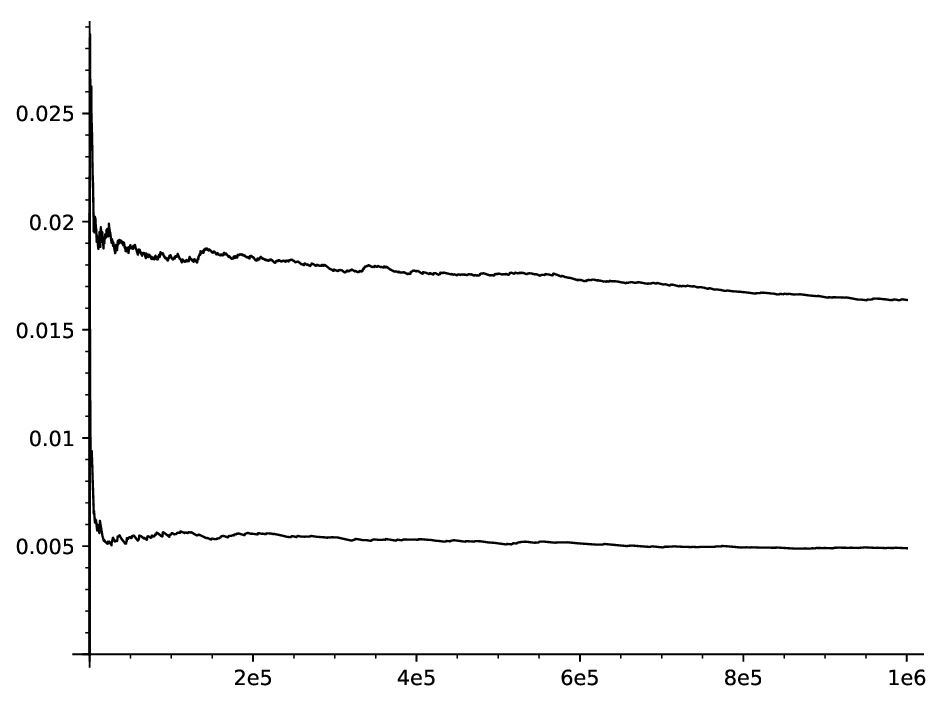}
\caption{$|l| = 3$: Top -3 bottom 3} \label{fig:14_6_odd_A_3}
\end{subfigure}
\hspace*{-2.3cm}
\begin{subfigure}[b]{0.4\linewidth}
\includegraphics[width=\linewidth]{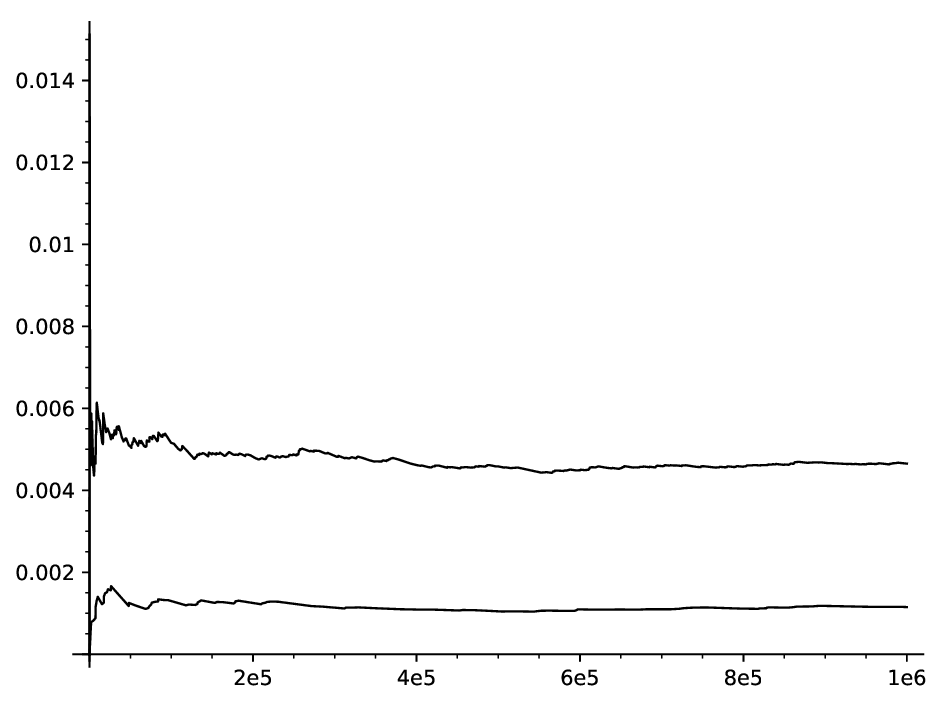}
\caption{$|l| = 4$: Top 4 bottom -4} \label{fig:14_6_odd_A_4}
\end{subfigure}\hspace*{\fill}
\begin{subfigure}[b]{0.4\linewidth}
\includegraphics[width=\linewidth]{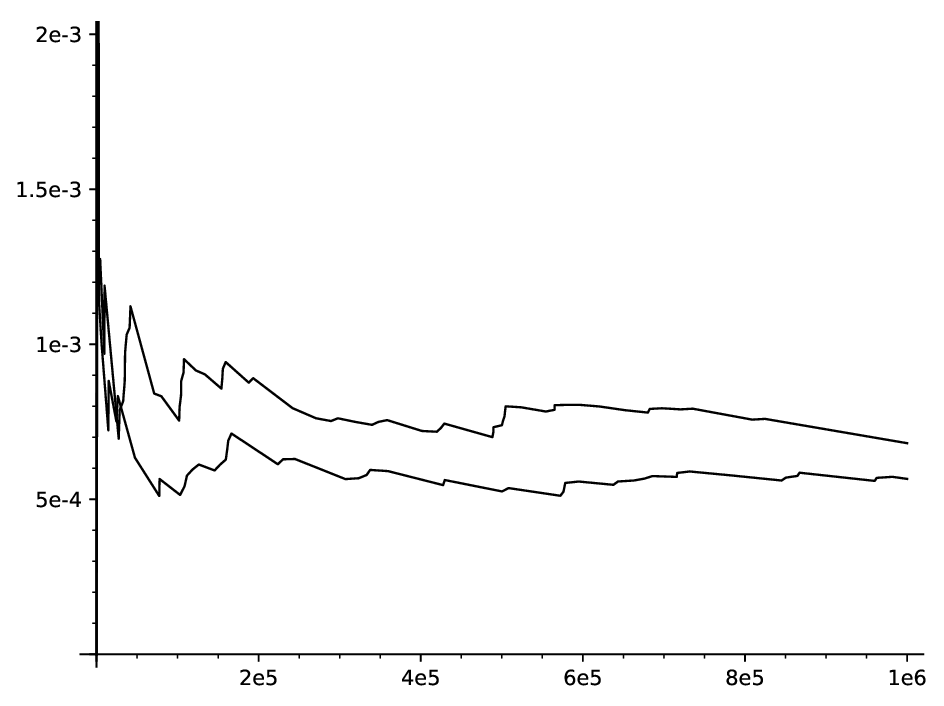}
\caption{$|l| = 5$: Top 5 bottom -5} \label{fig:14_6_odd_A_5}
\end{subfigure}\hspace*{\fill}
\begin{subfigure}[b]{0.4\linewidth}
\includegraphics[width=\linewidth]{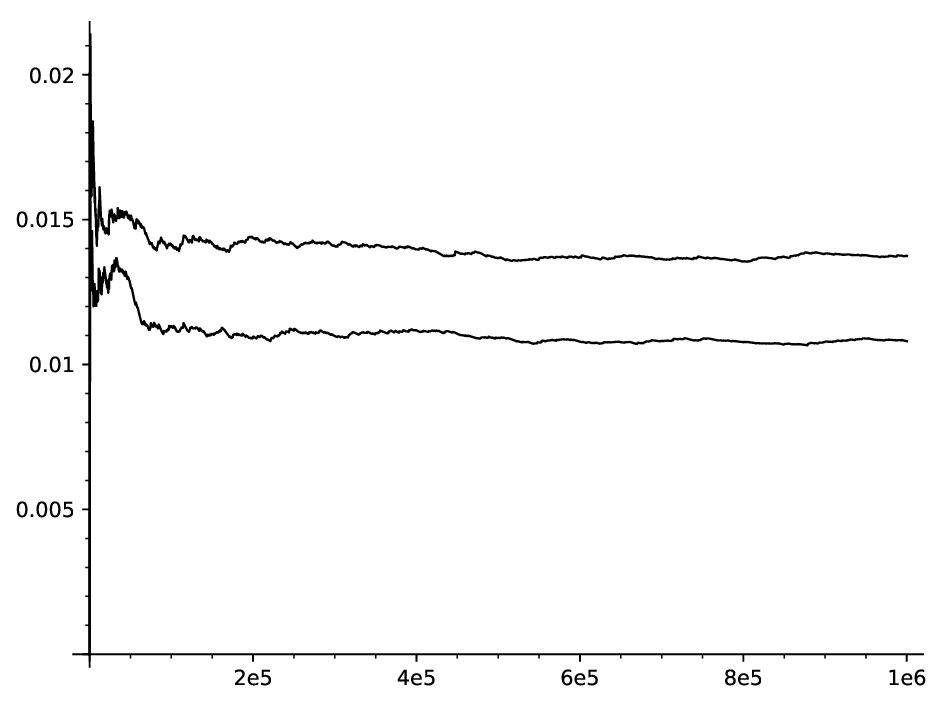}
\caption{$|l| = 6$: Top -6 bottom 6} \label{fig:14_6_odd_A_6}
\end{subfigure}
\hspace*{-2.3cm}
\begin{subfigure}[b]{0.4\linewidth}
\includegraphics[width=\linewidth]{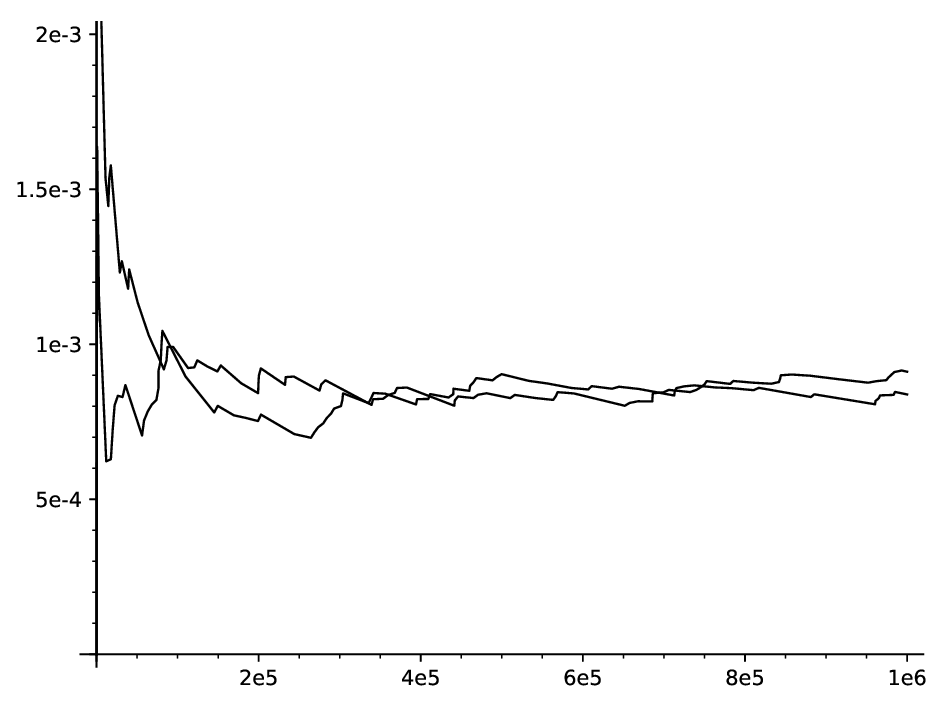}
\caption{$|l| = 7$: Top 7 bottom -7} \label{fig:14_6_odd_A_7}
\end{subfigure}\hspace*{\fill}
\begin{subfigure}[b]{0.4\linewidth}
\includegraphics[width=\linewidth]{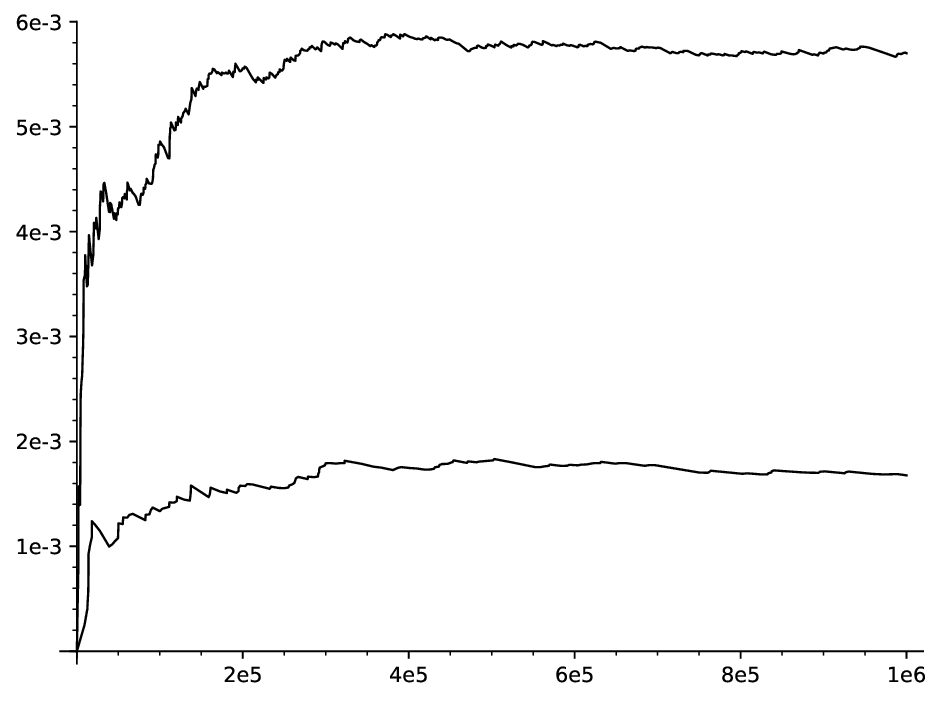}
\caption{$|l| = 8$: Top -8 bottom 8} \label{fig:14_6_od_A_8}
\end{subfigure}\hspace*{\fill}
\begin{subfigure}[b]{0.4\linewidth}
\includegraphics[width=\linewidth]{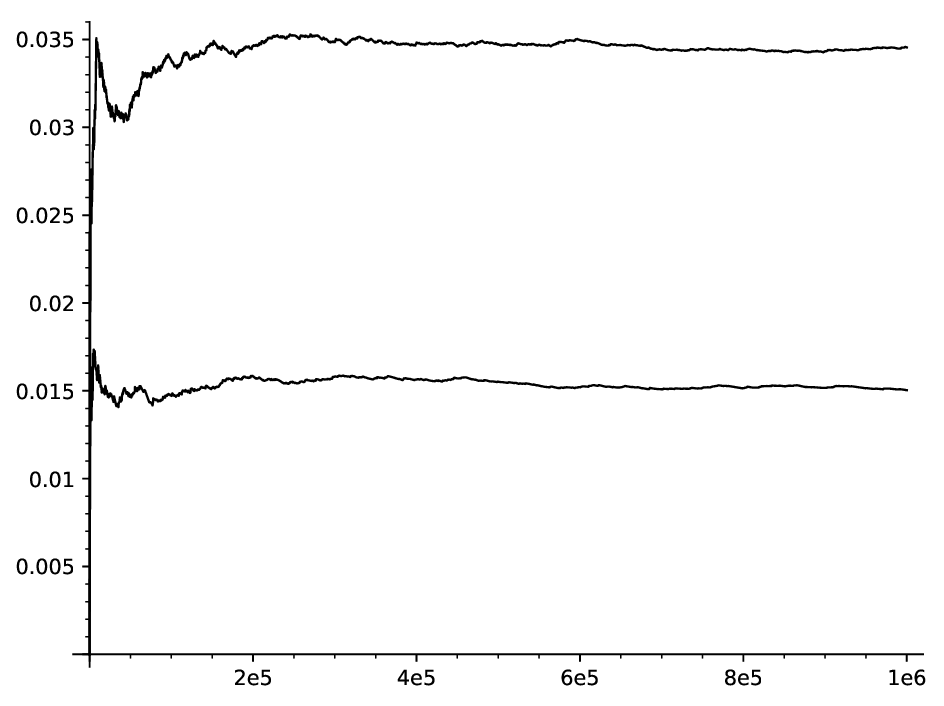}
\caption{$|l| = 9$: Top -9 bottom 9} \label{fig:14_6_odd_A_9}
\end{subfigure}
\caption{14a1: Ratio~\eqref{ratio_n_pm} $x_{6,E}^-(X;l)/X^{1/2}\log^2(X)$} \label{fig:14a1_6_odd_A_exact}
\end{figure}

\clearpage

\begin{figure}[t] 
\hspace*{-2.3cm}
\begin{subfigure}[b]{0.4\linewidth}
\includegraphics[width=\linewidth]{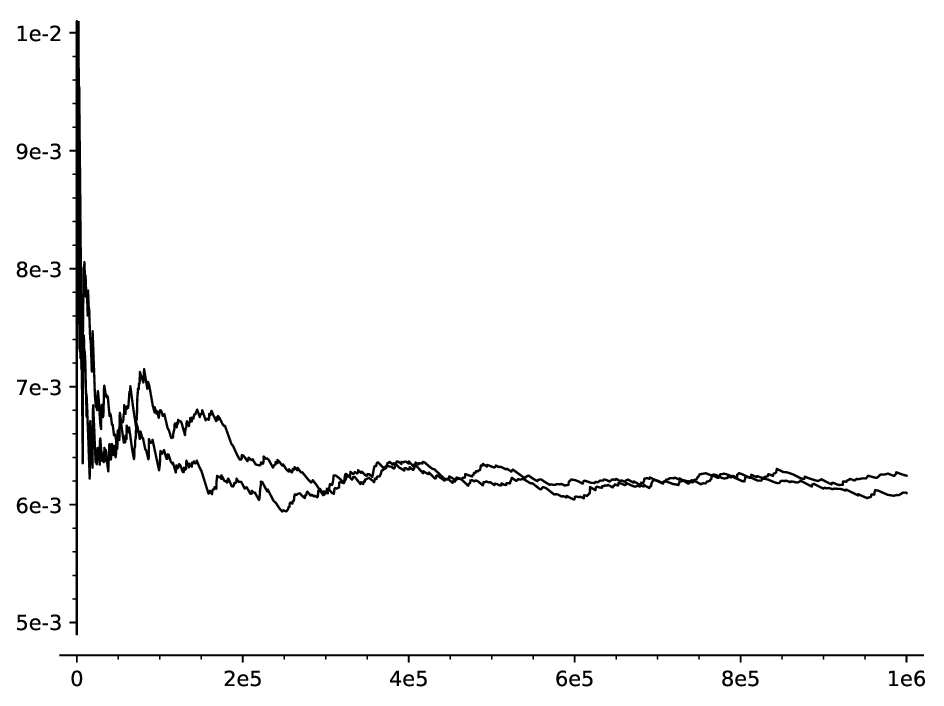}
\caption{$|l| = 1$: Top 1 bottom -1} \label{fig:15_6_even_A_1}
\end{subfigure}\hspace*{\fill}
\begin{subfigure}[b]{0.4\linewidth}
\includegraphics[width=\linewidth]{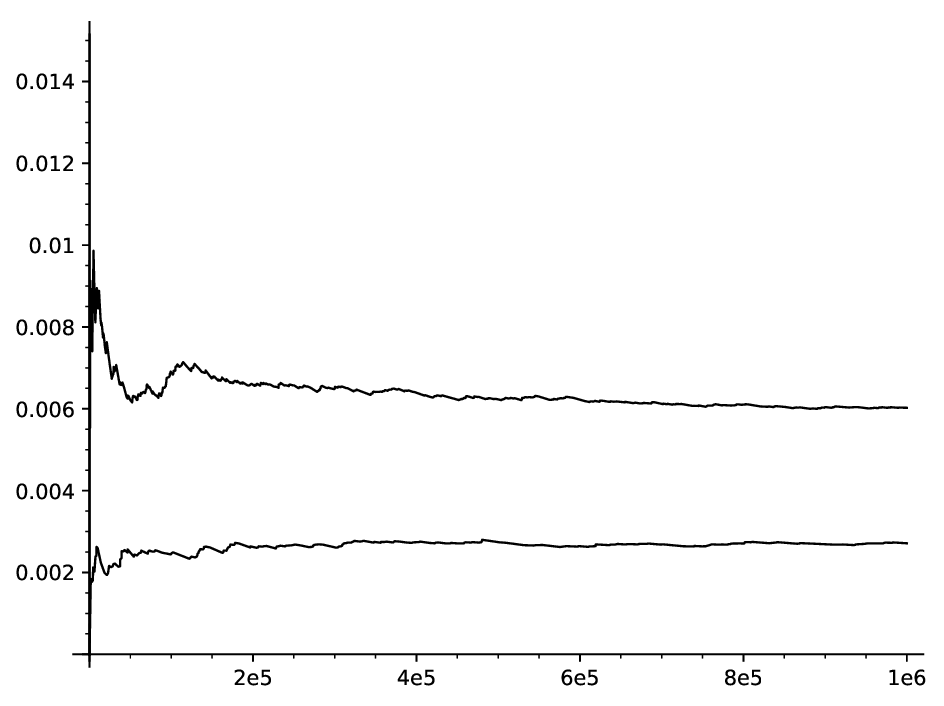}
\caption{$|l| = 2$: Top 2 bottom -2} \label{fig:15_6_even_A_2}
\end{subfigure}\hspace*{\fill}
\begin{subfigure}[b]{0.4\linewidth}
\includegraphics[width=\linewidth]{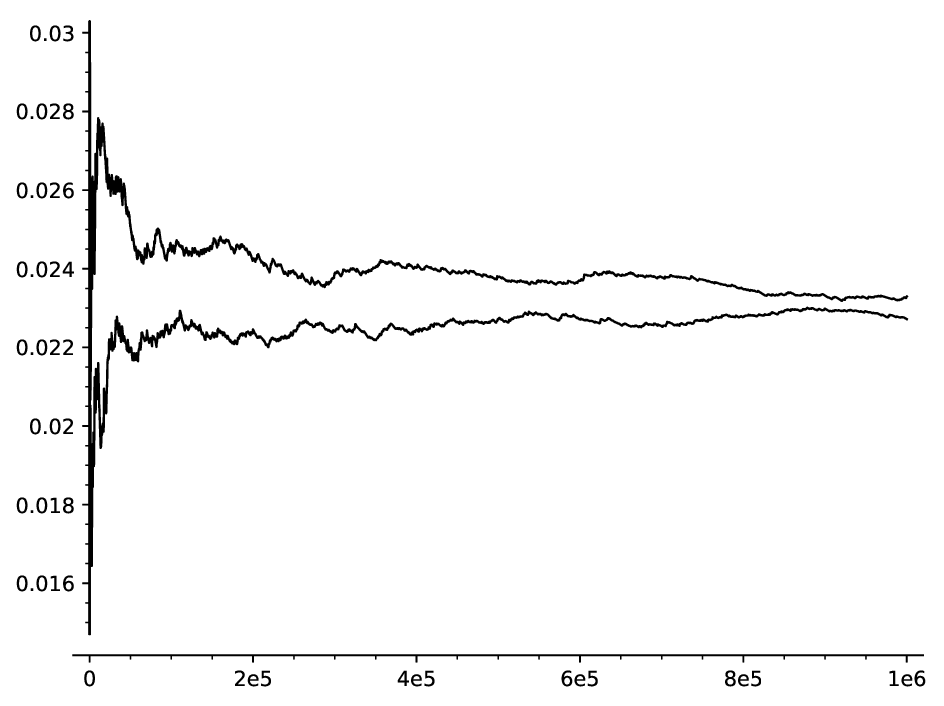}
\caption{$|l| = 3$: Top 3 bottom -3} \label{fig:15_6_even_A_3}
\end{subfigure}
\hspace*{-2.3cm}
\begin{subfigure}[b]{0.4\linewidth}
\includegraphics[width=\linewidth]{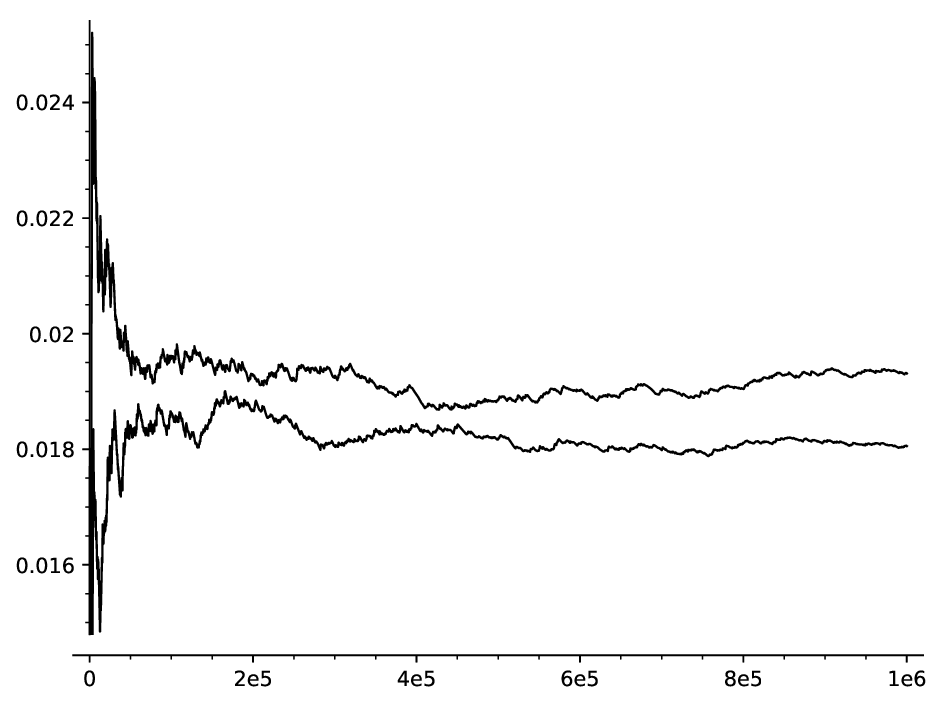}
\caption{$|l| = 4$: Top -4 bottom 4} \label{fig:15_6_even_A_4}
\end{subfigure}\hspace*{\fill}
\begin{subfigure}[b]{0.4\linewidth}
\includegraphics[width=\linewidth]{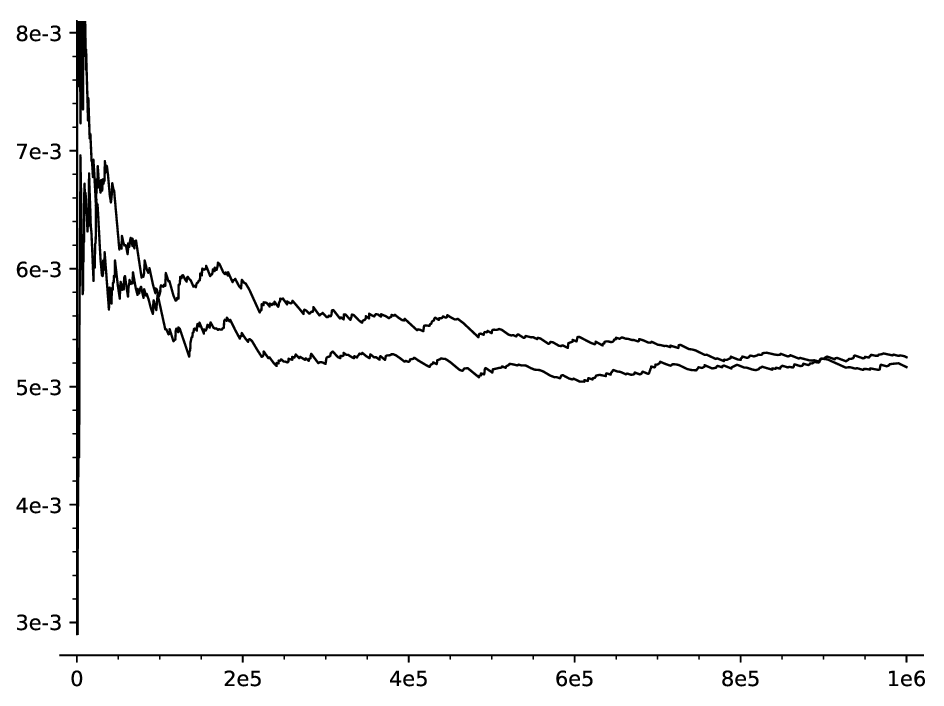}
\caption{$|l| = 5$: Top 5 bottom -5} \label{fig:15_6_even_A_5}
\end{subfigure}\hspace*{\fill}
\begin{subfigure}[b]{0.4\linewidth}
\includegraphics[width=\linewidth]{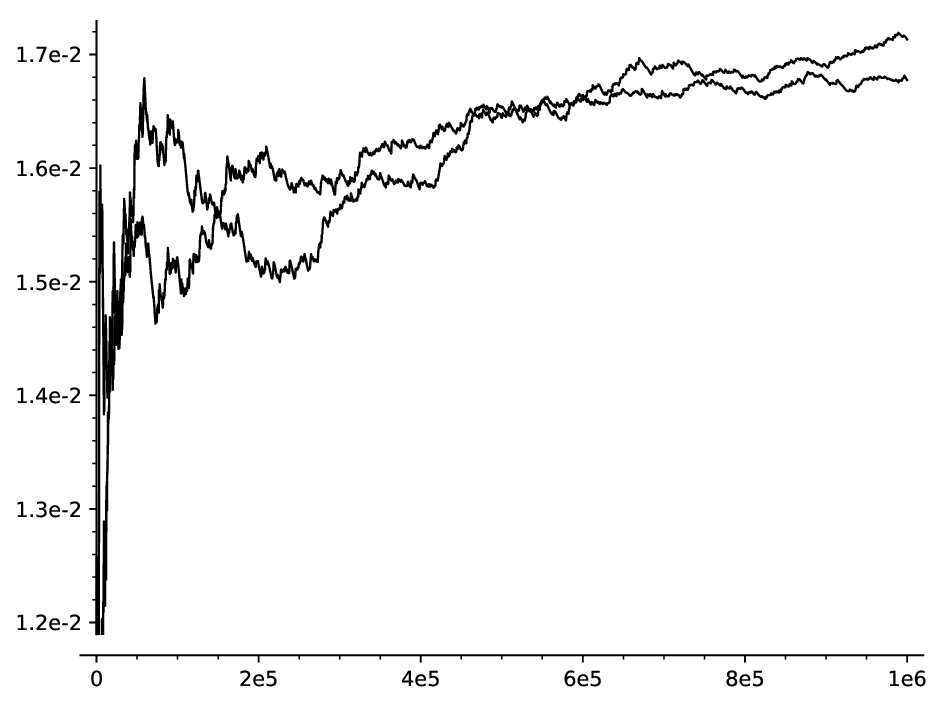}
\caption{$|l| = 6$: Top 6 bottom -6} \label{fig:15_6_even_A_6}
\end{subfigure}
\hspace*{-2.3cm}
\begin{subfigure}[b]{0.4\linewidth}
\includegraphics[width=\linewidth]{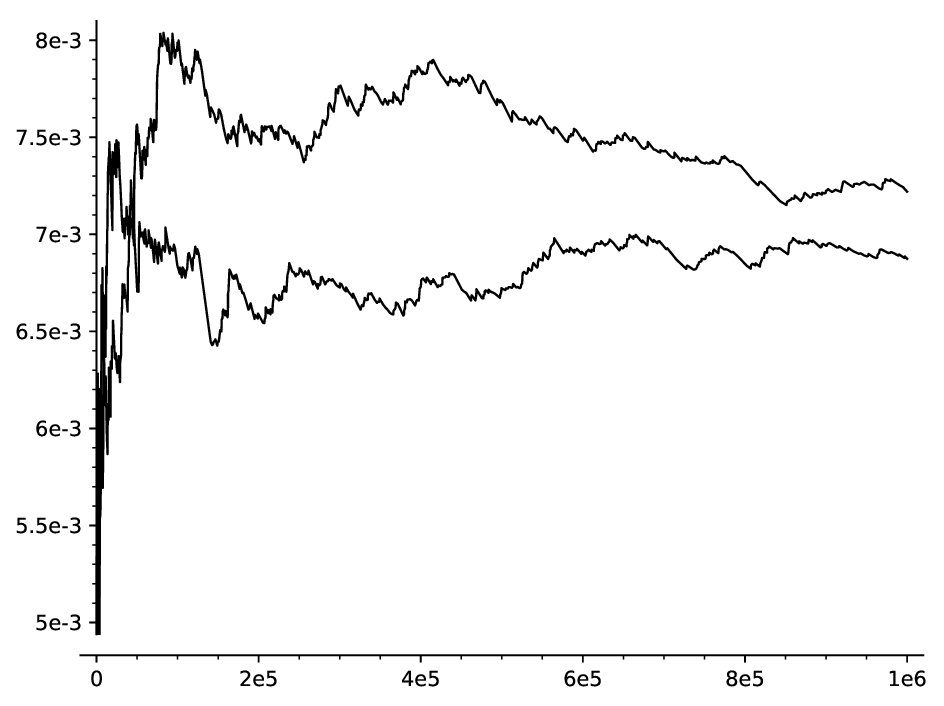}
\caption{$|l| = 7$: Top 7 bottom -7} \label{fig:15_6_even_A_7}
\end{subfigure}\hspace*{\fill}
\begin{subfigure}[b]{0.4\linewidth}
\includegraphics[width=\linewidth]{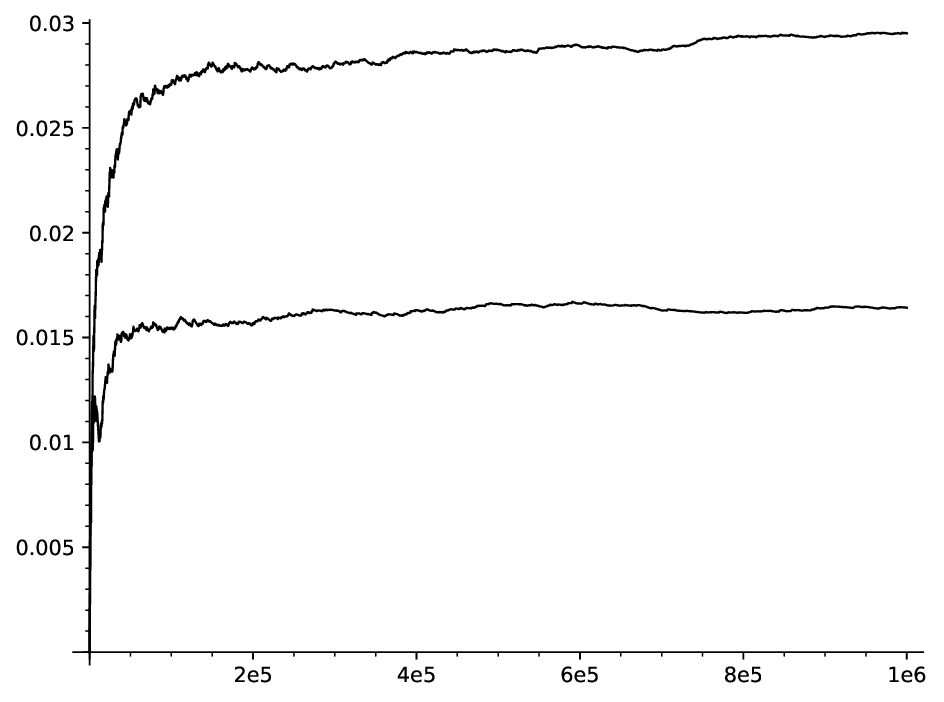}
\caption{$|l| = 8$: Top 8 bottom -8} \label{fig:15_6_even_A_8}
\end{subfigure}\hspace*{\fill}
\begin{subfigure}[b]{0.4\linewidth}
\includegraphics[width=\linewidth]{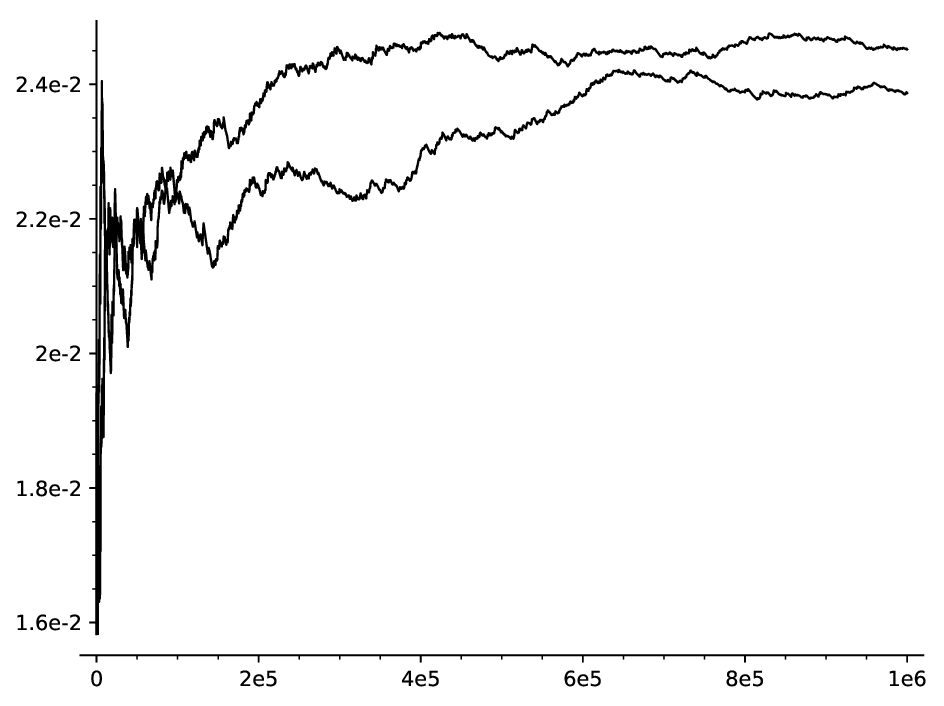}
\caption{$|l| = 9$: Top 9 bottom -9} \label{fig:15_6_even_A_9}
\end{subfigure}
\caption{15a1: Ratio~\eqref{ratio_n_pm} $x_{6,E}^+(X;l)/X^{1/2}\log^2(X)$} \label{fig:15a1_6_even_A_exact}
\end{figure}

\begin{figure}[b] 
\hspace*{-2.3cm}
\begin{subfigure}[b]{0.4\linewidth}
\includegraphics[width=\linewidth]{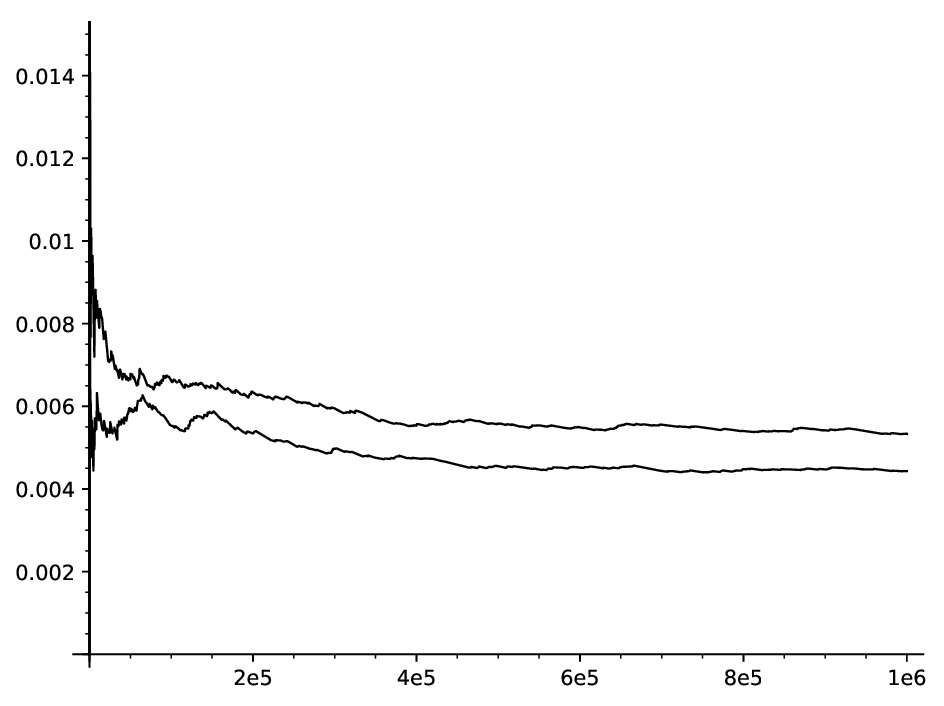}
\caption{$|l| = 1$: Top 1 bottom -1} \label{fig:15_6_odd_A_1}
\end{subfigure}\hspace*{\fill}
\begin{subfigure}[b]{0.4\linewidth}
\includegraphics[width=\linewidth]{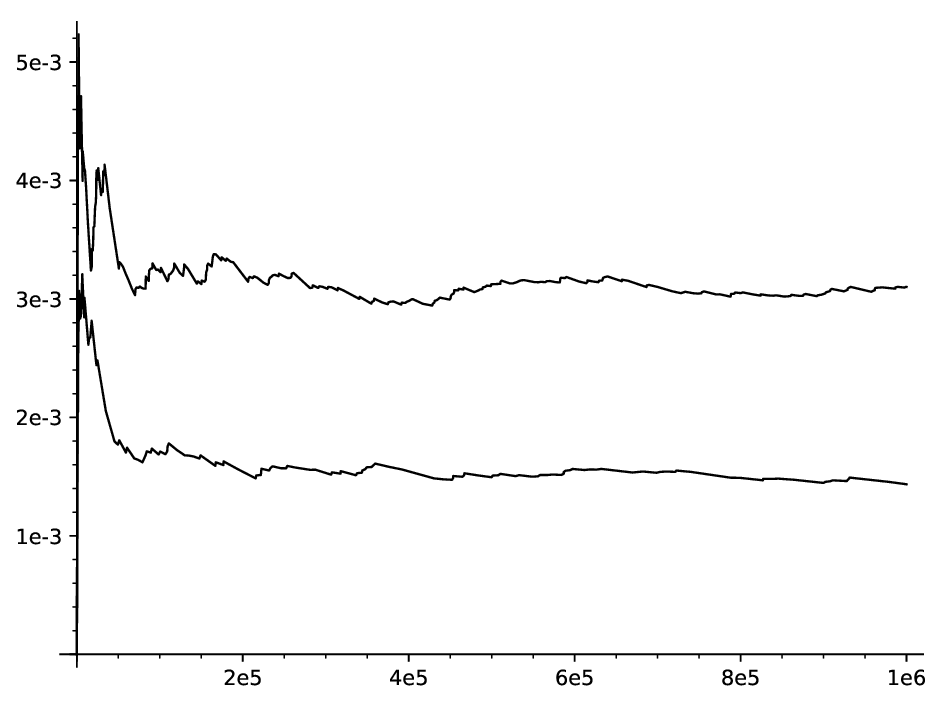}
\caption{$|l| = 2$: Top 2 bottom -2} \label{fig:15_6_odd_A_2}
\end{subfigure}\hspace*{\fill}
\begin{subfigure}[b]{0.4\linewidth}
\includegraphics[width=\linewidth]{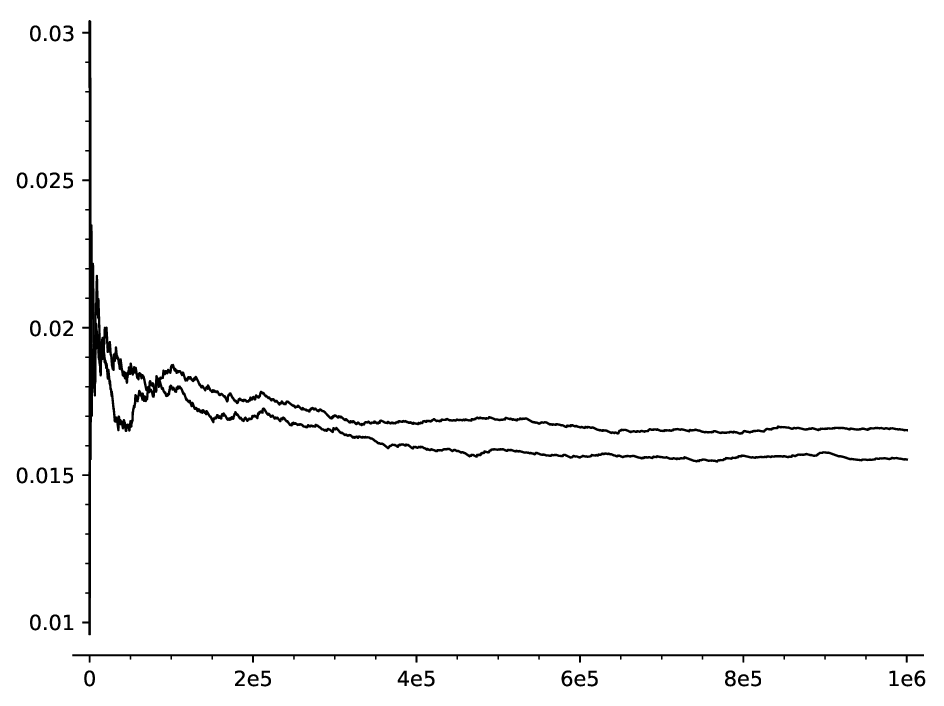}
\caption{$|l| = 3$: Top -3 bottom 3} \label{fig:15_6_odd_A_3}
\end{subfigure}
\hspace*{-2.3cm}
\begin{subfigure}[b]{0.4\linewidth}
\includegraphics[width=\linewidth]{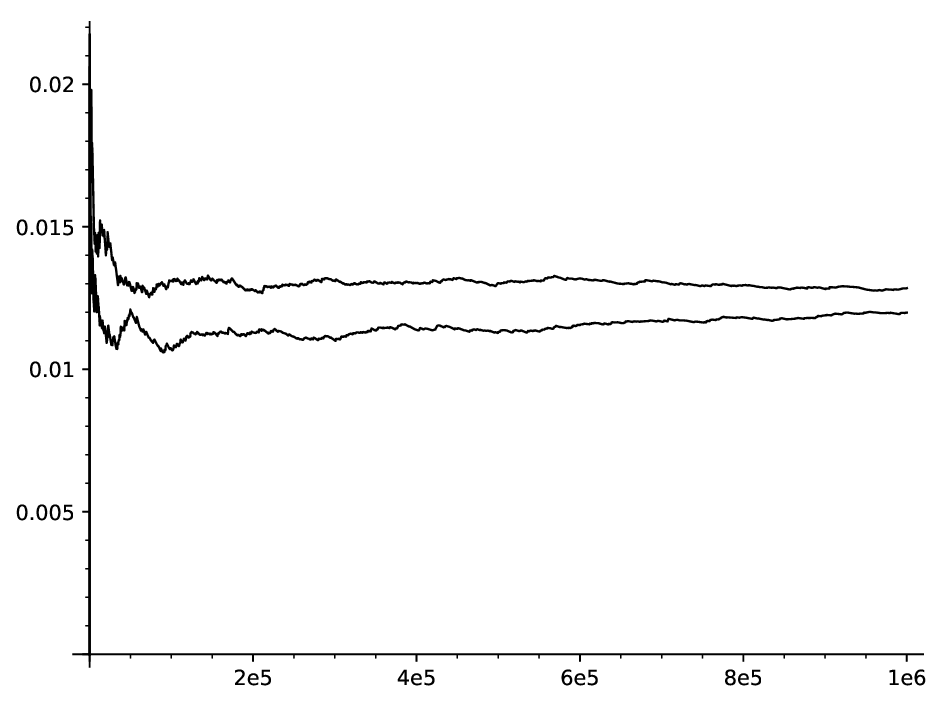}
\caption{$|l| = 4$: Top -4 bottom 4} \label{fig:15_6_odd_A_4}
\end{subfigure}\hspace*{\fill}
\begin{subfigure}[b]{0.4\linewidth}
\includegraphics[width=\linewidth]{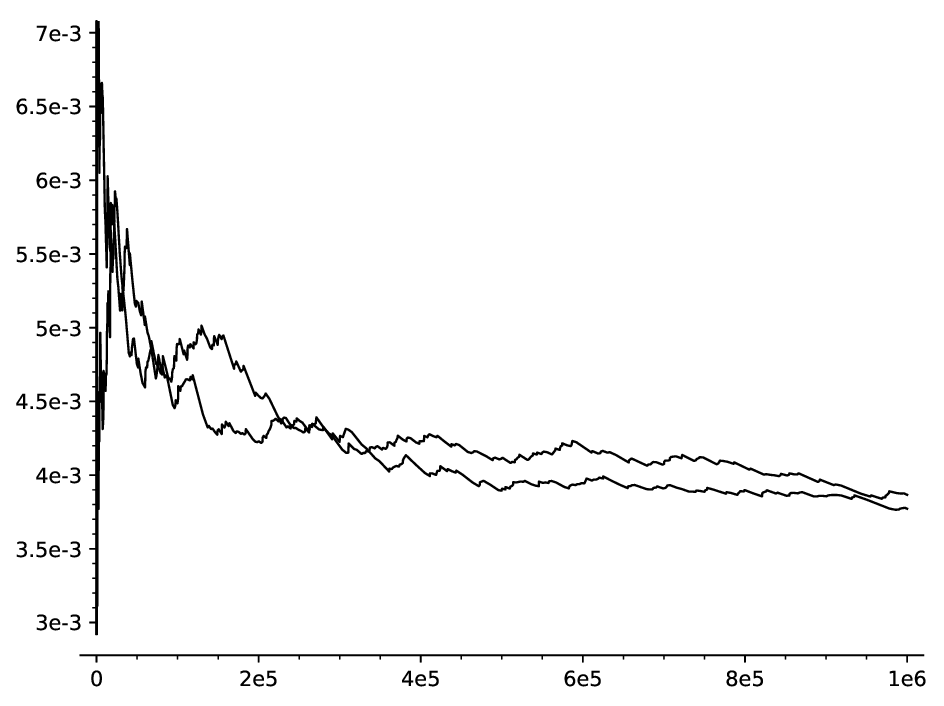}
\caption{$|l| = 5$: Top 5 bottom -5} \label{fig:15_6_odd_A_5}
\end{subfigure}\hspace*{\fill}
\begin{subfigure}[b]{0.4\linewidth}
\includegraphics[width=\linewidth]{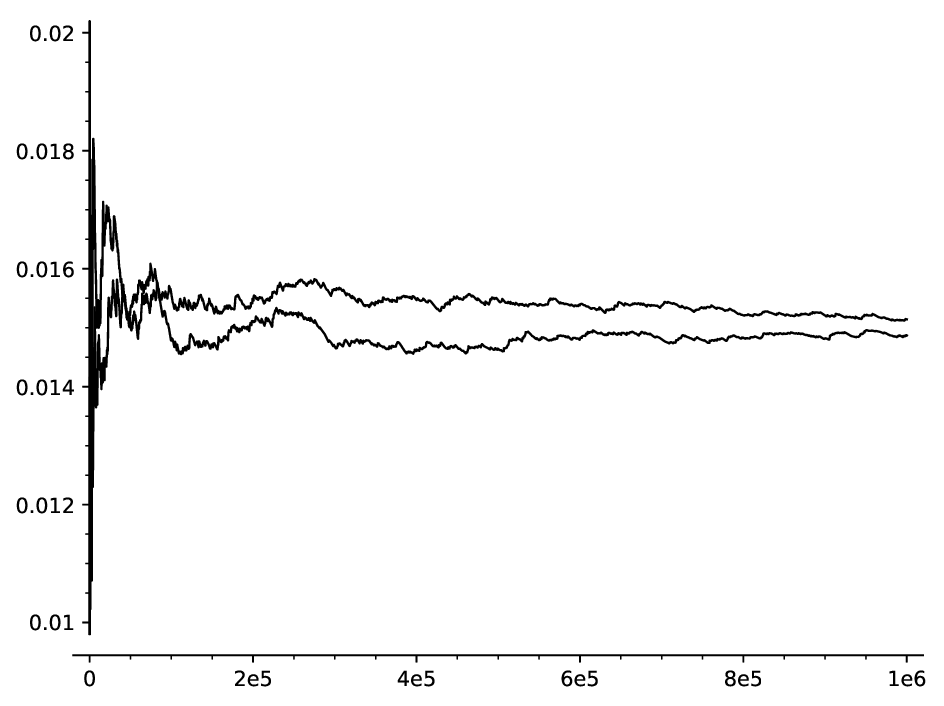}
\caption{$|l| = 6$: Top 6 bottom -6} \label{fig:15_6_odd_A_6}
\end{subfigure}
\hspace*{-2.3cm}
\begin{subfigure}[b]{0.4\linewidth}
\includegraphics[width=\linewidth]{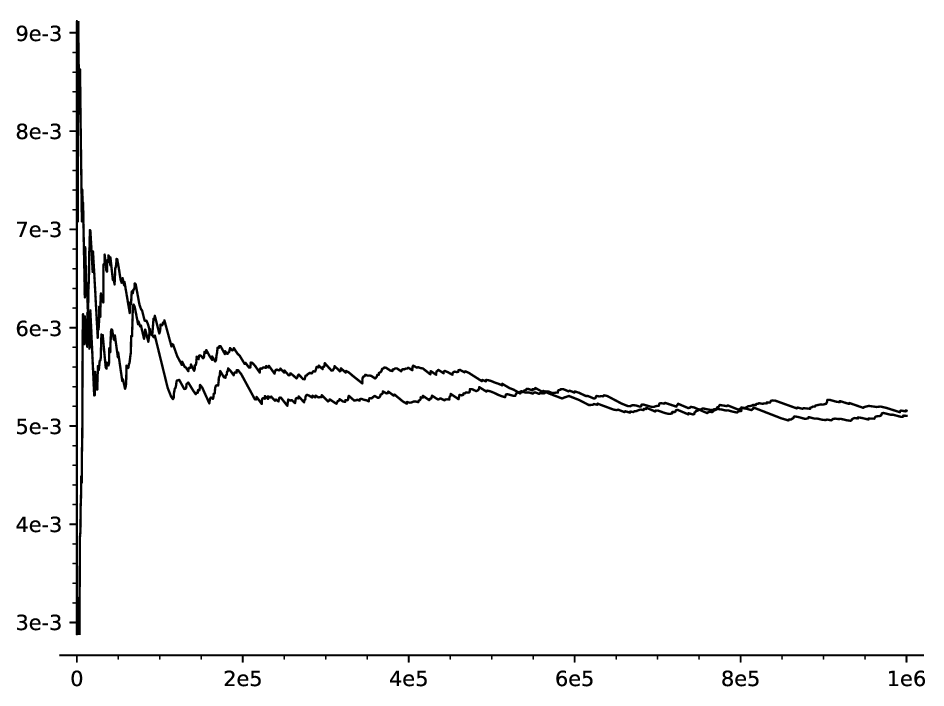}
\caption{$|l| = 7$: Top -7 bottom 7} \label{fig:15_6_odd_A_7}
\end{subfigure}\hspace*{\fill}
\begin{subfigure}[b]{0.4\linewidth}
\includegraphics[width=\linewidth]{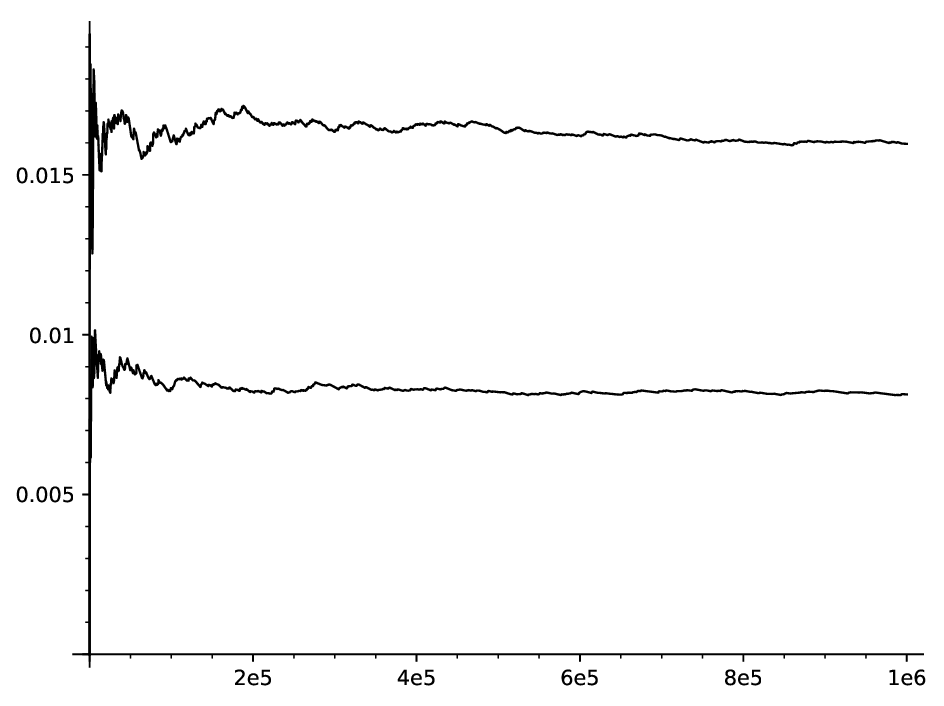}
\caption{$|l| = 8$: Top 8 bottom -8} \label{fig:15_6_od_A_8}
\end{subfigure}\hspace*{\fill}
\begin{subfigure}[b]{0.4\linewidth}
\includegraphics[width=\linewidth]{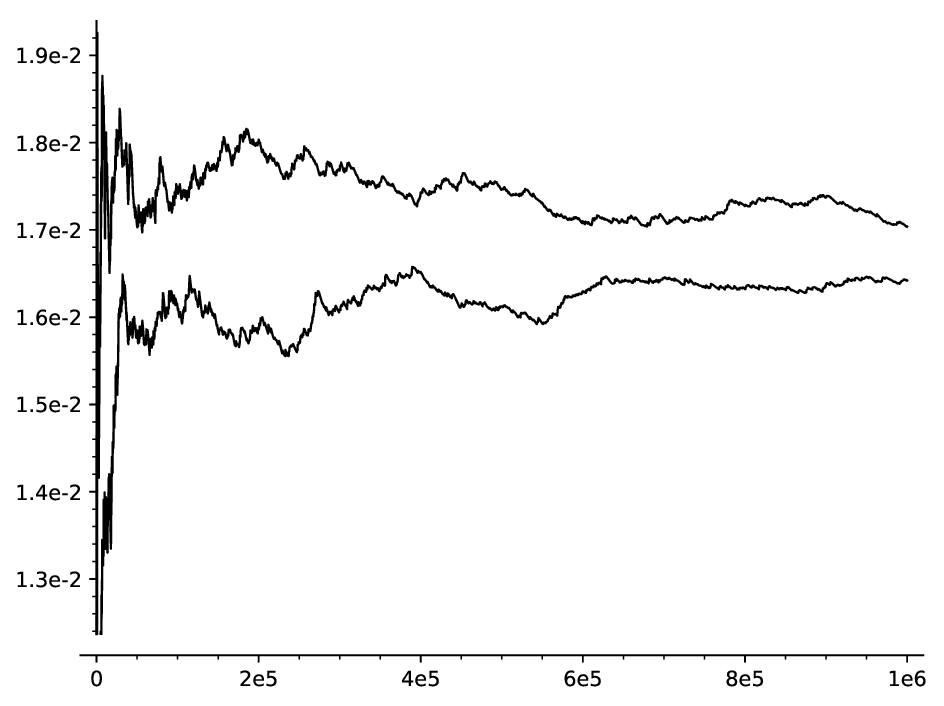}
\caption{$|l| = 9$: Top 9 bottom -9} \label{fig:15_6_odd_A_9}
\end{subfigure}
\caption{15a1: Ratio~\eqref{ratio_n_pm} $x_{6,E}^-(X;l)/X^{1/2}\log^2(X)$} \label{fig:15a1_6_odd_A_exact}
\end{figure}

\clearpage

\begin{figure}[t] 
\hspace*{-2.3cm}
\begin{subfigure}[b]{0.4\linewidth}
\includegraphics[width=\linewidth]{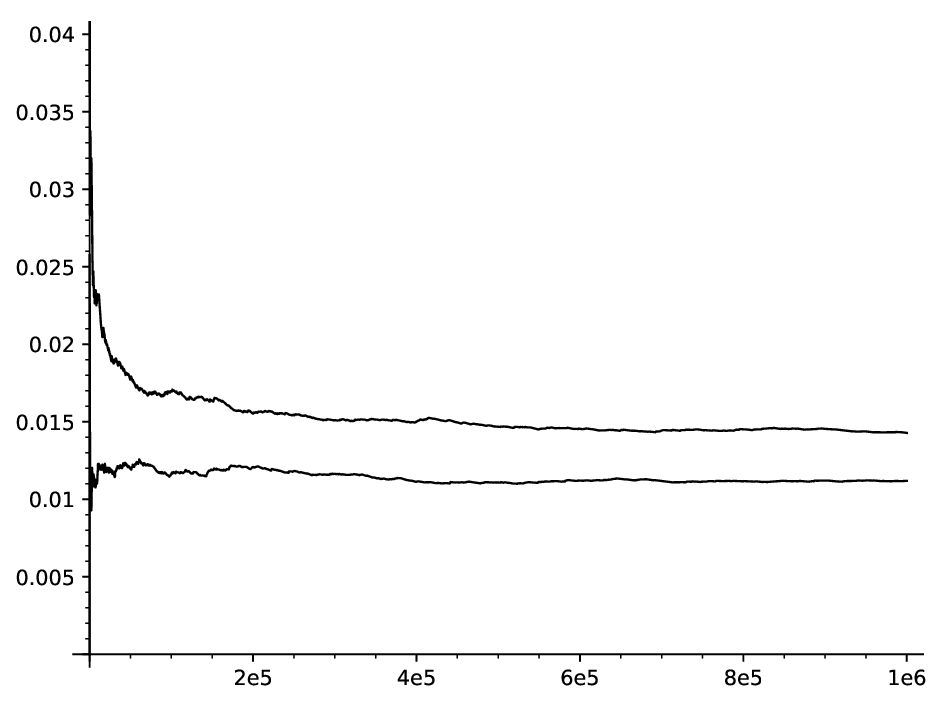}
\caption{$|l| = 1$: Top -1 bottom 1} \label{fig:17_6_even_A_1}
\end{subfigure}\hspace*{\fill}
\begin{subfigure}[b]{0.4\linewidth}
\includegraphics[width=\linewidth]{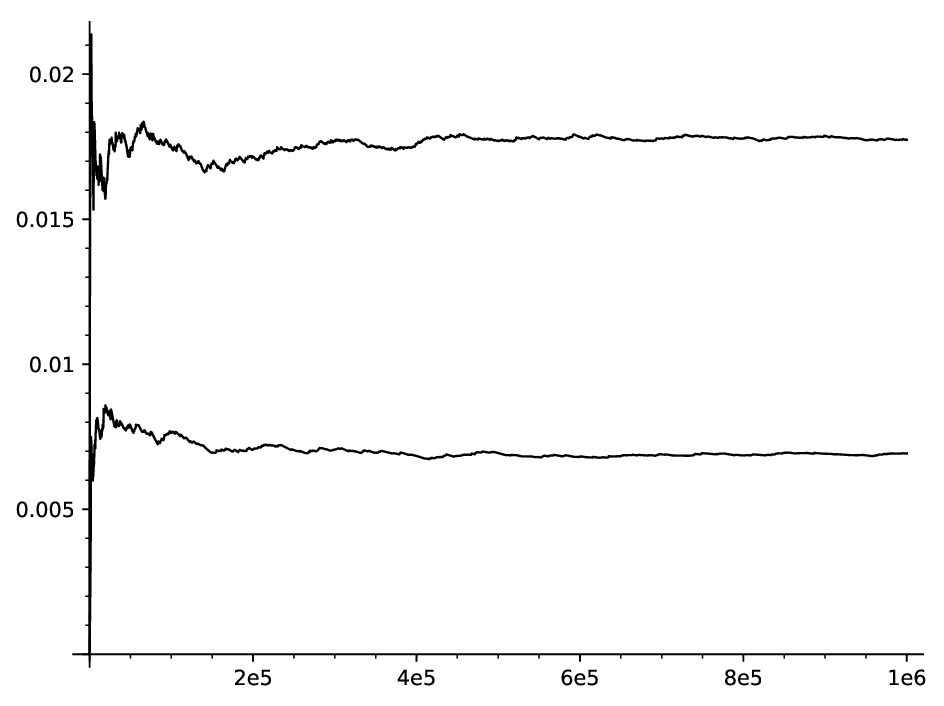}
\caption{$|l| = 2$: Top 2 bottom -2} \label{fig:17_6_even_A_2}
\end{subfigure}\hspace*{\fill}
\begin{subfigure}[b]{0.4\linewidth}
\includegraphics[width=\linewidth]{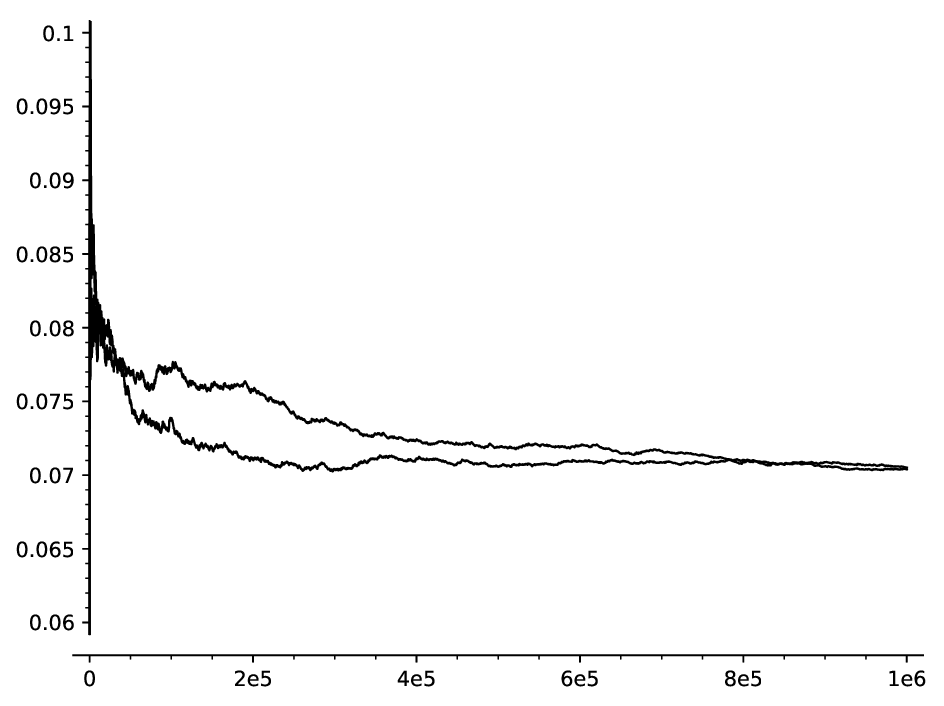}
\caption{$|l| = 3$: Top 3 bottom -3} \label{fig:17_6_even_A_3}
\end{subfigure}
\hspace*{-2.3cm}
\begin{subfigure}[b]{0.4\linewidth}
\includegraphics[width=\linewidth]{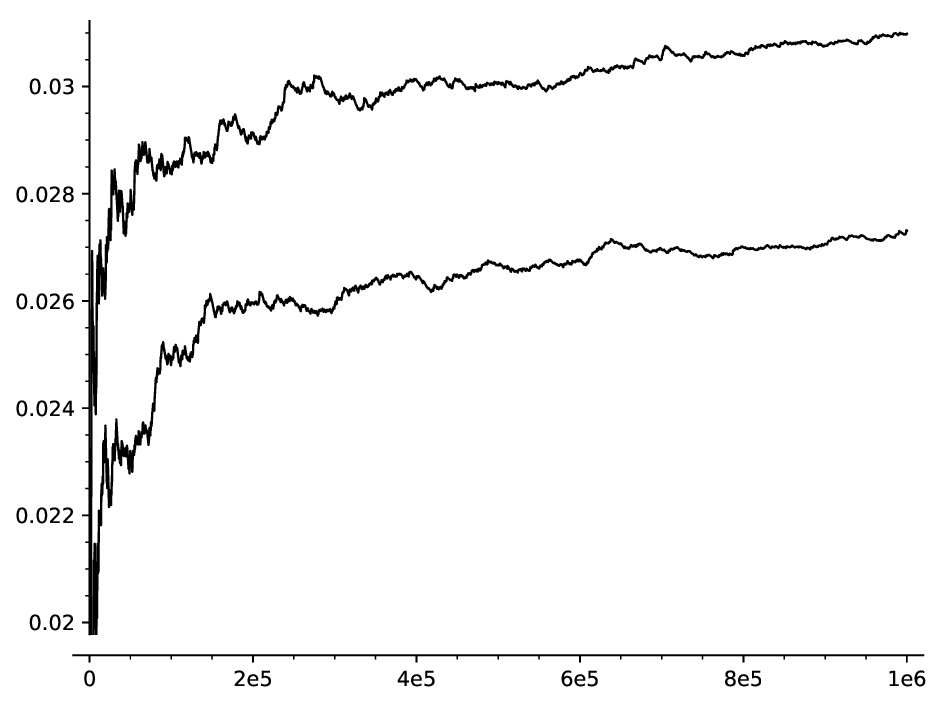}
\caption{$|l| = 4$: Top -4 bottom 4} \label{fig:17_6_even_A_4}
\end{subfigure}\hspace*{\fill}
\begin{subfigure}[b]{0.4\linewidth}
\includegraphics[width=\linewidth]{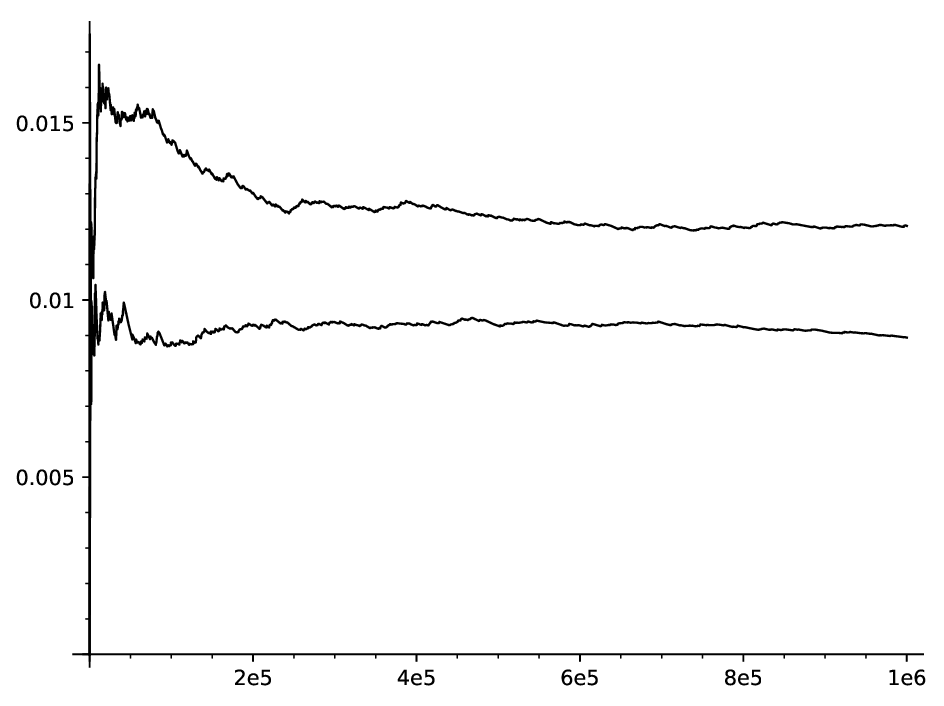}
\caption{$|l| = 5$: Top 5 bottom -5} \label{fig:17_6_even_A_5}
\end{subfigure}\hspace*{\fill}
\begin{subfigure}[b]{0.4\linewidth}
\includegraphics[width=\linewidth]{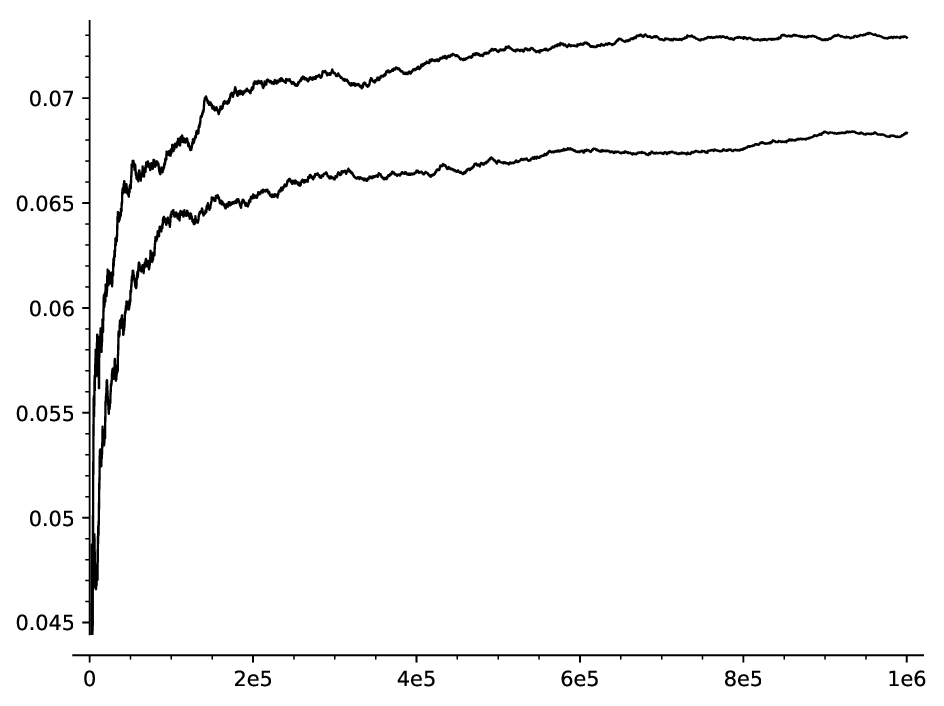}
\caption{$|l| = 6$: Top 6 bottom -6} \label{fig:17_6_even_A_6}
\end{subfigure}
\hspace*{-2.3cm}
\begin{subfigure}[b]{0.4\linewidth}
\includegraphics[width=\linewidth]{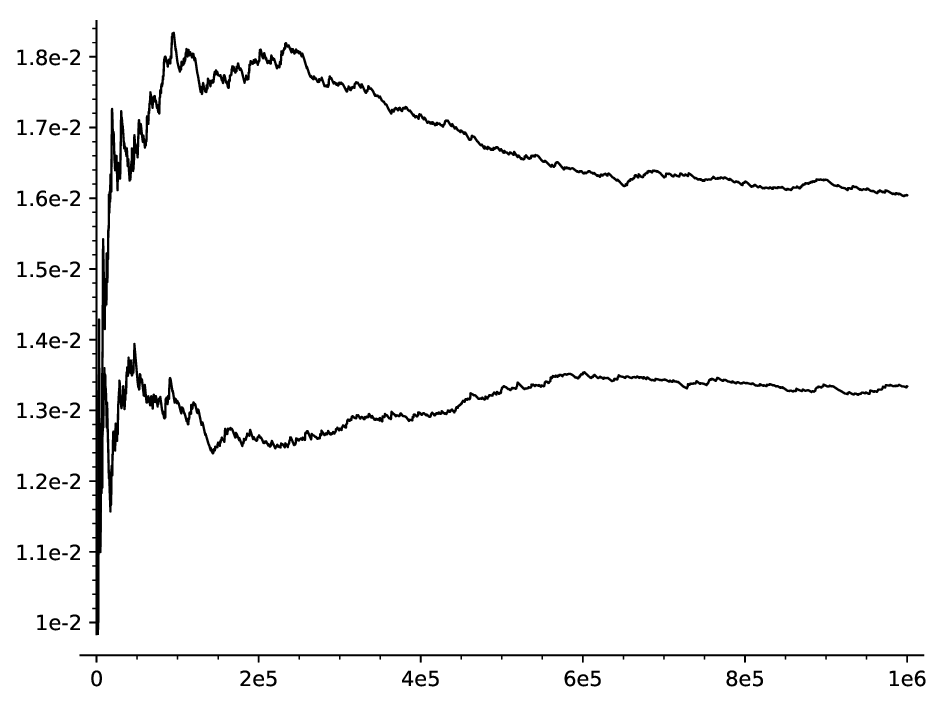}
\caption{$|l| = 7$: Top -7 bottom 7} \label{fig:17_6_even_A_7}
\end{subfigure}\hspace*{\fill}
\begin{subfigure}[b]{0.4\linewidth}
\includegraphics[width=\linewidth]{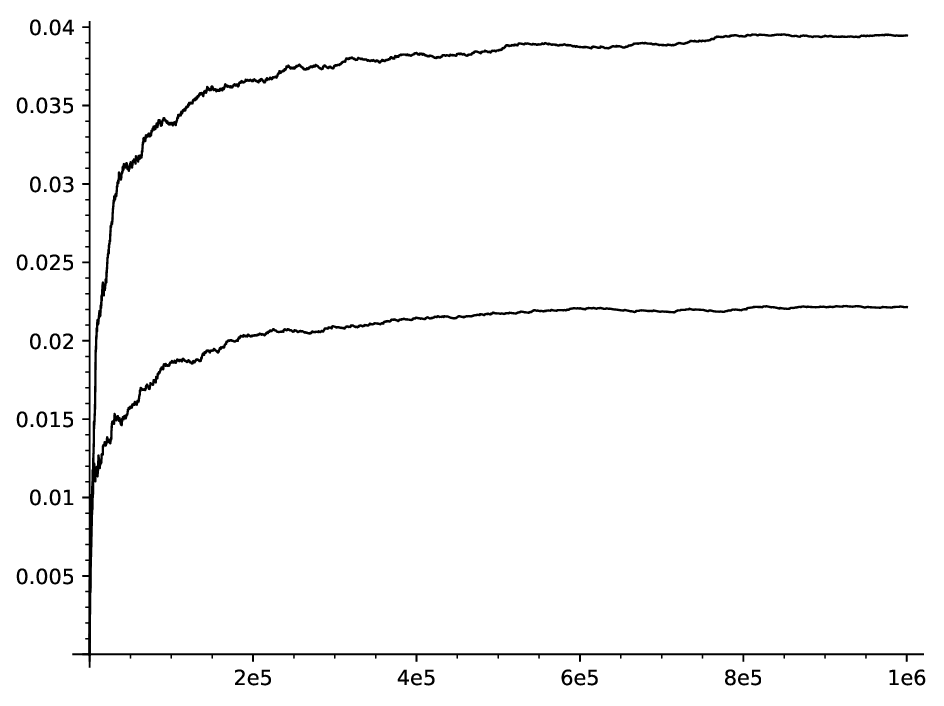}
\caption{$|l| = 8$: Top 8 bottom -8} \label{fig:17_6_even_A_8}
\end{subfigure}\hspace*{\fill}
\begin{subfigure}[b]{0.4\linewidth}
\includegraphics[width=\linewidth]{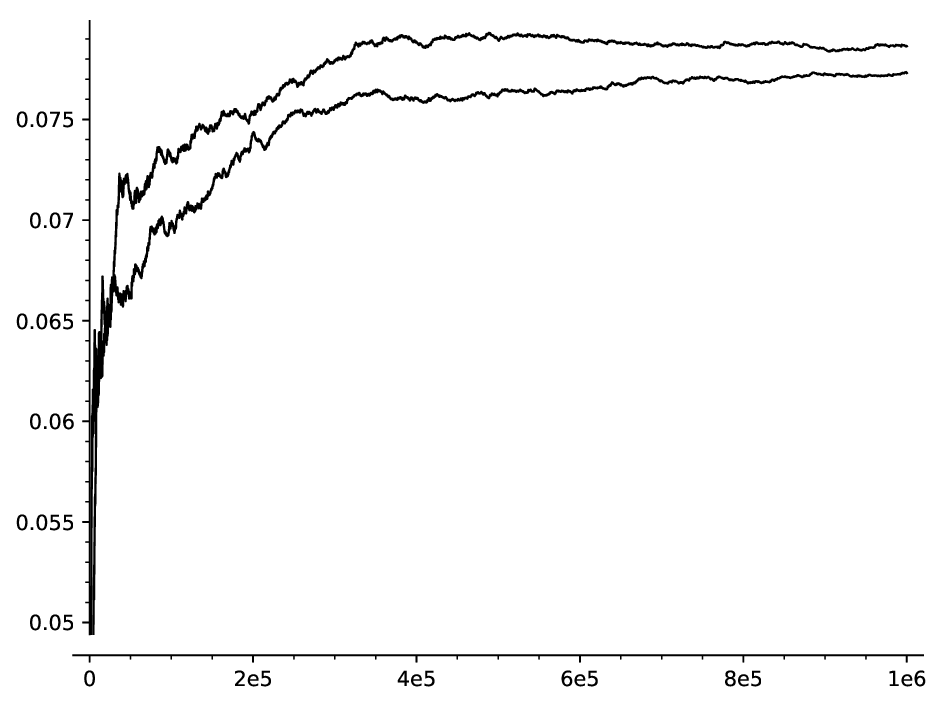}
\caption{$|l| = 9$: Top 9 bottom -9} \label{fig:17_6_even_A_9}
\end{subfigure}
\caption{17a1: Ratio~\eqref{ratio_n_pm} $x_{6,E}^+(X;l)/X^{1/2}\log^2(X)$} \label{fig:17a1_6_even_A_exact}
\end{figure}

\begin{figure}[b] 
\hspace*{-2.3cm}
\begin{subfigure}[b]{0.4\linewidth}
\includegraphics[width=\linewidth]{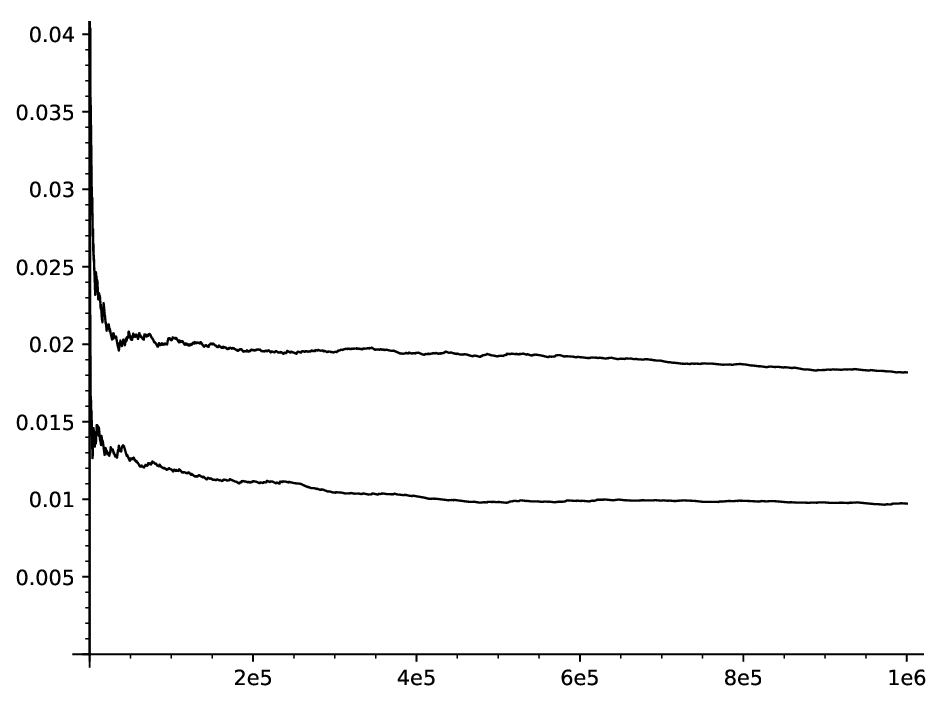}
\caption{$|l| = 1$: Top 1 bottom -1} \label{fig:17_6_odd_A_1}
\end{subfigure}\hspace*{\fill}
\begin{subfigure}[b]{0.4\linewidth}
\includegraphics[width=\linewidth]{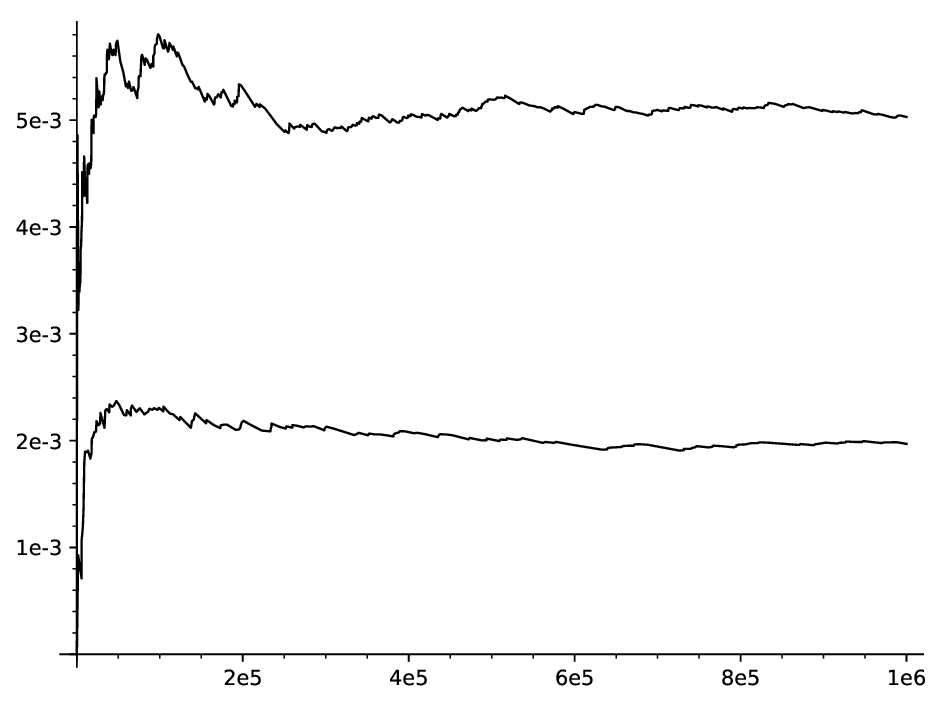}
\caption{$|l| = 2$: Top 2 bottom -2} \label{fig:17_6_odd_A_2}
\end{subfigure}\hspace*{\fill}
\begin{subfigure}[b]{0.4\linewidth}
\includegraphics[width=\linewidth]{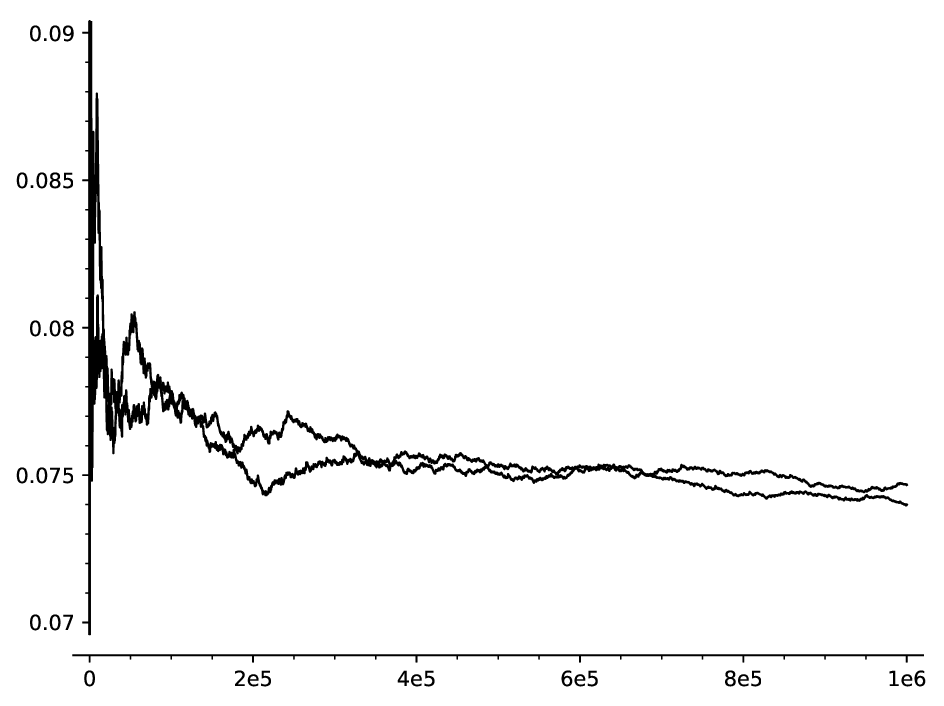}
\caption{$|l| = 3$: Top 3 bottom -3} \label{fig:17_6_odd_A_3}
\end{subfigure}
\hspace*{-2.3cm}
\begin{subfigure}[b]{0.4\linewidth}
\includegraphics[width=\linewidth]{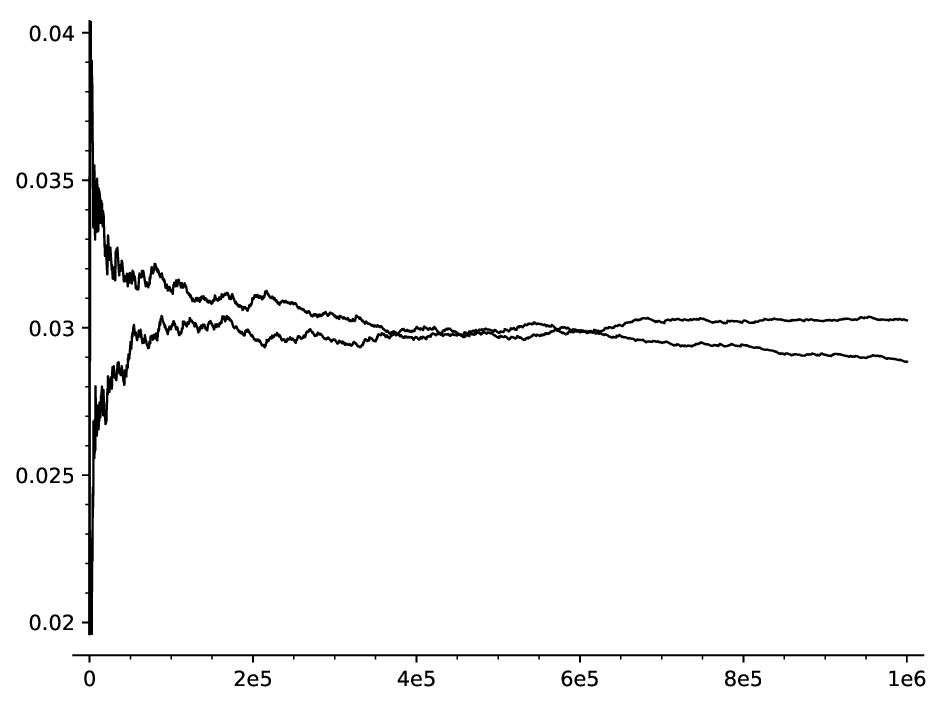}
\caption{$|l| = 4$: Top 4 bottom -4} \label{fig:17_6_odd_A_4}
\end{subfigure}\hspace*{\fill}
\begin{subfigure}[b]{0.4\linewidth}
\includegraphics[width=\linewidth]{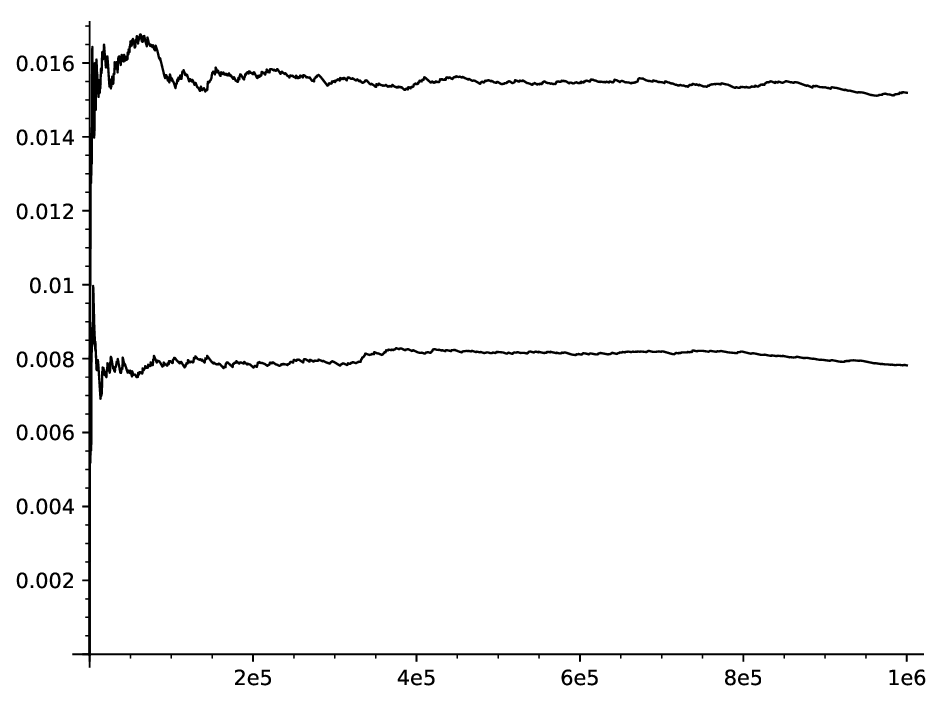}
\caption{$|l| = 5$: Top -5 bottom 5} \label{fig:17_6_odd_A_5}
\end{subfigure}\hspace*{\fill}
\begin{subfigure}[b]{0.4\linewidth}
\includegraphics[width=\linewidth]{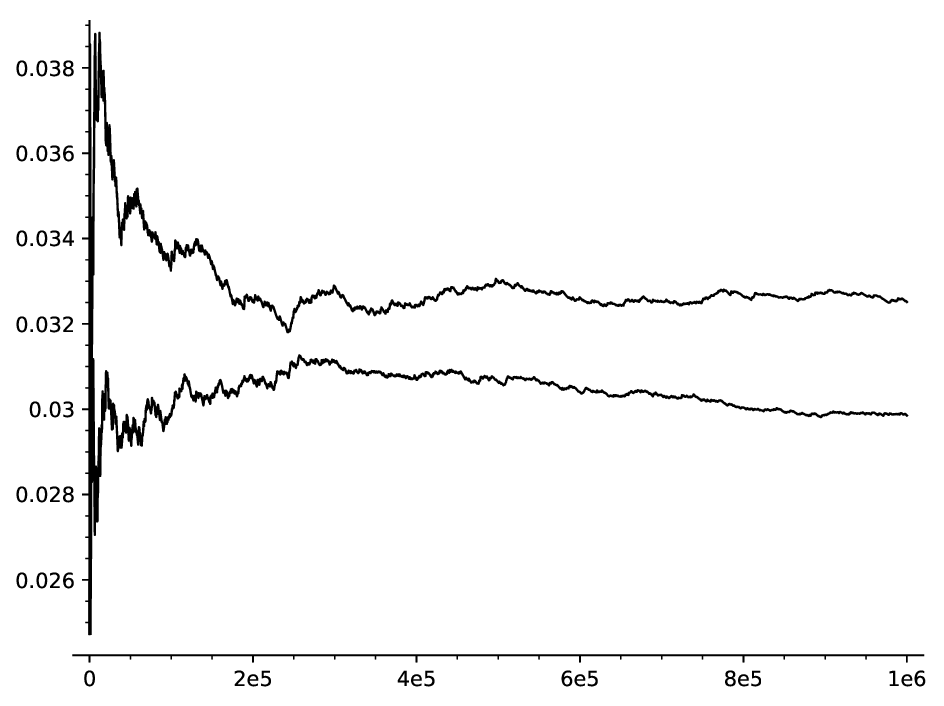}
\caption{$|l| = 6$: Top 6 bottom -6} \label{fig:17_6_odd_A_6}
\end{subfigure}
\hspace*{-2.3cm}
\begin{subfigure}[b]{0.4\linewidth}
\includegraphics[width=\linewidth]{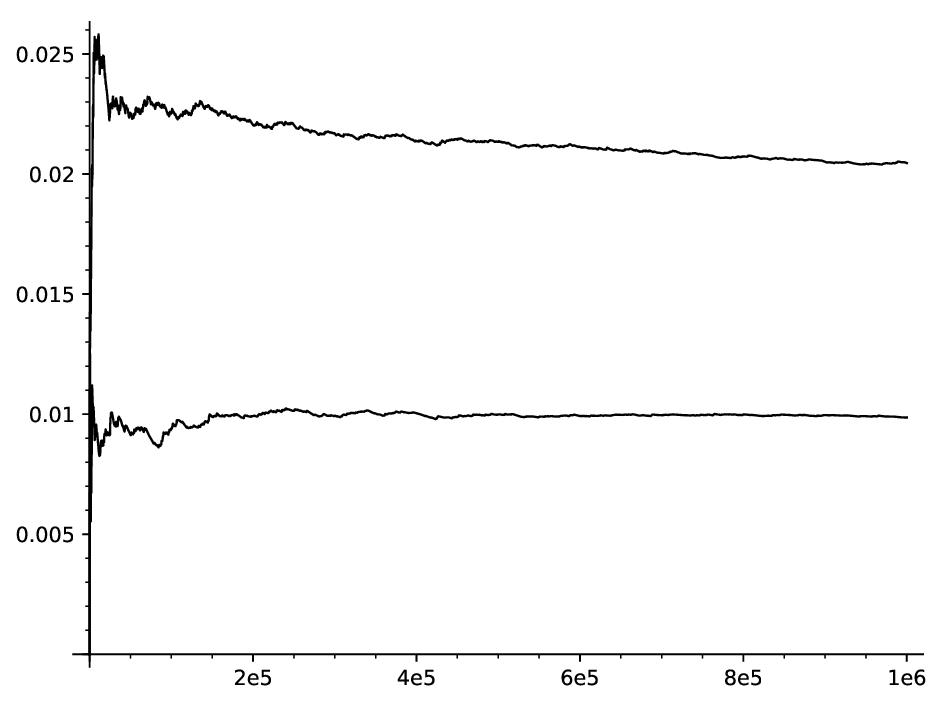}
\caption{$|l| = 7$: Top 7 bottom -7} \label{fig:17_6_odd_A_7}
\end{subfigure}\hspace*{\fill}
\begin{subfigure}[b]{0.4\linewidth}
\includegraphics[width=\linewidth]{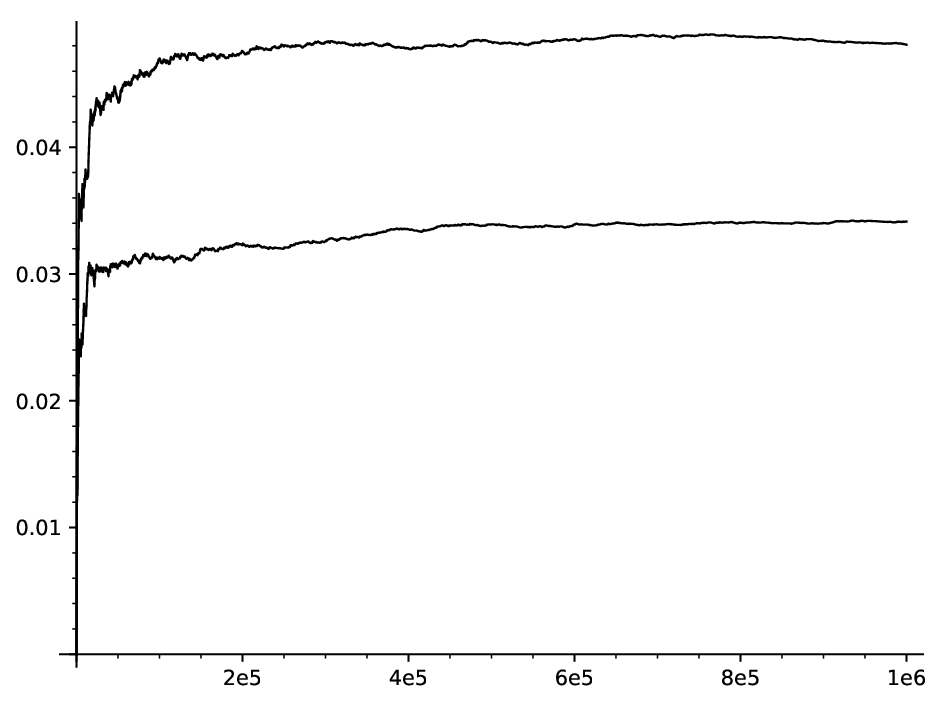}
\caption{$|l| = 8$: Top 8 bottom -8} \label{fig:17_6_od_A_8}
\end{subfigure}\hspace*{\fill}
\begin{subfigure}[b]{0.4\linewidth}
\includegraphics[width=\linewidth]{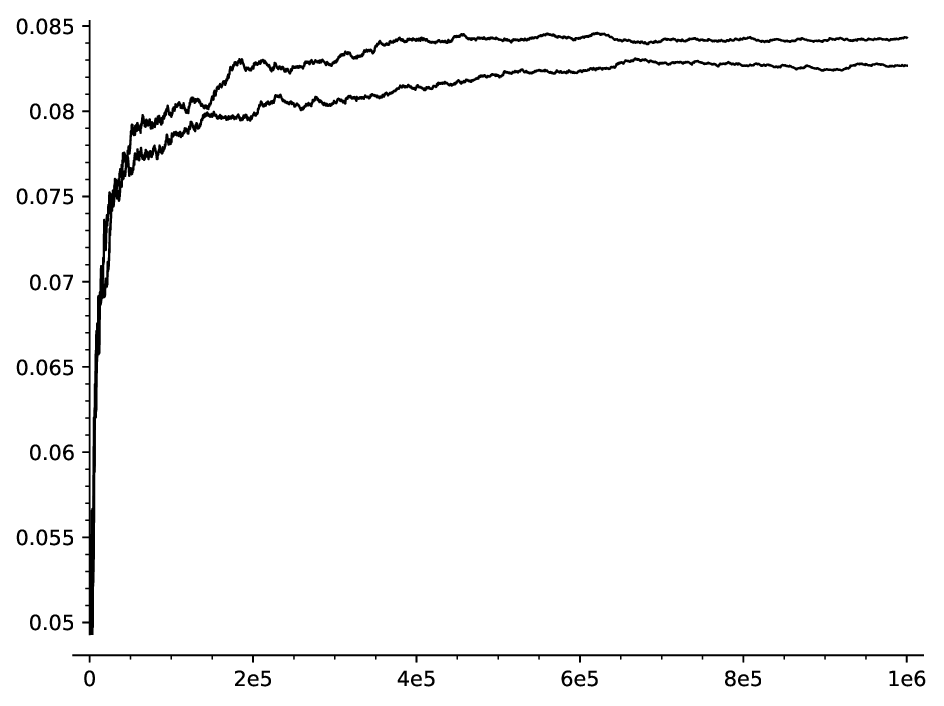}
\caption{$|l| = 9$: Top 9 bottom -9} \label{fig:17_6_odd_A_9}
\end{subfigure}
\caption{17a1: Ratio~\eqref{ratio_n_pm} $x_{6,E}^-(X;l)/X^{1/2}\log^2(X)$} \label{fig:17a1_6_odd_A_exact}
\end{figure}

\clearpage

\begin{figure}[t] 
\hspace*{-2.3cm}
\begin{subfigure}[b]{0.4\linewidth}
\includegraphics[width=\linewidth]{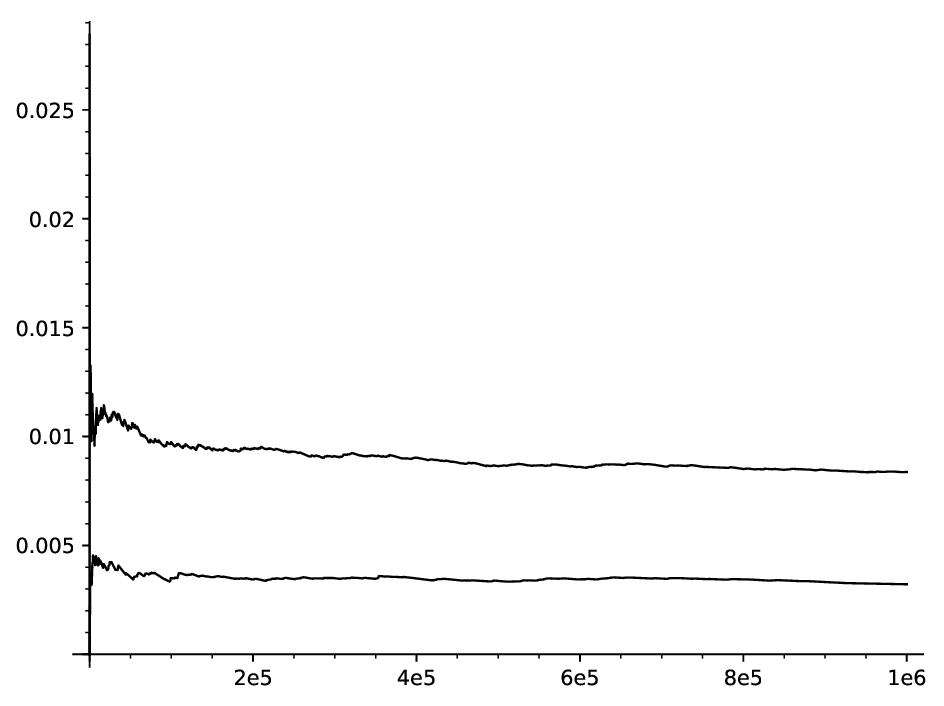}
\caption{$|l| = 1$: Top -1 bottom 1} \label{fig:19_6_even_A_1}
\end{subfigure}\hspace*{\fill}
\begin{subfigure}[b]{0.4\linewidth}
\includegraphics[width=\linewidth]{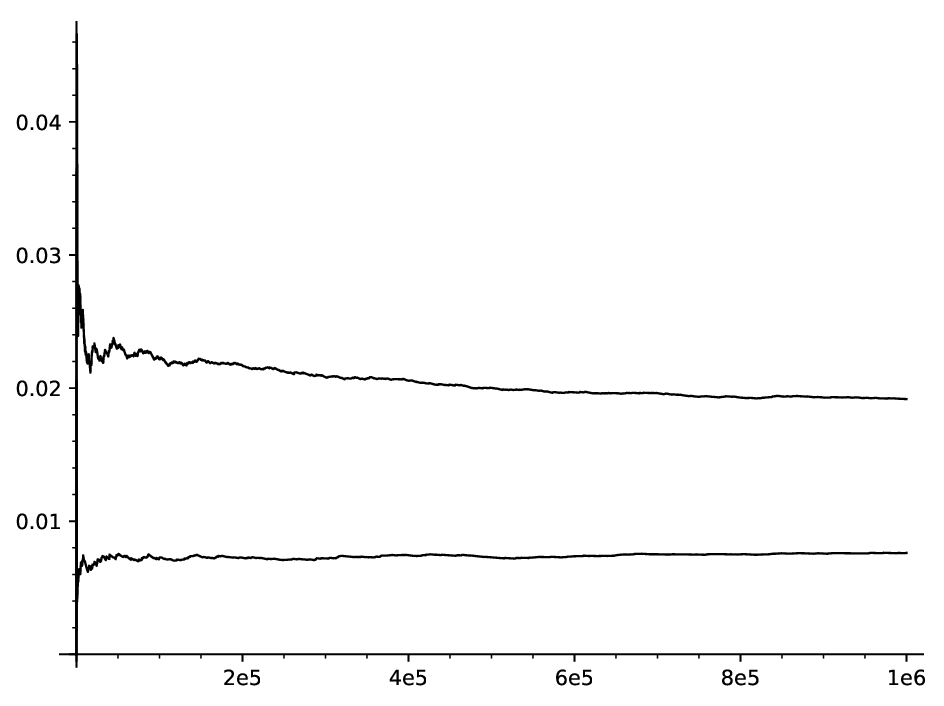}
\caption{$|l| = 2$: Top 2 bottom -2} \label{fig:19_6_even_A_2}
\end{subfigure}\hspace*{\fill}
\begin{subfigure}[b]{0.4\linewidth}
\includegraphics[width=\linewidth]{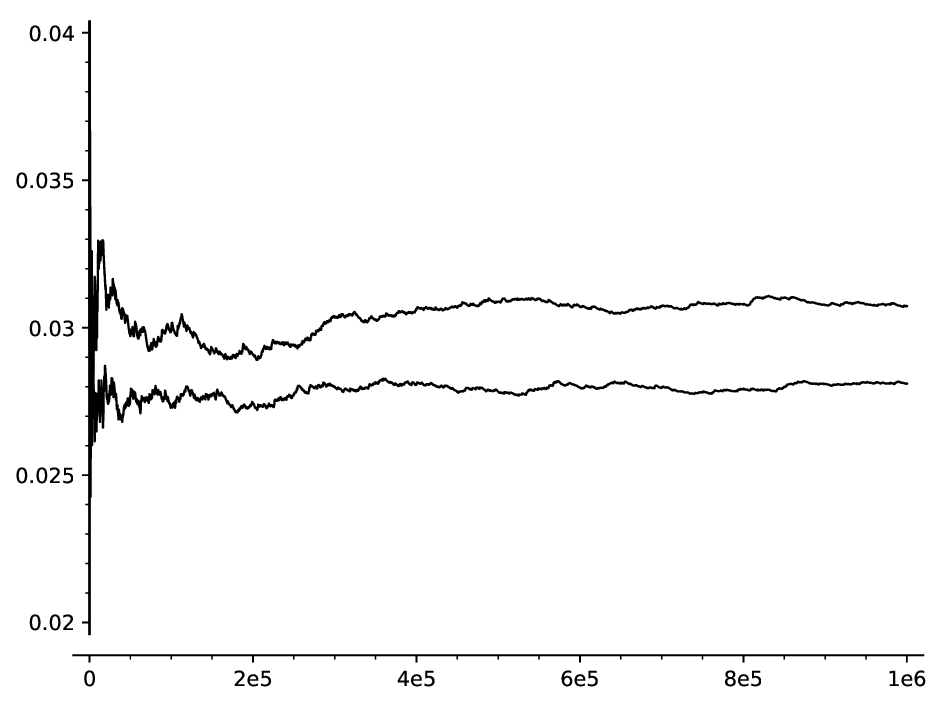}
\caption{$|l| = 3$: Top 3 bottom -3} \label{fig:19_6_even_A_3}
\end{subfigure}
\hspace*{-2.3cm}
\begin{subfigure}[b]{0.4\linewidth}
\includegraphics[width=\linewidth]{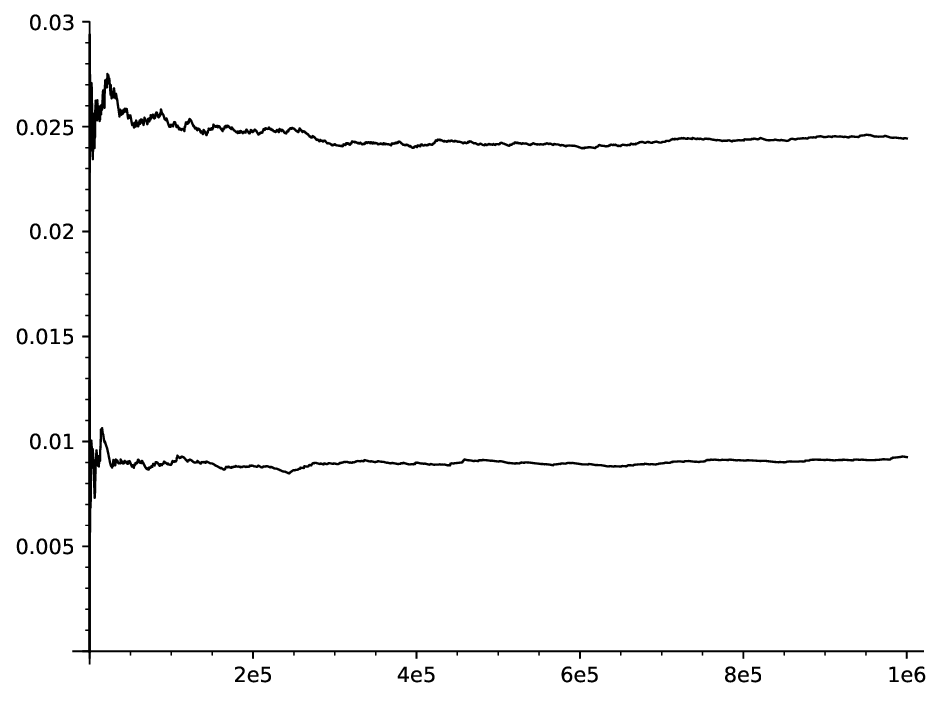}
\caption{$|l| = 4$: Top -4 bottom 4} \label{fig:19_6_even_A_4}
\end{subfigure}\hspace*{\fill}
\begin{subfigure}[b]{0.4\linewidth}
\includegraphics[width=\linewidth]{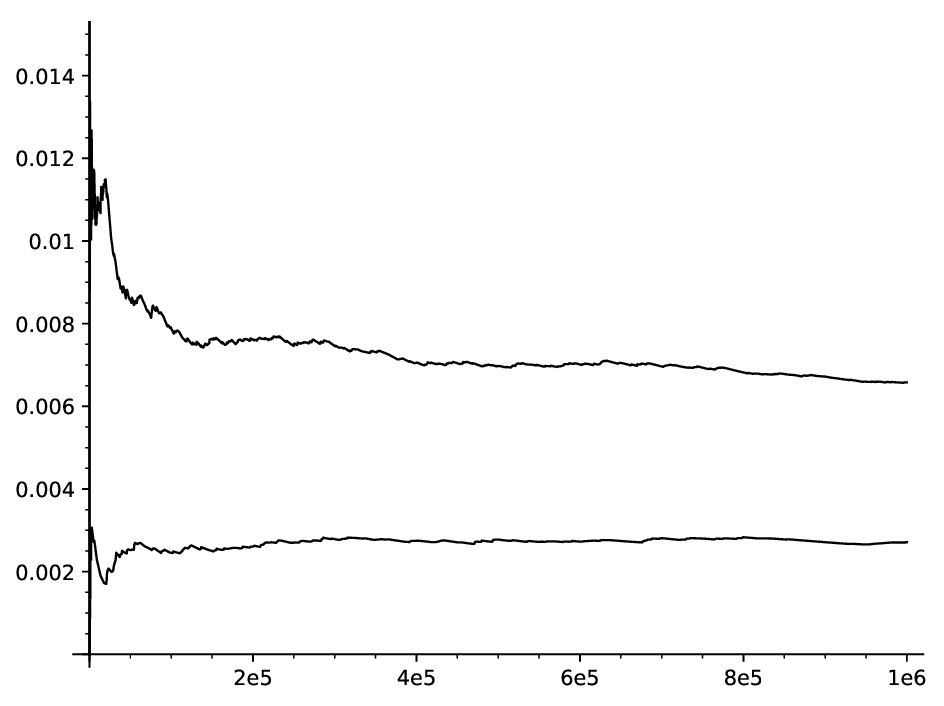}
\caption{$|l| = 5$: Top 5 bottom -5} \label{fig:19_6_even_A_5}
\end{subfigure}\hspace*{\fill}
\begin{subfigure}[b]{0.4\linewidth}
\includegraphics[width=\linewidth]{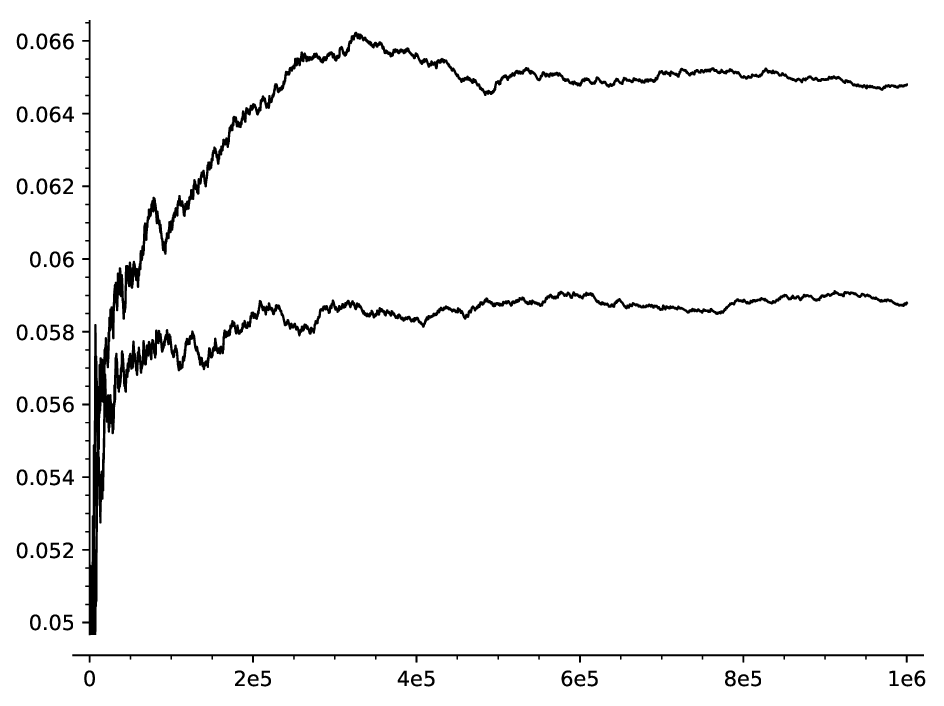}
\caption{$|l| = 6$: Top -6 bottom 6} \label{fig:19_6_even_A_6}
\end{subfigure}
\hspace*{-2.3cm}
\begin{subfigure}[b]{0.4\linewidth}
\includegraphics[width=\linewidth]{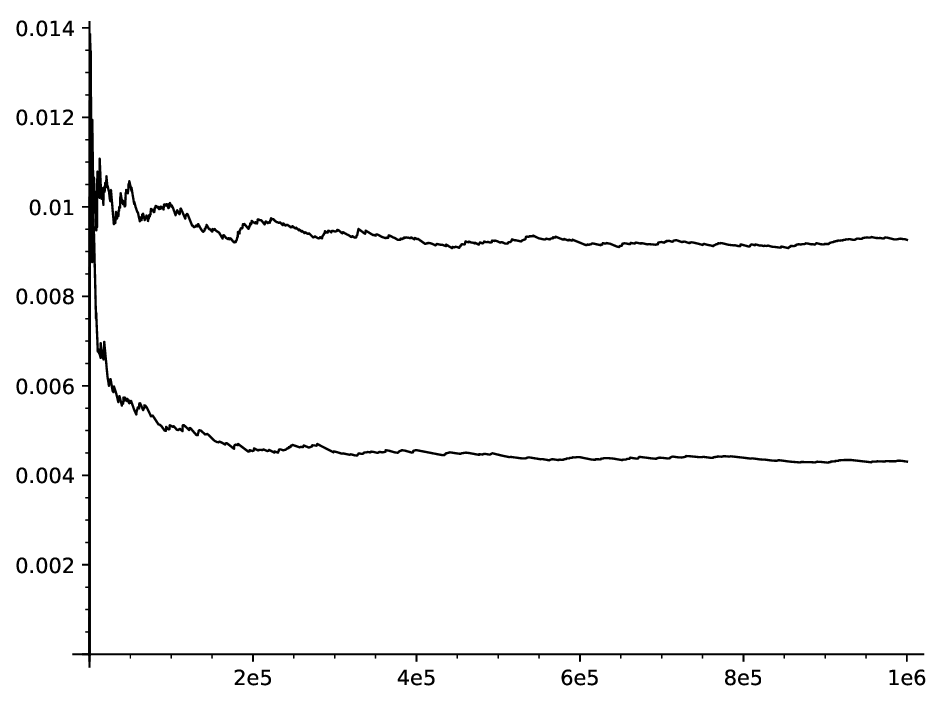}
\caption{$|l| = 7$: Top -7 bottom 7} \label{fig:19_6_even_A_7}
\end{subfigure}\hspace*{\fill}
\begin{subfigure}[b]{0.4\linewidth}
\includegraphics[width=\linewidth]{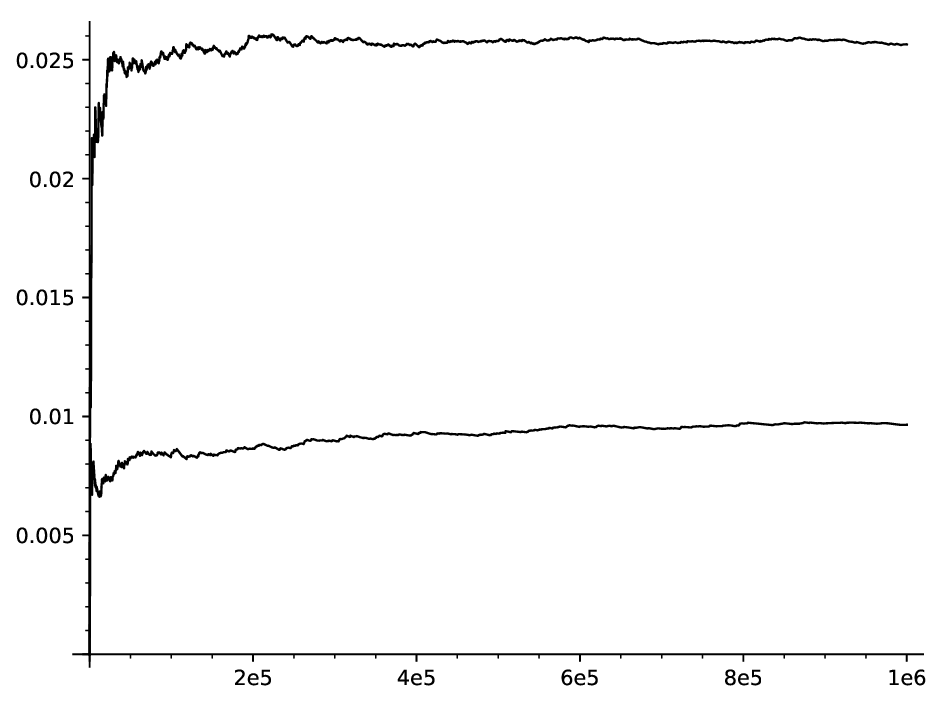}
\caption{$|l| = 8$: Top 8 bottom -8} \label{fig:19_6_even_A_8}
\end{subfigure}\hspace*{\fill}
\begin{subfigure}[b]{0.4\linewidth}
\includegraphics[width=\linewidth]{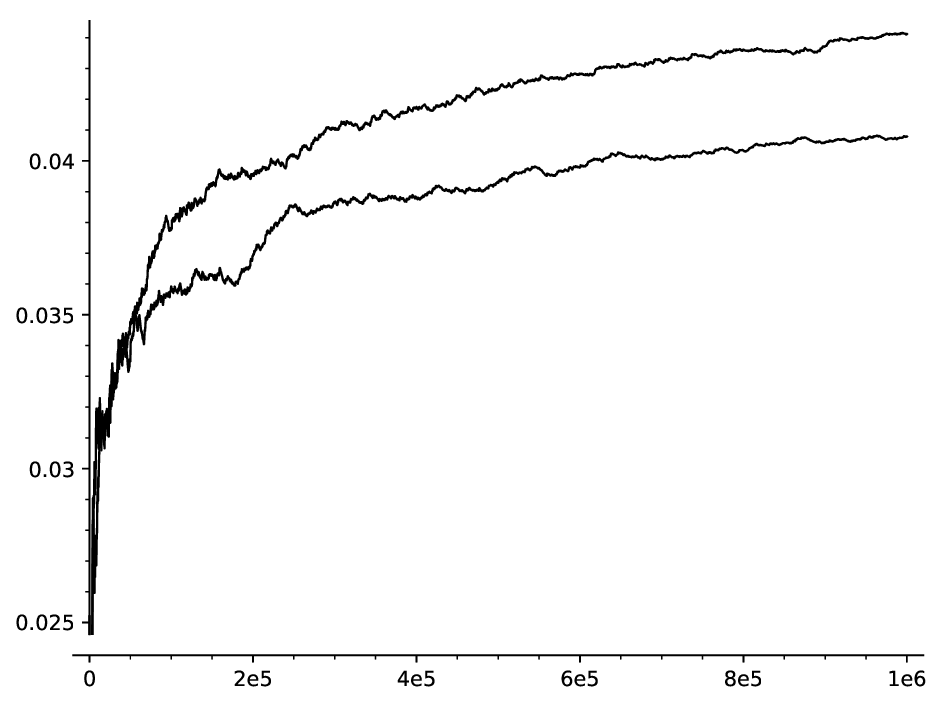}
\caption{$|l| = 9$: Top 9 bottom -9} \label{fig:19_6_even_A_9}
\end{subfigure}
\caption{19a1: Ratio~\eqref{ratio_n_pm} $x_{6,E}^+(X;l)/X^{1/2}\log^2(X)$} \label{fig:19a1_6_even_A_exact}
\end{figure}

\begin{figure}[b] 
\hspace*{-2.3cm}
\begin{subfigure}[b]{0.4\linewidth}
\includegraphics[width=\linewidth]{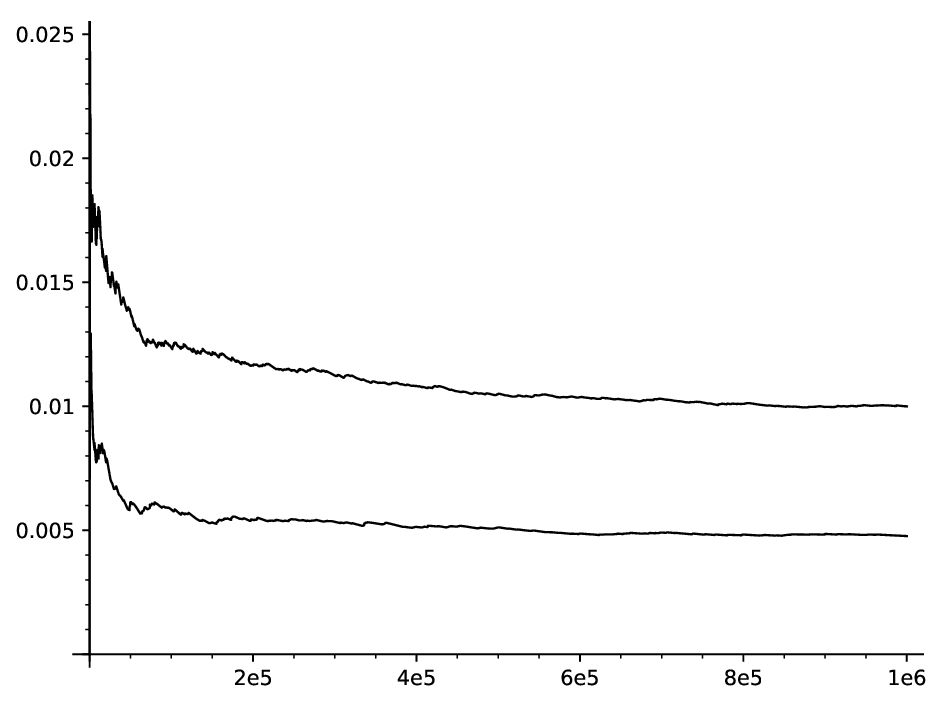}
\caption{$|l| = 1$: Top -1 bottom 1} \label{fig:19_6_odd_A_1}
\end{subfigure}\hspace*{\fill}
\begin{subfigure}[b]{0.4\linewidth}
\includegraphics[width=\linewidth]{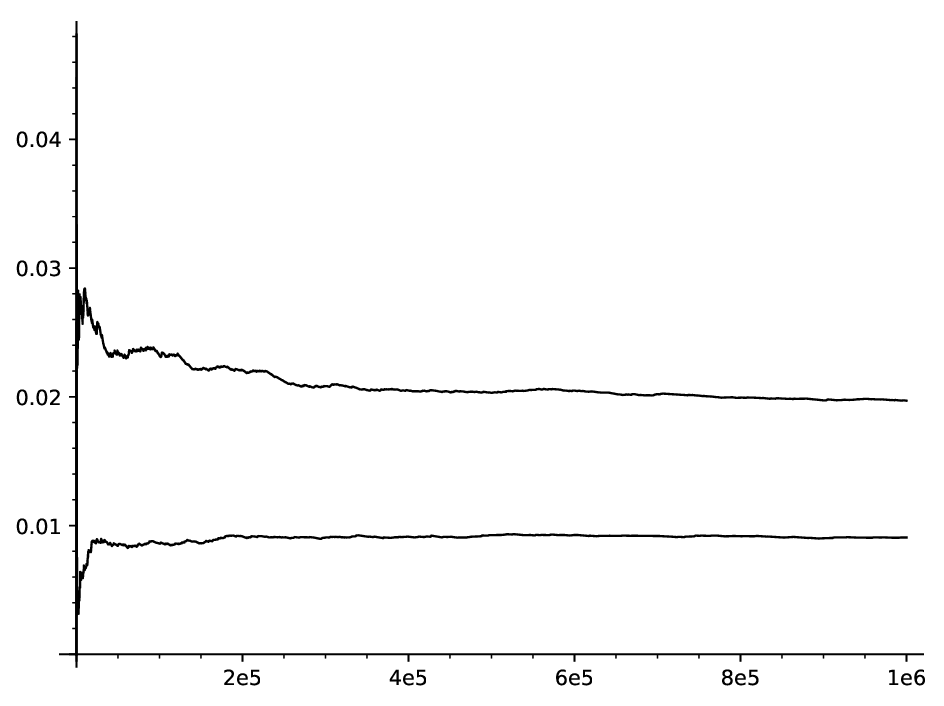}
\caption{$|l| = 2$: Top 2 bottom -2} \label{fig:19_6_odd_A_2}
\end{subfigure}\hspace*{\fill}
\begin{subfigure}[b]{0.4\linewidth}
\includegraphics[width=\linewidth]{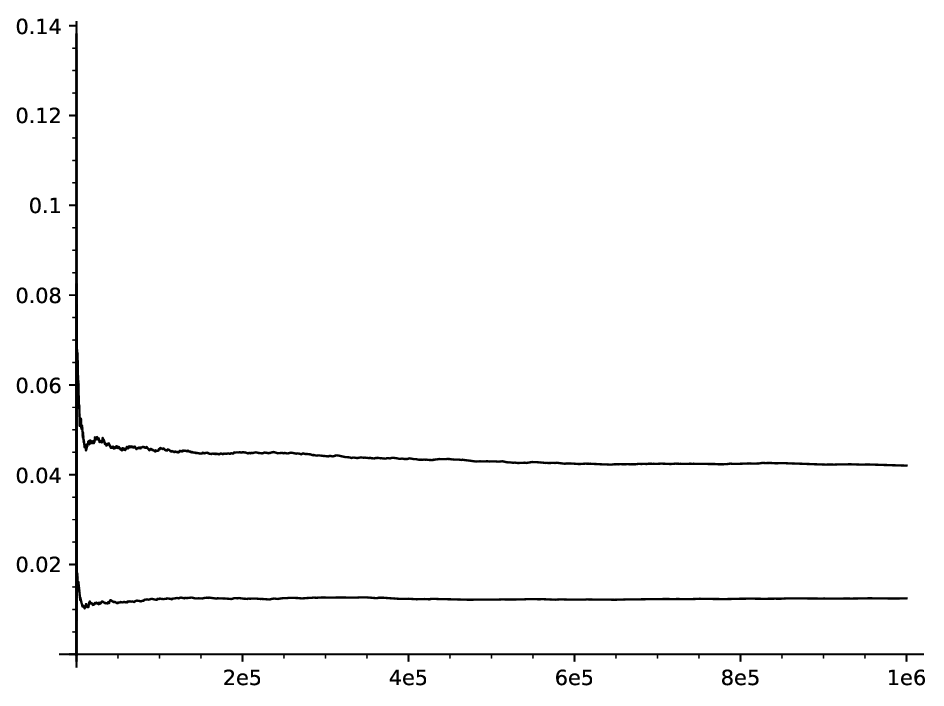}
\caption{$|l| = 3$: Top -3 bottom 3} \label{fig:19_6_odd_A_3}
\end{subfigure}
\hspace*{-2.3cm}
\begin{subfigure}[b]{0.4\linewidth}
\includegraphics[width=\linewidth]{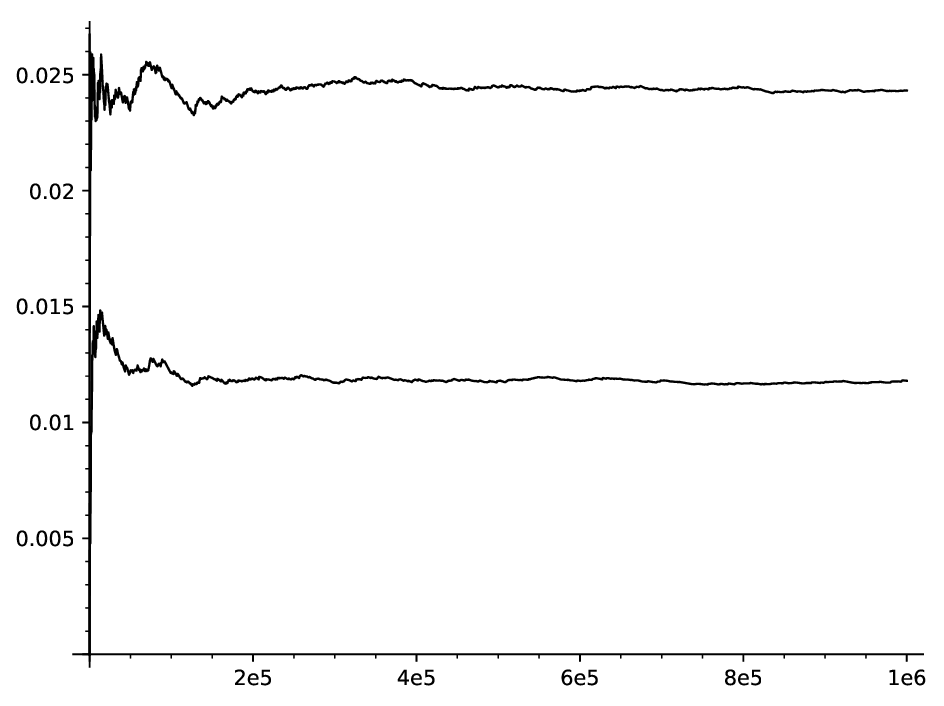}
\caption{$|l| = 4$: Top -4 bottom 4} \label{fig:19_6_odd_A_4}
\end{subfigure}\hspace*{\fill}
\begin{subfigure}[b]{0.4\linewidth}
\includegraphics[width=\linewidth]{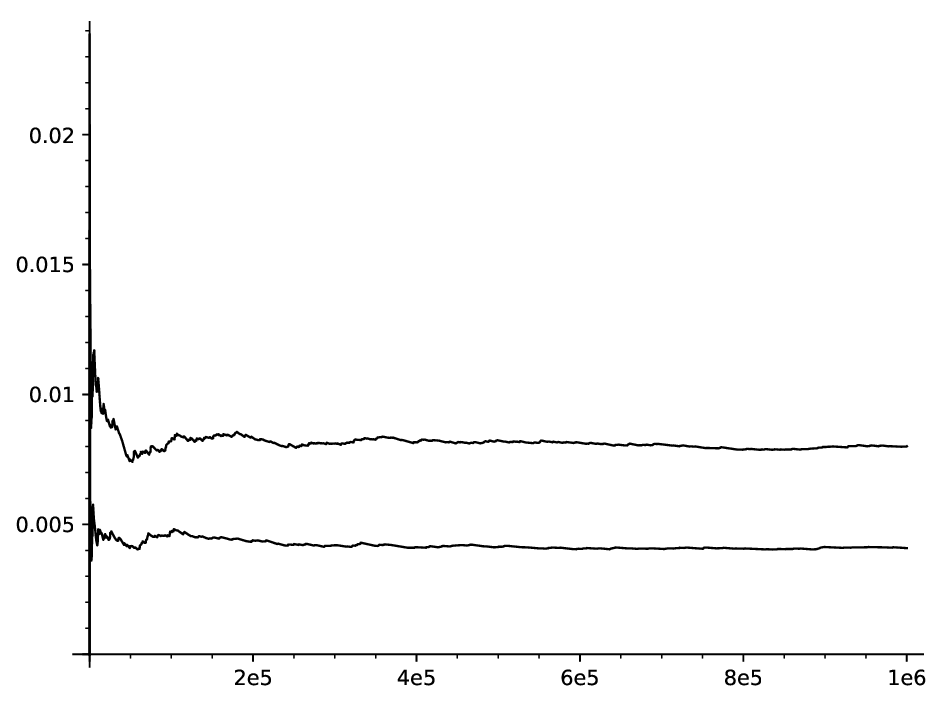}
\caption{$|l| = 5$: Top 5 bottom -5} \label{fig:19_6_odd_A_5}
\end{subfigure}\hspace*{\fill}
\begin{subfigure}[b]{0.4\linewidth}
\includegraphics[width=\linewidth]{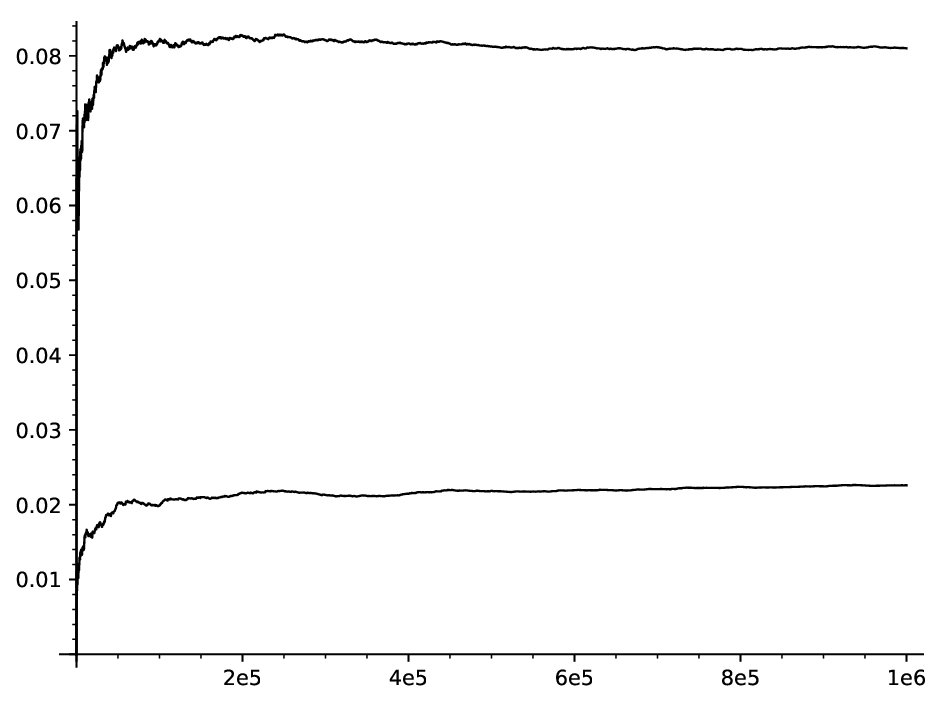}
\caption{$|l| = 6$: Top 6 bottom -6} \label{fig:19_6_odd_A_6}
\end{subfigure}
\hspace*{-2.3cm}
\begin{subfigure}[b]{0.4\linewidth}
\includegraphics[width=\linewidth]{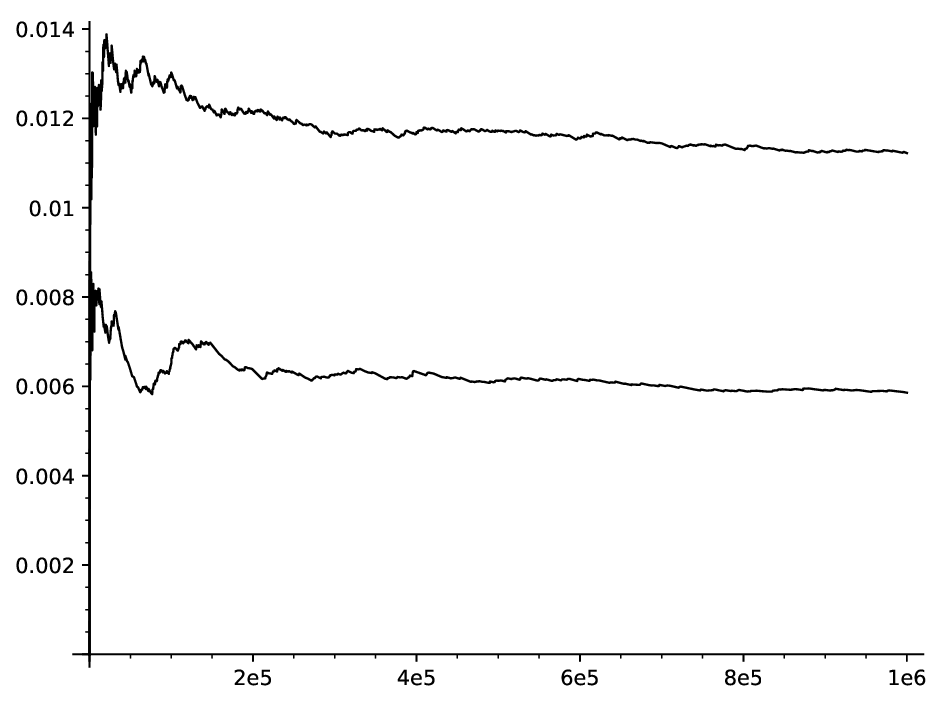}
\caption{$|l| = 7$: Top -7 bottom 7} \label{fig:19_6_odd_A_7}
\end{subfigure}\hspace*{\fill}
\begin{subfigure}[b]{0.4\linewidth}
\includegraphics[width=\linewidth]{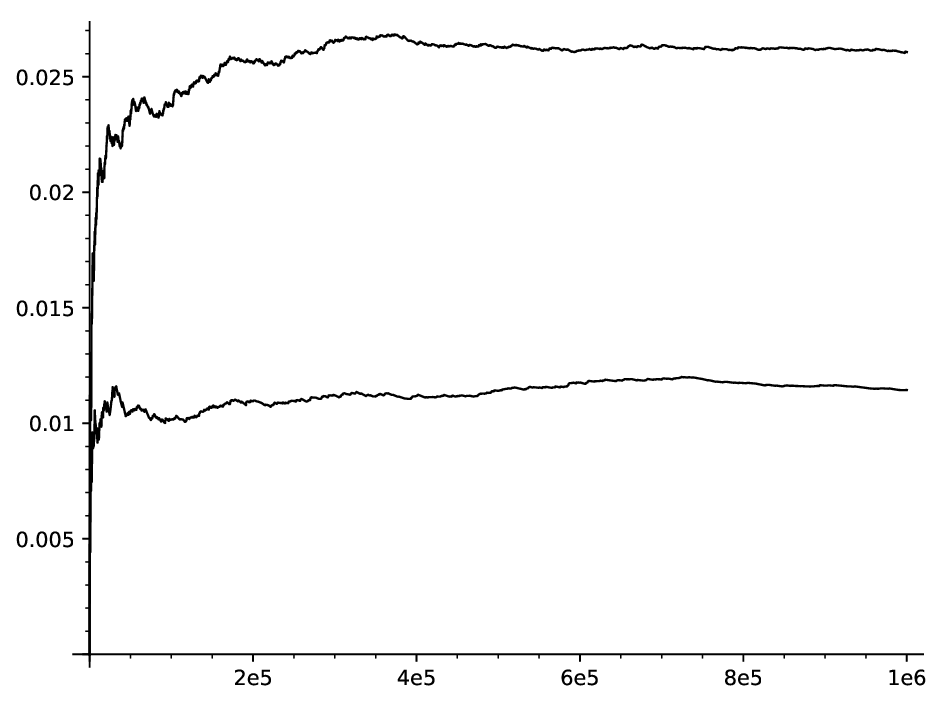}
\caption{$|l| = 8$: Top 8 bottom -8} \label{fig:19_6_od_A_8}
\end{subfigure}\hspace*{\fill}
\begin{subfigure}[b]{0.4\linewidth}
\includegraphics[width=\linewidth]{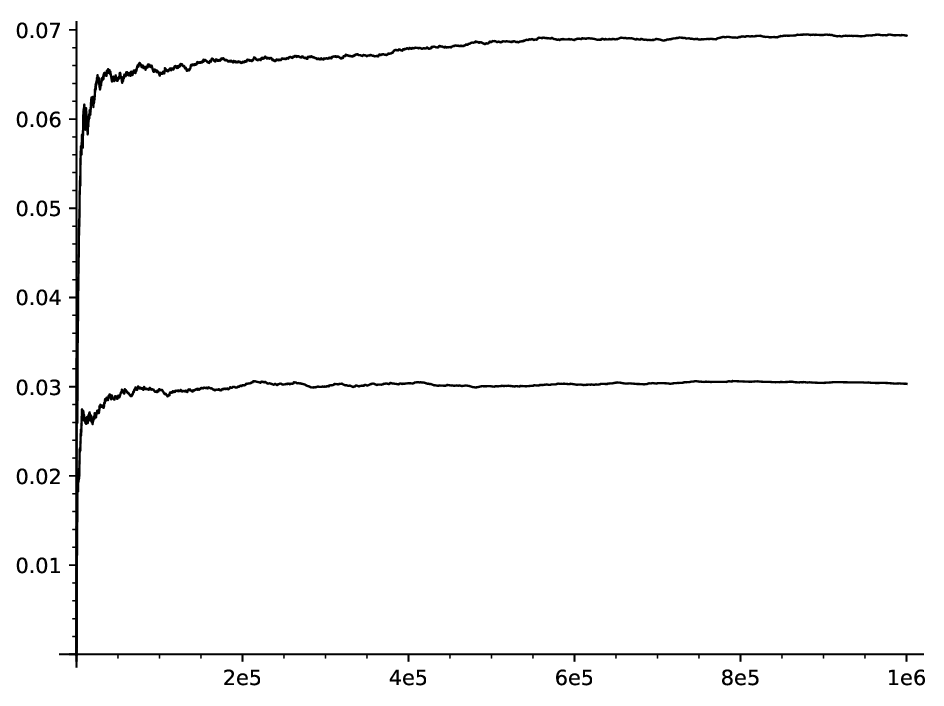}
\caption{$|l| = 9$: Top -9 bottom 9} \label{fig:19_6_odd_A_9}
\end{subfigure}
\caption{19a1: Ratio~\eqref{ratio_n_pm} $x_{6,E}^-(X;l)/X^{1/2}\log^2(X)$} \label{fig:19a1_6_odd_A_exact}
\end{figure}

\clearpage

\begin{figure}[t] 
\hspace*{-2.3cm}
\begin{subfigure}[b]{0.4\linewidth}
\includegraphics[width=\linewidth]{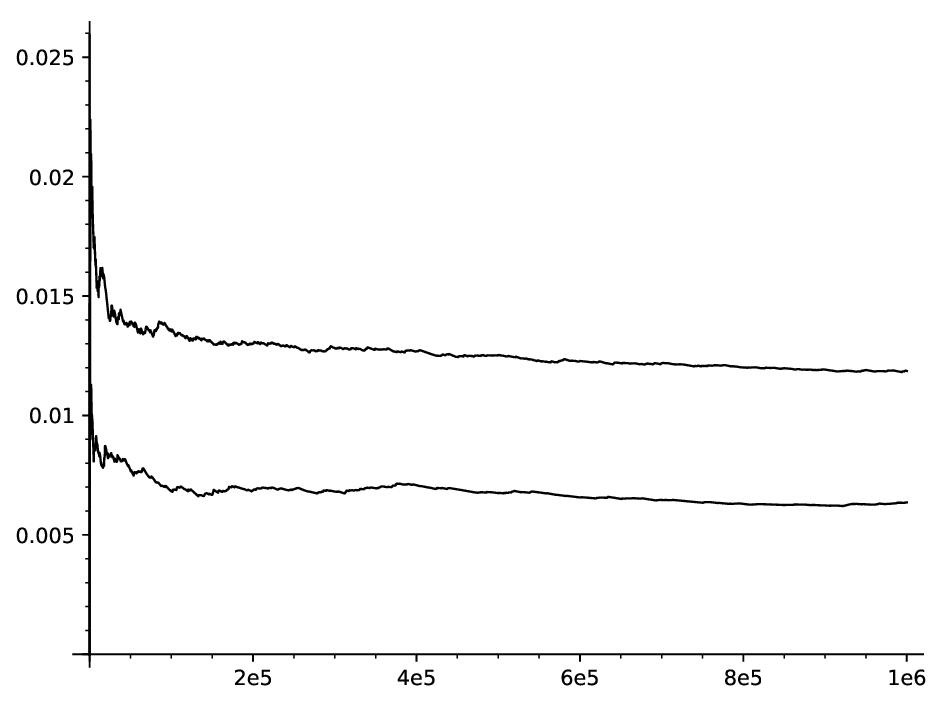}
\caption{$|l| = 1$: Top -1 bottom 1} \label{fig:37_6_even_A_1}
\end{subfigure}\hspace*{\fill}
\begin{subfigure}[b]{0.4\linewidth}
\includegraphics[width=\linewidth]{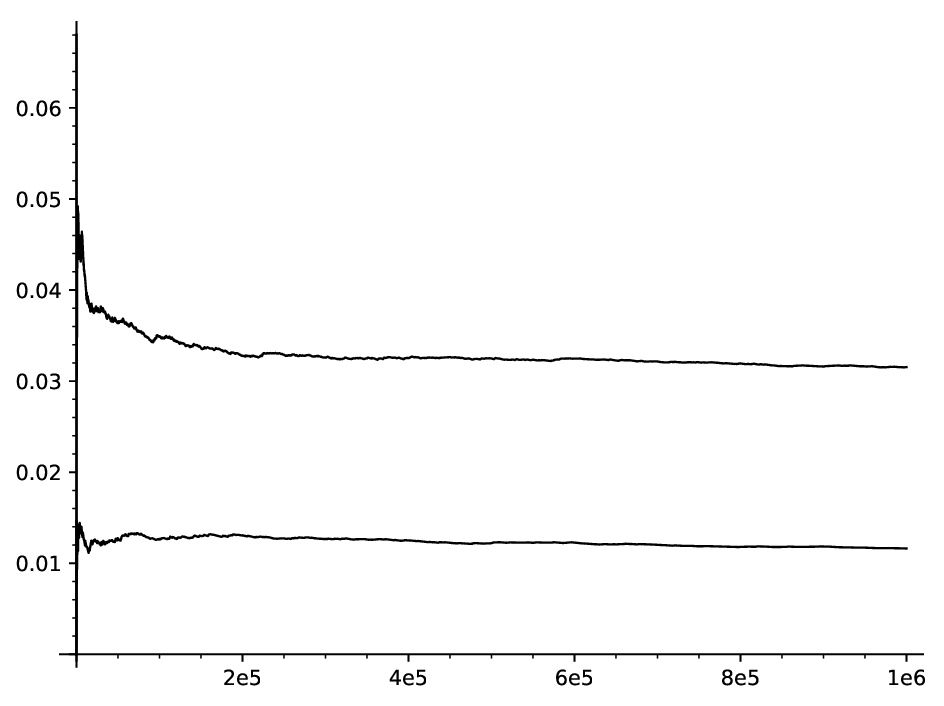}
\caption{$|l| = 2$: Top 2 bottom -2} \label{fig:37_6_even_A_2}
\end{subfigure}\hspace*{\fill}
\begin{subfigure}[b]{0.4\linewidth}
\includegraphics[width=\linewidth]{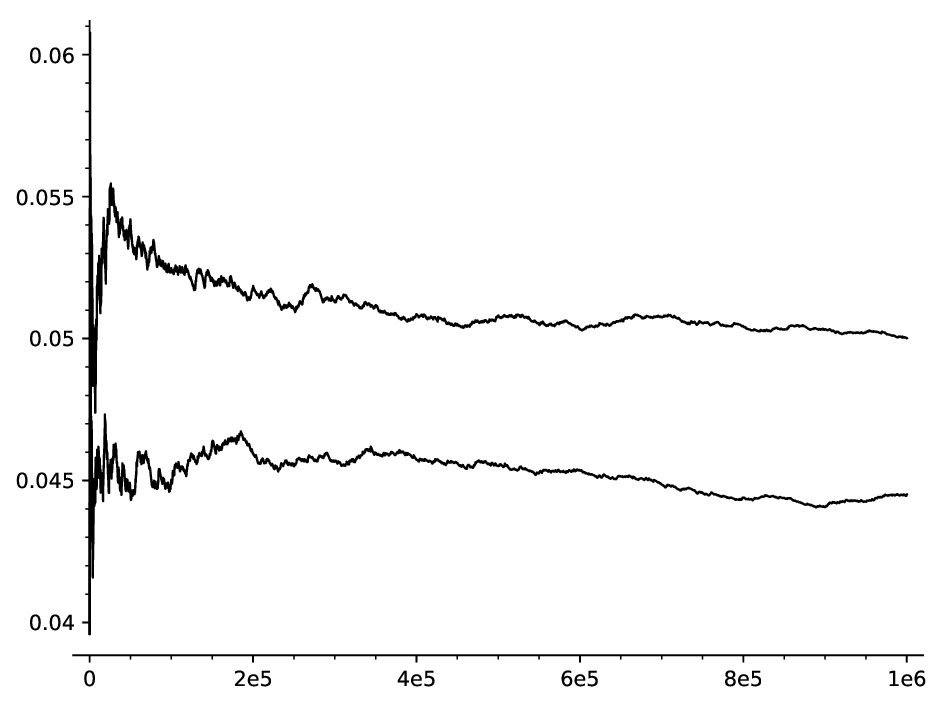}
\caption{$|l| = 3$: Top 3 bottom -3} \label{fig:37_6_even_A_3}
\end{subfigure}
\hspace*{-2.3cm}
\begin{subfigure}[b]{0.4\linewidth}
\includegraphics[width=\linewidth]{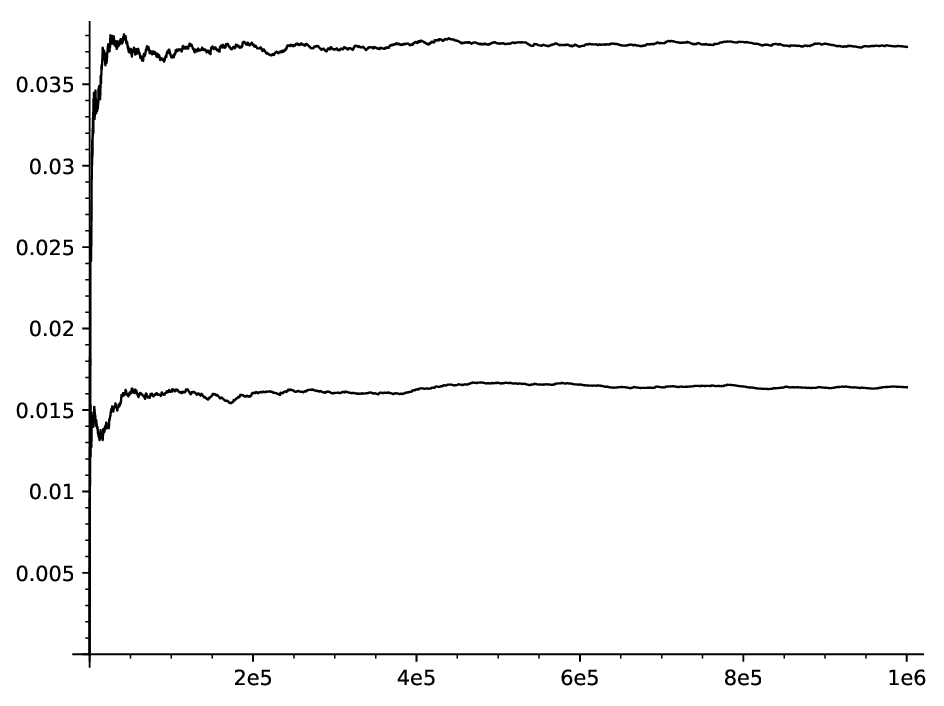}
\caption{$|l| = 4$: Top -4 bottom 4} \label{fig:37_6_even_A_4}
\end{subfigure}\hspace*{\fill}
\begin{subfigure}[b]{0.4\linewidth}
\includegraphics[width=\linewidth]{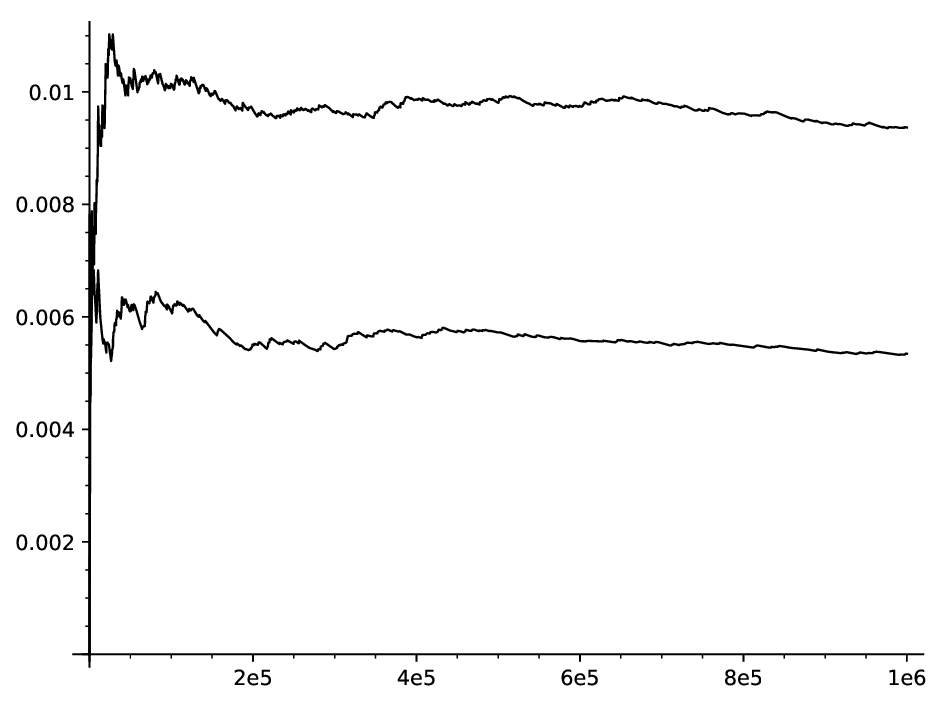}
\caption{$|l| = 5$: Top 5 bottom -5} \label{fig:37_6_even_A_5}
\end{subfigure}\hspace*{\fill}
\begin{subfigure}[b]{0.4\linewidth}
\includegraphics[width=\linewidth]{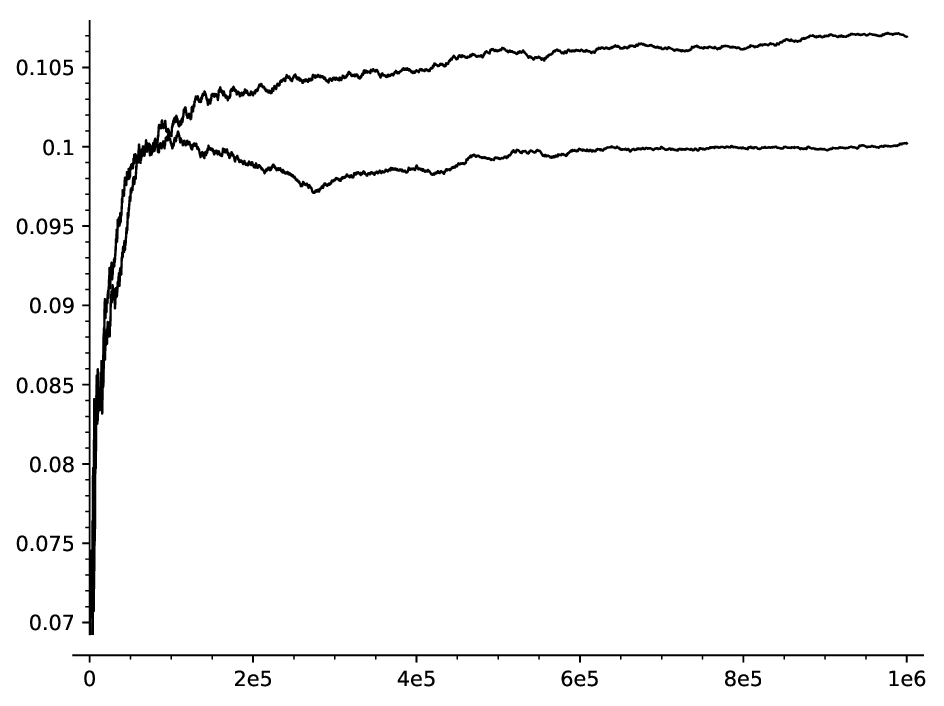}
\caption{$|l| = 6$: Top -6 bottom 6} \label{fig:37_6_even_A_6}
\end{subfigure}
\hspace*{-2.3cm}
\begin{subfigure}[b]{0.4\linewidth}
\includegraphics[width=\linewidth]{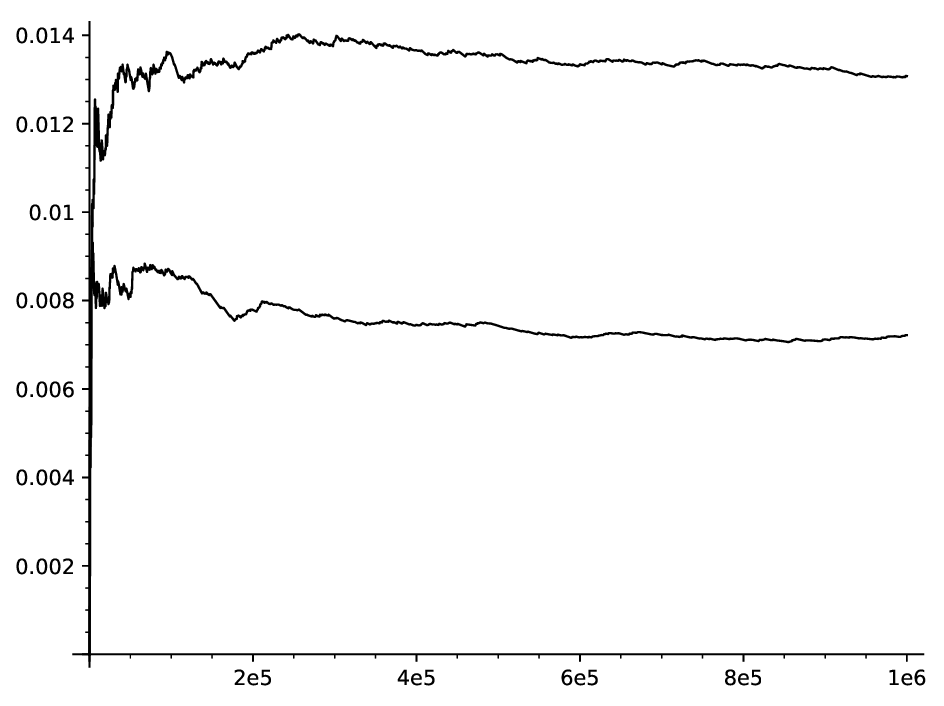}
\caption{$|l| = 7$: Top -7 bottom 7} \label{fig:37_6_even_A_7}
\end{subfigure}\hspace*{\fill}
\begin{subfigure}[b]{0.4\linewidth}
\includegraphics[width=\linewidth]{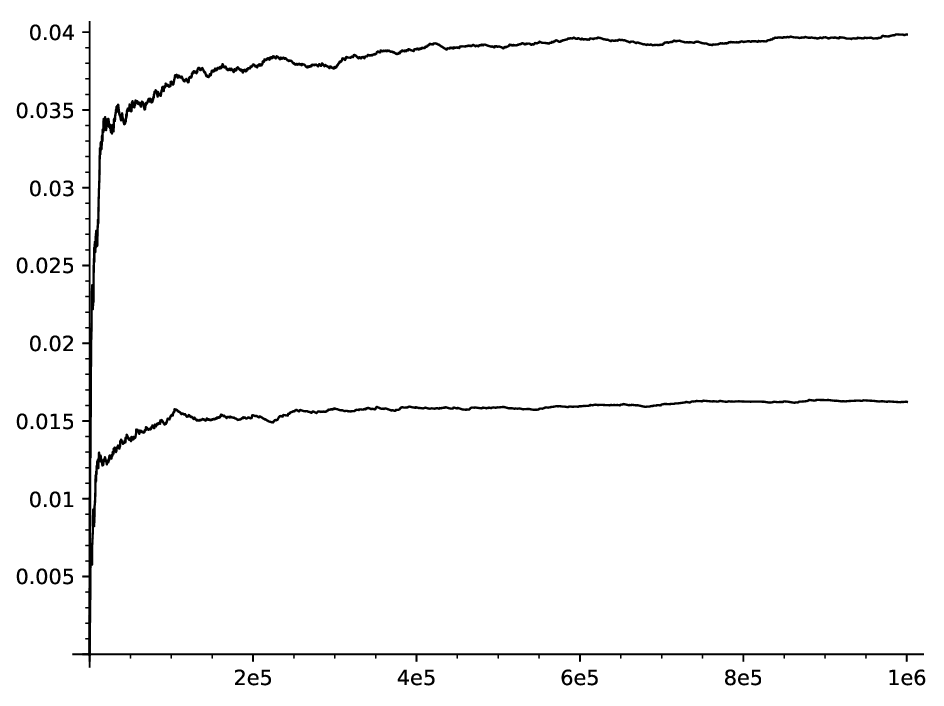}
\caption{$|l| = 8$: Top 8 bottom -8} \label{fig:37_6_even_A_8}
\end{subfigure}\hspace*{\fill}
\begin{subfigure}[b]{0.4\linewidth}
\includegraphics[width=\linewidth]{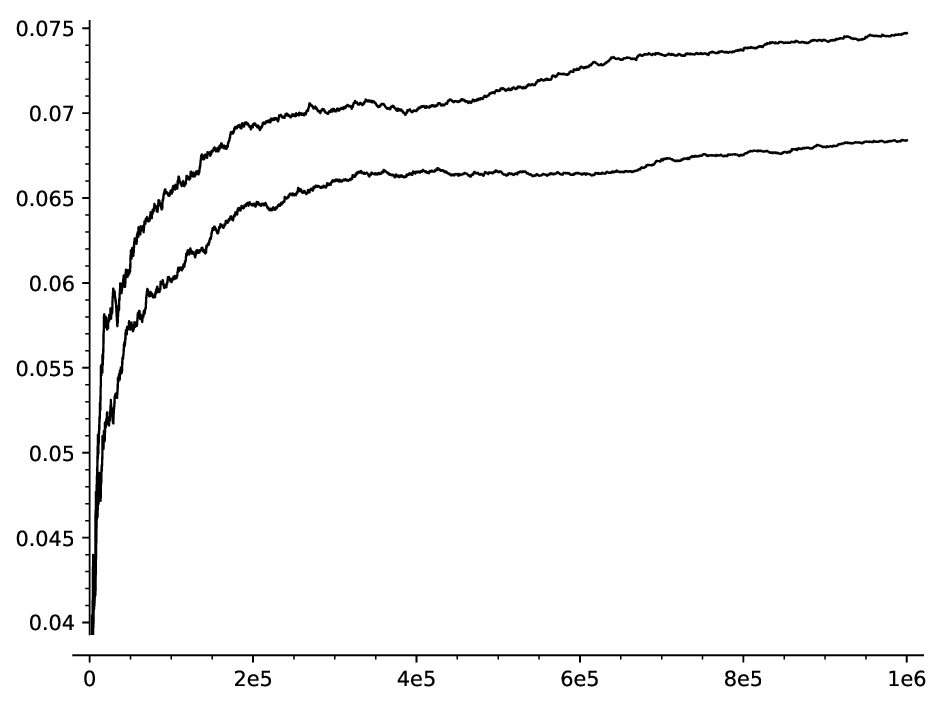}
\caption{$|l| = 9$: Top 9 bottom -9} \label{fig:37_6_even_A_9}
\end{subfigure}
\caption{37b1: Ratio~\eqref{ratio_n_pm} $x_{6,E}^+(X;l)/X^{1/2}\log^2(X)$} \label{fig:37b1_6_even_A_exact}
\end{figure}

\begin{figure}[b] 
\hspace*{-2.3cm}
\begin{subfigure}[b]{0.4\linewidth}
\includegraphics[width=\linewidth]{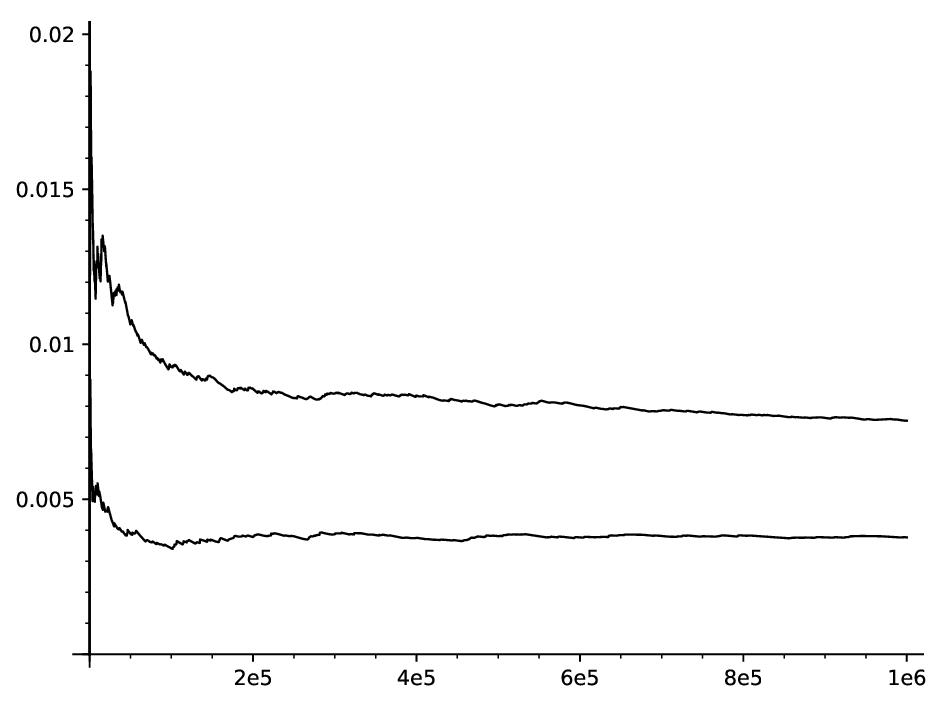}
\caption{$|l| = 1$: Top 1 bottom -1} \label{fig:37_6_odd_A_1}
\end{subfigure}\hspace*{\fill}
\begin{subfigure}[b]{0.4\linewidth}
\includegraphics[width=\linewidth]{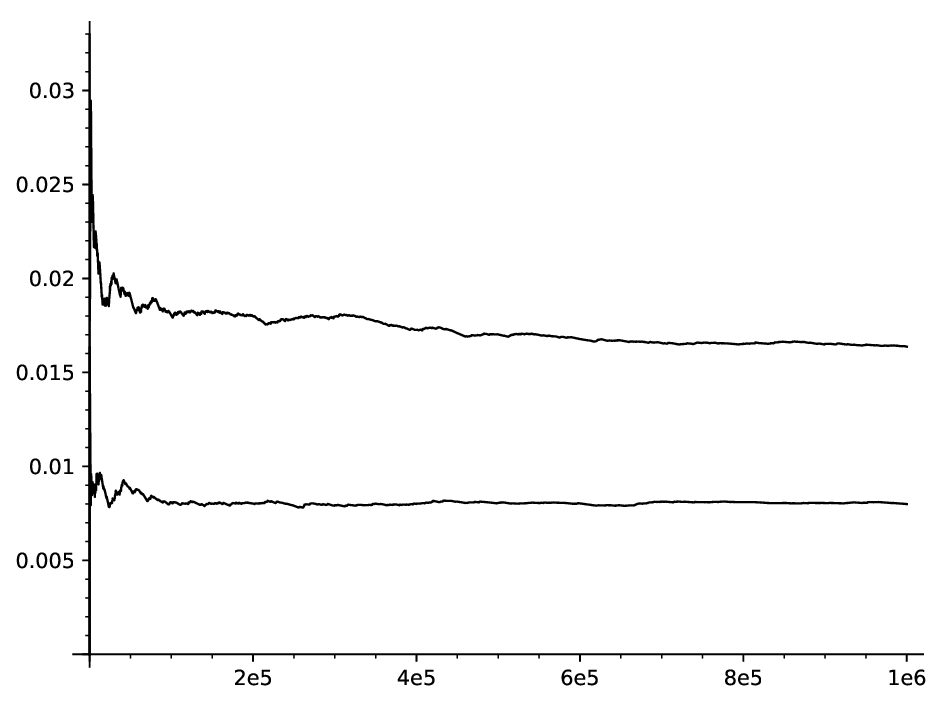}
\caption{$|l| = 2$: Top -2 bottom 2} \label{fig:37_6_odd_A_2}
\end{subfigure}\hspace*{\fill}
\begin{subfigure}[b]{0.4\linewidth}
\includegraphics[width=\linewidth]{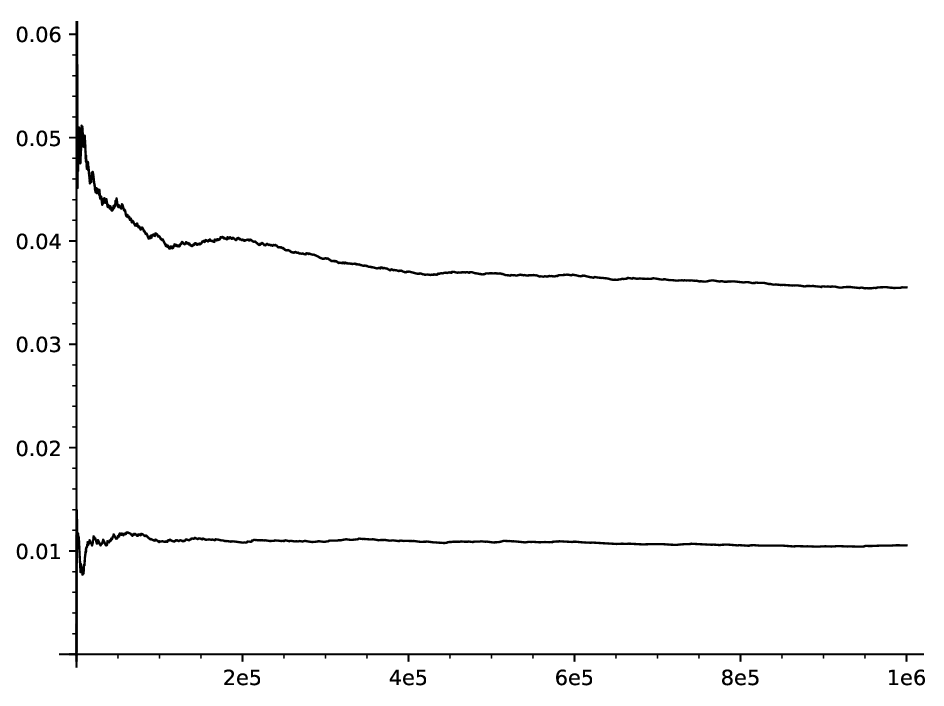}
\caption{$|l| = 3$: Top 3 bottom -3} \label{fig:37_6_odd_A_3}
\end{subfigure}
\hspace*{-2.3cm}
\begin{subfigure}[b]{0.4\linewidth}
\includegraphics[width=\linewidth]{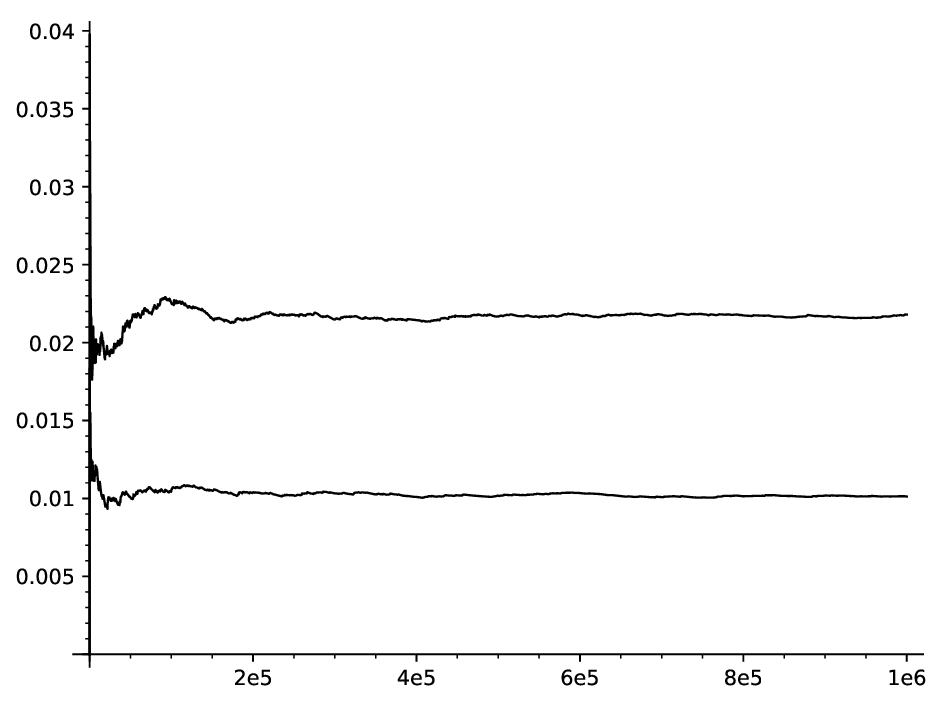}
\caption{$|l| = 4$: Top 4 bottom -4} \label{fig:37_6_odd_A_4}
\end{subfigure}\hspace*{\fill}
\begin{subfigure}[b]{0.4\linewidth}
\includegraphics[width=\linewidth]{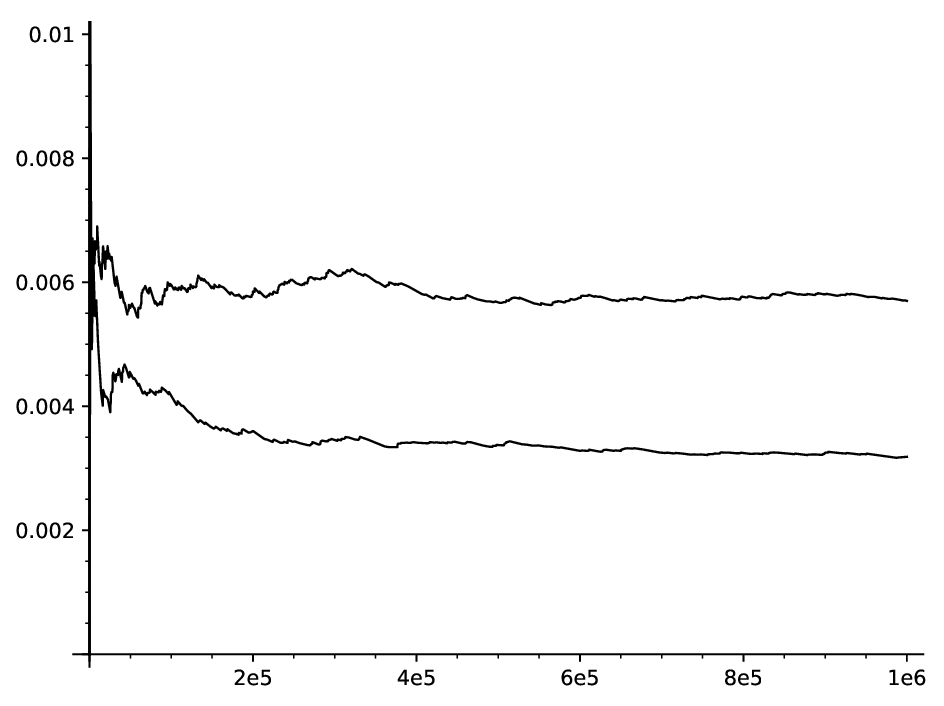}
\caption{$|l| = 5$: Top -5 bottom 5} \label{fig:37_6_odd_A_5}
\end{subfigure}\hspace*{\fill}
\begin{subfigure}[b]{0.4\linewidth}
\includegraphics[width=\linewidth]{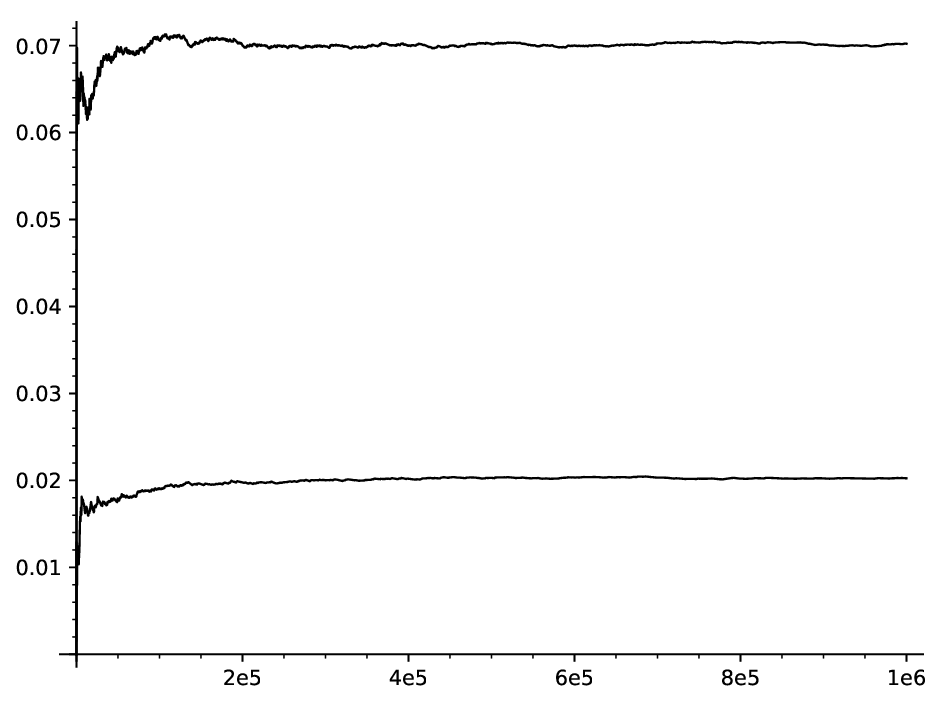}
\caption{$|l| = 6$: Top -6 bottom 6} \label{fig:37_6_odd_A_6}
\end{subfigure}
\hspace*{-2.3cm}
\begin{subfigure}[b]{0.4\linewidth}
\includegraphics[width=\linewidth]{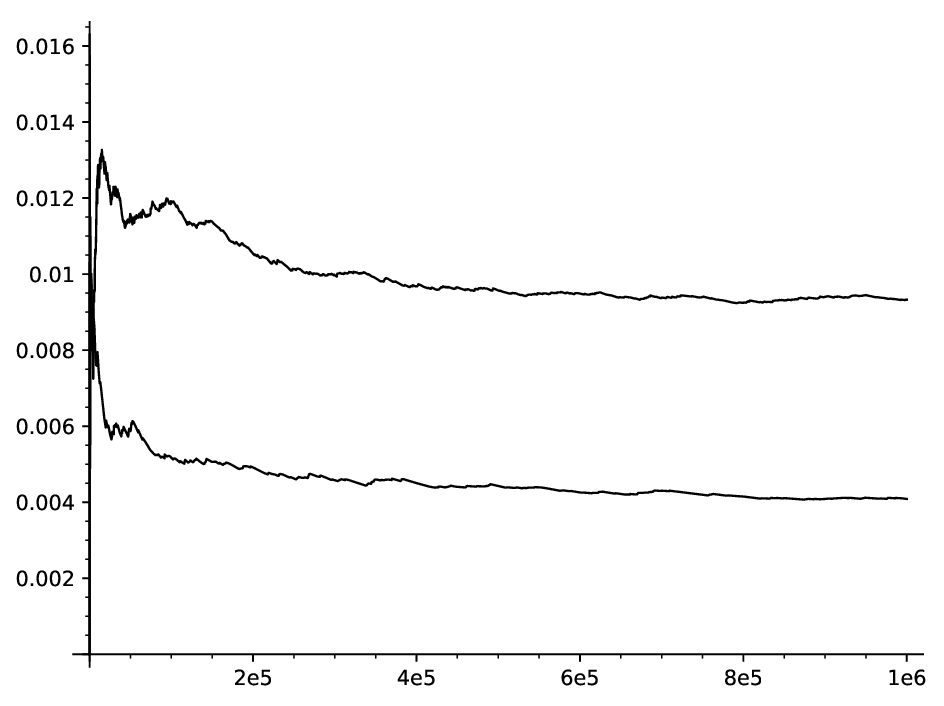}
\caption{$|l| = 7$: Top 7 bottom -7} \label{fig:37_6_odd_A_7}
\end{subfigure}\hspace*{\fill}
\begin{subfigure}[b]{0.4\linewidth}
\includegraphics[width=\linewidth]{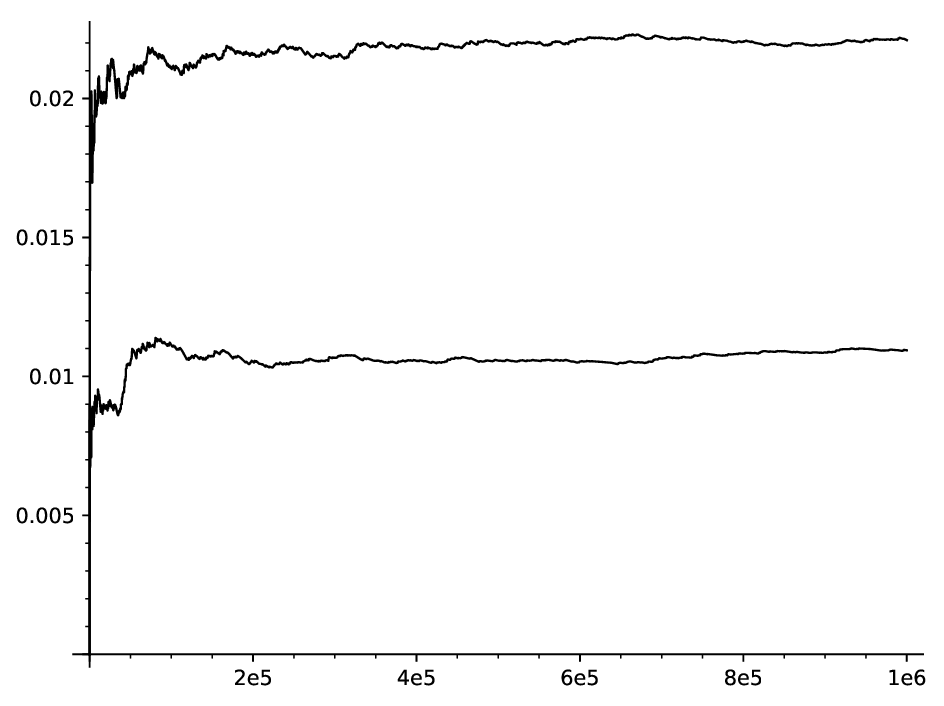}
\caption{$|l| = 8$: Top -8 bottom 8} \label{fig:37_6_od_A_8}
\end{subfigure}\hspace*{\fill}
\begin{subfigure}[b]{0.4\linewidth}
\includegraphics[width=\linewidth]{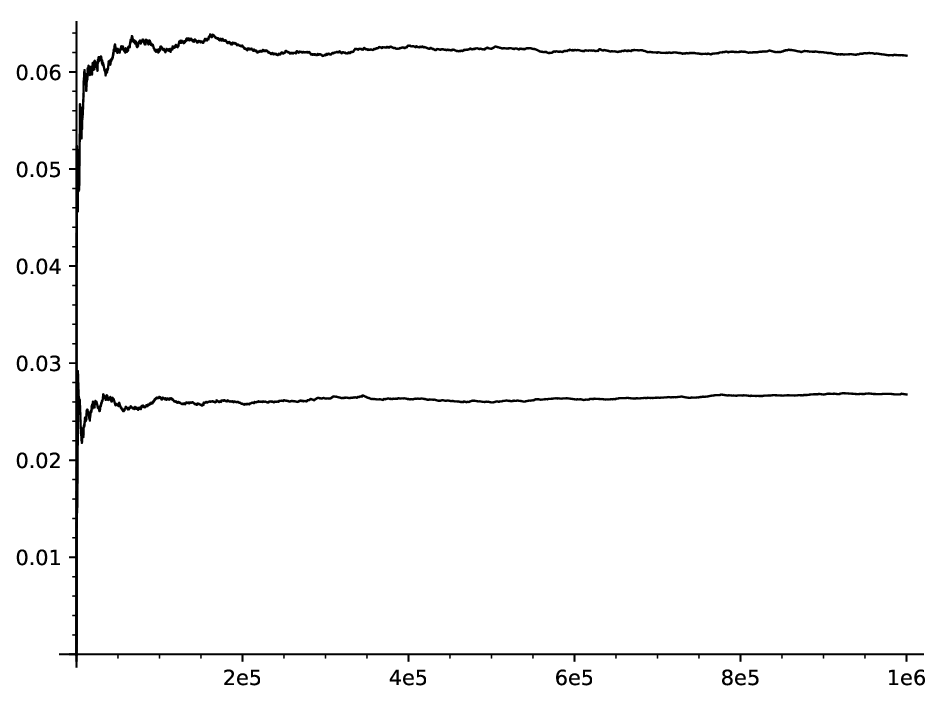}
\caption{$|l| = 9$: Top 9 bottom -9} \label{fig:37_6_odd_A_9}
\end{subfigure}
\caption{37b1: Ratio~\eqref{ratio_n_pm} $x_{6,E}^-(X;l)/X^{1/2}\log^2(X)$} \label{fig:37b1_6_odd_A_exact}
\end{figure}

\clearpage

\begin{figure}[b!] 
\hspace*{-.7cm}
\begin{subfigure}[b]{0.4\linewidth}
\includegraphics[width=\linewidth]{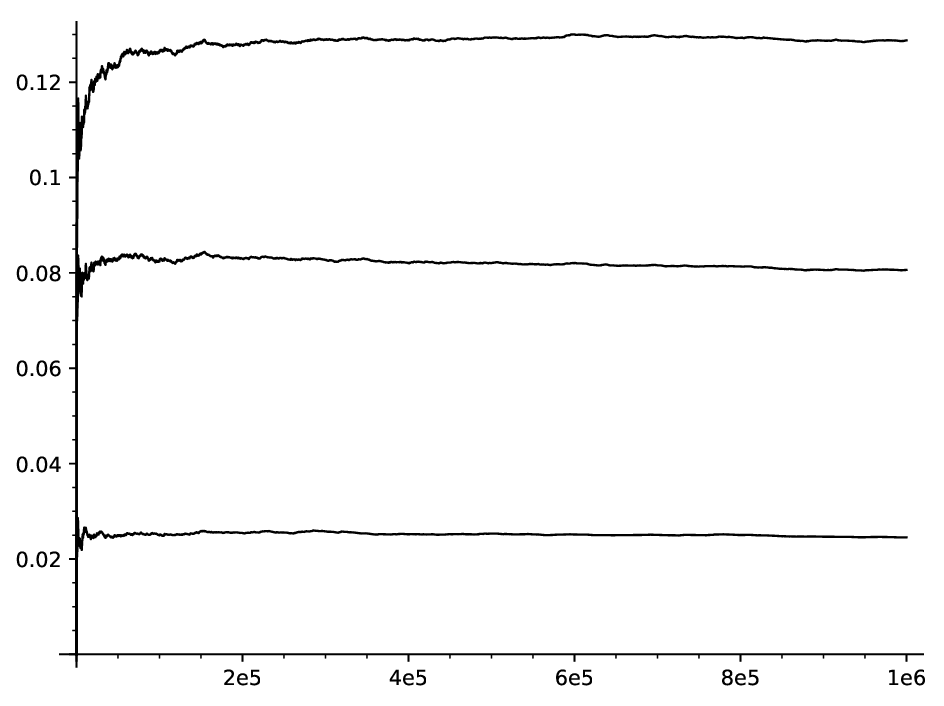}
\caption{11a1: $(\alpha, \beta) = (1,3)$} \label{fig:11_6_1_3_acc_A_orders}
\end{subfigure}\hspace*{\fill}
\begin{subfigure}[b]{0.4\linewidth}
\includegraphics[width=\linewidth]{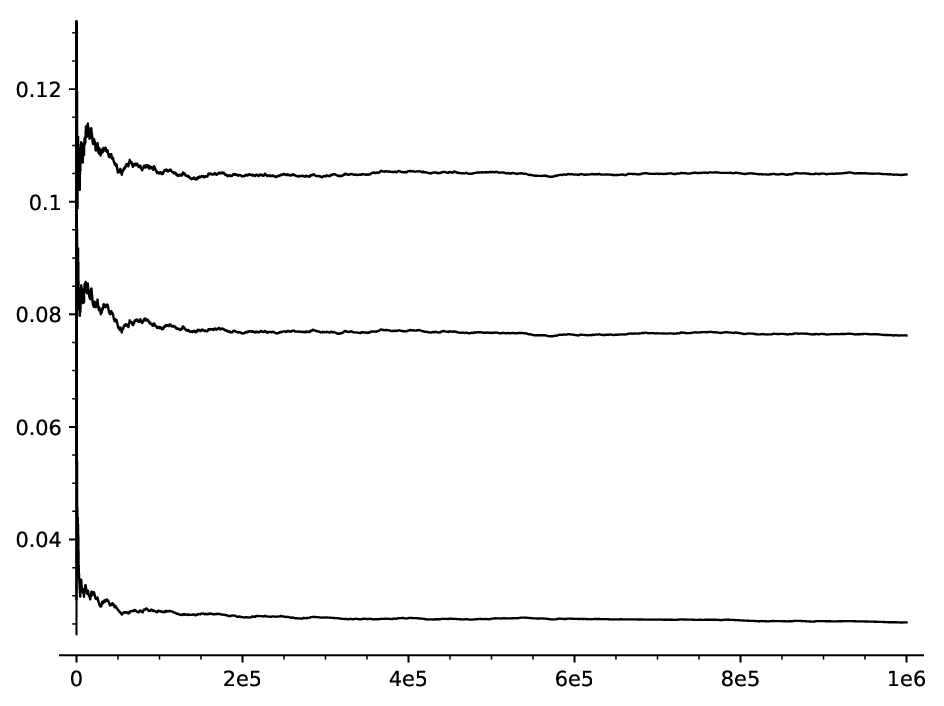}
\caption{11a1: $(\alpha, \beta) = (2,3)$} \label{fig:11_6_2_3_acc_A_orders}
\end{subfigure}
\hspace*{-.7cm}
\begin{subfigure}[b]{0.4\linewidth}
\includegraphics[width=\linewidth]{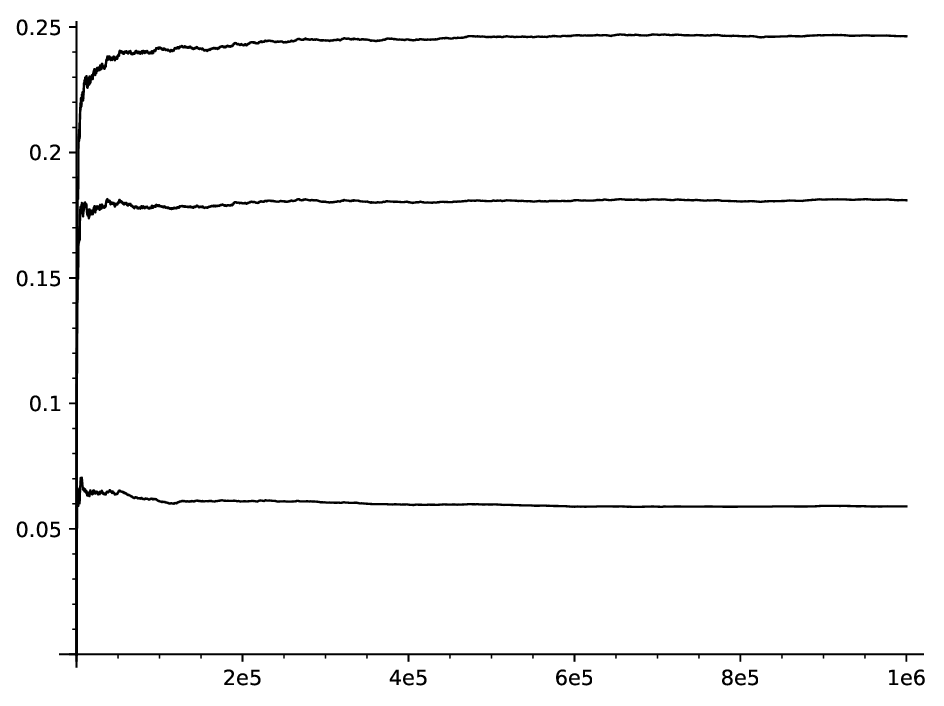}
\caption{11a1: $(\alpha, \beta) = (1,6)$} \label{fig:11_6_1_6_acc_A_orders}
\end{subfigure}\hspace*{\fill}
\begin{subfigure}[b]{0.4\linewidth}
\includegraphics[width=\linewidth]{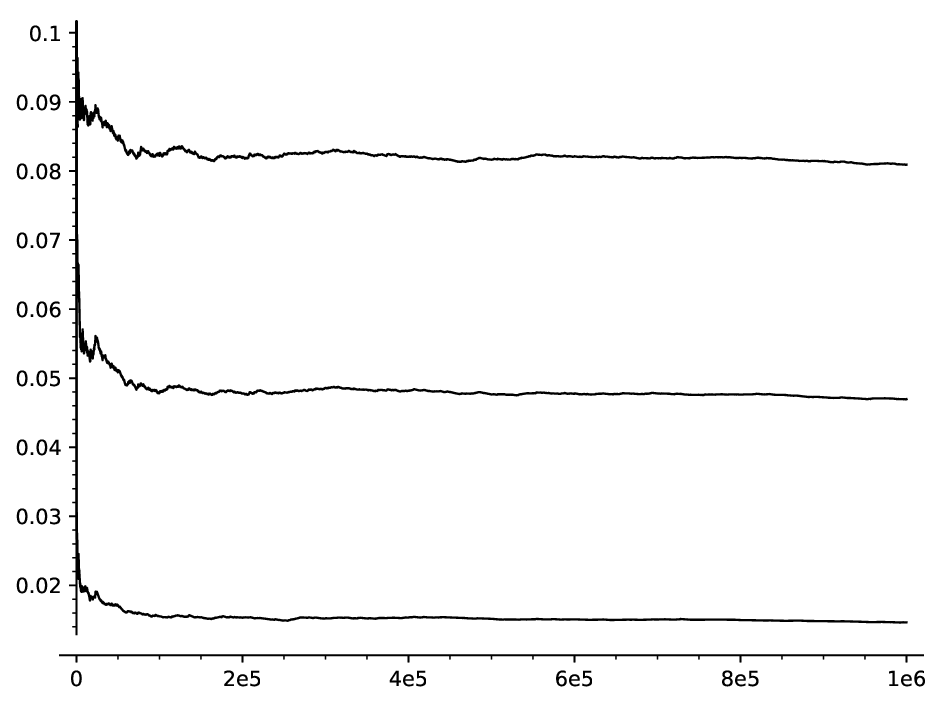}
\caption{11a1: $(\alpha, \beta) = (2,6)$} \label{fig:11_6_2_6_acc_A_orders}
\end{subfigure}
\hspace*{-.7cm}
\begin{subfigure}[b]{0.4\linewidth}
\includegraphics[width=\linewidth]{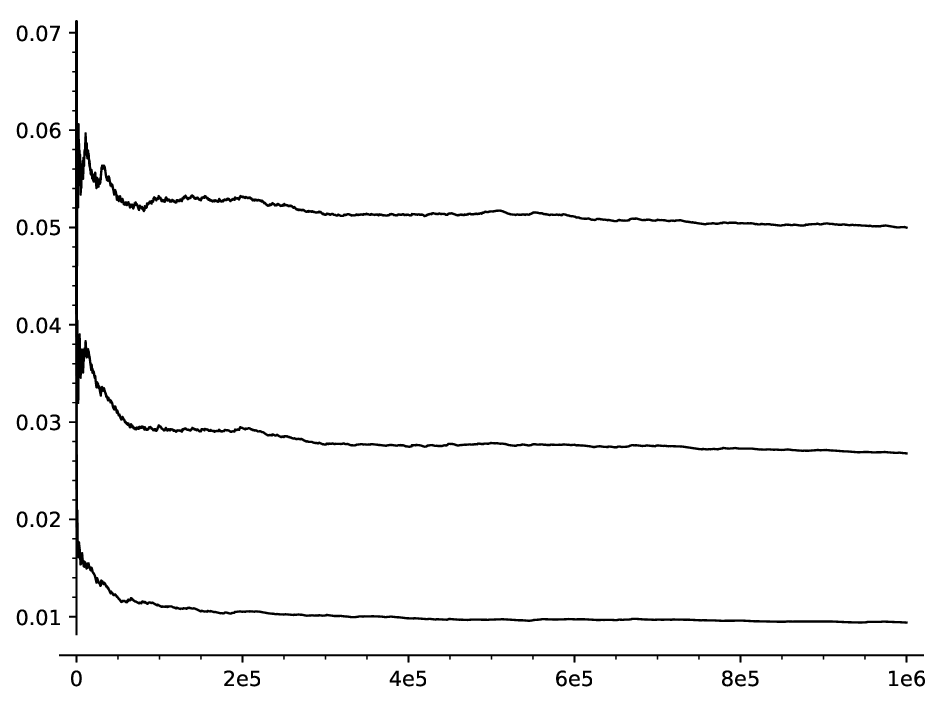}
\caption{14a1: $(\alpha, \beta) = (1,3)$} \label{fig:14_6_1_3_acc_A_orders}
\end{subfigure}\hspace*{\fill}
\begin{subfigure}[b]{0.4\linewidth}
\includegraphics[width=\linewidth]{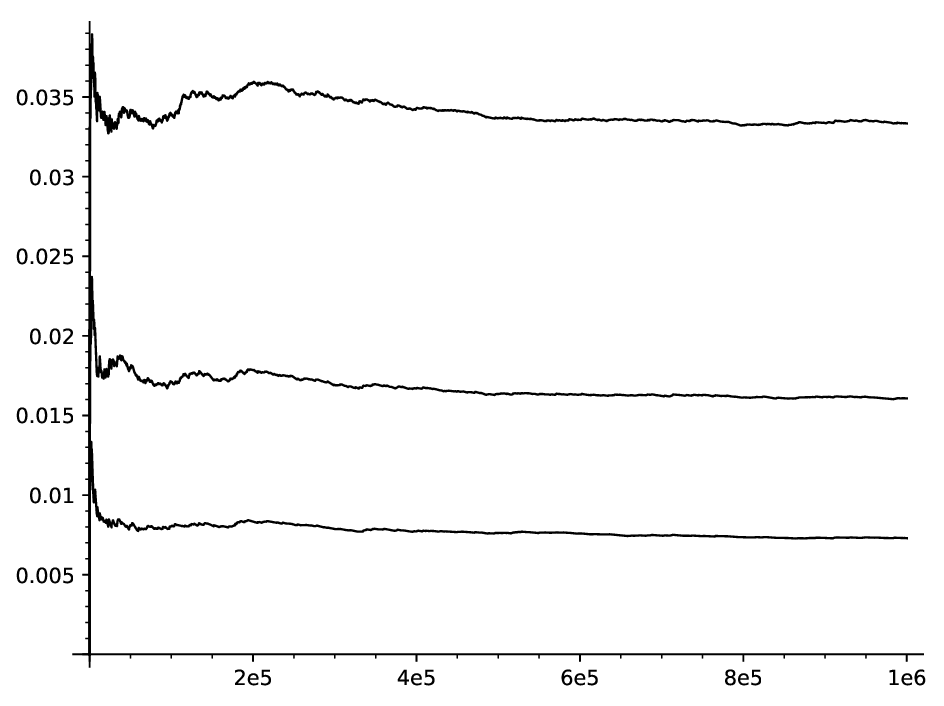}
\caption{14a1: $(\alpha, \beta) = (2,3)$} \label{fig:14_6_2_3_acc_A_orders}
\end{subfigure}
\hspace*{-.7cm}
\begin{subfigure}[b]{0.4\linewidth}
\includegraphics[width=\linewidth]{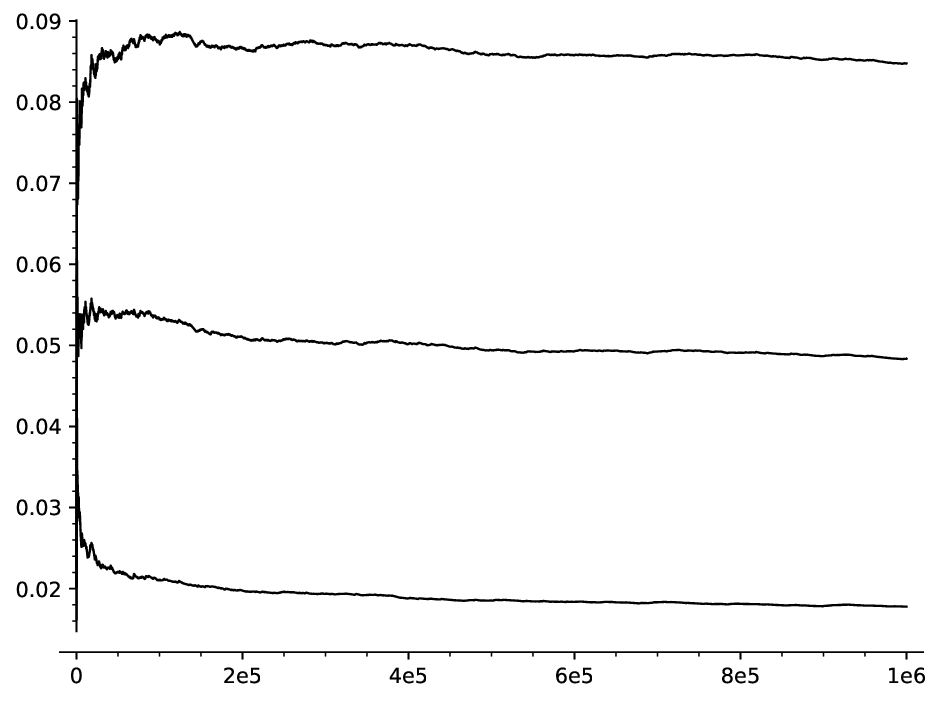}
\caption{14a1: $(\alpha, \beta) = (1,6)$} \label{fig:14_6_1_6_acc_A_orders}
\end{subfigure}\hspace*{\fill}
\begin{subfigure}[b]{0.4\linewidth}
\includegraphics[width=\linewidth]{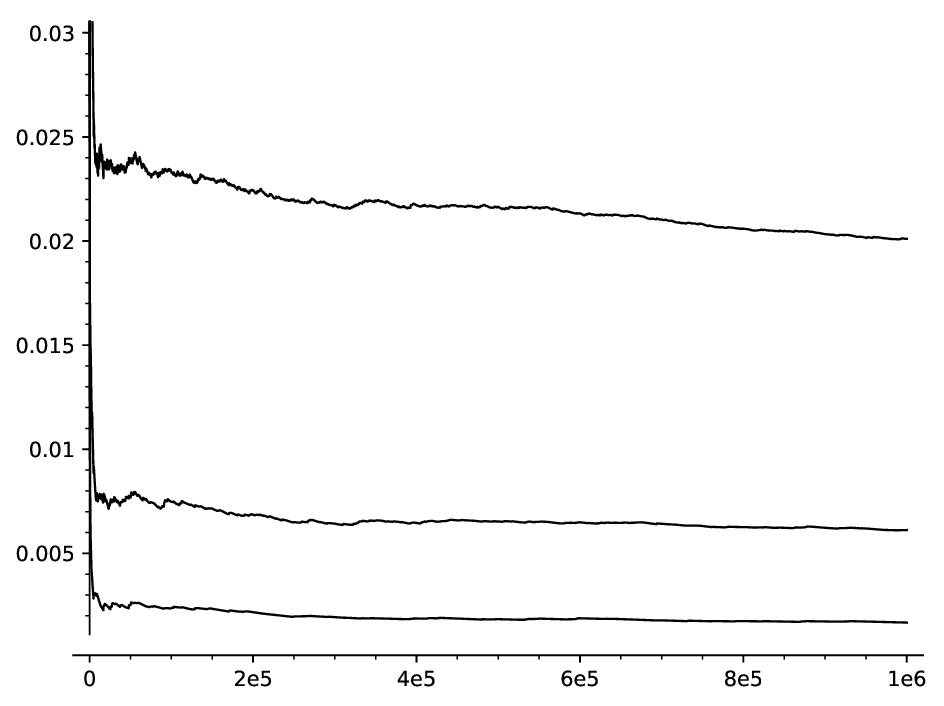}
\caption{14a1: $(\alpha, \beta) = (2,6)$} \label{fig:14_6_2_6_acc_A_orders}
\end{subfigure}
\caption{11a1, 14a1: Ratio~\eqref{ratio_N_orders} $n_{6,E}^{(\alpha, \beta)}(X;L)/X^{1/2}\log^2(X)$ for $L = $ 1, 2, 3 and $k = 6$. Note that the larger $L$ the higher its ratio graph is depicted.} \label{fig:6_alpha_beta_A_acc_11_14}
\end{figure}

\clearpage

\begin{figure}[b!] 
\hspace*{-.7cm}
\begin{subfigure}[b]{0.4\linewidth}
\includegraphics[width=\linewidth]{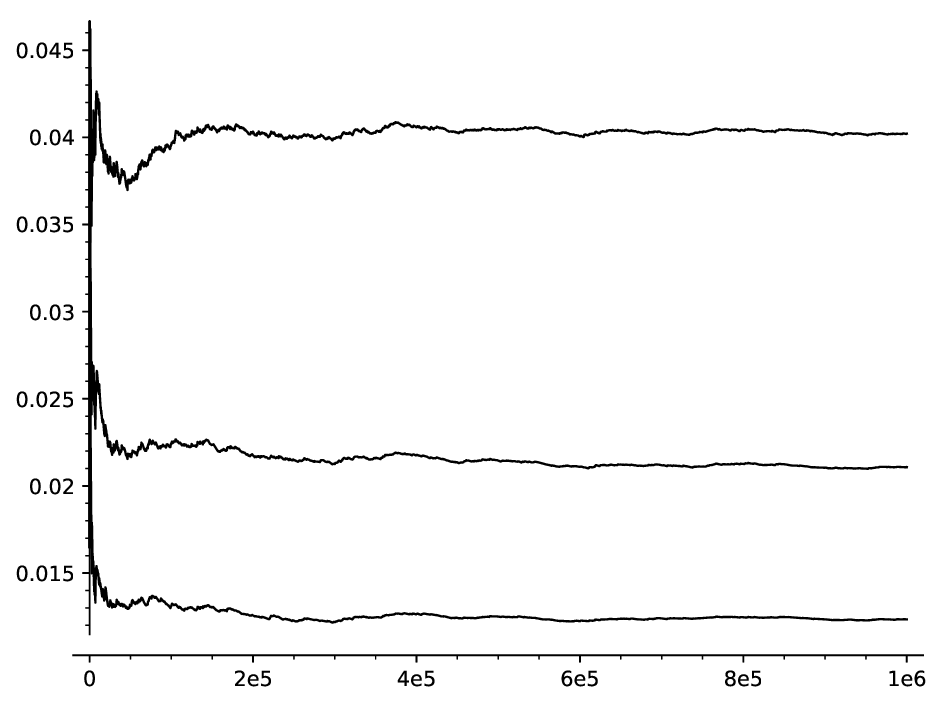}
\caption{15a1: $(\alpha, \beta) = (1,3)$} \label{fig:15_6_1_3_acc_A_orders}
\end{subfigure}\hspace*{\fill}
\begin{subfigure}[b]{0.4\linewidth}
\includegraphics[width=\linewidth]{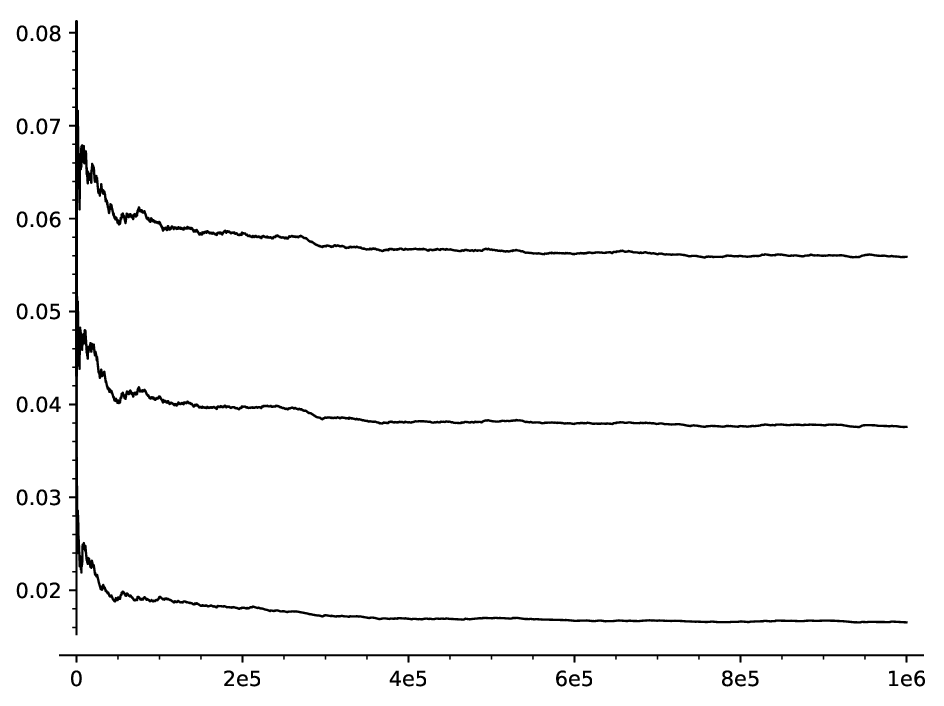}
\caption{15a1: $(\alpha, \beta) = (2,3)$} \label{fig:15_6_2_3_acc_A_orders}
\end{subfigure}
\hspace*{-.7cm}
\begin{subfigure}[b]{0.4\linewidth}
\includegraphics[width=\linewidth]{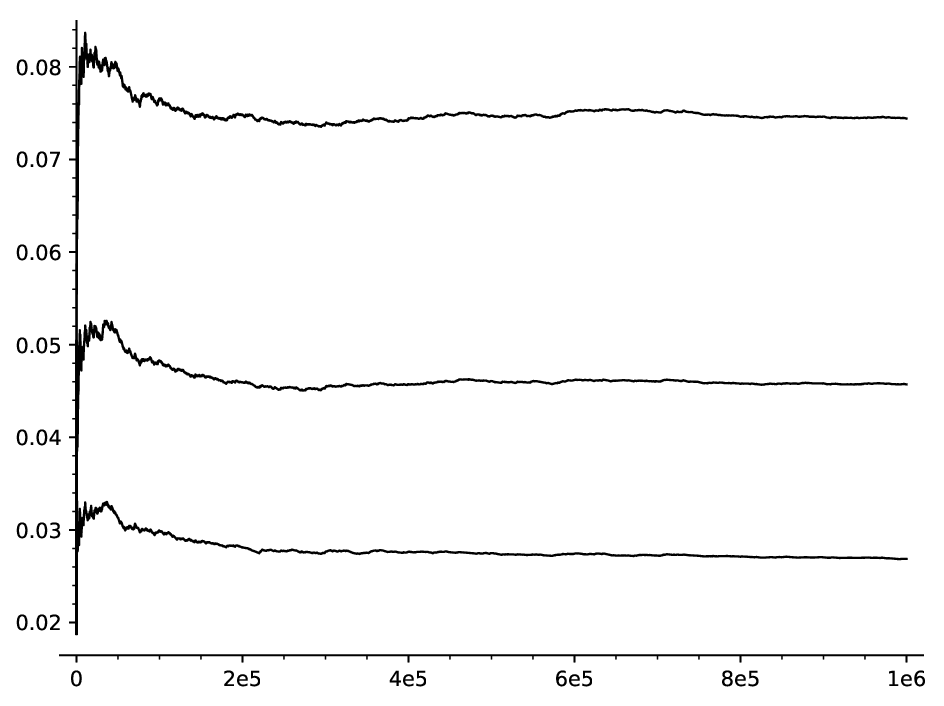}
\caption{15a1: $(\alpha, \beta) = (1,6)$} \label{fig:15_6_1_6_acc_A_orders}
\end{subfigure}\hspace*{\fill}
\begin{subfigure}[b]{0.4\linewidth}
\includegraphics[width=\linewidth]{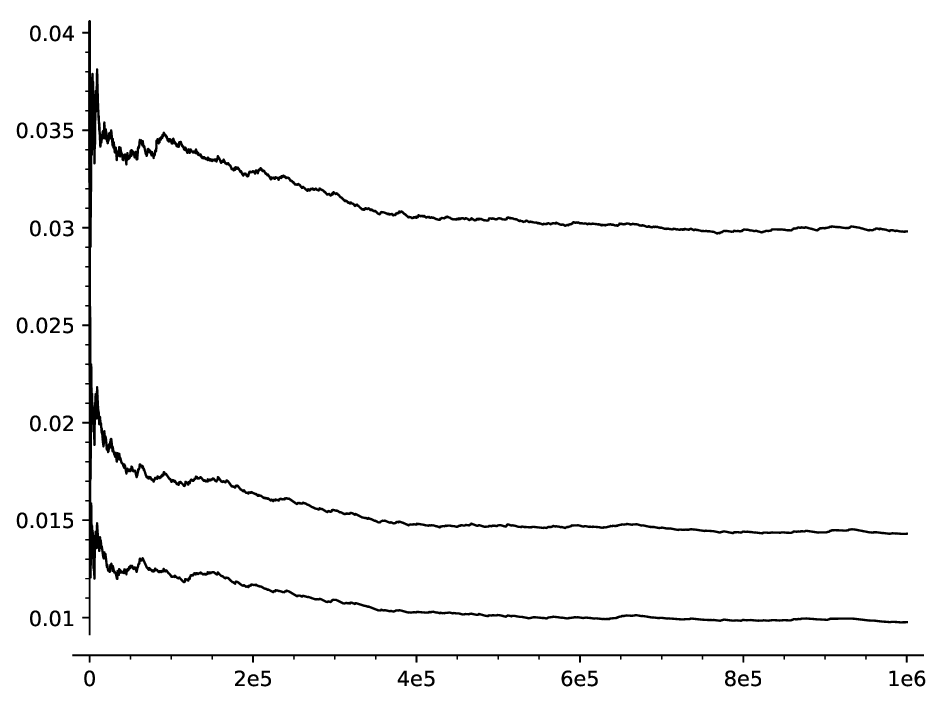}
\caption{15a1: $(\alpha, \beta) = (2,6)$} \label{fig:15_6_2_6_acc_A_orders}
\end{subfigure}
\hspace*{-.7cm}
\begin{subfigure}[b]{0.4\linewidth}
\includegraphics[width=\linewidth]{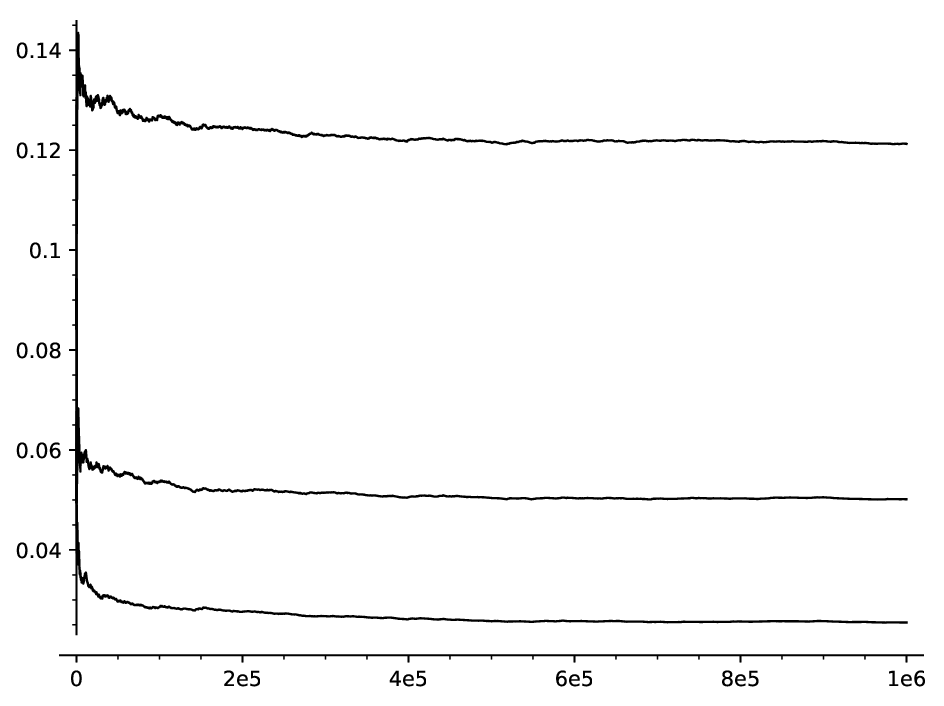}
\caption{17a1: $(\alpha, \beta) = (1,3)$} \label{fig:17_6_1_3_acc_A_orders}
\end{subfigure}\hspace*{\fill}
\begin{subfigure}[b]{0.4\linewidth}
\includegraphics[width=\linewidth]{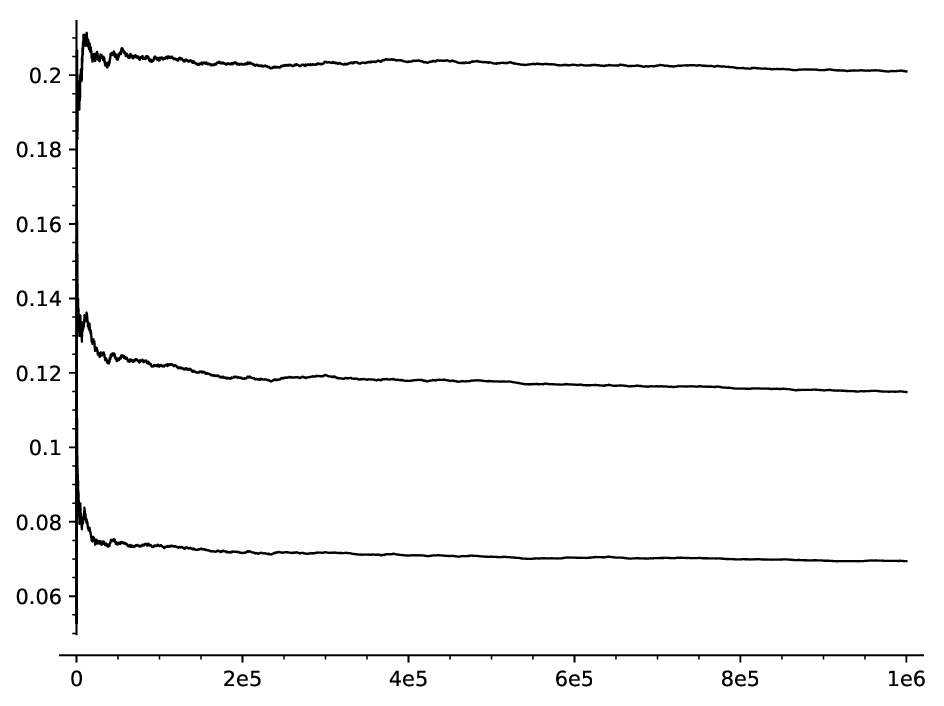}
\caption{17a1: $(\alpha, \beta) = (2,3)$} \label{fig:17_6_2_3_acc_A_orders}
\end{subfigure}
\hspace*{-.7cm}
\begin{subfigure}[b]{0.4\linewidth}
\includegraphics[width=\linewidth]{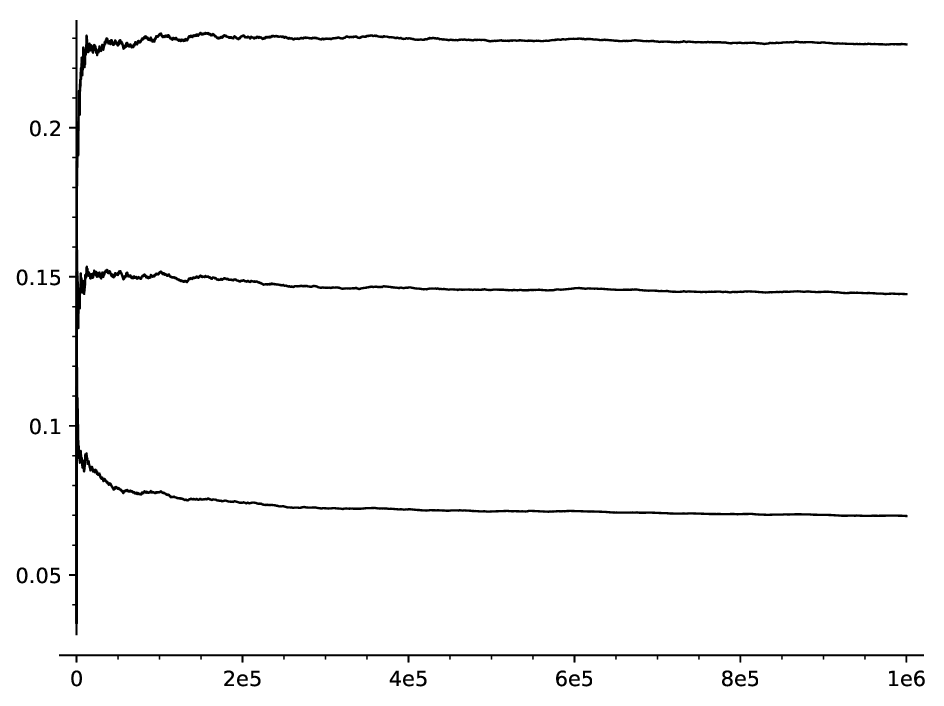}
\caption{17a1: $(\alpha, \beta) = (1,6)$} \label{fig:17_6_1_6_acc_A_orders}
\end{subfigure}\hspace*{\fill}
\begin{subfigure}[b]{0.4\linewidth}
\includegraphics[width=\linewidth]{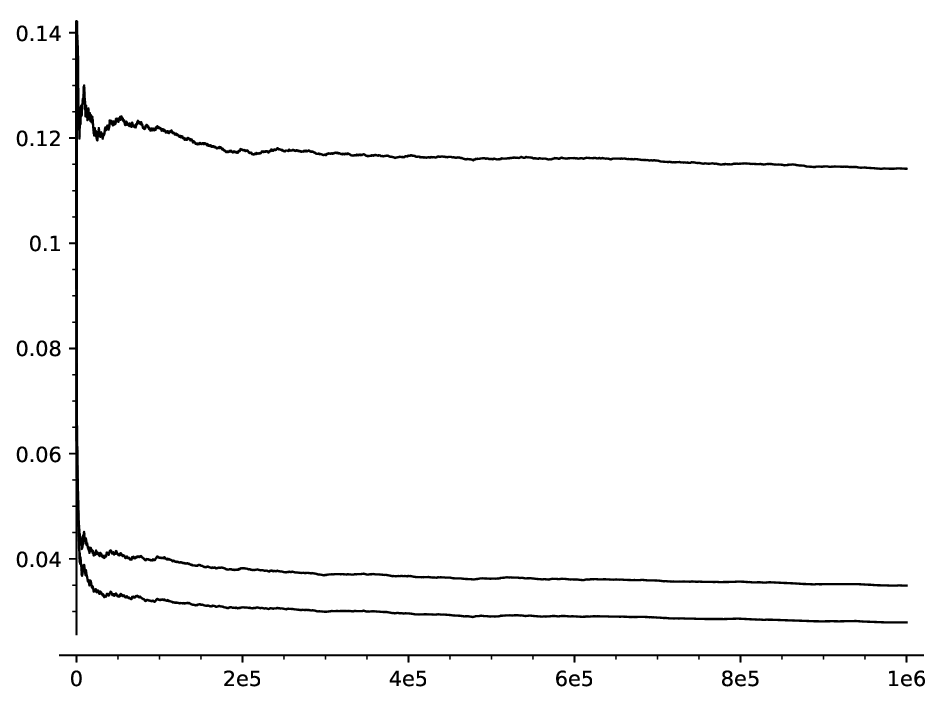}
\caption{17a1: $(\alpha, \beta) = (2,6)$} \label{fig:17_6_2_6_acc_A_orders}
\end{subfigure}
\caption{15a1, 17a1: Ratio~\eqref{ratio_N_orders} $n_{6,E}^{(\alpha, \beta)}(X;L)/X^{1/2}\log^2(X)$ for $L = $ 1, 2, 3 and $k = 6$. Note that the larger $L$ the higher its ratio graph is depicted.} \label{fig:6_alpha_beta_A_acc_15_17}
\end{figure}
\clearpage

\begin{figure}[b!] 
\hspace*{-.7cm}
\begin{subfigure}[b]{0.4\linewidth}
\includegraphics[width=\linewidth]{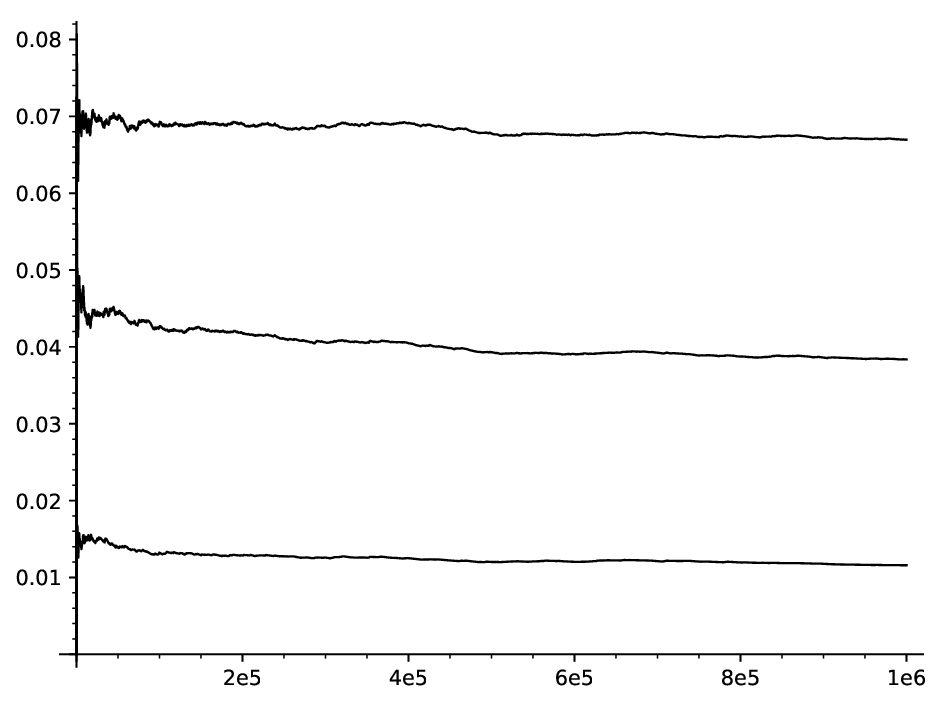}
\caption{19a1: $(\alpha, \beta) = (1,3)$} \label{fig:19_6_1_3_acc_A_orders}
\end{subfigure}\hspace*{\fill}
\begin{subfigure}[b]{0.4\linewidth}
\includegraphics[width=\linewidth]{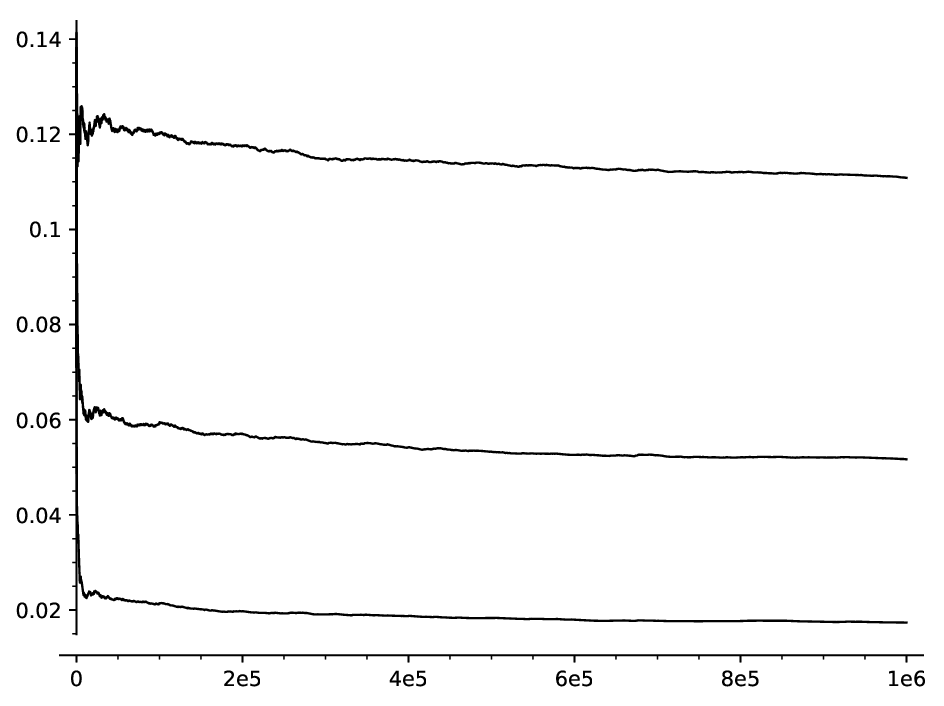}
\caption{19a1: $(\alpha, \beta) = (2,3)$} \label{fig:19_6_2_3_acc_A_orders}
\end{subfigure}
\hspace*{-.7cm}
\begin{subfigure}[b]{0.4\linewidth}
\includegraphics[width=\linewidth]{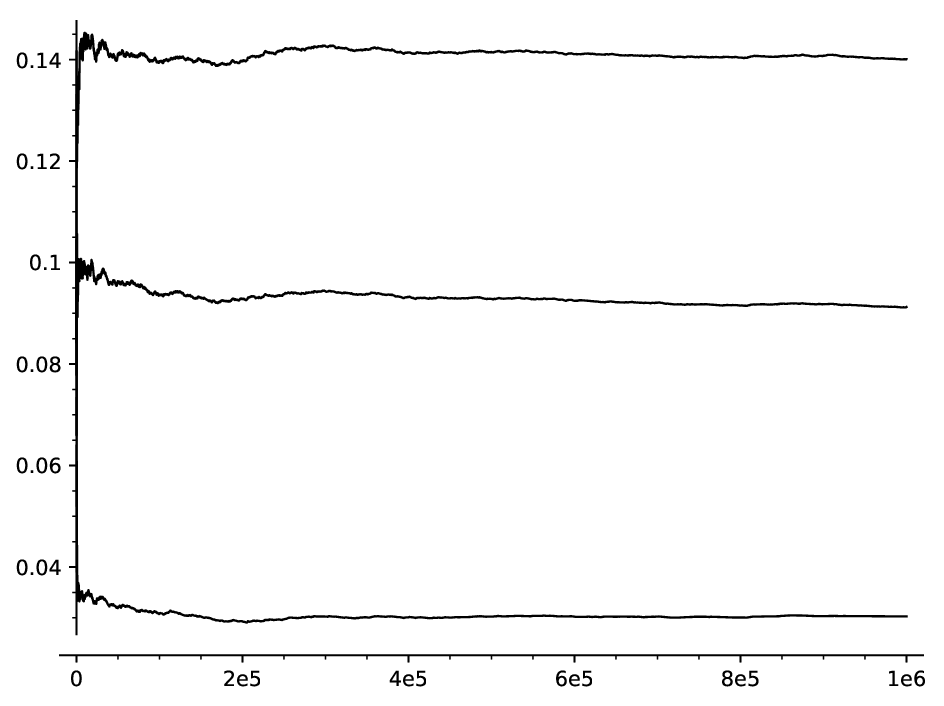}
\caption{19a1: $(\alpha, \beta) = (1,6)$} \label{fig:19_6_1_6_acc_A_orders}
\end{subfigure}\hspace*{\fill}
\begin{subfigure}[b]{0.4\linewidth}
\includegraphics[width=\linewidth]{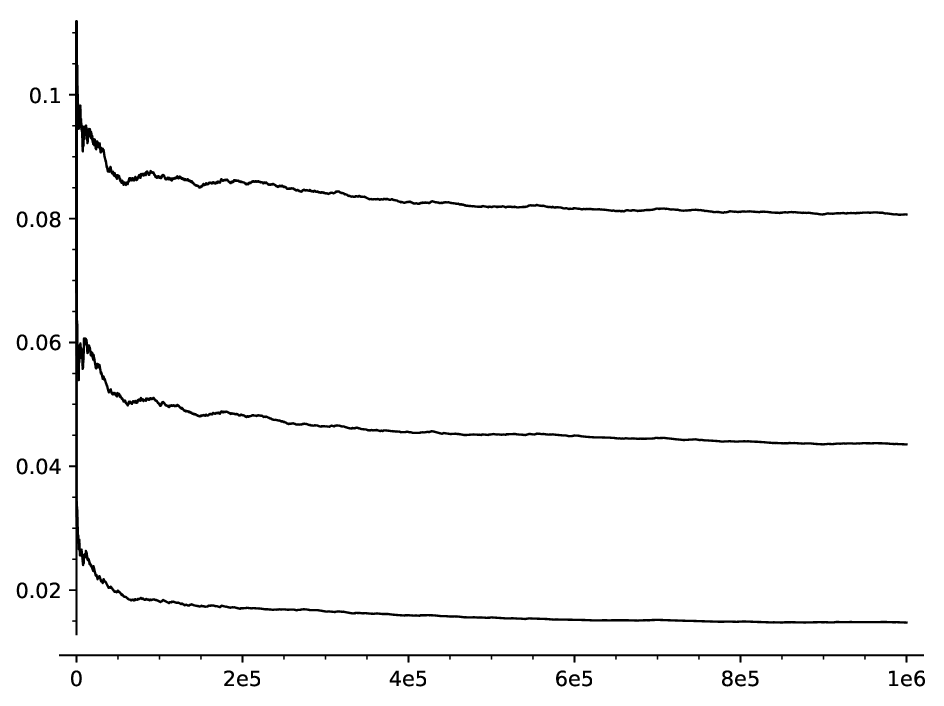}
\caption{19a1: $(\alpha, \beta) = (2,6)$} \label{fig:19_6_2_6_acc_A_orders}
\end{subfigure}
\hspace*{-.7cm}
\begin{subfigure}[b]{0.4\linewidth}
\includegraphics[width=\linewidth]{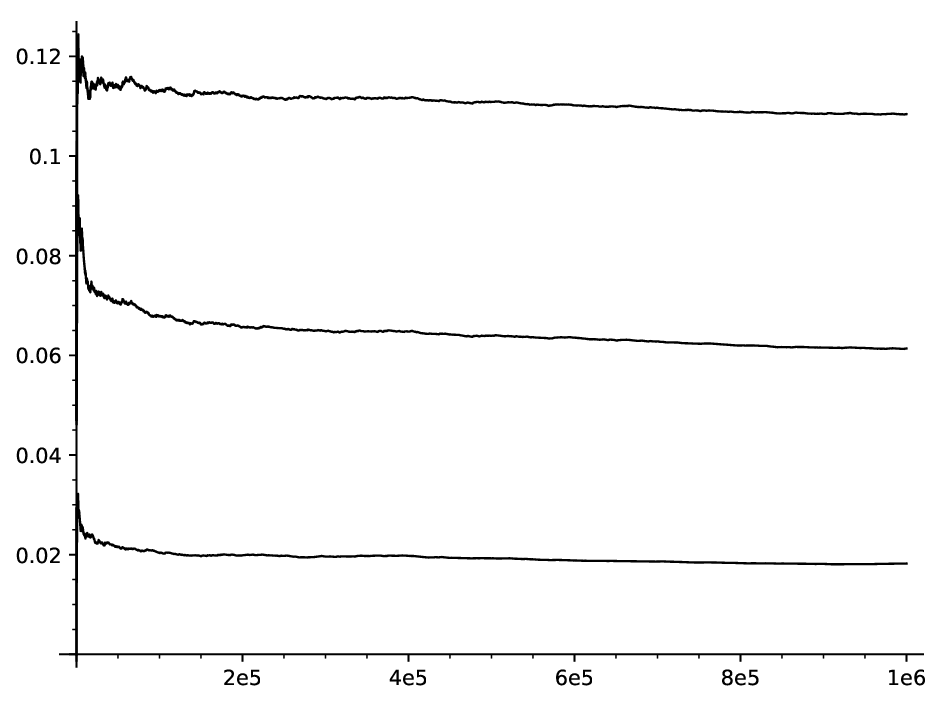}
\caption{37b1: $(\alpha, \beta) = (1,3)$} \label{fig:37_6_1_3_acc_A_orders}
\end{subfigure}\hspace*{\fill}
\begin{subfigure}[b]{0.4\linewidth}
\includegraphics[width=\linewidth]{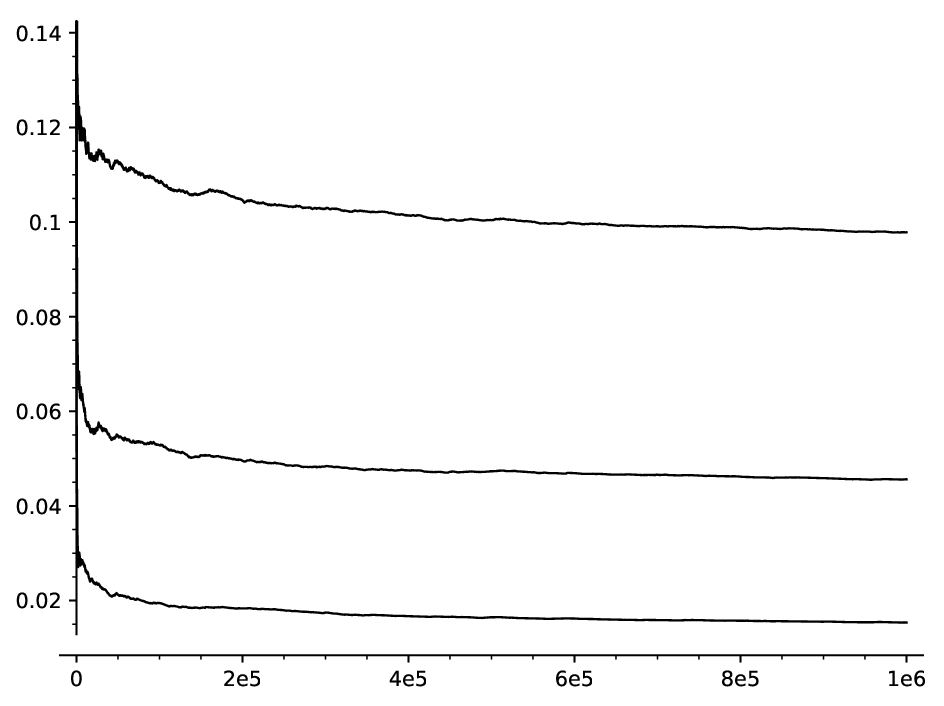}
\caption{37b1: $(\alpha, \beta) = (2,3)$} \label{fig:37_6_2_3_acc_A_orders}
\end{subfigure}
\hspace*{-.7cm}
\begin{subfigure}[b]{0.4\linewidth}
\includegraphics[width=\linewidth]{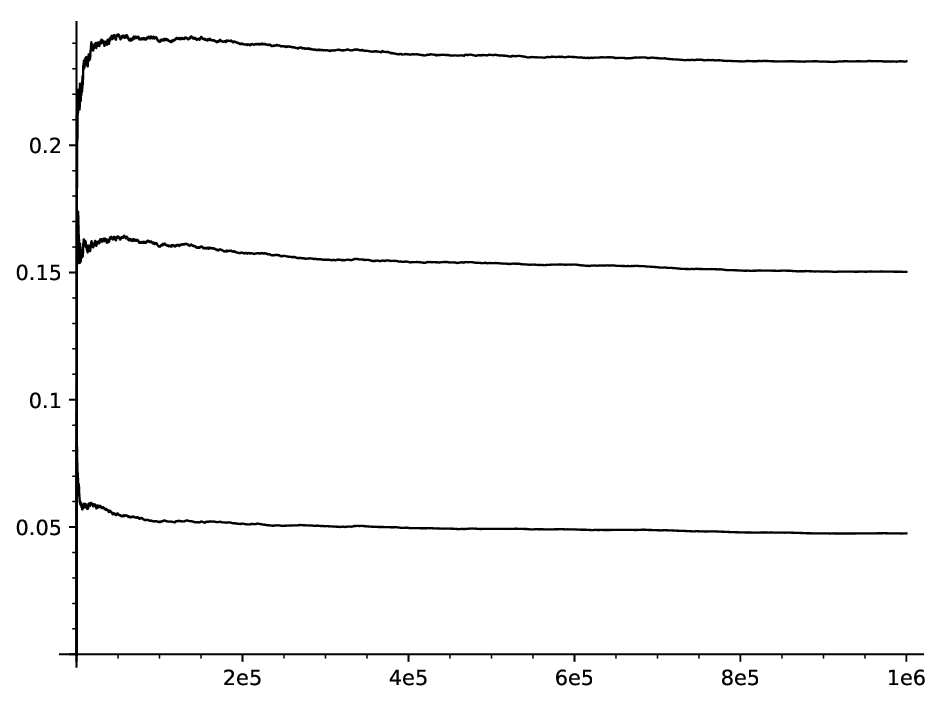}
\caption{37b1: $(\alpha, \beta) = (1,6)$} \label{fig:37_6_1_6_acc_A_orders}
\end{subfigure}\hspace*{\fill}
\begin{subfigure}[b]{0.4\linewidth}
\includegraphics[width=\linewidth]{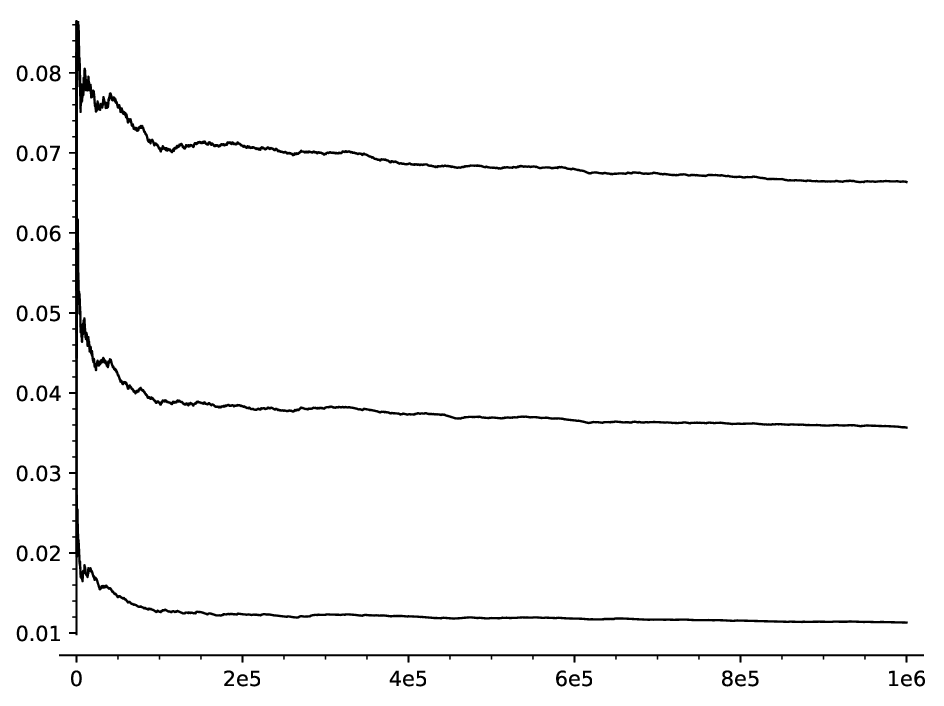}
\caption{37b1: $(\alpha, \beta) = (2,6)$} \label{fig:37_6_2_6_acc_A_orders}
\end{subfigure}
\caption{19a1, 37b1: Ratio~\eqref{ratio_N_orders} $n_{6,E}^{(\alpha, \beta)}(X;L)/X^{1/2}\log^2(X)$ for $L =$ 1, 2, 3 and $k = 6$. Note that the larger $L$ the higher its ratio graph is depicted.} \label{fig:6_alpha_beta_A_acc_19_37}
\end{figure}

\clearpage

\begin{figure}[b!] 
\hspace*{-.7cm}
\begin{subfigure}[b]{0.4\linewidth}
\includegraphics[width=\linewidth]{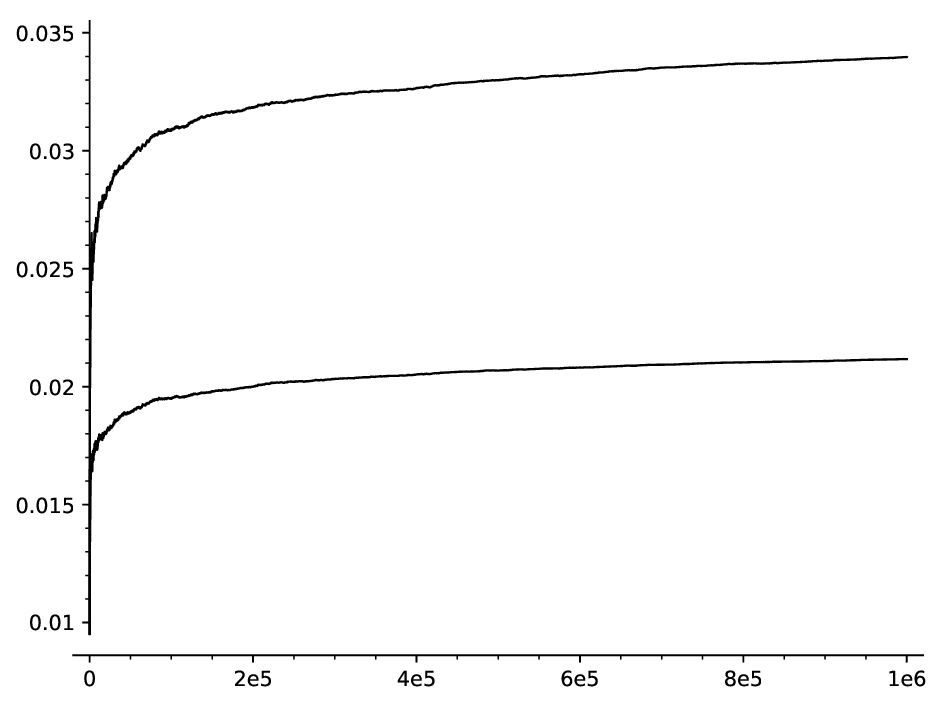}
\captionsetup{width=1.4\linewidth}
\caption{11a1: $(\alpha, \beta) = (1,3)$ Top to bottom $c =$ 0.3, 0.4} \label{fig:11_6_1_3_acc_c_orders}
\end{subfigure}\hspace*{\fill}
\begin{subfigure}[b]{0.4\linewidth}
\includegraphics[width=\linewidth]{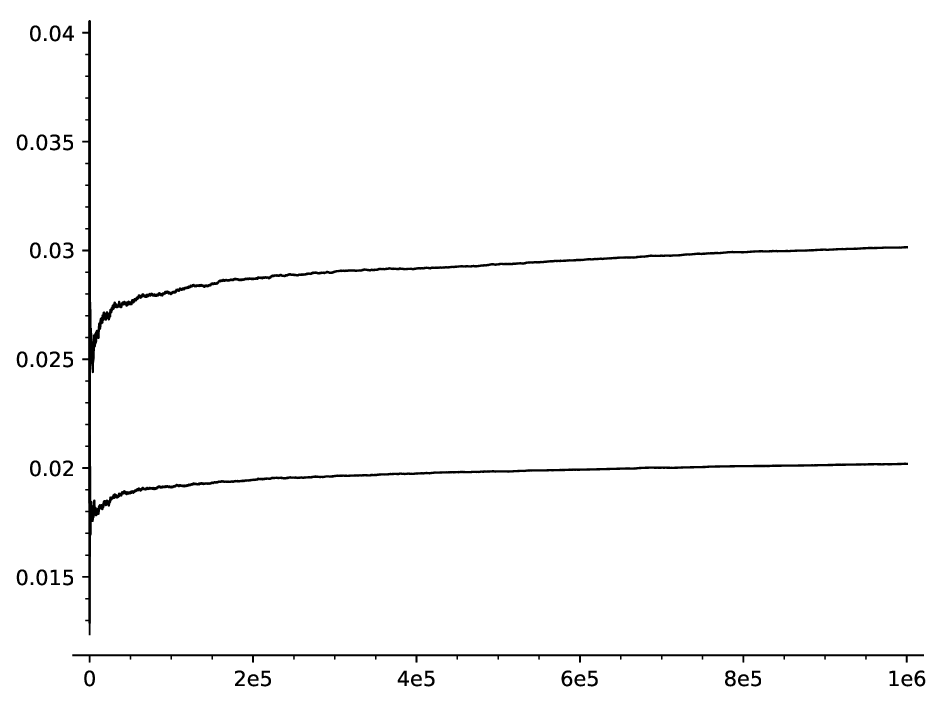}
\captionsetup{width=1.4\linewidth}
\caption{11a1: $(\alpha, \beta) = (2,3)$ Top to bottom $c =$ 0.3, 0.4} \label{fig:11_6_2_3_acc_c_orders}
\end{subfigure}
\hspace*{-.7cm}
\begin{subfigure}[b]{0.4\linewidth}
\includegraphics[width=\linewidth]{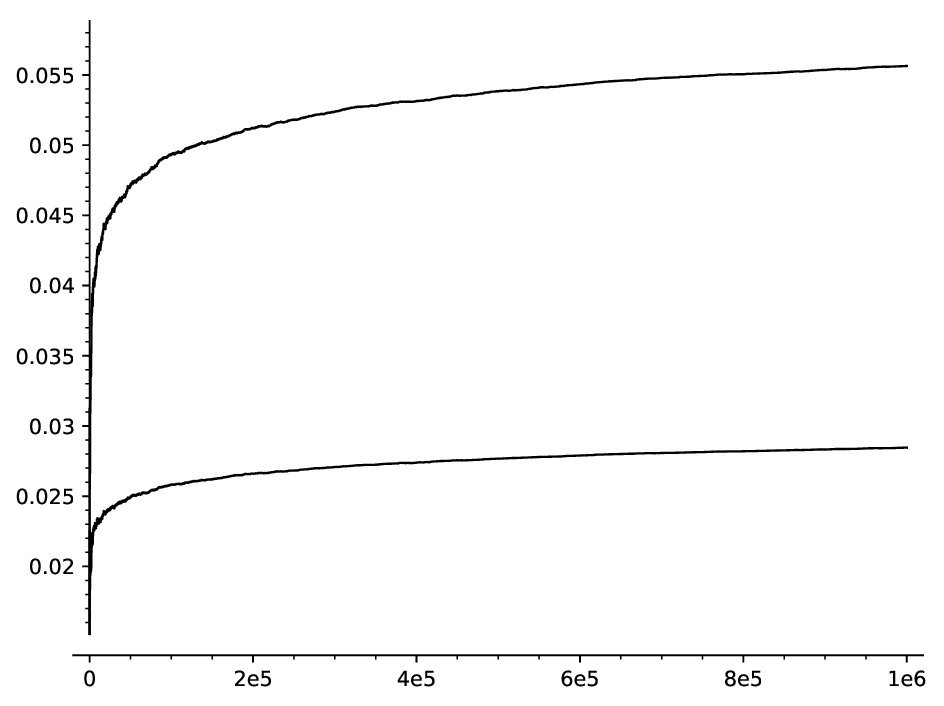}
\captionsetup{width=1.4\linewidth}
\caption{11a1: $(\alpha, \beta) = (1,6)$ Top to bottom $c =$ 0.3, 0.4} \label{fig:11_6_1_6_acc_c_orders}
\end{subfigure}\hspace*{\fill}
\begin{subfigure}[b]{0.4\linewidth}
\includegraphics[width=\linewidth]{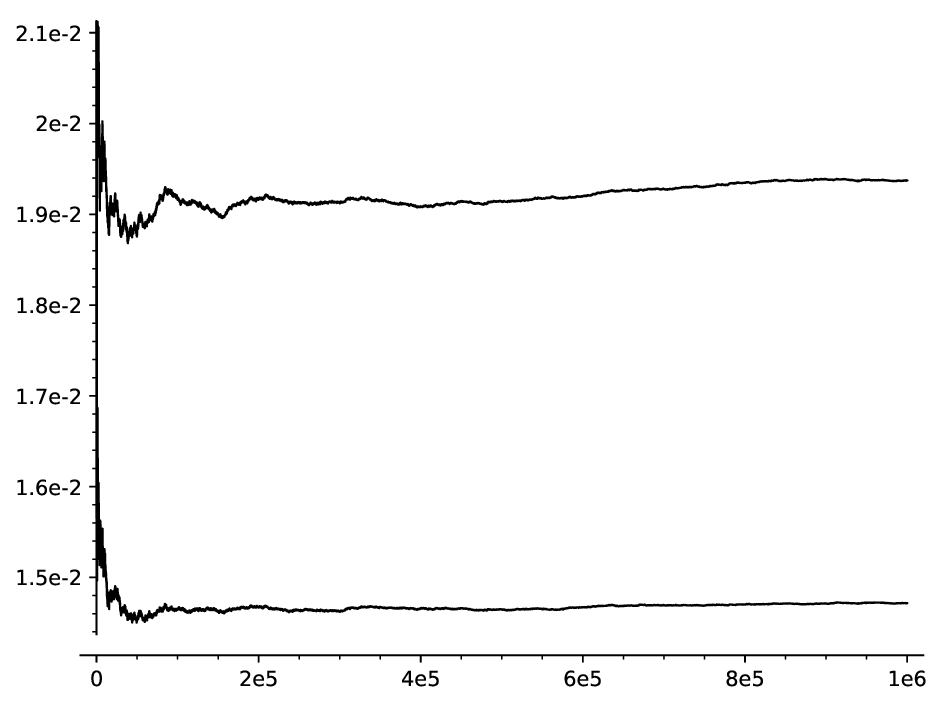}
\captionsetup{width=1.4\linewidth}
\caption{11a1: $(\alpha, \beta) = (2,6)$ Top to bottom $c =$ 0.3, 0.4} \label{fig:11_6_2_6_acc_c_orders}
\end{subfigure}
\hspace*{-.7cm}
\begin{subfigure}[b]{0.4\linewidth}
\includegraphics[width=\linewidth]{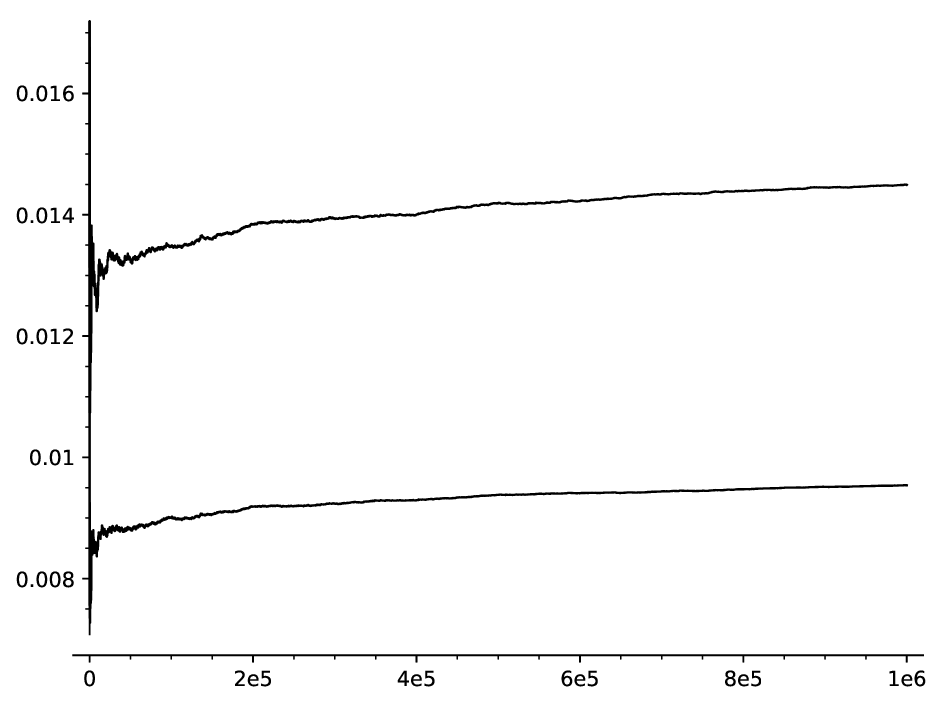}
\captionsetup{width=1.4\linewidth}
\caption{14a1: $(\alpha, \beta) = (1,3)$ Top to bottom $c =$ 0.3, 0.4} \label{fig:14_6_1_3_acc_c_orders}
\end{subfigure}\hspace*{\fill}
\begin{subfigure}[b]{0.4\linewidth}
\includegraphics[width=\linewidth]{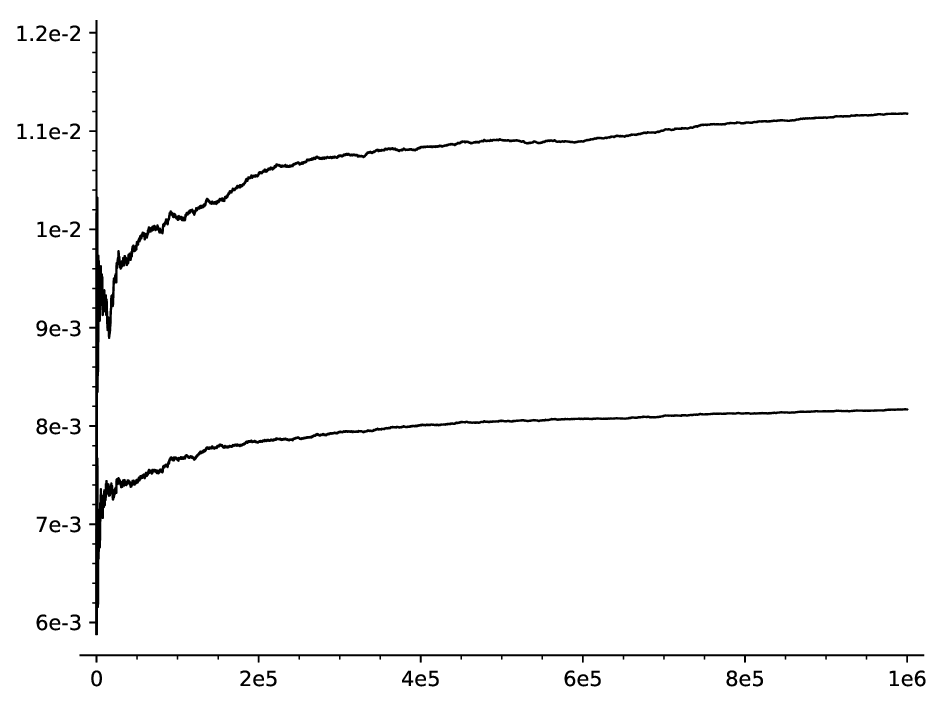}
\captionsetup{width=1.4\linewidth}
\caption{14a1: $(\alpha, \beta) = (2,3)$ Top to bottom $c =$ 0.3, 0.4} \label{fig:14_6_2_3_acc_c_orders}
\end{subfigure}
\hspace*{-.7cm}
\begin{subfigure}[b]{0.4\linewidth}
\includegraphics[width=\linewidth]{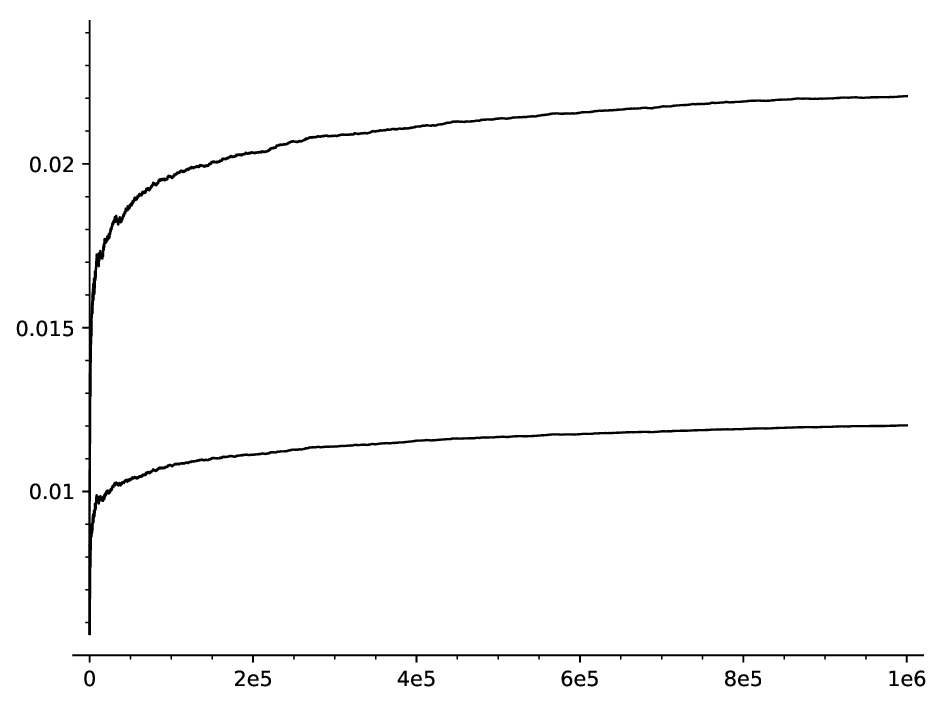}
\captionsetup{width=1.4\linewidth}
\caption{14a1: $(\alpha, \beta) = (1,6)$ Top to bottom $c =$ 0.3, 0.4} \label{fig:14_6_1_6_acc_c_orders}
\end{subfigure}\hspace*{\fill}
\begin{subfigure}[b]{0.4\linewidth}
\includegraphics[width=\linewidth]{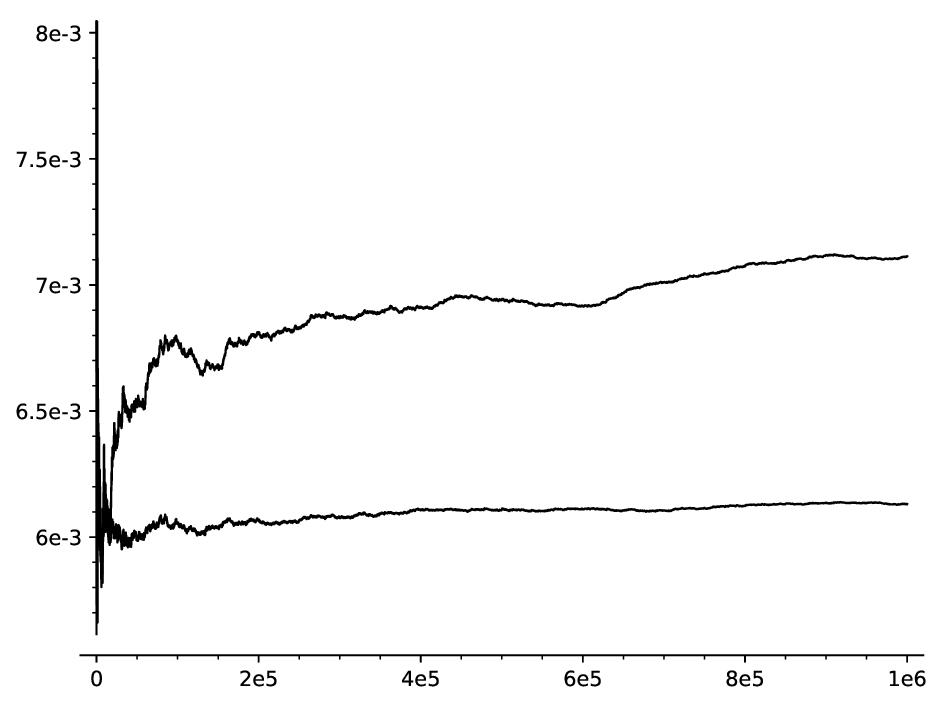}
\captionsetup{width=1.4\linewidth}
\caption{14a1: $(\alpha, \beta) = (2,6)$ Top to bottom $c =$ 0.3, 0.4} \label{fig:14_6_2_6_acc_c_orders}
\end{subfigure}
\caption{11a1, 14a1: Ratio~\eqref{ratio_M_orders} $m_{6,E}^{(\alpha, \beta)}(X;c)/X^{c +1/2}\log^2(X)$ for $c = $ 0.3, 0.4 and $k = 6$.} \label{fig:c_11_14_acc_6_orders}
\end{figure}

\clearpage

\begin{figure}[b!] 
\hspace*{-.7cm}
\begin{subfigure}[b]{0.4\linewidth}
\includegraphics[width=\linewidth]{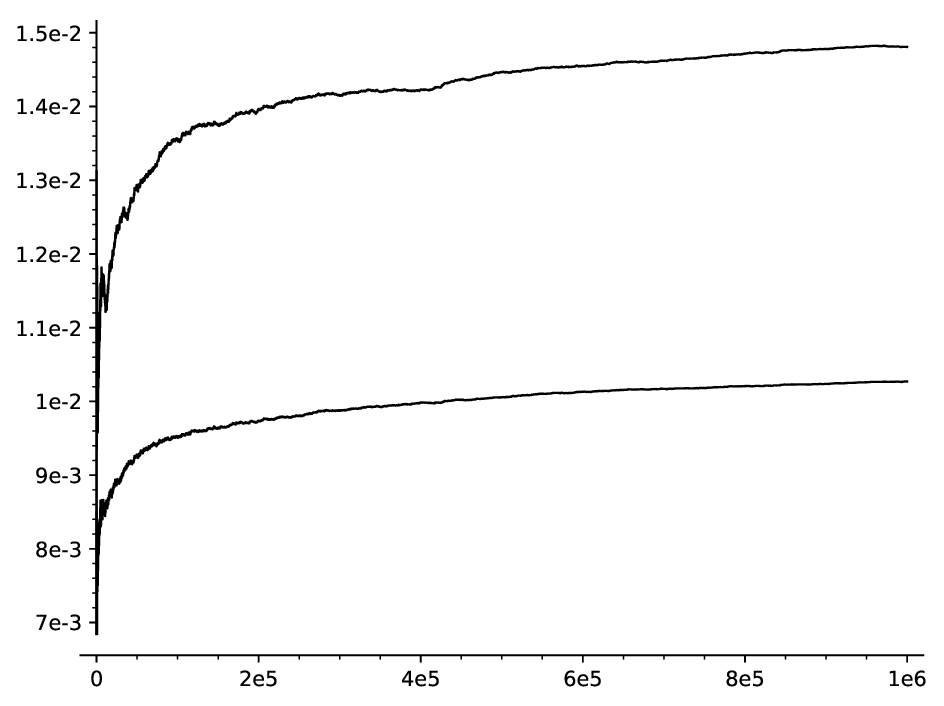}
\captionsetup{width=1.4\linewidth}
\caption{15a1: $(\alpha, \beta) = (1,3)$ Top to bottom $c =$ 0.3, 0.4} \label{fig:15_6_1_3_acc_c_orders}
\end{subfigure}\hspace*{\fill}
\begin{subfigure}[b]{0.4\linewidth}
\includegraphics[width=\linewidth]{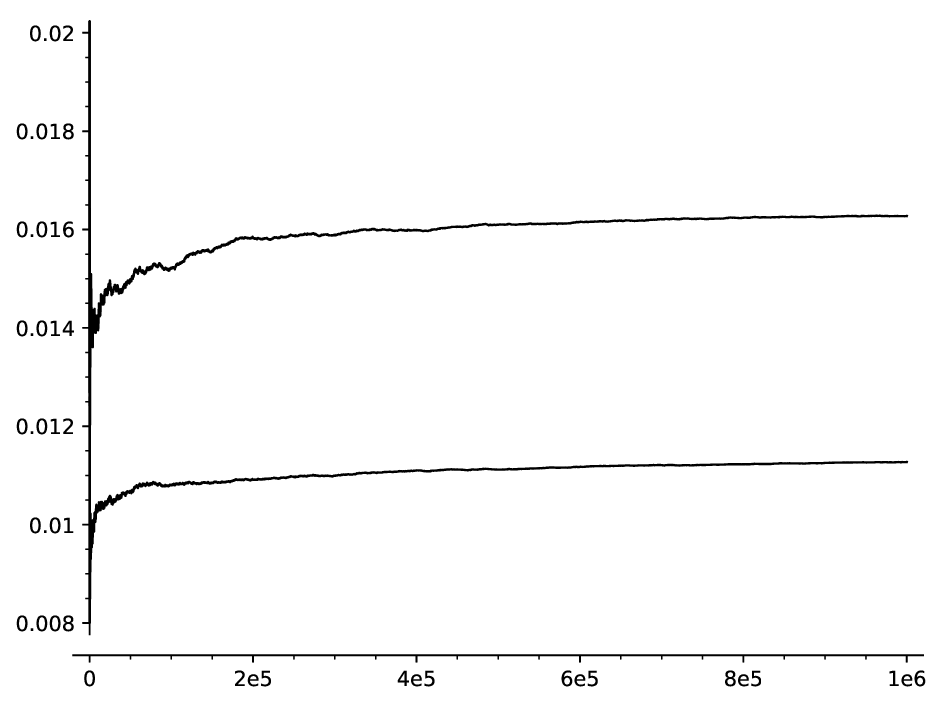}
\captionsetup{width=1.4\linewidth}
\caption{15a1: $(\alpha, \beta) = (2,3)$ Top to bottom $c =$ 0.3, 0.4} \label{fig:15_6_2_3_acc_c_orders}
\end{subfigure}
\hspace*{-.7cm}
\begin{subfigure}[b]{0.4\linewidth}
\includegraphics[width=\linewidth]{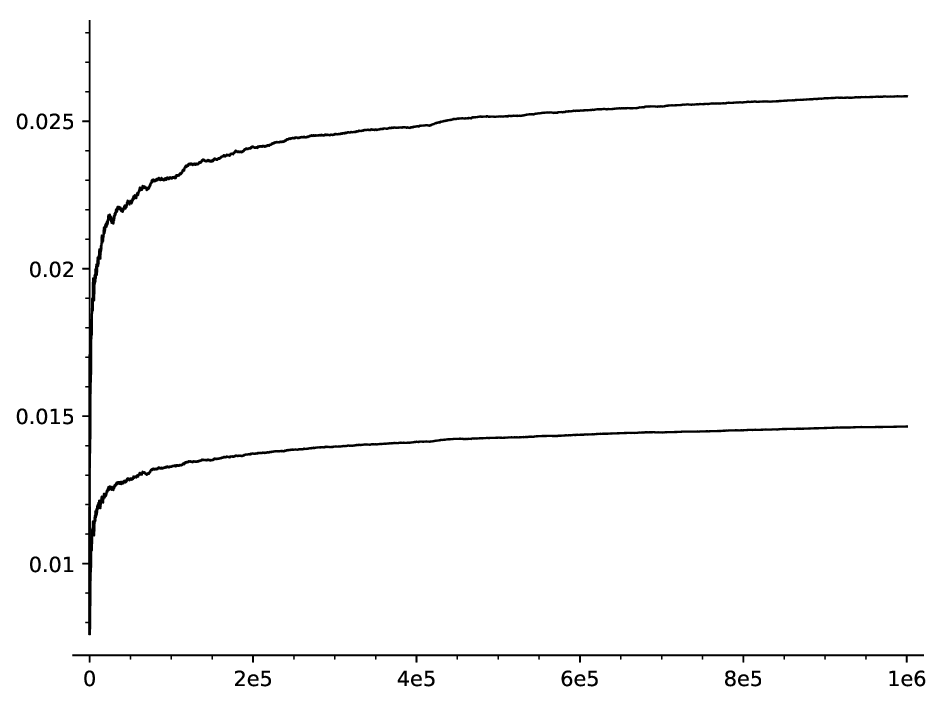}
\captionsetup{width=1.4\linewidth}
\caption{15a1: $(\alpha, \beta) = (1,6)$ Top to bottom $c =$ 0.3, 0.4} \label{fig:15_6_1_6_acc_c_orders}
\end{subfigure}\hspace*{\fill}
\begin{subfigure}[b]{0.4\linewidth}
\includegraphics[width=\linewidth]{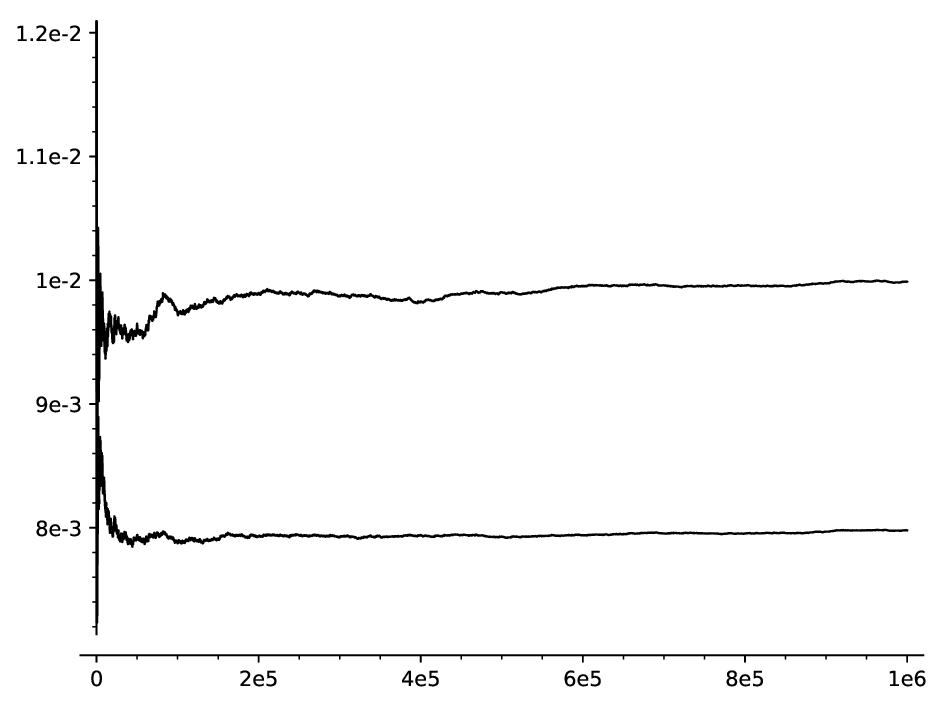}
\captionsetup{width=1.4\linewidth}
\caption{15a1: $(\alpha, \beta) = (2,6)$ Top to bottom $c =$ 0.3, 0.4} \label{fig:15_6_2_6_acc_c_orders}
\end{subfigure}
\hspace*{-.7cm}
\begin{subfigure}[b]{0.4\linewidth}
\includegraphics[width=\linewidth]{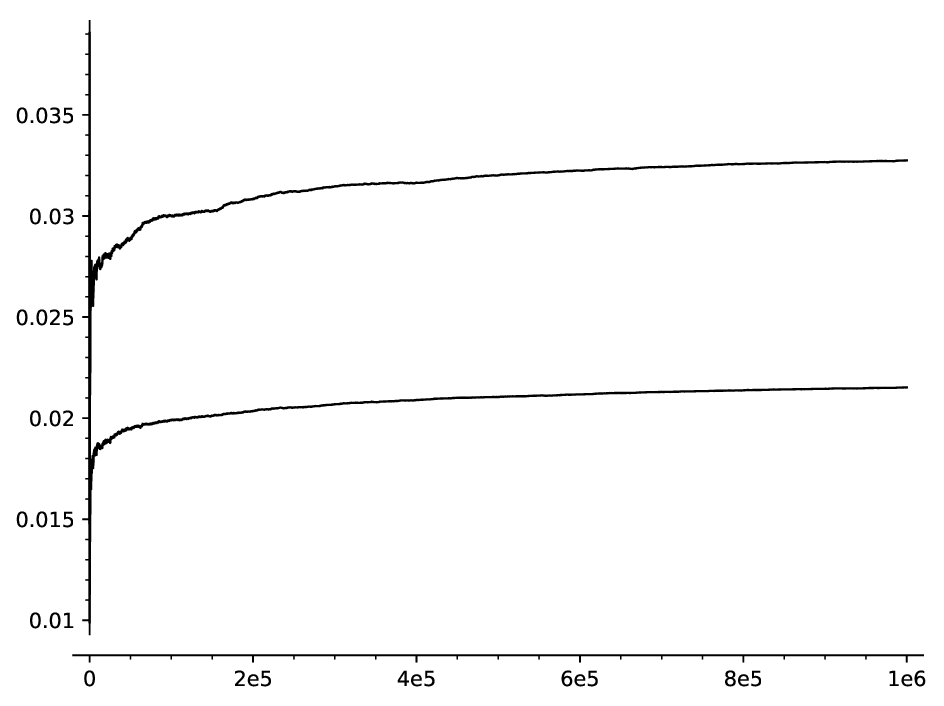}
\captionsetup{width=1.4\linewidth}
\caption{17a1: $(\alpha, \beta) = (1,3)$ Top to bottom $c =$ 0.3, 0.4} \label{fig:17_6_1_3_acc_c_orders}
\end{subfigure}\hspace*{\fill}
\begin{subfigure}[b]{0.4\linewidth}
\includegraphics[width=\linewidth]{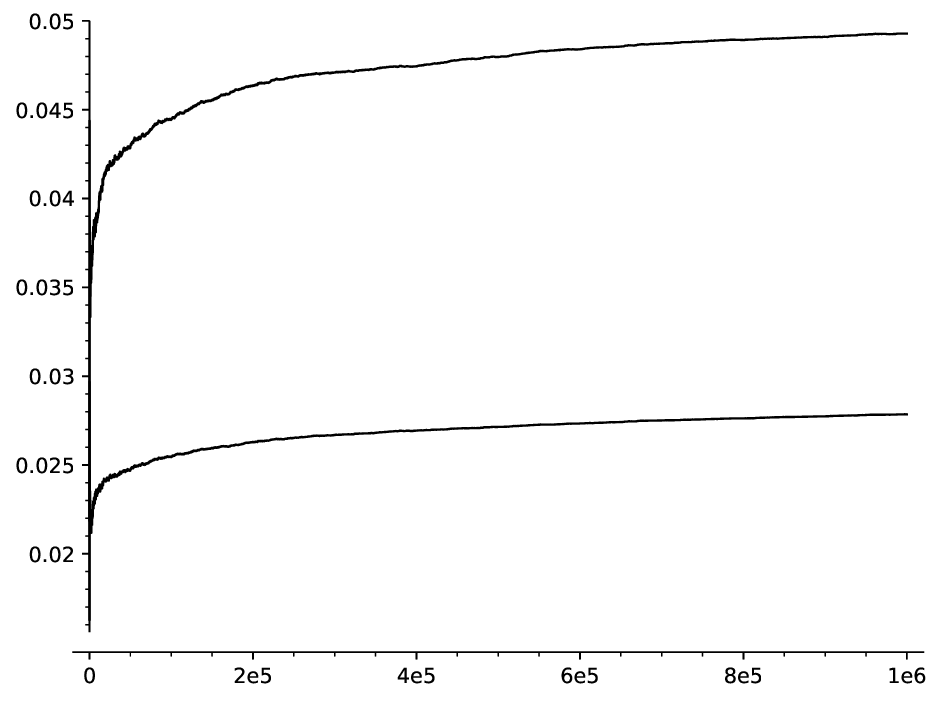}
\captionsetup{width=1.4\linewidth}
\caption{17a1: $(\alpha, \beta) = (2,3)$ Top to bottom $c =$ 0.3, 0.4} \label{fig:17_6_2_3_acc_c_orders}
\end{subfigure}
\hspace*{-.7cm}
\begin{subfigure}[b]{0.4\linewidth}
\includegraphics[width=\linewidth]{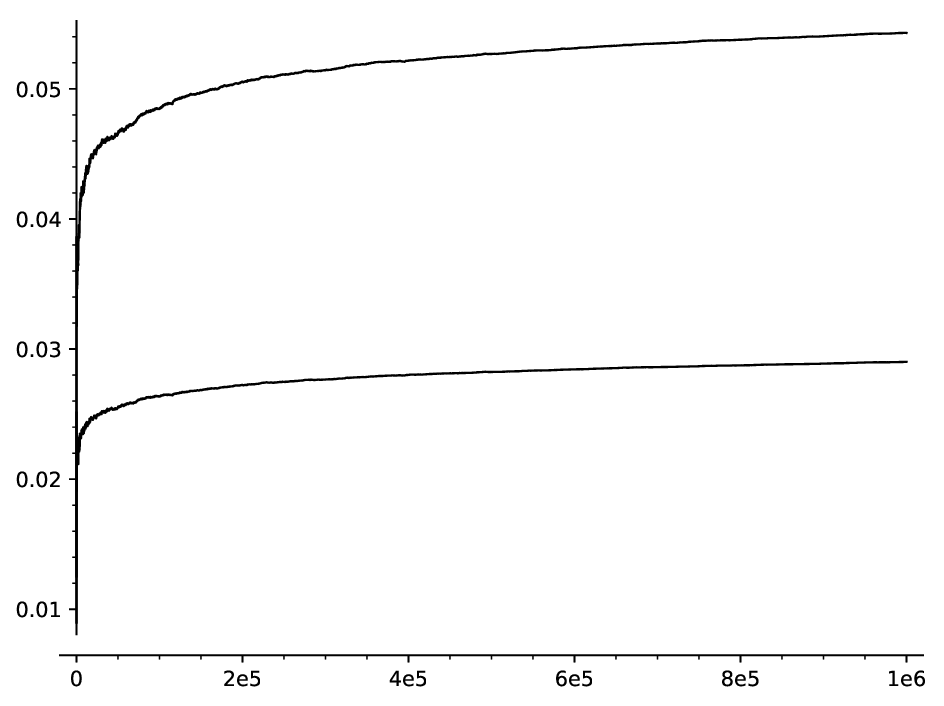}
\captionsetup{width=1.4\linewidth}
\caption{17a1: $(\alpha, \beta) = (1,6)$ Top to bottom $c =$ 0.3, 0.4} \label{fig:17_6_1_6_acc_c_orders}
\end{subfigure}\hspace*{\fill}
\begin{subfigure}[b]{0.4\linewidth}
\includegraphics[width=\linewidth]{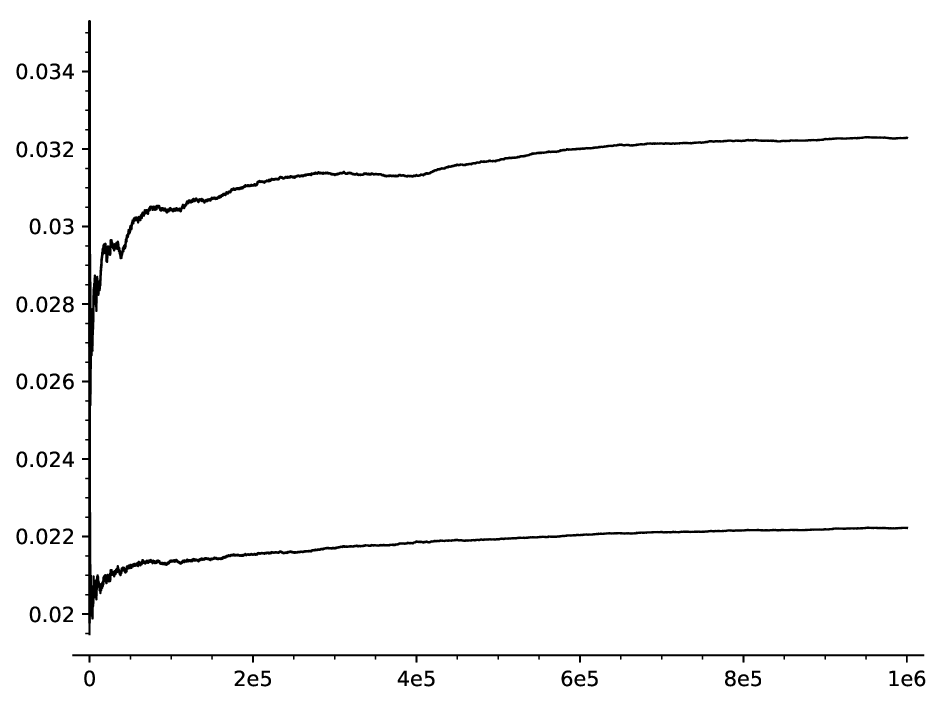}
\captionsetup{width=1.4\linewidth}
\caption{17a1: $(\alpha, \beta) = (2,6)$ Top to bottom $c =$ 0.3, 0.4} \label{fig:17_6_2_6_acc_c_orders}
\end{subfigure}
\caption{15a1, 17a1: Ratio~\eqref{ratio_M_orders} $m_{6,E}^{(\alpha, \beta)}(X;c)/X^{c +1/2}\log^2(X)$ for $c = $ 0.3, 0.4 and $k = 6$.} \label{fig:c_15_17_acc_6_orders}
\end{figure}

\clearpage

\begin{figure}[b!] 
\hspace*{-.7cm}
\begin{subfigure}[b]{0.4\linewidth}
\includegraphics[width=\linewidth]{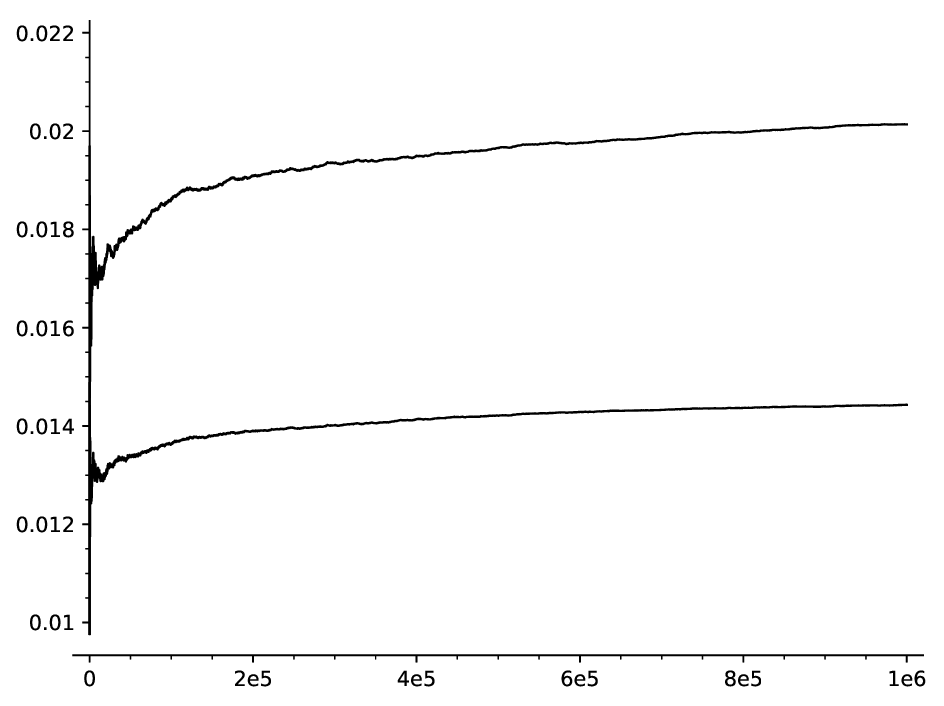}
\captionsetup{width=1.4\linewidth}
\caption{19a1: $(\alpha, \beta) = (1,3)$ Top to bottom $c =$ 0.3, 0.4} \label{fig:19_6_1_3_acc_c_orders}
\end{subfigure}\hspace*{\fill}
\begin{subfigure}[b]{0.4\linewidth}
\includegraphics[width=\linewidth]{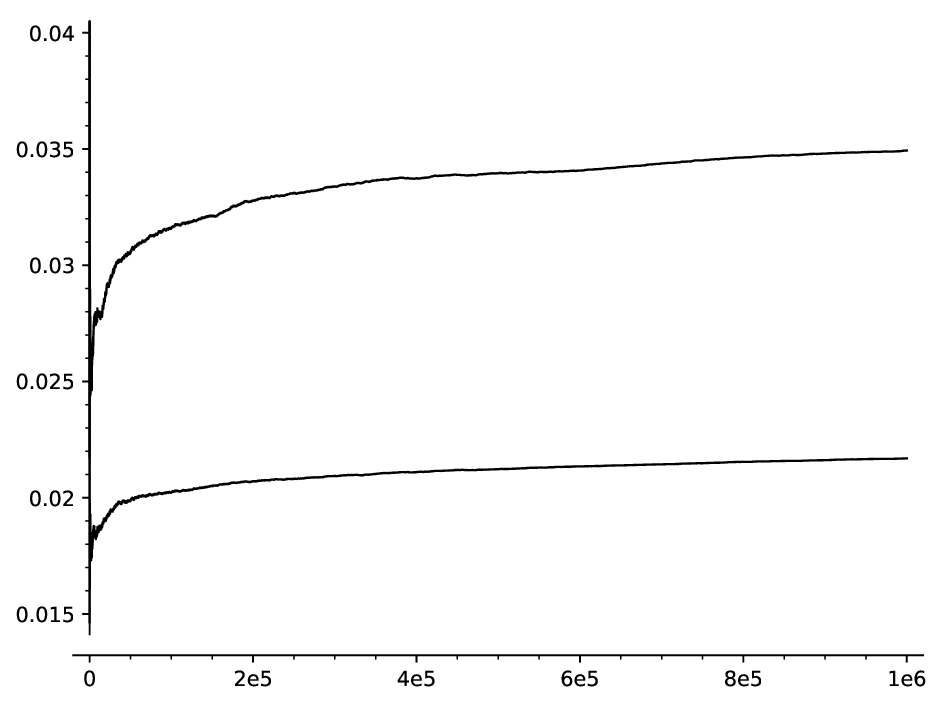}
\captionsetup{width=1.4\linewidth}
\caption{19a1: $(\alpha, \beta) = (2,3)$ Top to bottom $c =$ 0.3, 0.4} \label{fig:19_6_2_3_acc_c_orders}
\end{subfigure}
\hspace*{-.7cm}
\begin{subfigure}[b]{0.4\linewidth}
\includegraphics[width=\linewidth]{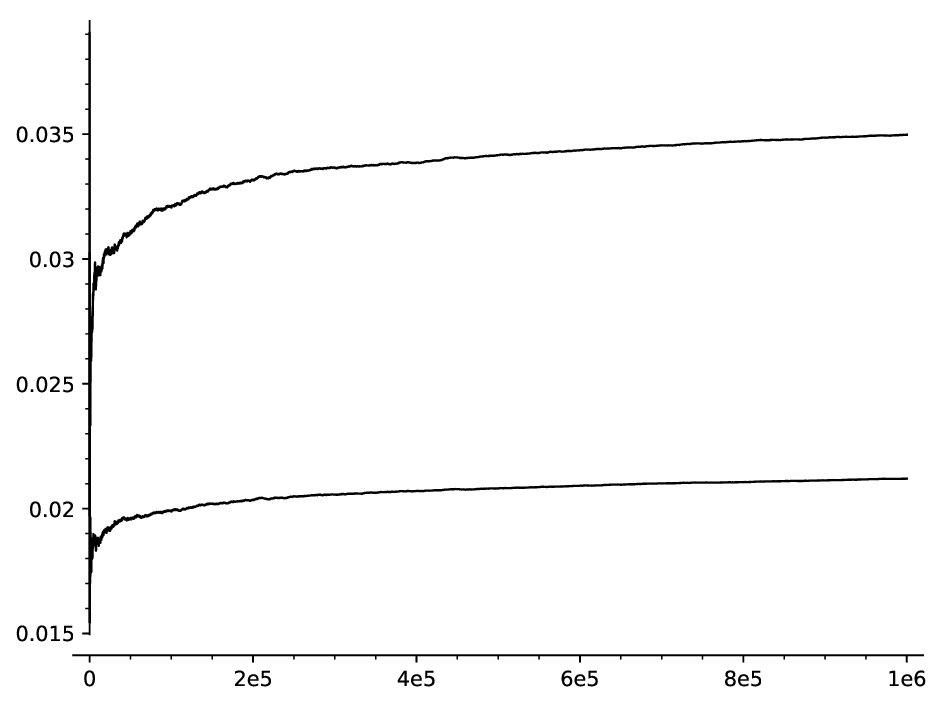}
\captionsetup{width=1.4\linewidth}
\caption{19a1: $(\alpha, \beta) = (1,6)$ Top to bottom $c =$ 0.3, 0.4} \label{fig:19_6_1_6_acc_c_orders}
\end{subfigure}\hspace*{\fill}
\begin{subfigure}[b]{0.4\linewidth}
\includegraphics[width=\linewidth]{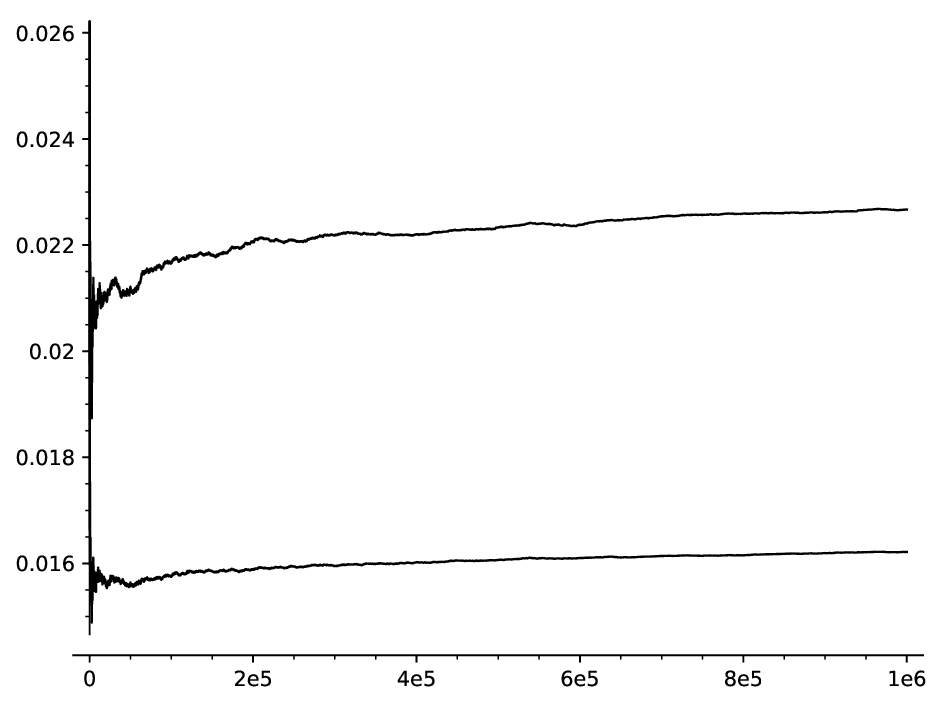}
\captionsetup{width=1.4\linewidth}
\caption{19a1: $(\alpha, \beta) = (2,6)$ Top to bottom $c =$ 0.3, 0.4} \label{fig:19_6_2_6_acc_c_orders}
\end{subfigure}
\hspace*{-.7cm}
\begin{subfigure}[b]{0.4\linewidth}
\includegraphics[width=\linewidth]{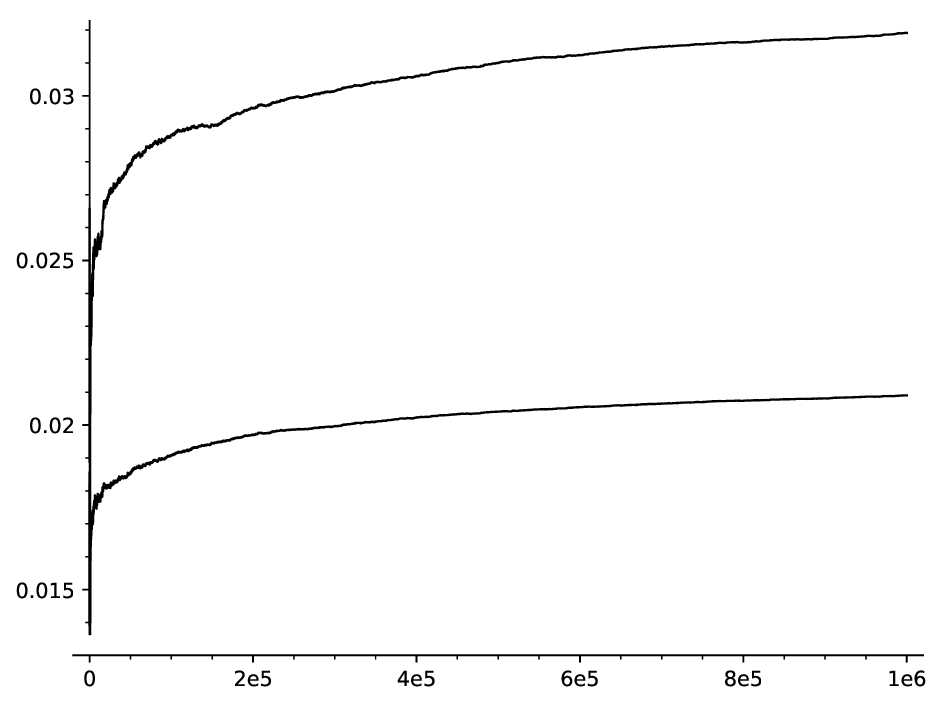}
\captionsetup{width=1.4\linewidth}
\caption{37b1: $(\alpha, \beta) = (1,3)$ Top to bottom $c =$ 0.3, 0.4} \label{fig:37_6_1_3_acc_c_orders}
\end{subfigure}\hspace*{\fill}
\begin{subfigure}[b]{0.4\linewidth}
\includegraphics[width=\linewidth]{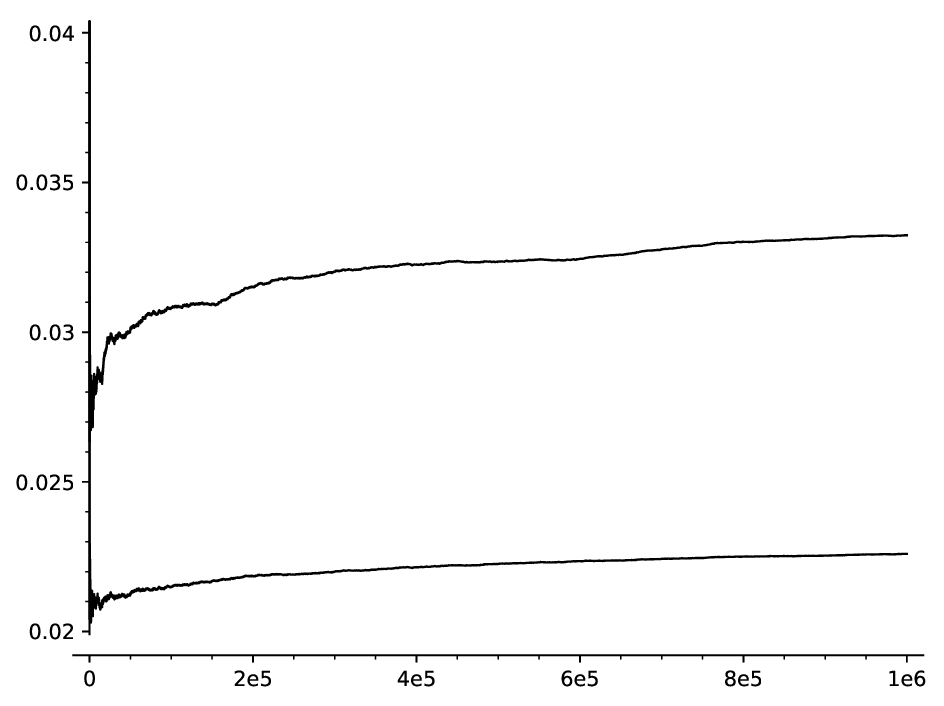}
\captionsetup{width=1.4\linewidth}
\caption{37b1: $(\alpha, \beta) = (2,3)$ Top to bottom $c =$ 0.3, 0.4} \label{fig:37_6_2_3_acc_c_orders}
\end{subfigure}
\hspace*{-.7cm}
\begin{subfigure}[b]{0.4\linewidth}
\includegraphics[width=\linewidth]{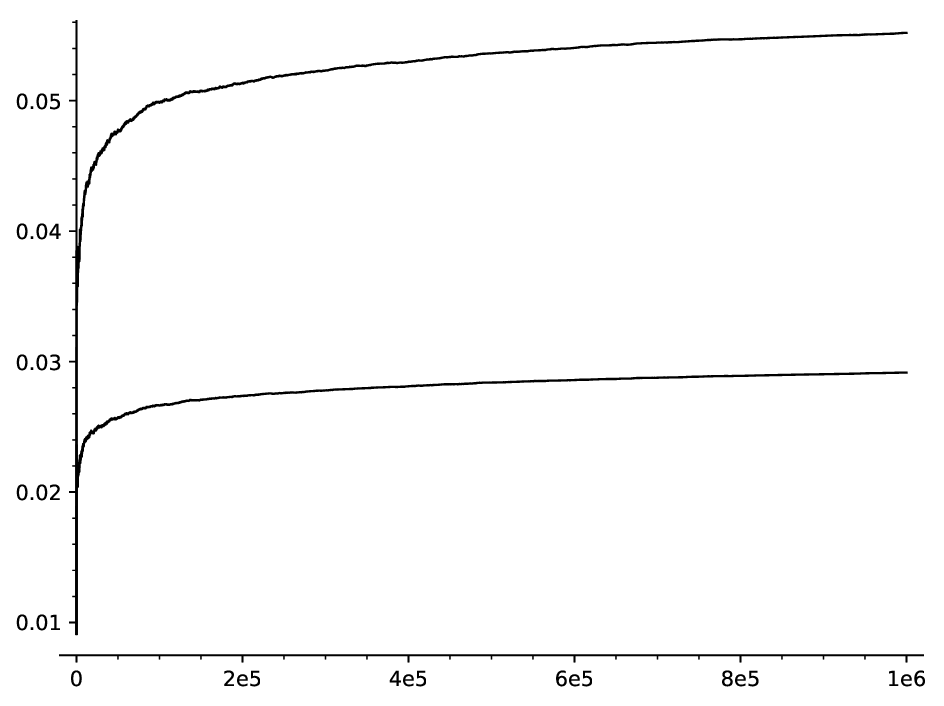}
\captionsetup{width=1.4\linewidth}
\caption{37b1: $(\alpha, \beta) = (1,6)$ Top to bottom $c =$ 0.3, 0.4} \label{fig:37_6_1_6_acc_c_orders}
\end{subfigure}\hspace*{\fill}
\begin{subfigure}[b]{0.4\linewidth}
\includegraphics[width=\linewidth]{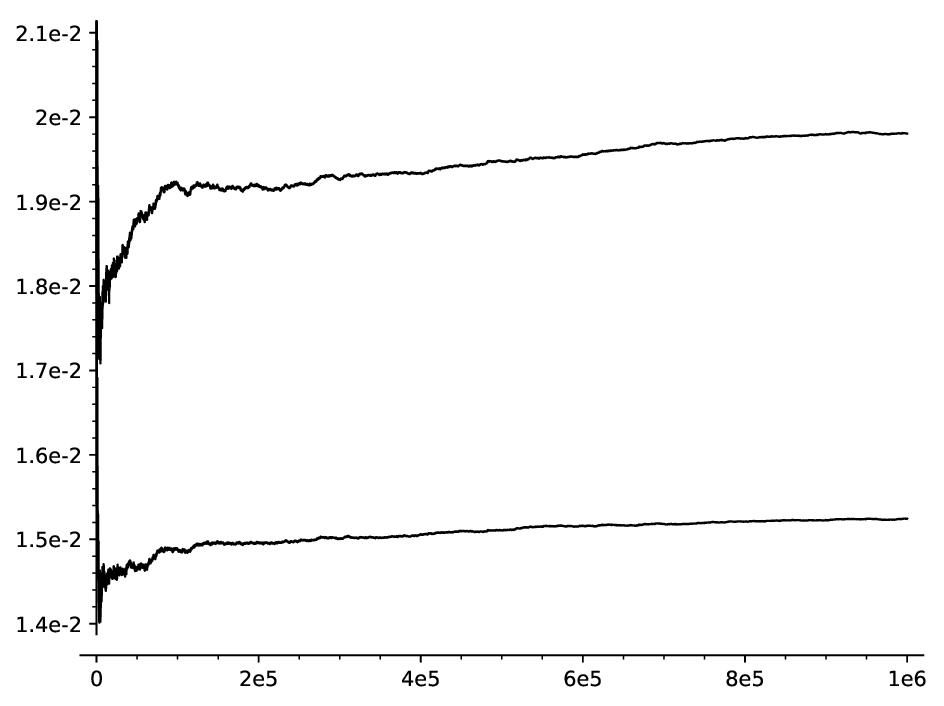}
\captionsetup{width=1.4\linewidth}
\caption{37b1: $(\alpha, \beta) = (2,6)$ Top to bottom $c =$ 0.3, 0.4} \label{fig:37_6_2_6_acc_c_orders}
\end{subfigure}
\caption{19a1, 37b1: Ratio~\eqref{ratio_M_orders} $m_{6,E}^{(\alpha, \beta)}(X;c)/X^{c +1/2}\log^2(X)$ for $c = $ 0.3, 0.4 and $k = 6$.} \label{fig:c_19_37_acc_6_orders}
\end{figure}

\clearpage

\begin{figure}[t] 
\hspace*{-2.3cm}
\begin{subfigure}[b]{0.4\linewidth}
\includegraphics[width=\linewidth]{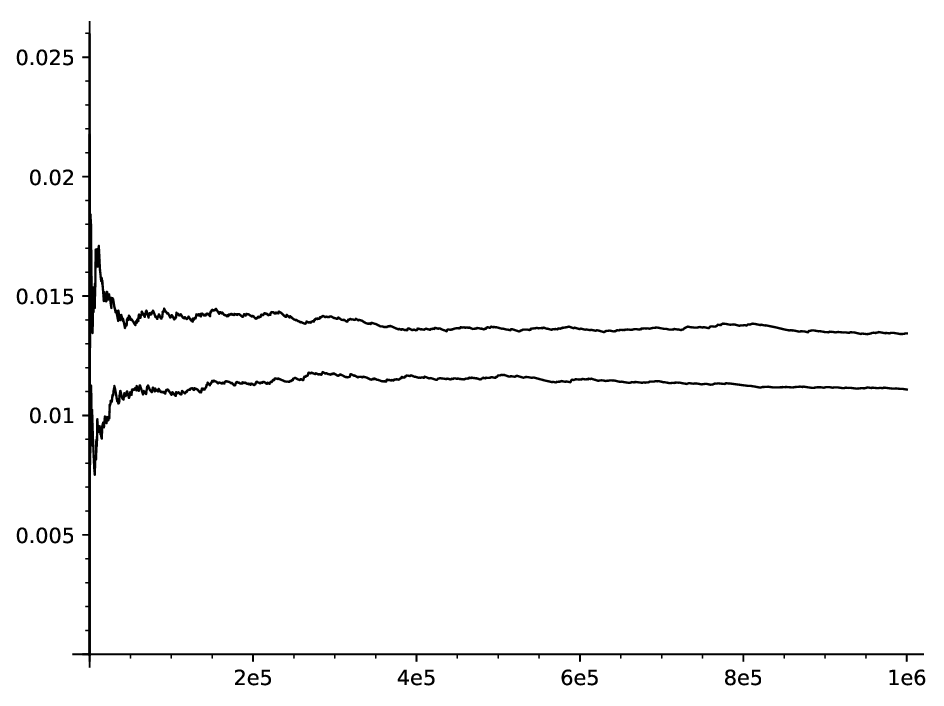}
\caption{$|l| = 1$: Top -1 bottom 1} \label{fig:11_6_1_3_A_1}
\end{subfigure}\hspace*{\fill}
\begin{subfigure}[b]{0.4\linewidth}
\includegraphics[width=\linewidth]{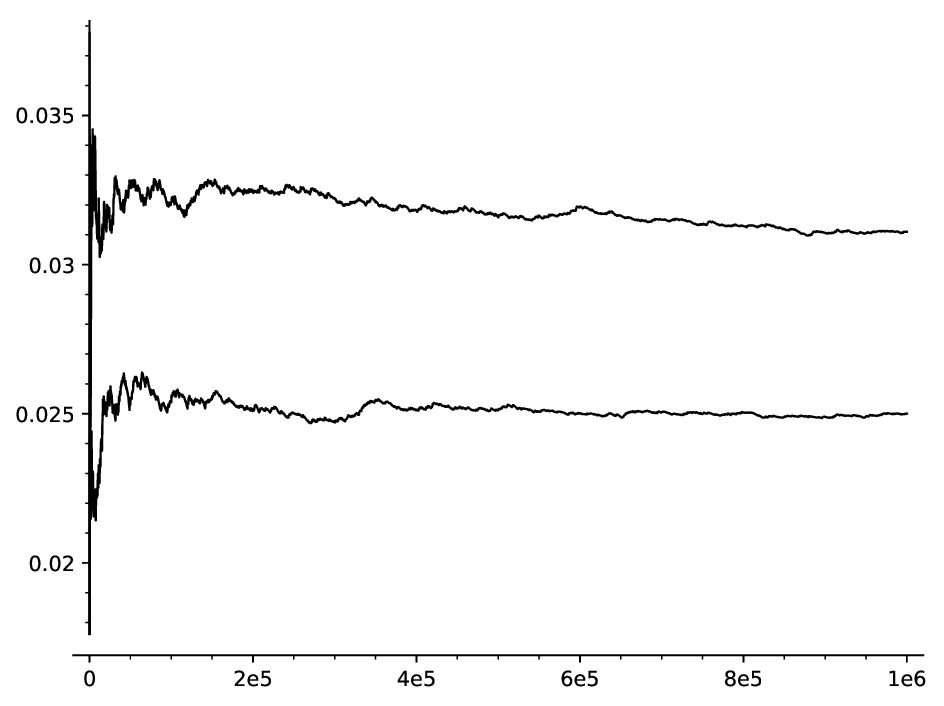}
\caption{$|l| = 2$: Top 2 bottom -2} \label{fig:11_6_1_3_A_2}
\end{subfigure}\hspace*{\fill}
\begin{subfigure}[b]{0.4\linewidth}
\includegraphics[width=\linewidth]{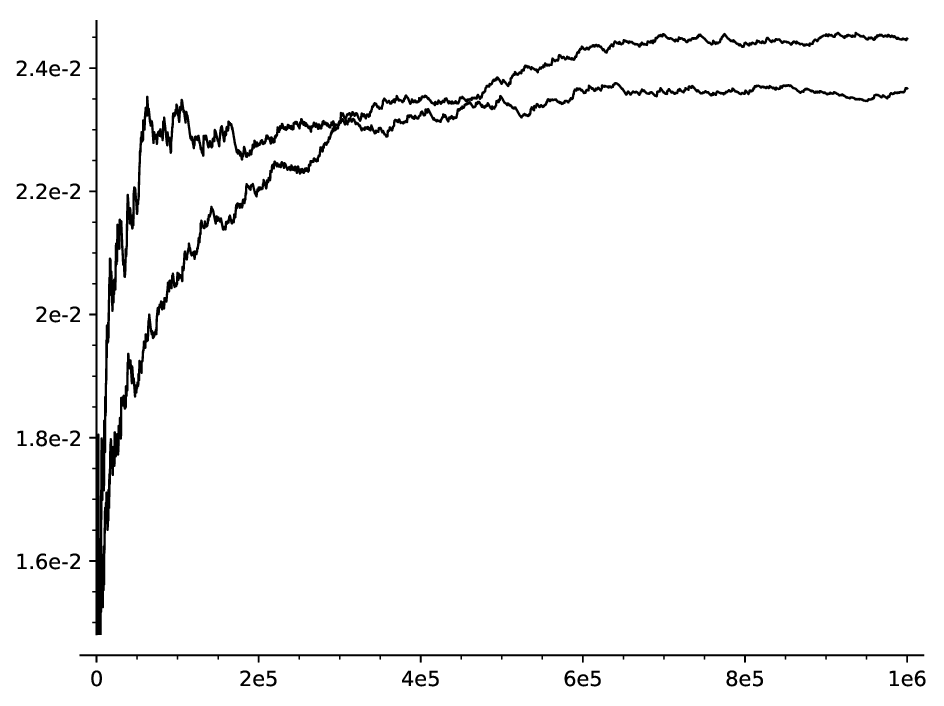}
\caption{$|l| = 3$: Top -3 bottom 3} \label{fig:11_6_1_3_A_3}
\end{subfigure}
\hspace*{-2.3cm}
\begin{subfigure}[b]{0.4\linewidth}
\includegraphics[width=\linewidth]{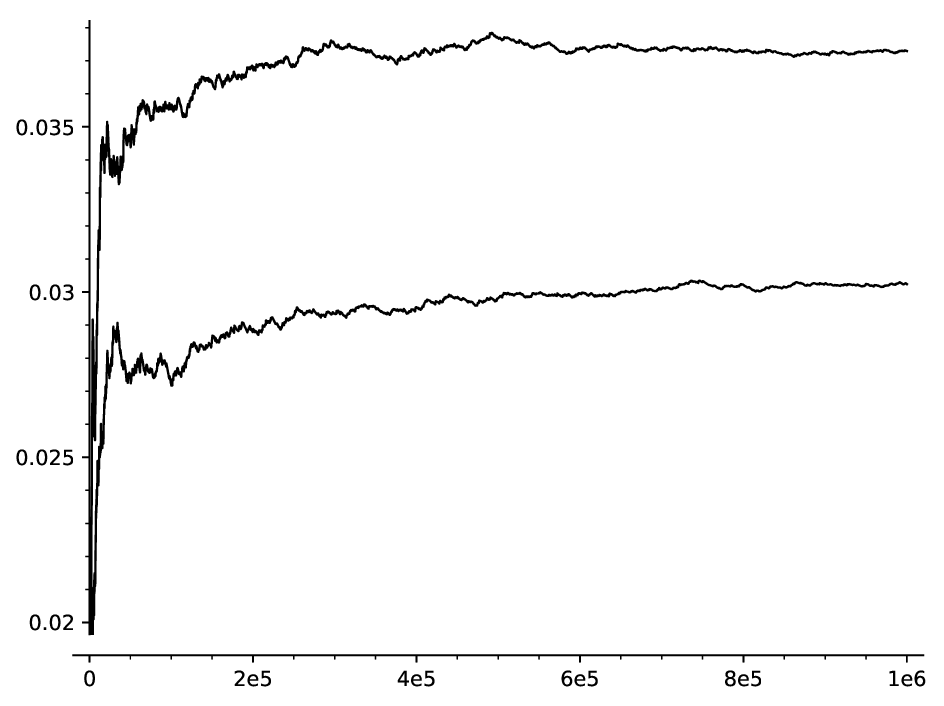}
\caption{$|l| = 4$: Top -4 bottom 4} \label{fig:11_6_1_3_A_4}
\end{subfigure}\hspace*{\fill}
\begin{subfigure}[b]{0.4\linewidth}
\includegraphics[width=\linewidth]{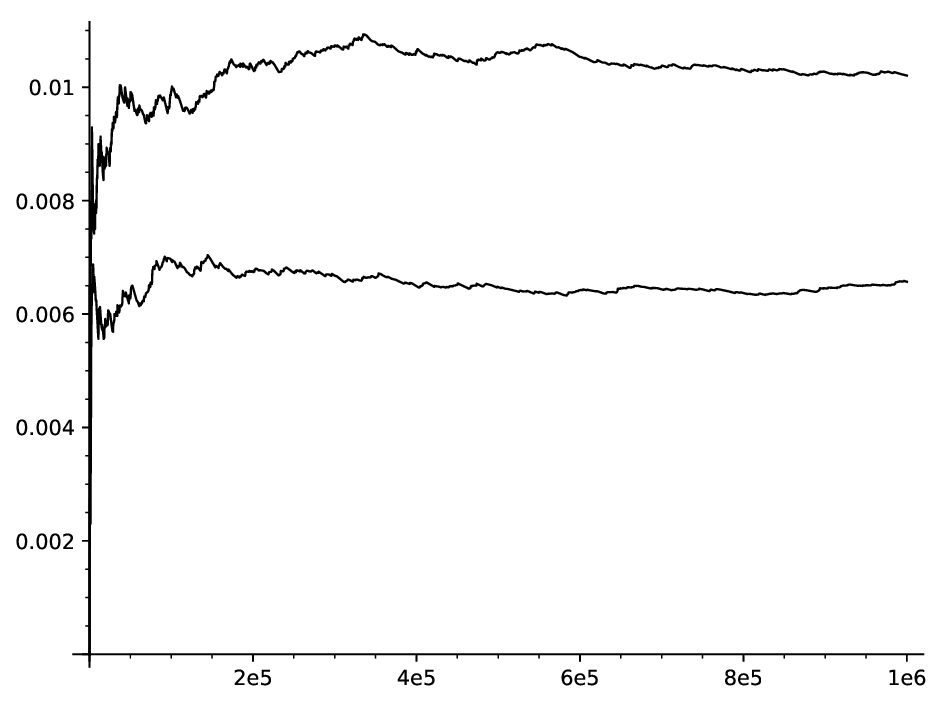}
\caption{$|l| = 5$: Top 5 bottom -5} \label{fig:11_6_1_3_A_5}
\end{subfigure}\hspace*{\fill}
\begin{subfigure}[b]{0.4\linewidth}
\includegraphics[width=\linewidth]{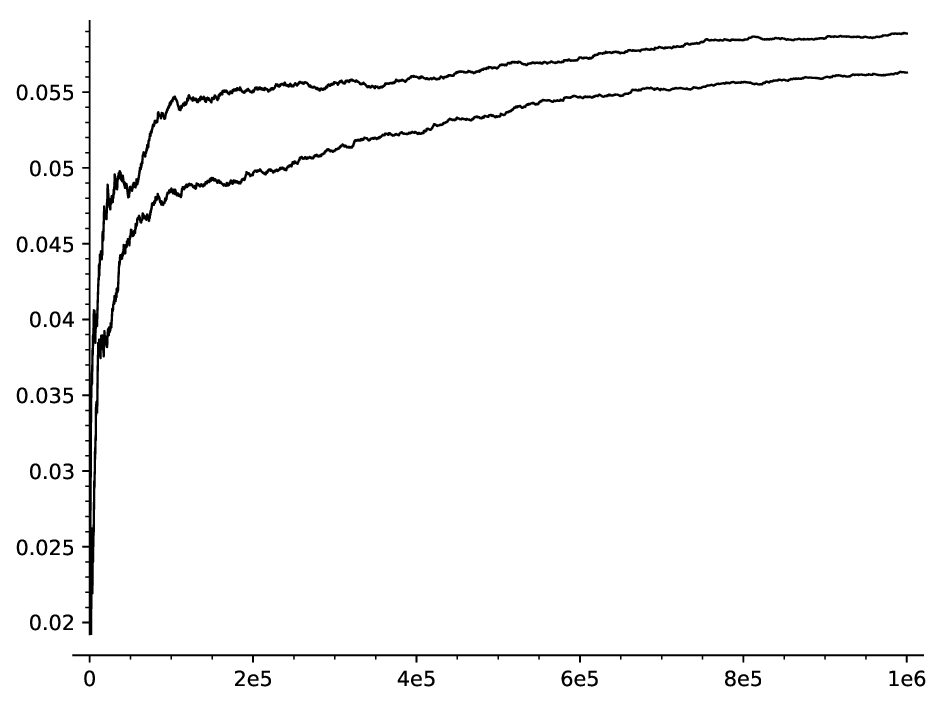}
\caption{$|l| = 6$: Top 6 bottom -6} \label{fig:11_6_1_3_A_6}
\end{subfigure}
\hspace*{-2.3cm}
\begin{subfigure}[b]{0.4\linewidth}
\includegraphics[width=\linewidth]{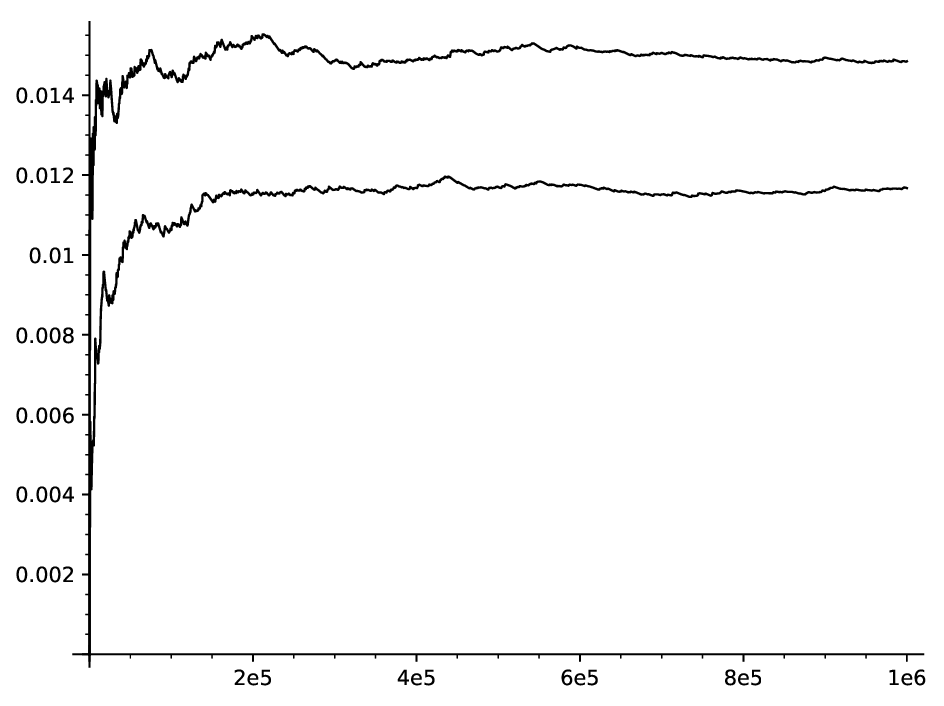}
\caption{$|l| = 7$: Top -7 bottom 7} \label{fig:11_6_1_3_A_7}
\end{subfigure}\hspace*{\fill}
\begin{subfigure}[b]{0.4\linewidth}
\includegraphics[width=\linewidth]{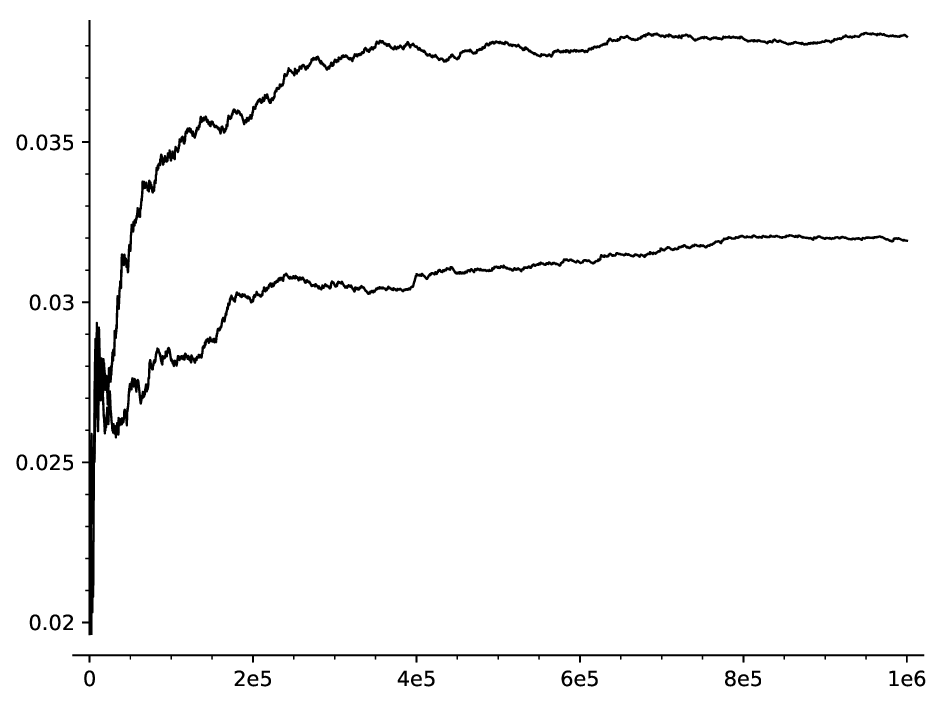}
\caption{$|l| = 8$: Top 8 bottom -8} \label{fig:11_6_1_3_A_8}
\end{subfigure}\hspace*{\fill}
\begin{subfigure}[b]{0.4\linewidth}
\includegraphics[width=\linewidth]{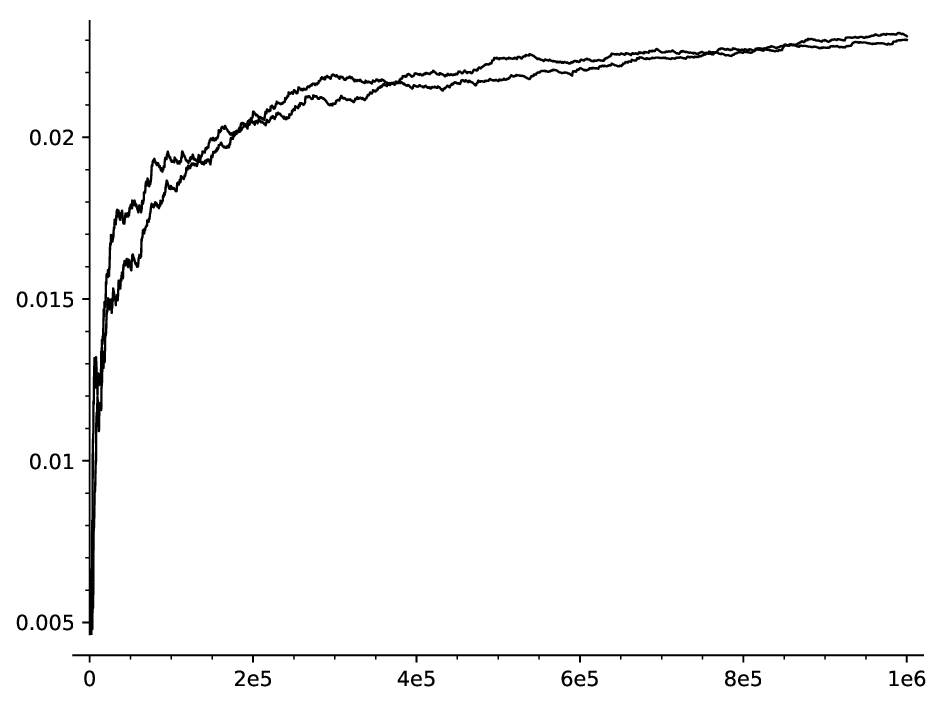}
\caption{$|l| = 9$: Top -9 bottom 9} \label{fig:11_6_1_3_A_9}
\end{subfigure}
\caption{11a1: $(\alpha, \beta) = (1,3)$ Ratio~\eqref{ratio_n_orders} $x_{6,E}^{(\alpha, \beta)}(X;l)/X^{1/2}\log^2(X)$} \label{fig:11a1_6_1_3_A_exact}
\end{figure}

\clearpage

\begin{figure}[t] 
\hspace*{-2.3cm}
\begin{subfigure}[b]{0.4\linewidth}
\includegraphics[width=\linewidth]{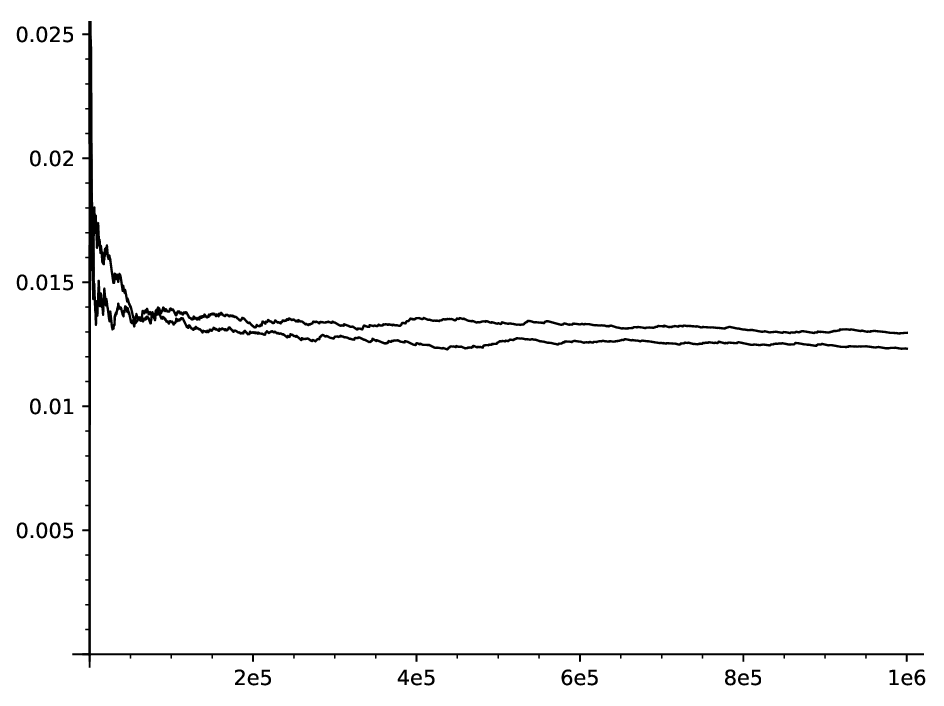}
\caption{$|l| = 1$: Top 1 bottom -1} \label{fig:11_6_2_3_A_1}
\end{subfigure}\hspace*{\fill}
\begin{subfigure}[b]{0.4\linewidth}
\includegraphics[width=\linewidth]{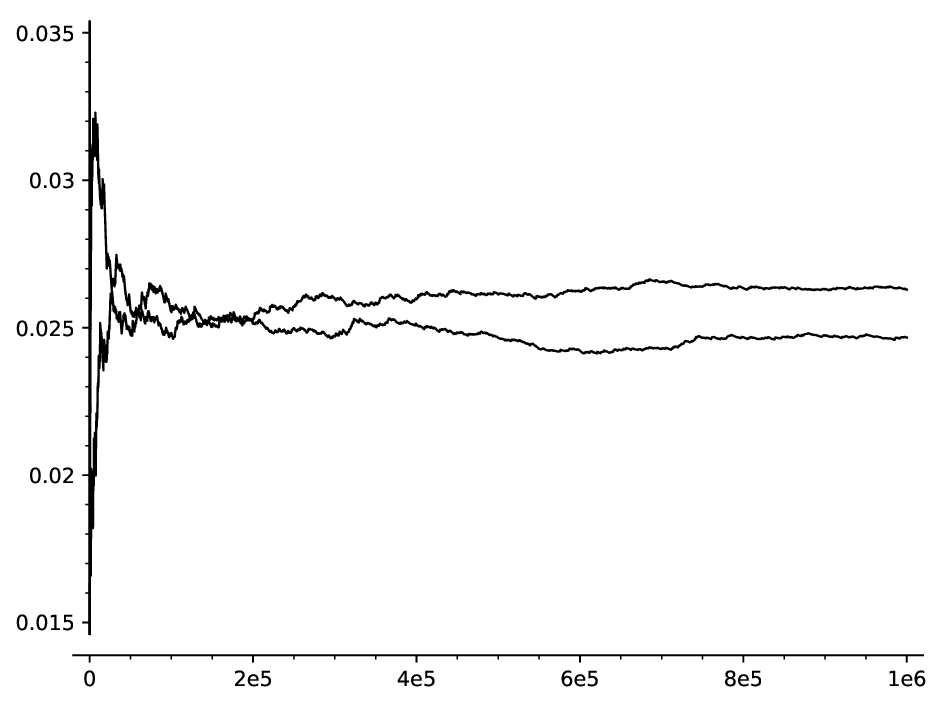}
\caption{$|l| = 2$: Top -2 bottom 2} \label{fig:11_6_2_3_A_2}
\end{subfigure}\hspace*{\fill}
\begin{subfigure}[b]{0.4\linewidth}
\includegraphics[width=\linewidth]{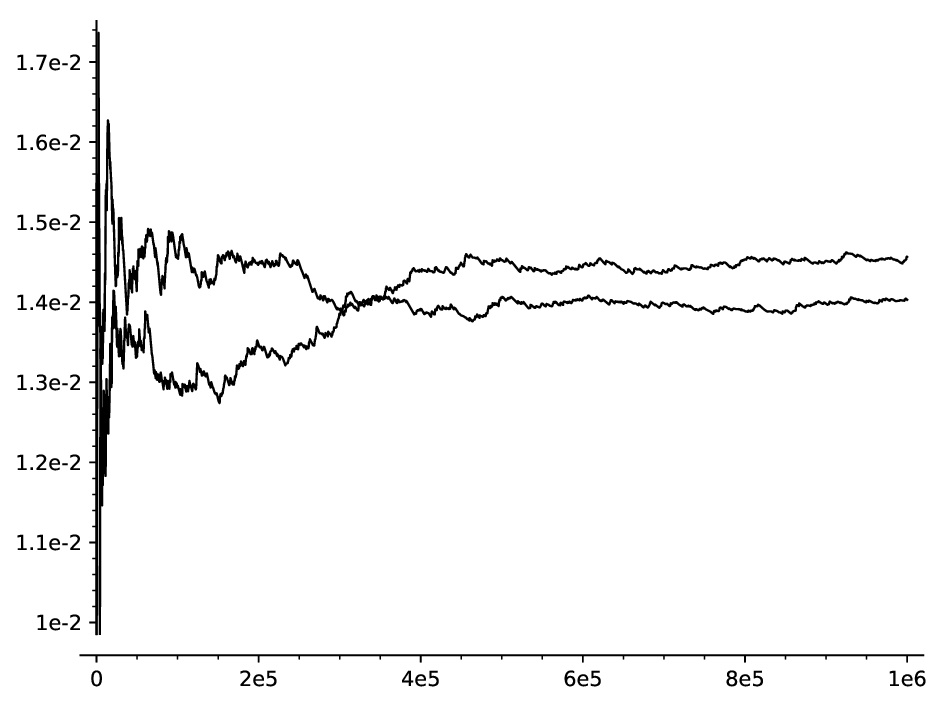}
\caption{$|l| = 3$: Top -3 bottom 3} \label{fig:11_6_2_3_A_3}
\end{subfigure}
\hspace*{-2.3cm}
\begin{subfigure}[b]{0.4\linewidth}
\includegraphics[width=\linewidth]{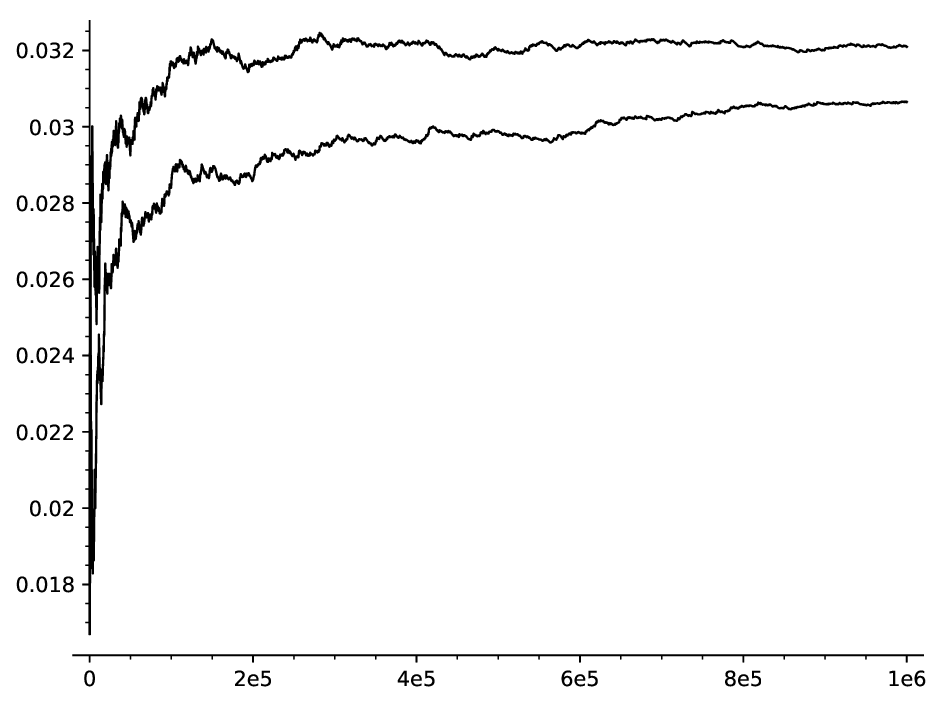}
\caption{$|l| = 4$: Top 4 bottom -4} \label{fig:11_6_2_3_A_4}
\end{subfigure}\hspace*{\fill}
\begin{subfigure}[b]{0.4\linewidth}
\includegraphics[width=\linewidth]{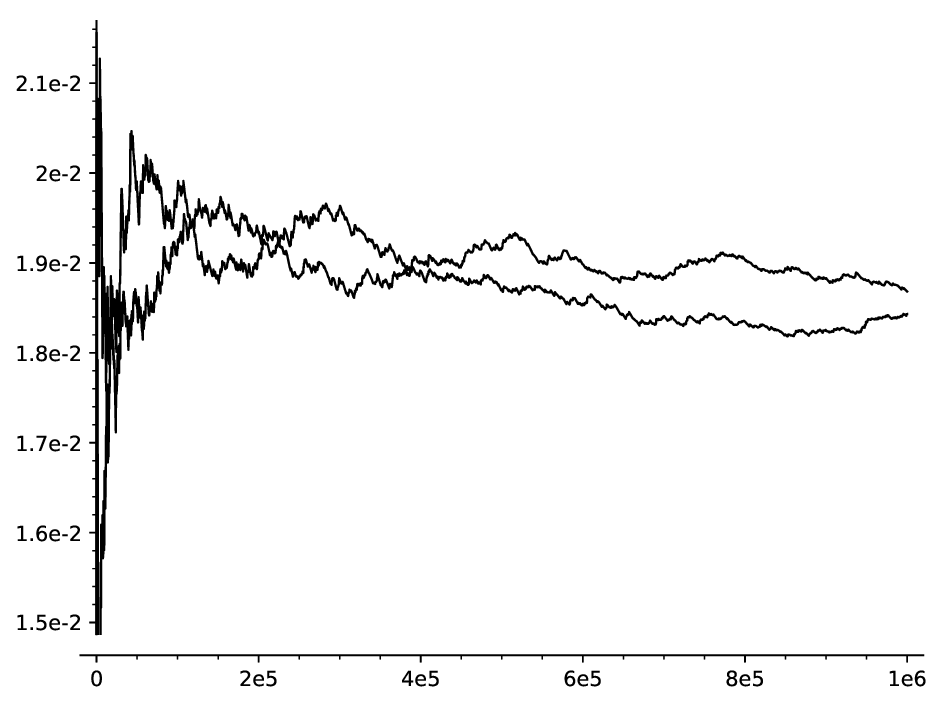}
\caption{$|l| = 5$: Top 5 bottom -5} \label{fig:11_6_2_3_A_5}
\end{subfigure}\hspace*{\fill}
\begin{subfigure}[b]{0.4\linewidth}
\includegraphics[width=\linewidth]{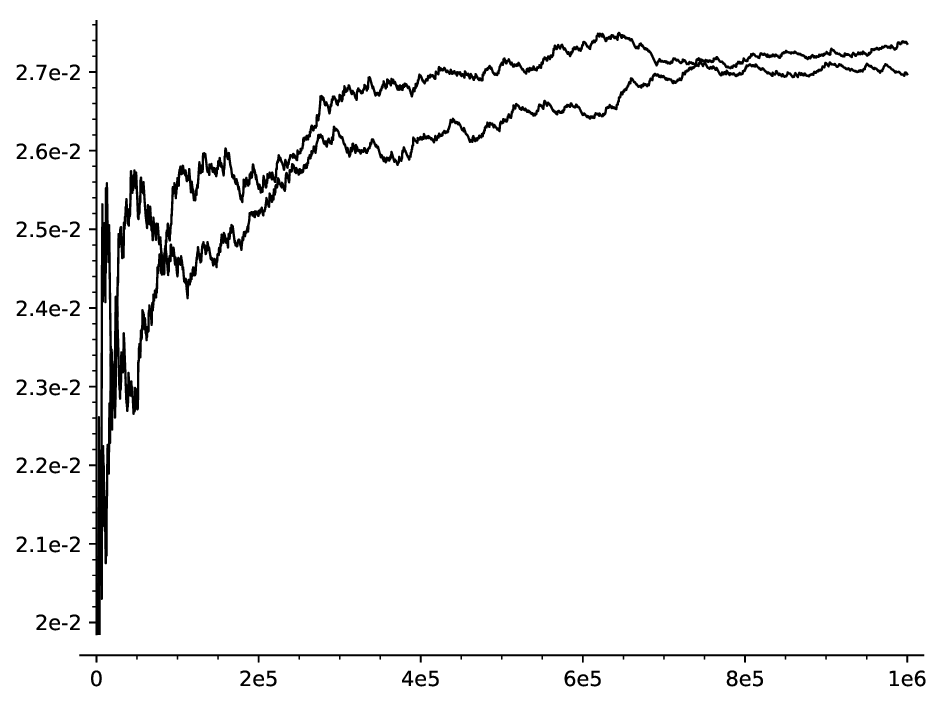}
\caption{$|l| = 6$: Top 6 bottom -6} \label{fig:11_6_2_3_A_6}
\end{subfigure}
\hspace*{-2.3cm}
\begin{subfigure}[b]{0.4\linewidth}
\includegraphics[width=\linewidth]{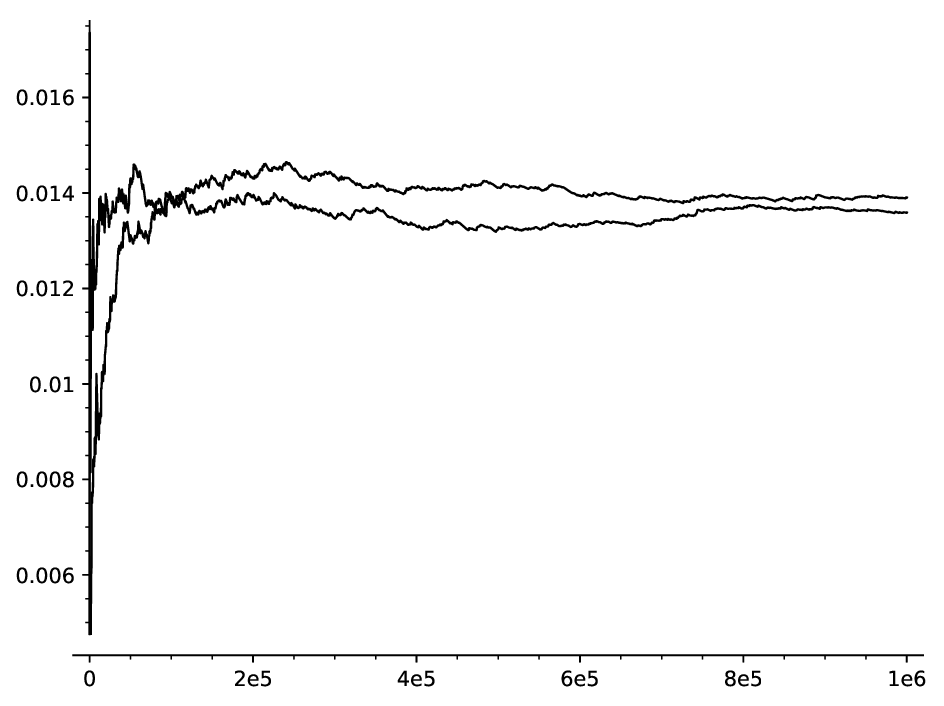}
\caption{$|l| = 7$: Top 7 bottom -7} \label{fig:11_6_2_3_A_7}
\end{subfigure}\hspace*{\fill}
\begin{subfigure}[b]{0.4\linewidth}
\includegraphics[width=\linewidth]{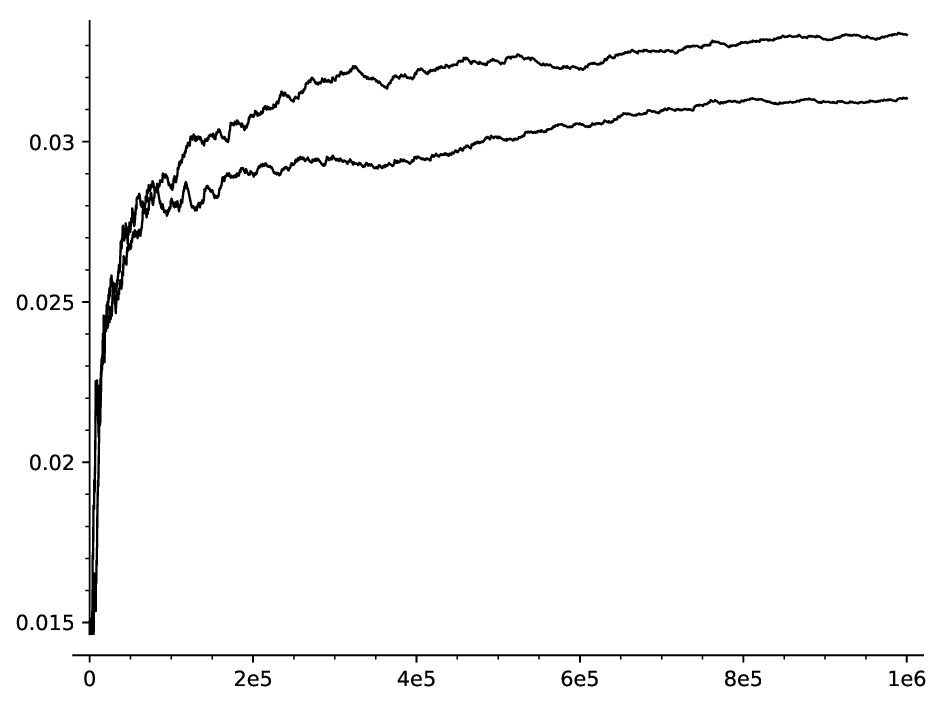}
\caption{$|l| = 8$: Top -8 bottom 8} \label{fig:11_6_2_3_A_8}
\end{subfigure}\hspace*{\fill}
\begin{subfigure}[b]{0.4\linewidth}
\includegraphics[width=\linewidth]{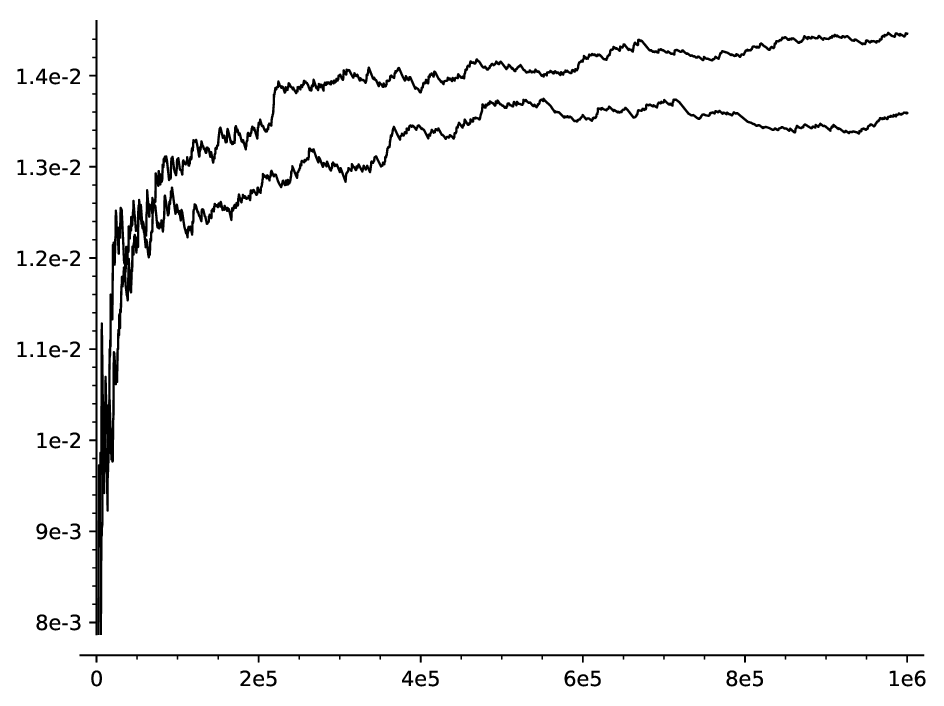}
\caption{$|l| = 9$: Top 9 bottom -9} \label{fig:11_6_2_3_A_9}
\end{subfigure}
\caption{11a1: $(\alpha, \beta) = (2,3)$ Ratio~\eqref{ratio_n_orders} $x_{6,E}^{(\alpha, \beta)}(X;l)/X^{1/2}\log^2(X)$} \label{fig:11a1_6_2_3_A_exact}
\end{figure}

\clearpage

\begin{figure}[t] 
\hspace*{-2.3cm}
\begin{subfigure}[b]{0.4\linewidth}
\includegraphics[width=\linewidth]{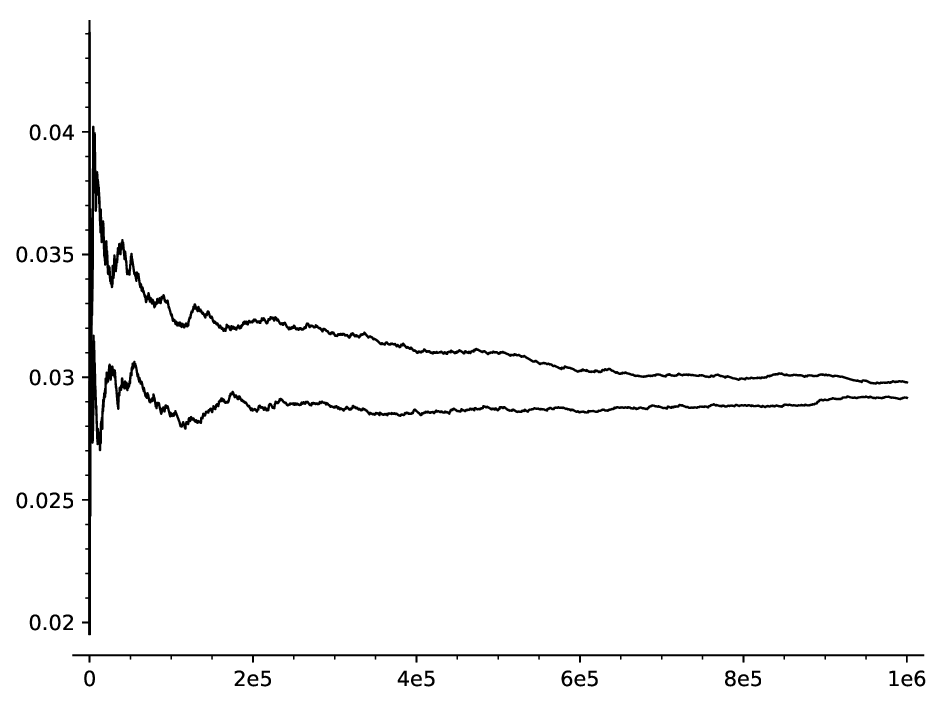}
\caption{$|l| = 1$: Top 1 bottom -1} \label{fig:11_6_1_6_A_1}
\end{subfigure}\hspace*{\fill}
\begin{subfigure}[b]{0.4\linewidth}
\includegraphics[width=\linewidth]{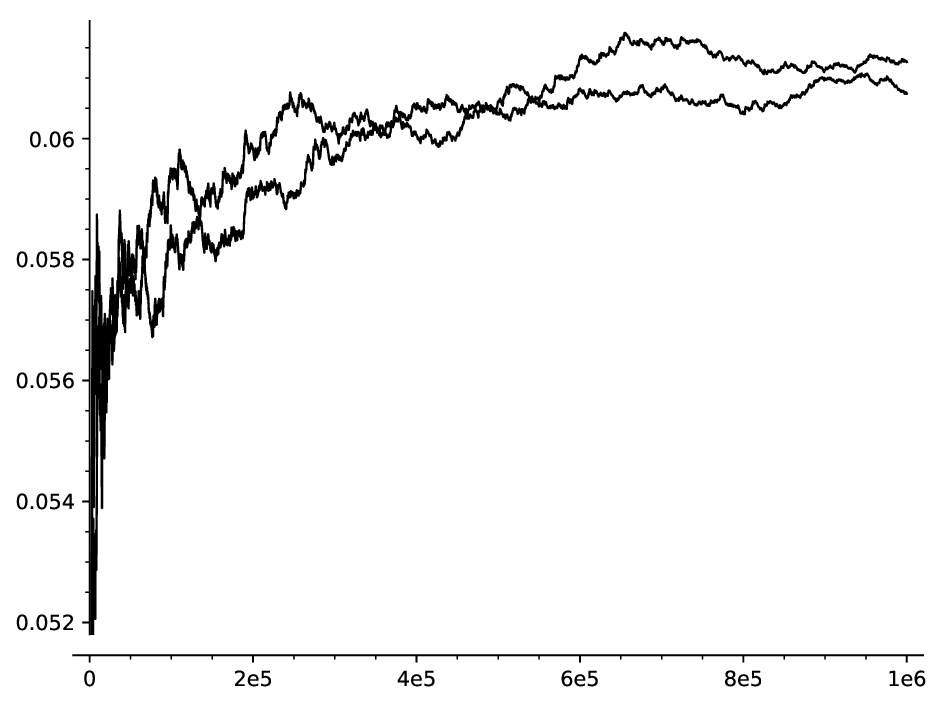}
\caption{$|l| = 2$: Top -2 bottom 2} \label{fig:11_6_1_6_A_2}
\end{subfigure}\hspace*{\fill}
\begin{subfigure}[b]{0.4\linewidth}
\includegraphics[width=\linewidth]{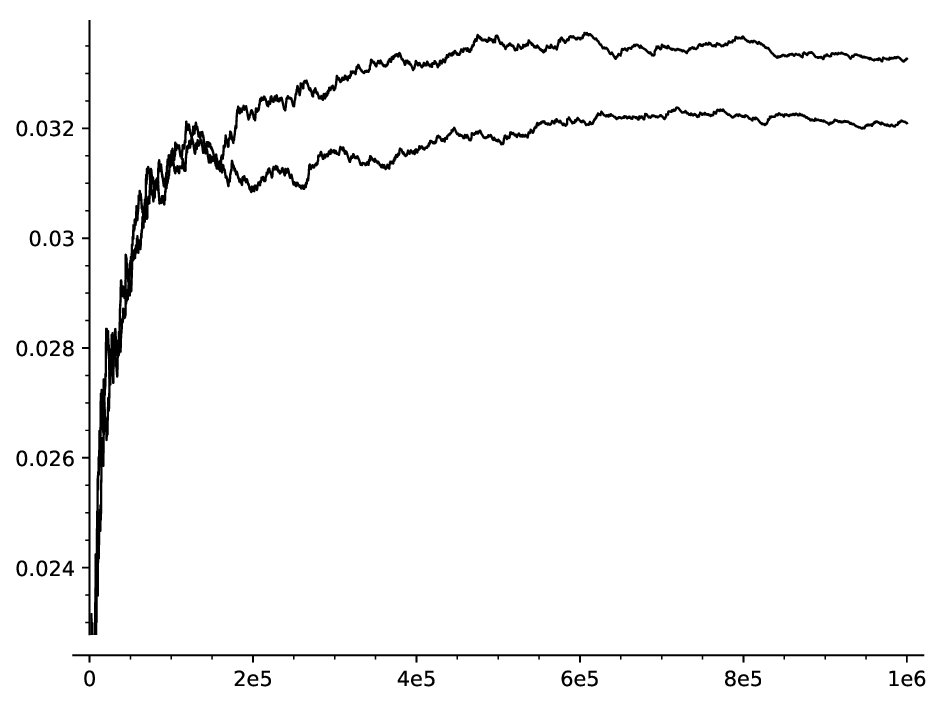}
\caption{$|l| = 3$: Top 3 bottom -3} \label{fig:11_6_1_6_A_3}
\end{subfigure}
\hspace*{-2.3cm}
\begin{subfigure}[b]{0.4\linewidth}
\includegraphics[width=\linewidth]{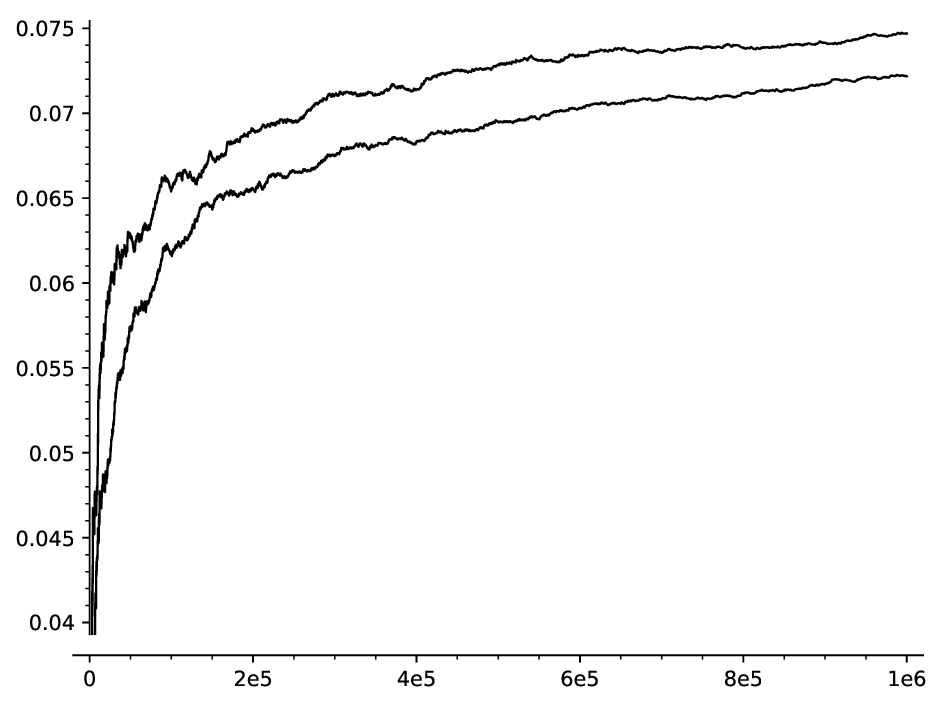}
\caption{$|l| = 4$: Top 4 bottom -4} \label{fig:11_6_1_6_A_4}
\end{subfigure}\hspace*{\fill}
\begin{subfigure}[b]{0.4\linewidth}
\includegraphics[width=\linewidth]{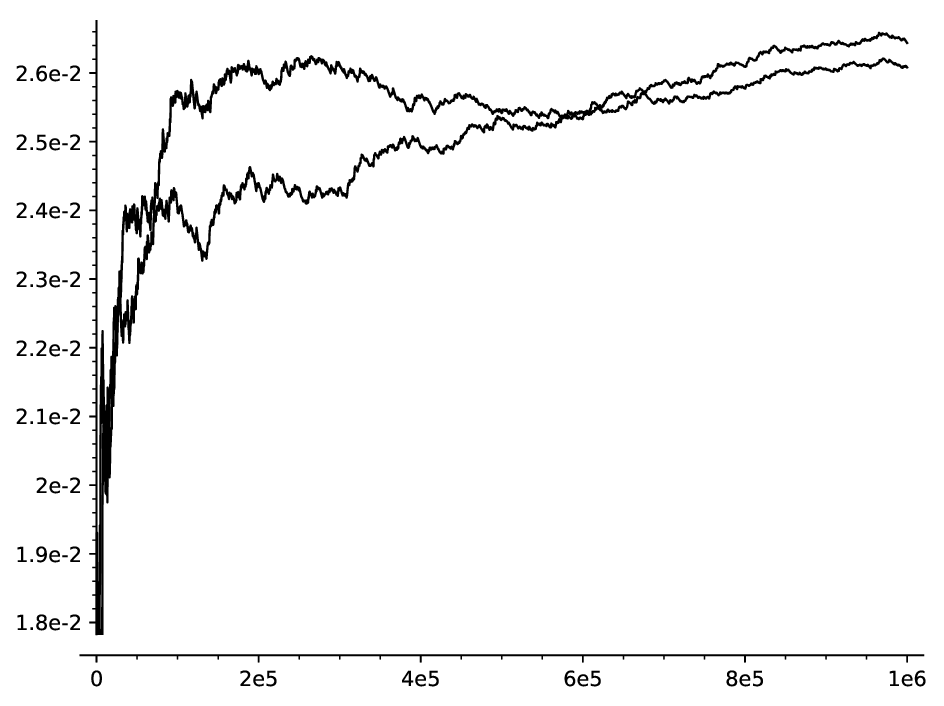}
\caption{$|l| = 5$: Top 5 bottom -5} \label{fig:11_6_1_6_A_5}
\end{subfigure}\hspace*{\fill}
\begin{subfigure}[b]{0.4\linewidth}
\includegraphics[width=\linewidth]{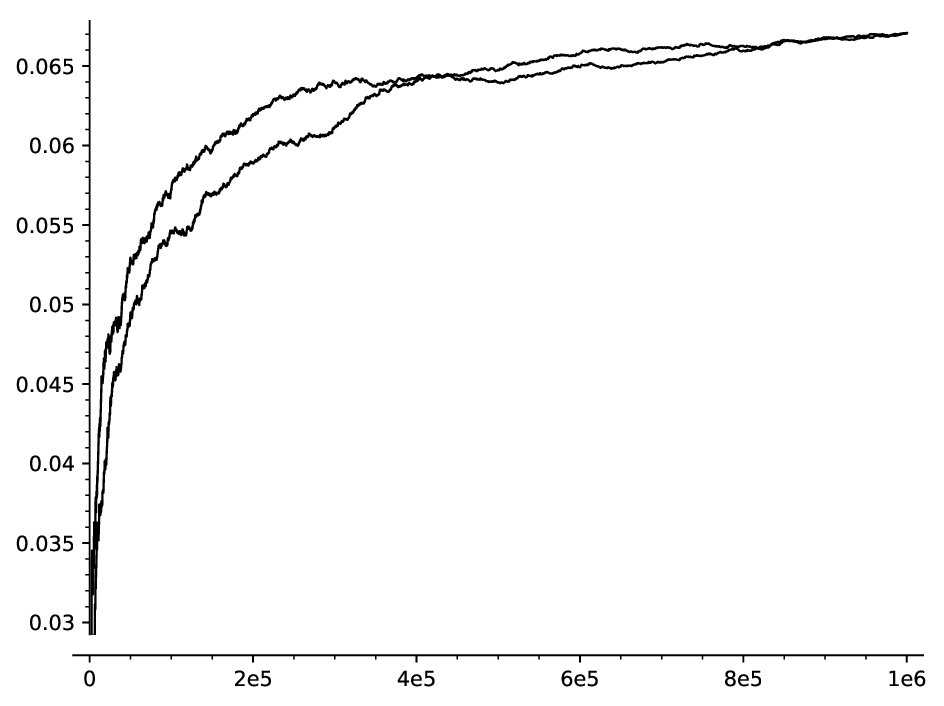}
\caption{$|l| = 6$: Top -6 bottom 6} \label{fig:11_6_1_6_A_6}
\end{subfigure}
\hspace*{-2.3cm}
\begin{subfigure}[b]{0.4\linewidth}
\includegraphics[width=\linewidth]{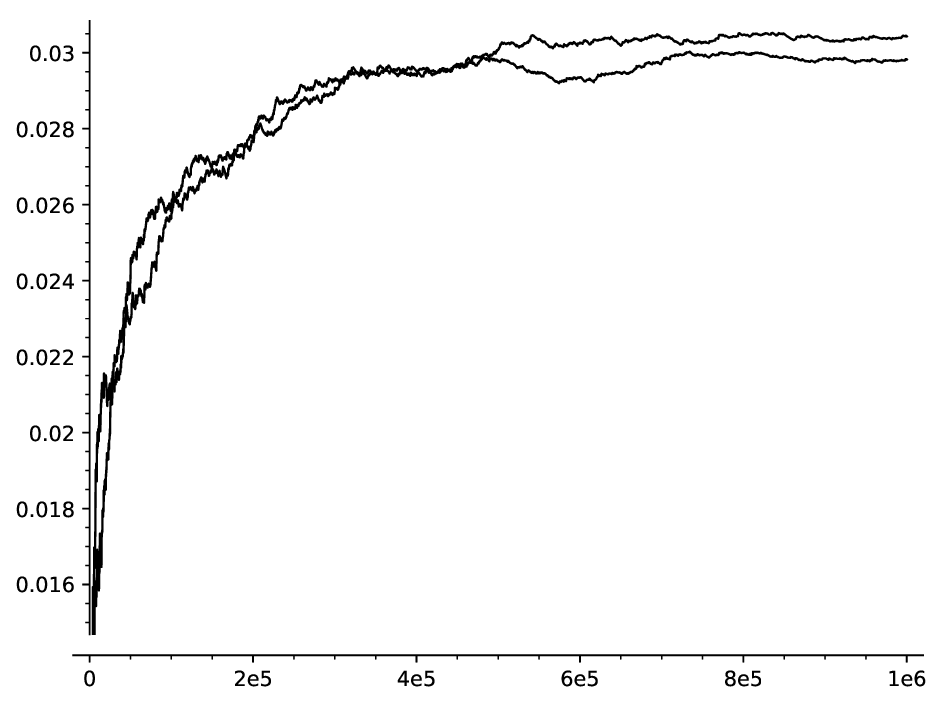}
\caption{$|l| = 7$: Top 7 bottom -7} \label{fig:11_6_1_6_A_7}
\end{subfigure}\hspace*{\fill}
\begin{subfigure}[b]{0.4\linewidth}
\includegraphics[width=\linewidth]{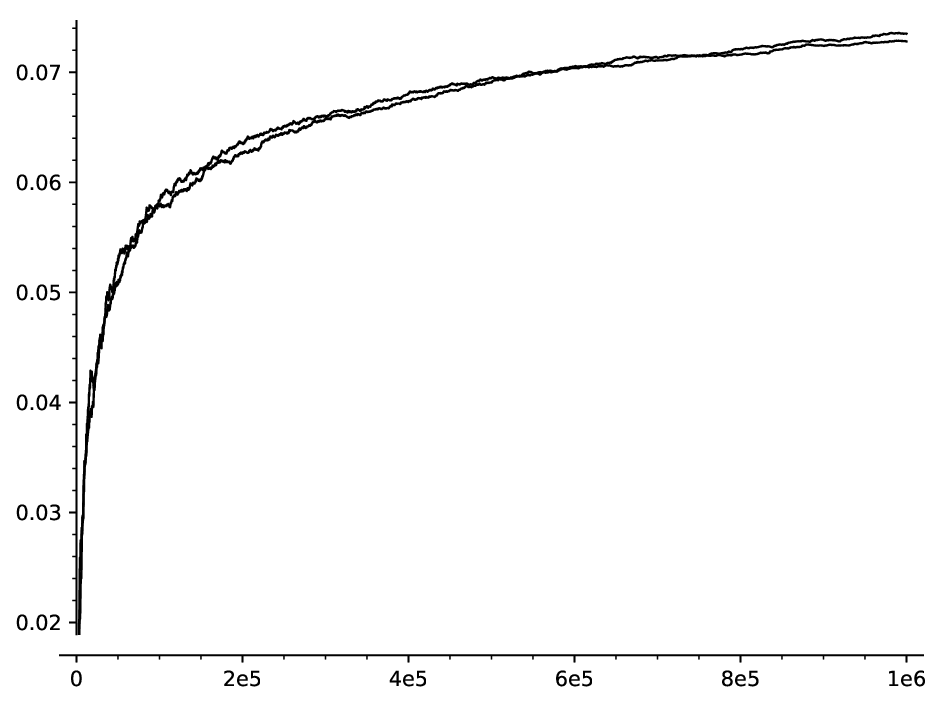}
\caption{$|l| = 8$: Top -8 bottom 8} \label{fig:11_6_1_6_A_8}
\end{subfigure}\hspace*{\fill}
\begin{subfigure}[b]{0.4\linewidth}
\includegraphics[width=\linewidth]{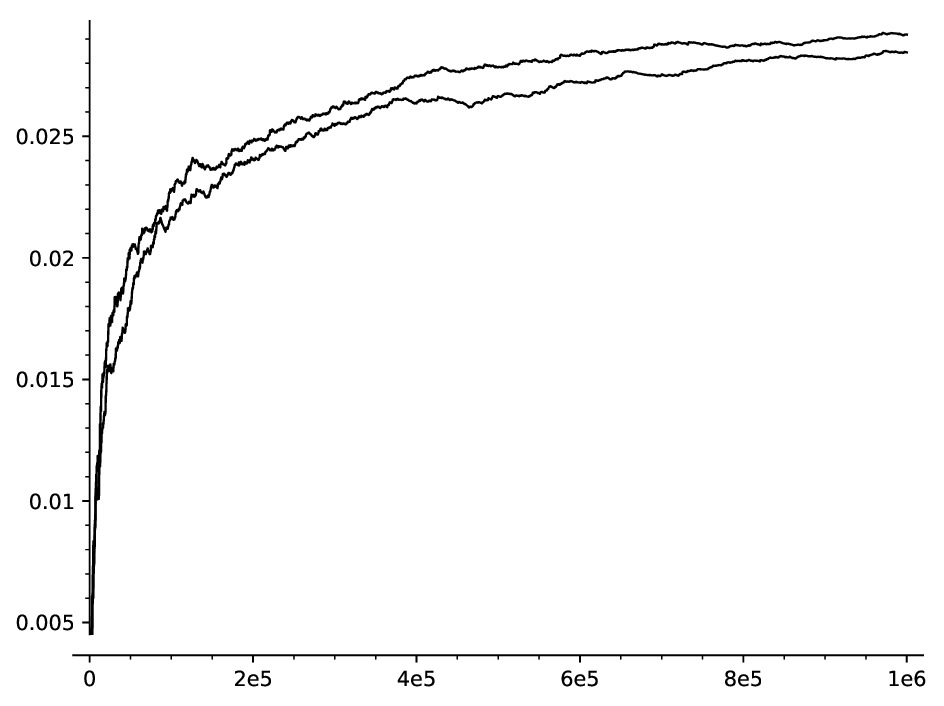}
\caption{$|l| = 9$: Top 9 bottom -9} \label{fig:11_6_1_6_A_9}
\end{subfigure}
\caption{11a1: $(\alpha, \beta) = (1,6)$ Ratio~\eqref{ratio_n_orders} $x_{6,E}^{(\alpha, \beta)}(X;l)/X^{1/2}\log^2(X)$} \label{fig:11a1_6_1_6_A_exact}
\end{figure}

\clearpage

\begin{figure}[t] 
\hspace*{-2.3cm}
\begin{subfigure}[b]{0.4\linewidth}
\includegraphics[width=\linewidth]{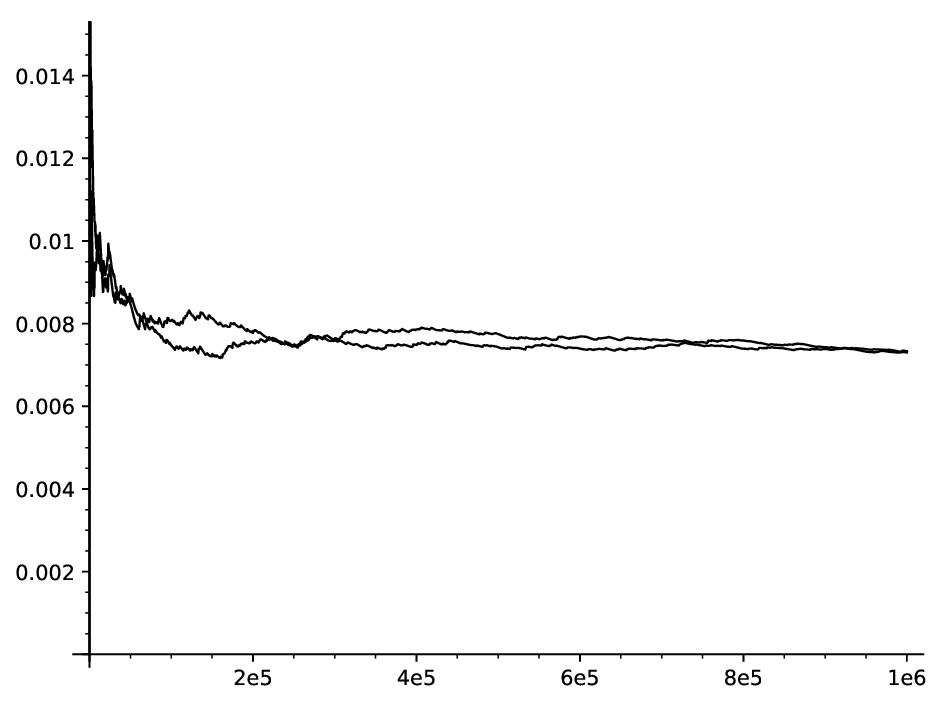}
\caption{$|l| = 1$: Top -1 bottom 1} \label{fig:11_6_2_6_A_1}
\end{subfigure}\hspace*{\fill}
\begin{subfigure}[b]{0.4\linewidth}
\includegraphics[width=\linewidth]{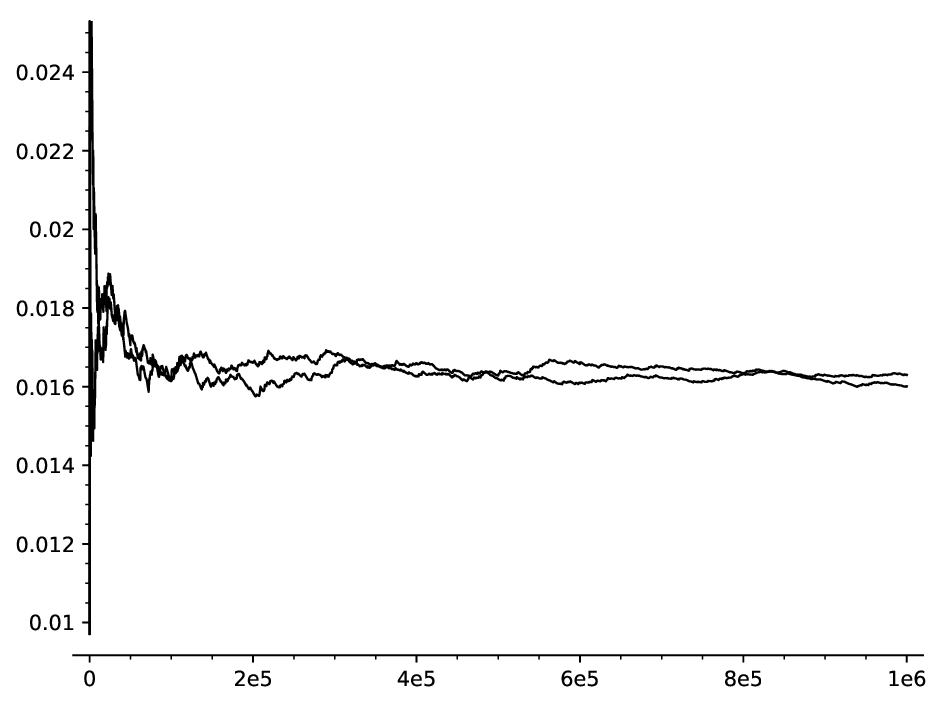}
\caption{$|l| = 2$: Top -2 bottom 2} \label{fig:11_6_2_6_A_2}
\end{subfigure}\hspace*{\fill}
\begin{subfigure}[b]{0.4\linewidth}
\includegraphics[width=\linewidth]{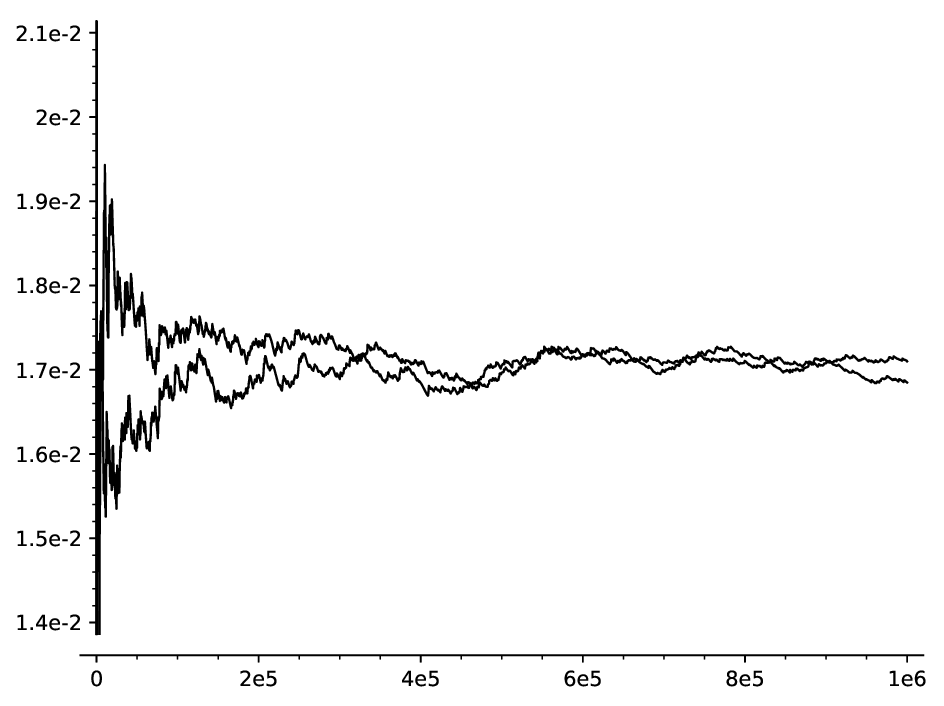}
\caption{$|l| = 3$: Top -3 bottom 3} \label{fig:11_6_2_6_A_3}
\end{subfigure}
\hspace*{-2.3cm}
\begin{subfigure}[b]{0.4\linewidth}
\includegraphics[width=\linewidth]{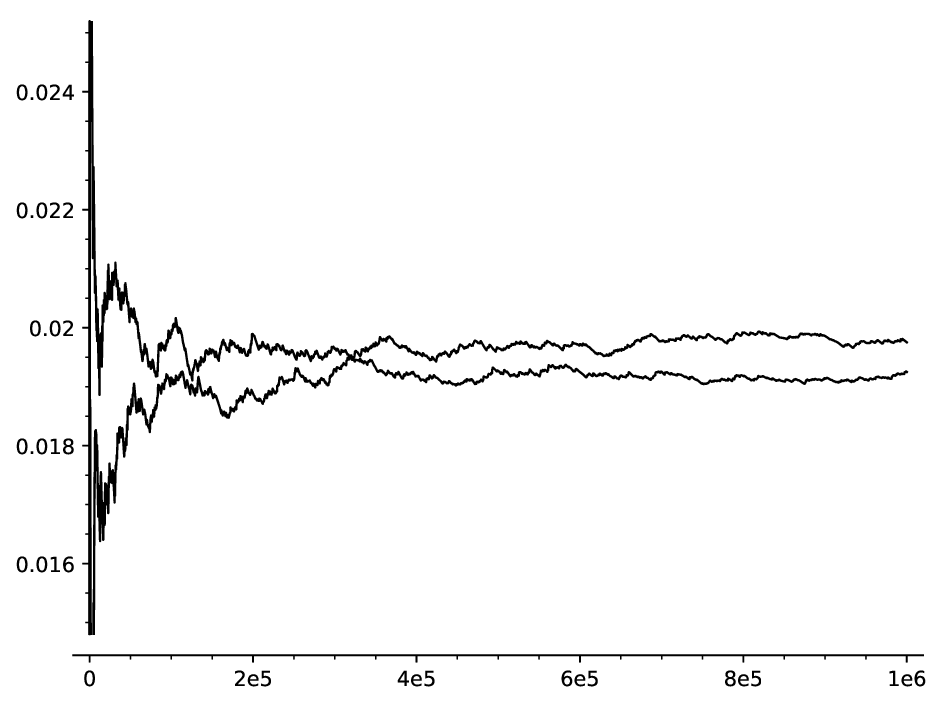}
\caption{$|l| = 4$: Top -4 bottom 4} \label{fig:11_6_2_6_A_4}
\end{subfigure}\hspace*{\fill}
\begin{subfigure}[b]{0.4\linewidth}
\caption{$|l| = 5$: Not occur} \label{fig:11_6_2_6_A_5}
\end{subfigure}\hspace*{\fill}
\begin{subfigure}[b]{0.4\linewidth}
\includegraphics[width=\linewidth]{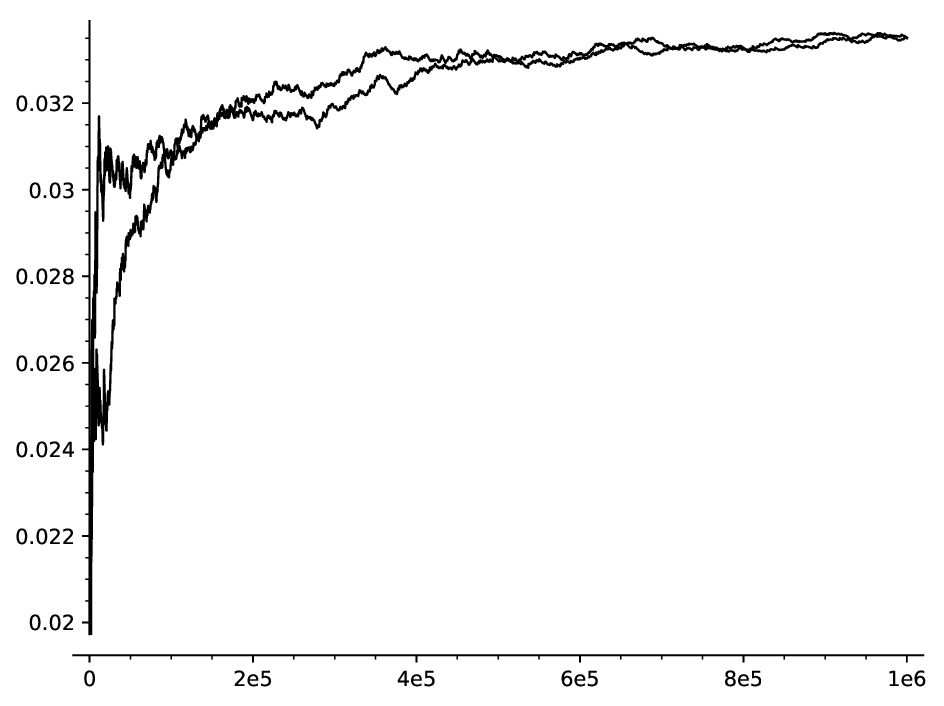}
\caption{$|l| = 6$: Top 6 bottom -6} \label{fig:11_6_2_6_A_6}
\end{subfigure}
\hspace*{-2.3cm}
\begin{subfigure}[b]{0.4\linewidth}
\includegraphics[width=\linewidth]{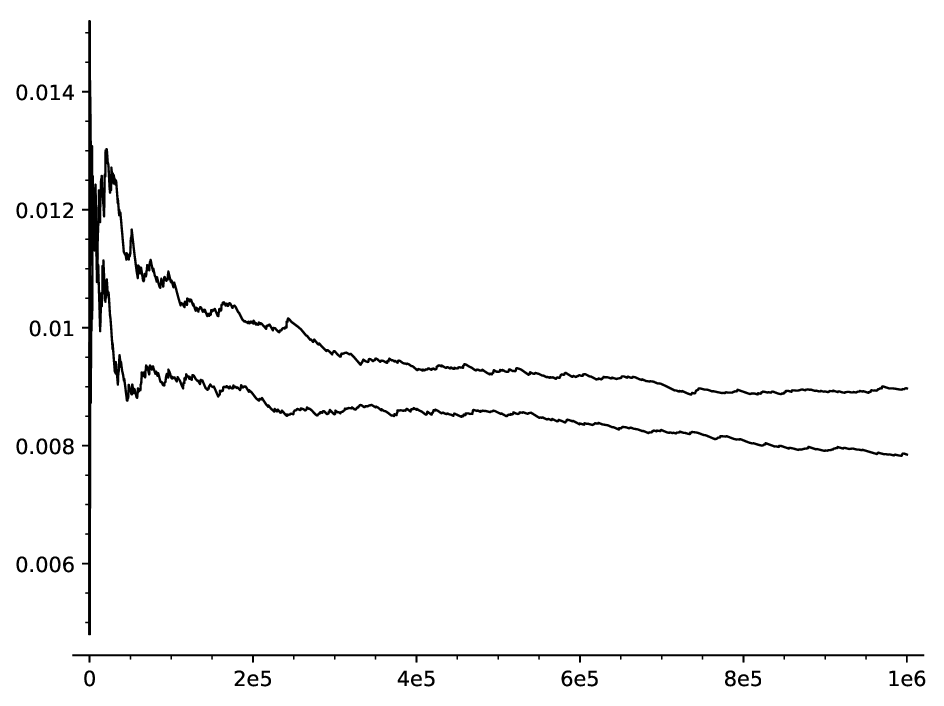}
\caption{$|l| = 7$: Top -7 bottom 7} \label{fig:11_6_2_6_A_7}
\end{subfigure}\hspace*{\fill}
\begin{subfigure}[b]{0.4\linewidth}
\includegraphics[width=\linewidth]{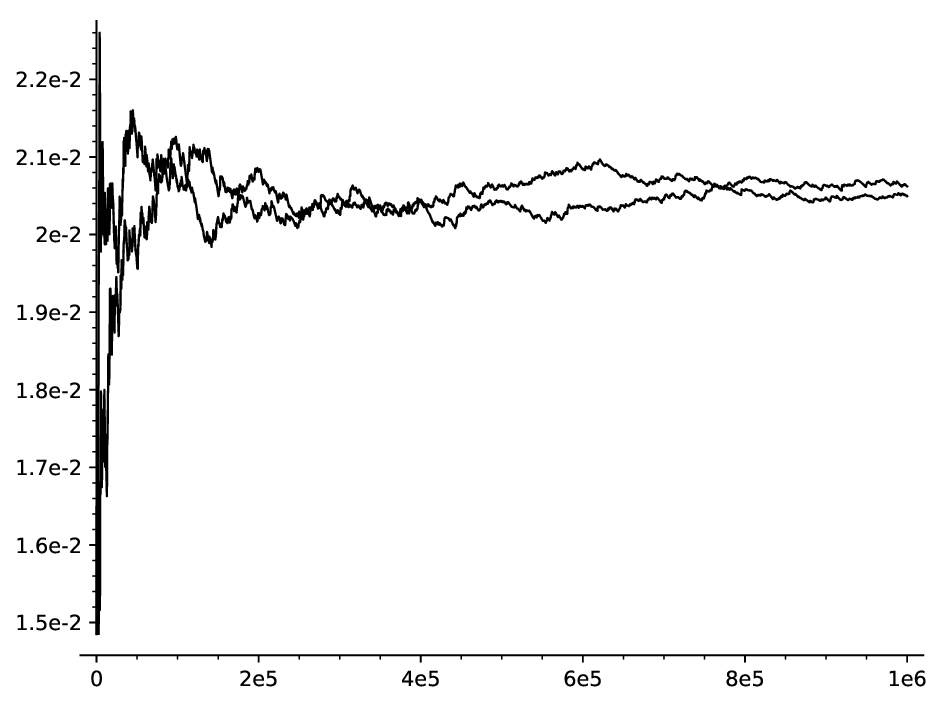}
\caption{$|l| = 8$: Top 8 bottom -8} \label{fig:11_6_2_6_A_8}
\end{subfigure}\hspace*{\fill}
\begin{subfigure}[b]{0.4\linewidth}
\includegraphics[width=\linewidth]{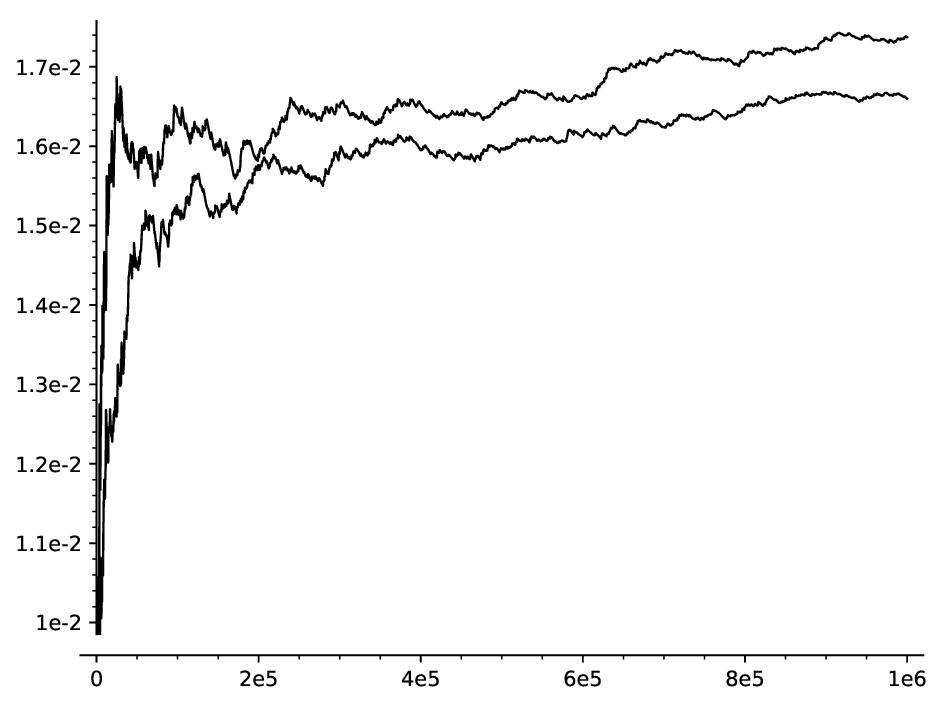}
\caption{$|l| = 9$: Top -9 bottom 9} \label{fig:11_6_2_6_A_9}
\end{subfigure}
\caption{11a1: $(\alpha, \beta) = (2,6)$ Ratio~\eqref{ratio_n_orders} $x_{6,E}^{(\alpha, \beta)}(X;l)/X^{1/2}\log^2(X)$} \label{fig:11a1_6_2_6_A_exact}
\end{figure}

\clearpage

\begin{figure}[t] 
\hspace*{-2.3cm}
\begin{subfigure}[b]{0.4\linewidth}
\includegraphics[width=\linewidth]{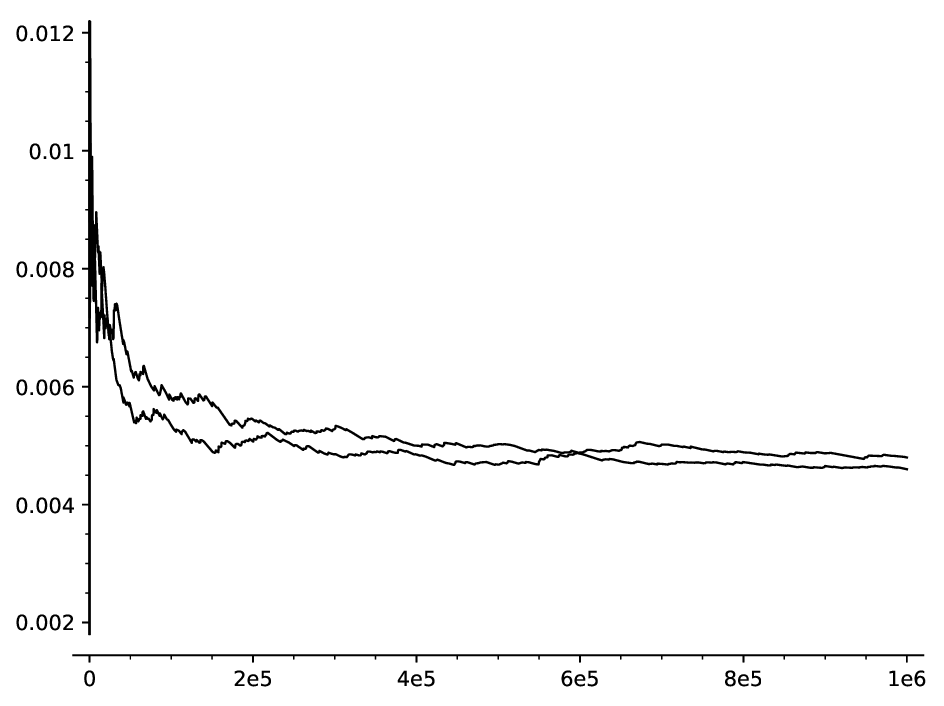}
\caption{$|l| = 1$: Top 1 bottom -1} \label{fig:14_6_1_3_A_1}
\end{subfigure}\hspace*{\fill}
\begin{subfigure}[b]{0.4\linewidth}
\includegraphics[width=\linewidth]{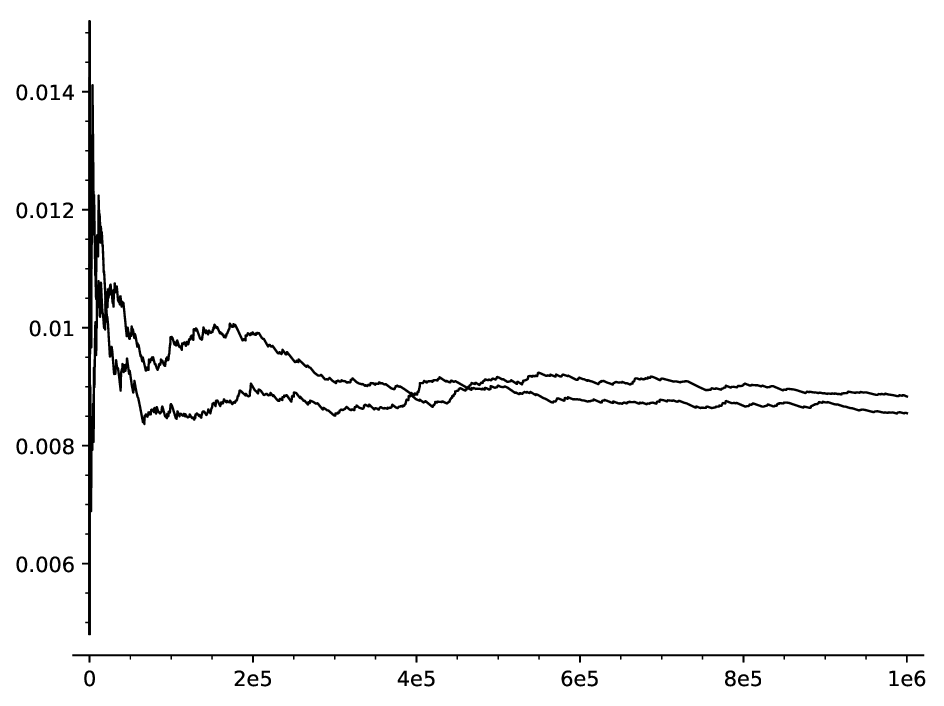}
\caption{$|l| = 2$: Top 2 bottom -2} \label{fig:14_6_1_3_A_2}
\end{subfigure}\hspace*{\fill}
\begin{subfigure}[b]{0.4\linewidth}
\includegraphics[width=\linewidth]{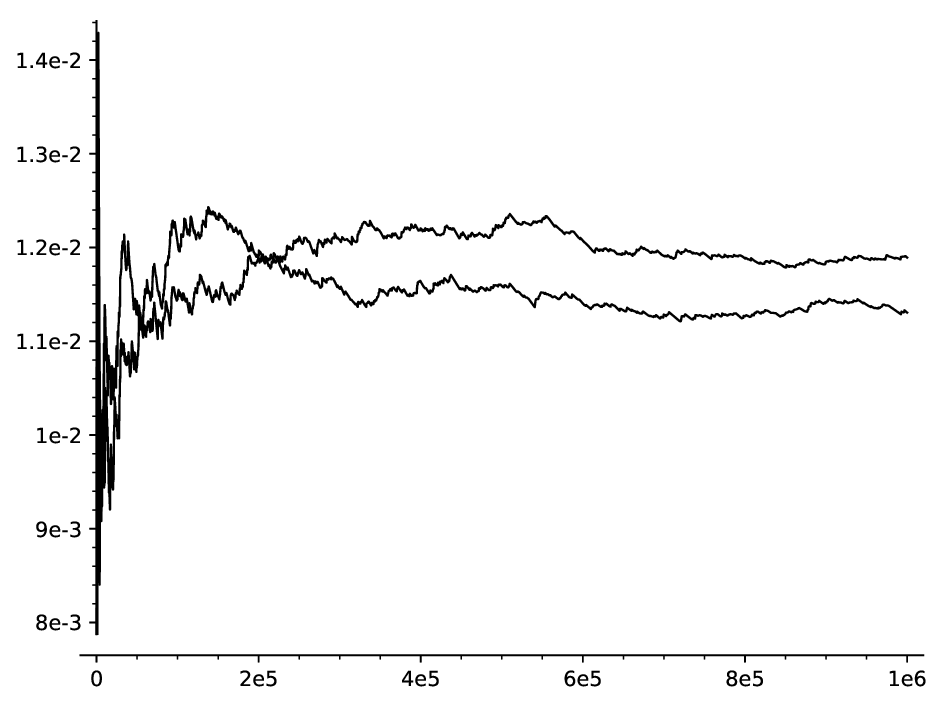}
\caption{$|l| = 3$: Top -3 bottom 3} \label{fig:14_6_1_3_A_3}
\end{subfigure}
\hspace*{-2.3cm}
\begin{subfigure}[b]{0.4\linewidth}
\includegraphics[width=\linewidth]{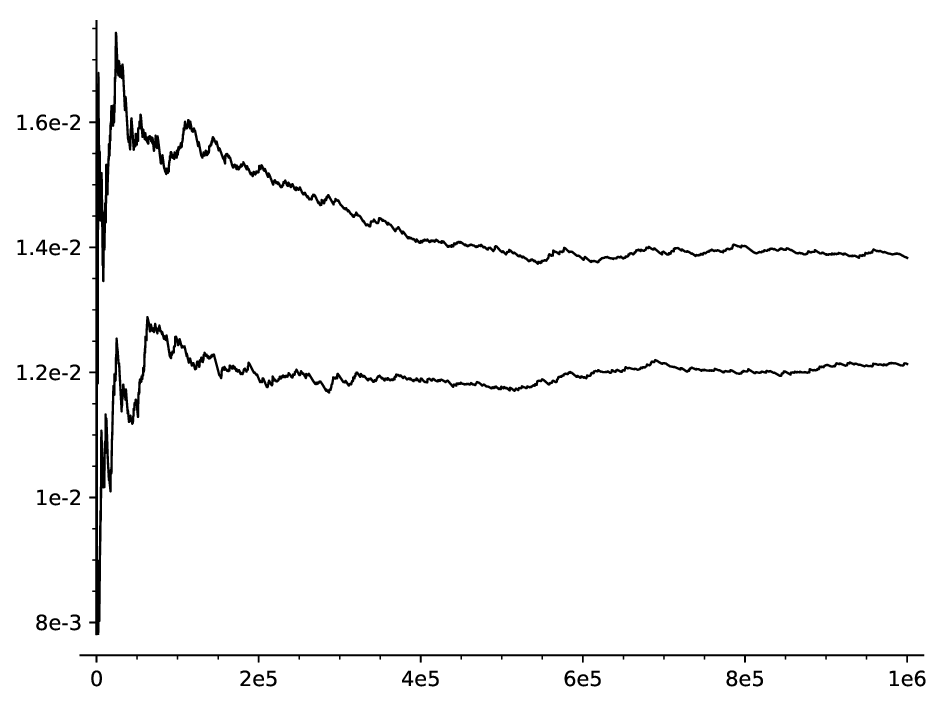}
\caption{$|l| = 4$: Top -4 bottom 4} \label{fig:14_6_1_3_A_4}
\end{subfigure}\hspace*{\fill}
\begin{subfigure}[b]{0.4\linewidth}
\includegraphics[width=\linewidth]{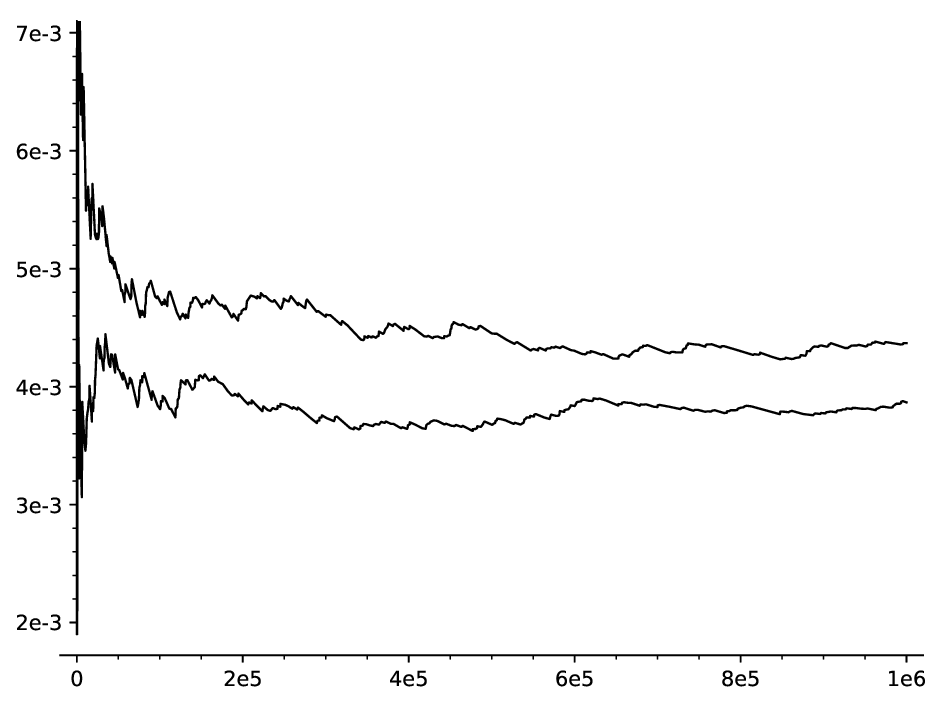}
\caption{$|l| = 5$: Top -5 bottom 5} \label{fig:14_6_1_3_A_5}
\end{subfigure}\hspace*{\fill}
\begin{subfigure}[b]{0.4\linewidth}
\includegraphics[width=\linewidth]{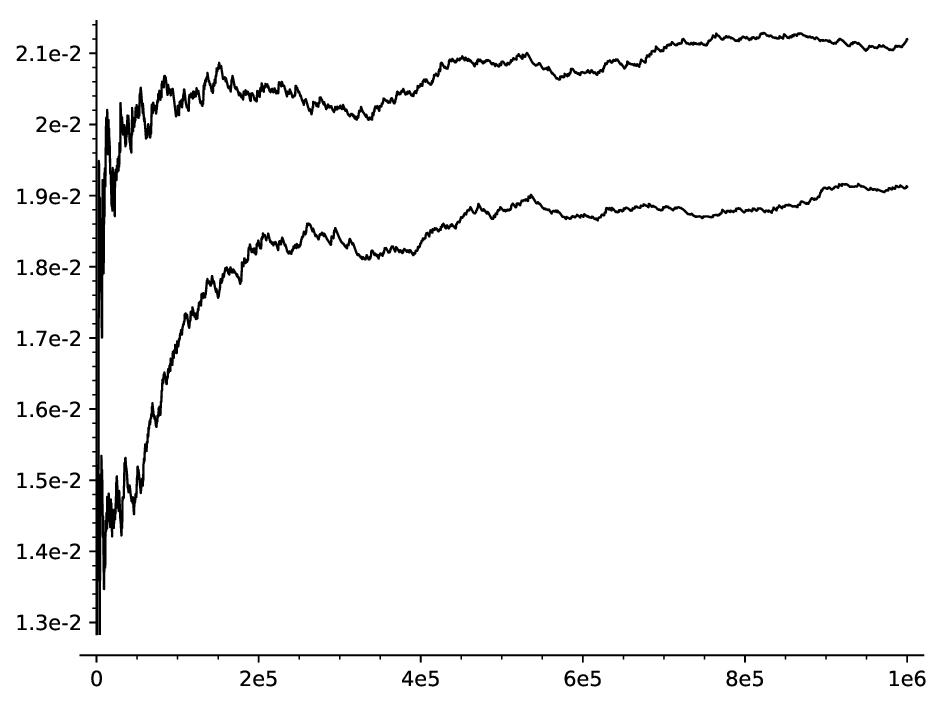}
\caption{$|l| = 6$: Top 6 bottom -6} \label{fig:14_6_1_3_A_6}
\end{subfigure}
\hspace*{-2.3cm}
\begin{subfigure}[b]{0.4\linewidth}
\includegraphics[width=\linewidth]{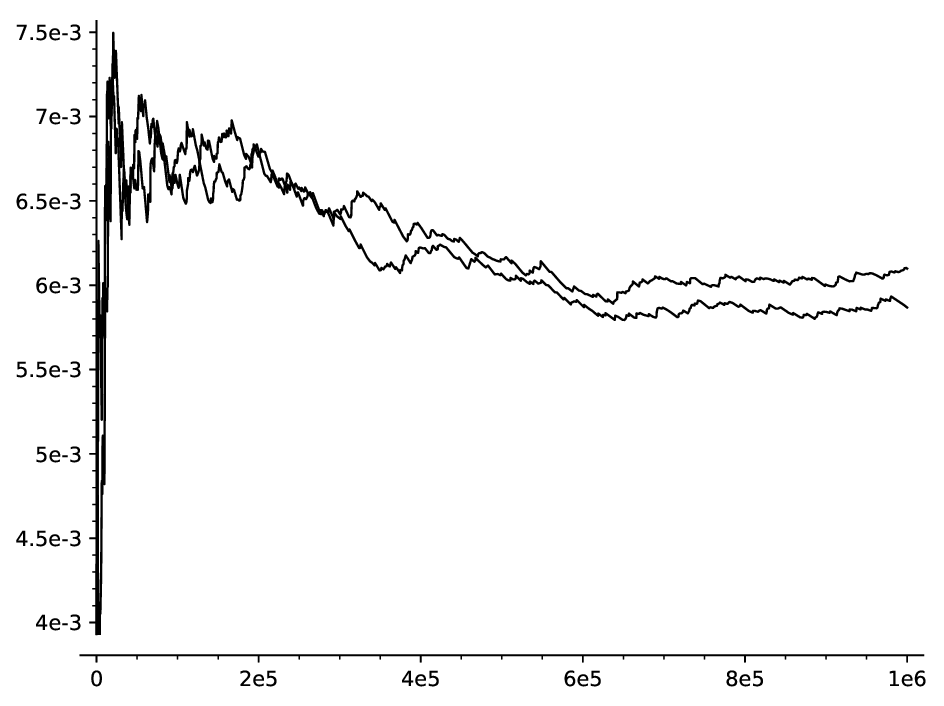}
\caption{$|l| = 7$: Top -7 bottom 7} \label{fig:14_6_1_3_A_7}
\end{subfigure}\hspace*{\fill}
\begin{subfigure}[b]{0.4\linewidth}
\includegraphics[width=\linewidth]{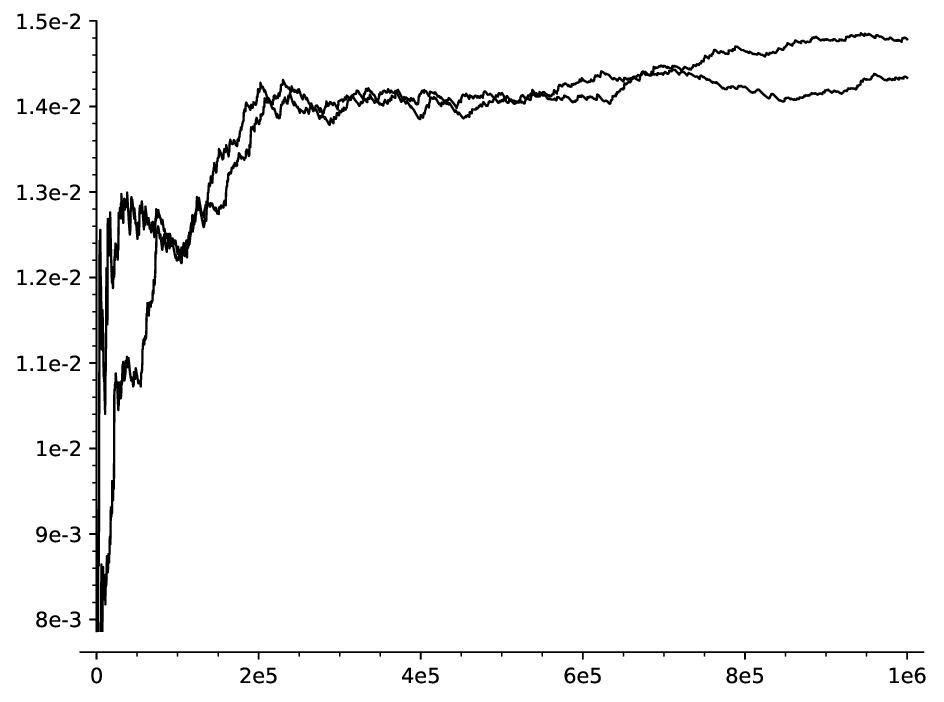}
\caption{$|l| = 8$: Top -8 bottom 8} \label{fig:14_6_1_3_A_8}
\end{subfigure}\hspace*{\fill}
\begin{subfigure}[b]{0.4\linewidth}
\includegraphics[width=\linewidth]{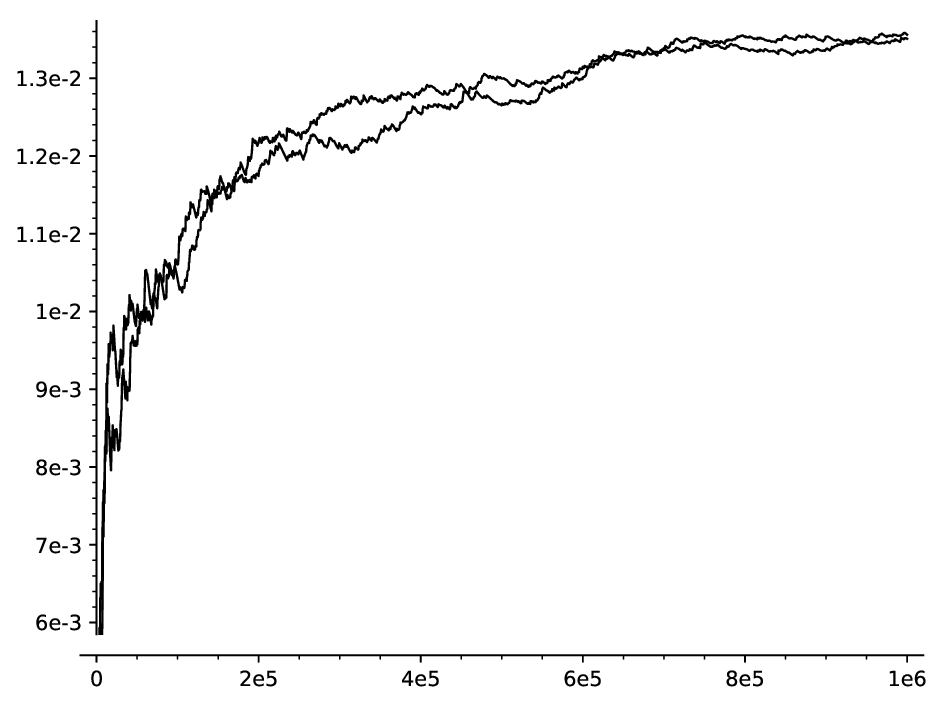}
\caption{$|l| = 9$: Top 9 bottom -9} \label{fig:14_6_1_3_A_9}
\end{subfigure}
\caption{14a1: $(\alpha, \beta) = (1,3)$ Ratio~\eqref{ratio_n_orders} $x_{6,E}^{(\alpha, \beta)}(X;l)/X^{1/2}\log^2(X)$} \label{fig:14a1_6_1_3_A_exact}
\end{figure}

\clearpage

\begin{figure}[t] 
\hspace*{-2.3cm}
\begin{subfigure}[b]{0.4\linewidth}
\includegraphics[width=\linewidth]{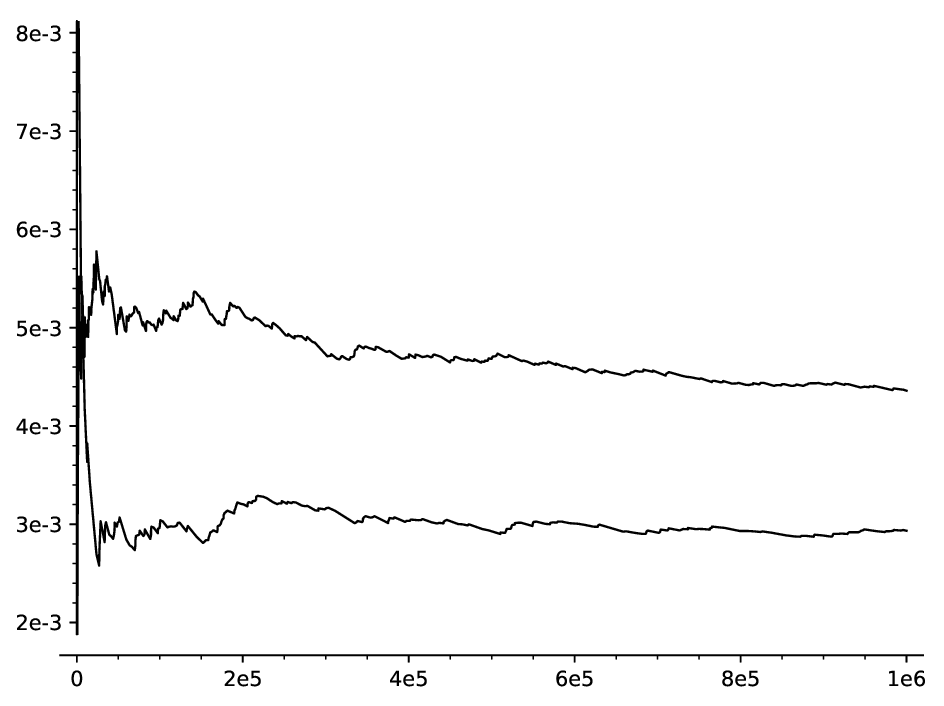}
\caption{$|l| = 1$: Top -1 bottom 1} \label{fig:14_6_2_3_A_1}
\end{subfigure}\hspace*{\fill}
\begin{subfigure}[b]{0.4\linewidth}
\includegraphics[width=\linewidth]{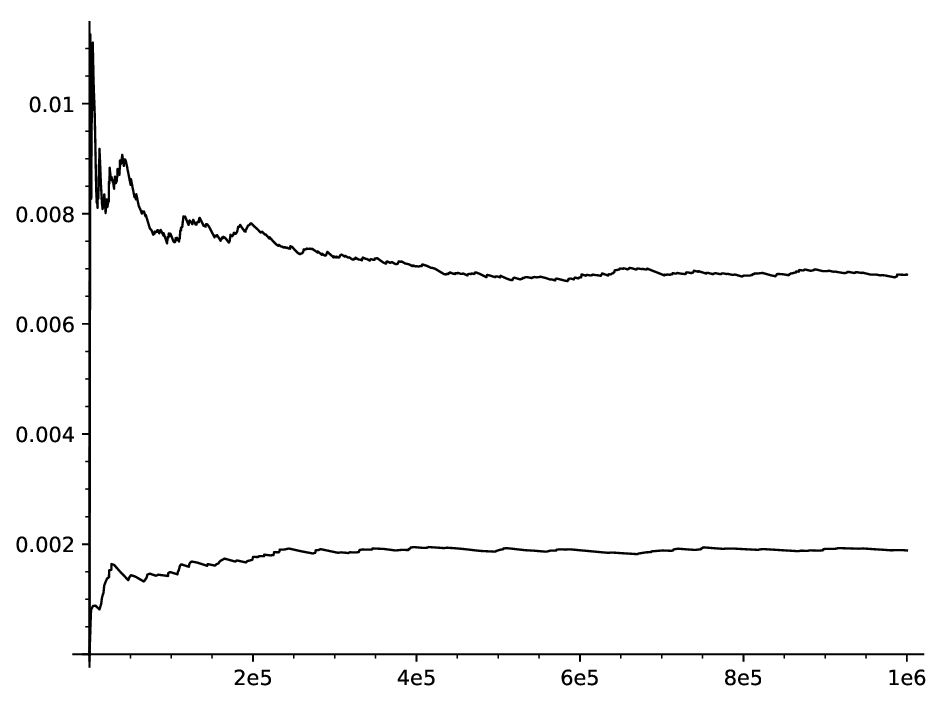}
\caption{$|l| = 2$: Top -2 bottom 2} \label{fig:14_6_2_3_A_2}
\end{subfigure}\hspace*{\fill}
\begin{subfigure}[b]{0.4\linewidth}
\includegraphics[width=\linewidth]{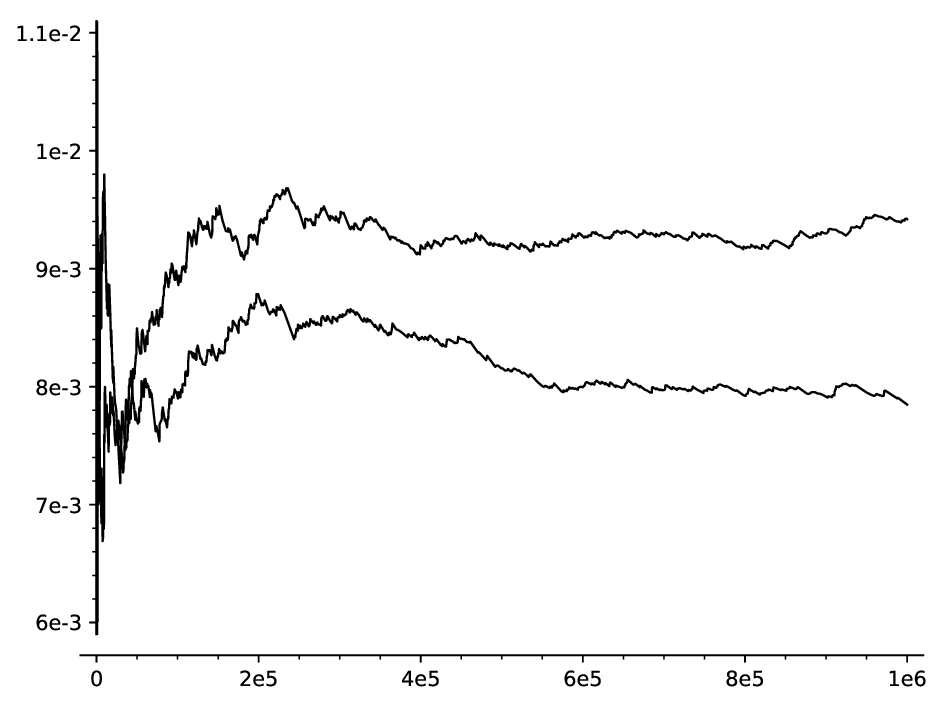}
\caption{$|l| = 3$: Top -3 bottom 3} \label{fig:14_6_2_3_A_3}
\end{subfigure}
\hspace*{-2.3cm}
\begin{subfigure}[b]{0.4\linewidth}
\includegraphics[width=\linewidth]{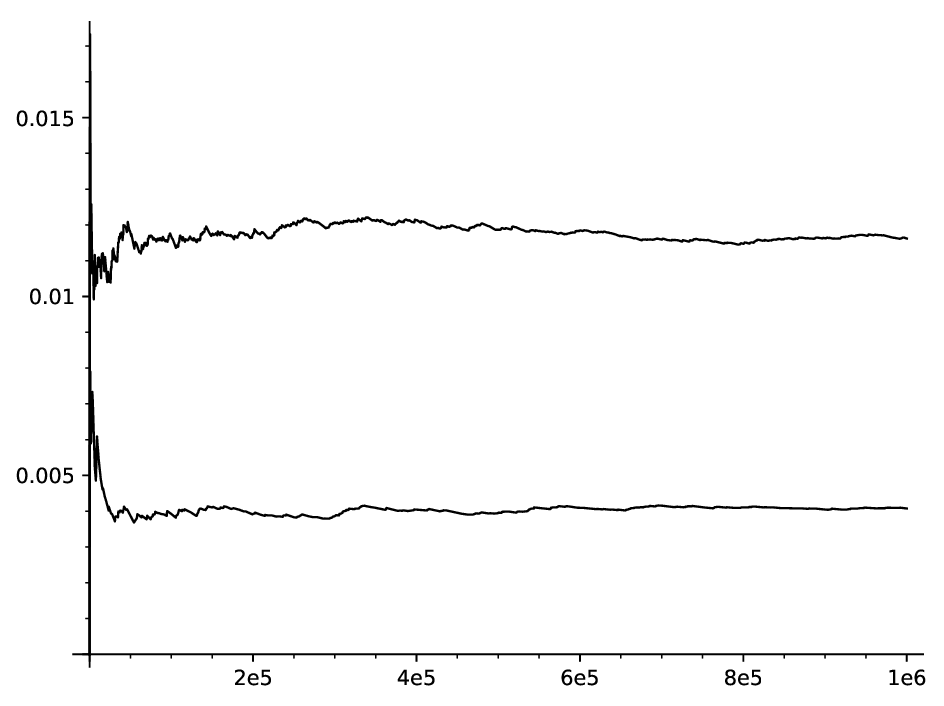}
\caption{$|l| = 4$: Top 4 bottom -4} \label{fig:14_6_2_3_A_4}
\end{subfigure}\hspace*{\fill}
\begin{subfigure}[b]{0.4\linewidth}
\includegraphics[width=\linewidth]{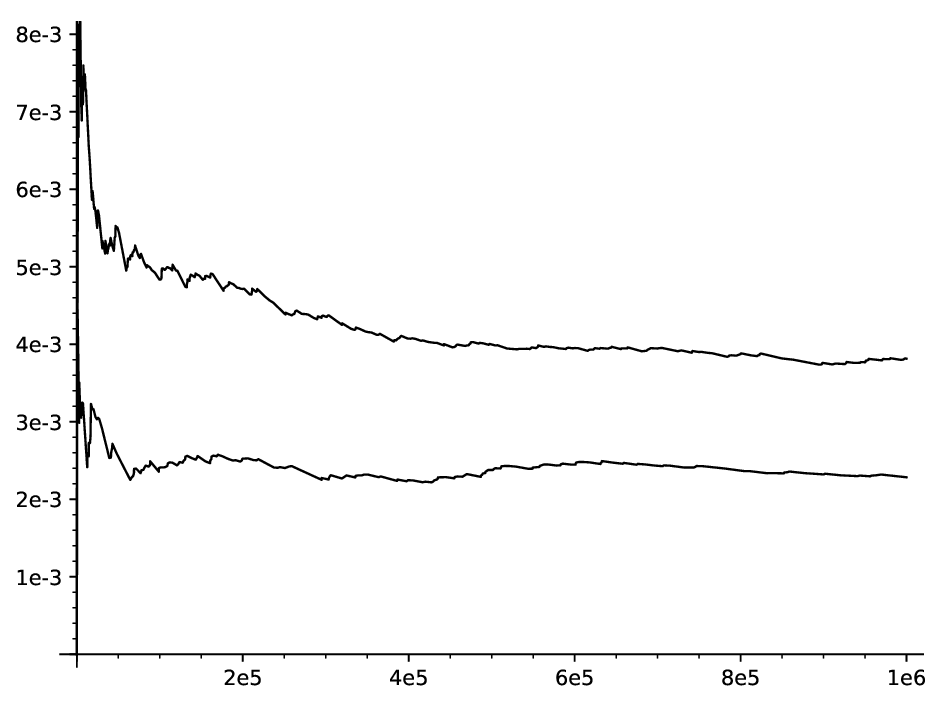}
\caption{$|l| = 5$: Top 5 bottom -5} \label{fig:14_6_2_3_A_5}
\end{subfigure}\hspace*{\fill}
\begin{subfigure}[b]{0.4\linewidth}
\includegraphics[width=\linewidth]{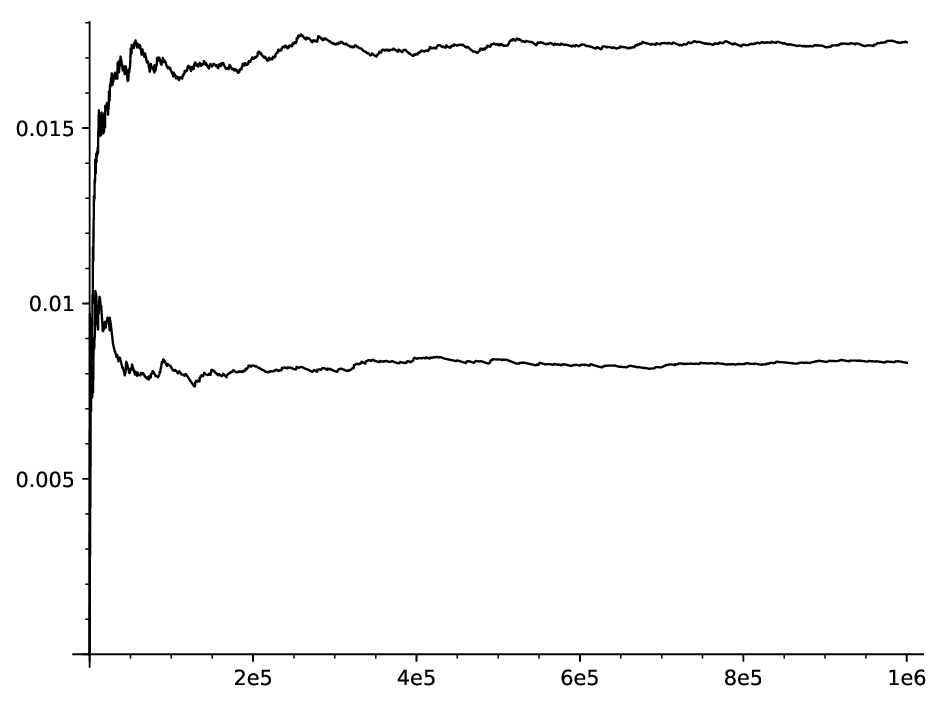}
\caption{$|l| = 6$: Top -6 bottom 6} \label{fig:14_6_2_3_A_6}
\end{subfigure}
\hspace*{-2.3cm}
\begin{subfigure}[b]{0.4\linewidth}
\includegraphics[width=\linewidth]{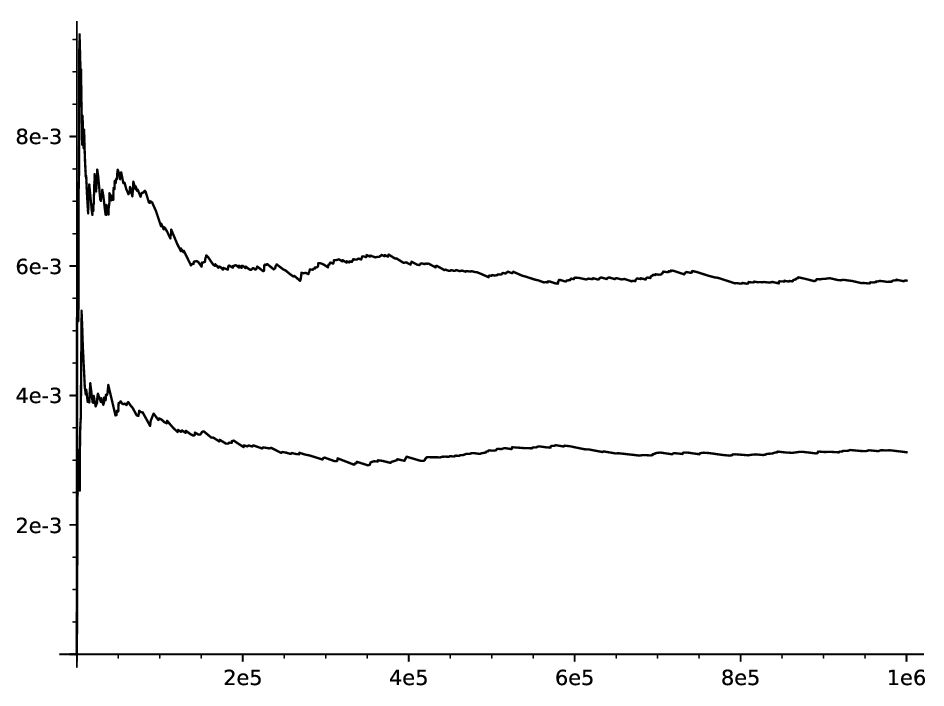}
\caption{$|l| = 7$: Top -7 bottom 7} \label{fig:14_6_2_3_A_7}
\end{subfigure}\hspace*{\fill}
\begin{subfigure}[b]{0.4\linewidth}
\includegraphics[width=\linewidth]{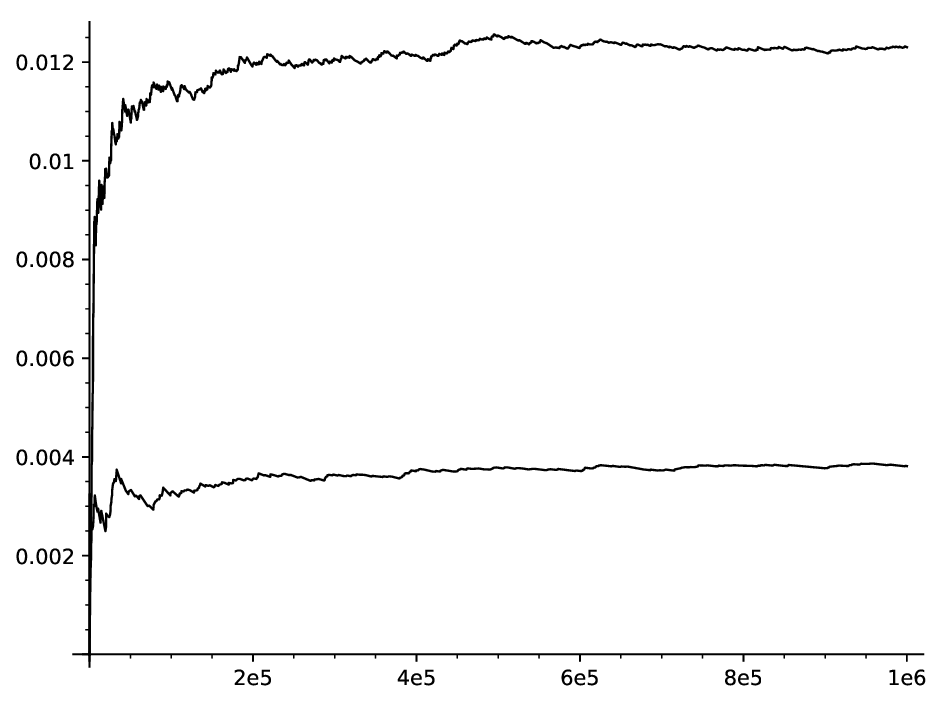}
\caption{$|l| = 8$: Top -8 bottom 8} \label{fig:14_6_2_3_A_8}
\end{subfigure}\hspace*{\fill}
\begin{subfigure}[b]{0.4\linewidth}
\includegraphics[width=\linewidth]{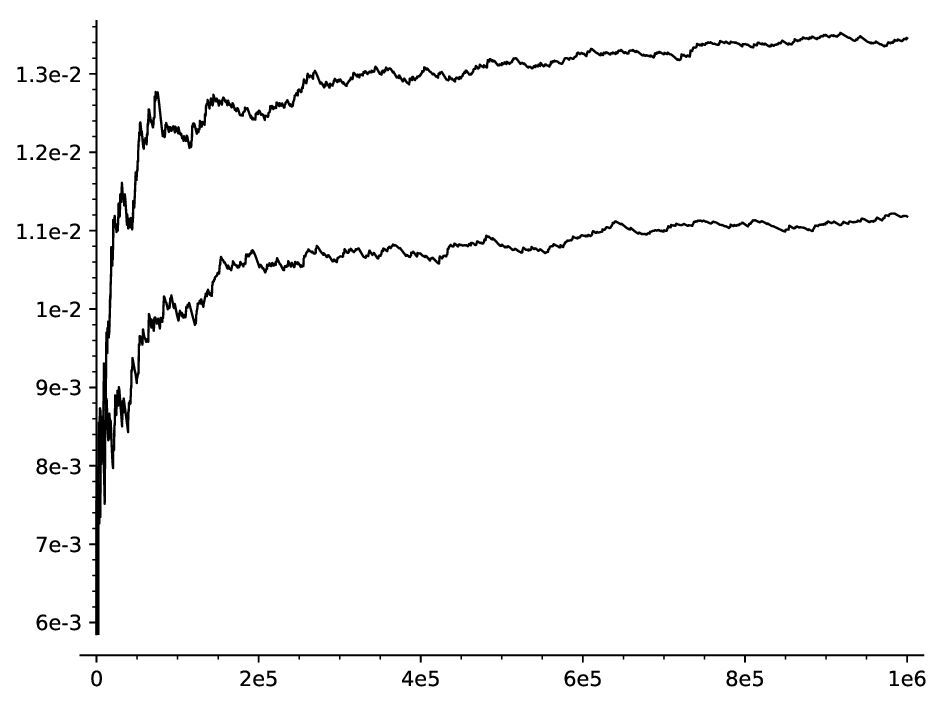}
\caption{$|l| = 9$: Top -9 bottom 9} \label{fig:14_6_2_3_A_9}
\end{subfigure}
\caption{14a1: $(\alpha, \beta) = (2,3)$ Ratio~\eqref{ratio_n_orders} $x_{6,E}^{(\alpha, \beta)}(X;l)/X^{1/2}\log^2(X)$} \label{fig:14a1_6_2_3_A_exact}
\end{figure}

\clearpage

\begin{figure}[t] 
\hspace*{-2.3cm}
\begin{subfigure}[b]{0.4\linewidth}
\includegraphics[width=\linewidth]{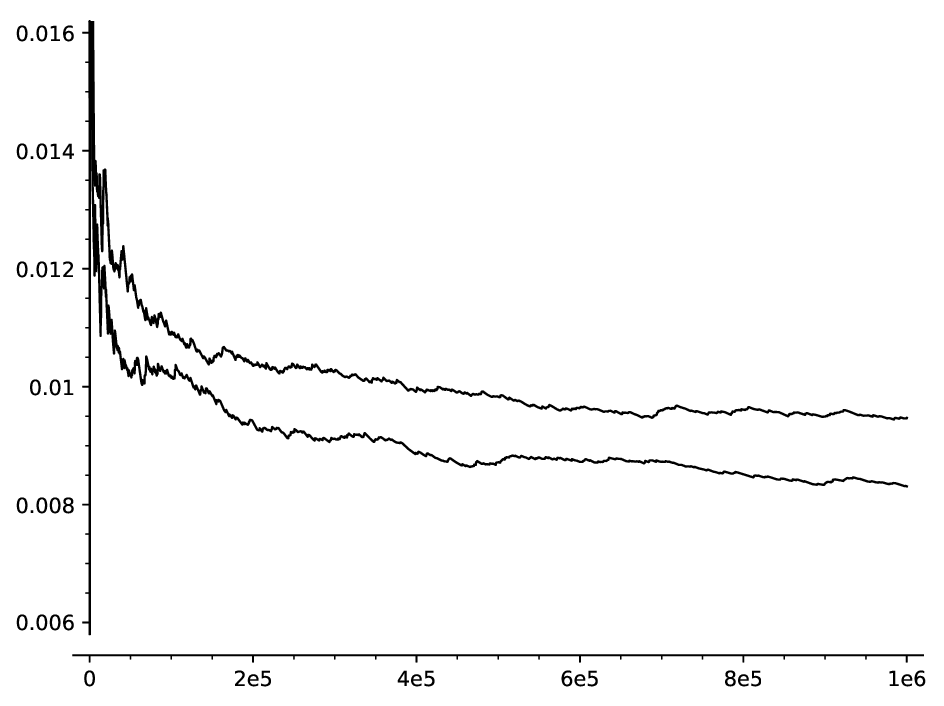}
\caption{$|l| = 1$: Top -1 bottom 1} \label{fig:14_6_1_6_A_1}
\end{subfigure}\hspace*{\fill}
\begin{subfigure}[b]{0.4\linewidth}
\includegraphics[width=\linewidth]{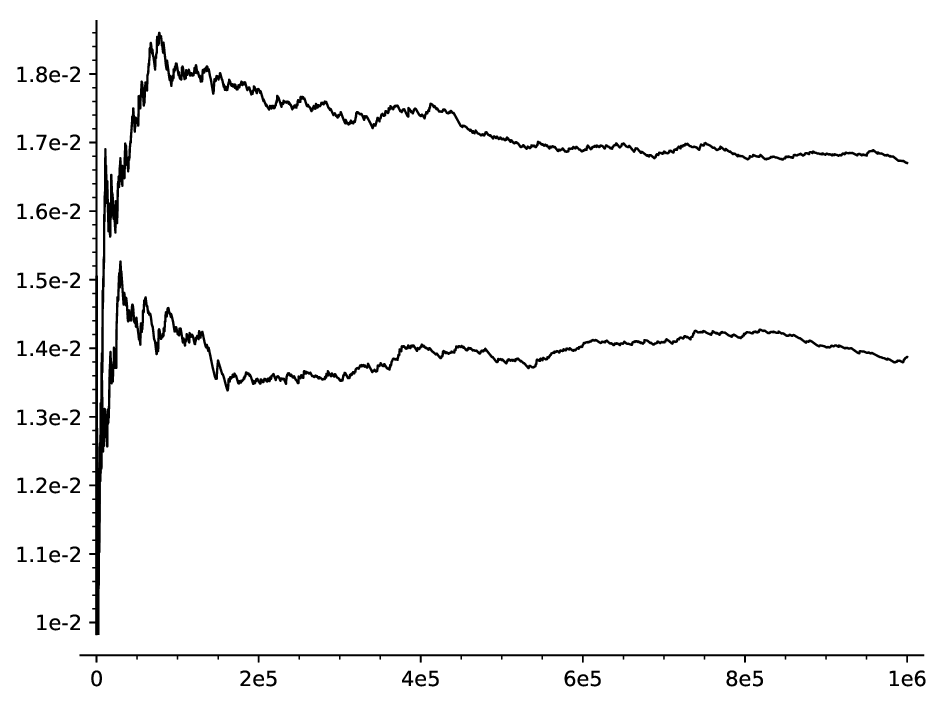}
\caption{$|l| = 2$: Top -2 bottom 2} \label{fig:14_6_1_6_A_2}
\end{subfigure}\hspace*{\fill}
\begin{subfigure}[b]{0.4\linewidth}
\includegraphics[width=\linewidth]{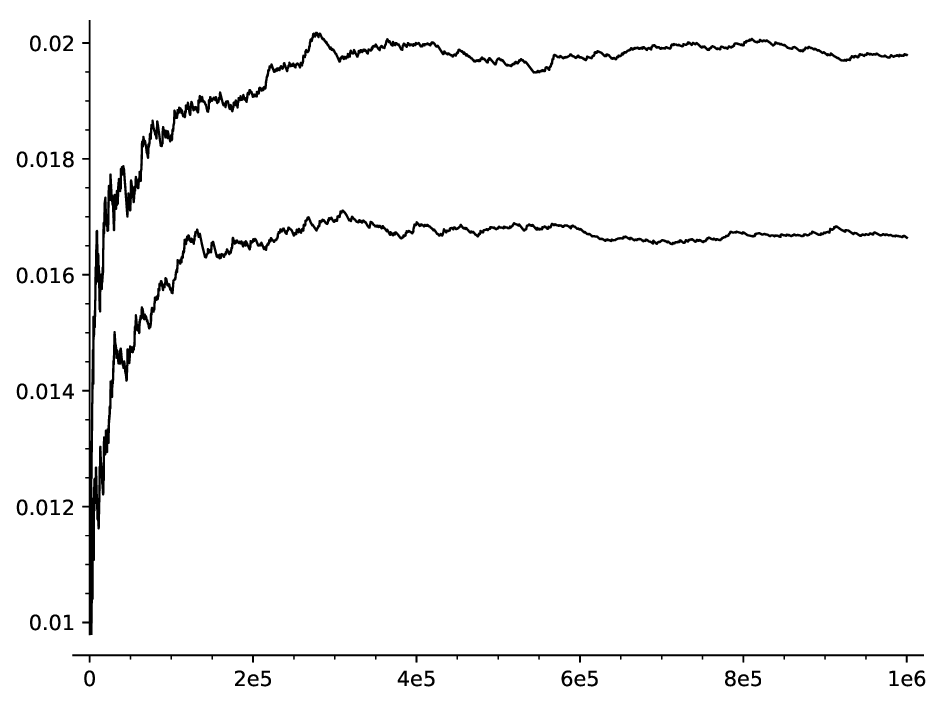}
\caption{$|l| = 3$: Top -3 bottom 3} \label{fig:14_6_1_6_A_3}
\end{subfigure}
\hspace*{-2.3cm}
\begin{subfigure}[b]{0.4\linewidth}
\includegraphics[width=\linewidth]{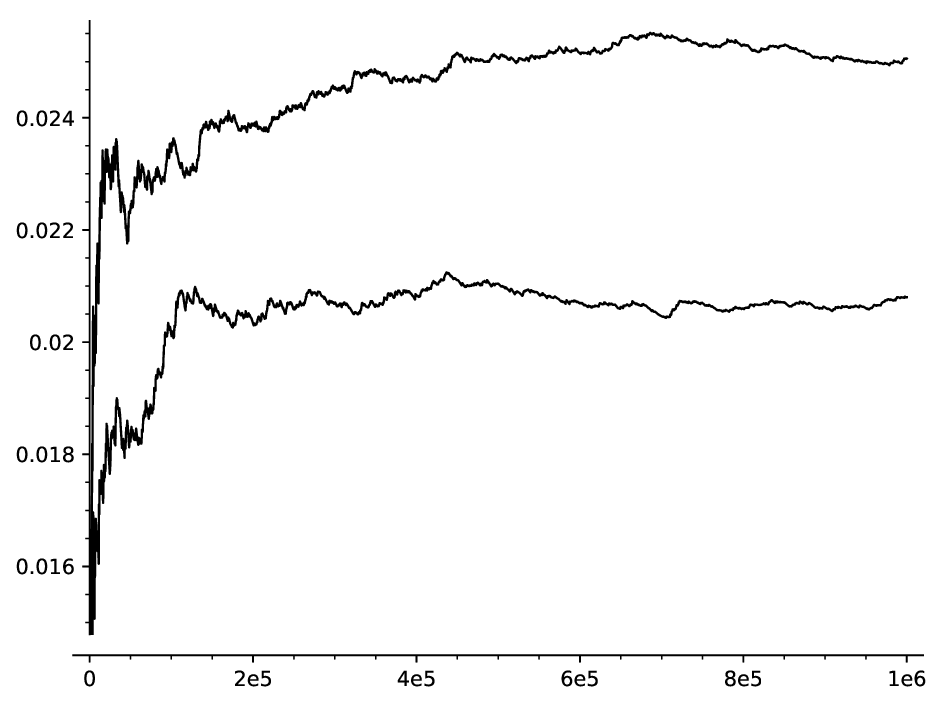}
\caption{$|l| = 4$: Top 4 bottom -4} \label{fig:14_6_1_6_A_4}
\end{subfigure}\hspace*{\fill}
\begin{subfigure}[b]{0.4\linewidth}
\includegraphics[width=\linewidth]{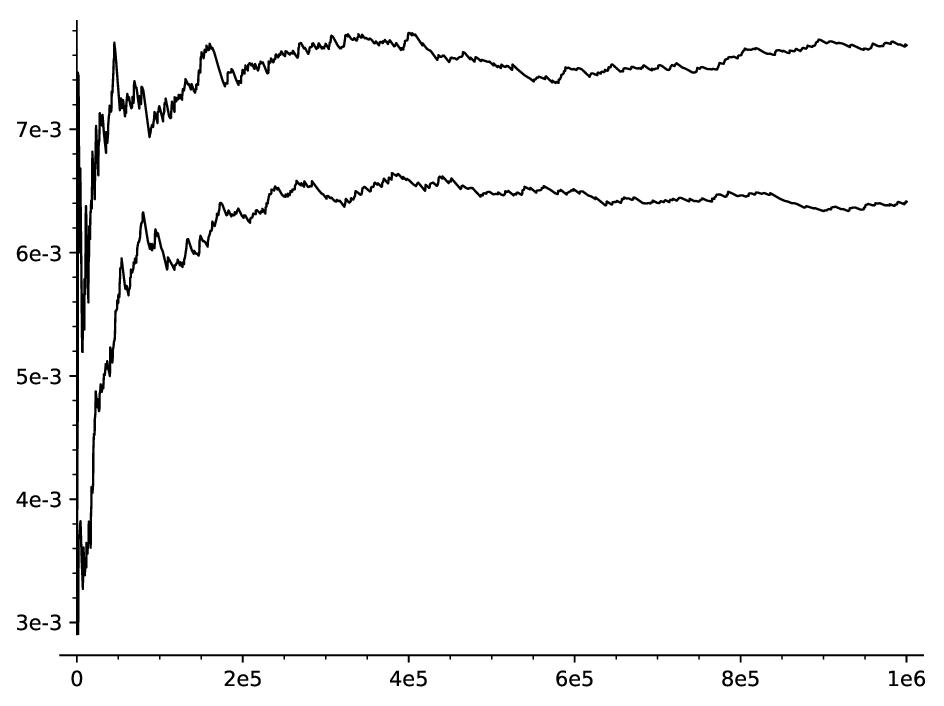}
\caption{$|l| = 5$: Top 5 bottom -5} \label{fig:14_6_1_6_A_5}
\end{subfigure}\hspace*{\fill}
\begin{subfigure}[b]{0.4\linewidth}
\includegraphics[width=\linewidth]{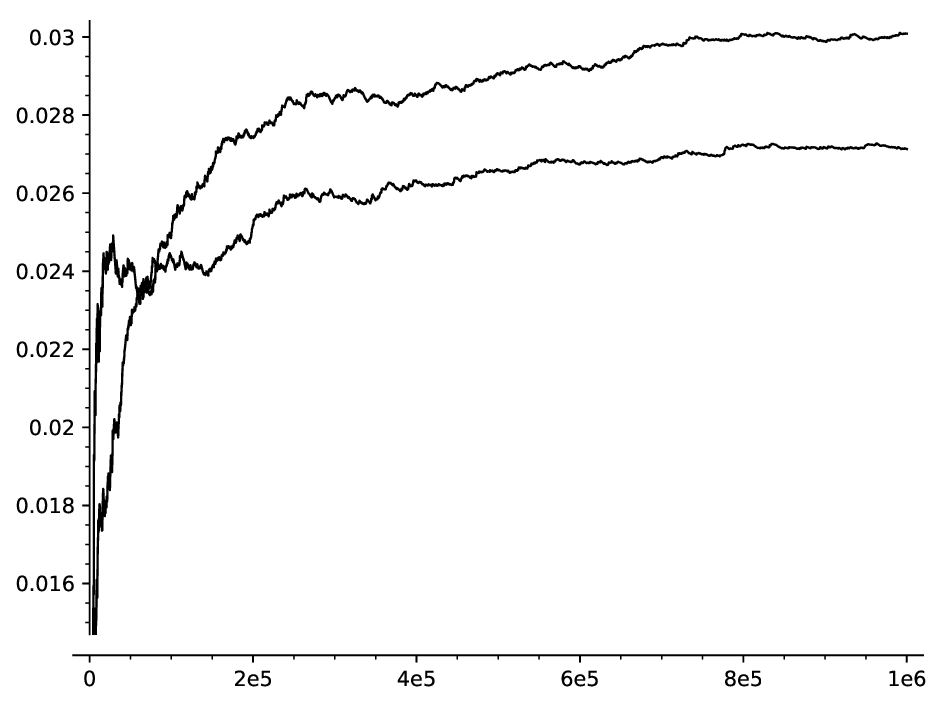}
\caption{$|l| = 6$: Top -6 bottom 6} \label{fig:14_6_1_6_A_6}
\end{subfigure}
\hspace*{-2.3cm}
\begin{subfigure}[b]{0.4\linewidth}
\includegraphics[width=\linewidth]{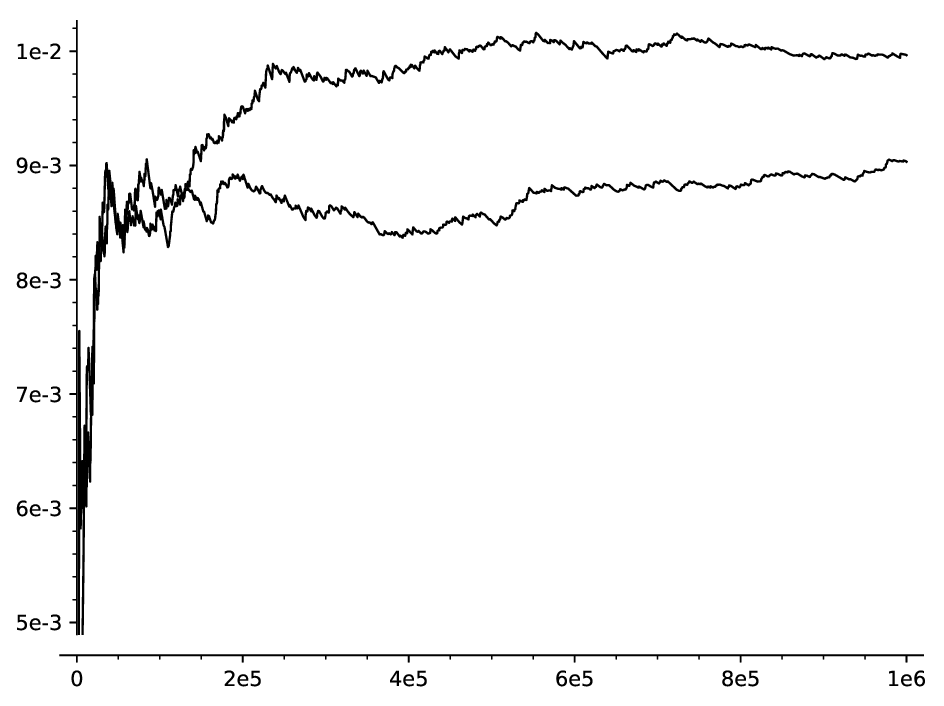}
\caption{$|l| = 7$: Top -7 bottom 7} \label{fig:14_6_1_6_A_7}
\end{subfigure}\hspace*{\fill}
\begin{subfigure}[b]{0.4\linewidth}
\includegraphics[width=\linewidth]{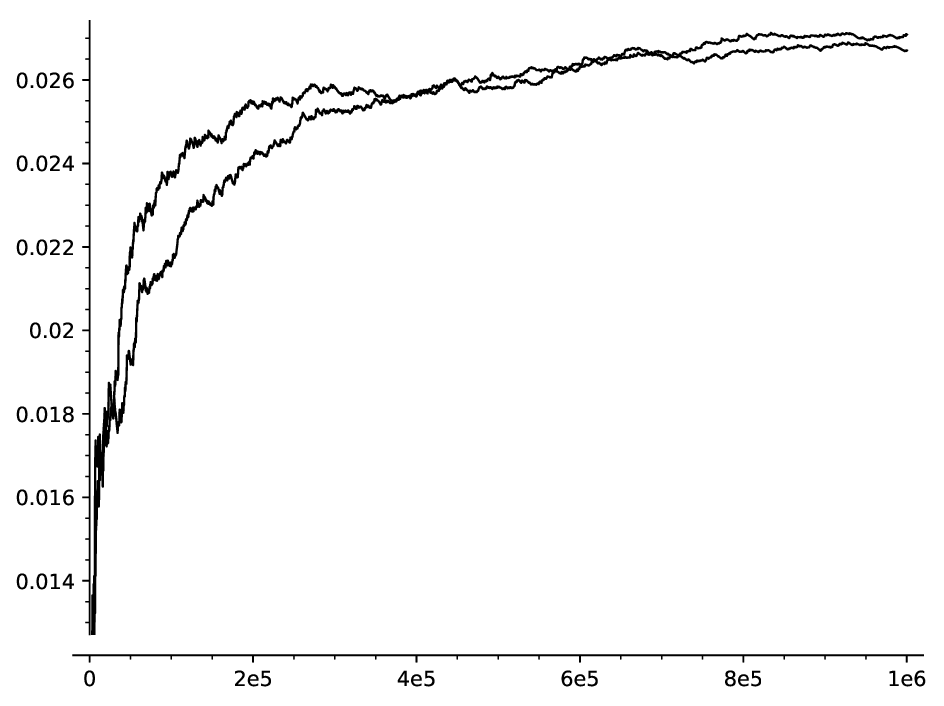}
\caption{$|l| = 8$: Top -8 bottom 8} \label{fig:14_6_1_6_A_8}
\end{subfigure}\hspace*{\fill}
\begin{subfigure}[b]{0.4\linewidth}
\includegraphics[width=\linewidth]{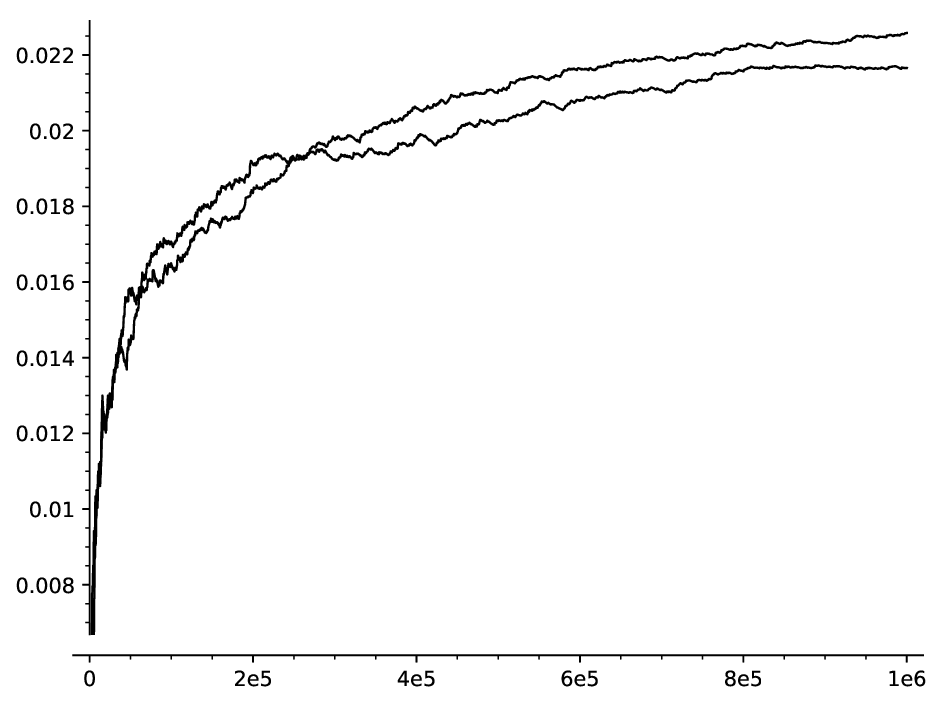}
\caption{$|l| = 9$: Top -9 bottom 9} \label{fig:14_6_1_6_A_9}
\end{subfigure}
\caption{14a1: $(\alpha, \beta) = (1,6)$ Ratio~\eqref{ratio_n_orders} $x_{6,E}^{(\alpha, \beta)}(X;l)/X^{1/2}\log^2(X)$} \label{fig:14a1_6_1_6_A_exact}
\end{figure}

\clearpage

\begin{figure}[t] 
\hspace*{-2.3cm}
\begin{subfigure}[b]{0.4\linewidth}
\includegraphics[width=\linewidth]{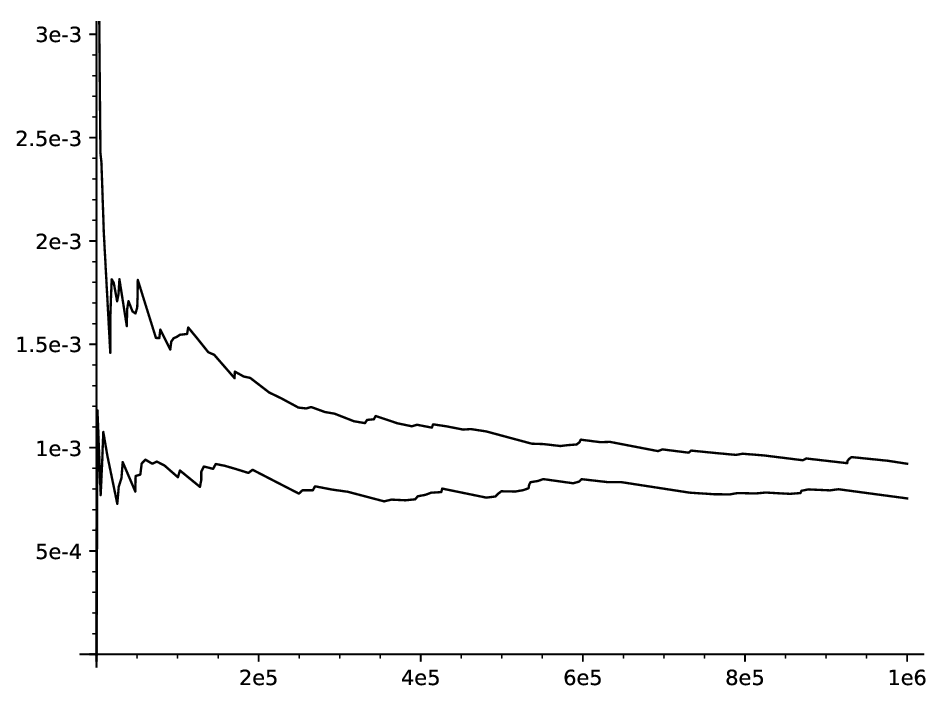}
\caption{$|l| = 1$: Top -1 bottom 1} \label{fig:14_6_2_6_A_1}
\end{subfigure}\hspace*{\fill}
\begin{subfigure}[b]{0.4\linewidth}
\includegraphics[width=\linewidth]{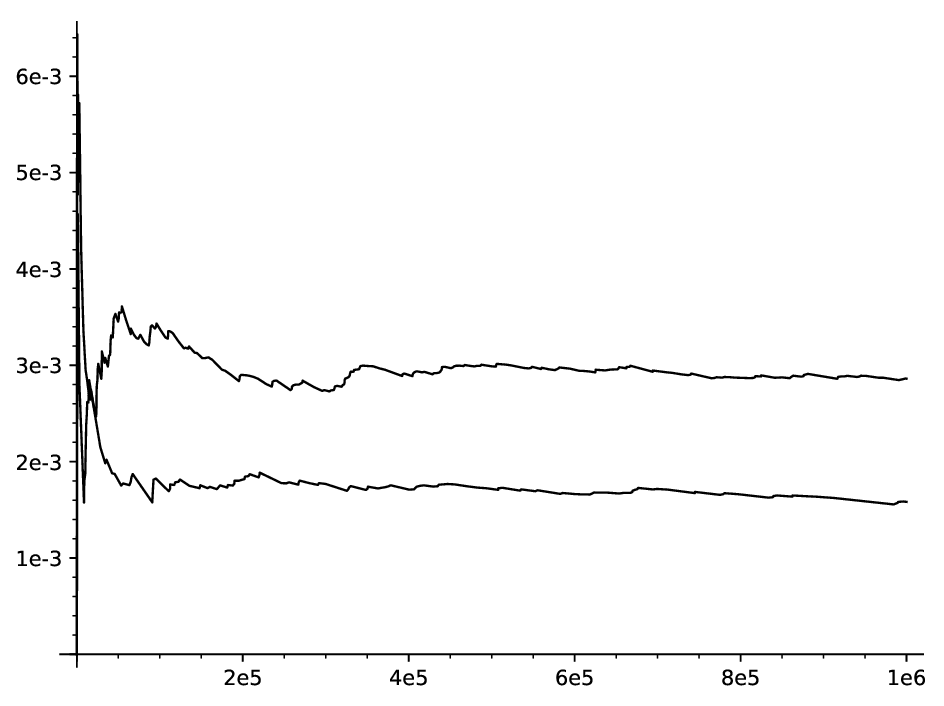}
\caption{$|l| = 2$: Top -2 bottom 2} \label{fig:14_6_2_6_A_2}
\end{subfigure}\hspace*{\fill}
\begin{subfigure}[b]{0.4\linewidth}
\includegraphics[width=\linewidth]{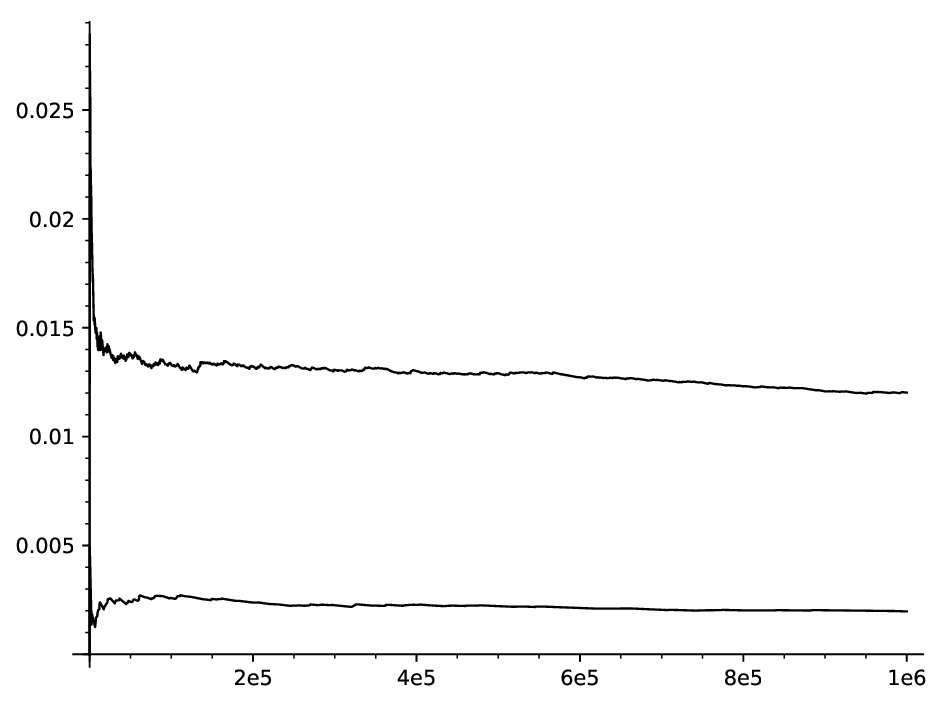}
\caption{$|l| = 3$: Top -3 bottom 3} \label{fig:14_6_2_6_A_3}
\end{subfigure}
\hspace*{-2.3cm}
\begin{subfigure}[b]{0.4\linewidth}
\includegraphics[width=\linewidth]{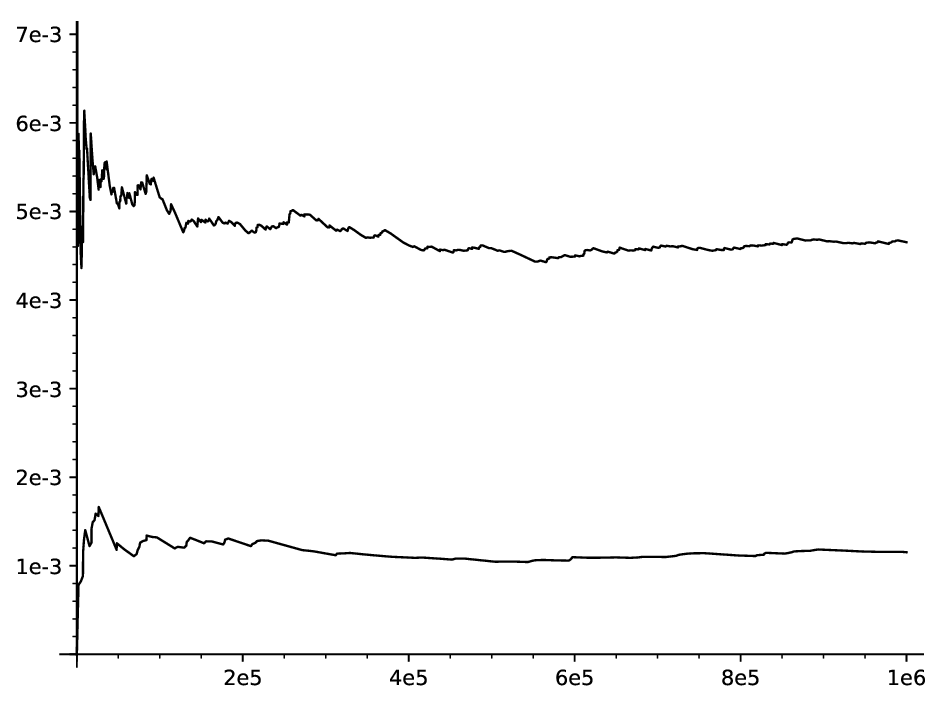}
\caption{$|l| = 4$: Top 4 bottom -4} \label{fig:14_6_2_6_A_4}
\end{subfigure}\hspace*{\fill}
\begin{subfigure}[b]{0.4\linewidth}
\includegraphics[width=\linewidth]{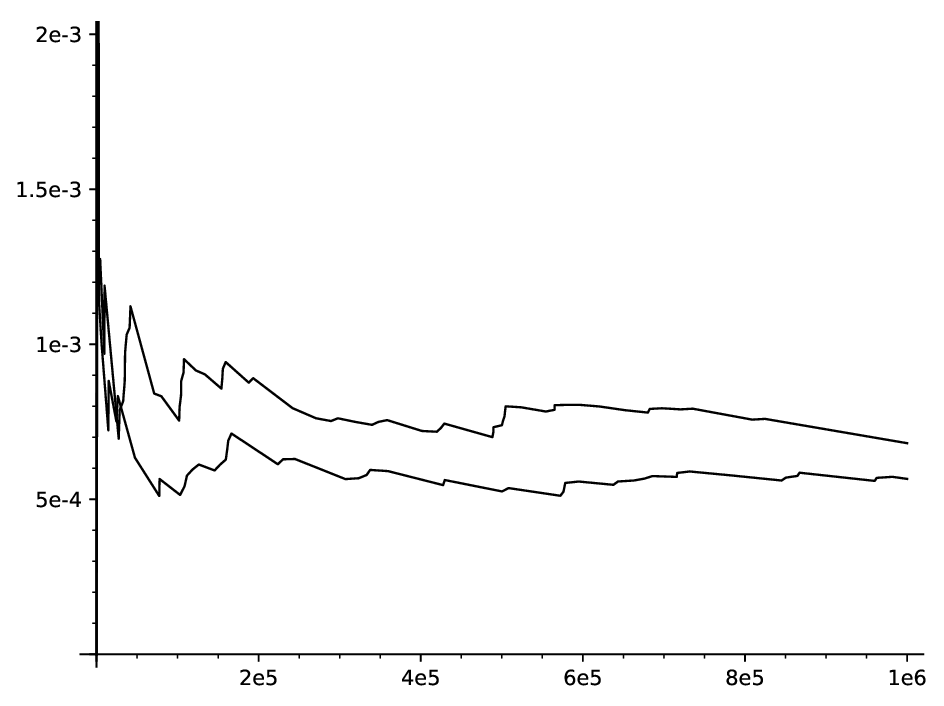}
\caption{$|l| = 5$: Top 5 bottom -5} \label{fig:14_6_2_6_A_5}
\end{subfigure}\hspace*{\fill}
\begin{subfigure}[b]{0.4\linewidth}
\includegraphics[width=\linewidth]{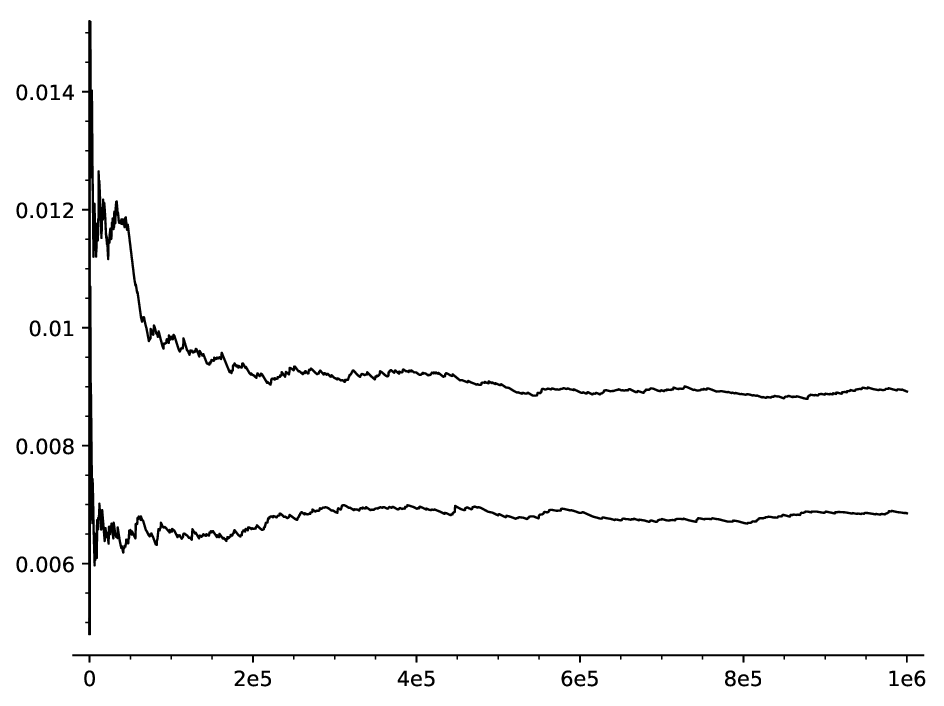}
\caption{$|l| = 6$: Top 6 bottom -6} \label{fig:14_6_2_6_A_6}
\end{subfigure}
\hspace*{-2.3cm}
\begin{subfigure}[b]{0.4\linewidth}
\includegraphics[width=\linewidth]{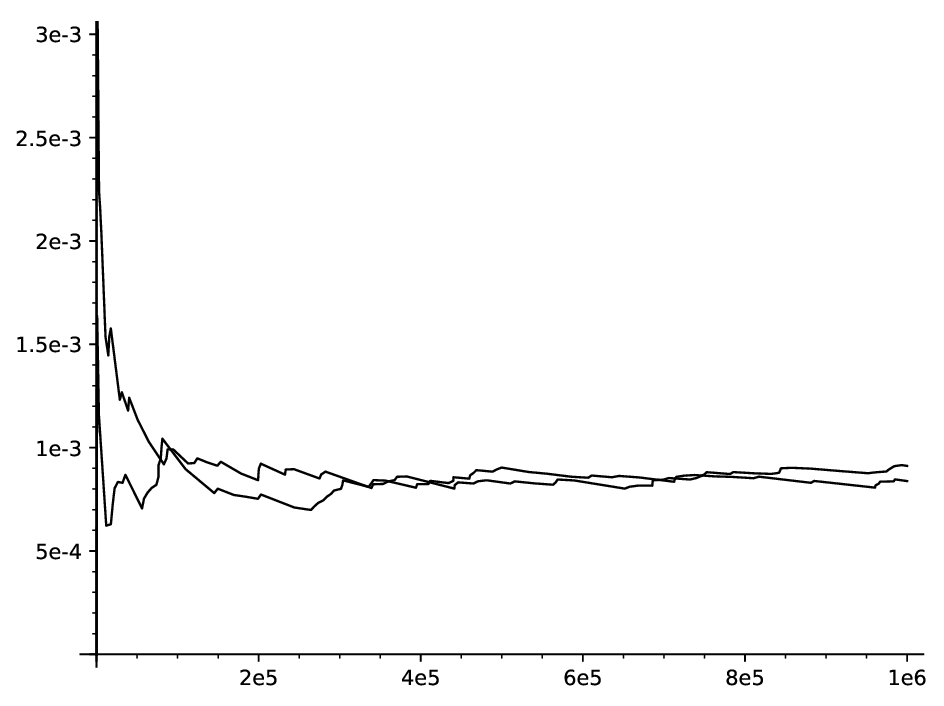}
\caption{$|l| = 7$: Top 7 bottom -7} \label{fig:14_6_2_6_A_7}
\end{subfigure}\hspace*{\fill}
\begin{subfigure}[b]{0.4\linewidth}
\includegraphics[width=\linewidth]{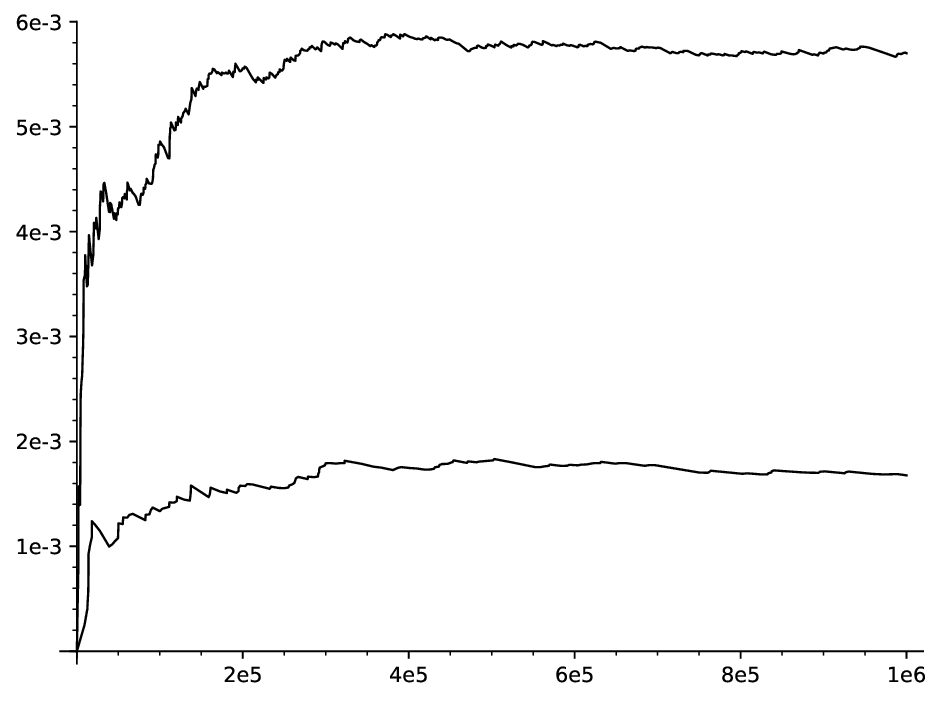}
\caption{$|l| = 8$: Top -8 bottom 8} \label{fig:14_6_2_6_A_8}
\end{subfigure}\hspace*{\fill}
\begin{subfigure}[b]{0.4\linewidth}
\includegraphics[width=\linewidth]{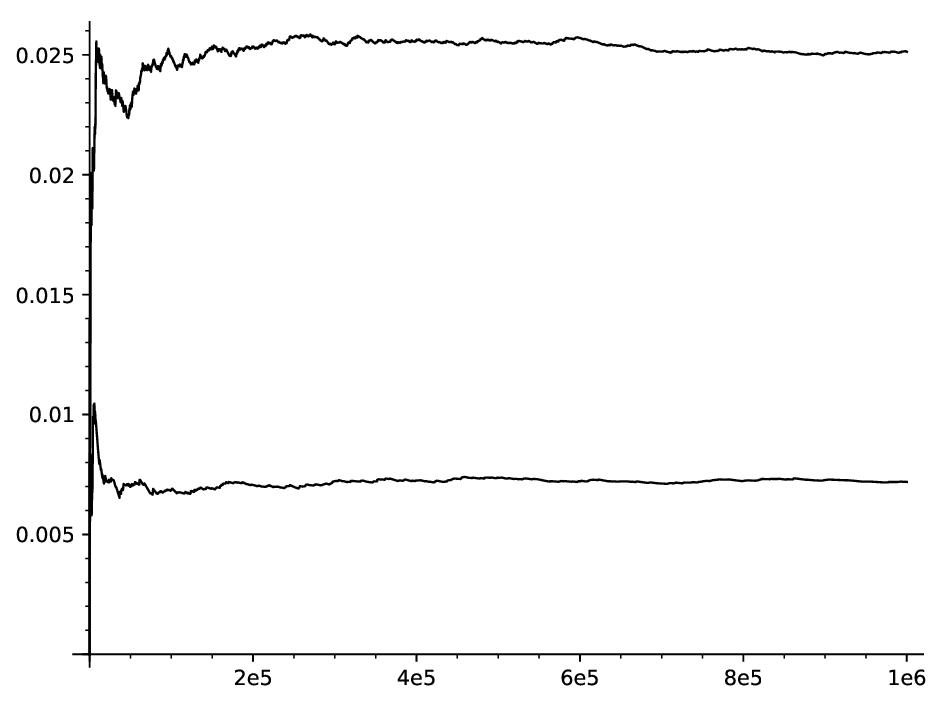}
\caption{$|l| = 9$: Top -9 bottom 9} \label{fig:14_6_2_6_A_9}
\end{subfigure}
\caption{14a1: $(\alpha, \beta) = (2,6)$ Ratio~\eqref{ratio_n_orders} $x_{6,E}^{(\alpha, \beta)}(X;l)/X^{1/2}\log^2(X)$} \label{fig:14a1_6_2_6_A_exact}
\end{figure}

\clearpage

\begin{figure}[t] 
\hspace*{-2.3cm}
\begin{subfigure}[b]{0.4\linewidth}
\includegraphics[width=\linewidth]{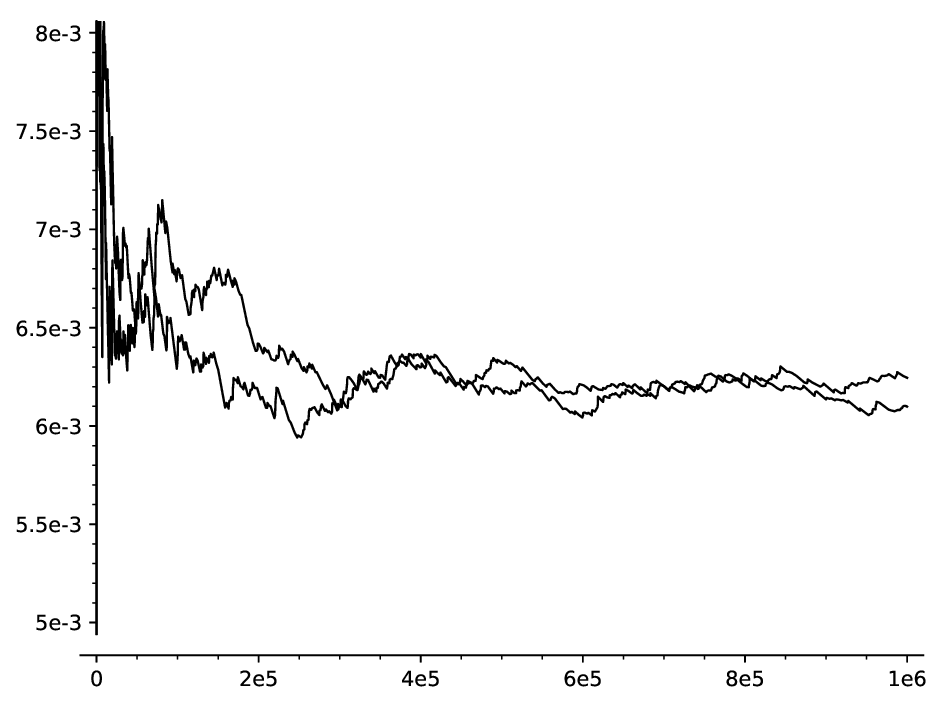}
\caption{$|l| = 1$: Top 1 bottom -1} \label{fig:15_6_1_3_A_1}
\end{subfigure}\hspace*{\fill}
\begin{subfigure}[b]{0.4\linewidth}
\includegraphics[width=\linewidth]{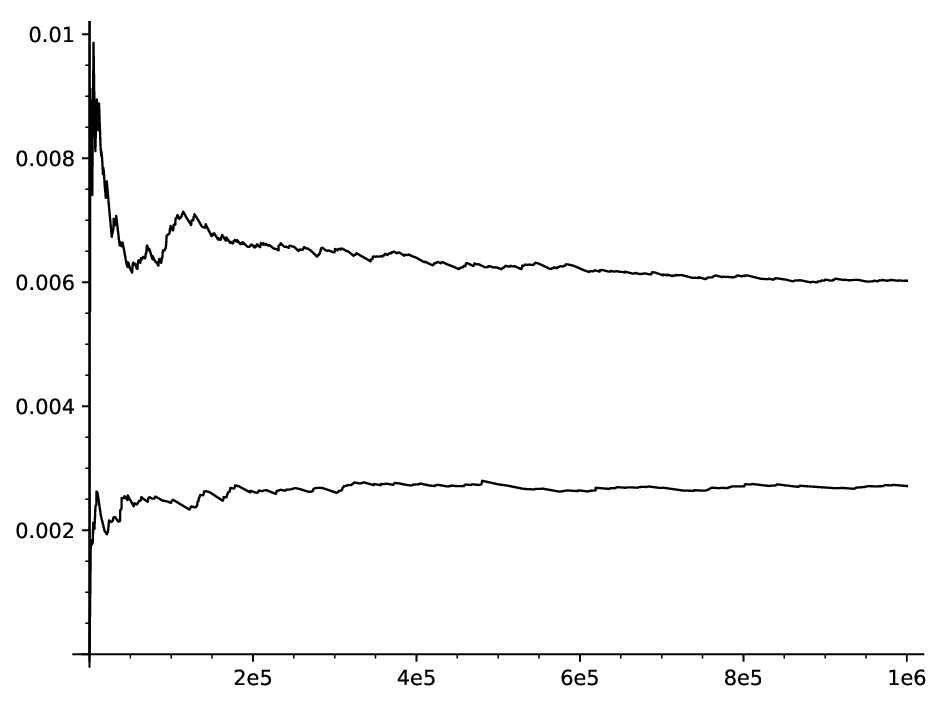}
\caption{$|l| = 2$: Top 2 bottom -2} \label{fig:15_6_1_3_A_2}
\end{subfigure}\hspace*{\fill}
\begin{subfigure}[b]{0.4\linewidth}
\includegraphics[width=\linewidth]{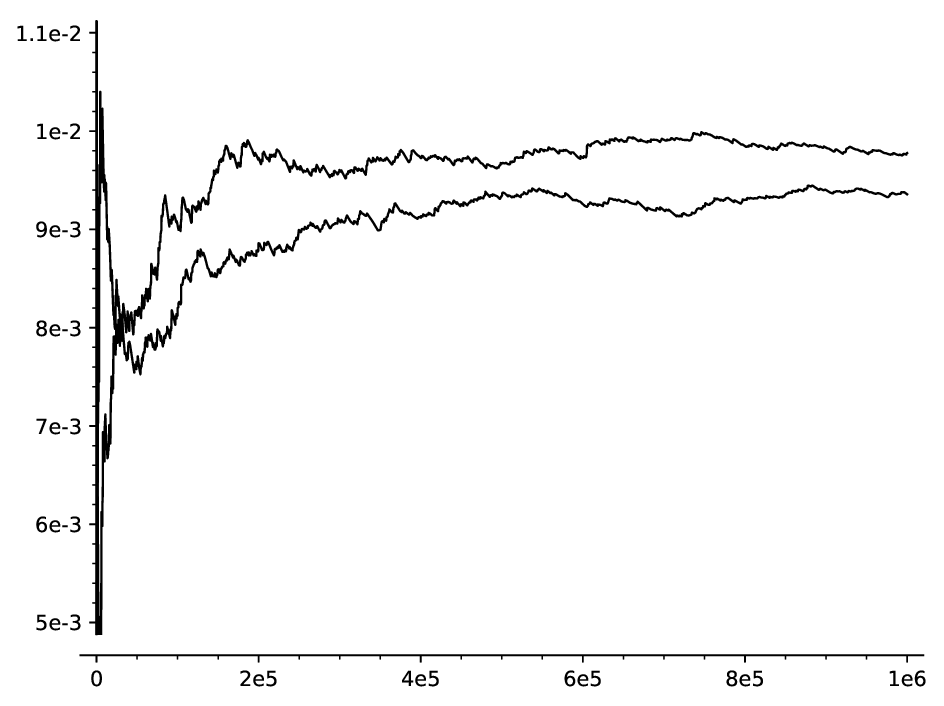}
\caption{$|l| = 3$: Top 3 bottom -3} \label{fig:15_6_1_3_A_3}
\end{subfigure}
\hspace*{-2.3cm}
\begin{subfigure}[b]{0.4\linewidth}
\includegraphics[width=\linewidth]{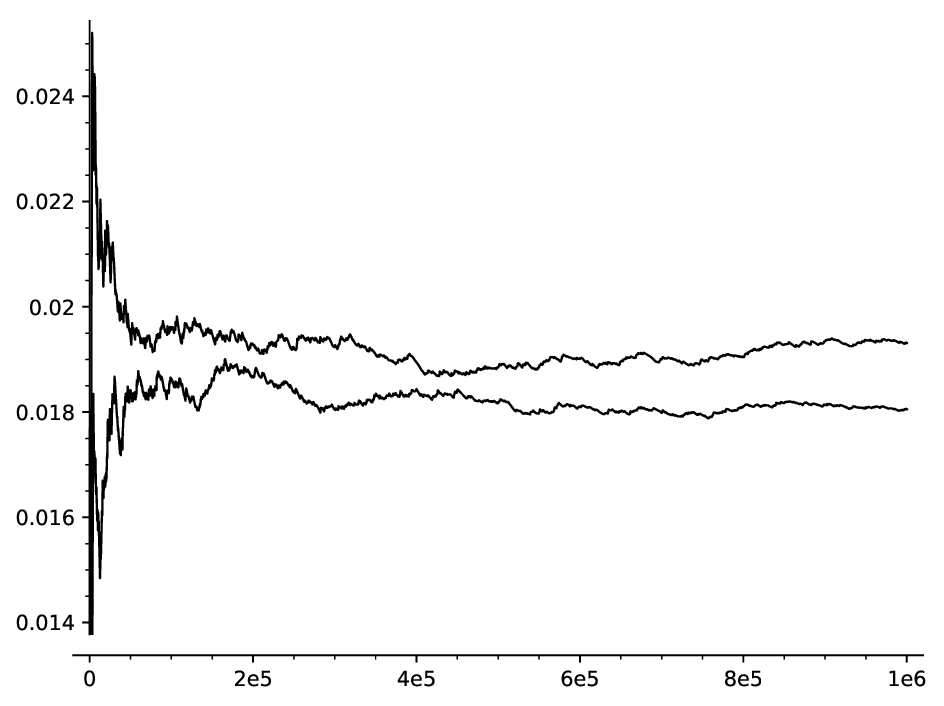}
\caption{$|l| = 4$: Top -4 bottom 4} \label{fig:15_6_1_3_A_4}
\end{subfigure}\hspace*{\fill}
\begin{subfigure}[b]{0.4\linewidth}
\includegraphics[width=\linewidth]{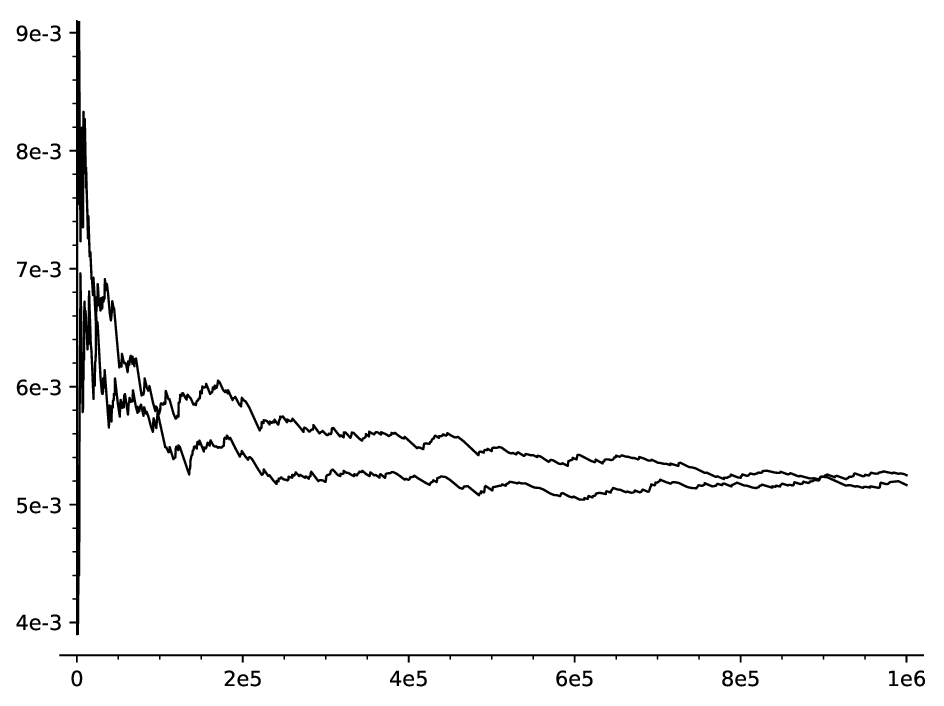}
\caption{$|l| = 5$: Top 5 bottom -5} \label{fig:15_6_1_3_A_5}
\end{subfigure}\hspace*{\fill}
\begin{subfigure}[b]{0.4\linewidth}
\includegraphics[width=\linewidth]{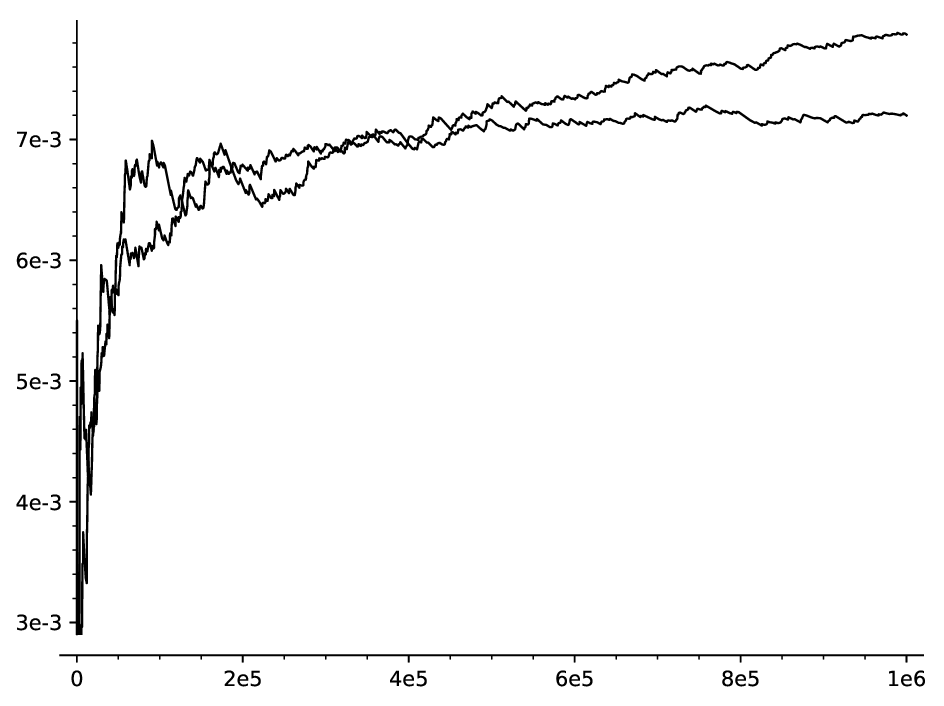}
\caption{$|l| = 6$: Top 6 bottom -6} \label{fig:15_6_1_3_A_6}
\end{subfigure}
\hspace*{-2.3cm}
\begin{subfigure}[b]{0.4\linewidth}
\includegraphics[width=\linewidth]{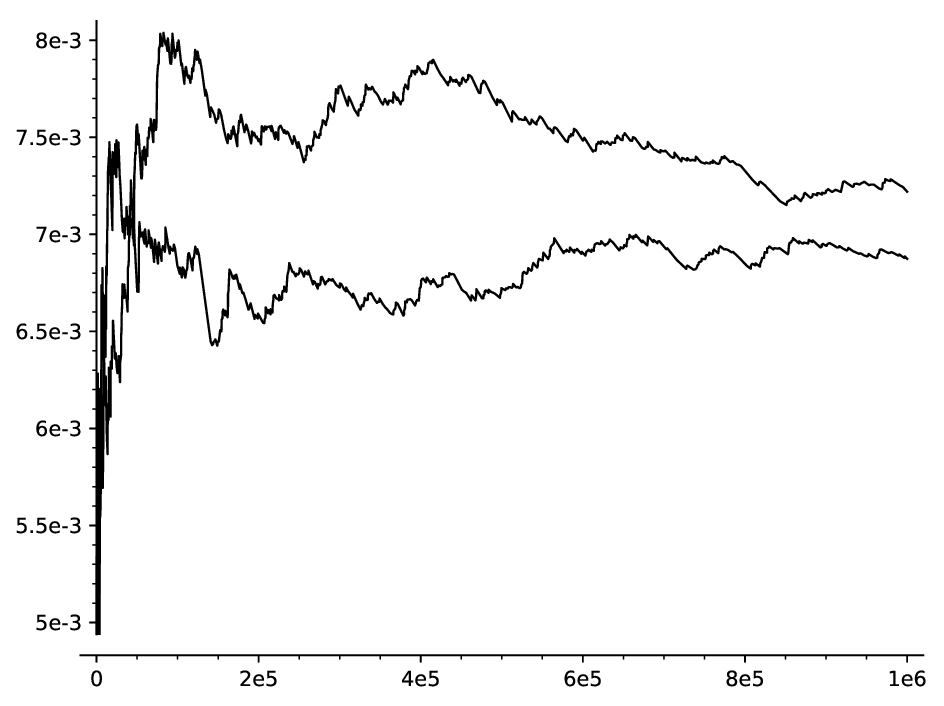}
\caption{$|l| = 7$: Top 7 bottom -7} \label{fig:15_6_1_3_A_7}
\end{subfigure}\hspace*{\fill}
\begin{subfigure}[b]{0.4\linewidth}
\includegraphics[width=\linewidth]{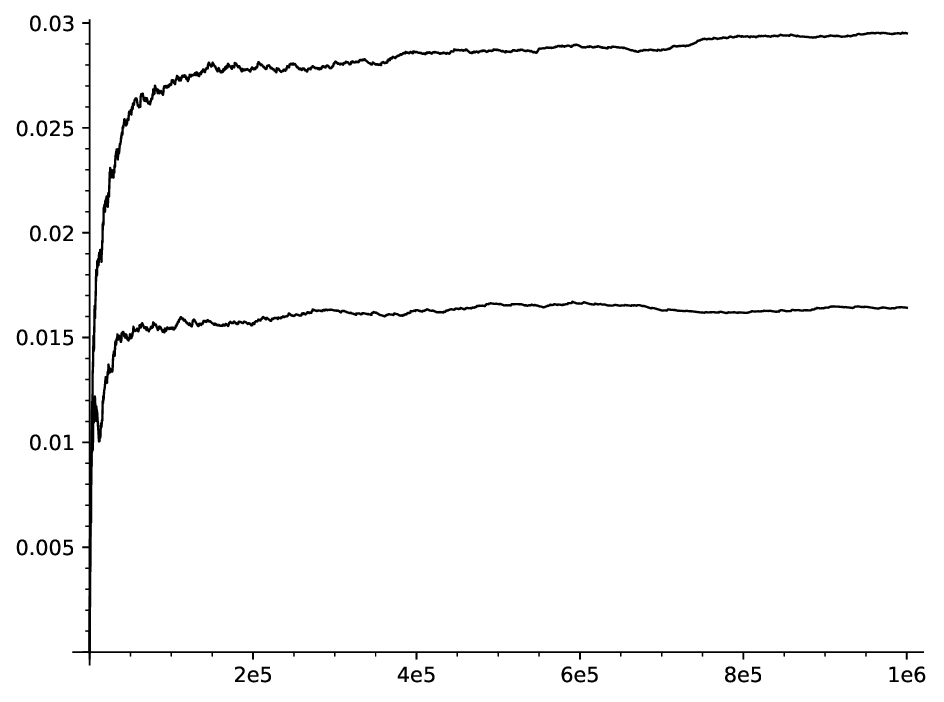}
\caption{$|l| = 8$: Top 8 bottom -8} \label{fig:15_6_1_3_A_8}
\end{subfigure}\hspace*{\fill}
\begin{subfigure}[b]{0.4\linewidth}
\includegraphics[width=\linewidth]{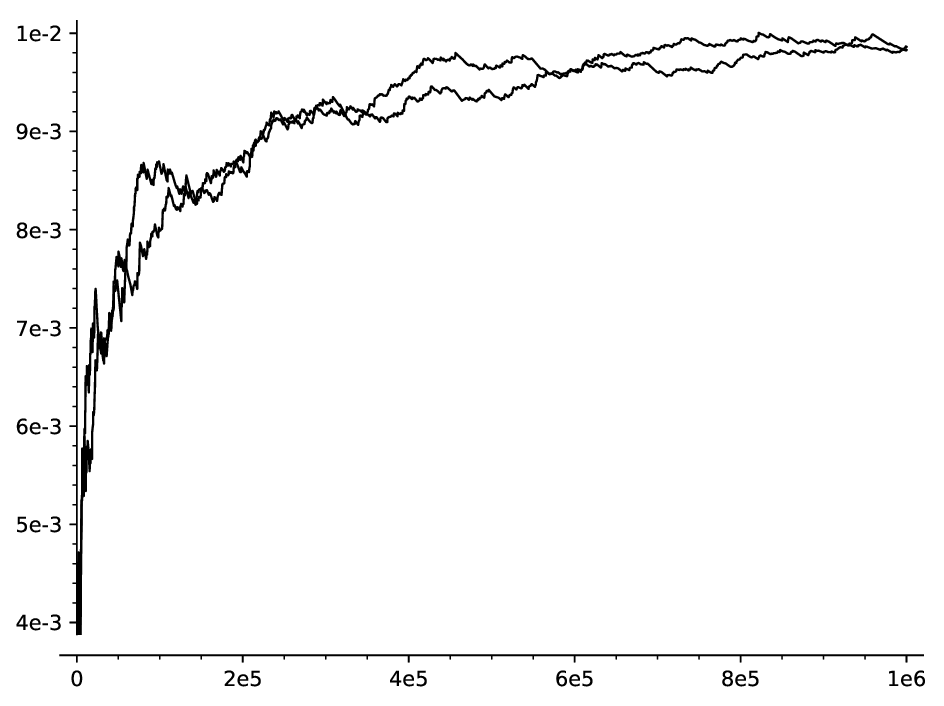}
\caption{$|l| = 9$: Top -9 bottom 9} \label{fig:15_6_1_3_A_9}
\end{subfigure}
\caption{15a1: $(\alpha, \beta) = (1,3)$ Ratio~\eqref{ratio_n_orders} $x_{6,E}^{(\alpha, \beta)}(X;l)/X^{1/2}\log^2(X)$} \label{fig:15a1_6_1_3_A_exact}
\end{figure}

\clearpage

\begin{figure}[t] 
\hspace*{-2.3cm}
\begin{subfigure}[b]{0.4\linewidth}
\includegraphics[width=\linewidth]{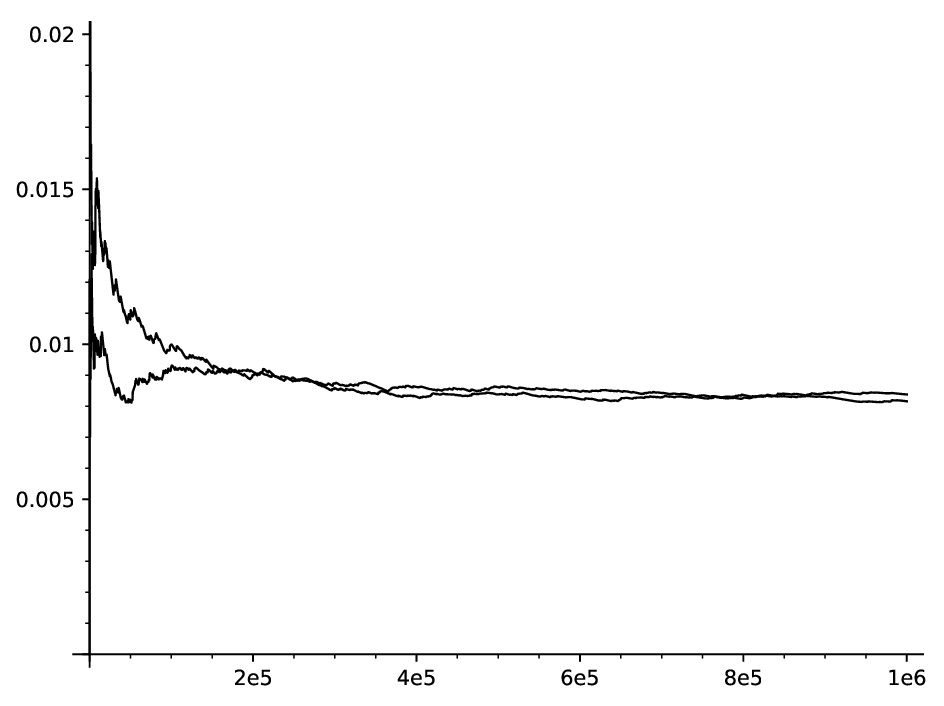}
\caption{$|l| = 1$: Top -1 bottom 1} \label{fig:15_6_2_3_A_1}
\end{subfigure}\hspace*{\fill}
\begin{subfigure}[b]{0.4\linewidth}
\includegraphics[width=\linewidth]{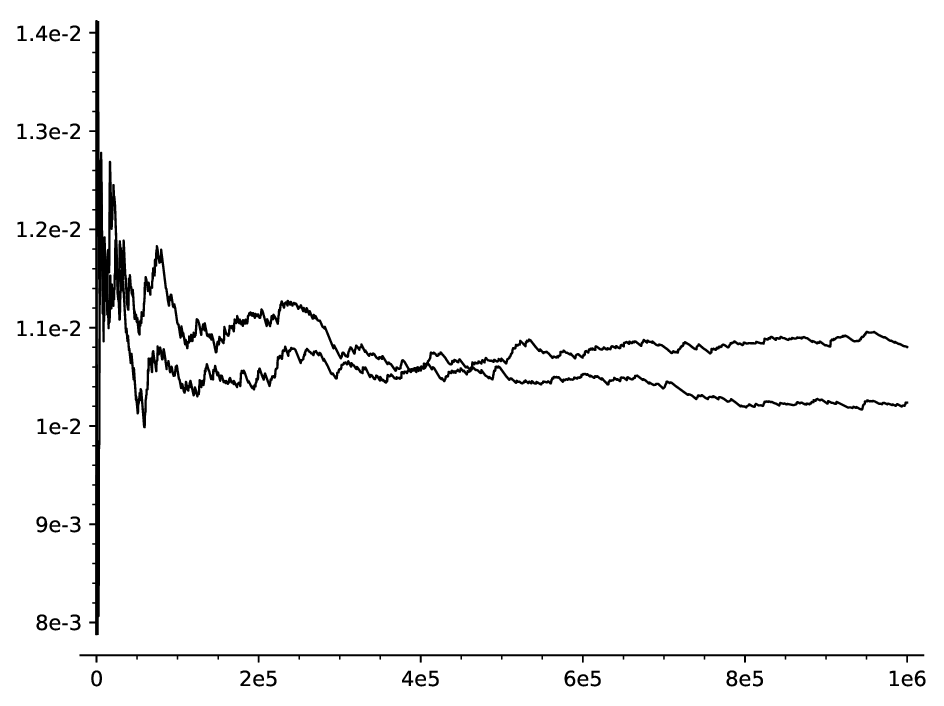}
\caption{$|l| = 2$: Top -2 bottom 2} \label{fig:15_6_2_3_A_2}
\end{subfigure}\hspace*{\fill}
\begin{subfigure}[b]{0.4\linewidth}
\includegraphics[width=\linewidth]{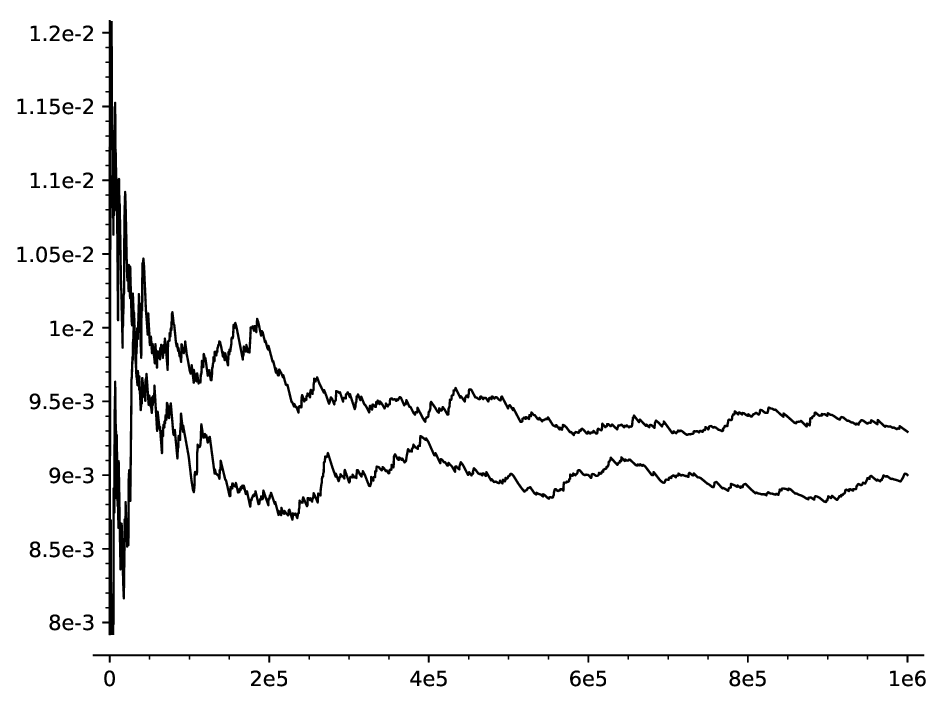}
\caption{$|l| = 3$: Top 3 bottom -3} \label{fig:15_6_2_3_A_3}
\end{subfigure}
\hspace*{-2.3cm}
\begin{subfigure}[b]{0.4\linewidth}
\includegraphics[width=\linewidth]{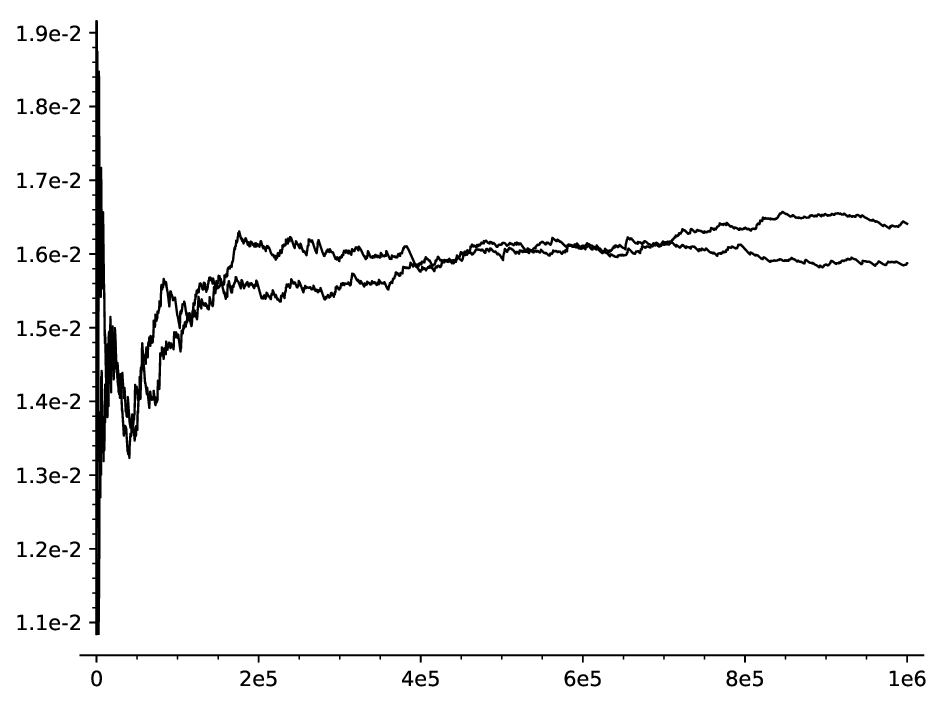}
\caption{$|l| = 4$: Top 4 bottom -4} \label{fig:15_6_2_3_A_4}
\end{subfigure}\hspace*{\fill}
\begin{subfigure}[b]{0.4\linewidth}
\includegraphics[width=\linewidth]{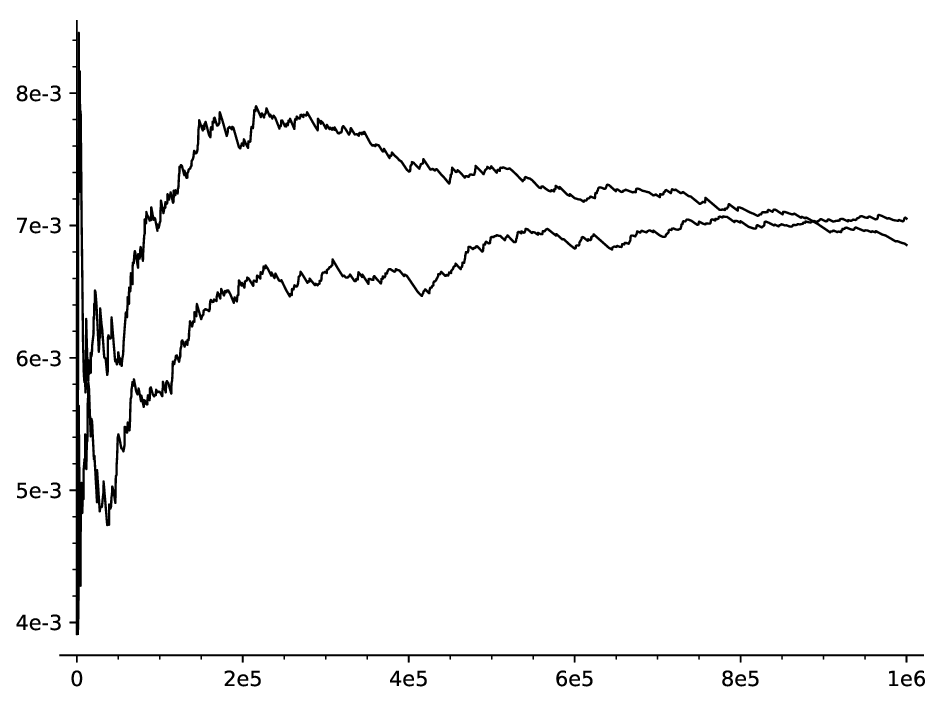}
\caption{$|l| = 5$: Top -5 bottom 5} \label{fig:15_6_2_3_A_5}
\end{subfigure}\hspace*{\fill}
\begin{subfigure}[b]{0.4\linewidth}
\includegraphics[width=\linewidth]{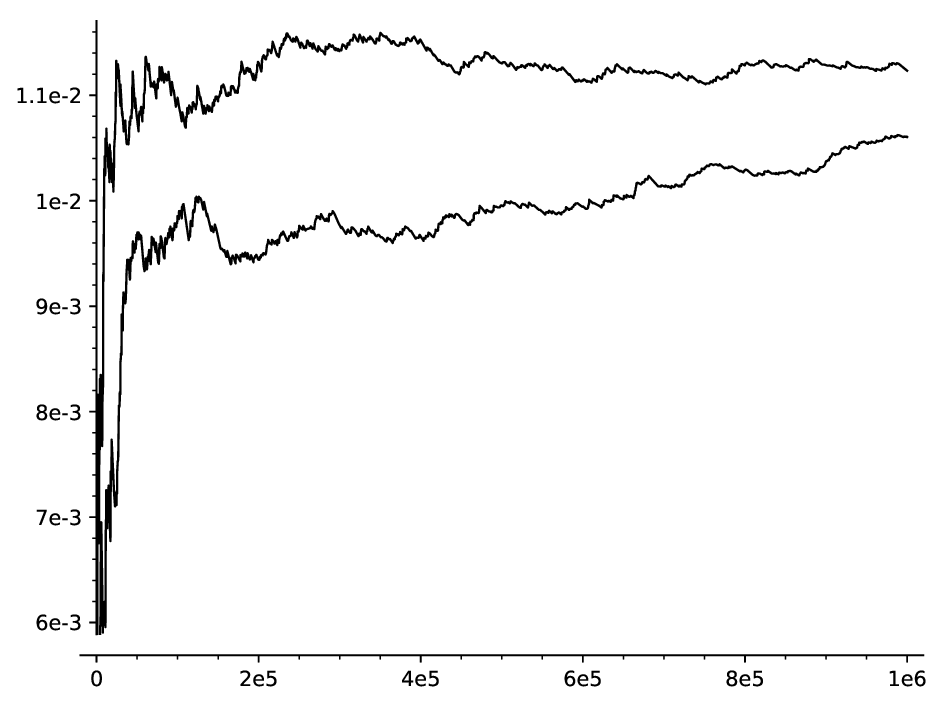}
\caption{$|l| = 6$: Top -6 bottom 6} \label{fig:15_6_2_3_A_6}
\end{subfigure}
\hspace*{-2.3cm}
\begin{subfigure}[b]{0.4\linewidth}
\includegraphics[width=\linewidth]{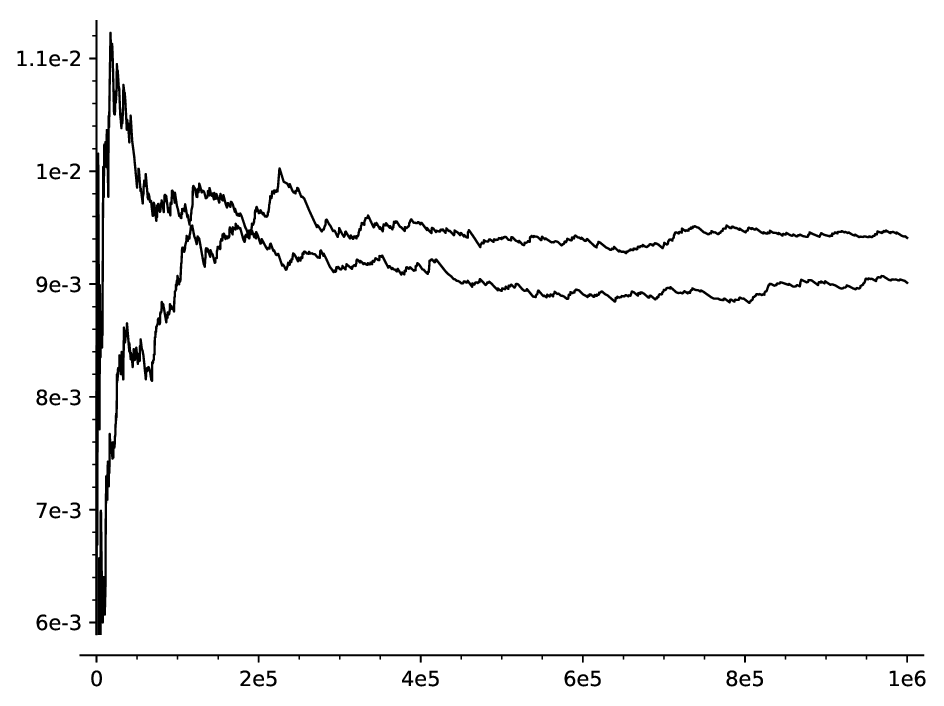}
\caption{$|l| = 7$: Top -7 bottom 7} \label{fig:15_6_2_3_A_7}
\end{subfigure}\hspace*{\fill}
\begin{subfigure}[b]{0.4\linewidth}
\includegraphics[width=\linewidth]{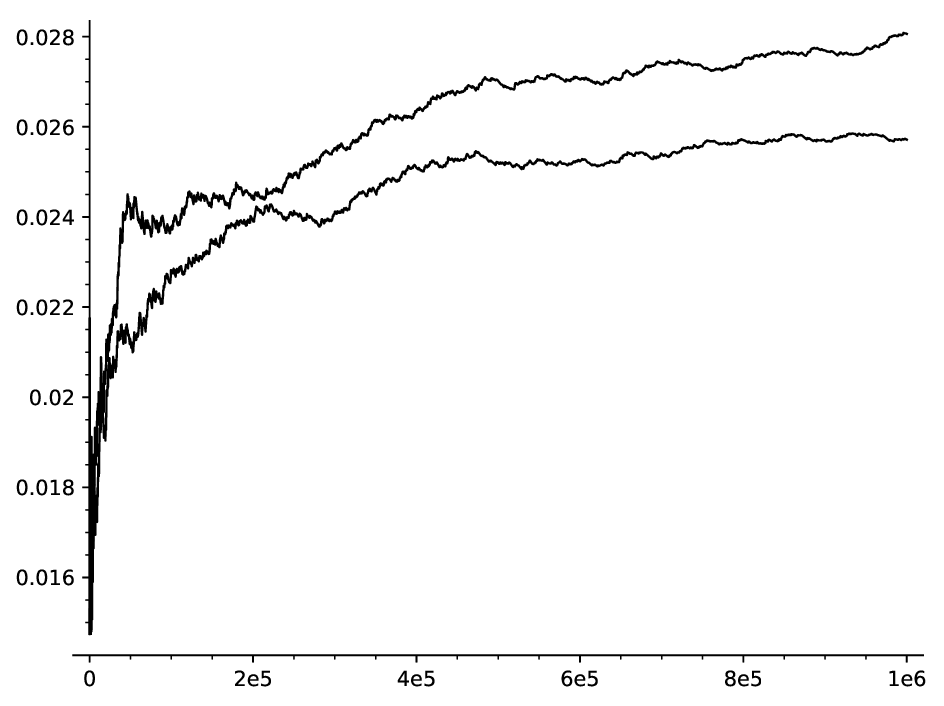}
\caption{$|l| = 8$: Top -8 bottom 8} \label{fig:15_6_2_3_A_8}
\end{subfigure}\hspace*{\fill}
\begin{subfigure}[b]{0.4\linewidth}
\includegraphics[width=\linewidth]{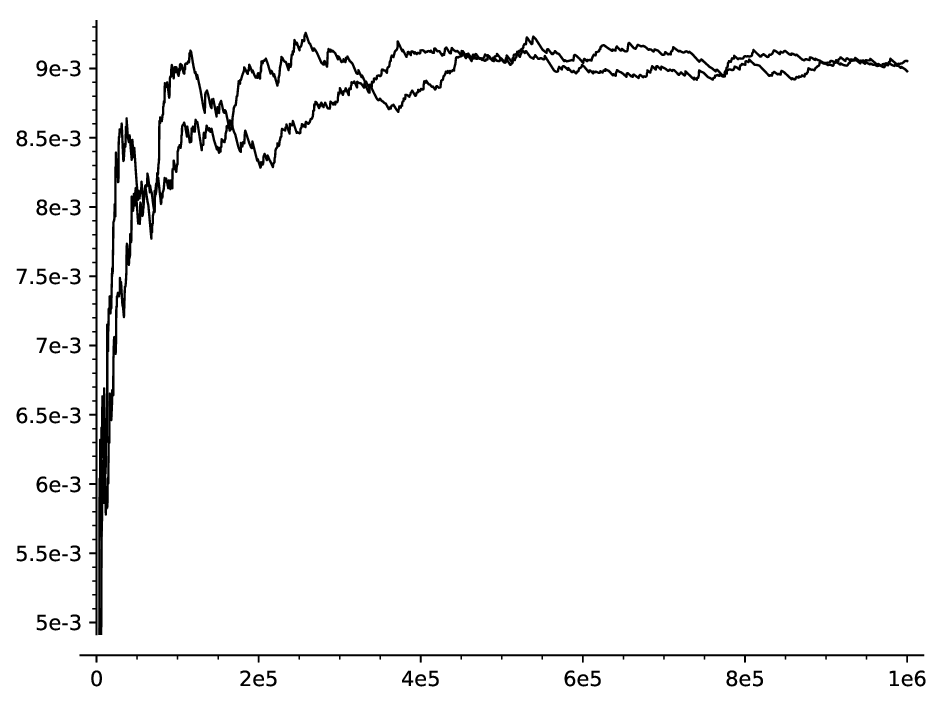}
\caption{$|l| = 9$: Top -9 bottom 9} \label{fig:15_6_2_3_A_9}
\end{subfigure}
\caption{15a1: $(\alpha, \beta) = (2,3)$ Ratio~\eqref{ratio_n_orders} $x_{6,E}^{(\alpha, \beta)}(X;l)/X^{1/2}\log^2(X)$} \label{fig:15a1_6_2_3_A_exact}
\end{figure}

\clearpage

\begin{figure}[t] 
\hspace*{-2.3cm}
\begin{subfigure}[b]{0.4\linewidth}
\includegraphics[width=\linewidth]{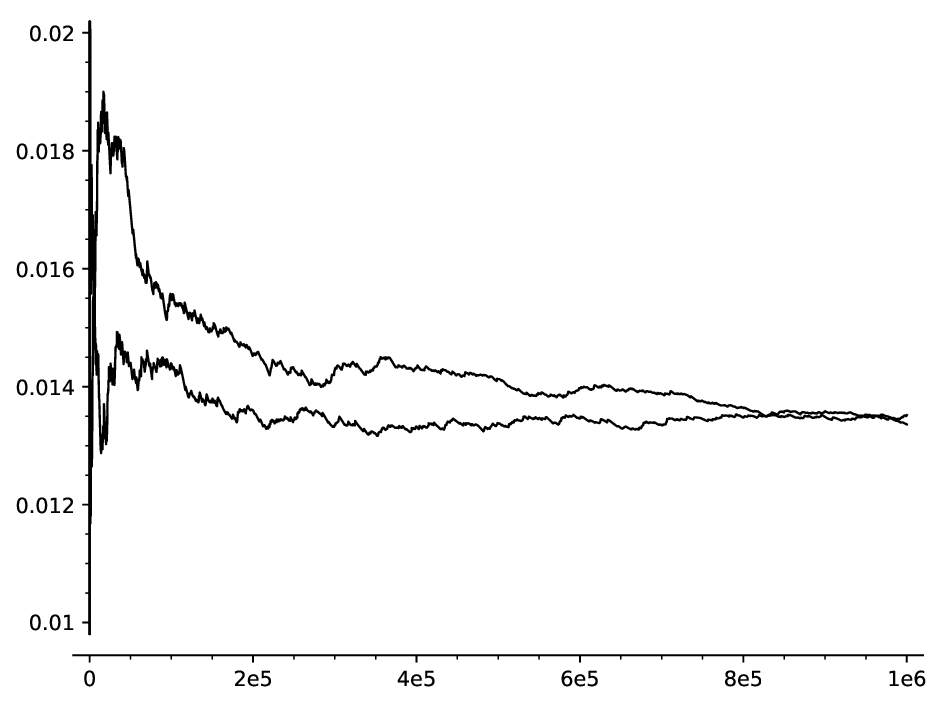}
\caption{$|l| = 1$: Top 1 bottom -1} \label{fig:15_6_1_6_A_1}
\end{subfigure}\hspace*{\fill}
\begin{subfigure}[b]{0.4\linewidth}
\includegraphics[width=\linewidth]{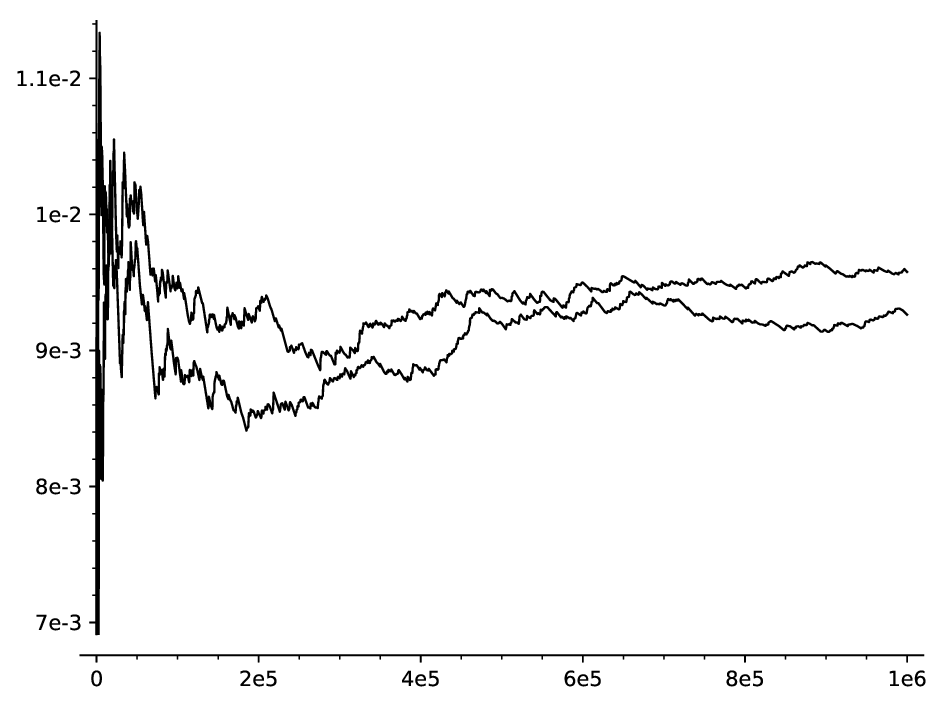}
\caption{$|l| = 2$: Top -2 bottom 2} \label{fig:15_6_1_6_A_2}
\end{subfigure}\hspace*{\fill}
\begin{subfigure}[b]{0.4\linewidth}
\includegraphics[width=\linewidth]{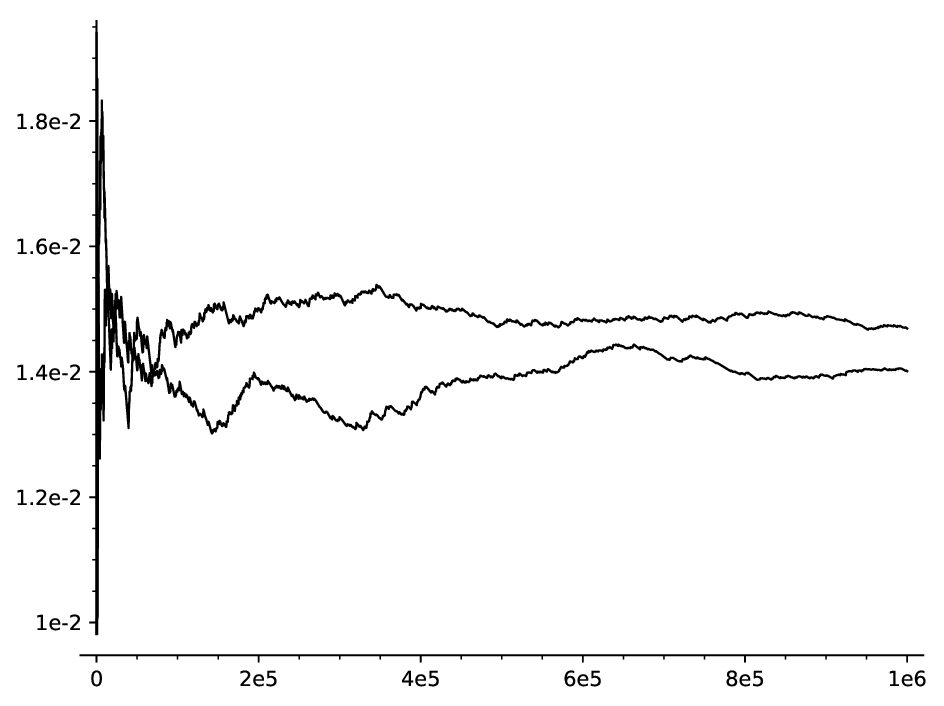}
\caption{$|l| = 3$: Top 3 bottom -3} \label{fig:15_6_1_6_A_3}
\end{subfigure}
\hspace*{-2.3cm}
\begin{subfigure}[b]{0.4\linewidth}
\includegraphics[width=\linewidth]{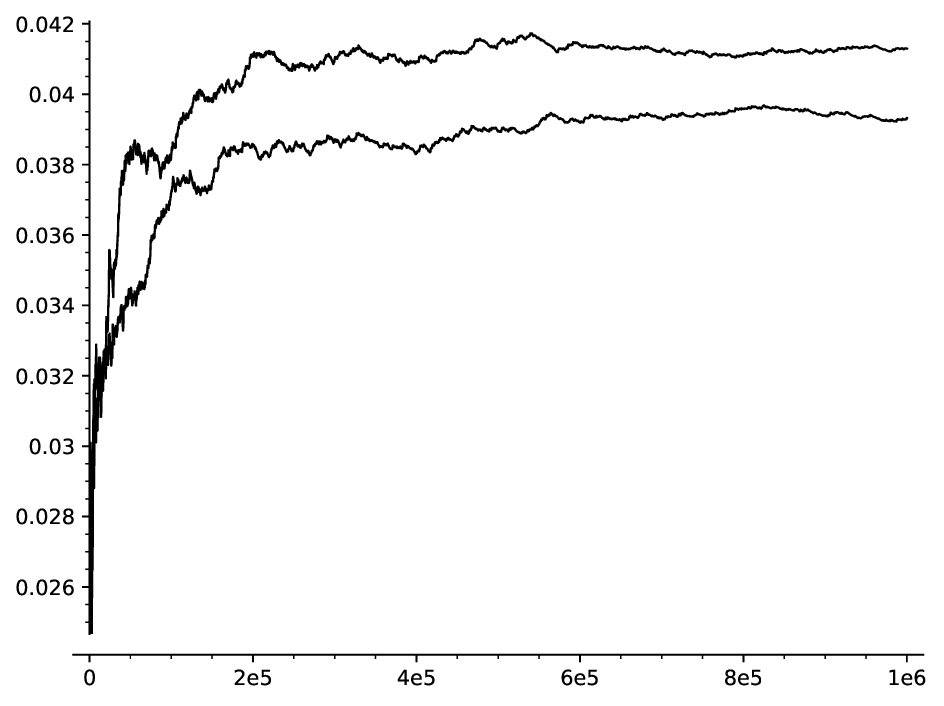}
\caption{$|l| = 4$: Top 4 bottom -4} \label{fig:15_6_1_6_A_4}
\end{subfigure}\hspace*{\fill}
\begin{subfigure}[b]{0.4\linewidth}
\includegraphics[width=\linewidth]{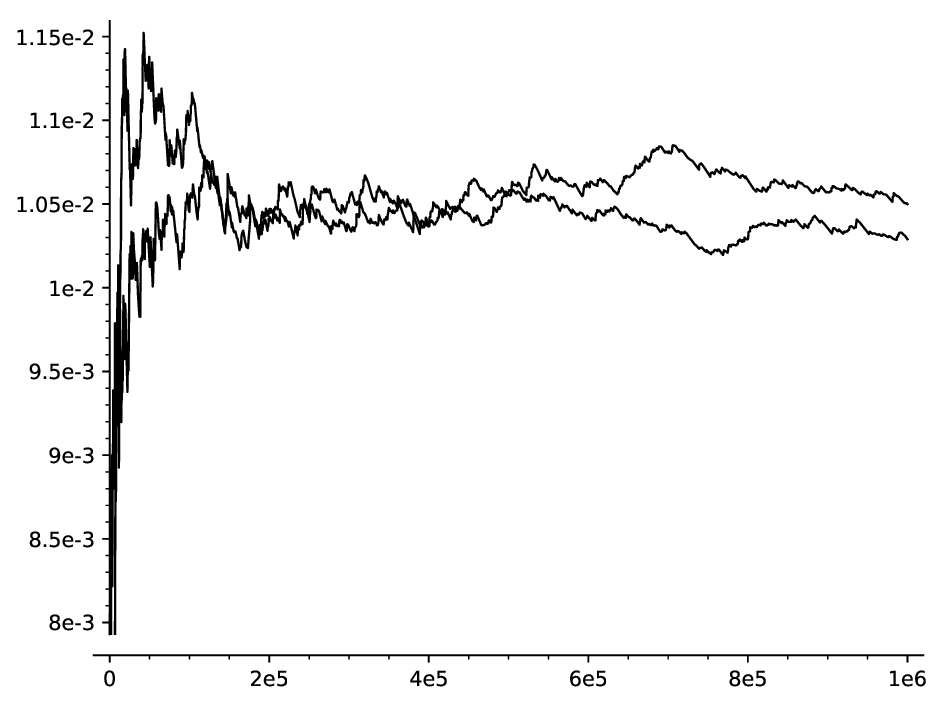}
\caption{$|l| = 5$: Top 5 bottom -5} \label{fig:15_6_1_6_A_5}
\end{subfigure}\hspace*{\fill}
\begin{subfigure}[b]{0.4\linewidth}
\includegraphics[width=\linewidth]{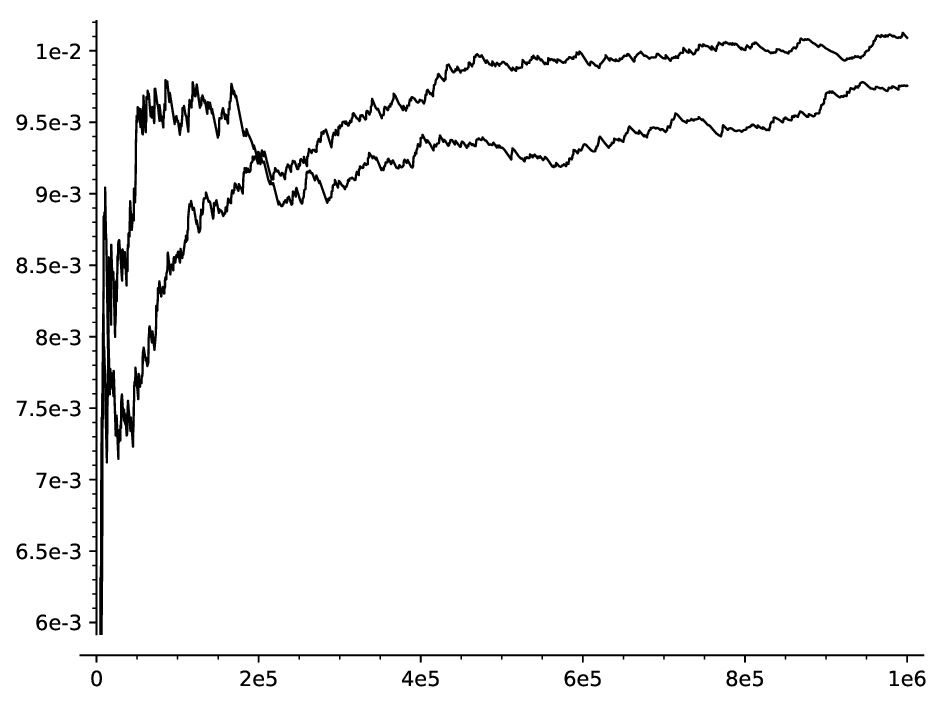}
\caption{$|l| = 6$: Top 6 bottom -6} \label{fig:15_6_1_6_A_6}
\end{subfigure}
\hspace*{-2.3cm}
\begin{subfigure}[b]{0.4\linewidth}
\includegraphics[width=\linewidth]{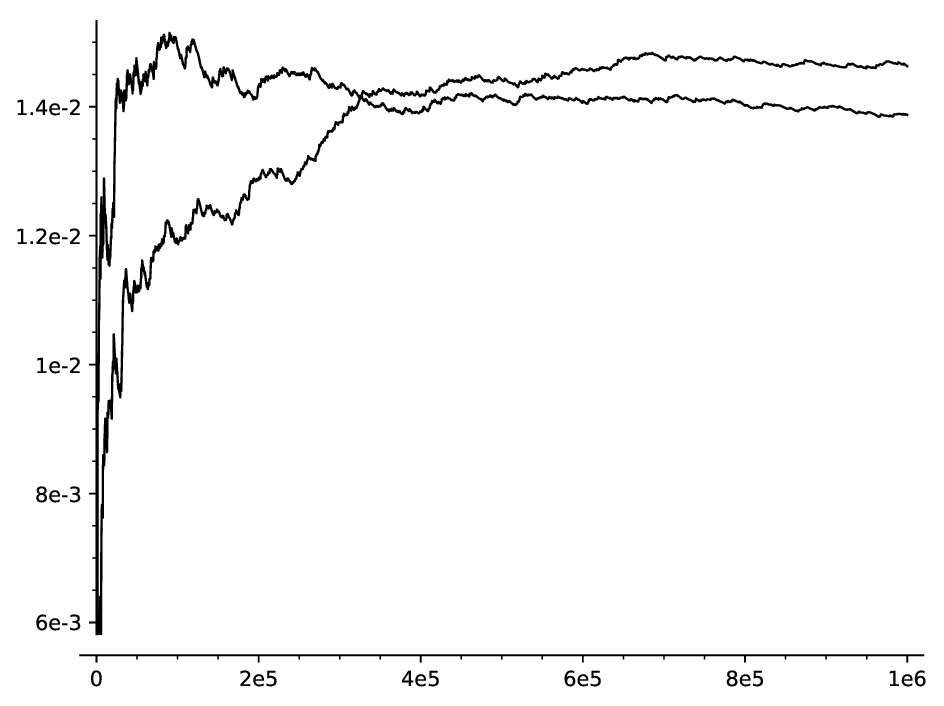}
\caption{$|l| = 7$: Top 7 bottom -7} \label{fig:15_6_1_6_A_7}
\end{subfigure}\hspace*{\fill}
\begin{subfigure}[b]{0.4\linewidth}
\includegraphics[width=\linewidth]{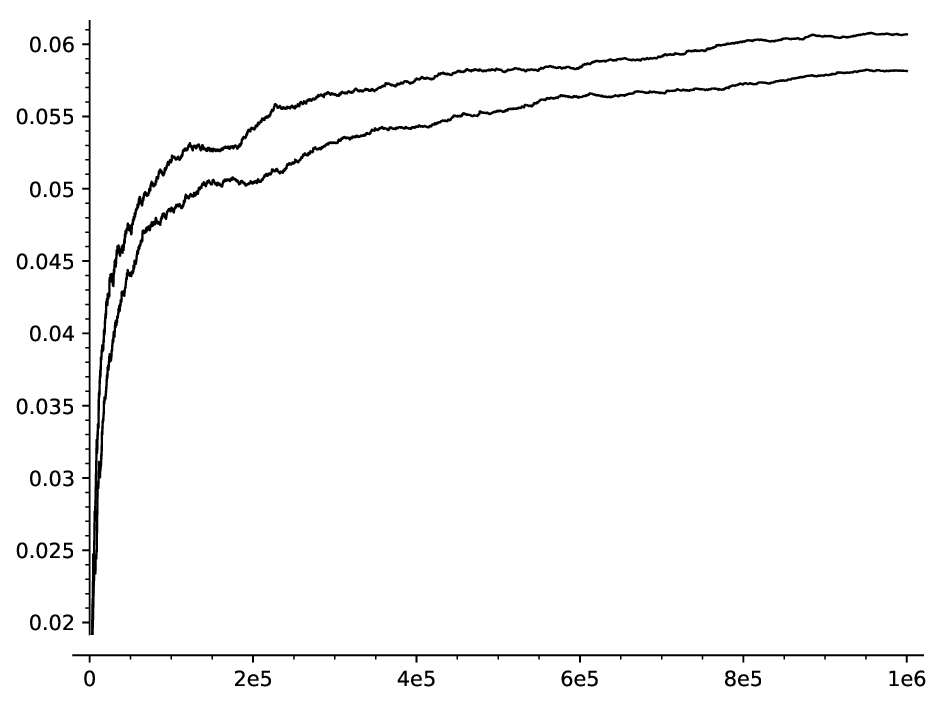}
\caption{$|l| = 8$: Top -8 bottom 8} \label{fig:15_6_1_6_A_8}
\end{subfigure}\hspace*{\fill}
\begin{subfigure}[b]{0.4\linewidth}
\includegraphics[width=\linewidth]{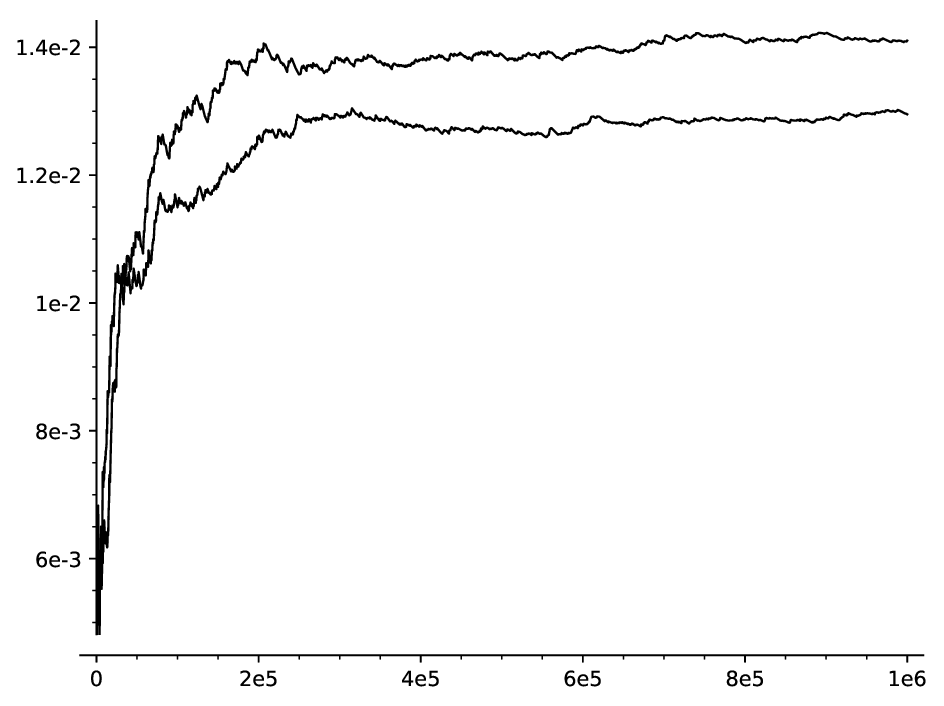}
\caption{$|l| = 9$: Top 9 bottom -9} \label{fig:15_6_1_6_A_9}
\end{subfigure}
\caption{15a1: $(\alpha, \beta) = (1,6)$ Ratio~\eqref{ratio_n_orders} $x_{6,E}^{(\alpha, \beta)}(X;l)/X^{1/2}\log^2(X)$} \label{fig:15a1_6_1_6_A_exact}
\end{figure}

\clearpage

\begin{figure}[t] 
\hspace*{-2.3cm}
\begin{subfigure}[b]{0.4\linewidth}
\includegraphics[width=\linewidth]{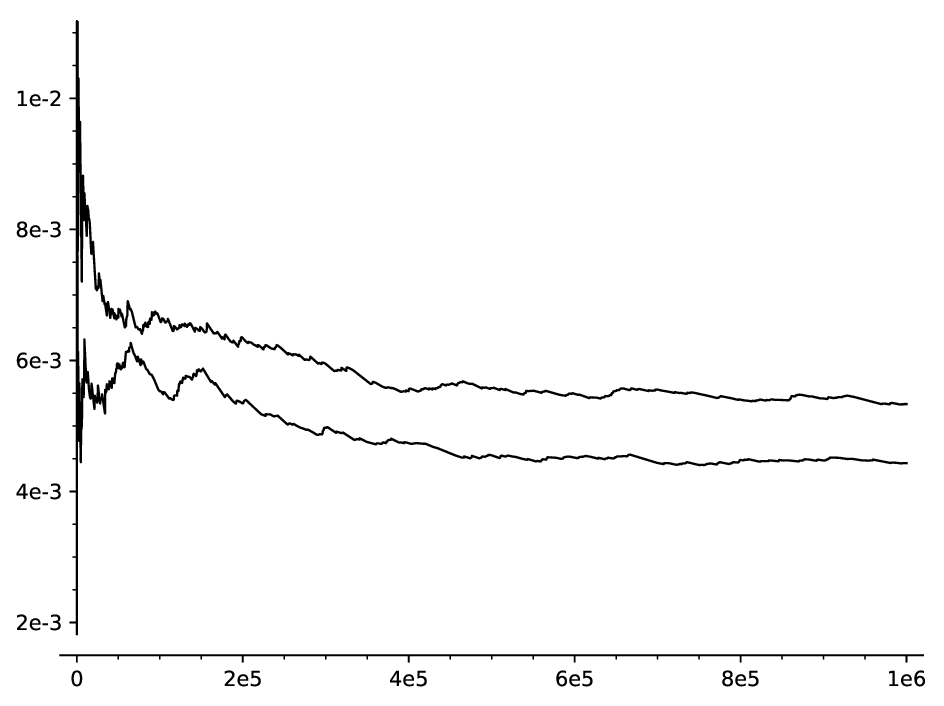}
\caption{$|l| = 1$: Top 1 bottom -1} \label{fig:15_6_2_6_A_1}
\end{subfigure}\hspace*{\fill}
\begin{subfigure}[b]{0.4\linewidth}
\includegraphics[width=\linewidth]{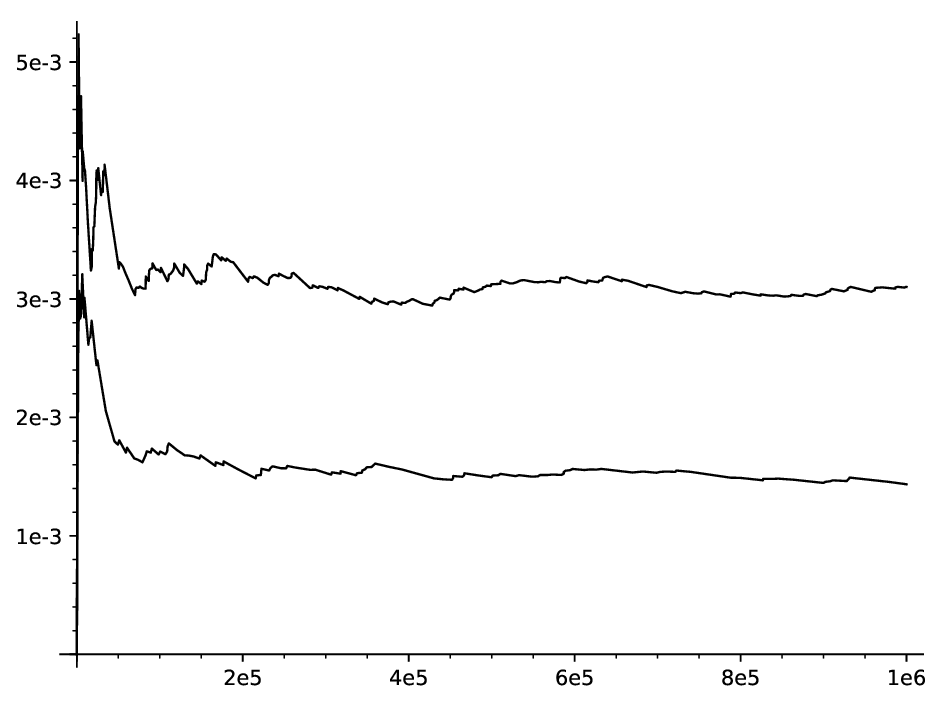}
\caption{$|l| = 2$: Top 2 bottom -2} \label{fig:15_6_2_6_A_2}
\end{subfigure}\hspace*{\fill}
\begin{subfigure}[b]{0.4\linewidth}
\includegraphics[width=\linewidth]{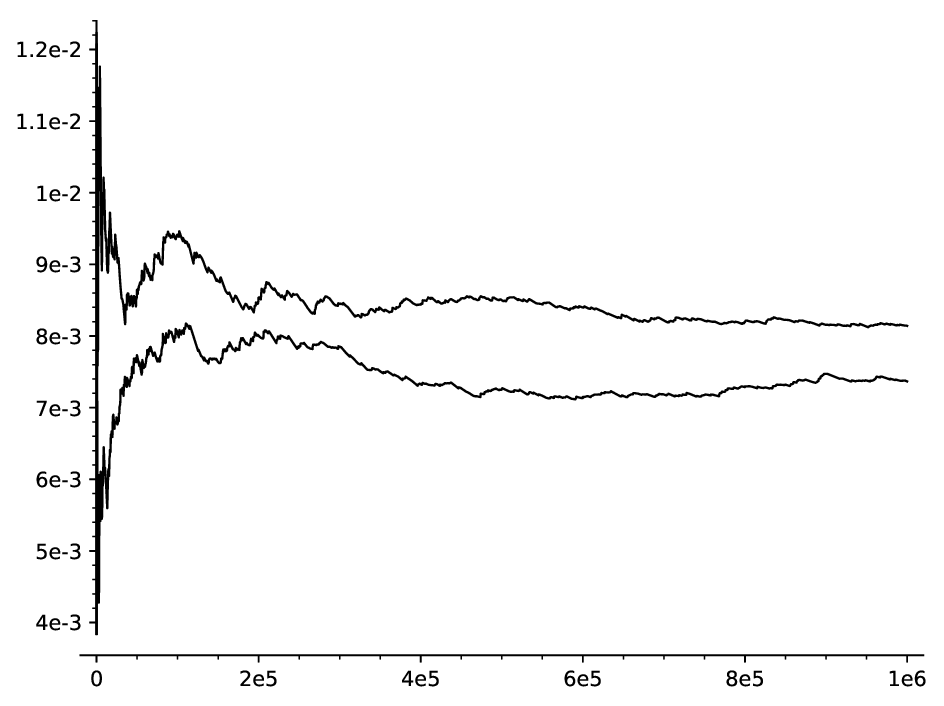}
\caption{$|l| = 3$: Top -3 bottom 3} \label{fig:15_6_2_6_A_3}
\end{subfigure}
\hspace*{-2.3cm}
\begin{subfigure}[b]{0.4\linewidth}
\includegraphics[width=\linewidth]{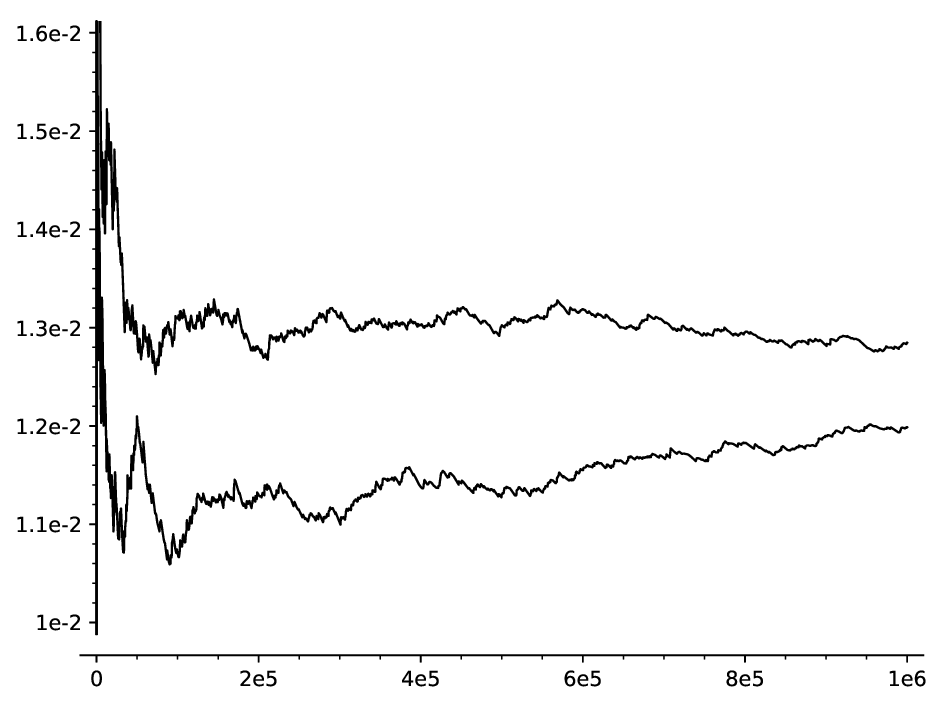}
\caption{$|l| = 4$: Top -4 bottom 4} \label{fig:15_6_2_6_A_4}
\end{subfigure}\hspace*{\fill}
\begin{subfigure}[b]{0.4\linewidth}
\includegraphics[width=\linewidth]{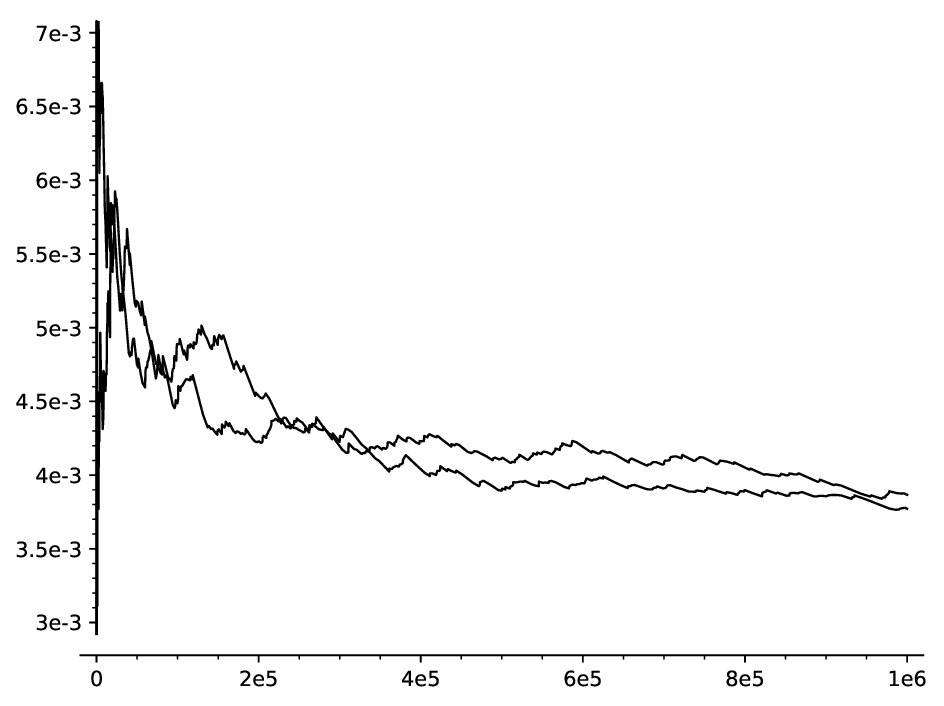}
\caption{$|l| = 5$: Top 5 bottom -5} \label{fig:15_6_2_6_A_5}
\end{subfigure}\hspace*{\fill}
\begin{subfigure}[b]{0.4\linewidth}
\includegraphics[width=\linewidth]{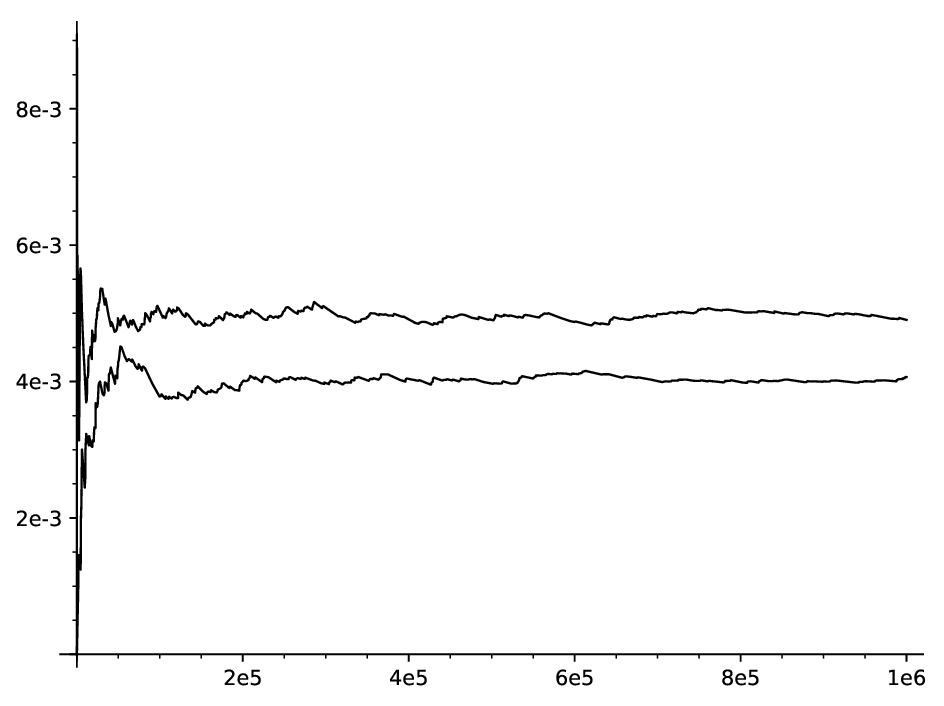}
\caption{$|l| = 6$: Top 6 bottom -6} \label{fig:15_6_2_6_A_6}
\end{subfigure}
\hspace*{-2.3cm}
\begin{subfigure}[b]{0.4\linewidth}
\includegraphics[width=\linewidth]{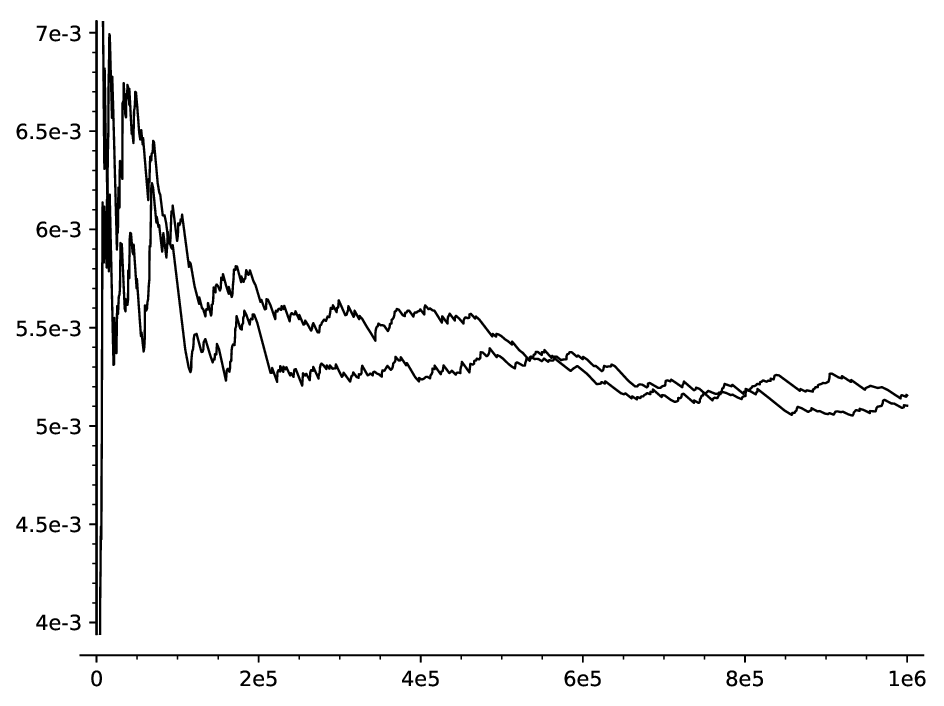}
\caption{$|l| = 7$: Top -7 bottom 7} \label{fig:15_6_2_6_A_7}
\end{subfigure}\hspace*{\fill}
\begin{subfigure}[b]{0.4\linewidth}
\includegraphics[width=\linewidth]{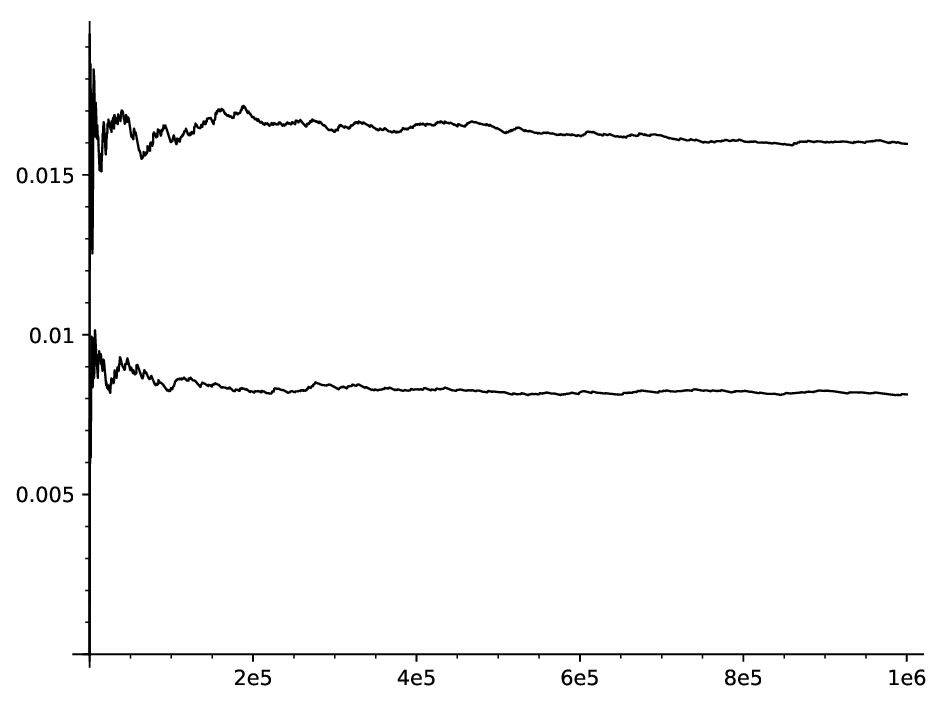}
\caption{$|l| = 8$: Top 8 bottom -8} \label{fig:15_6_2_6_A_8}
\end{subfigure}\hspace*{\fill}
\begin{subfigure}[b]{0.4\linewidth}
\includegraphics[width=\linewidth]{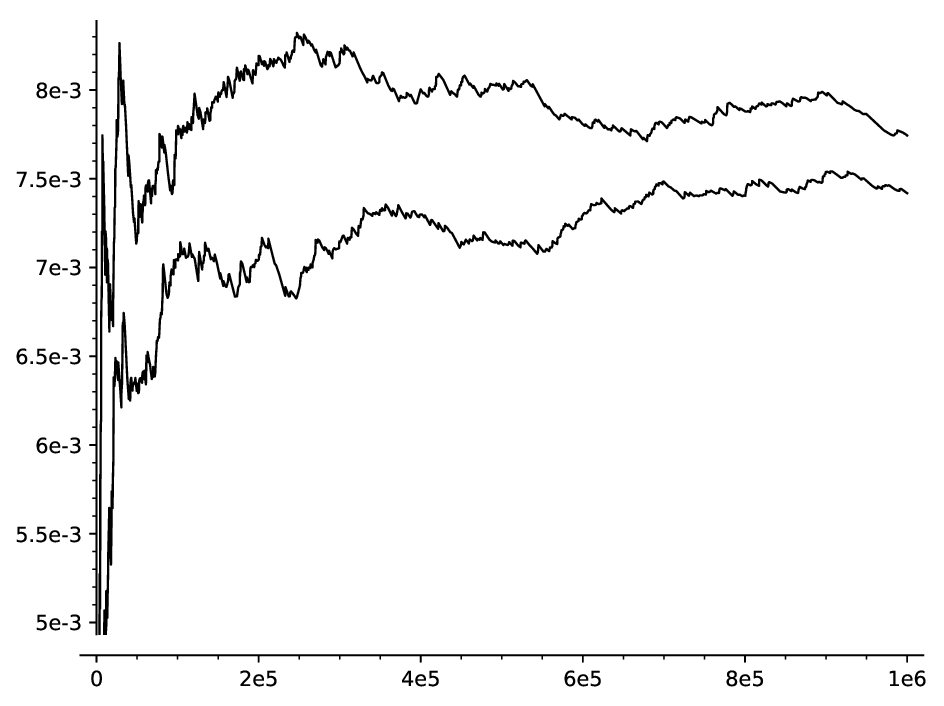}
\caption{$|l| = 9$: Top 9 bottom -9} \label{fig:15_6_2_6_A_9}
\end{subfigure}
\caption{15a1: $(\alpha, \beta) = (2,6)$ Ratio~\eqref{ratio_n_orders} $x_{6,E}^{(\alpha, \beta)}(X;l)/X^{1/2}\log^2(X)$} \label{fig:15a1_6_2_6_A_exact}
\end{figure}

\clearpage

\begin{figure}[t] 
\hspace*{-2.3cm}
\begin{subfigure}[b]{0.4\linewidth}
\includegraphics[width=\linewidth]{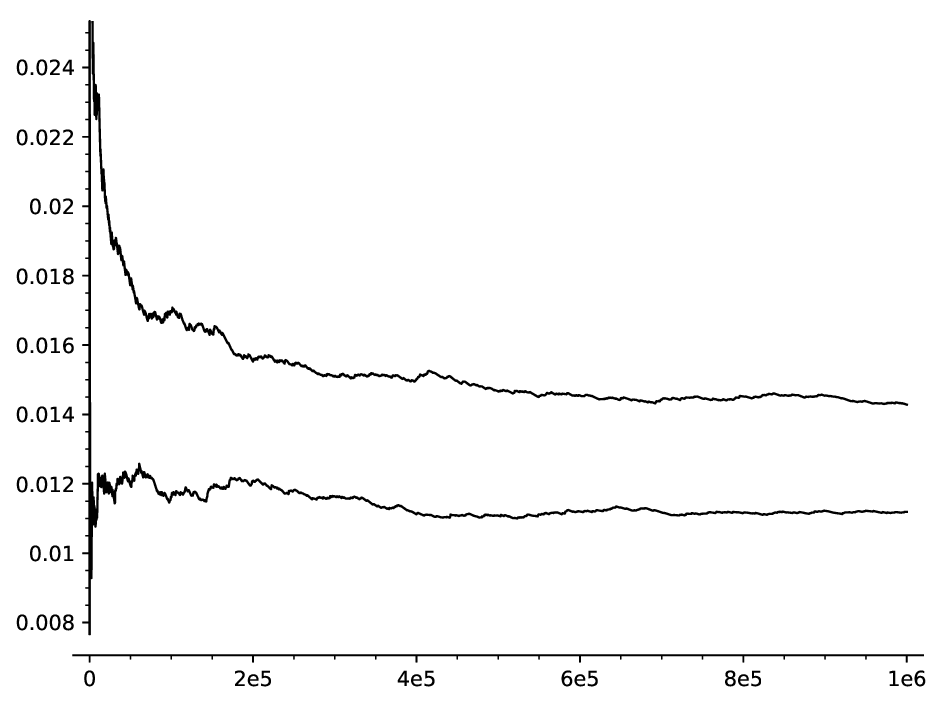}
\caption{$|l| = 1$: Top -1 bottom 1} \label{fig:17_6_1_3_A_1}
\end{subfigure}\hspace*{\fill}
\begin{subfigure}[b]{0.4\linewidth}
\includegraphics[width=\linewidth]{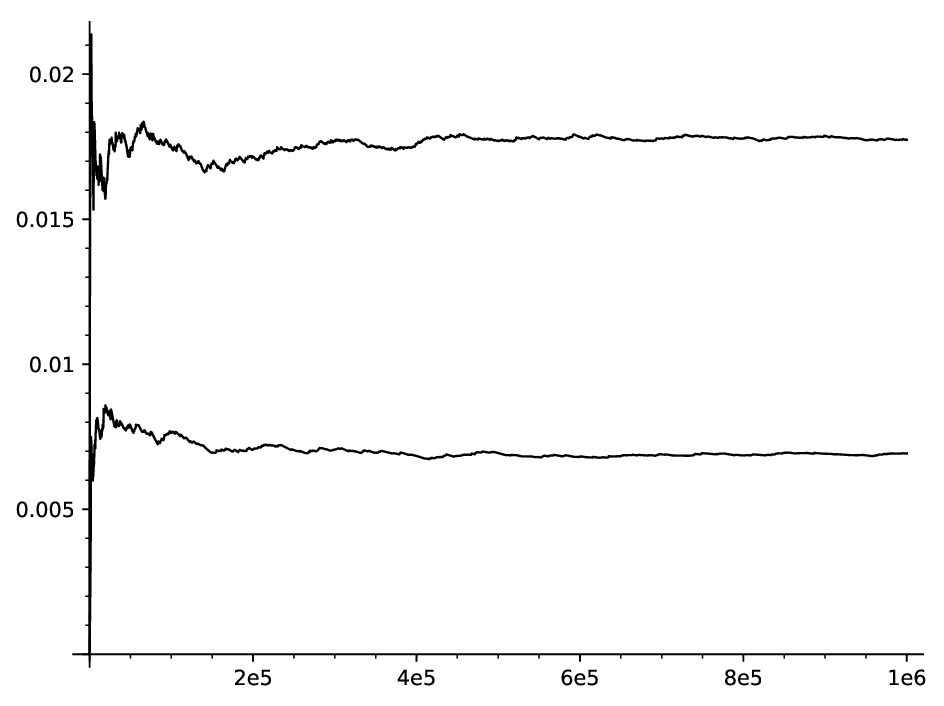}
\caption{$|l| = 2$: Top 2 bottom -2} \label{fig:17_6_1_3_A_2}
\end{subfigure}\hspace*{\fill}
\begin{subfigure}[b]{0.4\linewidth}
\includegraphics[width=\linewidth]{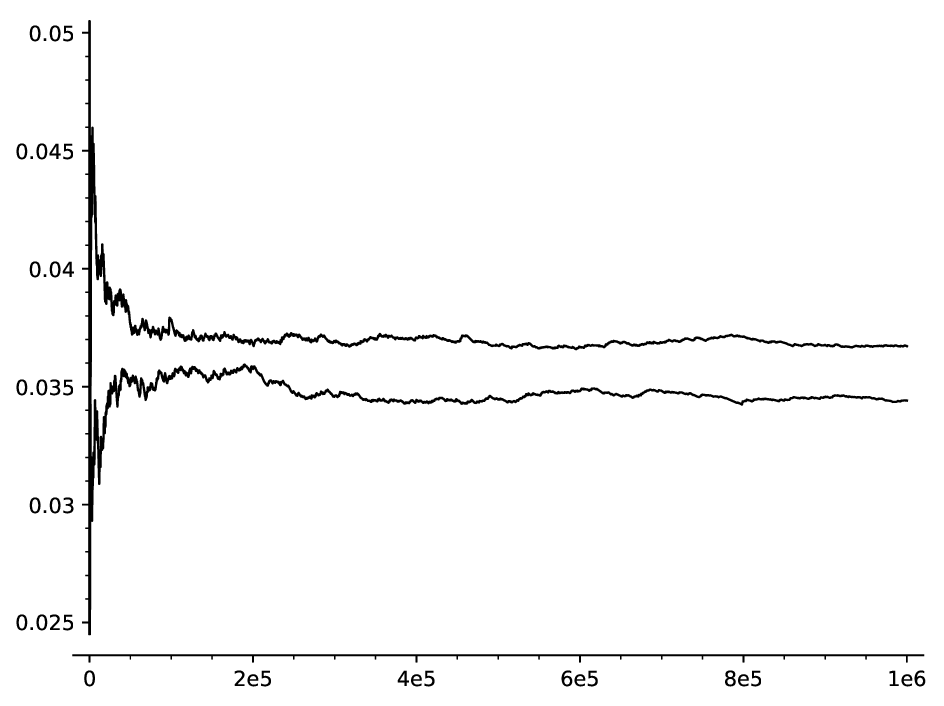}
\caption{$|l| = 3$: Top -3 bottom 3} \label{fig:17_6_1_3_A_3}
\end{subfigure}
\hspace*{-2.3cm}
\begin{subfigure}[b]{0.4\linewidth}
\includegraphics[width=\linewidth]{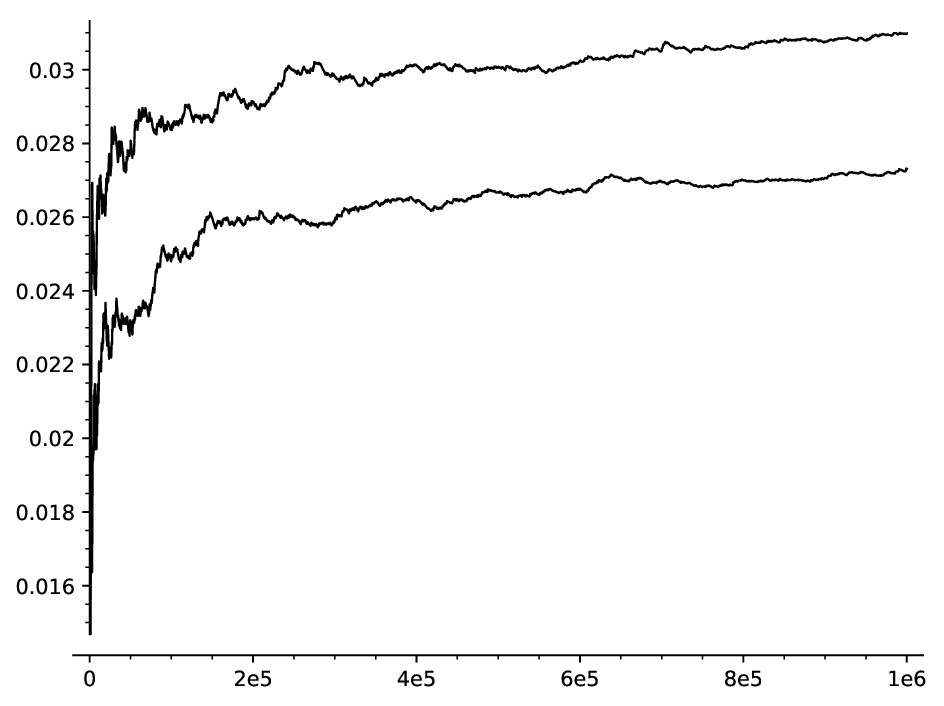}
\caption{$|l| = 4$: Top -4 bottom 4} \label{fig:17_6_1_3_A_4}
\end{subfigure}\hspace*{\fill}
\begin{subfigure}[b]{0.4\linewidth}
\includegraphics[width=\linewidth]{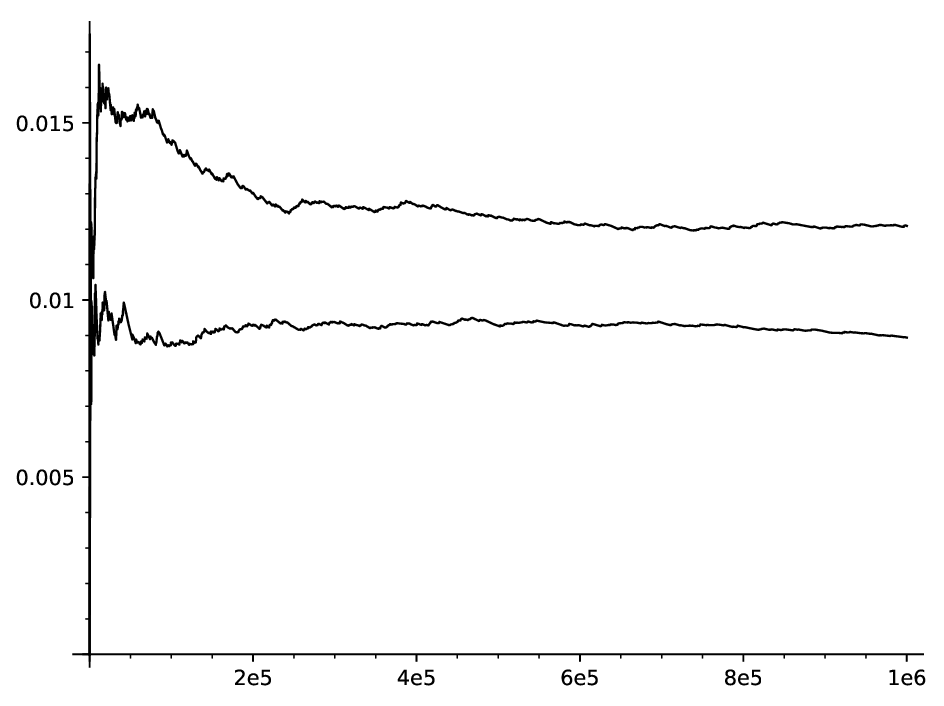}
\caption{$|l| = 5$: Top 5 bottom -5} \label{fig:17_6_1_3_A_5}
\end{subfigure}\hspace*{\fill}
\begin{subfigure}[b]{0.4\linewidth}
\includegraphics[width=\linewidth]{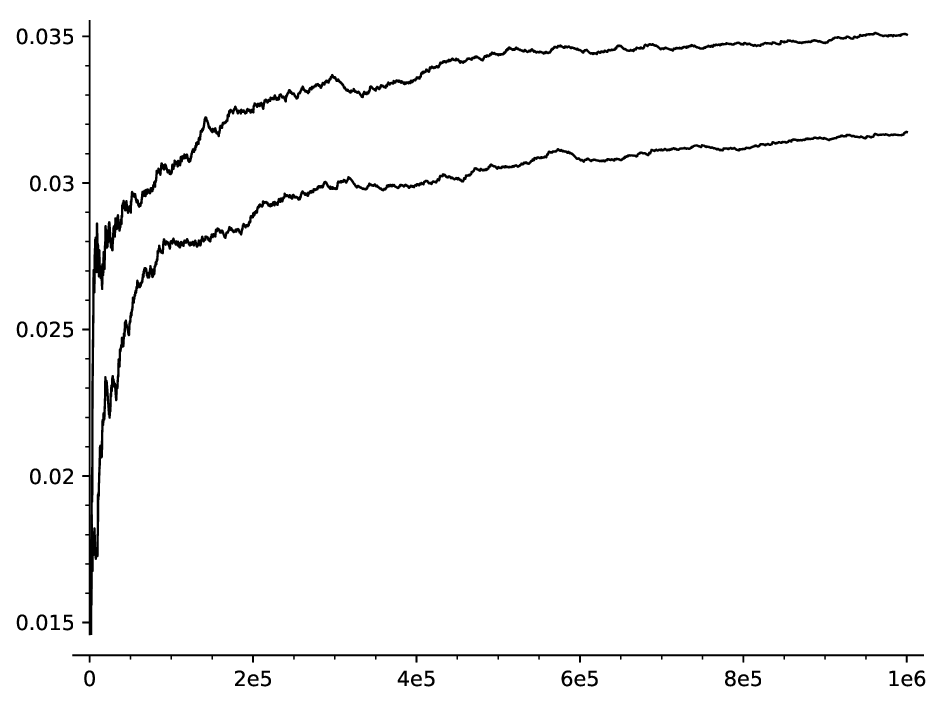}
\caption{$|l| = 6$: Top 6 bottom -6} \label{fig:17_6_1_3_A_6}
\end{subfigure}
\hspace*{-2.3cm}
\begin{subfigure}[b]{0.4\linewidth}
\includegraphics[width=\linewidth]{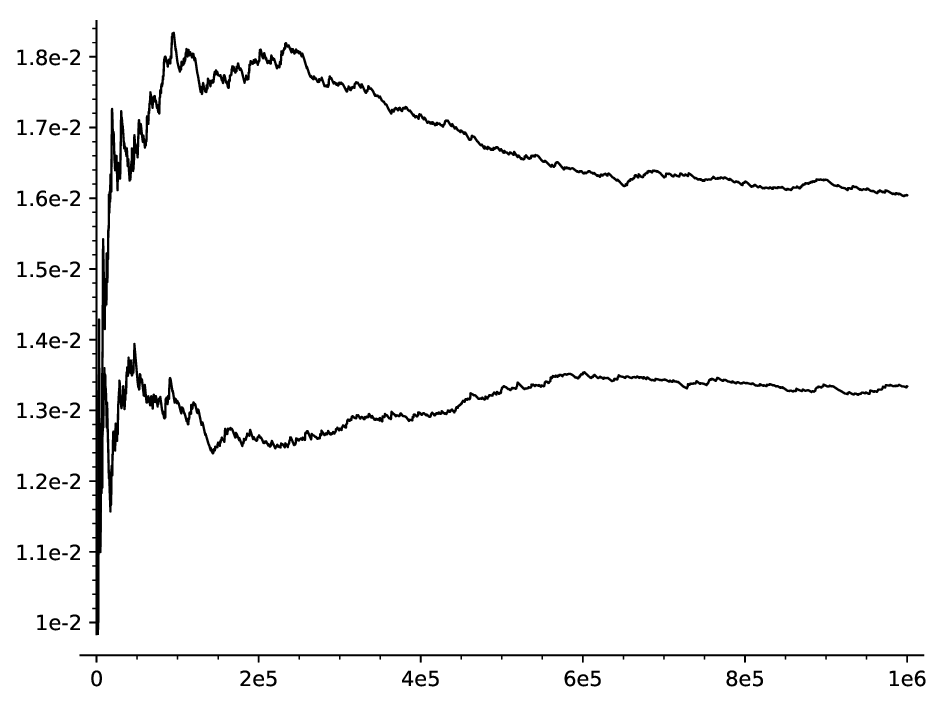}
\caption{$|l| = 7$: Top -7 bottom 7} \label{fig:17_6_1_3_A_7}
\end{subfigure}\hspace*{\fill}
\begin{subfigure}[b]{0.4\linewidth}
\includegraphics[width=\linewidth]{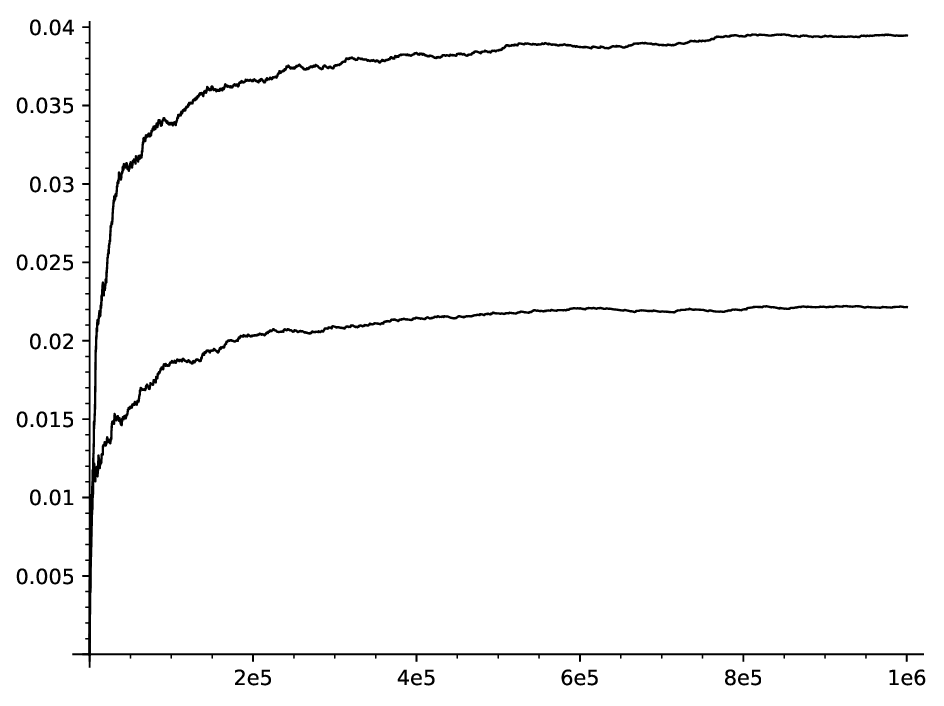}
\caption{$|l| = 8$: Top 8 bottom -8} \label{fig:17_6_1_3_A_8}
\end{subfigure}\hspace*{\fill}
\begin{subfigure}[b]{0.4\linewidth}
\includegraphics[width=\linewidth]{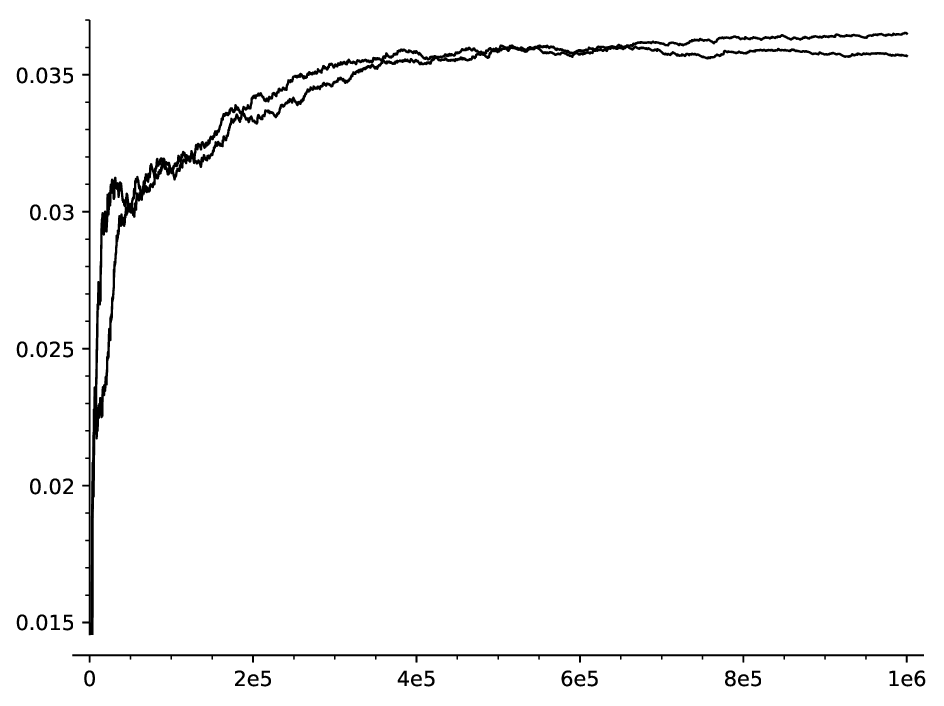}
\caption{$|l| = 9$: Top -9 bottom 9} \label{fig:17_6_1_3_A_9}
\end{subfigure}
\caption{17a1: $(\alpha, \beta) = (1,3)$ Ratio~\eqref{ratio_n_orders} $x_{6,E}^{(\alpha, \beta)}(X;l)/X^{1/2}\log^2(X)$} \label{fig:17a1_6_1_3_A_exact}
\end{figure}

\clearpage

\begin{figure}[t] 
\hspace*{-2.3cm}
\begin{subfigure}[b]{0.4\linewidth}
\includegraphics[width=\linewidth]{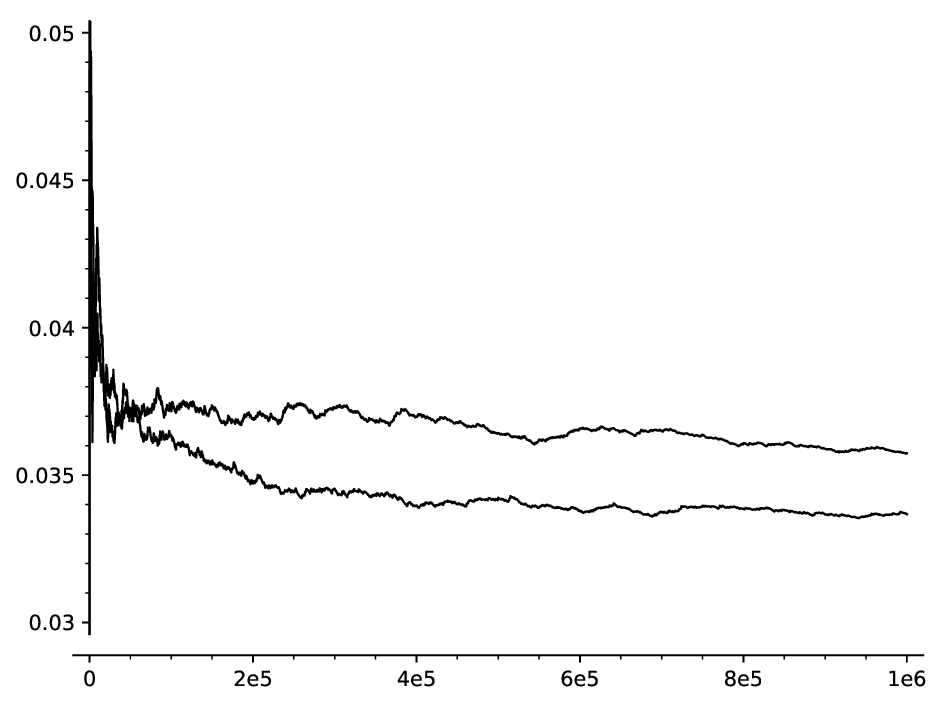}
\caption{$|l| = 1$: Top -1 bottom 1} \label{fig:17_6_2_3_A_1}
\end{subfigure}\hspace*{\fill}
\begin{subfigure}[b]{0.4\linewidth}
\includegraphics[width=\linewidth]{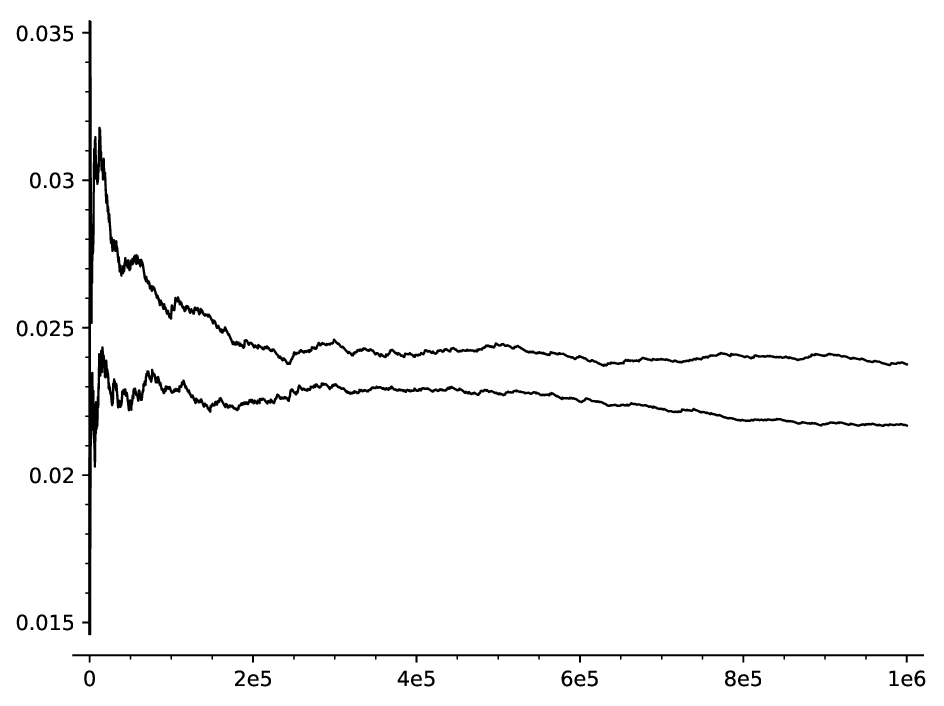}
\caption{$|l| = 2$: Top 2 bottom -2} \label{fig:17_6_2_3_A_2}
\end{subfigure}\hspace*{\fill}
\begin{subfigure}[b]{0.4\linewidth}
\includegraphics[width=\linewidth]{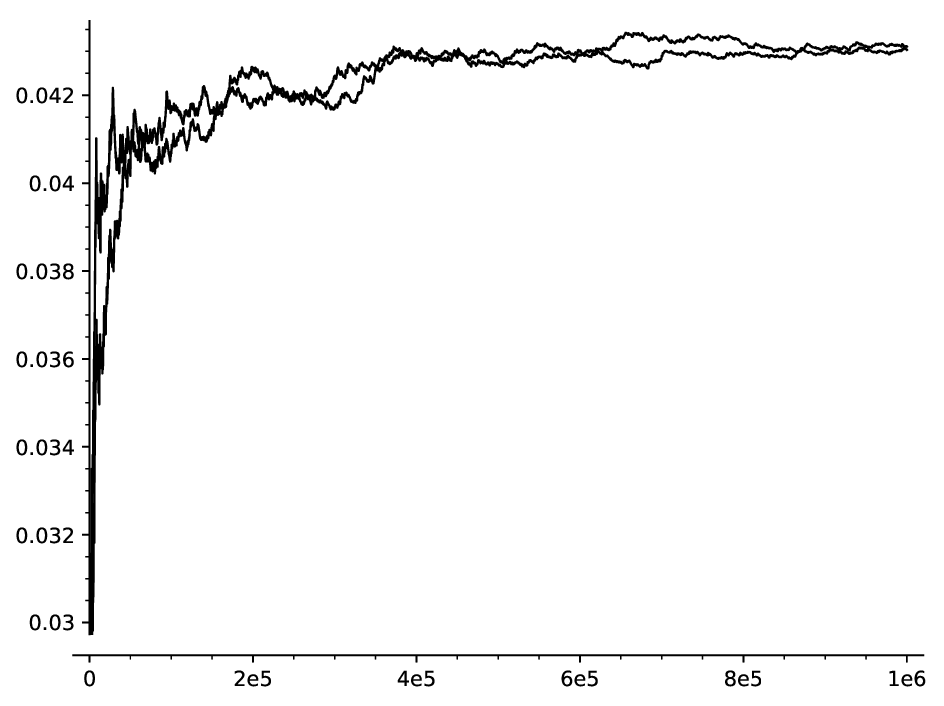}
\caption{$|l| = 3$: Top -3 bottom 3} \label{fig:17_6_2_3_A_3}
\end{subfigure}
\hspace*{-2.3cm}
\begin{subfigure}[b]{0.4\linewidth}
\includegraphics[width=\linewidth]{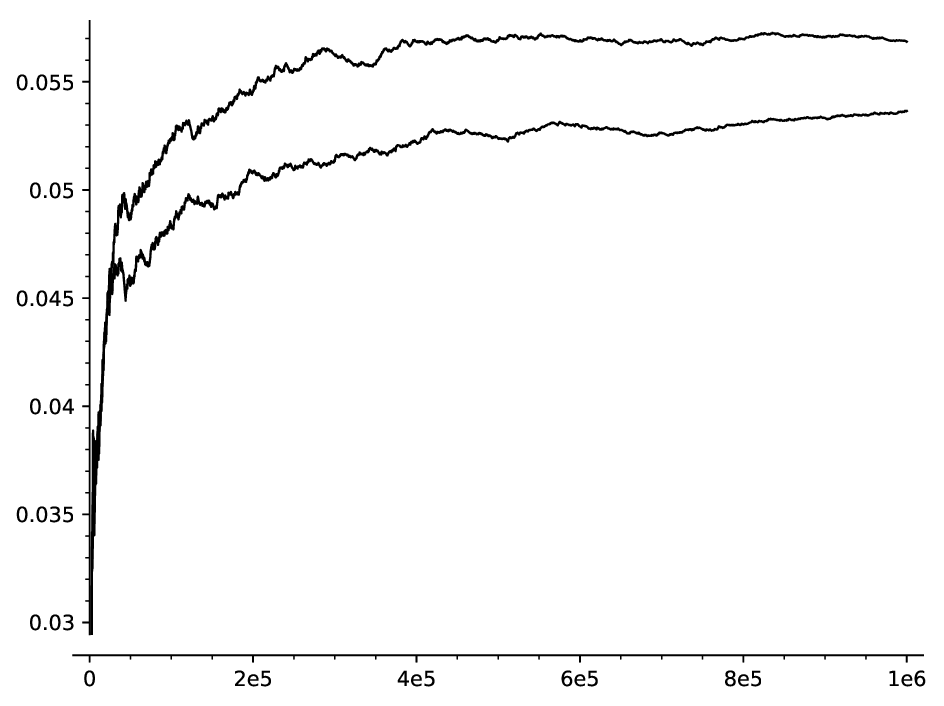}
\caption{$|l| = 4$: Top -4 bottom 4} \label{fig:17_6_2_3_A_4}
\end{subfigure}\hspace*{\fill}
\begin{subfigure}[b]{0.4\linewidth}
\includegraphics[width=\linewidth]{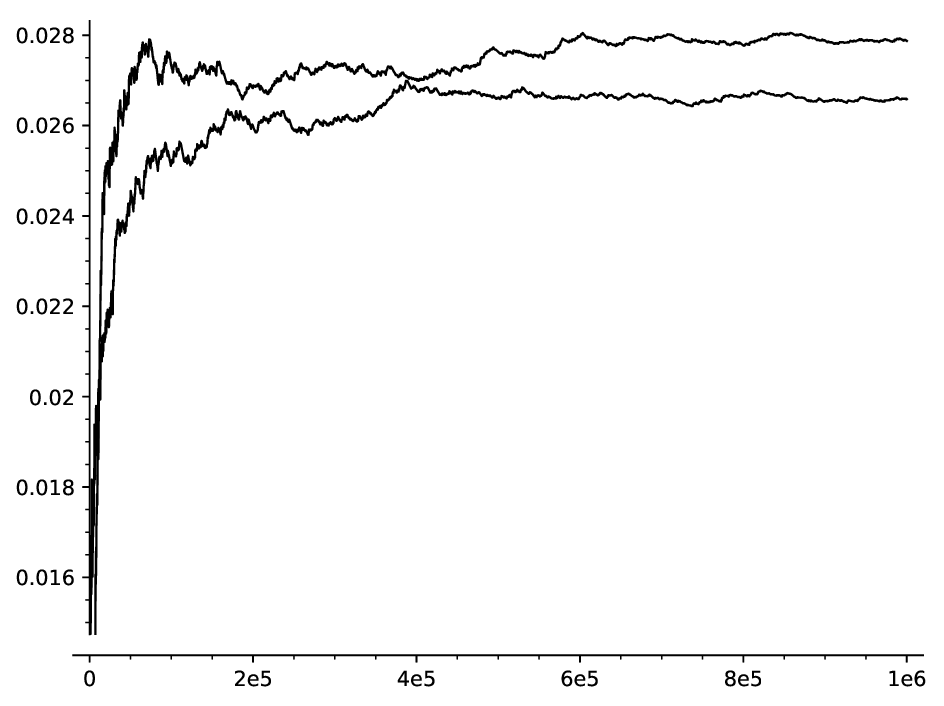}
\caption{$|l| = 5$: Top 5 bottom -5} \label{fig:17_6_2_3_A_5}
\end{subfigure}\hspace*{\fill}
\begin{subfigure}[b]{0.4\linewidth}
\includegraphics[width=\linewidth]{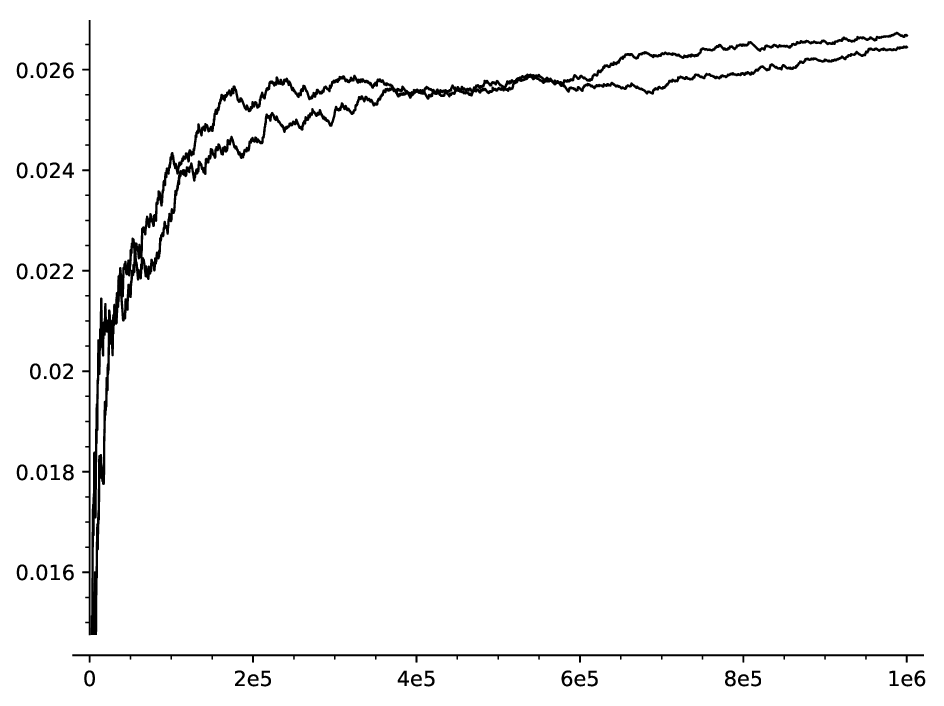}
\caption{$|l| = 6$: Top -6 bottom 6} \label{fig:17_6_2_3_A_6}
\end{subfigure}
\hspace*{-2.3cm}
\begin{subfigure}[b]{0.4\linewidth}
\includegraphics[width=\linewidth]{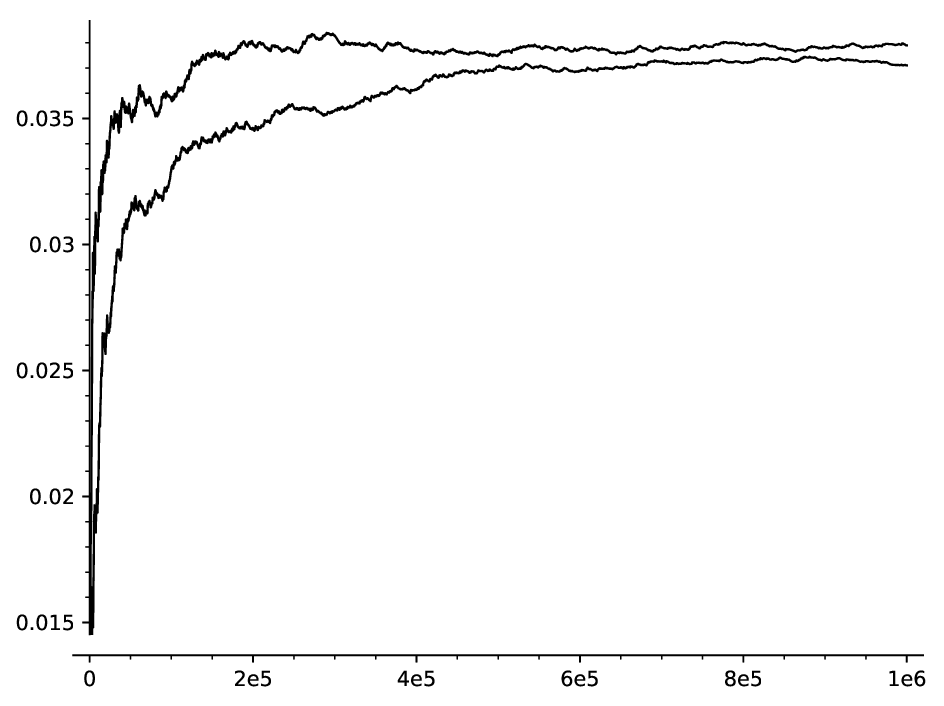}
\caption{$|l| = 7$: Top -7 bottom 7} \label{fig:17_6_2_3_A_7}
\end{subfigure}\hspace*{\fill}
\begin{subfigure}[b]{0.4\linewidth}
\includegraphics[width=\linewidth]{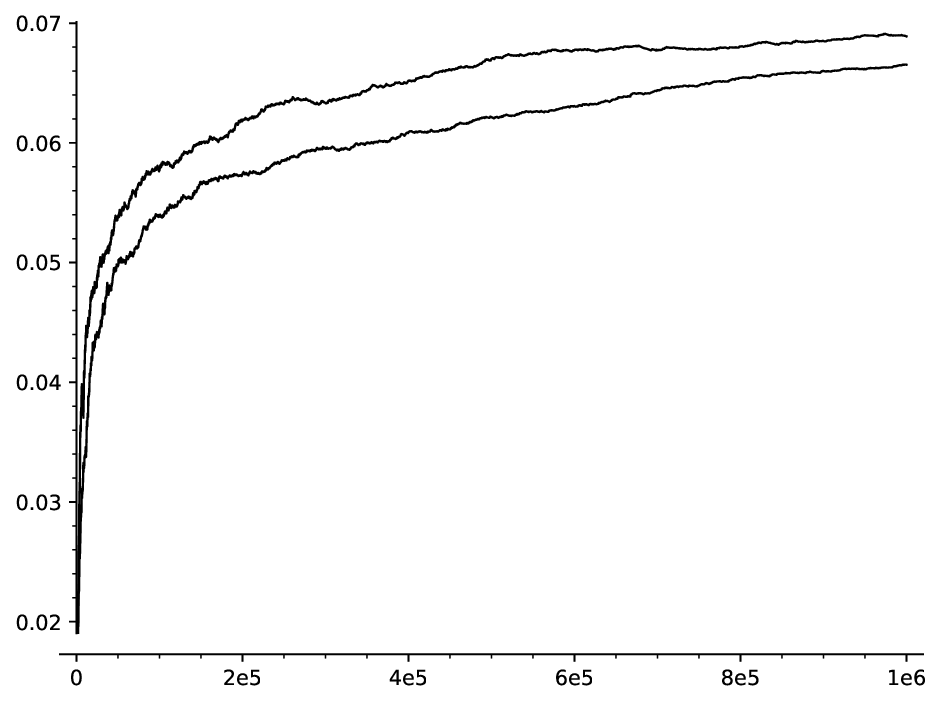}
\caption{$|l| = 8$: Top 8 bottom -8} \label{fig:17_6_2_3_A_8}
\end{subfigure}\hspace*{\fill}
\begin{subfigure}[b]{0.4\linewidth}
\includegraphics[width=\linewidth]{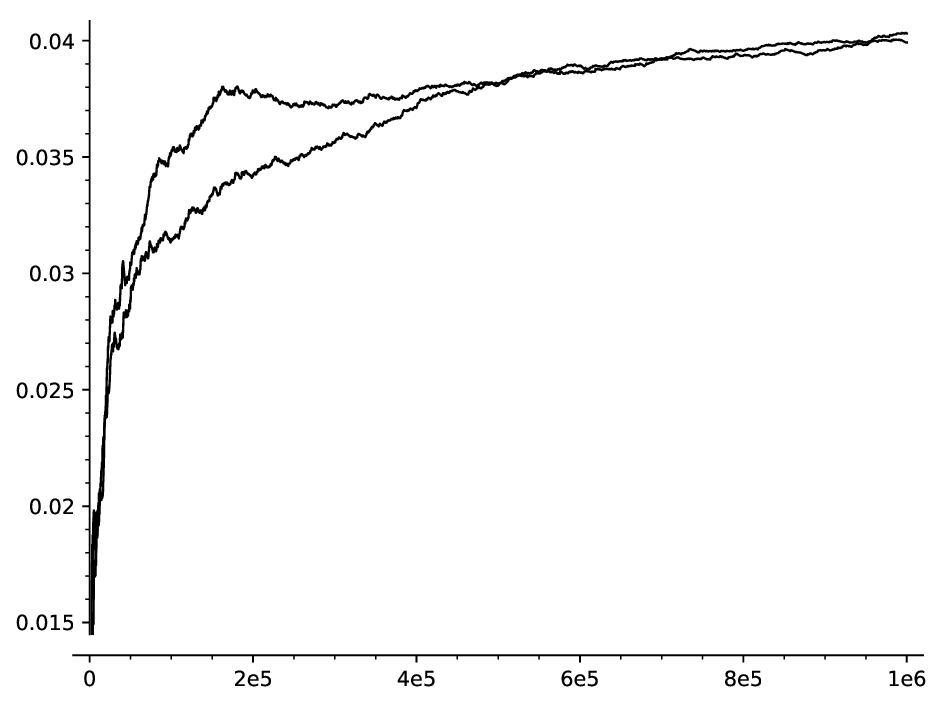}
\caption{$|l| = 9$: Top 9 bottom -9} \label{fig:17_6_2_3_A_9}
\end{subfigure}
\caption{17a1: $(\alpha, \beta) = (2,3)$ Ratio~\eqref{ratio_n_orders} $x_{6,E}^{(\alpha, \beta)}(X;l)/X^{1/2}\log^2(X)$} \label{fig:17a1_6_2_3_A_exact}
\end{figure}

\clearpage

\begin{figure}[t] 
\hspace*{-2.3cm}
\begin{subfigure}[b]{0.4\linewidth}
\includegraphics[width=\linewidth]{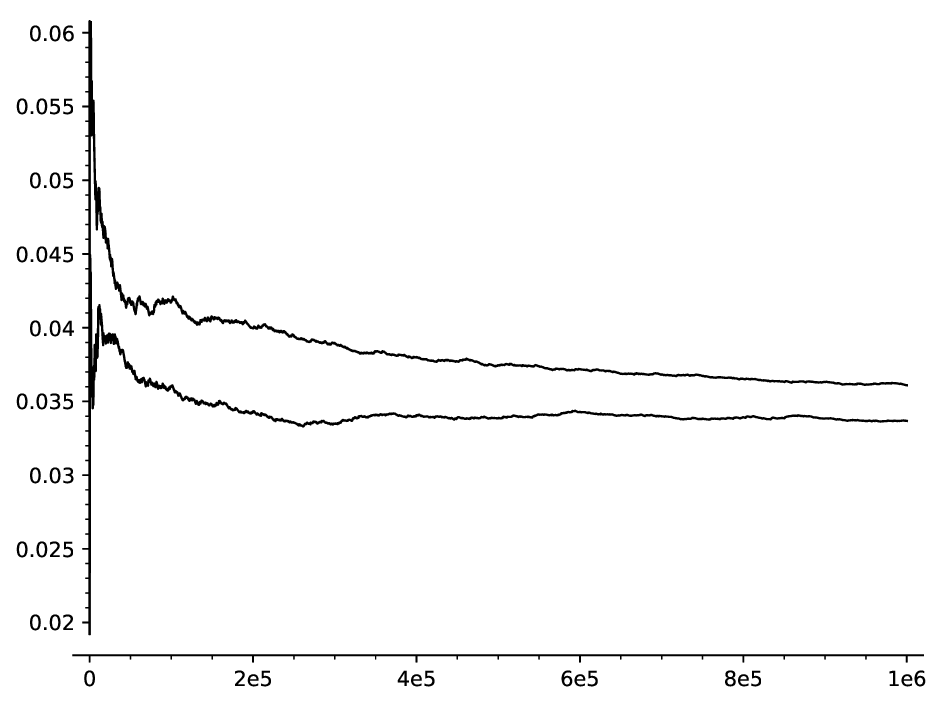}
\caption{$|l| = 1$: Top 1 bottom -1} \label{fig:17_6_1_6_A_1}
\end{subfigure}\hspace*{\fill}
\begin{subfigure}[b]{0.4\linewidth}
\includegraphics[width=\linewidth]{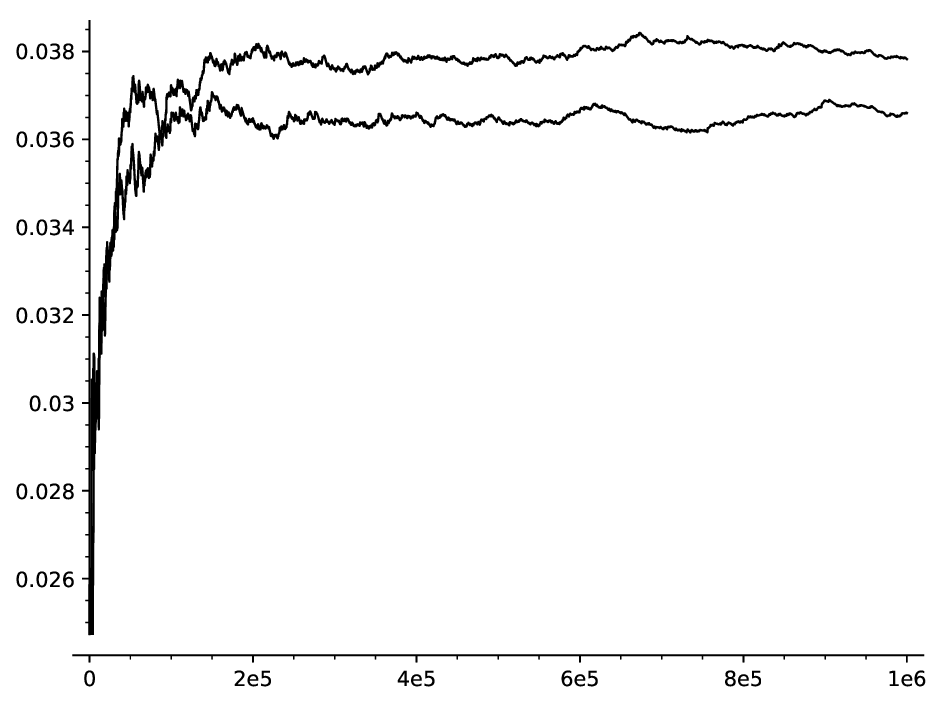}
\caption{$|l| = 2$: Top 2 bottom -2} \label{fig:17_6_1_6_A_2}
\end{subfigure}\hspace*{\fill}
\begin{subfigure}[b]{0.4\linewidth}
\includegraphics[width=\linewidth]{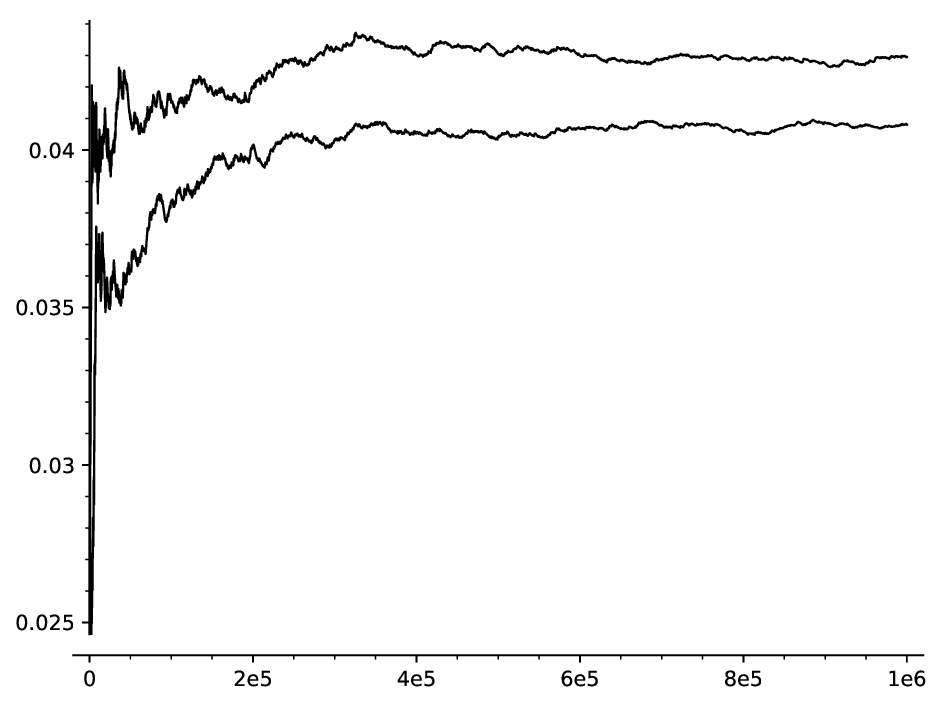}
\caption{$|l| = 3$: Top 3 bottom -3} \label{fig:17_6_1_6_A_3}
\end{subfigure}
\hspace*{-2.3cm}
\begin{subfigure}[b]{0.4\linewidth}
\includegraphics[width=\linewidth]{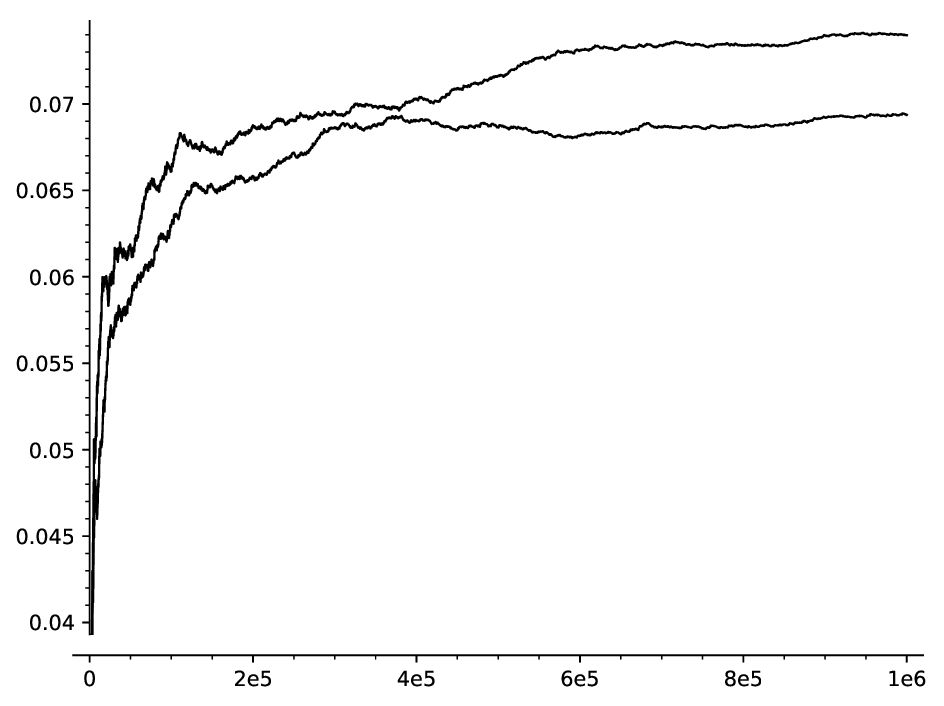}
\caption{$|l| = 4$: Top 4 bottom -4} \label{fig:17_6_1_6_A_4}
\end{subfigure}\hspace*{\fill}
\begin{subfigure}[b]{0.4\linewidth}
\includegraphics[width=\linewidth]{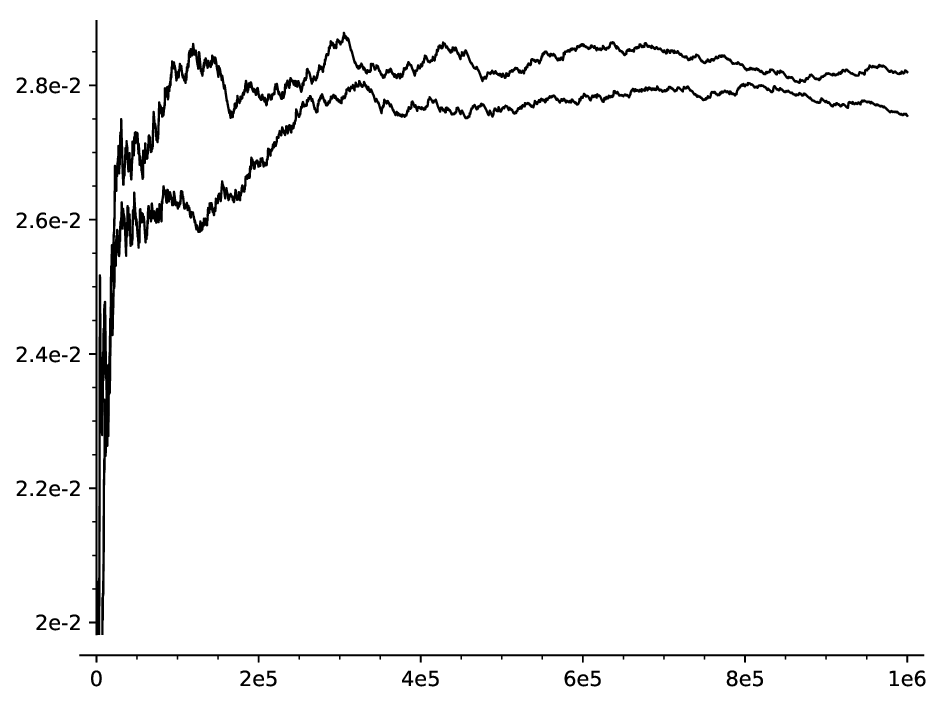}
\caption{$|l| = 5$: Top -5 bottom 5} \label{fig:17_6_1_6_A_5}
\end{subfigure}\hspace*{\fill}
\begin{subfigure}[b]{0.4\linewidth}
\includegraphics[width=\linewidth]{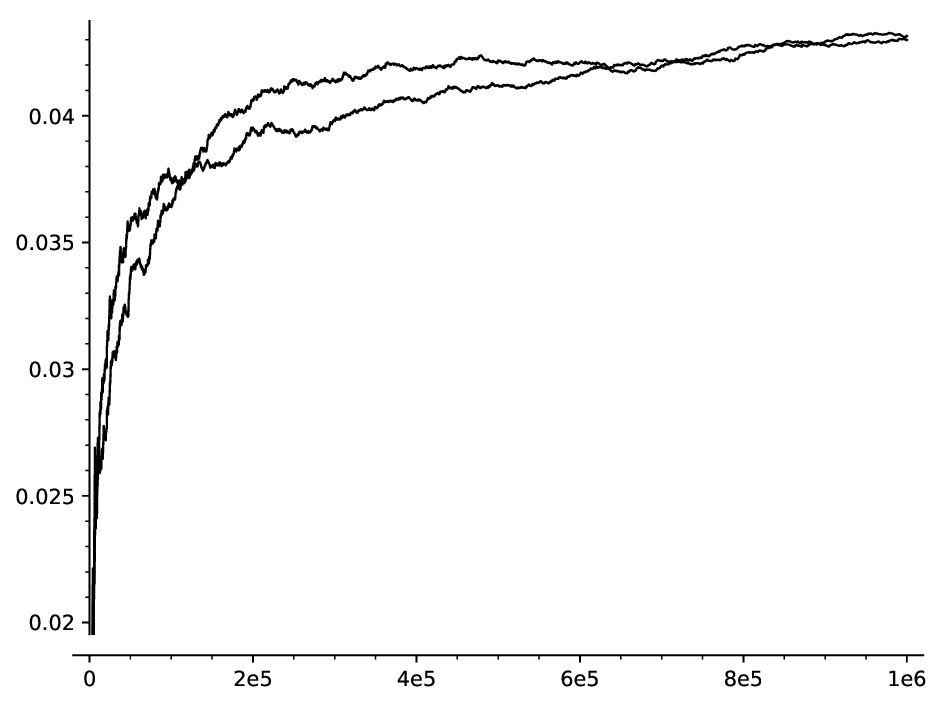}
\caption{$|l| = 6$: Top -6 bottom 6} \label{fig:17_6_1_6_A_6}
\end{subfigure}
\hspace*{-2.3cm}
\begin{subfigure}[b]{0.4\linewidth}
\includegraphics[width=\linewidth]{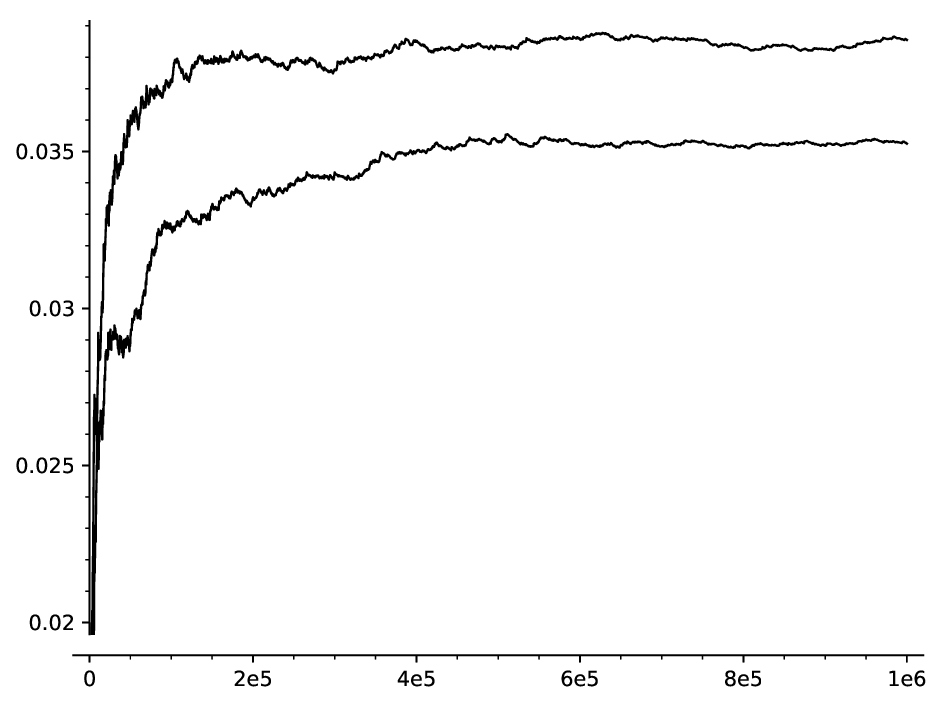}
\caption{$|l| = 7$: Top 7 bottom -7} \label{fig:17_6_1_6_A_7}
\end{subfigure}\hspace*{\fill}
\begin{subfigure}[b]{0.4\linewidth}
\includegraphics[width=\linewidth]{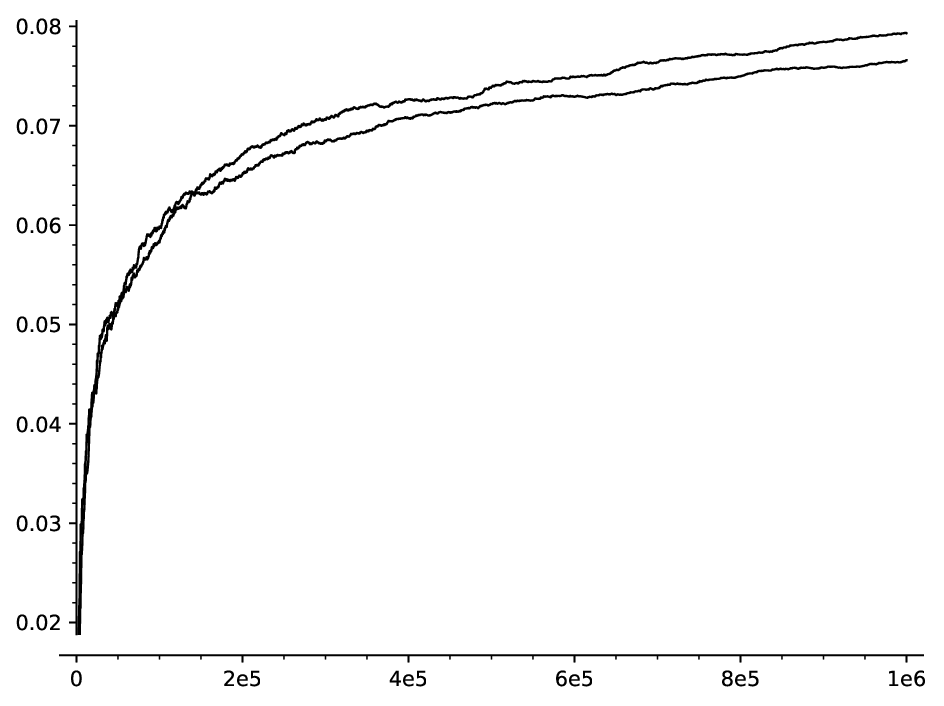}
\caption{$|l| = 8$: Top -8 bottom 8} \label{fig:17_6_1_6_A_8}
\end{subfigure}\hspace*{\fill}
\begin{subfigure}[b]{0.4\linewidth}
\includegraphics[width=\linewidth]{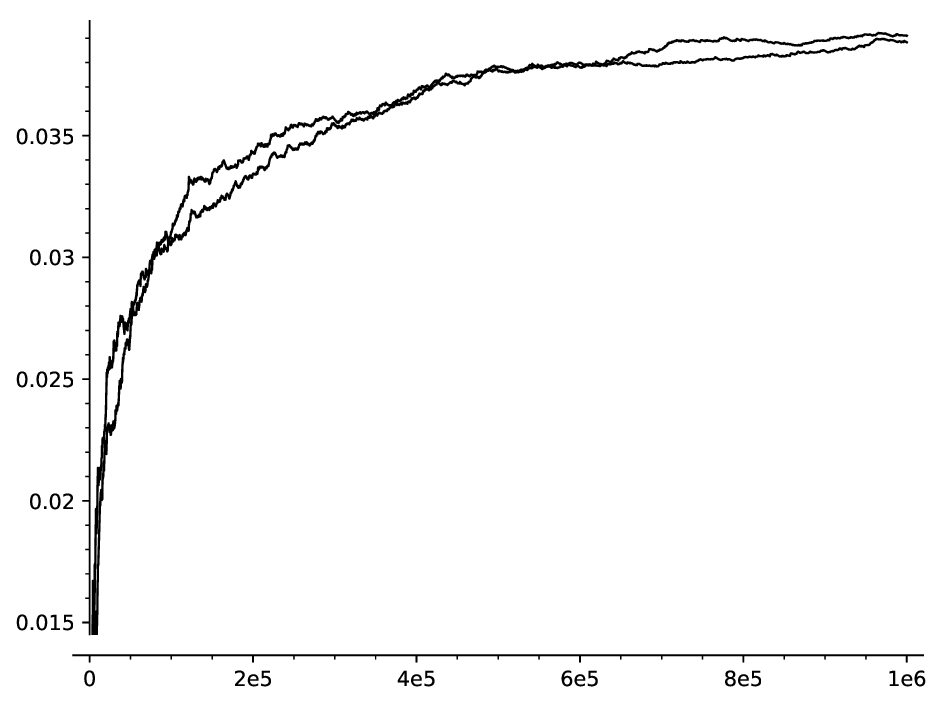}
\caption{$|l| = 9$: Top -9 bottom 9} \label{fig:17_6_1_6_A_9}
\end{subfigure}
\caption{17a1: $(\alpha, \beta) = (1,6)$ Ratio~\eqref{ratio_n_orders} $x_{6,E}^{(\alpha, \beta)}(X;l)/X^{1/2}\log^2(X)$} \label{fig:17a1_6_1_6_A_exact}
\end{figure}

\clearpage

\begin{figure}[t] 
\hspace*{-2.3cm}
\begin{subfigure}[b]{0.4\linewidth}
\includegraphics[width=\linewidth]{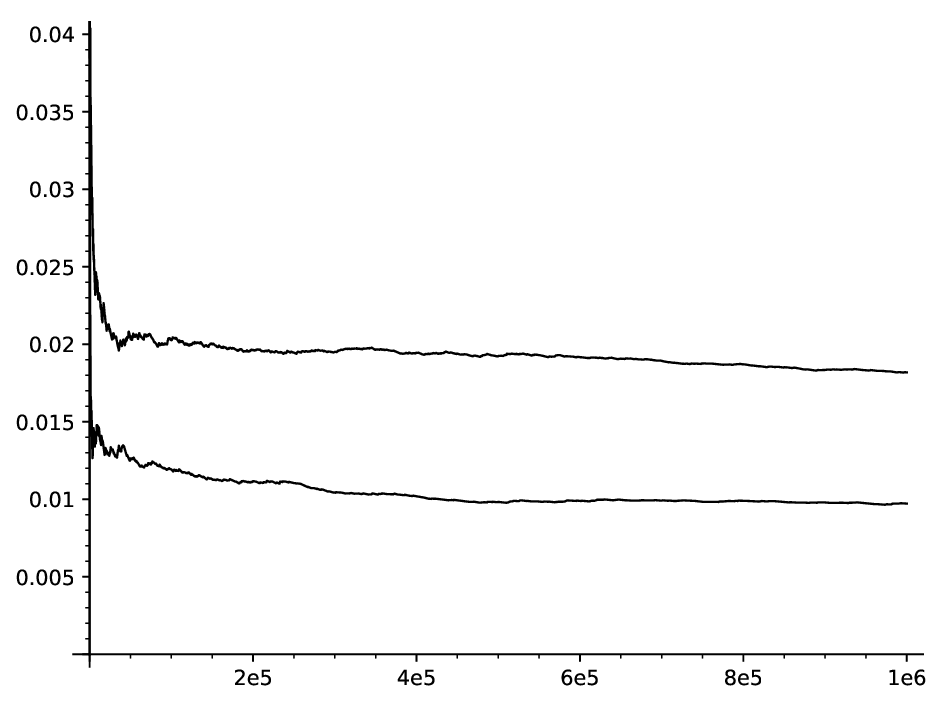}
\caption{$|l| = 1$: Top 1 bottom -1} \label{fig:17_6_2_6_A_1}
\end{subfigure}\hspace*{\fill}
\begin{subfigure}[b]{0.4\linewidth}
\includegraphics[width=\linewidth]{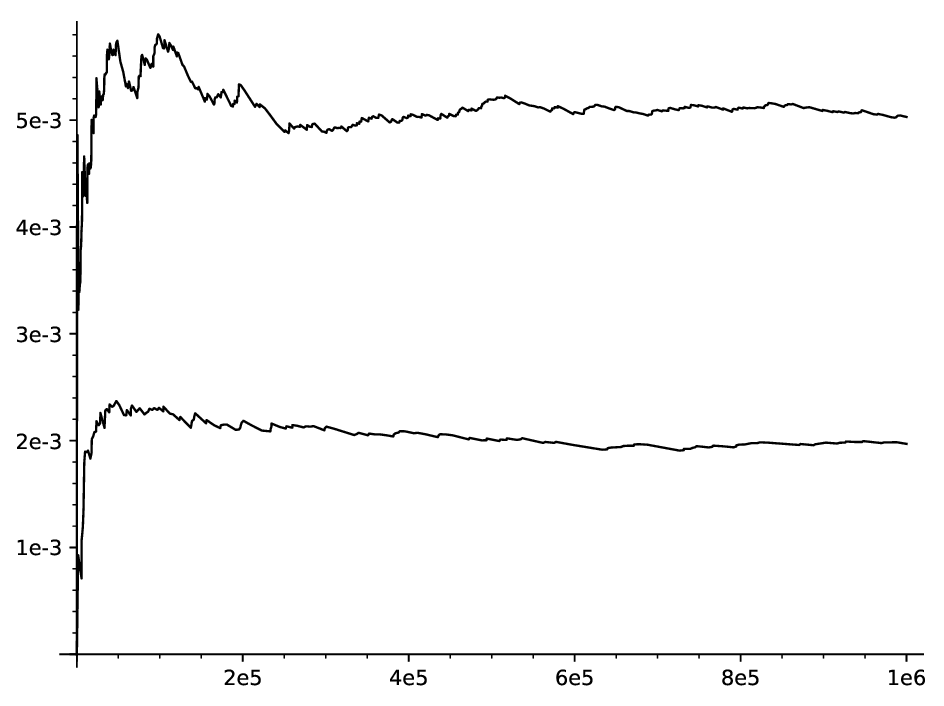}
\caption{$|l| = 2$: Top 2 bottom -2} \label{fig:17_6_2_6_A_2}
\end{subfigure}\hspace*{\fill}
\begin{subfigure}[b]{0.4\linewidth}
\includegraphics[width=\linewidth]{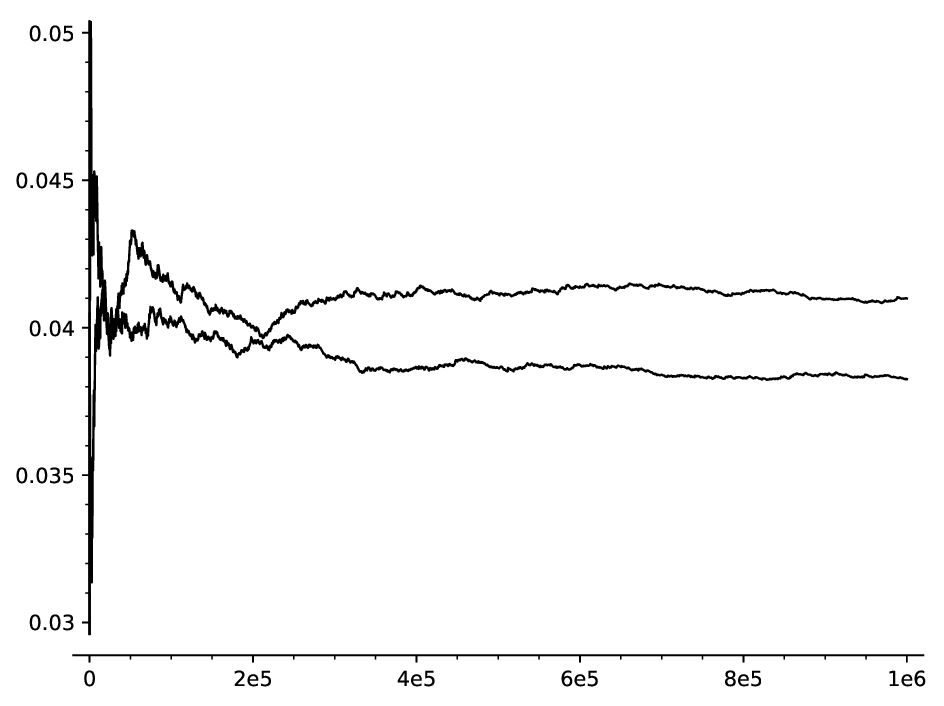}
\caption{$|l| = 3$: Top 3 bottom -3} \label{fig:17_6_2_6_A_3}
\end{subfigure}
\hspace*{-2.3cm}
\begin{subfigure}[b]{0.4\linewidth}
\includegraphics[width=\linewidth]{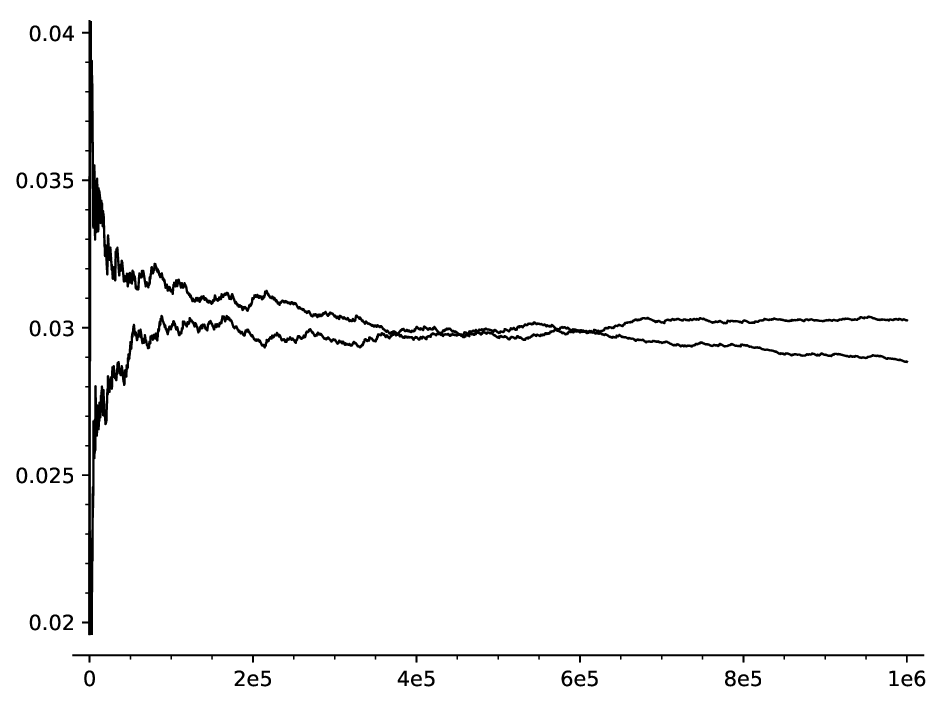}
\caption{$|l| = 4$: Top 4 bottom -4} \label{fig:17_6_2_6_A_4}
\end{subfigure}\hspace*{\fill}
\begin{subfigure}[b]{0.4\linewidth}
\includegraphics[width=\linewidth]{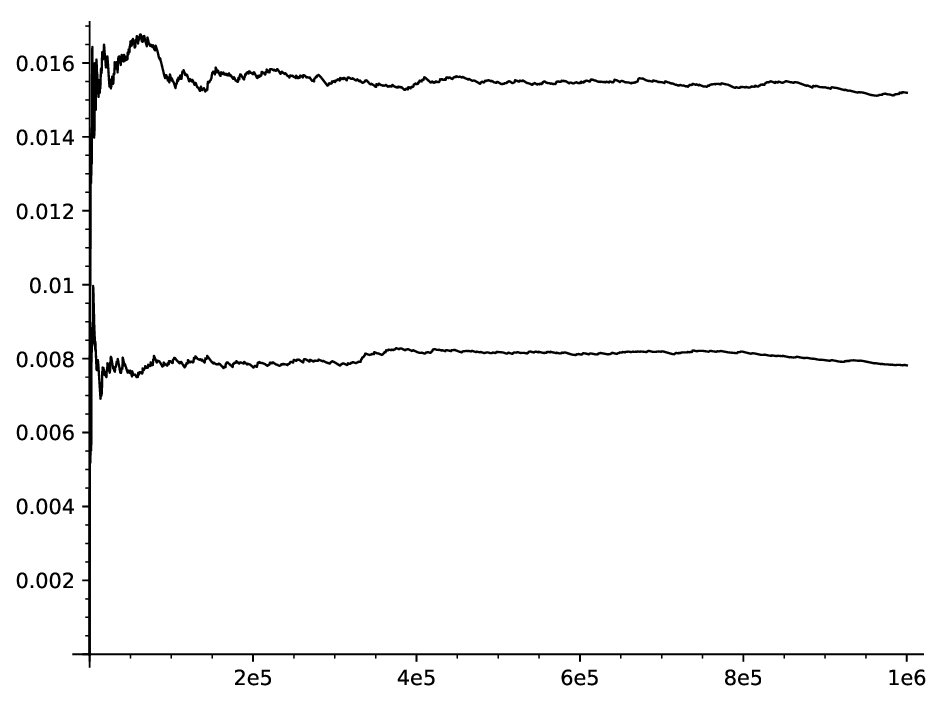}
\caption{$|l| = 5$: Top -5 bottom 5} \label{fig:17_6_2_6_A_5}
\end{subfigure}\hspace*{\fill}
\begin{subfigure}[b]{0.4\linewidth}
\includegraphics[width=\linewidth]{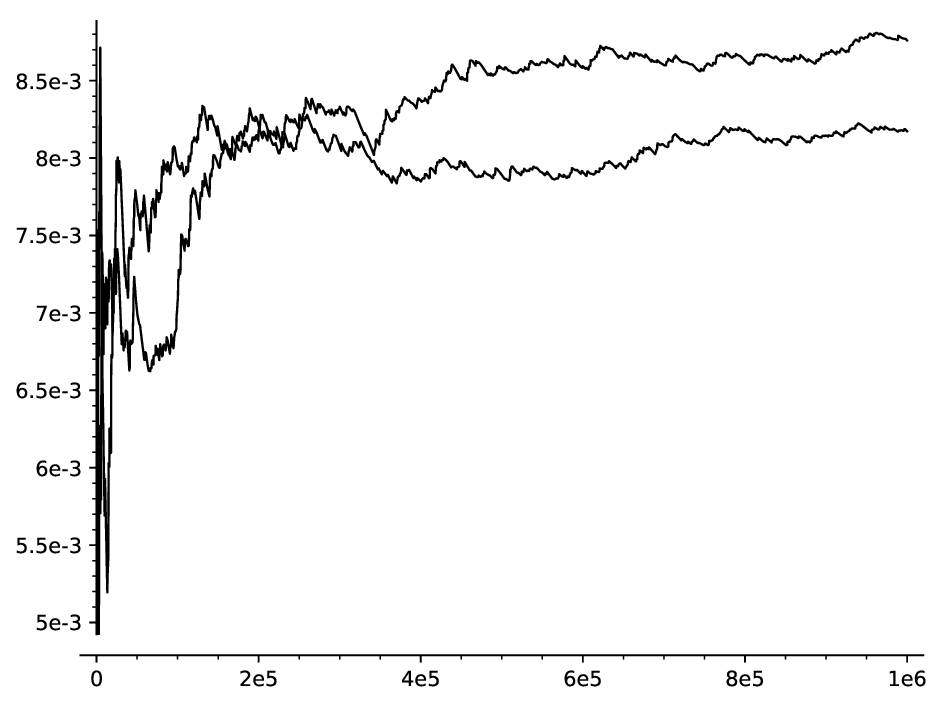}
\caption{$|l| = 6$: Top 6 bottom -6} \label{fig:17_6_2_6_A_6}
\end{subfigure}
\hspace*{-2.3cm}
\begin{subfigure}[b]{0.4\linewidth}
\includegraphics[width=\linewidth]{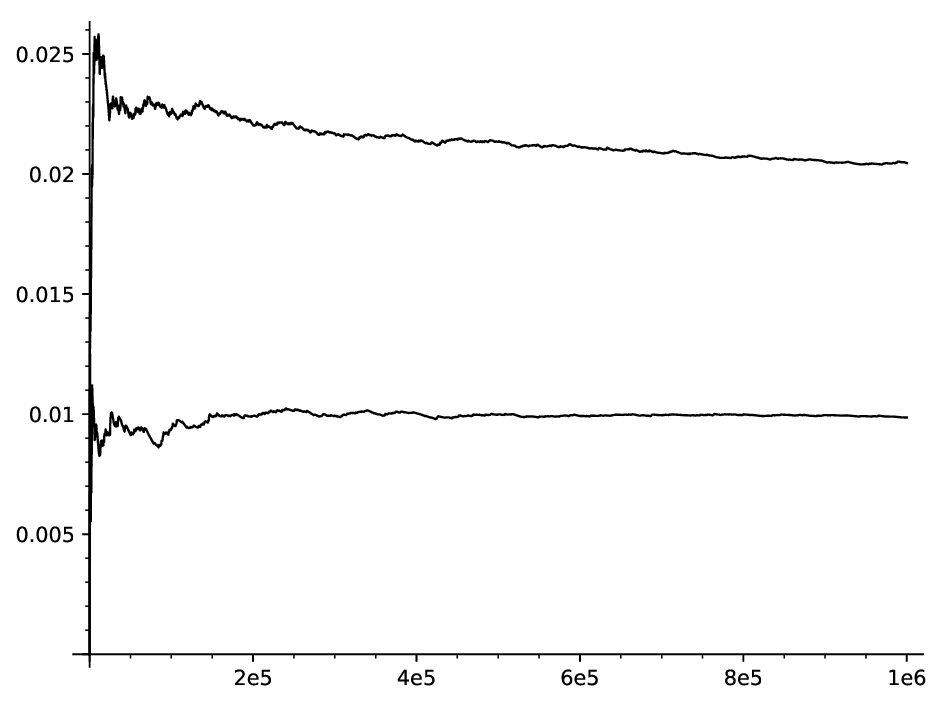}
\caption{$|l| = 7$: Top 7 bottom -7} \label{fig:17_6_2_6_A_7}
\end{subfigure}\hspace*{\fill}
\begin{subfigure}[b]{0.4\linewidth}
\includegraphics[width=\linewidth]{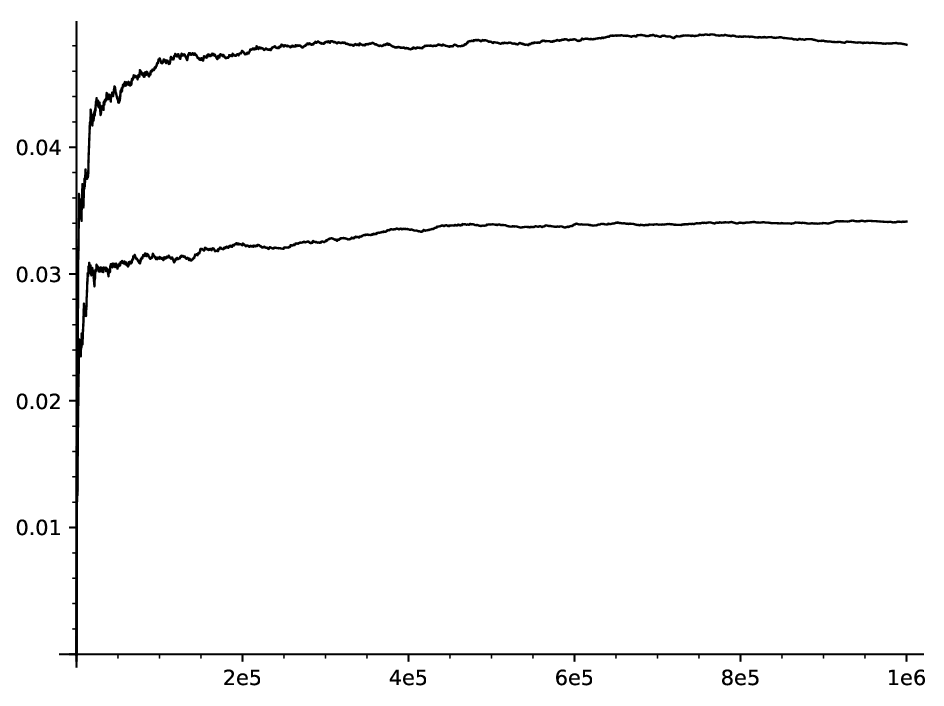}
\caption{$|l| = 8$: Top 8 bottom -8} \label{fig:17_6_2_6_A_8}
\end{subfigure}\hspace*{\fill}
\begin{subfigure}[b]{0.4\linewidth}
\includegraphics[width=\linewidth]{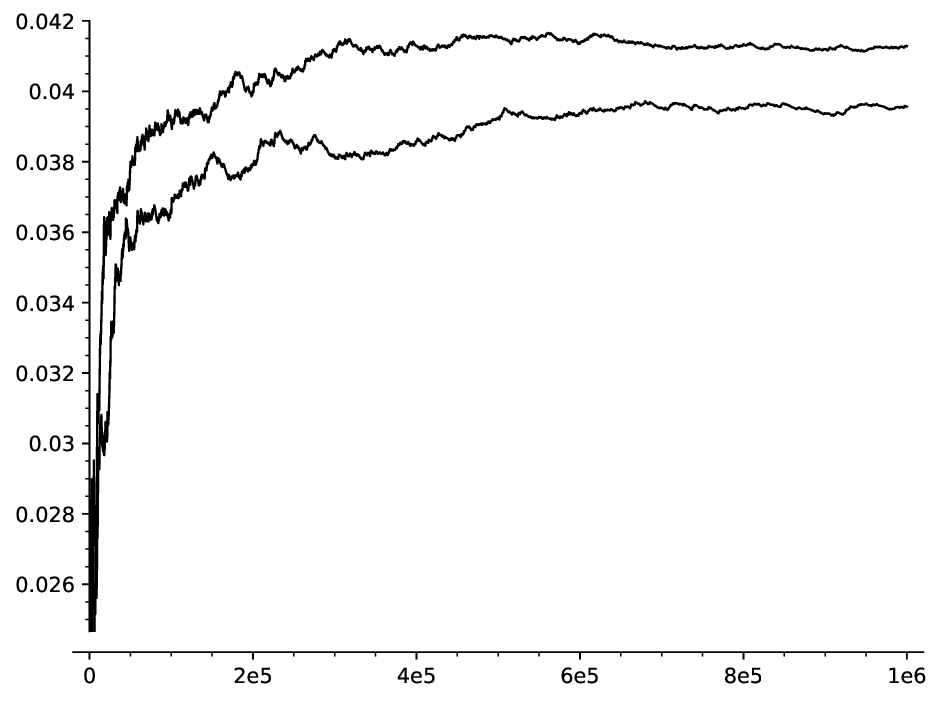}
\caption{$|l| = 9$: Top 9 bottom -9} \label{fig:17_6_2_6_A_9}
\end{subfigure}
\caption{17a1: $(\alpha, \beta) = (2,6)$ Ratio~\eqref{ratio_n_orders} $x_{6,E}^{(\alpha, \beta)}(X;l)/X^{1/2}\log^2(X)$} \label{fig:17a1_6_2_6_A_exact}
\end{figure}

\clearpage

\begin{figure}[t] 
\hspace*{-2.3cm}
\begin{subfigure}[b]{0.4\linewidth}
\includegraphics[width=\linewidth]{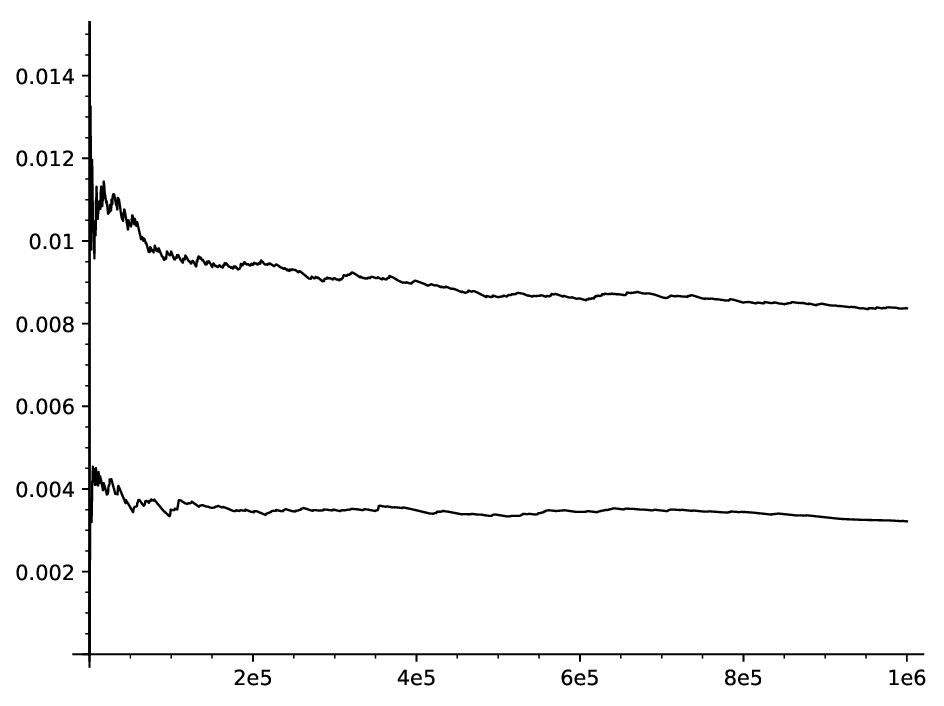}
\caption{$|l| = 1$: Top -1 bottom 1} \label{fig:19_6_1_3_A_1}
\end{subfigure}\hspace*{\fill}
\begin{subfigure}[b]{0.4\linewidth}
\includegraphics[width=\linewidth]{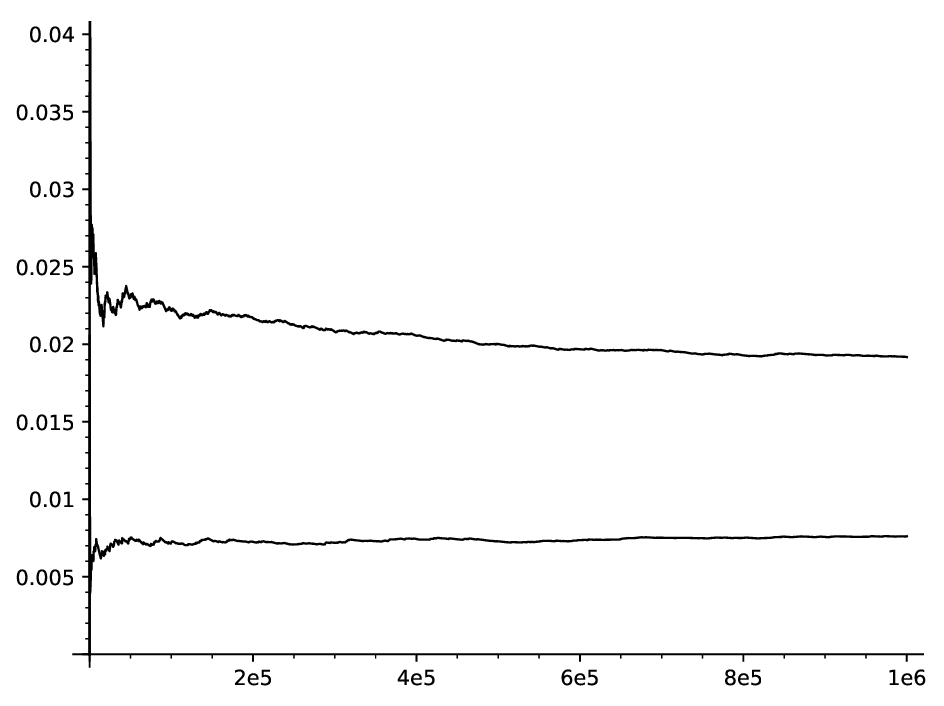}
\caption{$|l| = 2$: Top 2 bottom -2} \label{fig:19_6_1_3_A_2}
\end{subfigure}\hspace*{\fill}
\begin{subfigure}[b]{0.4\linewidth}
\includegraphics[width=\linewidth]{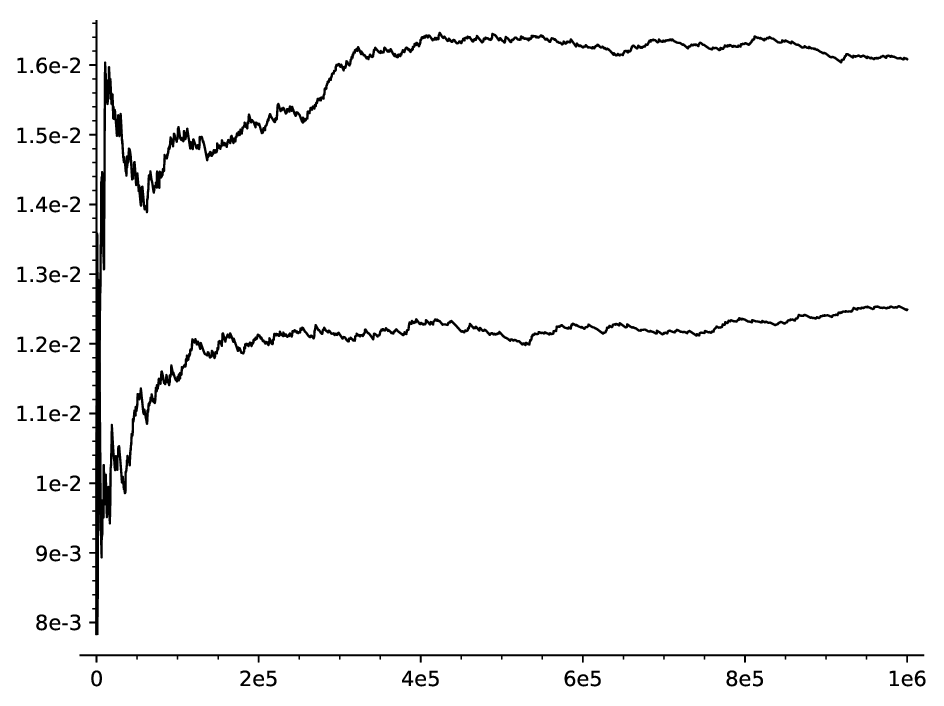}
\caption{$|l| = 3$: Top 3 bottom -3} \label{fig:19_6_1_3_A_3}
\end{subfigure}
\hspace*{-2.3cm}
\begin{subfigure}[b]{0.4\linewidth}
\includegraphics[width=\linewidth]{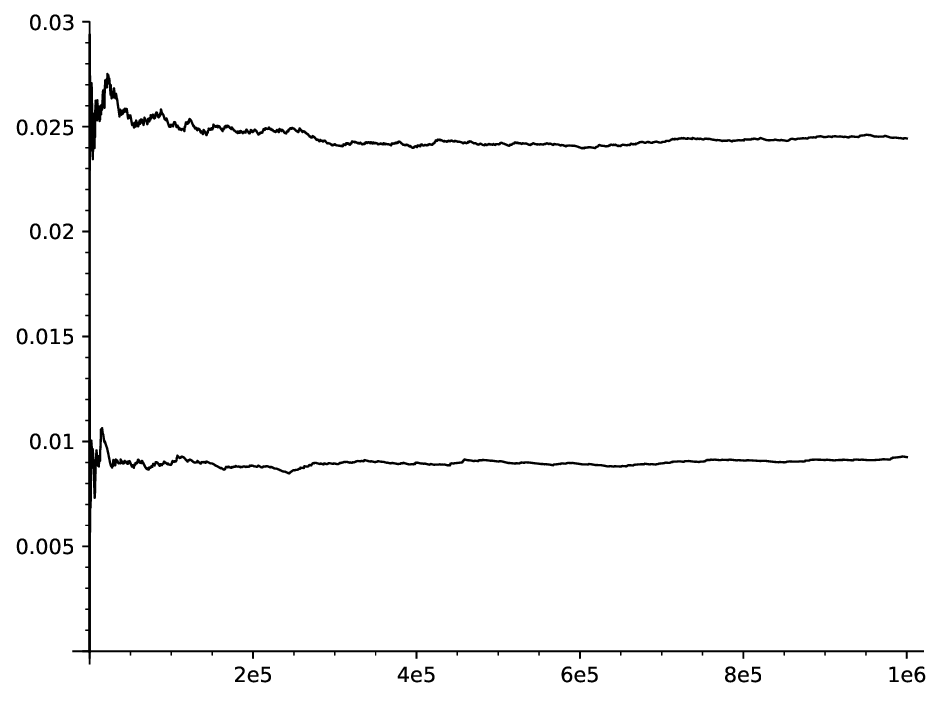}
\caption{$|l| = 4$: Top -4 bottom 4} \label{fig:19_6_1_3_A_4}
\end{subfigure}\hspace*{\fill}
\begin{subfigure}[b]{0.4\linewidth}
\includegraphics[width=\linewidth]{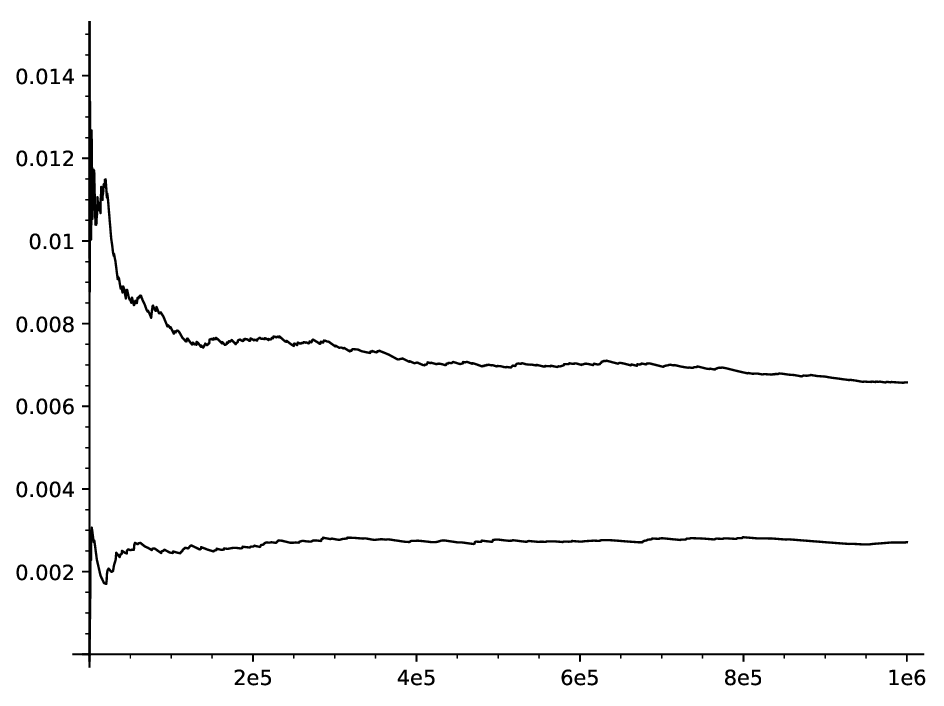}
\caption{$|l| = 5$: Top 5 bottom -5} \label{fig:19_6_1_3_A_5}
\end{subfigure}\hspace*{\fill}
\begin{subfigure}[b]{0.4\linewidth}
\includegraphics[width=\linewidth]{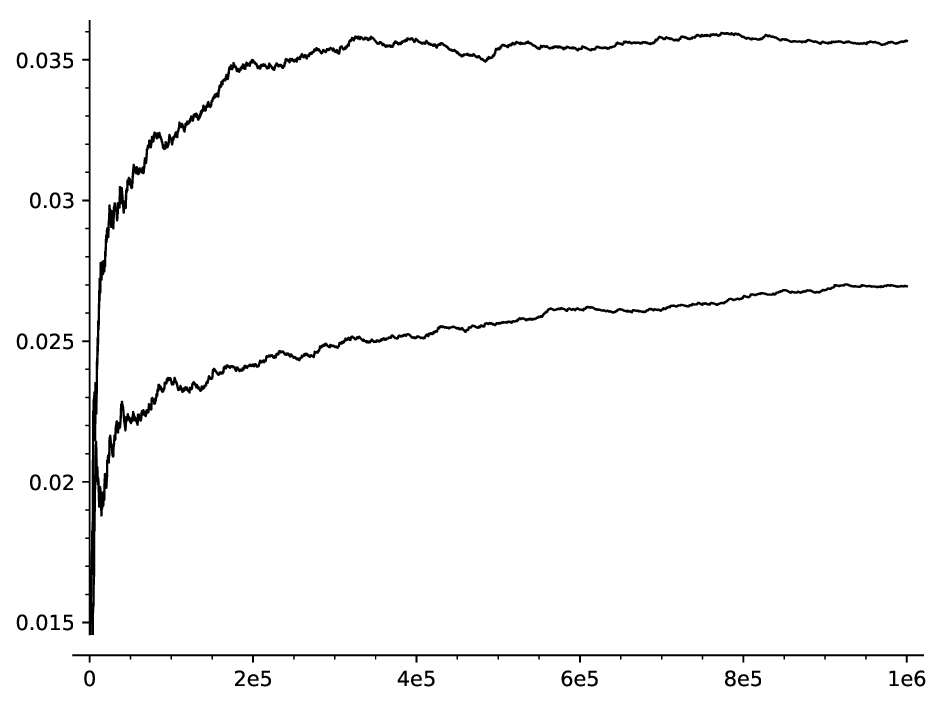}
\caption{$|l| = 6$: Top -6 bottom 6} \label{fig:19_6_1_3_A_6}
\end{subfigure}
\hspace*{-2.3cm}
\begin{subfigure}[b]{0.4\linewidth}
\includegraphics[width=\linewidth]{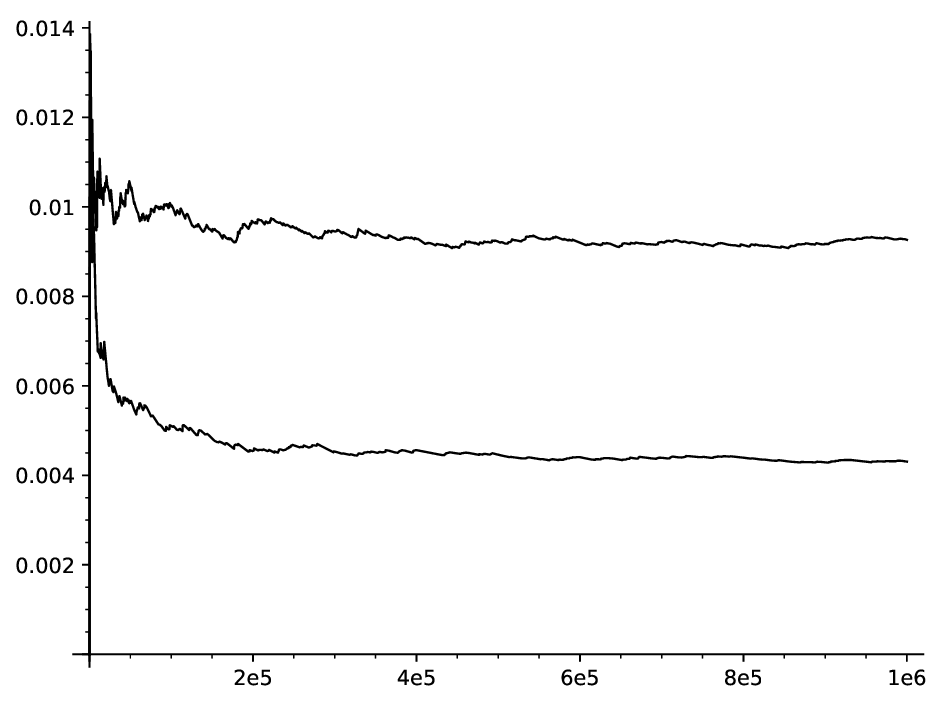}
\caption{$|l| = 7$: Top -7 bottom 7} \label{fig:19_6_1_3_A_7}
\end{subfigure}\hspace*{\fill}
\begin{subfigure}[b]{0.4\linewidth}
\includegraphics[width=\linewidth]{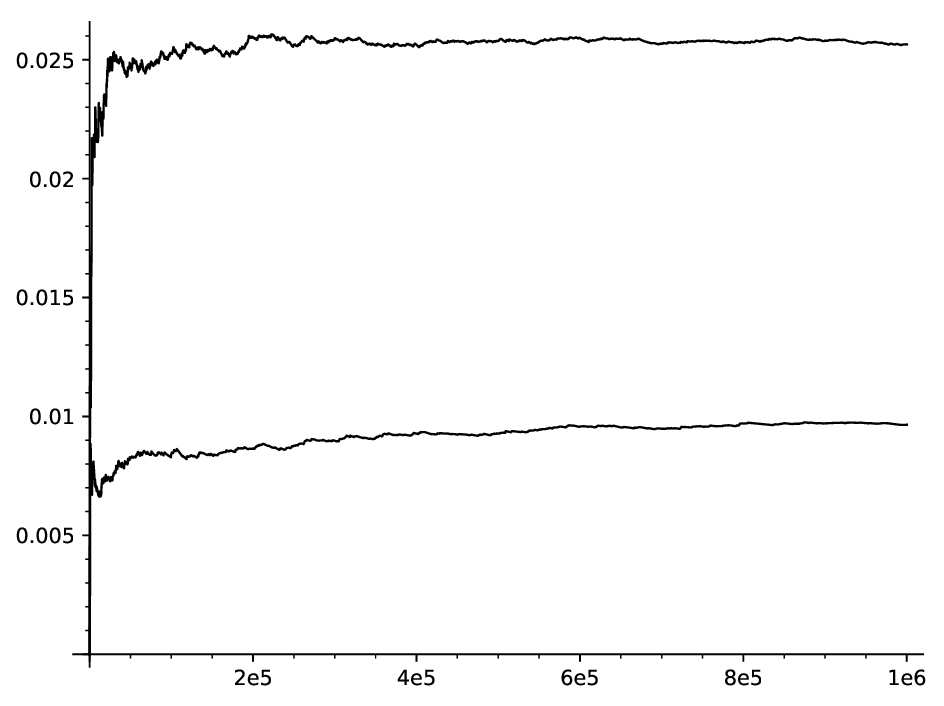}
\caption{$|l| = 8$: Top 8 bottom -8} \label{fig:19_6_1_3_A_8}
\end{subfigure}\hspace*{\fill}
\begin{subfigure}[b]{0.4\linewidth}
\includegraphics[width=\linewidth]{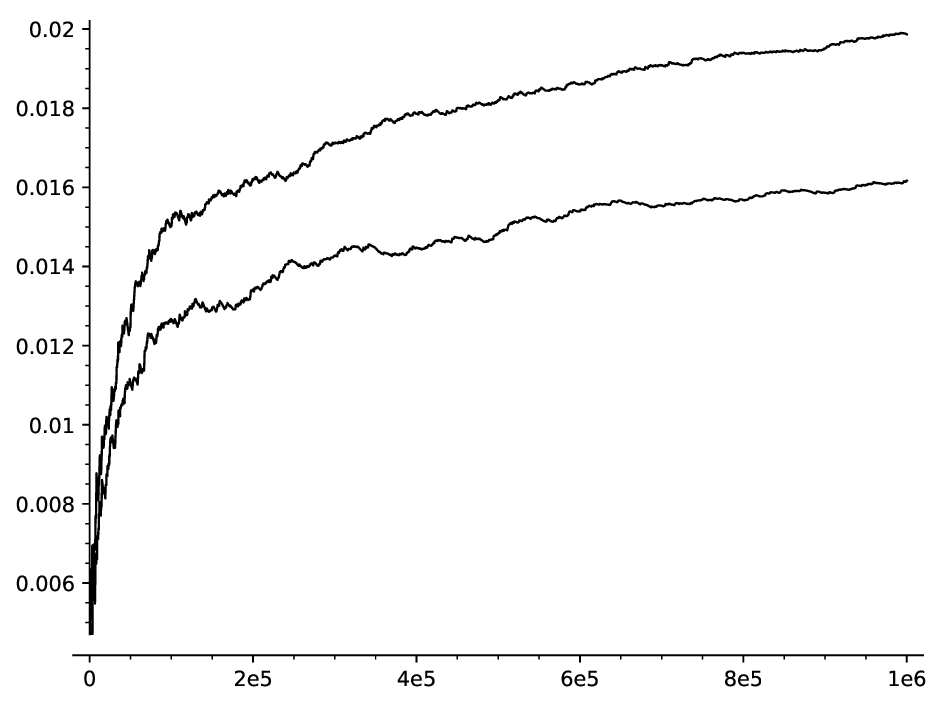}
\caption{$|l| = 9$: Top 9 bottom -9} \label{fig:19_6_1_3_A_9}
\end{subfigure}
\caption{19a1: $(\alpha, \beta) = (1,3)$ Ratio~\eqref{ratio_n_orders} $x_{6,E}^{(\alpha, \beta)}(X;l)/X^{1/2}\log^2(X)$} \label{fig:19a1_6_1_3_A_exact}
\end{figure}

\clearpage

\begin{figure}[t] 
\hspace*{-2.3cm}
\begin{subfigure}[b]{0.4\linewidth}
\includegraphics[width=\linewidth]{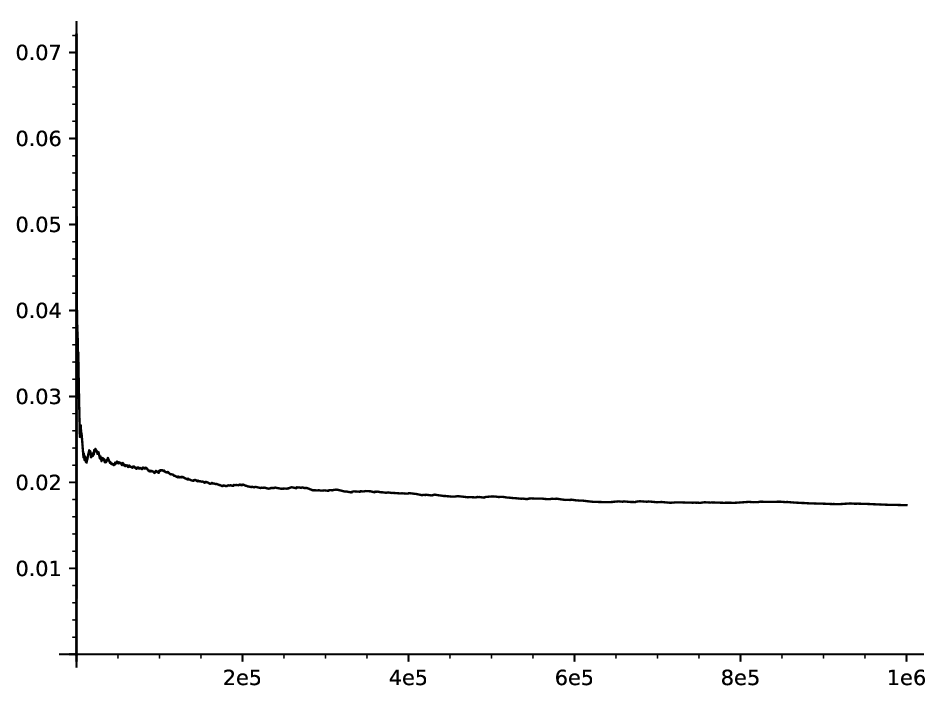}
\caption{$|l| = 1$: No 1 exists except for $\f = 9$} \label{fig:19_6_2_3_A_1}
\end{subfigure}\hspace*{\fill}
\begin{subfigure}[b]{0.4\linewidth}
\includegraphics[width=\linewidth]{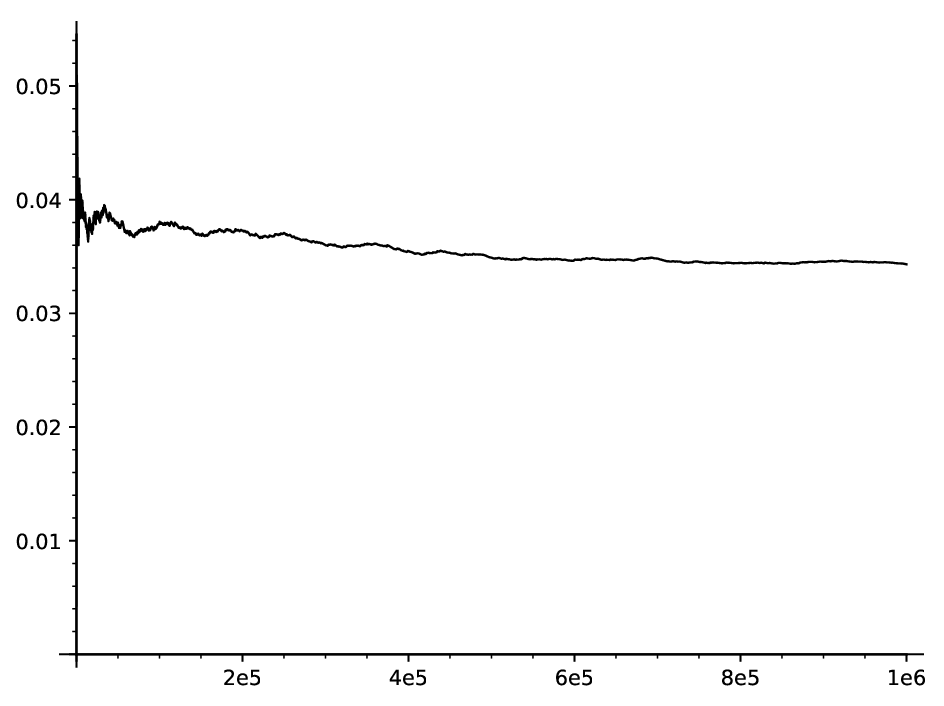}
\caption{$|l| = 2$: No -2 exists} \label{fig:19_6_2_3_A_2}
\end{subfigure}\hspace*{\fill}
\begin{subfigure}[b]{0.4\linewidth}
\includegraphics[width=\linewidth]{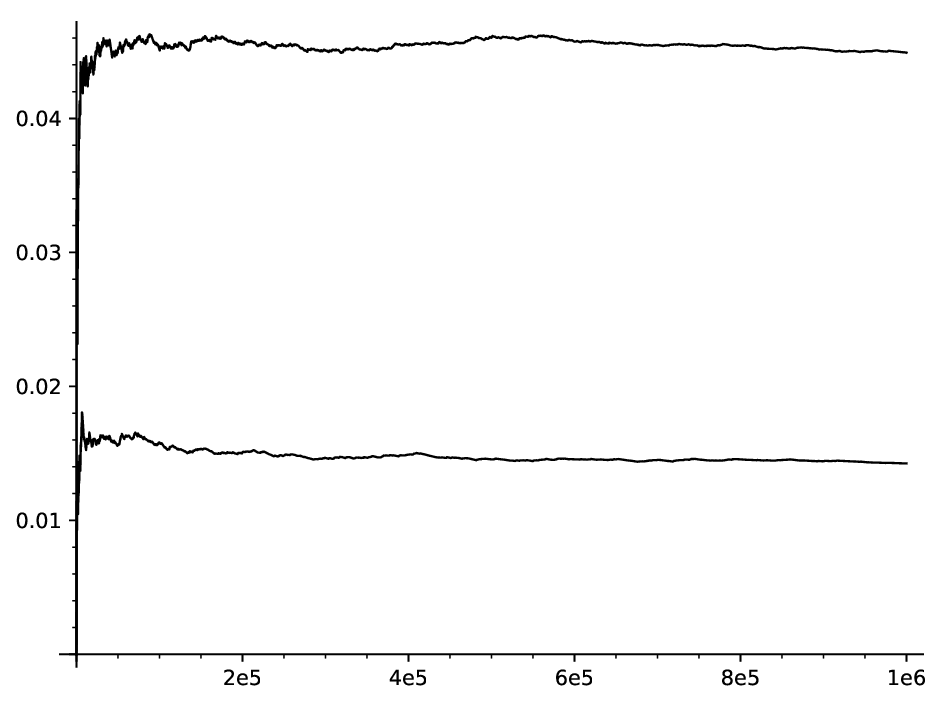}
\caption{$|l| = 3$: Top -3 bottom 3} \label{fig:19_6_2_3_A_3}
\end{subfigure}
\hspace*{-2.3cm}
\begin{subfigure}[b]{0.4\linewidth}
\includegraphics[width=\linewidth]{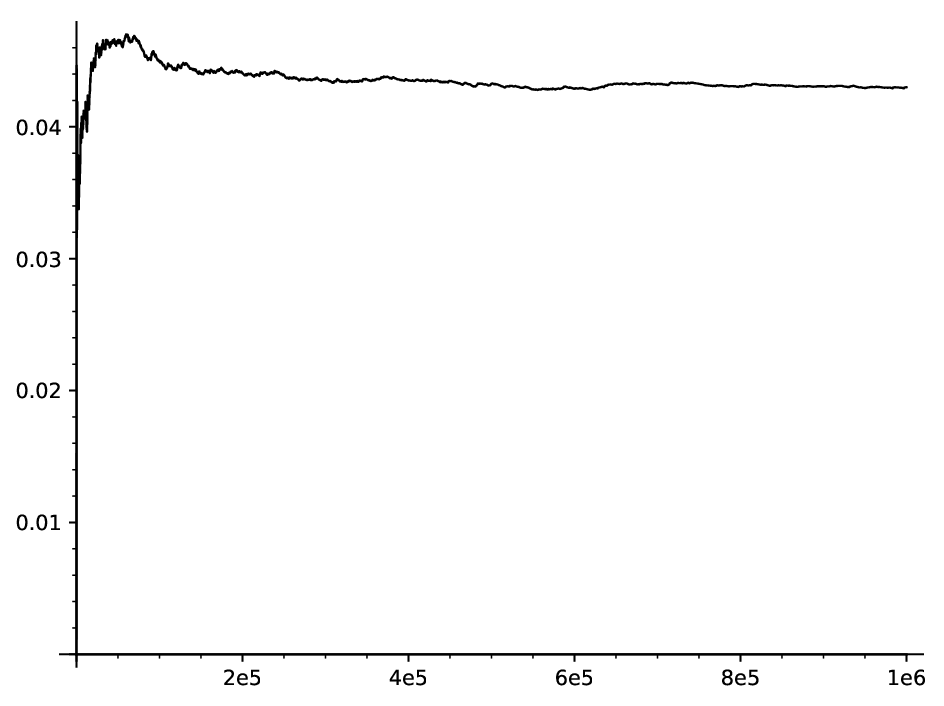}
\caption{$|l| = 4$: No 4 exists} \label{fig:19_6_2_3_A_4}
\end{subfigure}\hspace*{\fill}
\begin{subfigure}[b]{0.4\linewidth}
\includegraphics[width=\linewidth]{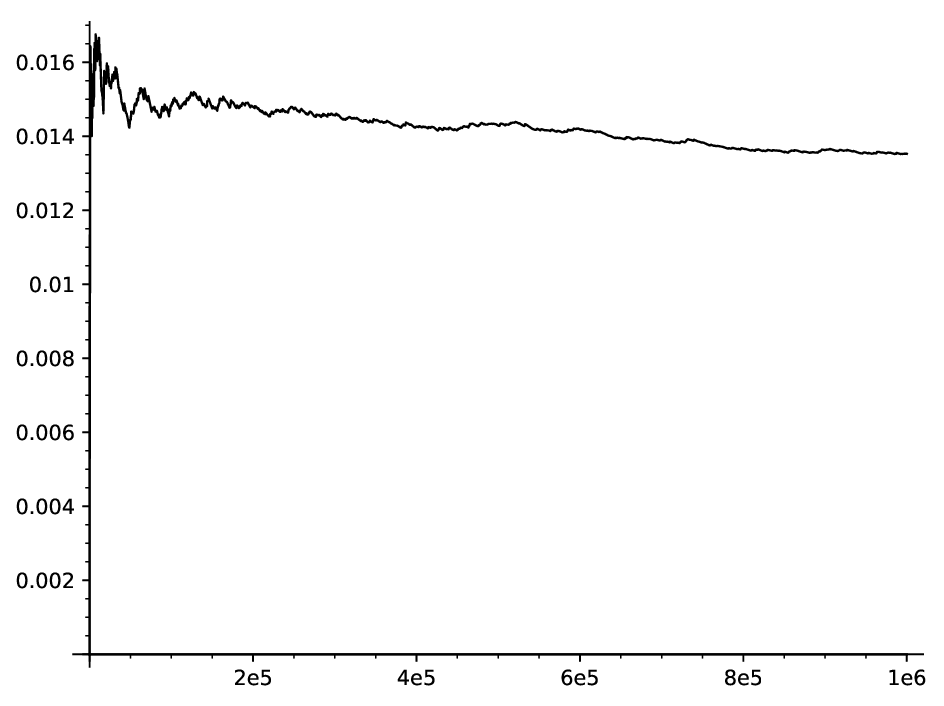}
\caption{$|l| = 5$: No -5 exists} \label{fig:19_6_2_3_A_5}
\end{subfigure}\hspace*{\fill}
\begin{subfigure}[b]{0.4\linewidth}
\includegraphics[width=\linewidth]{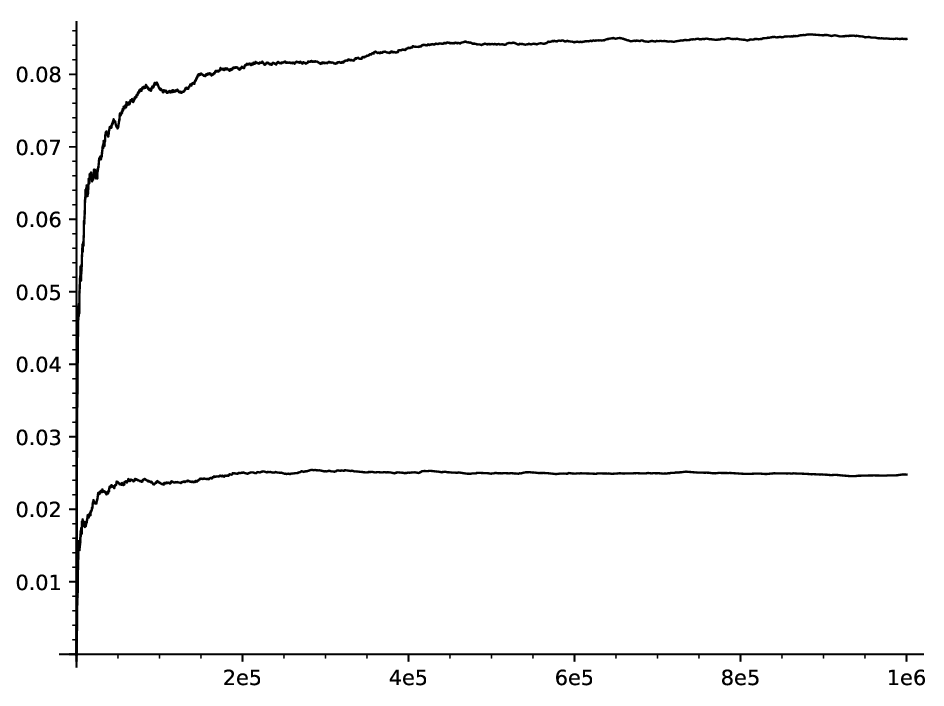}
\caption{$|l| = 6$: Top 6 bottom -6} \label{fig:19_6_2_3_A_6}
\end{subfigure}
\hspace*{-2.3cm}
\begin{subfigure}[b]{0.4\linewidth}
\includegraphics[width=\linewidth]{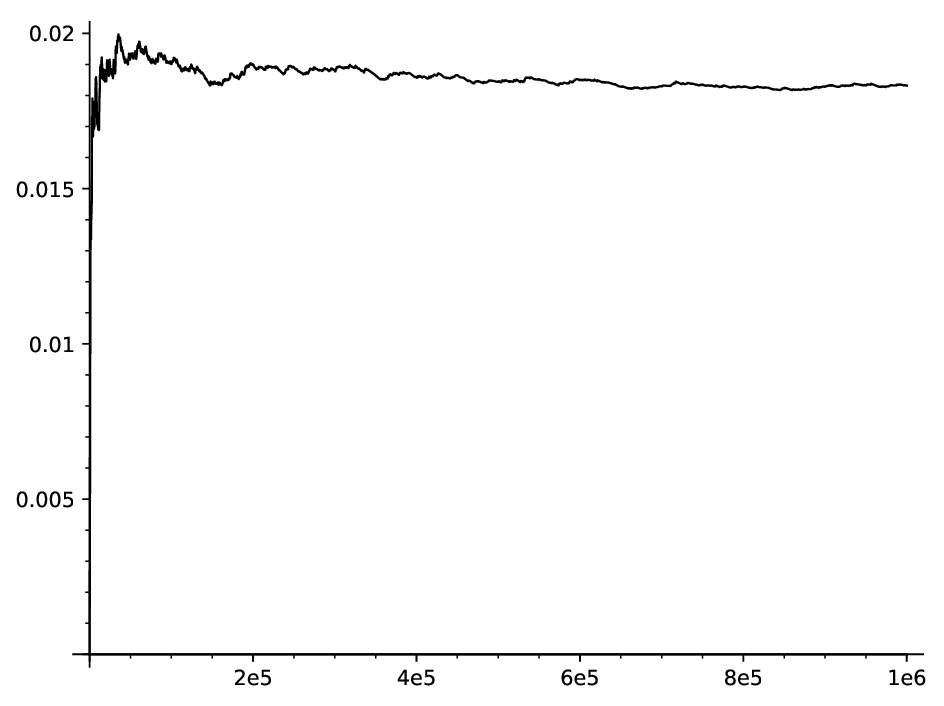}
\caption{$|l| = 7$: No 7 exists} \label{fig:19_6_2_3_A_7}
\end{subfigure}\hspace*{\fill}
\begin{subfigure}[b]{0.4\linewidth}
\includegraphics[width=\linewidth]{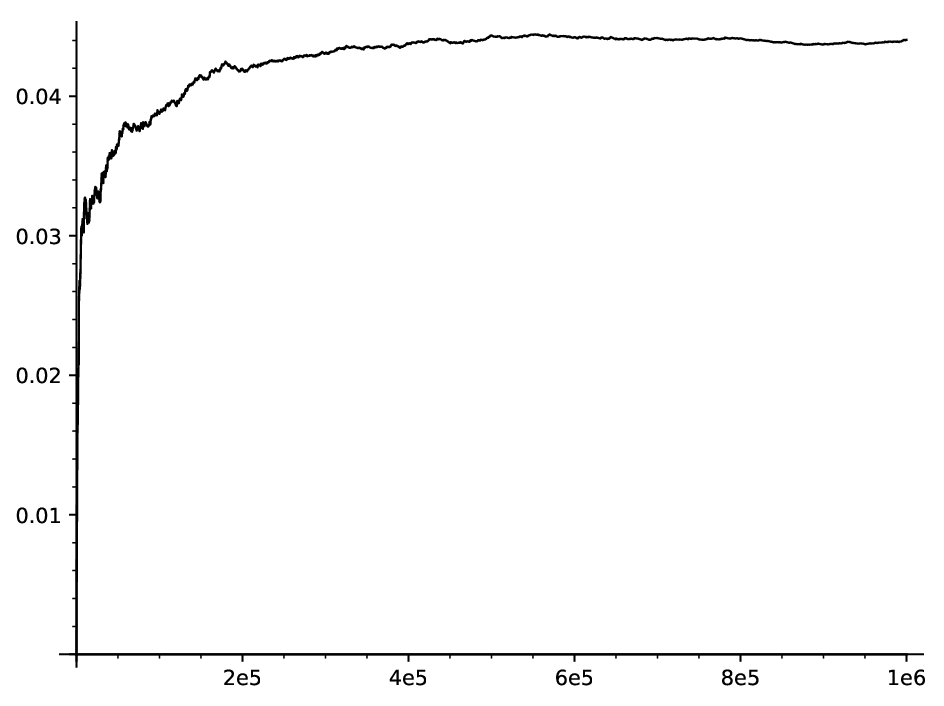}
\caption{$|l| = 8$: No -8 exists} \label{fig:19_6_2_3_A_8}
\end{subfigure}\hspace*{\fill}
\begin{subfigure}[b]{0.4\linewidth}
\includegraphics[width=\linewidth]{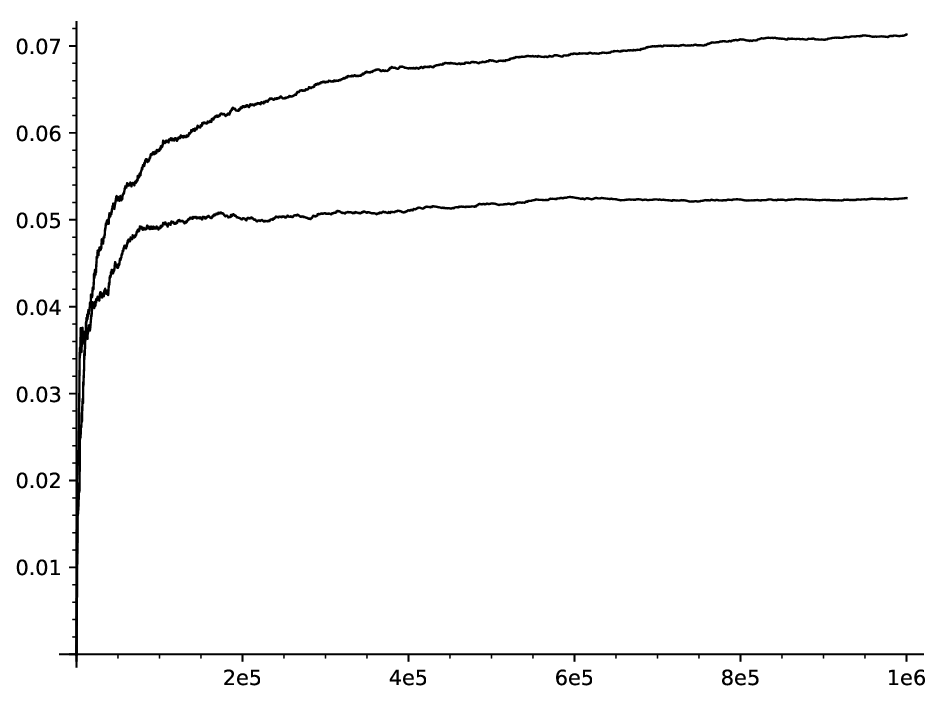}
\caption{$|l| = 9$: Top -9 bottom 9} \label{fig:19_6_2_3_A_9}
\end{subfigure}
\caption{19a1: $(\alpha, \beta) = (2,3)$ Ratio~\eqref{ratio_n_orders} $x_{6,E}^{(\alpha, \beta)}(X;l)/X^{1/2}\log^2(X)$} \label{fig:19a1_6_2_3_A_exact}
\end{figure}

\clearpage

\begin{figure}[t] 
\hspace*{-2.3cm}
\begin{subfigure}[b]{0.4\linewidth}
\includegraphics[width=\linewidth]{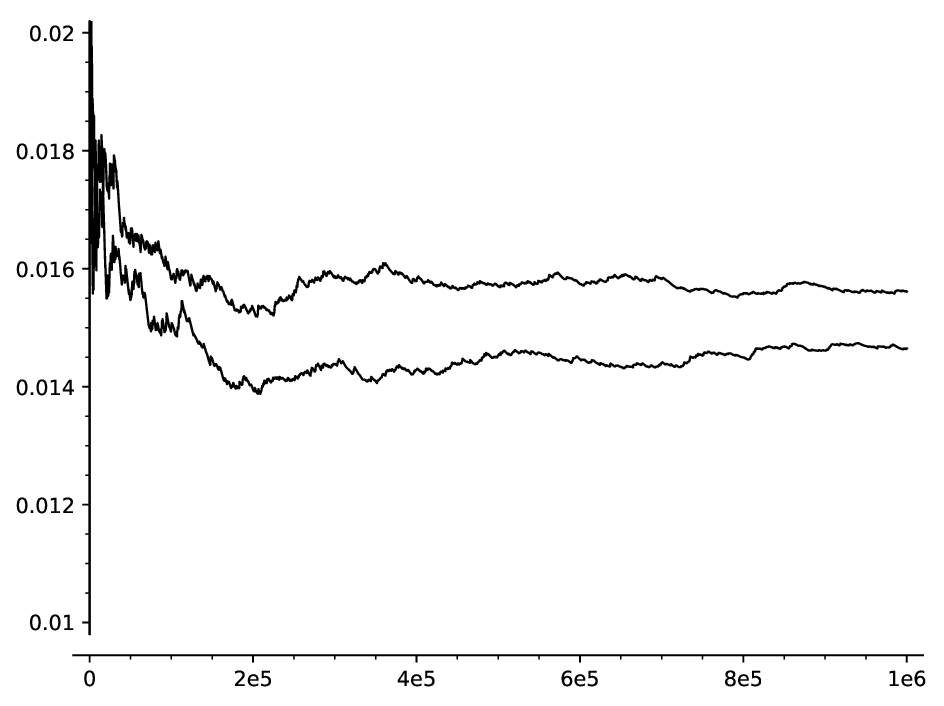}
\caption{$|l| = 1$: Top -1 bottom 1} \label{fig:19_6_1_6_A_1}
\end{subfigure}\hspace*{\fill}
\begin{subfigure}[b]{0.4\linewidth}
\includegraphics[width=\linewidth]{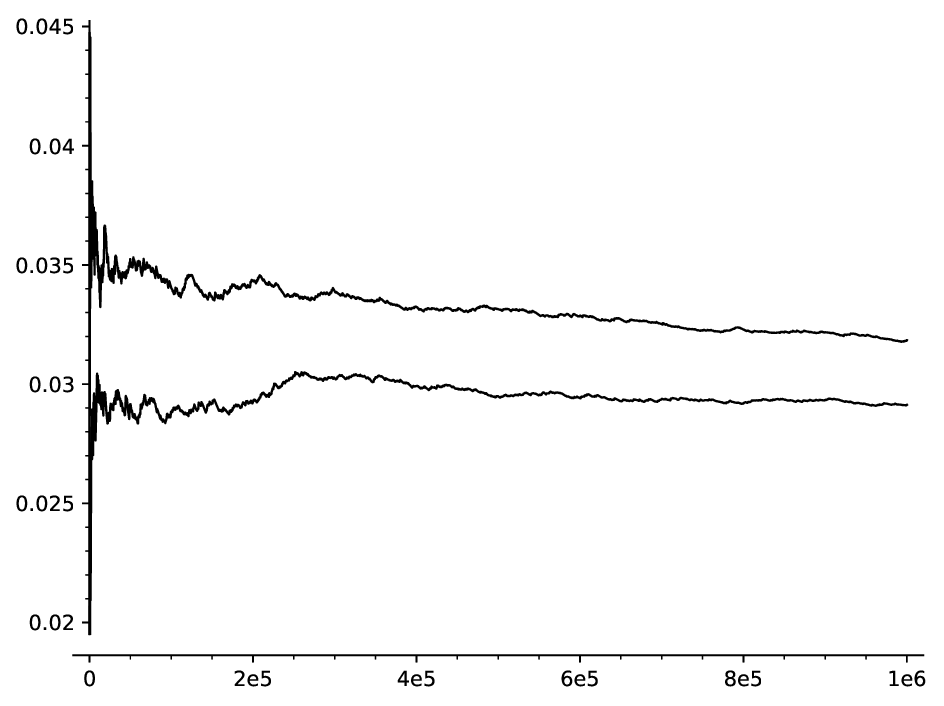}
\caption{$|l| = 2$: Top 2 bottom -2} \label{fig:19_6_1_6_A_2}
\end{subfigure}\hspace*{\fill}
\begin{subfigure}[b]{0.4\linewidth}
\includegraphics[width=\linewidth]{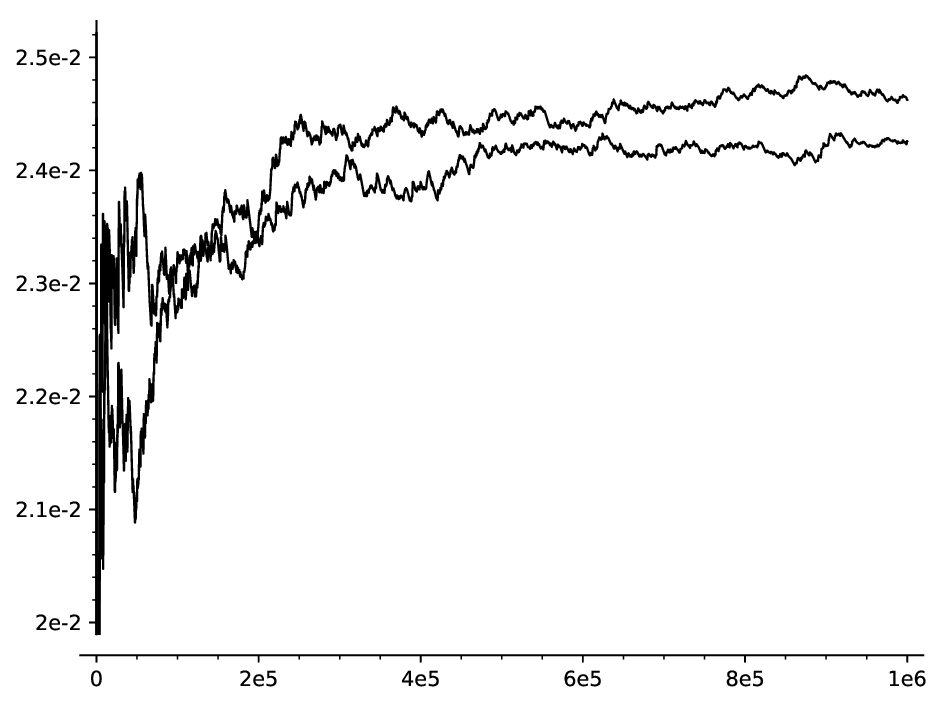}
\caption{$|l| = 3$: Top -3 bottom 3} \label{fig:19_6_1_6_A_3}
\end{subfigure}
\hspace*{-2.3cm}
\begin{subfigure}[b]{0.4\linewidth}
\includegraphics[width=\linewidth]{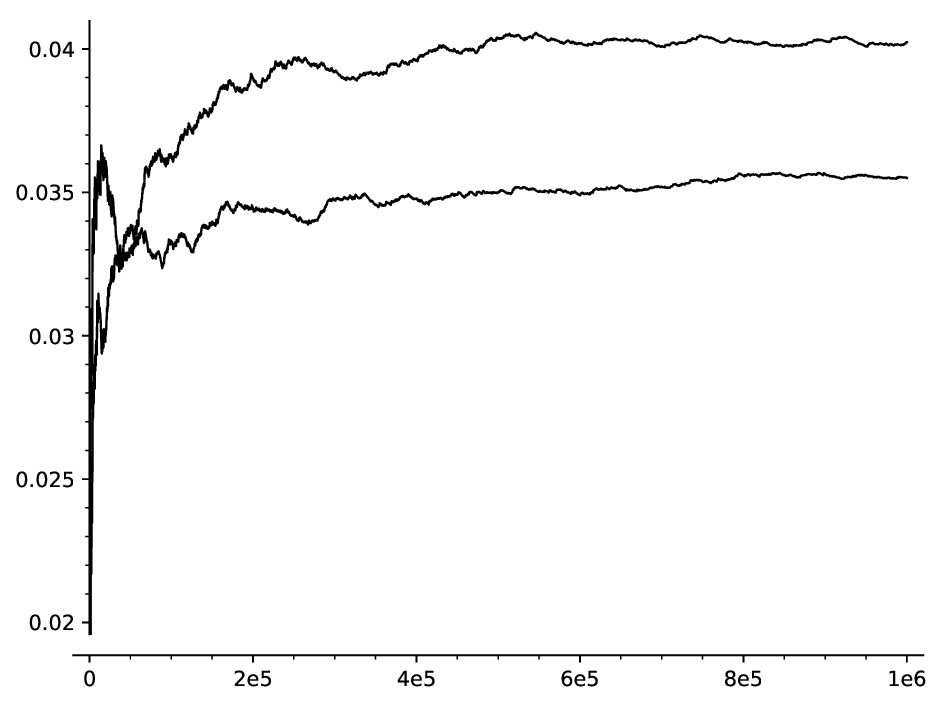}
\caption{$|l| = 4$: Top -4 bottom 4} \label{fig:19_6_1_6_A_4}
\end{subfigure}\hspace*{\fill}
\begin{subfigure}[b]{0.4\linewidth}
\includegraphics[width=\linewidth]{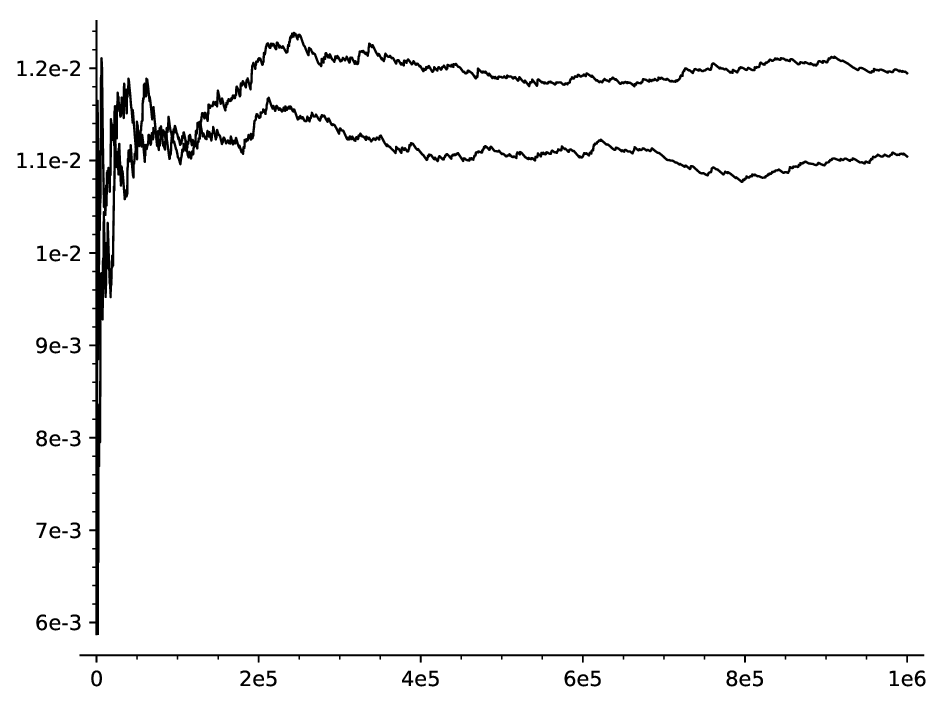}
\caption{$|l| = 5$: Top 5 bottom -5} \label{fig:19_6_1_6_A_5}
\end{subfigure}\hspace*{\fill}
\begin{subfigure}[b]{0.4\linewidth}
\includegraphics[width=\linewidth]{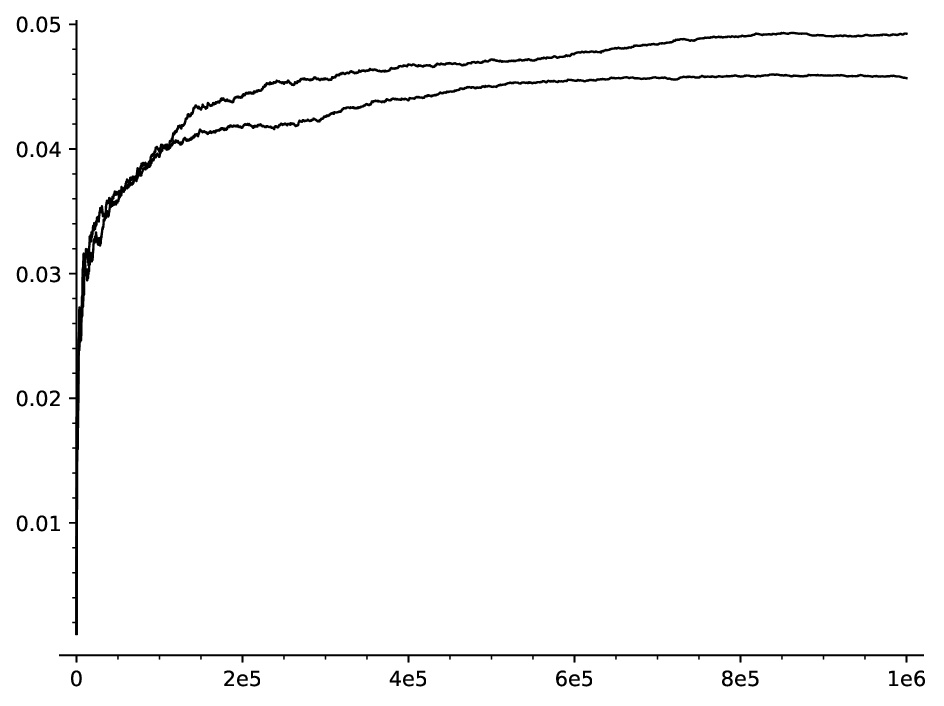}
\caption{$|l| = 6$: Top 6 bottom -6} \label{fig:19_6_1_6_A_6}
\end{subfigure}
\hspace*{-2.3cm}
\begin{subfigure}[b]{0.4\linewidth}
\includegraphics[width=\linewidth]{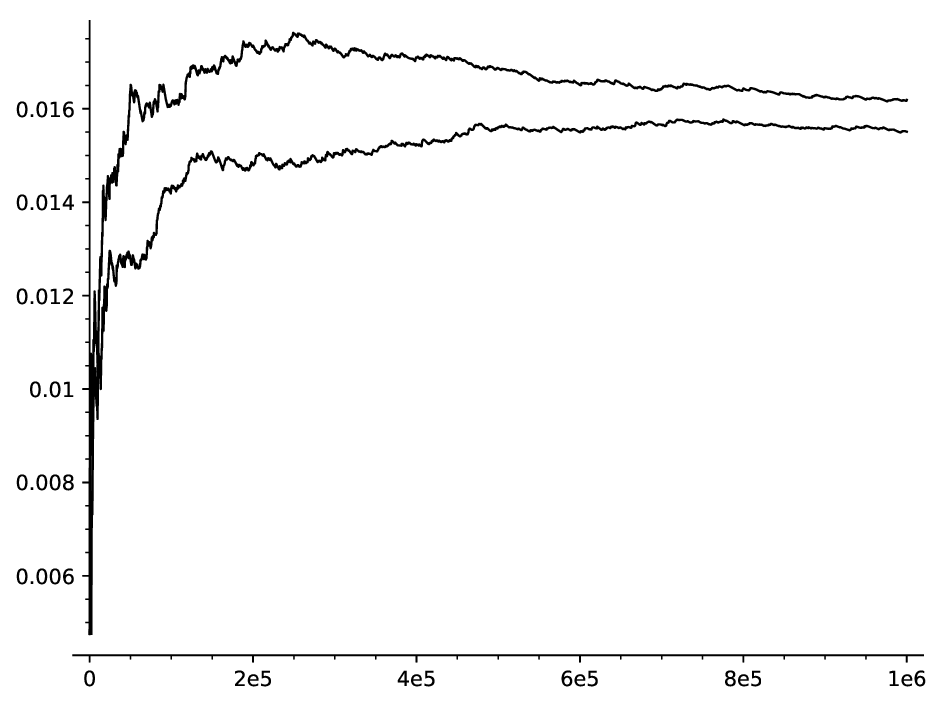}
\caption{$|l| = 7$: Top 7 bottom -7} \label{fig:19_6_1_6_A_7}
\end{subfigure}\hspace*{\fill}
\begin{subfigure}[b]{0.4\linewidth}
\includegraphics[width=\linewidth]{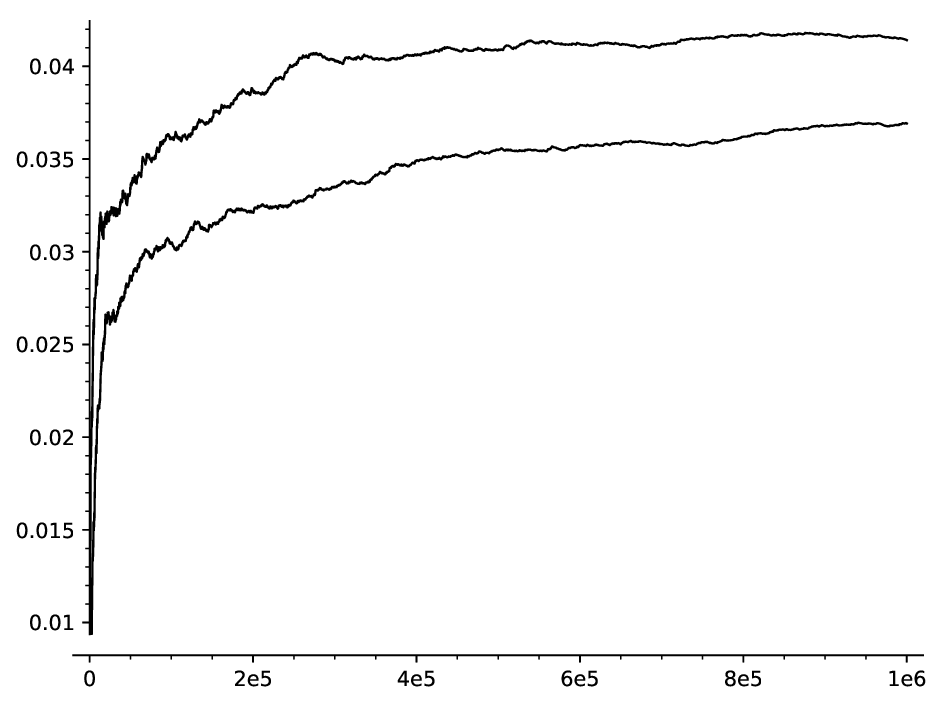}
\caption{$|l| = 8$: Top 8 bottom -8} \label{fig:19_6_1_6_A_8}
\end{subfigure}\hspace*{\fill}
\begin{subfigure}[b]{0.4\linewidth}
\includegraphics[width=\linewidth]{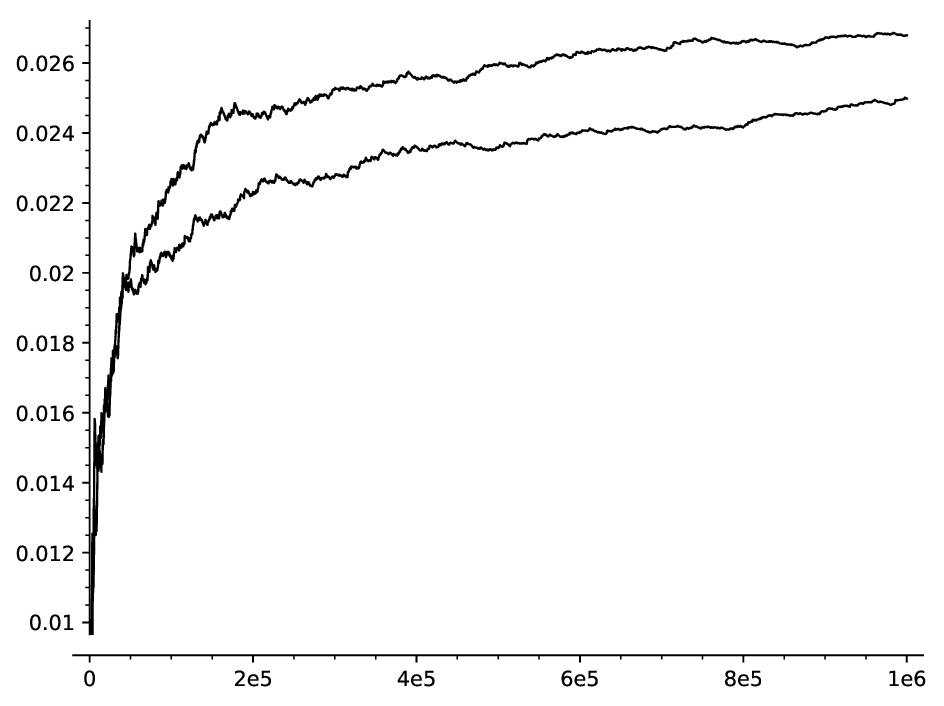}
\caption{$|l| = 9$: Top 9 bottom -9} \label{fig:19_6_1_6_A_9}
\end{subfigure}
\caption{19a1: $(\alpha, \beta) = (1,6)$ Ratio~\eqref{ratio_n_orders} $x_{6,E}^{(\alpha, \beta)}(X;l)/X^{1/2}\log^2(X)$} \label{fig:19a1_6_1_6_A_exact}
\end{figure}

\clearpage

\begin{figure}[t] 
\hspace*{-2.3cm}
\begin{subfigure}[b]{0.4\linewidth}
\includegraphics[width=\linewidth]{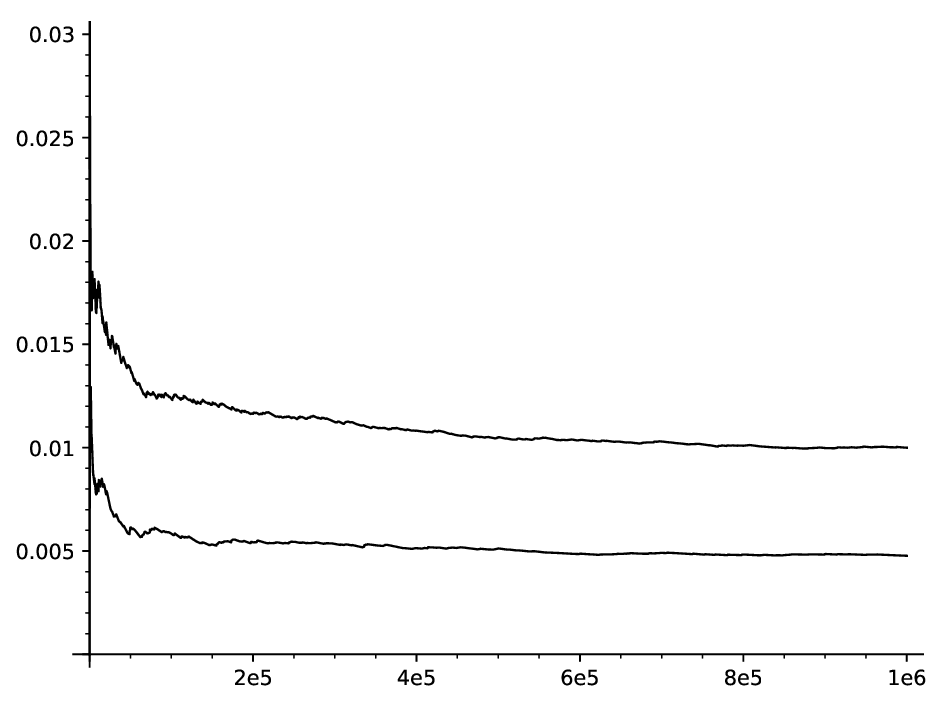}
\caption{$|l| = 1$: Top -1 bottom 1} \label{fig:19_6_2_6_A_1}
\end{subfigure}\hspace*{\fill}
\begin{subfigure}[b]{0.4\linewidth}
\includegraphics[width=\linewidth]{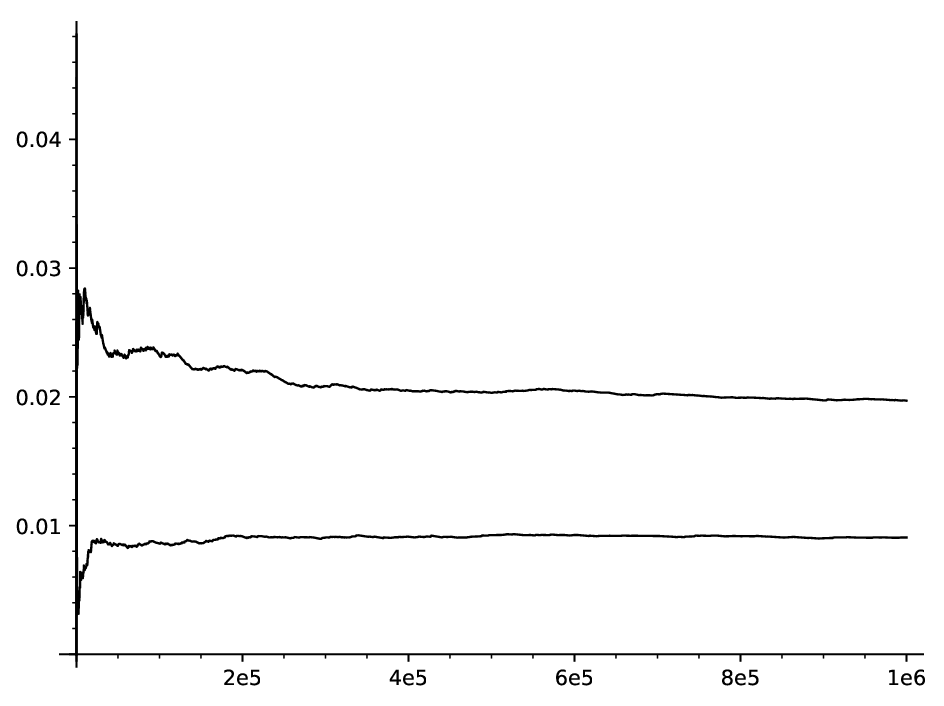}
\caption{$|l| = 2$: Top 2 bottom -2} \label{fig:19_6_2_6_A_2}
\end{subfigure}\hspace*{\fill}
\begin{subfigure}[b]{0.4\linewidth}
\includegraphics[width=\linewidth]{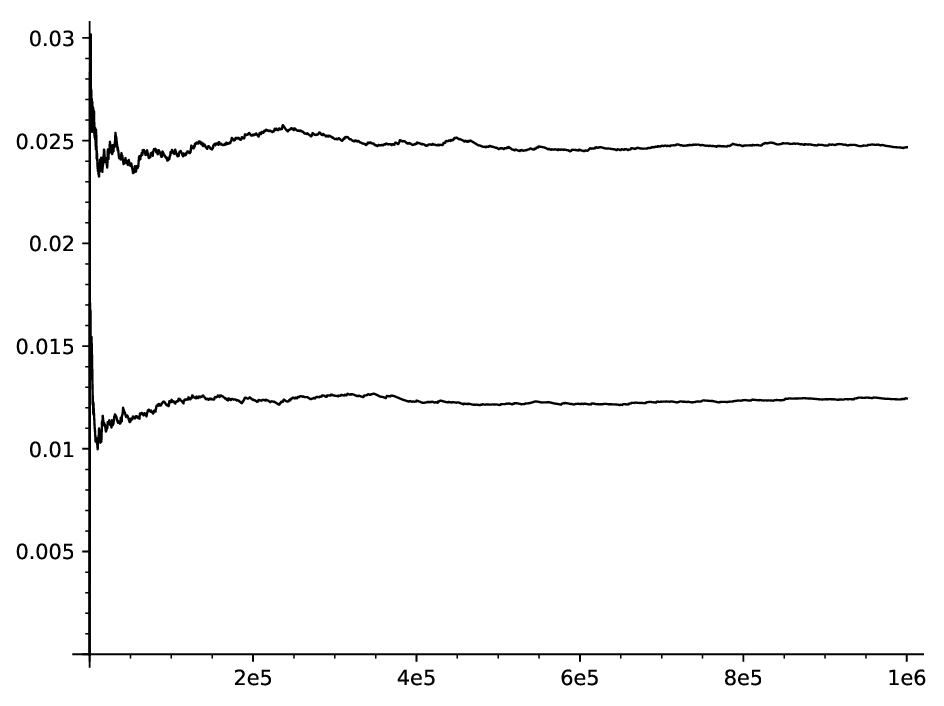}
\caption{$|l| = 3$: Top -3 bottom 3} \label{fig:19_6_2_6_A_3}
\end{subfigure}
\hspace*{-2.3cm}
\begin{subfigure}[b]{0.4\linewidth}
\includegraphics[width=\linewidth]{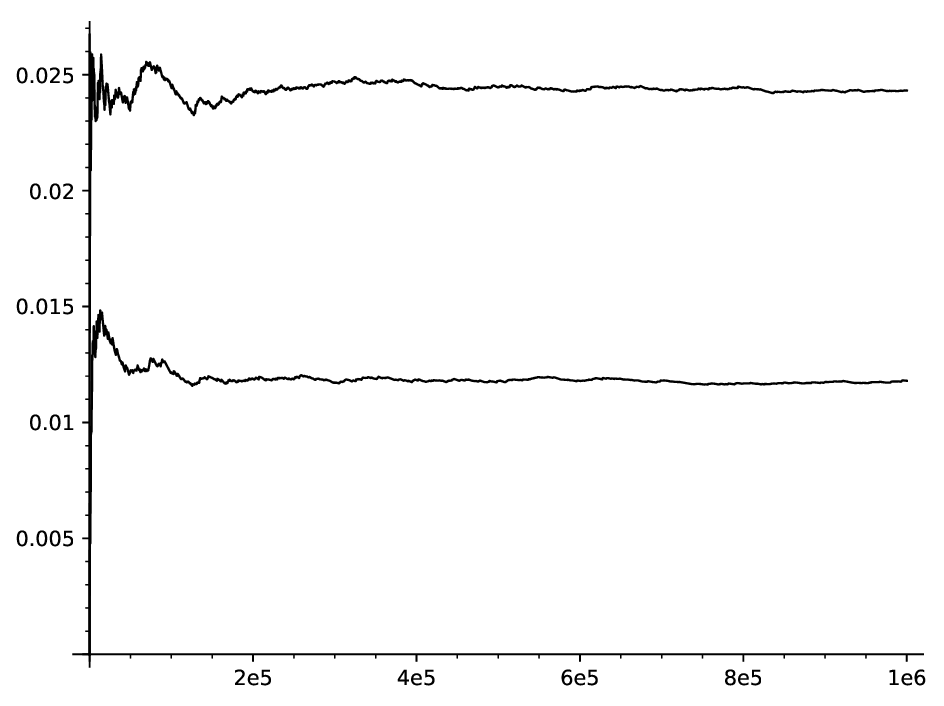}
\caption{$|l| = 4$: Top -4 bottom 4} \label{fig:19_6_2_6_A_4}
\end{subfigure}\hspace*{\fill}
\begin{subfigure}[b]{0.4\linewidth}
\includegraphics[width=\linewidth]{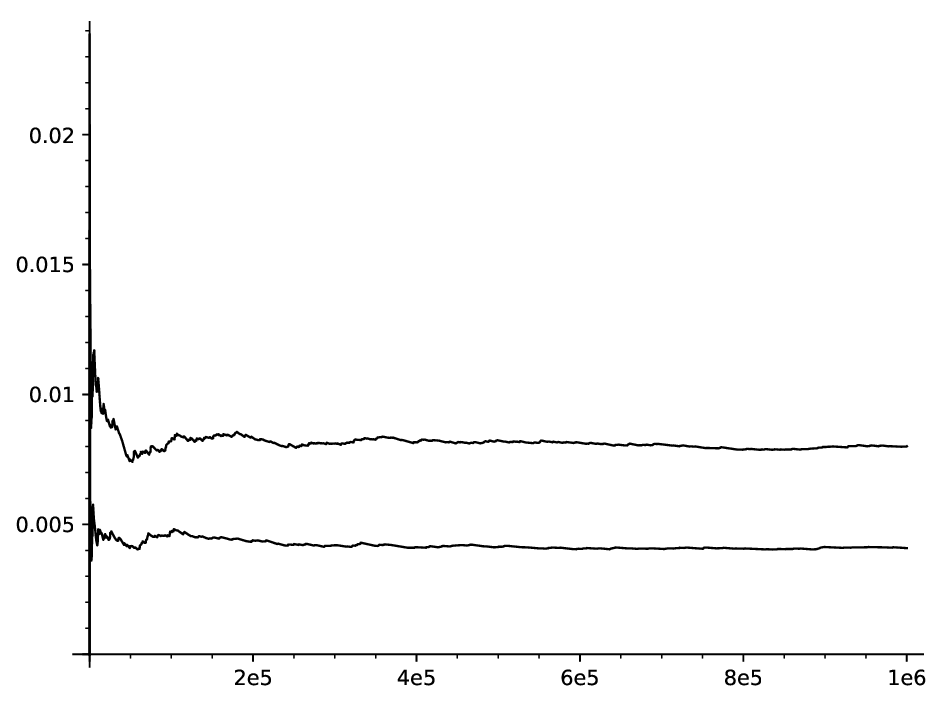}
\caption{$|l| = 5$: Top 5 bottom -5} \label{fig:19_6_2_6_A_5}
\end{subfigure}\hspace*{\fill}
\begin{subfigure}[b]{0.4\linewidth}
\includegraphics[width=\linewidth]{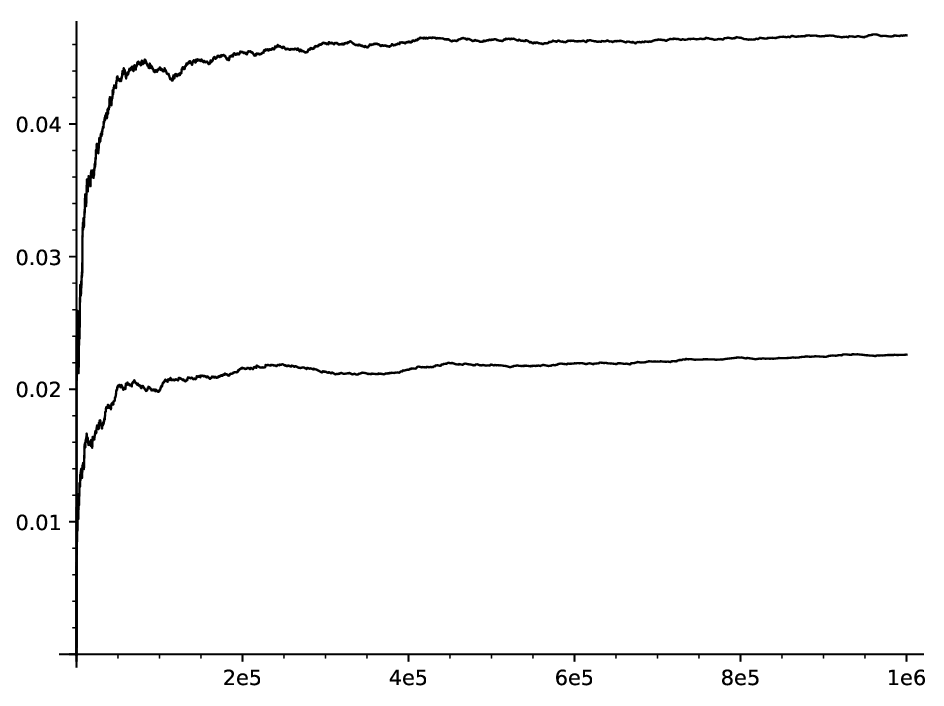}
\caption{$|l| = 6$: Top 6 bottom -6} \label{fig:19_6_2_6_A_6}
\end{subfigure}
\hspace*{-2.3cm}
\begin{subfigure}[b]{0.4\linewidth}
\includegraphics[width=\linewidth]{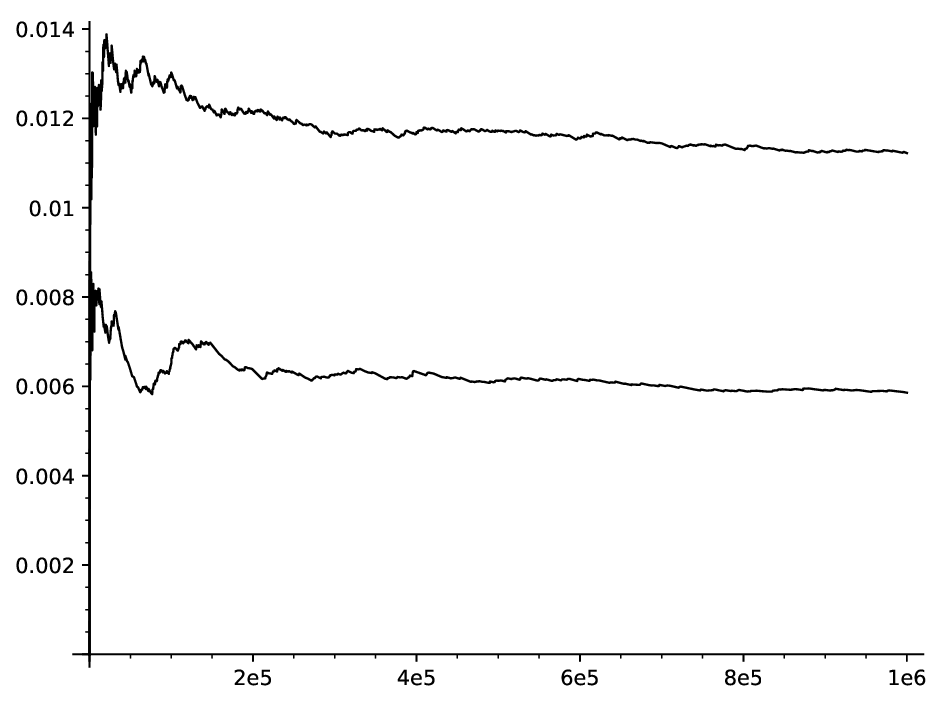}
\caption{$|l| = 7$: Top -7 bottom 7} \label{fig:19_6_2_6_A_7}
\end{subfigure}\hspace*{\fill}
\begin{subfigure}[b]{0.4\linewidth}
\includegraphics[width=\linewidth]{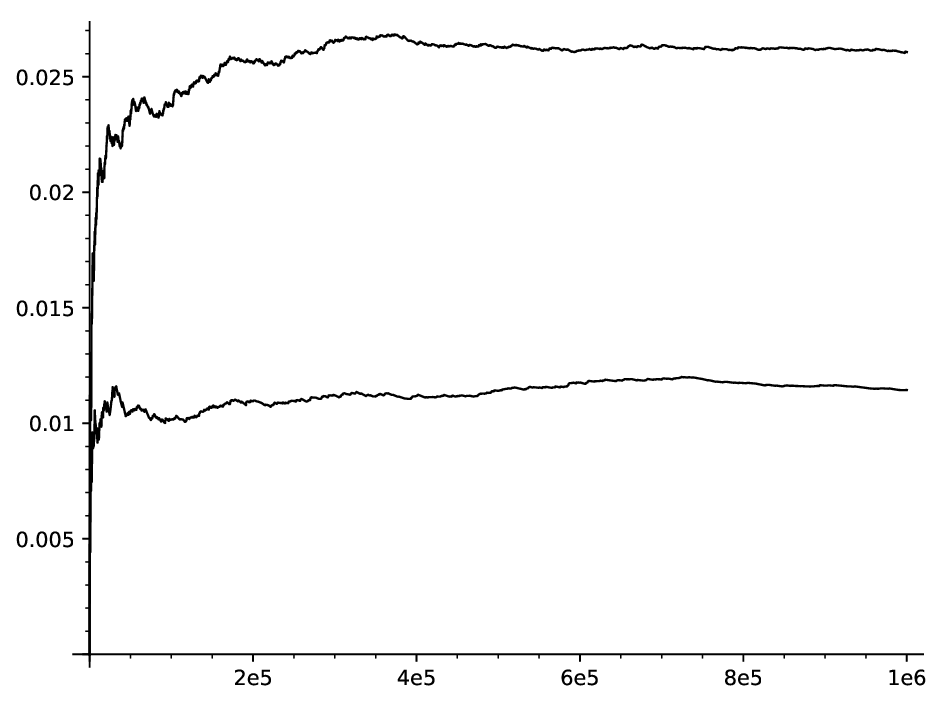}
\caption{$|l| = 8$: Top 8 bottom -8} \label{fig:19_6_2_6_A_8}
\end{subfigure}\hspace*{\fill}
\begin{subfigure}[b]{0.4\linewidth}
\includegraphics[width=\linewidth]{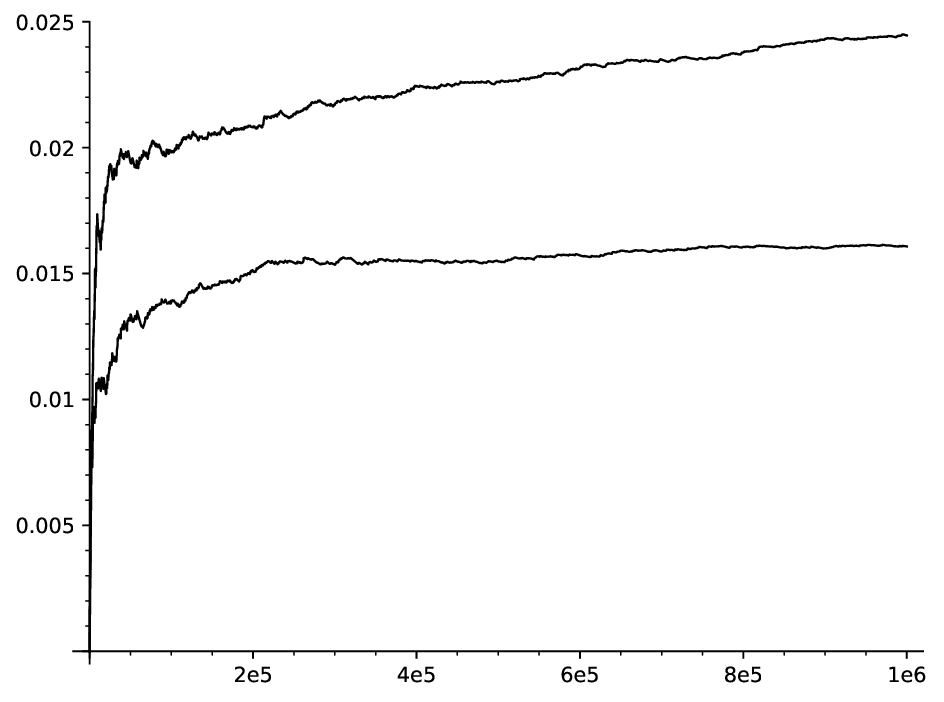}
\caption{$|l| = 9$: Top -9 bottom 9} \label{fig:19_6_2_6_A_9}
\end{subfigure}
\caption{19a1: $(\alpha, \beta) = (2,6)$ Ratio~\eqref{ratio_n_orders} $x_{6,E}^{(\alpha, \beta)}(X;l)/X^{1/2}\log^2(X)$} \label{fig:19a1_6_2_6_A_exact}
\end{figure}

\clearpage

\begin{figure}[t] 
\hspace*{-2.3cm}
\begin{subfigure}[b]{0.4\linewidth}
\includegraphics[width=\linewidth]{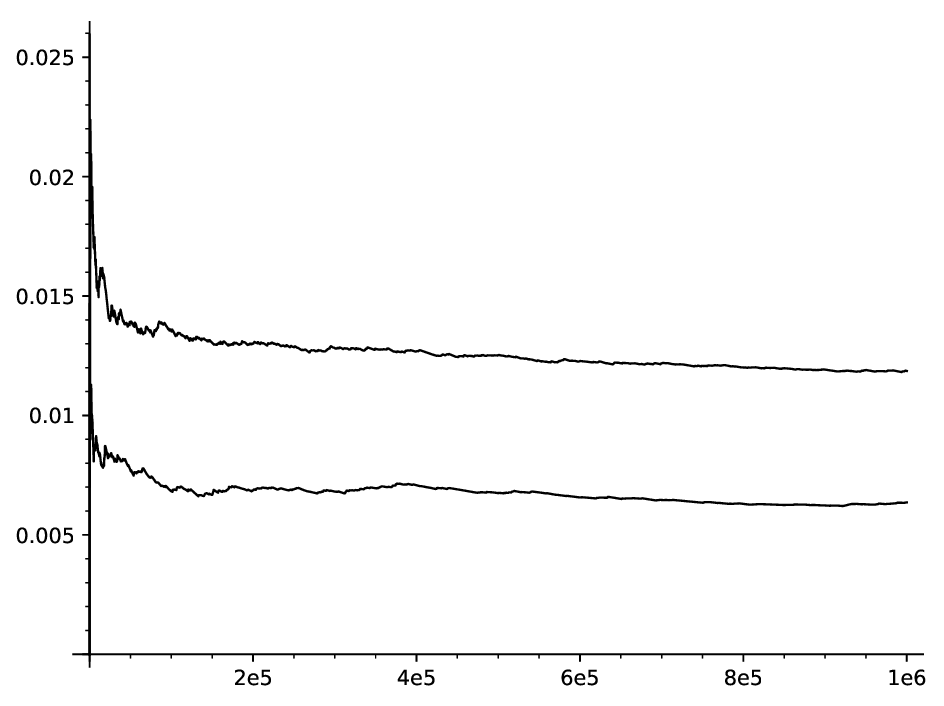}
\caption{$|l| = 1$: Top -1 bottom 1} \label{fig:37_6_1_3_A_1}
\end{subfigure}\hspace*{\fill}
\begin{subfigure}[b]{0.4\linewidth}
\includegraphics[width=\linewidth]{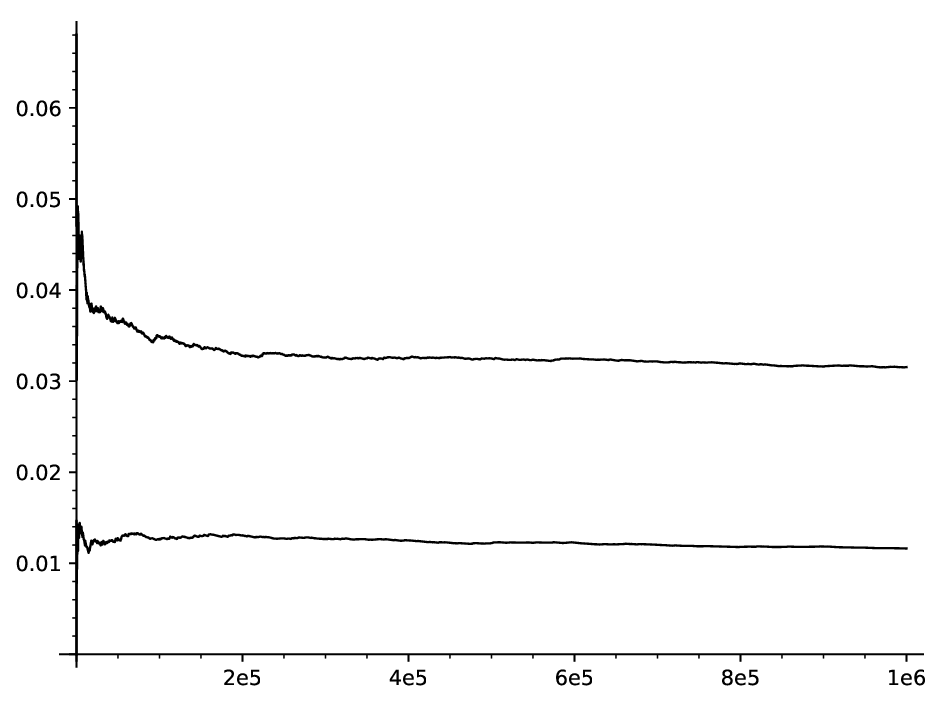}
\caption{$|l| = 2$: Top 2 bottom -2} \label{fig:37_6_1_3_A_2}
\end{subfigure}\hspace*{\fill}
\begin{subfigure}[b]{0.4\linewidth}
\includegraphics[width=\linewidth]{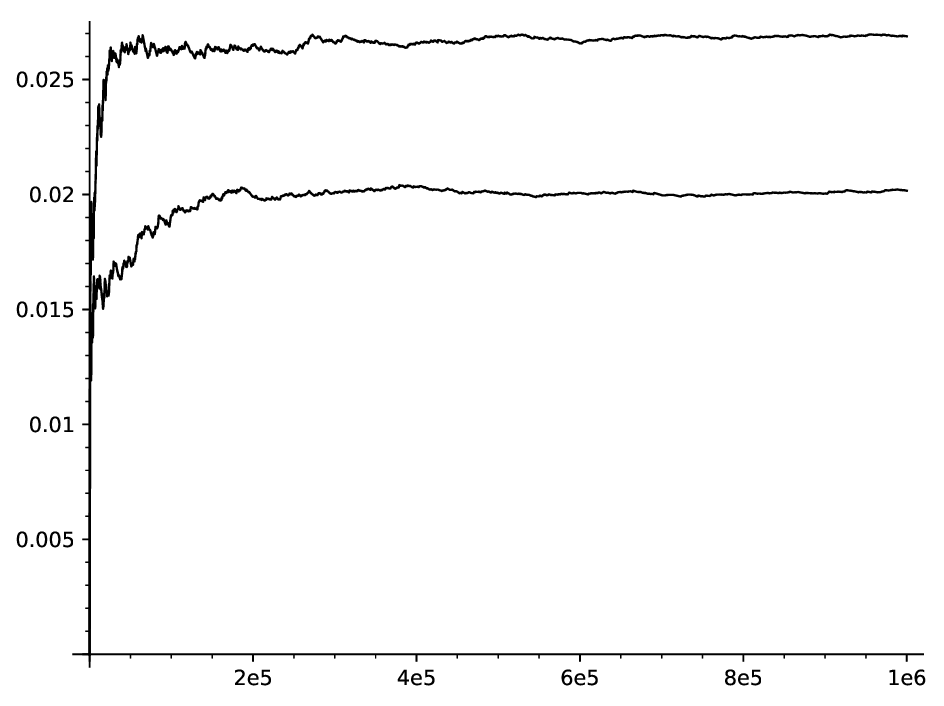}
\caption{$|l| = 3$: Top 3 bottom -3} \label{fig:37_6_1_3_A_3}
\end{subfigure}
\hspace*{-2.3cm}
\begin{subfigure}[b]{0.4\linewidth}
\includegraphics[width=\linewidth]{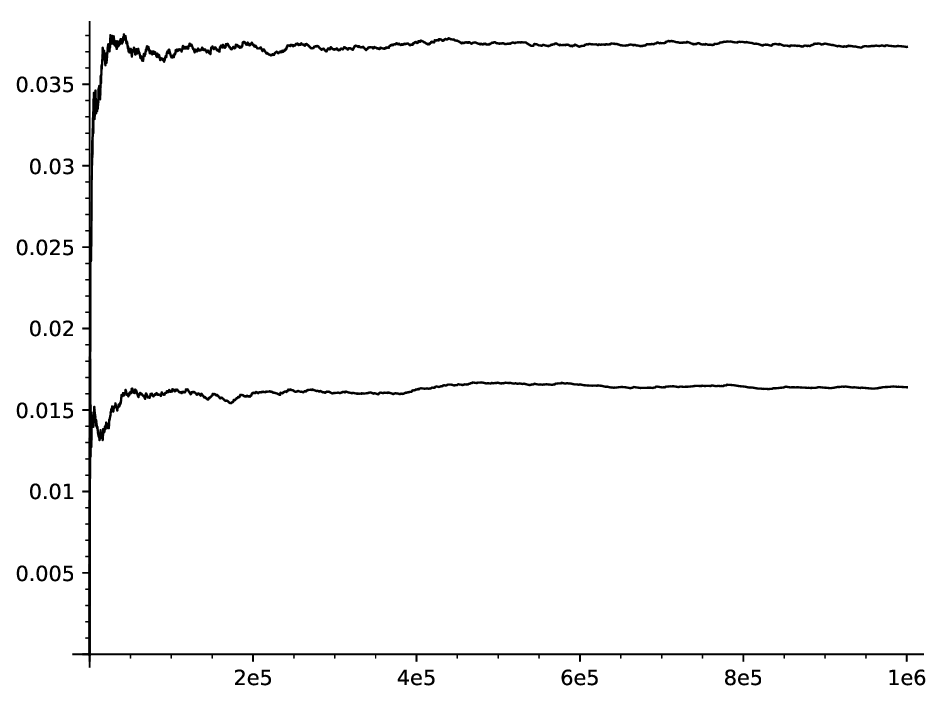}
\caption{$|l| = 4$: Top -4 bottom 4} \label{fig:37_6_1_3_A_4}
\end{subfigure}\hspace*{\fill}
\begin{subfigure}[b]{0.4\linewidth}
\includegraphics[width=\linewidth]{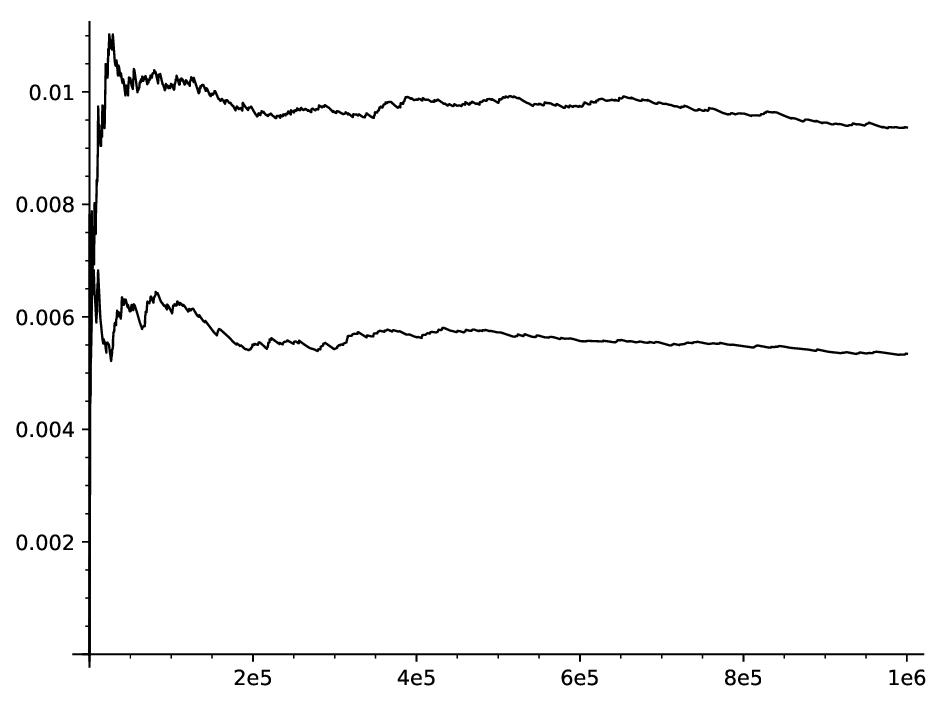}
\caption{$|l| = 5$: Top 5 bottom -5} \label{fig:37_6_1_3_A_5}
\end{subfigure}\hspace*{\fill}
\begin{subfigure}[b]{0.4\linewidth}
\includegraphics[width=\linewidth]{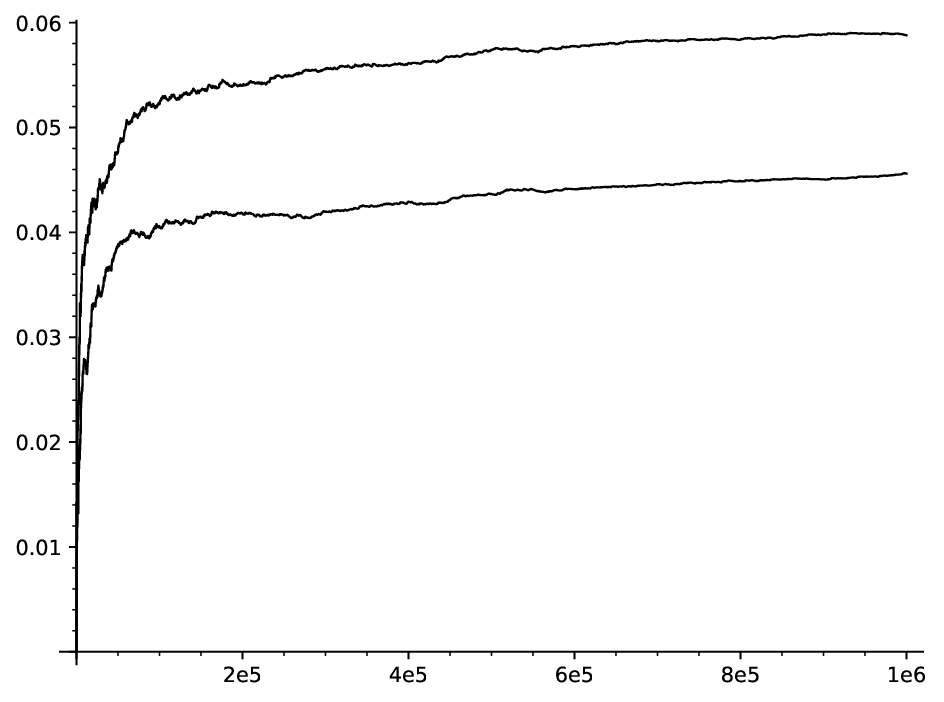}
\caption{$|l| = 6$: Top -6 bottom 6} \label{fig:37_6_1_3_A_6}
\end{subfigure}
\hspace*{-2.3cm}
\begin{subfigure}[b]{0.4\linewidth}
\includegraphics[width=\linewidth]{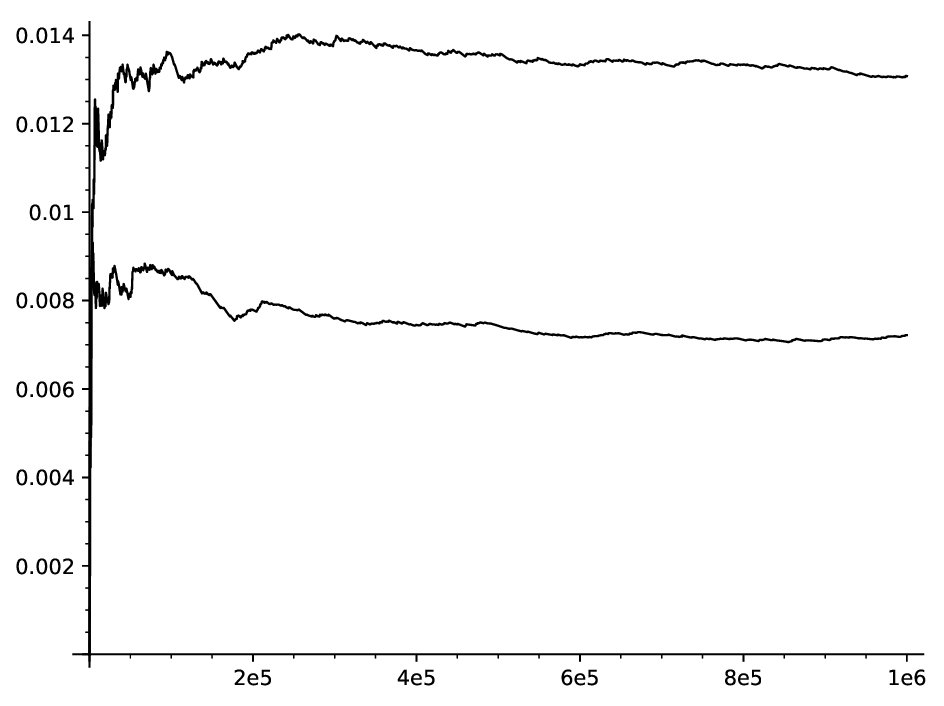}
\caption{$|l| = 7$: Top -7 bottom 7} \label{fig:37_6_1_3_A_7}
\end{subfigure}\hspace*{\fill}
\begin{subfigure}[b]{0.4\linewidth}
\includegraphics[width=\linewidth]{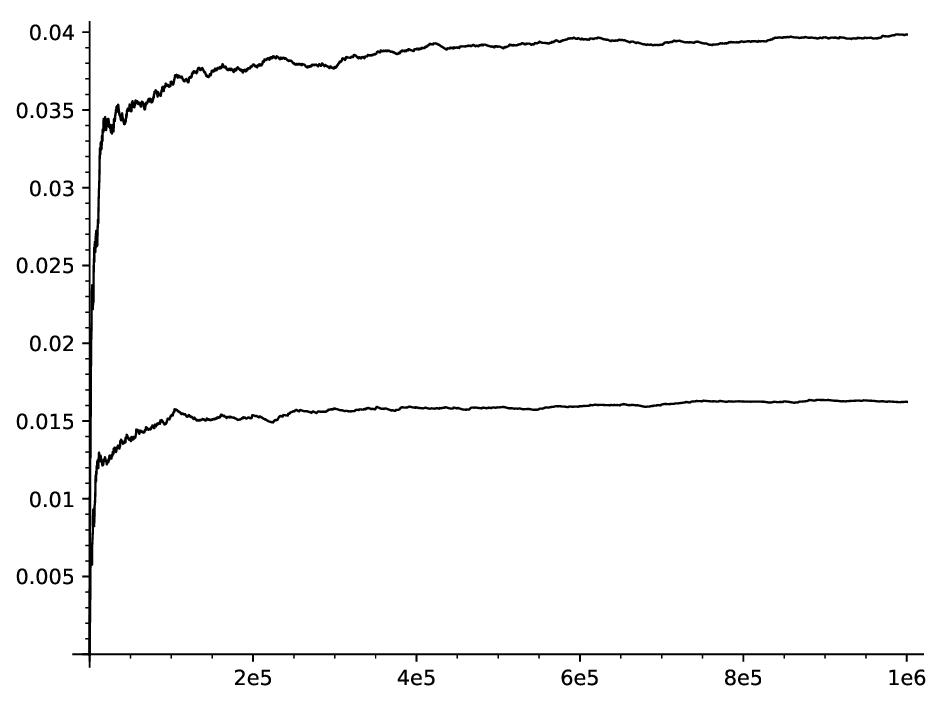}
\caption{$|l| = 8$: Top 8 bottom -8} \label{fig:37_6_1_3_A_8}
\end{subfigure}\hspace*{\fill}
\begin{subfigure}[b]{0.4\linewidth}
\includegraphics[width=\linewidth]{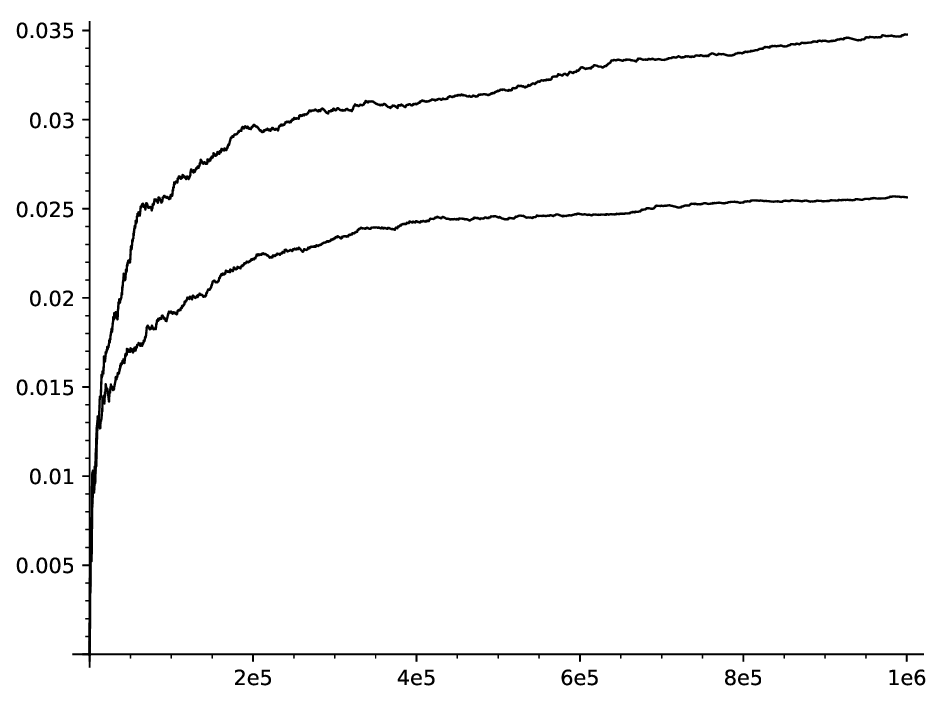}
\caption{$|l| = 9$: Top 9 bottom -9} \label{fig:37_6_1_3_A_9}
\end{subfigure}
\caption{37b1: $(\alpha, \beta) = (1,3)$ Ratio~\eqref{ratio_n_orders} $x_{6,E}^{(\alpha, \beta)}(X;l)/X^{1/2}\log^2(X)$} \label{fig:37b1_6_1_3_A_exact}
\end{figure}

\clearpage

\begin{figure}[t] 
\hspace*{-2.3cm}
\begin{subfigure}[b]{0.4\linewidth}
\includegraphics[width=\linewidth]{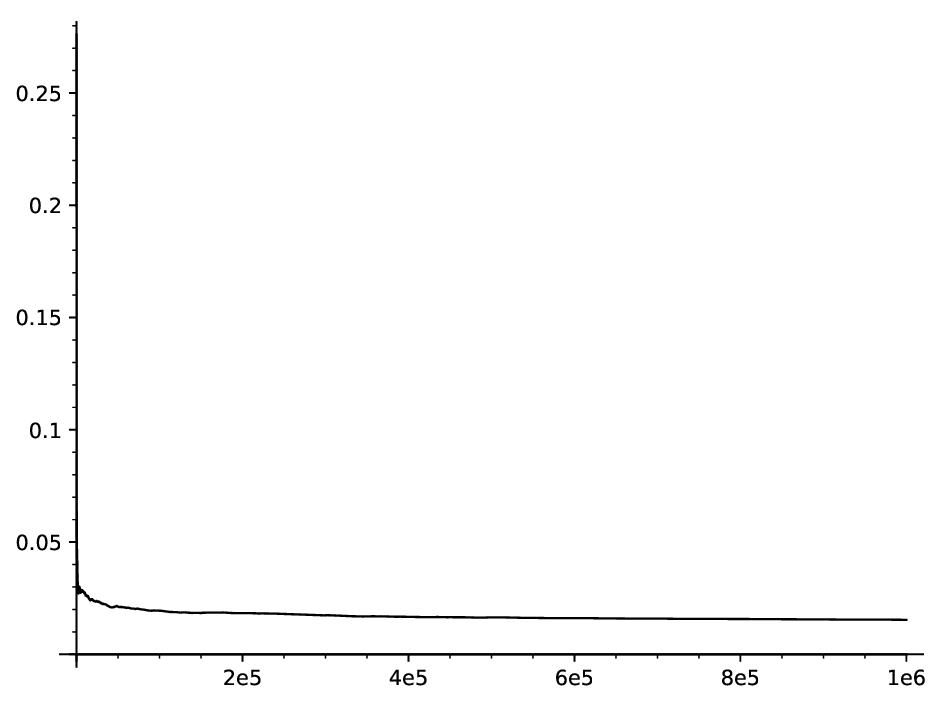}
\caption{$|l| = 1$: No -1 exists} \label{fig:37_6_2_3_A_1}
\end{subfigure}\hspace*{\fill}
\begin{subfigure}[b]{0.4\linewidth}
\includegraphics[width=\linewidth]{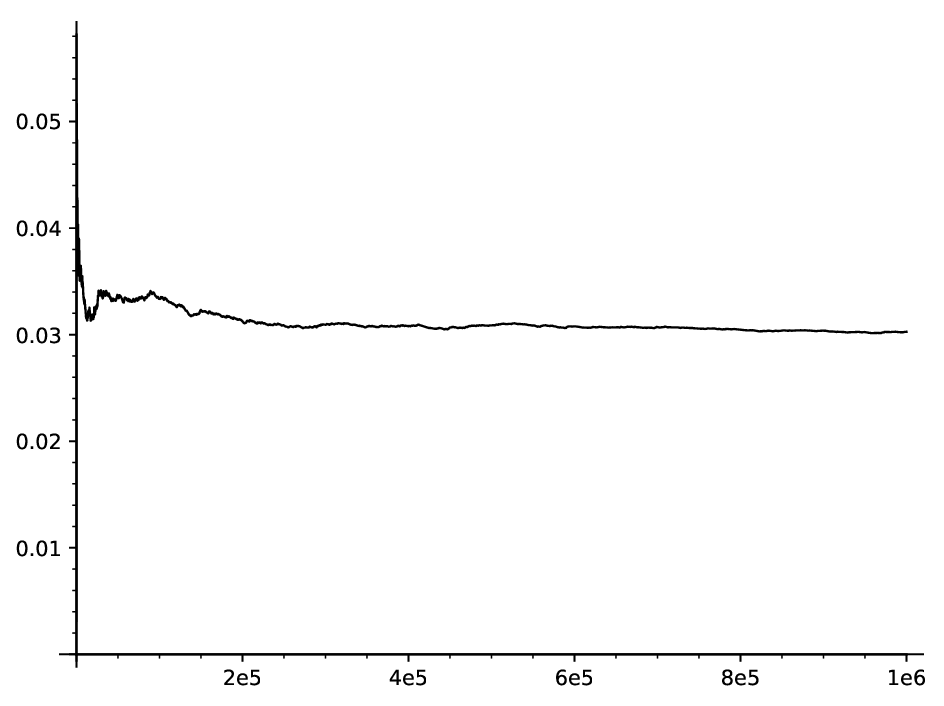}
\caption{$|l| = 2$: No 2 exists} \label{fig:37_6_2_3_A_2}
\end{subfigure}\hspace*{\fill}
\begin{subfigure}[b]{0.4\linewidth}
\includegraphics[width=\linewidth]{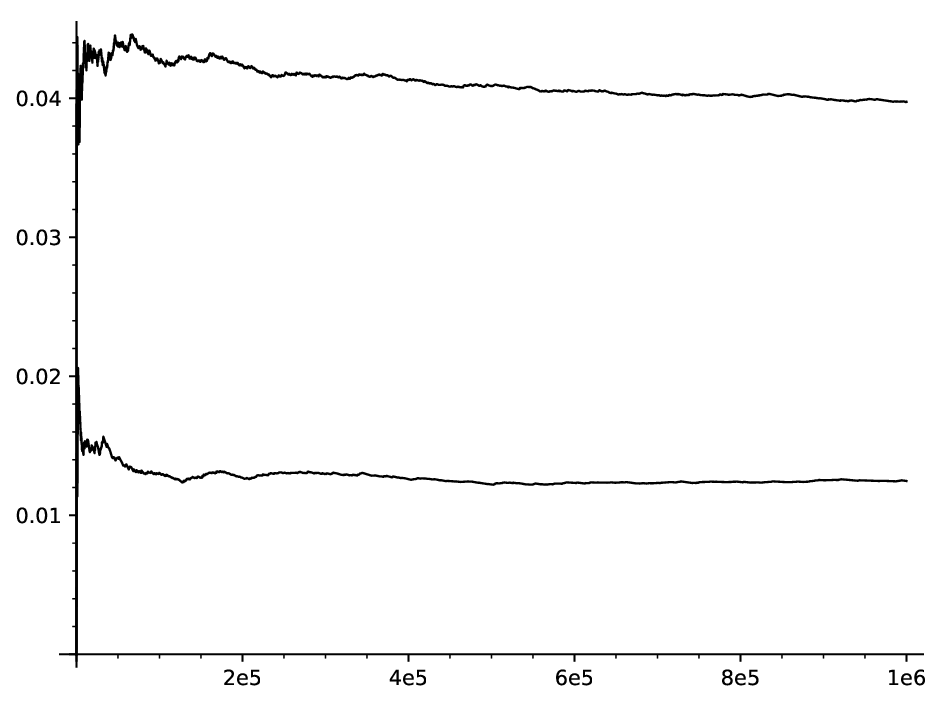}
\caption{$|l| = 3$: Top 3 bottom -3} \label{fig:37_6_2_3_A_3}
\end{subfigure}
\hspace*{-2.3cm}
\begin{subfigure}[b]{0.4\linewidth}
\includegraphics[width=\linewidth]{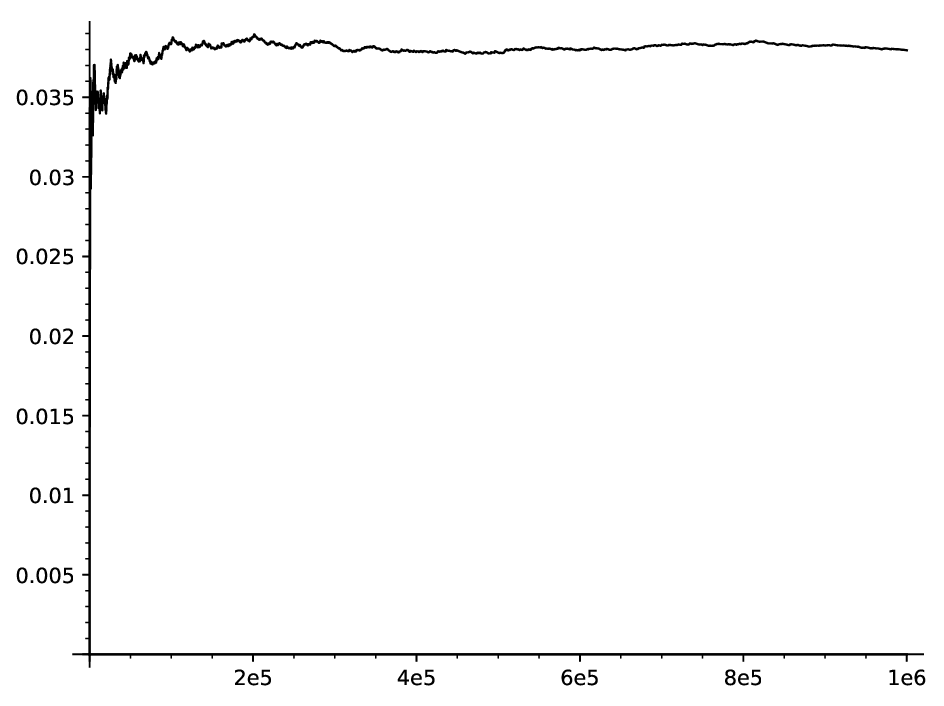}
\caption{$|l| = 4$: No -4 exists} \label{fig:37_6_2_3_A_4}
\end{subfigure}\hspace*{\fill}
\begin{subfigure}[b]{0.4\linewidth}
\includegraphics[width=\linewidth]{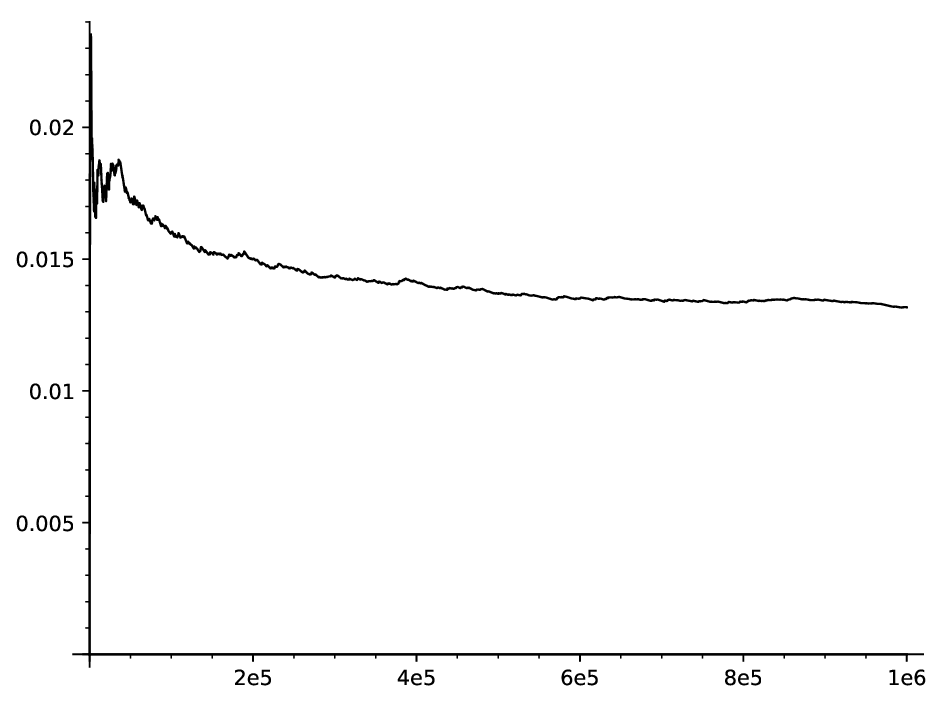}
\caption{$|l| = 5$: No 5 exists} \label{fig:37_6_2_3_A_5}
\end{subfigure}\hspace*{\fill}
\begin{subfigure}[b]{0.4\linewidth}
\includegraphics[width=\linewidth]{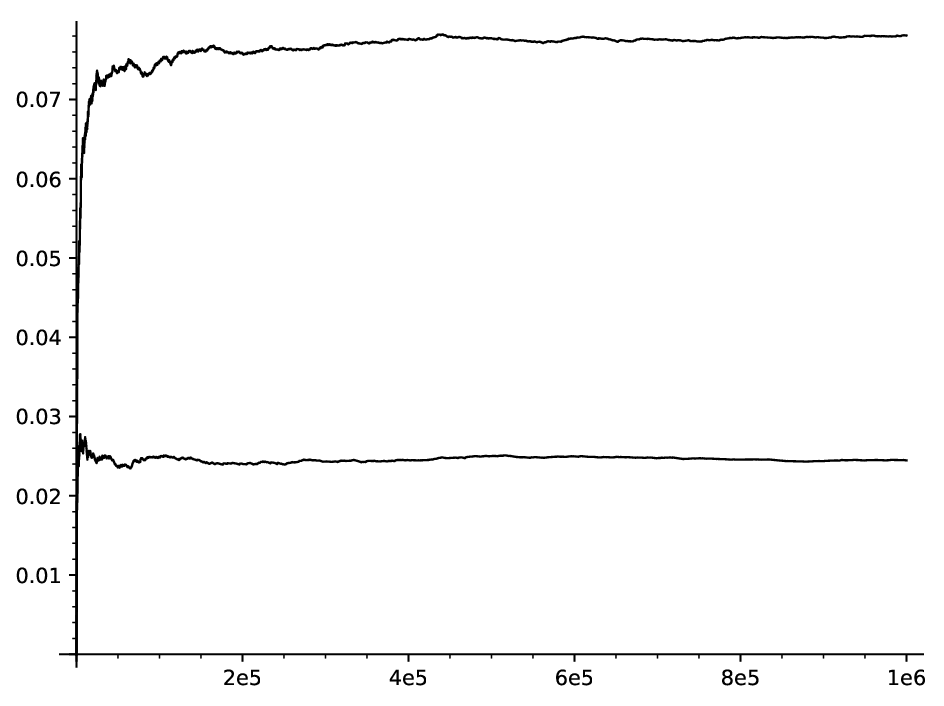}
\caption{$|l| = 6$: Top -6 bottom 6} \label{fig:37_6_2_3_A_6}
\end{subfigure}
\hspace*{-2.3cm}
\begin{subfigure}[b]{0.4\linewidth}
\includegraphics[width=\linewidth]{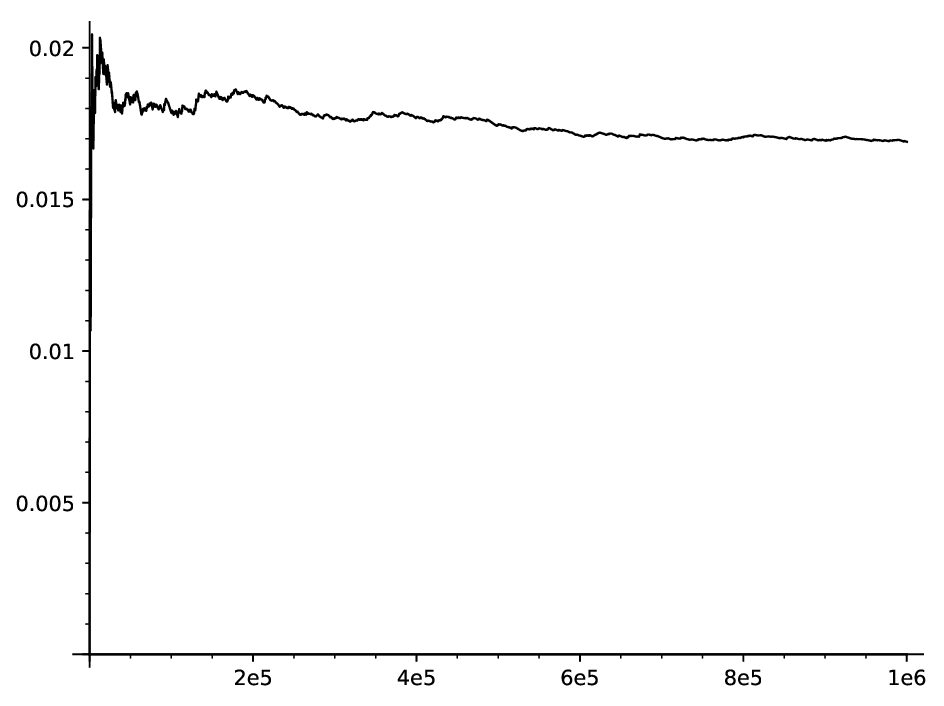}
\caption{$|l| = 7$: No -7 exists} \label{fig:37_6_2_3_A_7}
\end{subfigure}\hspace*{\fill}
\begin{subfigure}[b]{0.4\linewidth}
\includegraphics[width=\linewidth]{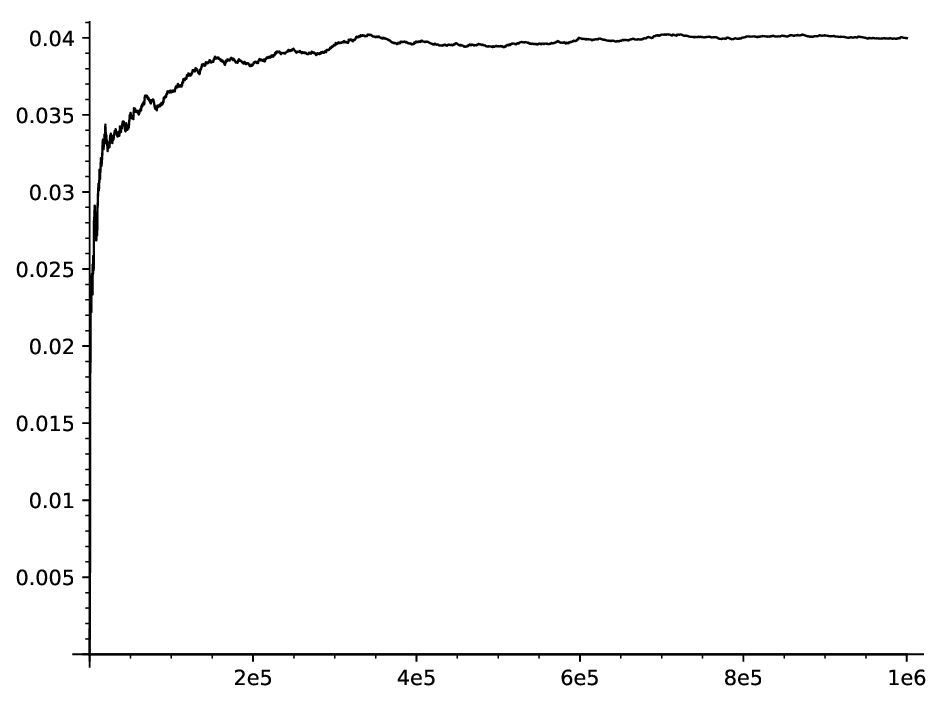}
\caption{$|l| = 8$: No 8 exists} \label{fig:37_6_2_3_A_8}
\end{subfigure}\hspace*{\fill}
\begin{subfigure}[b]{0.4\linewidth}
\includegraphics[width=\linewidth]{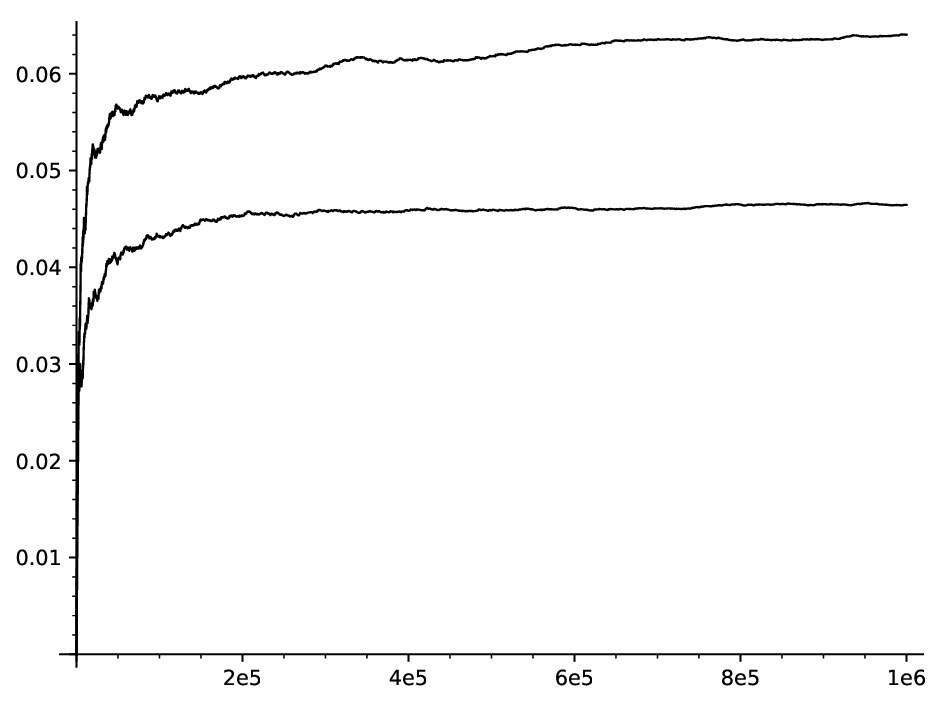}
\caption{$|l| = 9$: Top 9 bottom -9} \label{fig:37_6_2_3_A_9}
\end{subfigure}
\caption{37b1: $(\alpha, \beta) = (2,3)$ Ratio~\eqref{ratio_n_orders} $x_{6,E}^{(\alpha, \beta)}(X;l)/X^{1/2}\log^2(X)$} \label{fig:37b1_6_2_3_A_exact}
\end{figure}

\clearpage

\begin{figure}[t] 
\hspace*{-2.3cm}
\begin{subfigure}[b]{0.4\linewidth}
\includegraphics[width=\linewidth]{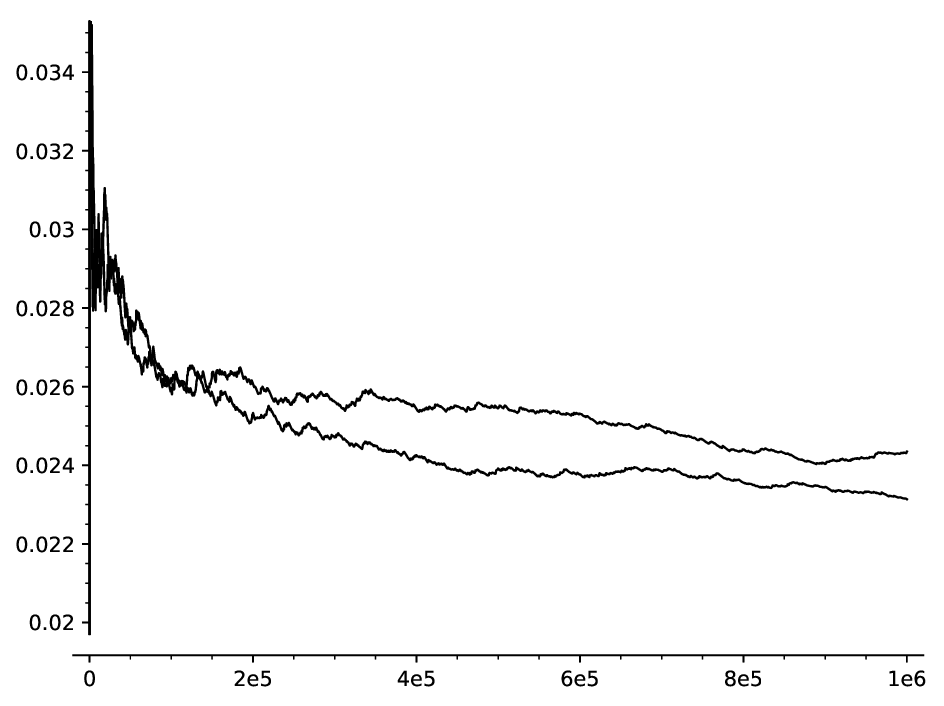}
\caption{$|l| = 1$: Top -1 bottom 1} \label{fig:37_6_1_6_A_1}
\end{subfigure}\hspace*{\fill}
\begin{subfigure}[b]{0.4\linewidth}
\includegraphics[width=\linewidth]{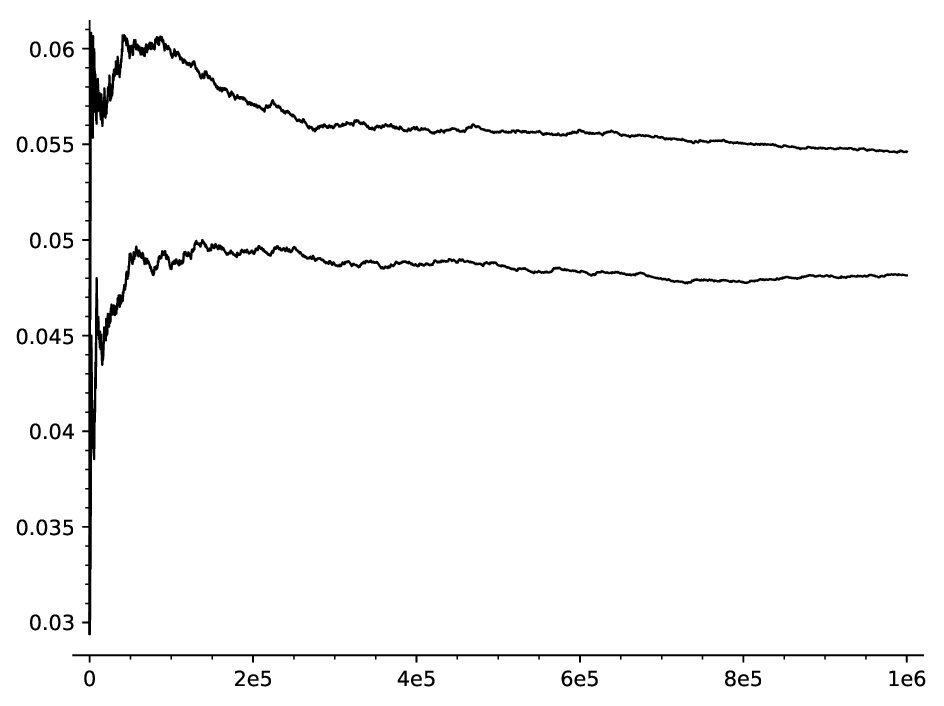}
\caption{$|l| = 2$: Top 2 bottom -2} \label{fig:37_6_1_6_A_2}
\end{subfigure}\hspace*{\fill}
\begin{subfigure}[b]{0.4\linewidth}
\includegraphics[width=\linewidth]{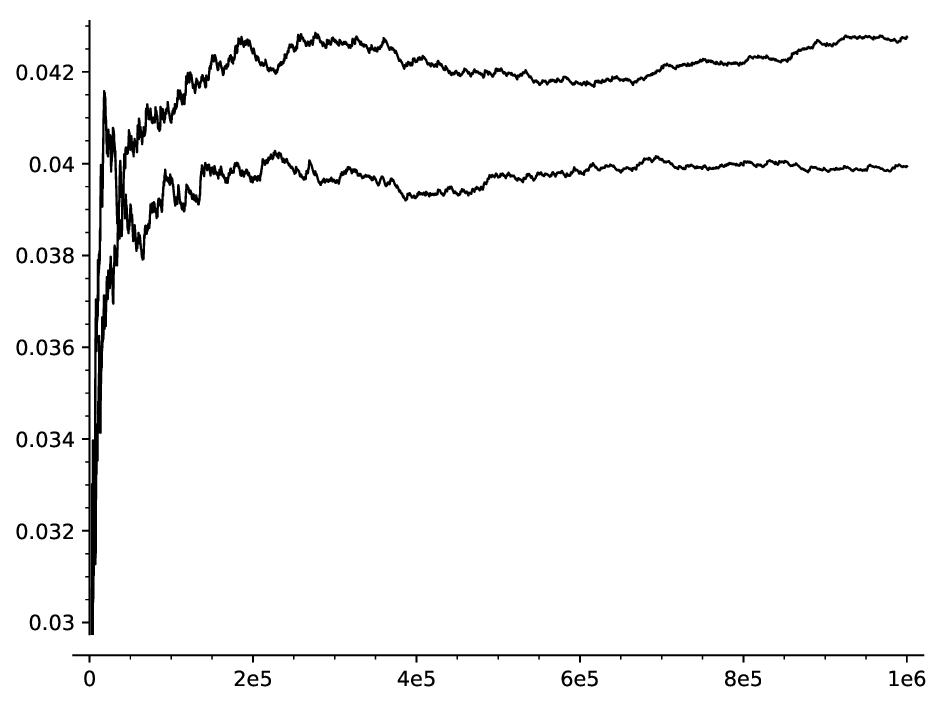}
\caption{$|l| = 3$: Top -3 bottom 3} \label{fig:37_6_1_6_A_3}
\end{subfigure}
\hspace*{-2.3cm}
\begin{subfigure}[b]{0.4\linewidth}
\includegraphics[width=\linewidth]{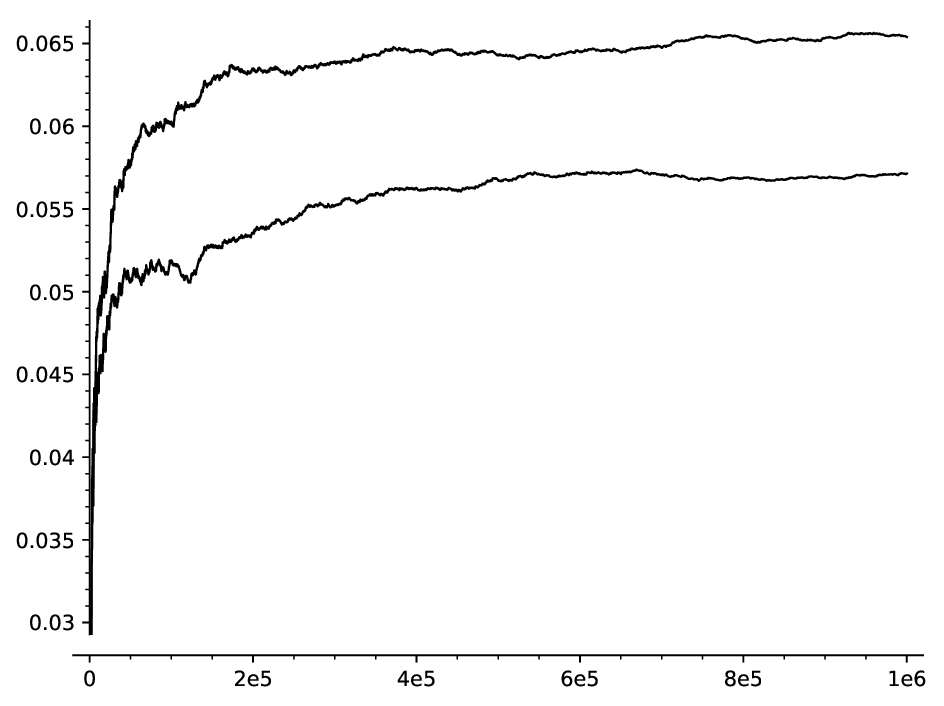}
\caption{$|l| = 4$: Top -4 bottom 4} \label{fig:37_6_1_6_A_4}
\end{subfigure}\hspace*{\fill}
\begin{subfigure}[b]{0.4\linewidth}
\includegraphics[width=\linewidth]{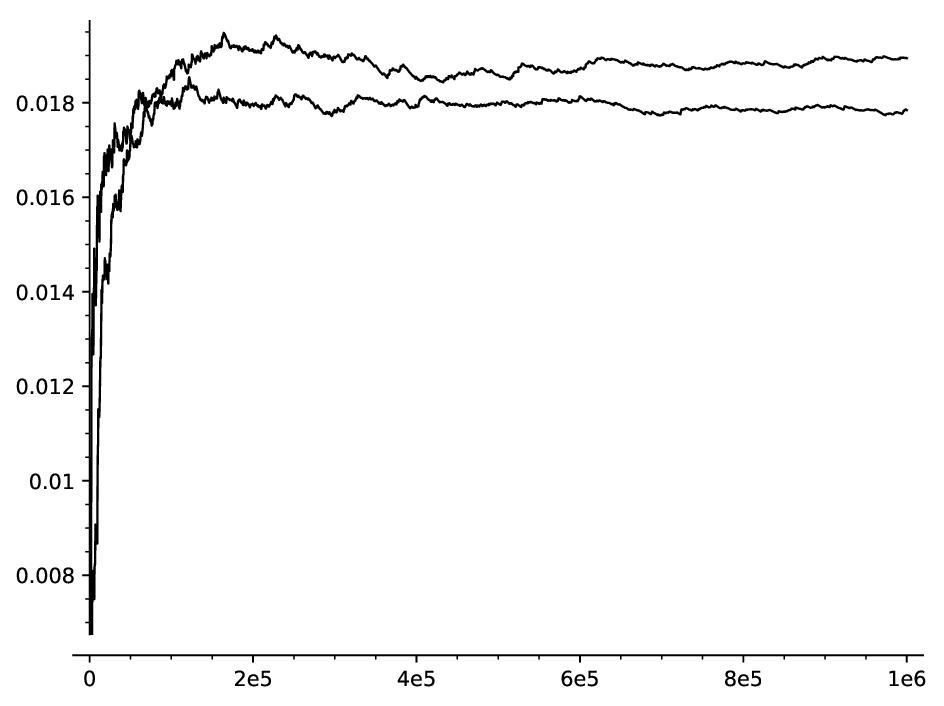}
\caption{$|l| = 5$: Top 5 bottom -5} \label{fig:37_6_1_6_A_5}
\end{subfigure}\hspace*{\fill}
\begin{subfigure}[b]{0.4\linewidth}
\includegraphics[width=\linewidth]{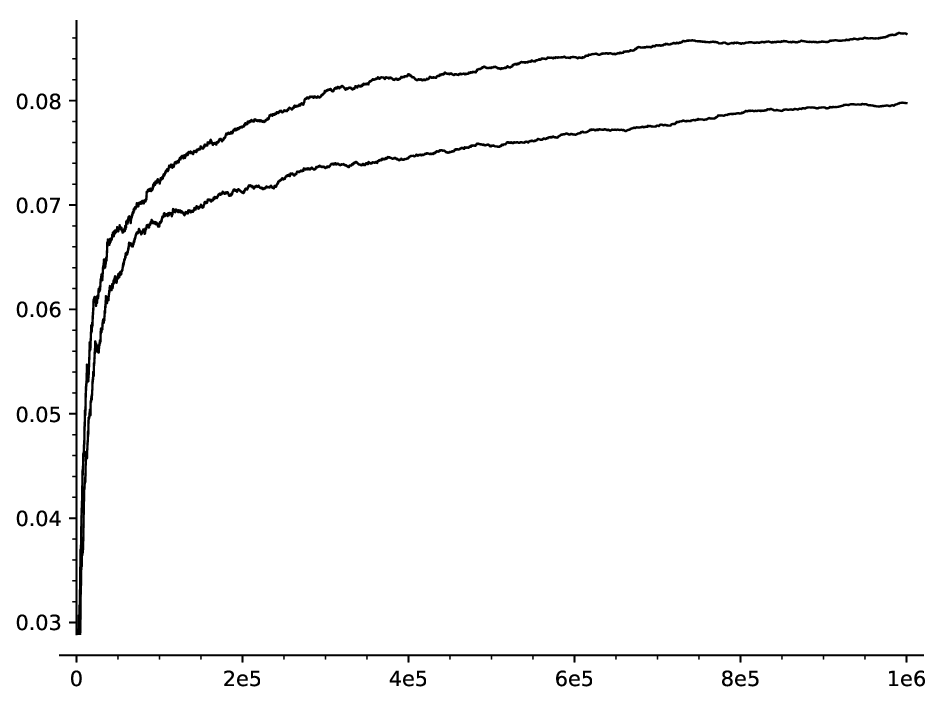}
\caption{$|l| = 6$: Top 6 bottom -6} \label{fig:37_6_1_6_A_6}
\end{subfigure}
\hspace*{-2.3cm}
\begin{subfigure}[b]{0.4\linewidth}
\includegraphics[width=\linewidth]{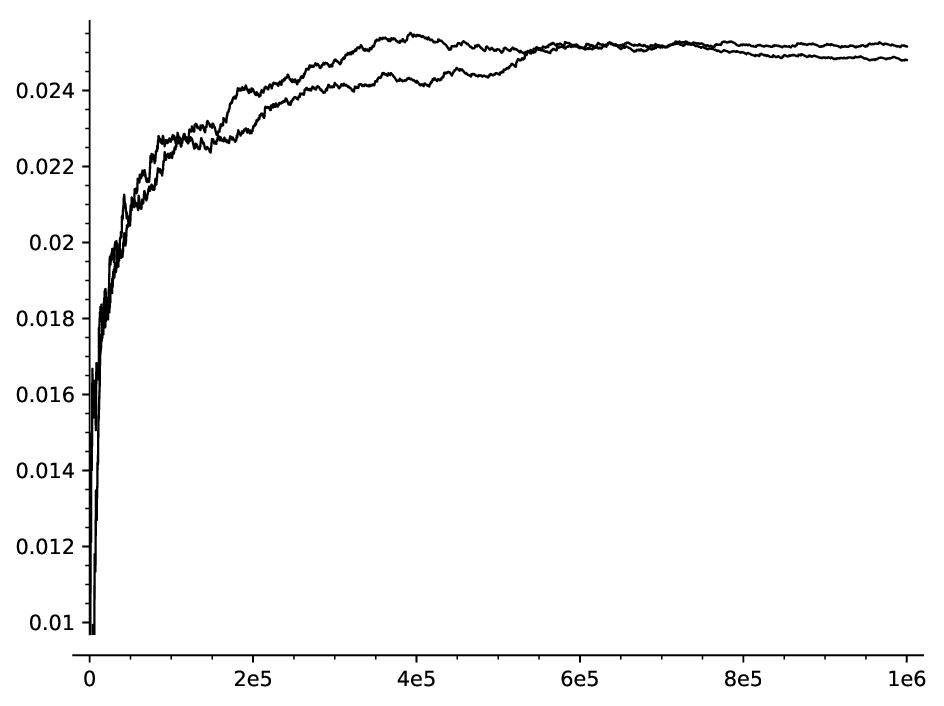}
\caption{$|l| = 7$: Top -7 bottom 7} \label{fig:37_6_1_6_A_7}
\end{subfigure}\hspace*{\fill}
\begin{subfigure}[b]{0.4\linewidth}
\includegraphics[width=\linewidth]{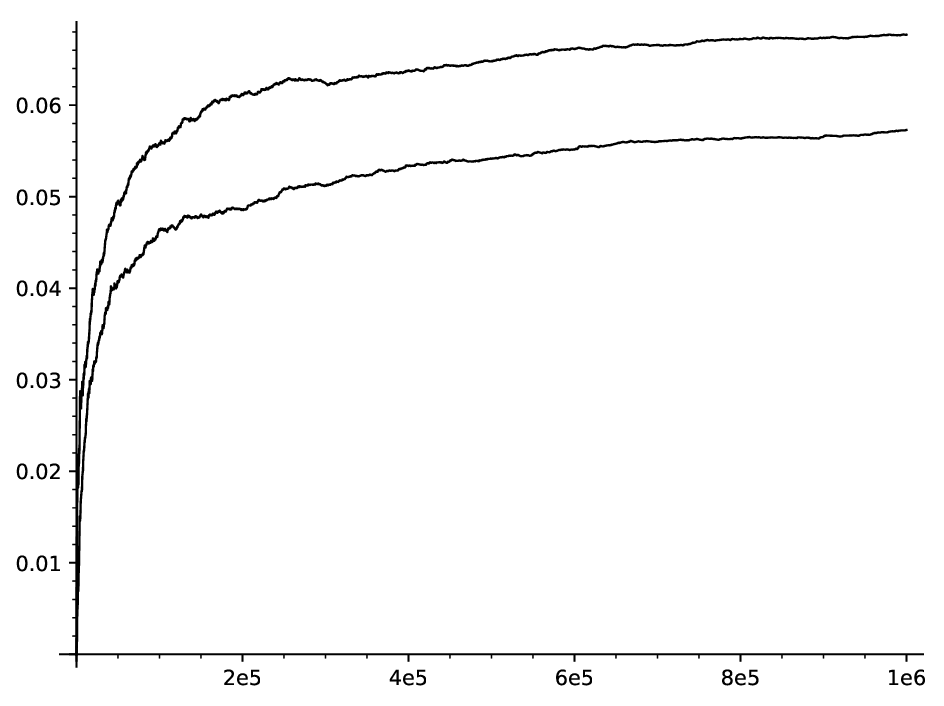}
\caption{$|l| = 8$: Top 8 bottom -8} \label{fig:37_6_1_6_A_8}
\end{subfigure}\hspace*{\fill}
\begin{subfigure}[b]{0.4\linewidth}
\includegraphics[width=\linewidth]{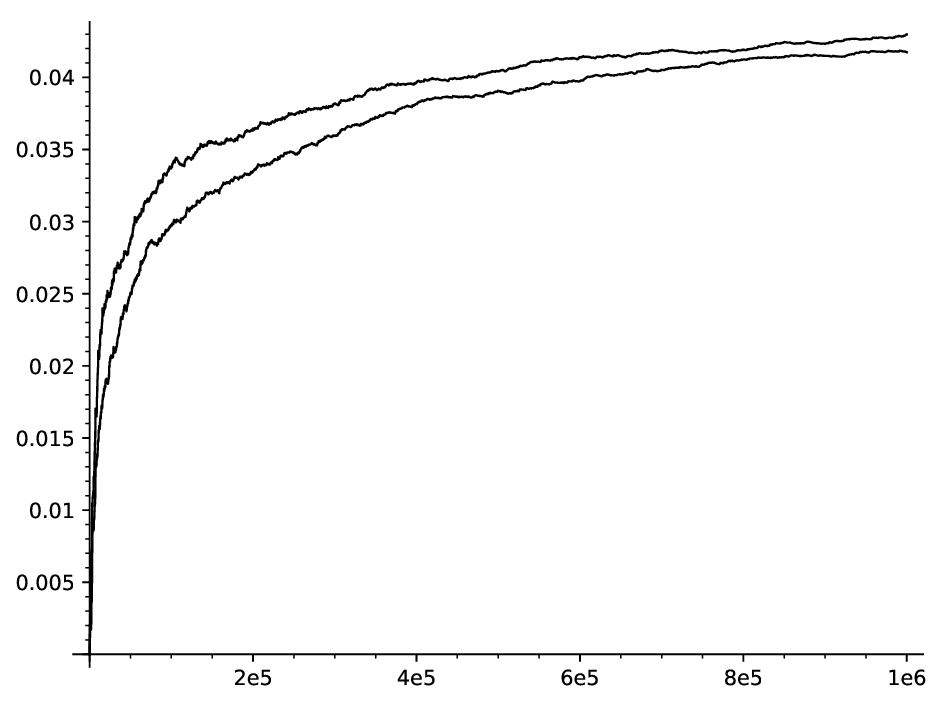}
\caption{$|l| = 9$: Top 9 bottom -9} \label{fig:37_6_1_6_A_9}
\end{subfigure}
\caption{37b1: $(\alpha, \beta) = (1,6)$ Ratio~\eqref{ratio_n_orders} $x_{6,E}^{(\alpha, \beta)}(X;l)/X^{1/2}\log^2(X)$} \label{fig:37b1_6_1_6_A_exact}
\end{figure}

\clearpage

\begin{figure}[t] 
\hspace*{-2.3cm}
\begin{subfigure}[b]{0.4\linewidth}
\includegraphics[width=\linewidth]{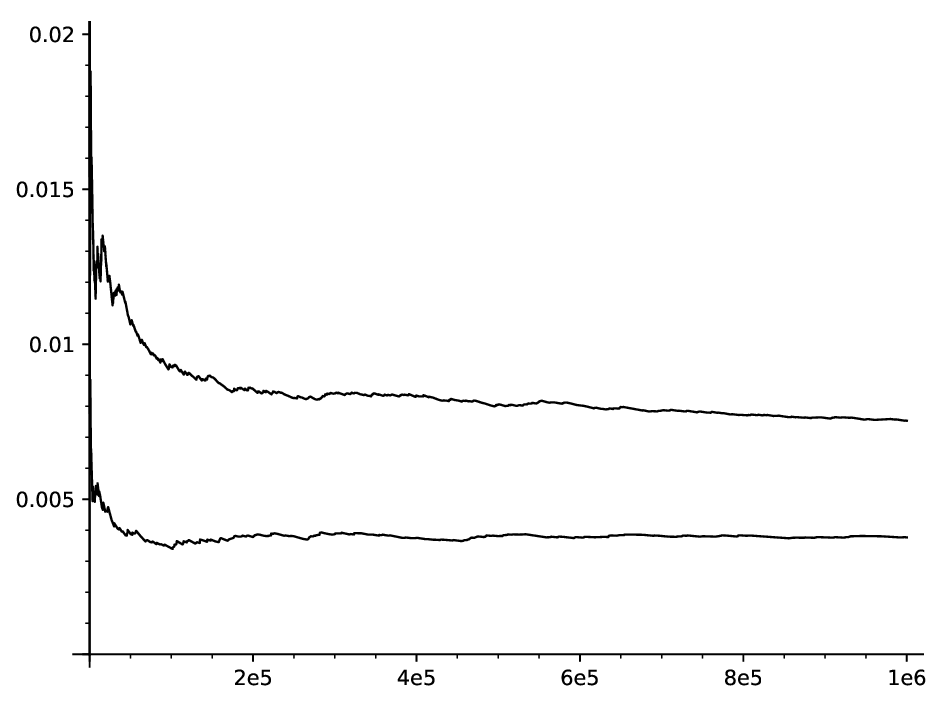}
\caption{$|l| = 1$: Top 1 bottom -1} \label{fig:37_6_2_6_A_1}
\end{subfigure}\hspace*{\fill}
\begin{subfigure}[b]{0.4\linewidth}
\includegraphics[width=\linewidth]{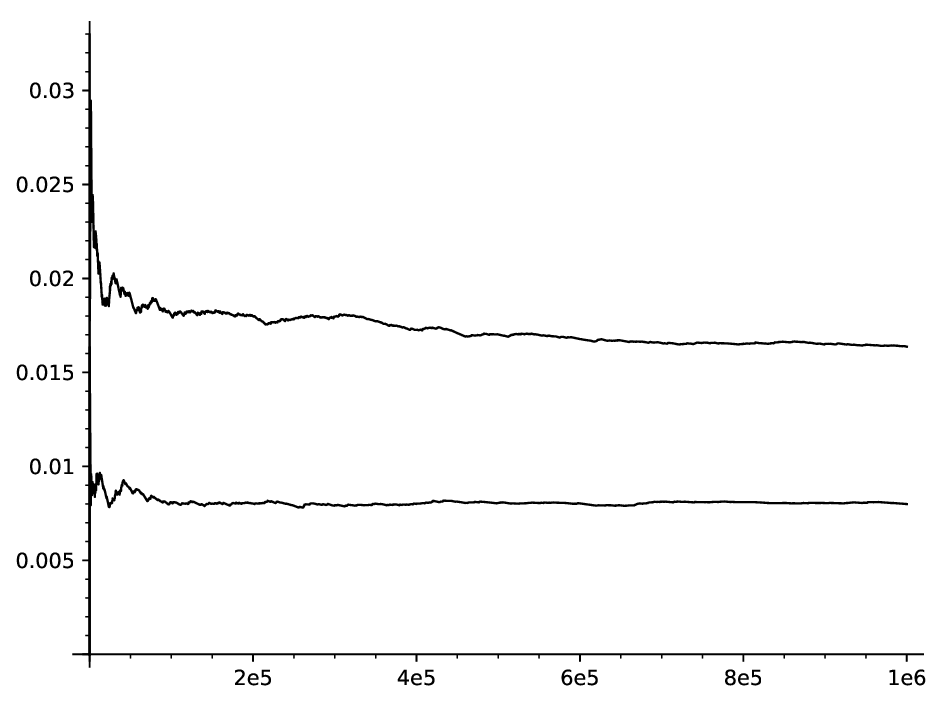}
\caption{$|l| = 2$: Top -2 bottom 2} \label{fig:37_6_2_6_A_2}
\end{subfigure}\hspace*{\fill}
\begin{subfigure}[b]{0.4\linewidth}
\includegraphics[width=\linewidth]{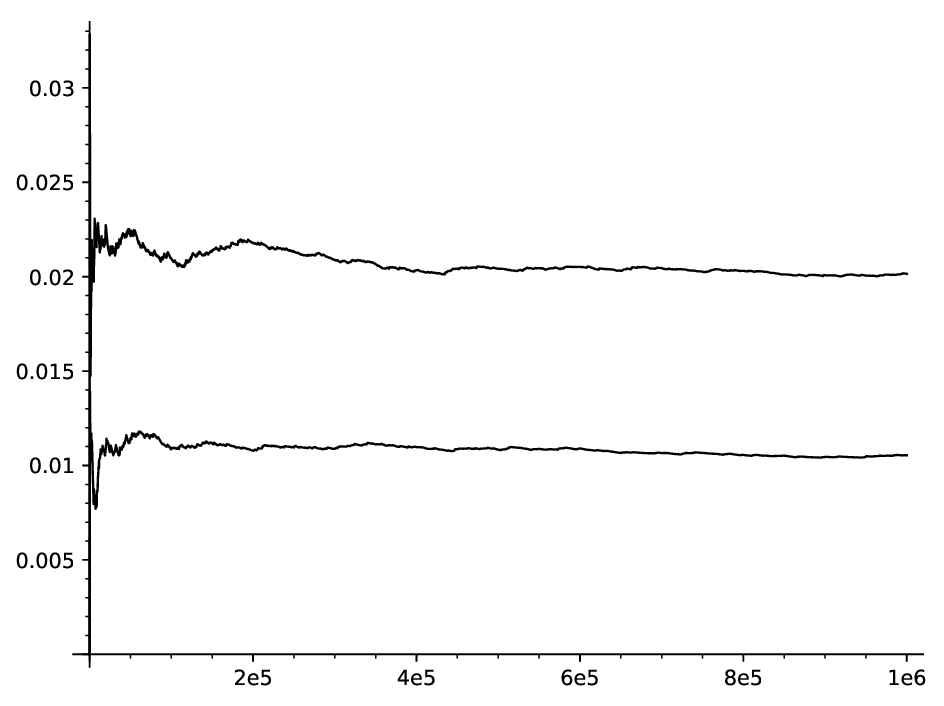}
\caption{$|l| = 3$: Top 3 bottom -3} \label{fig:37_6_2_6_A_3}
\end{subfigure}
\hspace*{-2.3cm}
\begin{subfigure}[b]{0.4\linewidth}
\includegraphics[width=\linewidth]{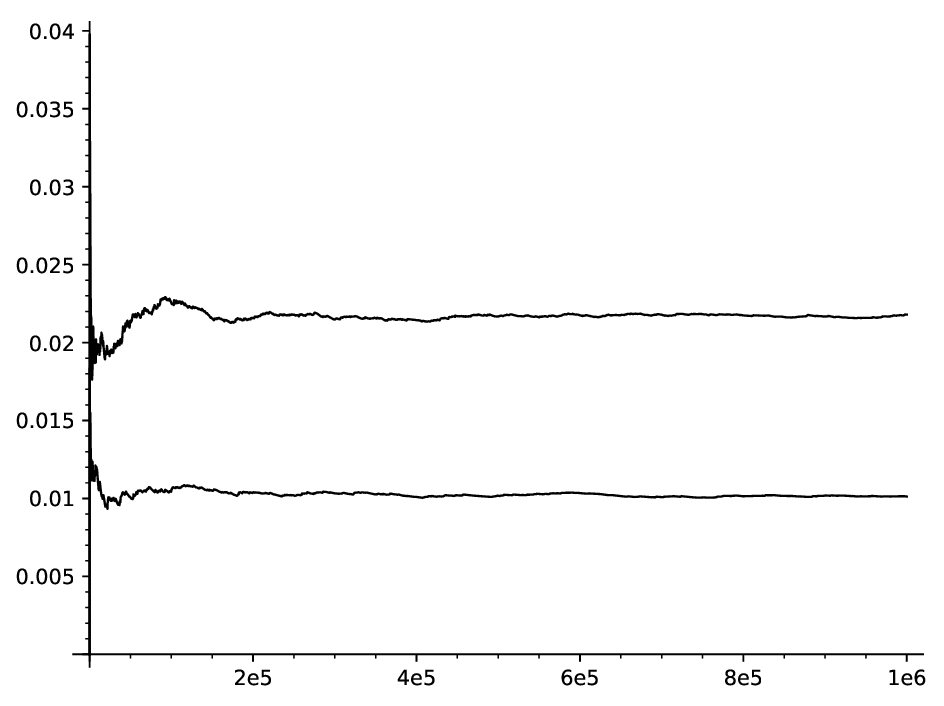}
\caption{$|l| = 4$: Top 4 bottom -4} \label{fig:37_6_2_6_A_4}
\end{subfigure}\hspace*{\fill}
\begin{subfigure}[b]{0.4\linewidth}
\includegraphics[width=\linewidth]{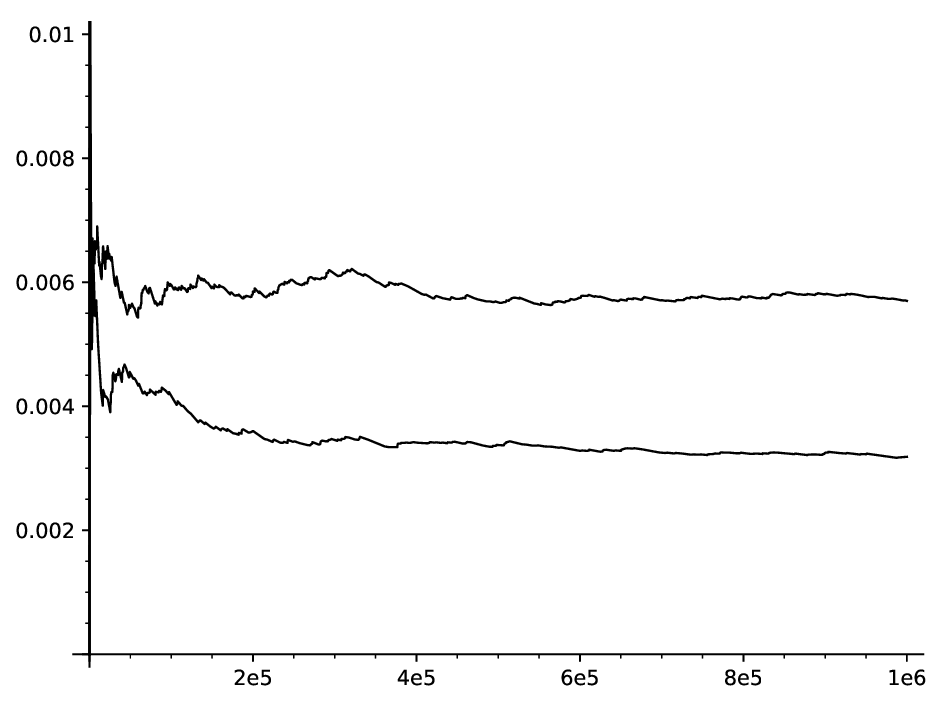}
\caption{$|l| = 5$: Top -5 bottom 5} \label{fig:37_6_2_6_A_5}
\end{subfigure}\hspace*{\fill}
\begin{subfigure}[b]{0.4\linewidth}
\includegraphics[width=\linewidth]{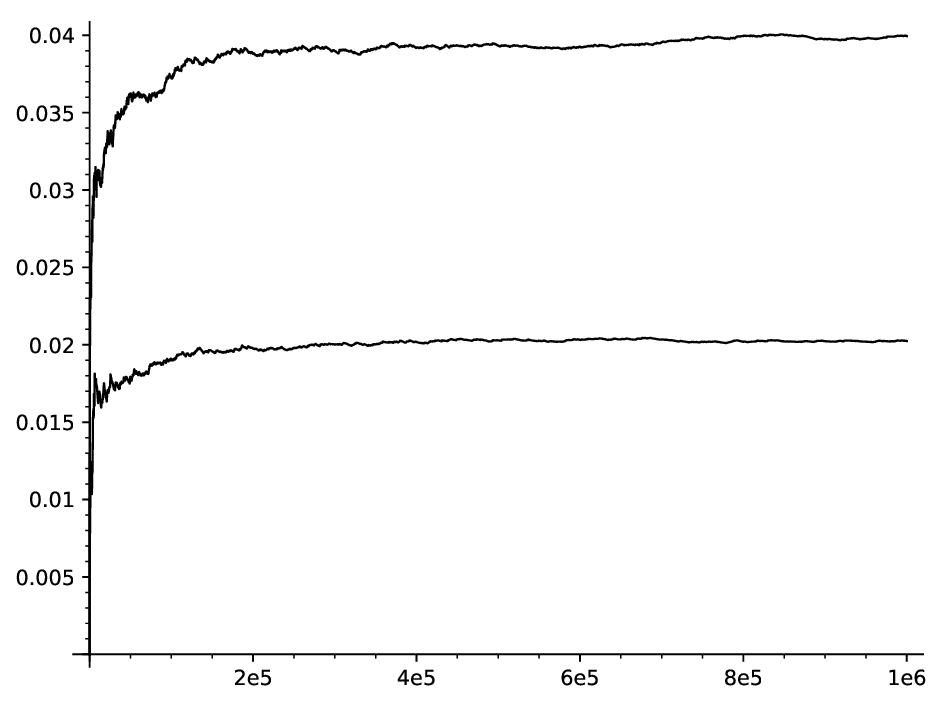}
\caption{$|l| = 6$: Top -6 bottom 6} \label{fig:37_6_2_6_A_6}
\end{subfigure}
\hspace*{-2.3cm}
\begin{subfigure}[b]{0.4\linewidth}
\includegraphics[width=\linewidth]{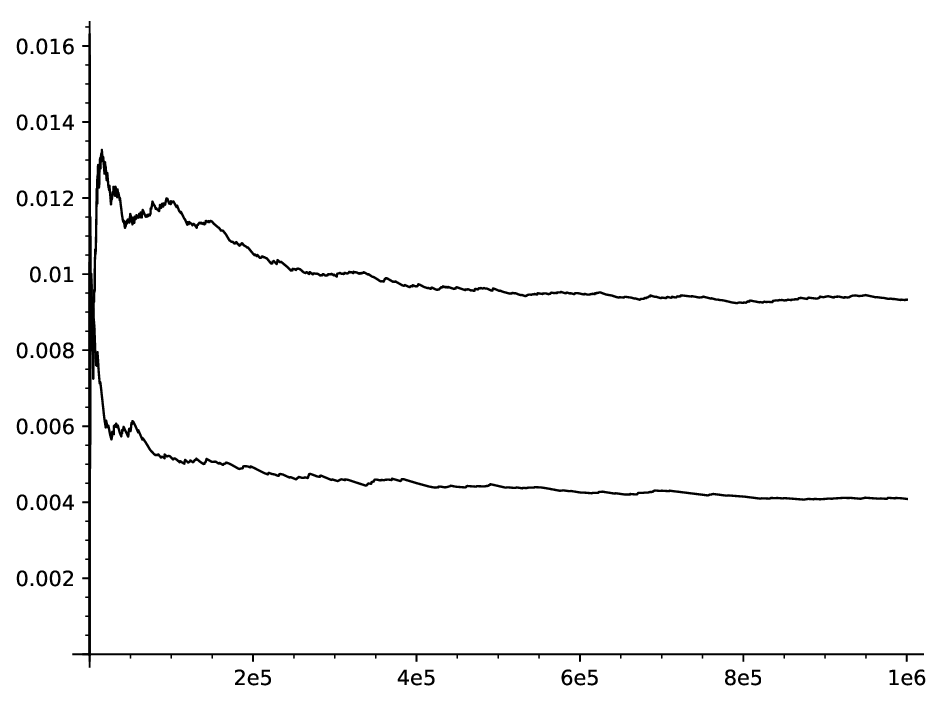}
\caption{$|l| = 7$: Top 7 bottom -7} \label{fig:37_6_2_6_A_7}
\end{subfigure}\hspace*{\fill}
\begin{subfigure}[b]{0.4\linewidth}
\includegraphics[width=\linewidth]{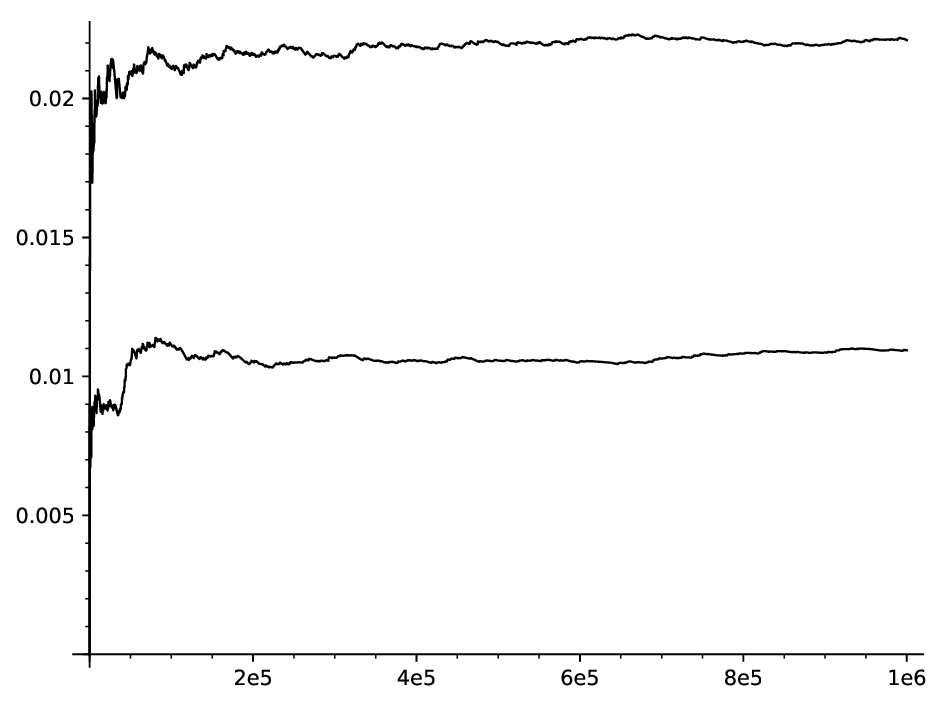}
\caption{$|l| = 8$: Top -8 bottom 8} \label{fig:37_6_2_6_A_8}
\end{subfigure}\hspace*{\fill}
\begin{subfigure}[b]{0.4\linewidth}
\includegraphics[width=\linewidth]{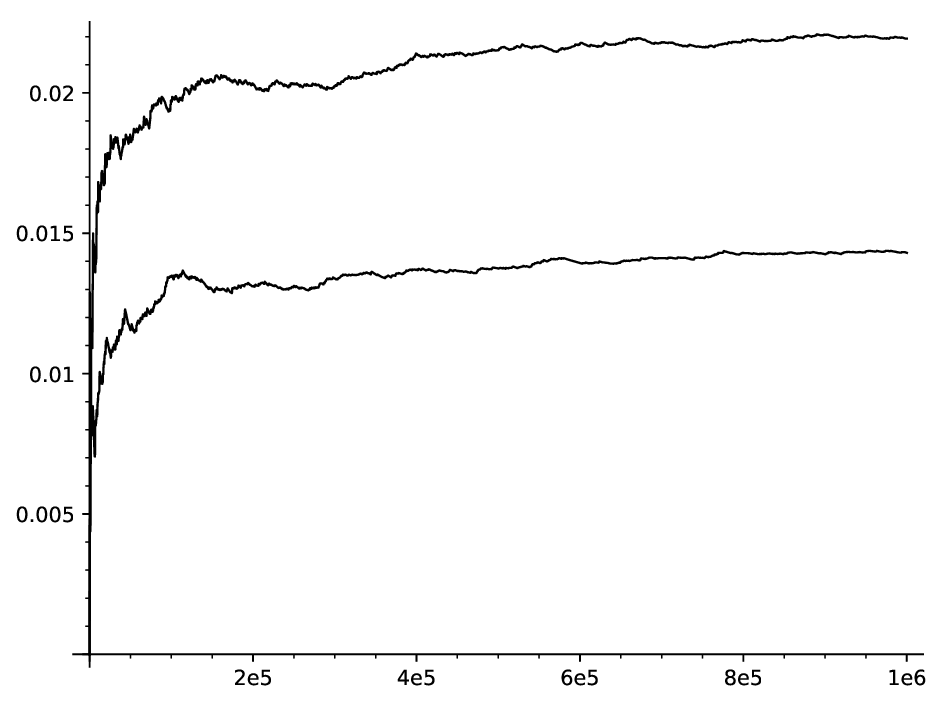}
\caption{$|l| = 9$: Top 9 bottom -9} \label{fig:37_6_2_6_A_9}
\end{subfigure}
\caption{37b1: $(\alpha, \beta) = (2,6)$ Ratio~\eqref{ratio_n_orders} $x_{6,E}^{(\alpha, \beta)}(X;l)/X^{1/2}\log^2(X)$} \label{fig:37b1_6_2_6_A_exact}
\end{figure}




\begin{thebibliography}{BCDTXX}

\bibitem[An]{An}
David~K.~Angdinata, \emph{$L$-values of elliptic curves twisted by cubic characters}, arXiv:2401.099227v1 [math.NT] 18 Jan 2024.


\bibitem[BRY]{BRY}
Jennifer Berg, Nathan~C. Ryan, and Matthew~P. Young, \emph{Vanishing of quartic
  and sextic twists of $l$-functions}, arXiv:2301.05329, 2023.


\bibitem[CKRS]{CKRS}
B.~Conrey, J.~Keating, M.~Rubinstein and N.~Snaith,\emph{On the frequency of vanishing
of quadratic twists of modular $L$-functions}, Number theory for the
Millenium I,  A. K. Peters Ltd., Natick, (2002),  pp.~301-315.

\bibitem[DEW]{D-E-W}
V.~Dokchitser, R.~Evans, and H.~Wiersema,\emph{On a B-SD-type formula for $L$-values of Artin twists of elliptic curves},
J. reine angew. Math. \textbf{773} (2021), 199-230.




\bibitem[DFK1]{DFK1}
 C.~David, J.~Fearnley and H.~Kisilevsky, \emph{On the vanishing of
  twisted {$L$}-functions of elliptic curves}, Experiment. Math. \textbf{13}
  (2004), no.~2, 185--198. MR2068892 (2005e:11082)

\bibitem[DFK2]{DFK2}
C.~David, J.~Fearnley and H.~Kisilevsky, \emph{Vanishing of
$L$-functions of elliptic curves over number fields}, Ranks of
Elliptic Curves and Random Matrix Theory, London Mathematical Society
Lecture Note Series, \textbf{341}, Cambridge University Press (2007), pp.~247-259.





\bibitem[DW]{D-W}
A. ~D\k{a}browski and M. ~Wodzicki, \emph{Elliptic curves with large analytic order of $\Sha(E)$}, 
In: Algebra, Arithmetic and Geometry (in honour of Yu. I. Manin, vol. I), Progress in Math. 269 (2009), 407–421. 
MR2641177

\bibitem[FKK]{FKK}
J.~Fearnley, H.~Kisilevsky and M.~Kuwata \emph{Vanishing and non-vanishing Dirichlet twists of $L$-functions of elliptic curves}, Journal of the London Mathematical Society, Volume 86, Issue 2, (Oct. 2012), pp.~539-557.

\bibitem[FL21]{FLINT}
W.~Hart, F.~Johansson, and S.~Pancratz, \emph{{FLINT}: {F}ast {L}ibrary for {N}umber {T}heory ({V}ersion 2.7.1)}, 2021, {\tt http://flintlib.org}.

\bibitem[Gol]{G}
Dorian Goldfeld, \emph{Conjectures on elliptic curves over quadratic fields},
  Number theory, Carbondale 1979 (Proc. Southern Illinois Conf., Southern
  Illinois Univ., Carbondale, Ill., 1979), Lecture Notes in Math., vol. 751,
  Springer, Berlin, 1979, pp.~108--118. MR564926 (81i:12014) 

\bibitem[GHP]{G-H-P}
D.~Goldfeld, J.~Hoffstein, and S.J.~Patterson, \emph{On automorphic functions of half-integral weight
with applications to elliptic curves}, In: Number Theory Related to Fermat's Last Theorem 
(Birkhauser, Boston, MA, 1982) pp.152 -- 193

\bibitem[GM]{Go-M}
F.~Gouv\^ea and B.~Mazur, \emph{The square-free sieve and the rank of elliptic
  curves}, J. Amer. Math. Soc. \textbf{4} (1991), 1--23.

\bibitem[HiPa]{Hi-Pa}
M.~Hindry and A.~Pacheco, \emph{An analogue of the Brauer-Siegel theorem for abelian varieties in positive characteristic}, Mosc. Math. J. \textbf{16} (2016) no. ~1, 45--93. MR3470576

\bibitem[Hi]{Hi}
M.~Hindry, \emph{Analogues of the Brauer-Siegel theorem in arithmetic geometry}, Contemporary Mathematics, vol.\textbf{722}, (2019) 19--41.



\bibitem[KS]{Ka-Sa}
Nicholas~M.~Katz and Peter Sarnak, \emph{Random matrices, {F}robenius
  eigenvalues, and monodromy}, American Mathematical Society Colloquium
  Publications, vol.~45, American Mathematical Society, Providence, RI, 1999.
  MR1659828 (2000b:11070)


\bibitem[Kat]{Kato}
Kazuya Kato, \emph{{$p$}-adic {H}odge theory and values of zeta functions of
  modular forms}, no. 295, 2004, Cohomologies $p$-adiques et applications
  arithm\'{e}tiques. III, pp.~ix, 117--290. 2104361

\bibitem[KN]{Ki-Na}
H.~Kisilevsky and J.~Nam, \emph{Small Algebraic Central Values of Twists of Elliptic {$L$}-Functions},
 arXiv:2001.03547v3, 2022

\bibitem[Kol]{Ko}
V.~A. Kolyvagin, \emph{Finiteness of {$E({\bf Q})$} and {SH{$(E,{\bf Q})$}} for
  a subclass of {W}eil curves}, Izv. Akad. Nauk SSSR Ser. Mat. \textbf{52}
  (1988), no.~3, 522--540, 670--671. MR954295 (89m:11056)

\bibitem[L]{Lang}
S.~Lang, \emph{Algebraic Number Theory}, Addison-Wesley;Reading MA, 1970

\bibitem[LMFDB]{LMFDB}
The LMFDB Collaboration, \emph{The $L$-functions and modular forms database}, 
https://www.lmfdb.org, 2024



\bibitem[MR]{M-R}
B.~Mazur and K.~Rubin, \emph{Arithmetic Conjectures Suggested by the Statistical Behavior of Modular Symbols},
Experimental Math. (2021). pp 1-16,  arXiv1910.12798, 2019

\bibitem[MTT]{M-T-T}
B.~Mazur, J.~Tate, and J.~Teitelbaum, \emph{On {$p$}-adic analogues of the
  conjectures of {B}irch and {S}winnerton-{D}yer}, Invent. Math. \textbf{84}
  (1986), no.~1, 1--48. MR830037 (87e:11076)

\bibitem[Nam]{Nam}
J.~Nam, \emph{Data for the Central Values of {$L$}-functions of Elliptic Curves Twisted by Dirichlet Characters 
and Their Algebraic and Integer Values}, 0.1.0, Zenodo, (2022), doi:10.5281/zenodo.7255396.

\bibitem[Na]{Nark}
W.~Narkiewicz, \emph{Elementary and Analytic Theory of Algebraic Numbers} Third Edition,
Springer Monographs in Mathematics, Springer (2004).

\bibitem[NVF]{NVF}
NVIDIA\textregistered, P.~Vingelmann and F.~H.P.~Fitzek, \emph{{CUDA} (release: 11.5.50)}, 2021, {\tt https://developer.nvidia.com/cuda-toolkit}.

\bibitem[PA19]{PARI2} 
The PARI Group, \emph{{PARI/GP} ({V}ersion 2.13.2)}, Univ. Bordeaux, 2021, {\tt http://pari.math.u-bordeaux.fr}.

\bibitem[Roh]{Ro}
David~E. Rohrlich, \emph{The vanishing of certain {R}ankin-{S}elberg
  convolutions}, Automorphic forms and analytic number theory (Montreal, PQ,
  1989), Univ. Montr\'eal, Montreal, QC, 1990, pp.~123--133. MR1111015
  (92d:11051)
  
  
\bibitem[TsVl]{Ts-Vl}
 Tsafsman,~M, and Vl\u{a}du\c{t}, \emph{Infinite global fields and the generalized Brauer--Siegel theorem}, Mosc. Math. \textbf{2}, no.2, 329--402.
 MR1944510
  
  
  
  
  
\bibitem[Sch]{Sch}
A.~J. Scholl, \emph{An introduction to {K}ato's {E}uler systems}, Galois
  representations in arithmetic algebraic geometry (Durham, 1996), London Math.
  Soc. Lecture Note Ser., vol. 254, Cambridge Univ. Press, Cambridge, 1998,
  pp.~379--460. MR1696501 (2000g:11057)




\bibitem[S{\etalchar{+}}19]{sage}
SageMath, the {S}age {M}athematics {S}oftware {S}ystem ({V}ersion
  8.6.rc0), The Sage Developers, 2019, {\tt https://www.sagemath.org}.
  
  
  \bibitem[Ul]{Ul}
  Ulmer, ~D, \emph{On the Brauer-Siegel ratio for abelian varieties over function fields}, arXiv:1806.01961v2 [math.NT], 28 Feb. 2019
  

\bibitem[W-W]{Wuthrich}
H.~Wiersema and C.~Wuthrich, \emph{Integrality of twisted $L$-values of elliptic curves}, arXiv:2004.05492v2 [math.NT] 8 Feb. 2021

\bibitem[Zy]{Zy}
Zykin,~A, \emph{On the generalizations of the Brauer-Siegel theorem}, 
   Contemp. Math., \textbf{487}, 195--206,
Amer. Math. Soc, Providence R.I., 2009.


\end{thebibliography}
\end{document}